\documentclass[10pt,letterpaper]{amsart}
\usepackage[T1]{fontenc}

\makeatletter
\let\@wraptoccontribs\wraptoccontribs
\makeatother

\usepackage[
  left=2.54cm,
  right=2.54cm,
  top=3.3cm,
  bottom=4.5cm
]{geometry}

\usepackage{amssymb}
\usepackage{comment}
\usepackage[all,curve]{xy}
\usepackage{enumerate}

\usepackage{bbm}

\usepackage[dvipsnames]{xcolor}
\usepackage{graphicx}

\usepackage{enumitem}

\usepackage{quiver}
\usetikzlibrary{spath3,nfold}

\tikzset{
  between/.style n args={2}{
    /tikz/execute at end to={
      \tikzset{spath/split at keep middle={current}{#1}{#2}}
    }
  }
}

\usepackage{amssymb}
\usepackage[bbgreekl]{mathbbol}

\DeclareSymbolFontAlphabet{\mathbb}{AMSb}
\DeclareSymbolFontAlphabet{\mathbbl}{bbold}

\newcommand{\Prism}{\mathbbl{\Delta}}

\usepackage{amsmath}
\usepackage{amsxtra}
\usepackage{amscd}
\usepackage{amsthm}
\usepackage{amsfonts}
\usepackage{amssymb}
\usepackage{eucal}
\usepackage{verbatim}
\usepackage[all]{xy}
\usepackage{mathrsfs}
\usepackage{lineno}
\usepackage{color}
\definecolor{deepjunglegreen}{rgb}{0.0, 0.29, 0.29}
\definecolor{darkspringgreen}{rgb}{0.09, 0.45, 0.27}

\usepackage{stmaryrd}

\makeatletter
\usepackage{etex,etoolbox,needspace}
\pretocmd\section{\Needspace*{4\baselineskip}}{}{}
\makeatletter

\usepackage[pdftex,bookmarks=false,colorlinks=true,citecolor=darkspringgreen,debug=true,
  naturalnames=true,pdfnewwindow=true]{hyperref}

\newcommand{\red}[1]{\leavevmode{\color{red}{#1}}}
\newcommand{\green}[1]{\leavevmode{\color{ForestGreen}{#1}}}

\DeclareMathOperator{\Char}{{char}}

\DeclareMathOperator{\Gra}{{Gr}}

\usepackage{tikz-cd}

\newcommand{\rar}[1]{\stackrel{#1}{\longrightarrow}}

\newcommand{\reduced}{\mathrm{red}}

\newcommand{\sn}{\mathrm{sn}}

\newcommand{\msf}{\mathsf}
\newcommand{\bA}{{\mathbb A}}

\newcommand{\bC}{{\mathbb C}}

\newcommand{\bF}{{\mathbb F}}
\newcommand{\bG}{{\mathbb G}}

\newcommand{\bP}{{\mathbb P}}
\newcommand{\bQ}{{\mathbb Q}}

\newcommand{\bZ}{{\mathbb Z}}

\newcommand{\cA}{{\mathcal A}}
\newcommand{\cB}{{\mathcal B}}
\newcommand{\cC}{{\mathcal C}}
\newcommand{\cD}{{\mathcal D}}
\newcommand{\cE}{{\mathcal E}}
\newcommand{\cF}{{\mathcal F}}

\newcommand{\cI}{{\mathcal I}}

\newcommand{\cL}{{\mathcal L}}
\newcommand{\cM}{{\mathcal M}}

\newcommand{\cO}{{\mathcal O}}
\newcommand{\cP}{{\mathcal P}}
\newcommand{\cQ}{{\mathcal Q}}
\newcommand{\cR}{{\mathcal R}}

\newcommand{\cV}{{\mathcal V}}
\newcommand{\cW}{{\mathcal W}}
\newcommand{\cX}{{\mathcal X}}
\newcommand{\cY}{{\mathcal Y}}

\newcommand{\cHom}{\underline{\operatorname{Hom}}}
\newcommand{\cExt}{\underline{\operatorname{Ext}}}
\newcommand{\cTor}{\underline{\operatorname{Tor}}}

\newcommand{\sC}{{\mathscr C}}

\newcommand{\fC}{{\mathfrak C}}

\newcommand{\fH}{{\mathfrak H}}

\newcommand{\fS}{{\mathfrak S}}

\newcommand{\fg}{{\mathfrak g}}

\newcommand{\fm}{{\mathfrak m}}

\newcommand{\et}{\text{et}}

\newcommand{\colim}{\text{colim}}

\newcommand{\ta}{\tilde{a}}
\newcommand{\tb}{\tilde{b}}

\newcommand{\tf}{\tilde{f}}
\newcommand{\tg}{\tilde{g}}

\newcommand{\tx}{\widetilde{x}}
\newcommand{\ty}{\widetilde{y}}

\newcommand{\nc}{\newcommand}

\nc\wh{\widehat}

\nc\on{\operatorname}

\nc\Gr{\on{Gr}}

\nc\Fl{\on{Fl}}

\newtheorem{problem}{Problem}

\DeclareMathOperator{\gr}{{gr}}

\DeclareMathOperator{\Lie}{{Lie}}

\DeclareMathOperator{\Mat}{{Mat}}
 \DeclareMathOperator{\Spf}{{Spf}}
\DeclareMathOperator{\Sym}{{Sym}}

\DeclareMathOperator{\Mod}{{Mod}}

\DeclareMathOperator{\coker}{{coker}}
\DeclareMathOperator{\Sp}{{Sp}}

\DeclareMathOperator{\Fib}{{Fib}}

\newcommand{\limto}{{\displaystyle\lim_{\longrightarrow}}}
\newcommand{\rightlim}{\mathop{\limto}}

\newcommand{\leftlim}{\mathop{\displaystyle\lim_{\longleftarrow}}}
\newcommand{\limfromn}{\leftlim\limits_{\raise3pt\hbox{$n$}}}
\newcommand{\limton}{\rightlim\limits_{\raise3pt\hbox{$n$}}}

\newcommand{\rightlimit}[1]{\mathop{\lim\limits_{\longrightarrow}}\limits%
                    _{\raise3pt\hbox{$\scriptstyle #1$}}}

\newcommand{\leftlimit}[1]{\mathop{\lim\limits_{\longleftarrow}}\limits%
                    _{\raise3pt\hbox{$\scriptstyle #1$}}}

\newcommand{\epi}{\twoheadrightarrow}
\newcommand{\iso}{\buildrel{\sim}\over{\longrightarrow}}
\newcommand{\mono}{\hookrightarrow}

\newcommand{\lmod}{{\operatorname{\mathsf{-Mod}}}}
\newcommand{\bimod}{{\operatorname{\mathsf{-BiMod}}}}

\DeclareMathOperator{\Id}{{Id}}
\DeclareMathOperator{\Aut}{{Aut}}

\DeclareMathOperator{\End}{{End}} \DeclareMathOperator{\Ext}{{Ext}}
\DeclareMathOperator{\Hom}{{Hom}} \DeclareMathOperator{\Ind}{{Ind}}

\DeclareMathOperator{\Ker}{{Ker}} \DeclareMathOperator{\id}{{id}}
\DeclareMathOperator{\im}{{Im}} 

 \DeclareMathOperator{\op}{{op}}

\DeclareMathOperator{\Spec}{{Spec}}

\DeclareMathOperator{\PGL}{{PGL}}
\DeclareMathOperator{\GL}{{GL}}
\DeclareMathOperator{\Pic}{{Pic}}

\DeclareMathOperator{\QCoh}{{QCoh}}

\DeclareMathOperator{\rk}{{rk}}
\DeclareMathOperator{\ev}{{ev}}
\DeclareMathOperator{\Span}{{span}}

\newcommand{\pro}{{\mr{pr}}}

\DeclareMathOperator{\Ad}{{Ad}}

\makeatletter

\newcommand{\Rmnum}[1]{\expandafter\@slowromancap\romannumeral #1@}
\makeatother

\newtheorem{pr}{Proposition}[section]
\newtheorem{Th}[pr]{Theorem}
\newtheorem{lm}[pr]{Lemma}

\newtheorem{cor}[pr]{Corollary}

\theoremstyle{definition}

\newtheorem{example}[pr]{Example}
\newtheorem{df}[pr]{Definition}
\newtheorem{rem}[pr]{Remark}

\newtheorem{question}[pr]{Question}
\newtheorem{warning}[pr]{Warning}

\numberwithin{equation}{section}

\newcommand{\Der}{\operatorname{Der}}

\newcommand{\Br}{\operatorname{Br}}

\newcommand{\Fun}{\operatorname{Fun}}

\newcommand{\tr}{\operatorname{tr}}
\newcommand{\modulo}{\operatorname{mod}}

\DeclareMathOperator{\NPic}{NPic}
\newcommand{\dR}{\mathrm{dR}}

\newcommand{\mbb}{\mathbb}

\DeclareMathOperator{\ad}{{ad}}

\DeclareMathOperator{\Center}{{Center}}

\DeclareMathOperator{\CAlg}{{CAlg}}
\DeclareMathOperator{\Aff}{{Aff}}
\DeclareMathOperator{\Sch}{{Sch}}
\newcommand{\sep}{\mathrm{sep}}
\newcommand{\ul}{\underline}
\usepackage{mathtools}
\usepackage{appendix}
\newcommand{\ra}{\rightarrow}
\newcommand{\xra}{\xrightarrow}

\newcommand{\mc}{\mathcal}

\newcommand{\mr}{\mathrm}
\newcommand{\twt}{^{(1)}}
\newcommand{\lp}{{\{ \!\!\!\ p \ \!\!\!\}}}
\newtheorem{constr}[pr]{Construction}

\nc\hattimes{\mathbin{\hat\times}}
\begin{document}

\title[The canonical twistor quantization of symplectic varieties in char.\ $p$]
{The canonical twistor quantization of symplectic varieties in characteristic $p$}




\author[E.~Bogdanova]{Ekaterina Bogdanova}
\address{Yale University,  USA}
\email{ekaterina.bogdanova@yale.edu}
\author[D.~Kubrak]{Dmitry Kubrak}
\address{CNRS, Universit\'e Paris Saclay, France}
\email{dmkubrak@gmail.com}
\author[R.~Travkin]{Roman Travkin}
\address{Skolkovo Institute of Science and Technology, Russia}
\email{roman.travkin2012@gmail.com}
\author[V. ~Vologodsky]{Vadim Vologodsky}
\address{University of Toronto,  Canada}
\email{vologod@gmail.com}

\begin{abstract}
 Let  $X$ be  a smooth symplectic variety over a  field $k$ of characteristic $p>2$ equipped with a restricted structure,  which is a class $[\eta] \in H^0(X, \Omega^1_X/d\mathcal O_X)$ whose de Rham differential equals the symplectic form. In this paper we construct a functorial in $(X, [\eta])$ sheaf of categories $\mathrm{QCoh}_{\mathbb P^{\mathsf{tw}}}$ over $\mathbb P^1$, which we call the ``canonical twistor quantization'' of $(X,[\eta])$. Its fiber over $0$ is $\QCoh(X)$, its restriction to the formal disc around $0$ is Zariski locally on $X$ given by modules over the Bezrukavnikov--Kaledin quantization, while its fiber over $\infty$ is the category $\QCoh(X^{(1)})$ of quasi-coherent sheaves on the Frobenius-twist. In the presence of a $\mathbb G_m$-action on $X$ that rescales $[\eta]$ by a character this sheaf of categories naturally descends to $\mathbb P^1/\mathbb G_m$. We also identify the Rozansky-Witten invariant $\mathrm{RW}_\Theta(X,[\eta])$ as an obstruction for $\mathrm{QCoh}_{\mathbb P^{\mathsf{tw}}}$ to come as modules over a sheaf of algebras, and show that it is non-zero in the case of a superspecial K3 surface.

 \end{abstract}

\maketitle

\tableofcontents



\section{Introduction}

\subsection{Plan} The main goal of our paper is, given a restricted symplectic variety in characteristic $p$, to construct a certain family of categories over $\mathbb P^1$ containing $\QCoh(X)$ as the fiber over $0$. We propose this family as a quantization of the category $\QCoh(X)$ which is parametrized by  $\mathbb P^1$ instead of just the formal disk around $\{0\}\in\mathbb P^1$. To better motivate our results we have made an attempt to put them into a larger context, including a rather detailed discussion of categorical quantizations in characteristic 0. This resulted in writing a long introduction, where one can find the actual results of this paper only closer to the end. To better navigate the reader, we feel that a short plan of the introduction itself is needed. \begin{itemize}
    \item 
In Section \ref{intro:Quantization of algebras} we review the general concept of a quantization of a Poisson structure in algebro-geometric context.
\item In Section \ref{ssec:formal quantizations of schemes in char 0} we discuss the quantizations of \textit{symplectic} schemes in characteristic 0, and explain and review a construction of the  Bezrukavnikov--Kaledin ``period map''.
\item In Section \ref{ssec:categorical quantizations in char 0} we introduce the concept of categorical quantizations, and discuss their classification in the symplectic case in characteristic $0$. This leads to a reinterpretation of the Bezrukavnikov-Kaledin period map as the map associating to a quantization of algebras its underlying categorical quantization.

\item In Section \ref{ssec:restricted symplectic structure and Frobenius-constant quantization} we review the notions of restricted symplectic variety in characteristic $p$ and its Frobenius-constant quantization.

\item In Section \ref{ssec:canonical categorical quantization in char p} we introduce the notion of
Frobenius-constant categorical quantization. We also start to present the results of our paper: in particular, we review our construction of the canonical Frobenius-constant categorical quantization $\QCoh_h$, as well as a full classification of Frobenius-constant categorical quantizations. We also identify the Rozansky-Witten invariant associated to the $\Theta$-graph as the 1-st order obstruction for the canonical categorical quantization to come from an algebra quantization.

\item Finally, in Section \ref{ssec:categorical twistor quantization} we present the main construction of our paper, namely the twistor categorical quantization $\QCoh_{\mathbb P^\mathsf{tw}}$: this is a quasi-coherent sheaf of categories over $X^{(1)}\times \mathbb P^1$ with the fiber over $X^{(1)}\times \{0\}$ given by $\QCoh(X)$, whose restriction to the formal disk $X^{(1)}\times \Spf k[[h]]$ is identified with the canonical categorical quantization $\QCoh_h$ and whose fiber over $X^{(1)}\times \{\infty\}$ is $\QCoh(X^{(1)})$. We describe its general properties, discuss some examples and the relation to other notions, like mod $p$ $F$-gauges. We also explain that a smooth restricted Lagrangian subvariety $Y\subset X$ gives rise to a canonical global object of $\QCoh_{\mathbb P^\mathsf{tw}}$. 
\end{itemize}

A reader interested in learning specifically about the new results in this paper is referred directly to Sections \ref{ssec:canonical categorical quantization in char p} and \ref{ssec:categorical twistor quantization}.

\subsection{Quantization of algebras}\label{intro:Quantization of algebras}
Let $k$ be a commutative ring. Originally motivated by physics, mathematically the concept of (formal) quantization comes naturally when trying to solve the following basic problem in abstract algebra.

\begin{problem}\label{problem:deformation}Let $B_0$ be a commutative algebra over $k$. Describe \textbf{formal deformations} of $B_0$; namely, $k[[h]]$-algebras $B_h$, endowed with an isomorphism $\alpha\colon B_h/hB_h\xra{\sim} B_0$, and satisfying the following two conditions:
\begin{itemize}
    \item $B_h$ is $h$-adically complete: $B_h\xra{\sim}\lim_{m\ge 0} B_h/h^{m+1}B_h$;
    \item $B_{m}:=B_h/h^{m+1}B_h$ is a flat $k[h]/h^{m+1}$-algebra for any $m\ge 0$.
\end{itemize}
\end{problem}
Note that by $h$-completeness, $B_h$ is uniquely recovered from the data of flat $k[h]/h^{m+1}$-algebras $B_{m}$ together with isomorphisms $\alpha_m\colon B_{m}\xra{\sim}B_{m+1}/h^{m+1}$ for all $m\ge 0$.
 We emphasize that $B_h$ itself is not assumed to be commutative, and in fact further we will be interested in the deformations that are maximally non-commutative in a certain sense. Any formal deformation $B_h$ endows $B_0$ with a $k$-linear bracket; namely, the commutator map $$[-,-]\colon \wedge^2_{k[[h]]}B_h \ra B_h$$ is divisible by $h$ (since $B_0$ is commutative), and reducing modulo $h$ the divided commutator $\frac{1}{h}[-,-]$ gives a bracket 
 $$
\{-,-\}_{B_h}\colon \wedge_k^2 B_0 \ra B_0.
 $$
 This bracket is a derivation in each variable and satisfies the Jacobi identity (which is enforced by associativity of $B_h$); in other words, $\{-,-\}_{B_h}$ defines a \textit{Poisson structure} on $B_0$. It thus makes sense to refine Problem \ref{problem:deformation}, and fix a Poisson structure on $B_0$ as a part of the initial data for deformation. 
 
 \begin{df}\label{def:quantization}
     Let $(B_0,\{-,-\})$ be a Poisson $k$-algebra. A \textit{quantization} of $(B_0,\{-,-\})$ is a formal deformation $B_h$ of $B_0$ such that $\{-,-\}_{B_h}=\{-,-\}$.
 \end{df}
Given a Poisson $k$-algebra $(B_0,\{-,-\})$ one can now ask to classify all its quantizations. Note that quantizations naturally form a groupoid rather than a set: morphisms between two quantizations $B_h$ and $B'_h$ are given by isomorphisms $\varphi \colon B_h\xrightarrow{\sim} B'_h$ that intertwine the identifications $\alpha\colon B_h/hB_h \xra{\sim} B_0$ and $\alpha'\colon B'_h/hB'_h \xra{\sim} B_0$. The association to $(B_0,\{-,-\})$ the groupoid $\mathsf{Quant}(B_0,\{-,-\})$ of its quantizations is functorial under Zariski localization\footnote{It is also known that it is functorial with respect to \'etale maps due to vanishing of relative Hochschild cohomology; however, to our knowledge there is no explicit way to write down the corresponding quantization.}. Namely, given a quantization $B_h$ of $(B_0,\{-,-\})$ and an element $b\in B_0$, one can take any lift $\tilde b\in B_h$: the limit of localizations
$\lim_{m\ge 0}(B_h/h^m)[\widetilde b^{-1}]$ does not depend on $\widetilde b$ and defines a quantization of 
 $(B_0[b^{-1}],f^*\{-,-\})$ (also see Section \ref{app:categorical sheaf property} for a more canonical construction). The association $(B_0,\{-,-\})\mapsto \mathsf{Quant}(B_0,\{-,-\})$ forms a groupoid-valued \emph{sheaf} in Zariski topology on the category of \textit{affine} Poisson $k$-schemes (with open embeddings as maps).
 Any such sheaf extends uniquely to the category of \textit{all} Poisson $k$-schemes and, given a general $(X,\{-,-\})$, one can define the groupoid $\mathsf{Quant}(X,\{-,-\})$ of quantizations of $(X,\{-,-\})$ as
the corresponding sections. 
 
 Explicitly, an object $\mc O_{X,h}\in \mathsf{Quant}(X,\{-,-\})$ can be described as a collection of sheaves $\{\mc{O}_{X,m}\}_{m\ge 0}$ of $k[h]/h^{m+1}$-algebras on $X$, endowed with isomorphisms $\mc{O}_{X,m}\xra{\sim} \mc O_{X,m+1}/h^{m+1}$ for all $m\ge 0$, and a further identification  $\mc O_{X,0}\xra{\sim}\mc O_X$. Moreover, on any affine open $j\colon U=\Spec B_0\subset X$, the restriction $\{j^*\mc{O}_{X,m}\}_{m\ge 0}$ should be identified with the collection of sheaves defined by $\{B_h/h^{m+1}\}_{m\ge 0}$, for some quantization $B_h$ of $(B_0,j^*\{-,-\})$, with the corresponding natural identifications.
 
 Let $X$ be a smooth scheme over $k$ and denote by $T_X$ the tangent bundle. A Poisson structure $\{-,-\}$ on $\mc O_X$ is called \emph{non-degenerate} if the map $\Omega^1_X=T_X ^\vee \to T_X$,  defined by the corresponding bivector $\mu\in H^0(X,\wedge^2T_X)$, is an isomorphism. In this situation, we also have the dual non-degenerate 2-form $\omega\coloneqq \mu^{-1}\in H^0(X,\Omega^2_X)$. The Jacobi identity for $\mu$ is equivalent to the equation $d\omega=0$, making $(X,\omega)$ a \textit{symplectic} scheme over $k$. In this case, by a quantization of $(X,\omega)$ we will implicitly mean a quantization of $(X,\mu=\omega^{-1})$.

\subsection{Quantizations of symplectic schemes: characteristic 0}\label{ssec:formal quantizations of schemes in char 0} When $\mbb Q\subset k$, the formality theorem of Kontsevich \cite{kont} (see also \cite{y} and \cite{vdb}) implies that every Poisson structure on a smooth algebra $B_0$ can be quantized. It also gives some control over $\mathsf{Quant}(B_0,\{-,-\})$ in terms of the moduli space of Poisson deformations of $(B_0,\{-,-\})$. However, for a general Poisson structure, its Poisson deformations are hard to describe explicitly.

On the other hand, in the (affine) symplectic case, Bezrukavnikov and Kaledin gave a uniform description of the set $\pi_0(\mathsf{Quant}(X,\omega))$ of isomorphism classes of quantizations in \cite{bk2}, building on the works of Fedosov \cite{fedosov} and  De Wilde--Lecomte \cite{DeWildeLecomte1983}, \cite{DeWildeLecomte1985} in the real $C^\infty$-context. More precisely, they constructed a ``period map'' $\mathsf{Per}$ from $\pi_0(\mathsf{Quant}(X,\omega))$, and showed that it is a bijection for a large class of schemes, including the affine ones. 

Below we will formulate the existence (and then explain a construction) of $\mathsf{Per}$ in a more general setting than in the original \cite{bk2}. For this it will be convenient to introduce some notation. Consider the following complex of sheaves on $X$:
\begin{equation}\label{eq:slightly rescaled de Rham complex}
F:=[h\cdot \mc O_{X}[[h]] \xra{d} \Omega^1_{X/k}[[h]] \xra{d}\Omega^2_{X/k}[[h]] \xra{d}\ldots];
\end{equation}
$F$ is a subcomplex in the $h$-complete $k[[h]]$-linear de Rham complex $\Omega^\bullet_{X/k,\dR}[[h]]:=[\mc O_X[[h]] \xra{d} \Omega^1_{X/k}[[h]] \xra{d}\Omega^2_{X/k}[[h]] \xra{d}\ldots]$, and the cofiber of the embedding is given by $\mc O_X$.
Let $H^i(X,F)$ denote the hypercohomology of $X$ with coefficients in $F$. Note that if $H^{1}(X,\mc O_X)=H^2(X,\mc O_X)=0$ then $H^2(X,F)\simeq H^2_\dR(X/k)[[h]]$.
\begin{Th}[compare to {\cite[Theorem 1.8]{bk2}}]\label{thm:BK theorem}
   Assume $\mbb Q\subset k$ and let $(X,\omega)$ be a smooth symplectic $k$-scheme. There is a natural map 
\begin{equation}\label{eq: period map char 0}
  \begin{tikzcd}
\mathsf{Per}\colon  \pi_0(\mathsf{Quant}(X,\omega))= \{\textup{Quantizations of }(X,\omega)\}/_{\textup{iso}}
        \arrow[r] & H^2(X,F), \end{tikzcd}  
\end{equation}
   functorial with respect to open embeddings\footnote{We expect it to also be functorial with respect to \'etale maps, but this functoriality is not addressed in \cite{bk2}.} of symplectic $k$-schemes. 
   
   If $H^{1}(X,\mc O_X)=H^2(X,\mc O_X)=0$ (in particular, if $X$ is affine), the period map\footnote{Strictly speaking, the notation in \cite{bk2} is different: their period map $\mathsf{Per}_{BK}$ is defined as $[\omega]+h\cdot \mathsf{Per}$, and it is an isomorphism on its image, given by $[\omega]+hH^2_\dR(X/k)[[h]]$ under the above conditions.} is an isomorphism\footnote{In fact, Bezrukavnikov and Kaledin prove a more general result. Let $X$ be \textit{admissible} in the sense of \cite{bk2}: namely, assume that $H^i_\dR(X)\to H^i(X,\mc O_X)$ is a surjection for $i=1,2$. In this case $H^2(F)\subset H^2_\dR(X)[[h]]$ is identified with the elements whose constant term lies in the 1-st step $\mr{Fil}^1_\mr{H}(H^2_\dR(X))\subset H^2_\dR(X)$ of the Hodge filtration. Bezrukavnikov and Kaledin show that in this case $\mathsf{Per}\colon \pi_0(\mathsf{Quant}(X,\omega)) \to H^2(F)$ is an embedding and (non-canonically) identify the image with $\mr{Fil}^1_\mr{H}(H^2_\dR(X))[[h]]$. } and we get a bijection
$$
  \begin{tikzcd}
\mathsf{Per}\colon  \pi_0(\mathsf{Quant}(X,\omega))
        \arrow[r,"\sim"] & H^2_\dR(X/k)[[h]]. \end{tikzcd}  
$$
\end{Th}

One calls a quantization $B_h$ with $\mathsf{Per}(B_h)=0$ \textit{distinguished}. By Theorem \ref{thm:BK theorem}, if $X$ is affine, it is unique up to an isomorphism. 

\begin{example}\label{ex:examples of quantizations}\begin{enumerate}
    \item 

    Let $X:=T^*\mathbb A^n\simeq \Spec k[x_1,\ldots ,x_n,y_1,\ldots,y_n]$ with the canonical symplectic form $\omega_{\mathsf{can}}:=\sum_{i=1}^n dy_i\wedge dx_i$ that is dual to the Poisson bivector $\mu=\sum_{i=1}^n\partial/\partial{y_i}\wedge\partial/\partial{x_i}$. Since $H^2_\dR(T^*\mbb A^n)=0$, Theorem \ref{thm:BK theorem} states that in this case there is a unique quantization of $X$ up to isomorphism. A representative is given by the \textit{asymptotic Weyl algebra}:
    $$
W_{n,h}(k):=\left(k[h]\langle x_i,y_j\rangle_{1\le i,j\le n}/([x_i,x_j]=[y_i,y_j]=0, [y_i,x_j]=h\delta_{ij})\right)^\wedge_h,
    $$
    where $(-)^\wedge_h$ denotes $h$-adic completion.
\item Let $Y=\Spec C$ be a smooth affine $k$-scheme, and $X:=T^*Y$ endowed with the canonical symplectic form $\omega:=d\eta$. A natural choice of a quantization of $X$ is given by the algebra $\mc D_{Y,h}$ of asymptotic differential operators on $Y$: explicitly, $\mc D_{Y,h}$ is the $h$-adic completion of the $k[h]$-algebra generated by functions $f\in C$ and vector fields $v\in T_Y$, with the relations $[f_1,f_2]_{\mc D_{Y,h}}=0$, $[v,f]_{\mc D_{Y,h}}=h\cdot v(f)$ and $[v_1,v_2]_{\mc D_{Y,h}}=h\cdot [v_1,v_2]_{T_{Y}}$, $(f\cdot v)_{\mc D_{Y,h}}=fv$,  where $[v_1,v_2]_{T_{Y}}\in T_Y$ is the Lie bracket of vector fields. One has 
$$\mathsf{Per}(\mc D_{Y,h}) = -\tfrac{1}{2}\cdot c_1(K_Y)\in H^2_\dR(Y/k)\simeq H^2_\dR(X/k)\subset H^2_\dR(X/k)[[h]],$$
where $K_Y=\wedge^{\dim Y}_{\mc O_Y}\Omega^1_Y$ is the canonical line bundle on $Y$, and $c_1$ is the first Chern class. In particular, for a general $Y$, $\mc D_{Y,h}$ is \textit{not} the distinguished quantization of $T^*Y$; the latter is instead given by the algebra of $K_Y^{1/2}$-twisted asymptotic differential operators\footnote{
For any quantization $B_h$, denoting by $\mathsf{Per}(B_h)(h)$ the corresponding power series one has $\mathsf{Per}(B_{-h}^{\op})= -\mathsf{Per}(B_h)(-h)$ (note that $B_h^{\op}$ is not a quantization of $(\Spec B_0,\{-,-\})$, but of $(\Spec B_0,-\{-,-\})$). For $\mc D_{Y,h}$, the opposite algebra is given by $K_Y$-twisted differential operators, rather than itself, so this shows that the constant term in $\mathsf{Per}(\mc D_{Y,h})$ is $-\tfrac{1}{2}\cdot c_1(K_Y)$. The full computation of the period can be found in \cite[Proposition 4.4]{BPW2016}.}.

    \end{enumerate}
\end{example}
Let us now explain a construction of the period map $\mathsf{Per}$. The algebra $W_{n,h}$ is a model example of quantization. To be more precise, let ${W}_{n,h}^\wedge(R)$ be the completion of $W_{n,h}(R)$ with respect to the two-sided ideal given by $(h, x_i, y_j)_{1\leq i,j \leq n}$.  A key input of \cite{bk2} is the quantum Darboux lemma, which says that, given a ring $R$ and an $R$-linear quantization $B_h$ of the symplectic formal disk $(\mathbb D_{2n,R},\omega_{\mathsf{can}}):=(\Spf R[[x_i,y_j]]_{1\le i,j\le n},\sum_{i=1}^n dy_i\wedge dx_i)$, it is fppf-locally on $\Spec R$ isomorphic to $W_{n,h}^\wedge(R)$. 

Recall that to any $k$-scheme $X$ we can associate the corresponding relative de Rham stack $X_\dR$: it is a functor on commutative $k$-algebras defined by $X_\dR(R):=X(R^{\mr{red}})$, with $R\twoheadrightarrow R^{\mr{red}}$ being the quotient by the nil-radical of $R$. Denote by $G_0$ the group sheaf whose $R$-points are given by the group of symplectomorphisms of the formal disk $(\mathbb D_{2n,R},\omega_{\mathsf{can}})$.  A symplectic structure $\omega$ on a smooth variety $X$ produces a canonical $G_0$-torsor $\mc M_{X,\omega}$ over $X_\dR$: an $R$-point of $\mc M_{X,\omega}$ is given by the data of a point $x\colon \Spec R^{\mr{red}} \ra X$ and an isomorphism  $\alpha\colon ((X\times \Spec R)^\wedge_{\Gamma_x}, p_X^* \omega) \simeq (\mathbb D_{2n,R},\omega_{\mathsf{can}})$ of symplectic formal $R$-schemes, where $(X\times \Spec R)^\wedge_{\Gamma_x}$ is the formal neighborhood of  the graph of $x$ inside $X\times \Spec R$, and $p_X$ is the projection to $X$. The assignment $(x,\alpha)\mapsto x$ defines a map $\mc M_{X,\omega}\ra X_\dR$ and, by a formal version of the classical Darboux lemma, defines a $G_0$-torsor over $X_\dR$. Now let $G$ be the group sheaf whose $R$-points are given by continuous $R[[h]]$-linear automorphisms of $W_{n,h}^\wedge(R)$; reduction modulo $h$ defines a surjection of sheaves $G\twoheadrightarrow G_0$. A quantization $\mc O_{X,h}$ of $(X,\omega)$, via the quantum Darboux lemma, produces a $G$-torsor $\mc M_{X, \mc O_{X,h}}\ra X_\dR$ in an analogous way to the construction of $\mc M_{X,\omega}$. Moreover, the $G$-torsor $\mc M_{X,\mc O_{X,h}}\ra X_\dR$ defines a lift of the $G_0$-torsor $\mc M_{X,\omega}\ra X_\dR$ to a $G$-torsor via the map $G\twoheadrightarrow G_0$, and one can show that 
the groupoid of such liftings is equivalent to the groupoid of quantizations $\mathsf{Quant}(X,\omega)$.

The kernel of $G\twoheadrightarrow G_0$ is identified with the group sheaf $\mr{Inn}(W_{n,h}^\wedge)$ of inner automorphisms of $W_{n,h}^\wedge$: the quotient of units $W_{n,h}^{\wedge,\times}$ by its center, that is the positive loop group scheme $L^+\mathbb G_m$ (here, $L^+\mathbb G_m(R):=R[[h]]^\times$).
The next key observation of Bezrukavnikov and Kaledin \cite{bk2}, also going back to the works of Fedosov and De Wilde--Lecomte, is that $G$ lifts, along the map $W_{n,h}^{\wedge,\times} \twoheadrightarrow \mr{Inn}(W_{n,h}^\wedge)$, to an extension $1\to W_{n,h}^{\wedge,\times} \to \widehat G \to G_0 \to 1.$ Note that such $\widehat G$ also gives a (central) extension 
\begin{equation}\label{eq:central extension by loops}
1\ra L^+\mathbb G_m \ra \widehat G \ra G \ra 1.
\end{equation}
We now discuss the construction of $\widehat{G}$ and how it allows to define the map $\mathsf{Per}$ in Theorem \ref{thm:BK theorem}. Let us for simplicity only describe the corresponding extension of Lie algebras. The Lie algebra of $G$ is given by derivations $\Der(W_{n,h}^\wedge)$  of $W_{n,h}^\wedge$. Almost all derivations of $W_{n,h}^\wedge$ are actually inner; more precisely, the subspace $\tfrac{1}{h}\cdot W_{n,h}^\wedge\subset  W_{n,h}^\wedge[h^{-1}]$ is closed under commutator and the map of Lie algebras
$$
\ad\colon \tfrac{1}{h}\cdot W_{n,h}^\wedge \ra \Der(W_{n,h}^\wedge)
$$
sending $a$ to $[a,-]$ is surjective. Since the center of $W_{n,h}^\wedge$ is given by $k[[h]]$ we get a central extension 
$$
0\ra \tfrac{1}{h} \cdot k[[h]] \ra \tfrac{1}{h} \cdot W_{n,h}^\wedge \twoheadrightarrow \Der(W_{n,h}^\wedge) \ra 0,
$$
and, taking quotient by the subspace $\frac{1}{h}\cdot k\subset \tfrac{1}{h} k[[h]]$, also a central extension 
$$
0\ra k[[h]] \ra \tfrac{1}{h} \cdot W_{n,h}^\wedge/(\tfrac{1}{h}\cdot k) \twoheadrightarrow \Der(W_{n,h}^\wedge) \ra 0.
$$
In \cite[Sections 3.2, 3.3]{bk2} the authors show (in a slightly different language of Harish-Chandra pairs) that this central extension of Lie algebras exponentiates to a central extension \eqref{eq:central extension by loops}
of the corresponding group sheaves. The kernel of the projection $\widehat G \twoheadrightarrow G_0$ can be naturally identified with $W_{n,h}^{\wedge,\times}$. 

Recall that, by the preceding discussion, the set $\mathsf{Quant}(X,\omega)$ of isomorphism classes of quantizations of $(X,\omega)$ is identified with the groupoid of liftings of the $G_0$-torsor of formal Darboux frames $[\mc M_{X,\omega}]$ on $X_\dR$ to a $G$-torsor. In particular, we have a map $\pi_0(\mathsf{Quant}(X,\omega))\to H^1_\dR(X,G)$, and the extension \eqref{eq:central extension by loops} produces a boundary map $H^1(X_\dR, G) \ra H^2(X_\dR, L^+\mathbb G_m)$. The BK-period map $\mathsf{Per}$ is a slight refinement of the composite map: namely, (see Remark \ref{rem:reinterpreting period map}) the composition $\delta\colon \pi_0(\mathsf{Quant}(X,\omega))\to H^1(X_\dR, G) \to H^2(X_\dR, L^+\mathbb G_m)$ factors as 
\begin{equation*}
\begin{tikzcd}
  \pi_0(\mathsf{Quant}(X,\omega))  \arrow[dr,"\delta"']\arrow[rr, dashed, "\mathsf{Per}"] && H^2\!\left(\Fib(R\Gamma(X_\dR, L^+\mathbb G_m) \to R\Gamma(X, \mbb G_m))\right)\arrow[dl,"H^2(p)"] \\ 
  &H^2(X_\dR, L^+\mathbb G_m), &
\end{tikzcd} 
\end{equation*} where the map $R\Gamma(X_\dR, L^+\mathbb G_m) \to R\Gamma(X, \mbb G_m)$ inside $\Fib(-)$ is induced by pull-back along $X\to X_\dR$ and reduction $L^+\mathbb G_m \to \mathbb G_m$, and $p$ is the projection from the fiber $\Fib\!\left(R\Gamma(X_\dR, L^+\mathbb G_m) \to R\Gamma(X, \mbb G_m)\right)$ to $R\Gamma(X_\dR, L^+\mathbb G_m)$\footnote{This factorization is a formal consequence of the following property of $\widehat{G}$: its quotient by the subgroup of $W_{n,h}^{\wedge,\times}$ of units congruent to $1$ modulo $h$ splits as   $G_0 \ltimes (W_{n,h}/(h))^\times$.
}. This defines the period map 
\begin{equation}\label{eq:period map char 0}
\mathsf{Per} \colon \pi_0(\mathsf{Quant}(X,\omega)) \to H^2\!\left(\Fib(R\Gamma(X_\dR, L^+\mathbb G_m) \to R\Gamma(X, \mbb G_m))\right).
\end{equation}

Note that $R\Gamma(X_\dR, L^+\mathbb G_m)$ is computed by the global sections of the $d\log$-version of de Rham complex $C:=[(\mc O_X[[h]])^\times \xra{d\log} \Omega^1_X[[h]] \xra{\ d\ } \Omega^2_X[[h]] \xra{\ d\ }\ldots]$, and the map to $R\Gamma(X, \mbb G_m)$ is given by the composite projection $C\twoheadrightarrow (\mc O_X[[h]])^\times \twoheadrightarrow \mc O_X^\times$. This way $\Fib(R\Gamma(X_\dR, L^+\mathbb G_m) \to R\Gamma(X, \mbb G_m))$ is identified with the global sections of the complex 
\begin{equation}\label{eq:two formulas for fiber F}
F:=\left((1+h\mc O_X[[h]])^\times \xra{} \Omega^1_X[[h]] \xra{\ d\ } \Omega^2_X[[h]] \xra{\ d\ }\ldots \right) \simeq \left(h\cdot \mc O_X[[h]] \xra{d} \Omega^1_X[[h]] \xra{d}\Omega^2_X[[h]] \xra{d}\ldots\right)
\end{equation}
with the isomorphism given by applying logarithm to the $0$-th term. Note that this is exactly the complex \eqref{eq:slightly rescaled de Rham complex} that featured in Theorem \ref{thm:BK theorem}. It is then the map \eqref{eq:period map char 0} which gives the period map \eqref{eq: period map char 0} in \textit{loc.cit}. 

Note that Theorem \ref{thm:BK theorem} implies that \textit{Zariski locally} on $X$ there exists a canonical isomorphism class $[B_{h}]$ of the distinguished quantization (one for which $\mathsf{Per}([B_h])=0\in H^2_\dR(X)[[h]]$). However, since Theorem \ref{thm:BK theorem} only describes the \textit{set} of \textit{isomorphism classes} of quantizations, and not the whole \textit{groupoid} $\mathsf{Quant}(X,\omega)$, it doesn't allow to glue the locally defined distinguished quantizations $B_h$ into a well-defined canonical sheaf of algebras. In particular, for a general non-affine $X$ a distinguished quantization doesn't exist, namely the point $0\in H^2_\dR(X)[[h]]$ is not in the image of period map (see Remark \ref{rem:rozansky-witten}). It also doesn't allow to globalize the isomorphism in Theorem \ref{thm:BK theorem}: in general, the natural target of the period map \eqref{eq:period map char 0} is more complicated, and it doesn't need to be either an embedding or a surjection (Remark \ref{rem:relation between categorical and non-categorical quantizations}).

We also don't know the answer to the following.
\begin{question} Is there an example of $(X,\omega)$ in characteristic 0, for which $\pi_0(\mathsf{Quant}(X,\omega))=\varnothing$? Note that by \cite[Theorem 1.8]{bk2} such $X$ can't be admissible in the sense of \cite[Definition 1.1]{bk2}; in particular $X$ can't be proper or affine.
\end{question}

\subsection{Categorical quantizations: characteristic 0}\label{ssec:categorical quantizations in char 0} An insight, going back to Gerstenhaber--Schack \cite{GerstenhaberSchack1988}, Kontsevich \cite{Kontsevich2001}, Van der Bergh \cite{vdb}, Yekutieli \cite{ye} is that in the setting of algebraic geometry it is in fact more natural to deform not the algebra $B_0$, but rather the (abelian) category $B_0\lmod$. In particular, the problem of non-functoriality of the ``distinguished quantization'' mentioned in the previous section disappears once we go one categorical level higher. To explain this phenomenon we need to introduce some further notation and definitions.

For a commutative ring $R$ let $\mr{LinCat}_R^{\mr{Ab}}$ denote the $(2,1)$-category of Grothendieck abelian categories over $R$ (see Section \ref{Sheavesofcategories}). Given a ring homomorphism $R\to R'$ there is a natural functor $\mr{LinCat}_R^{\mr{Ab}}\to \mr{LinCat}_{R'}^{\mr{Ab}}$, $\mc C\mapsto \mc C\otimes_{R} R'$: in the case $\mc C=A\lmod$, for an $R$-algebra $A$, one has $A\lmod\otimes_R R'\simeq (A\otimes_R R')\lmod$.  Set $\mr{LinCat}_{\Spf k[[h]]}^{\mr{Ab}}:= \lim_{m\ge 0} \mr{LinCat}_{k[h]/h^{m+1}}^{\mr{Ab}}$; this is a direct analogue of the category of $h$-complete $k[[h]]$-modules in the setting of abelian categories. An object of $\mr{LinCat}_{\Spf k[[h]]}^{\mr{Ab}}$ is given by a collection of Grothendieck abelian categories 
$\mc C_m$ over $k[h]/h^{m+1}$, $m\ge 0$, together with $k[h]/h^{m+1}$-linear equivalences 
$\mc C_{m}\simeq \mc C_{m+1}\otimes_{k[h]/h^{m+2}} k[h]/h^{m+1}$.

\begin{constr} Let $B_h$ be a quantization of a Poisson $k$-algebra $(B_0,\{,\})$ and set $B_m:=B_h/h^{m+1}$. Given a principal open subset $U=\Spec B_0[b^{-1}]\subset X$, let $S_b$ be the  multiplicative set of elements $x$ in $B_m$ such that $x\equiv b^r \mod h$ for some $r\ge 0$.  The assignment $U\mapsto B_m[S_b^{-1}]$
    defines a Zariski sheaf of flat $k[h]/h^{m+1}$-algebras on the site $\mr{Prin}(X)$ of principal open subsets of\footnote{Since principal opens form a basis for Zariski topology, this sheaf canonically extends to the entire Zariski site of $X$.} $X$. Moreover, one can show that the analogous statement one category level higher remains true: namely, the assignment $U\mapsto B_m[S_b^{-1}]\lmod$ defines a Zariski sheaf on $\mr{Prin}(X)$ with values in $\mr{LinCat}_{k[h]/h^{m+1}}^{\mr{Ab}}$ (see Proposition \ref{prop:Zariski descent for modules over deformation}). Let us denote this sheaf by $\ul{B_m\lmod}$.
Letting $m$ vary, we also get a $\mr{LinCat}_{\Spf k[[h]]}^{\mr{Ab}}$-valued Zariski sheaf which we denote $\ul{B_h\lmod}$.
\end{constr}
\begin{df}\label{def:categorical quantization char 0}
    Let $(X,\{,\})$ be a Poisson scheme over $k$. 
A \textit{categorical quantization} of $(X,\{,\})$ is a datum of $\mr{LinCat}_{\Spf k[[h]]}^{\mr{Ab}}$-valued Zariski sheaf $\mathscr C=\{\mathscr C_m\}_{m\ge 0}$ on $X$, together with a $k$-linear equivalence $\alpha\colon \mathscr C_0 \xra{\sim} \mc O_X\lmod$. It should satisfy the following local condition:
    \begin{itemize}
        \item[($\vcenter{\hbox{\large $*$}}$)] For any affine open $j\colon \Spec B_0 \subset X$, the restriction $j^*\mathscr C$ is equivalent\footnote{Here, we also want the equivalence $j^*(\alpha)$ to be compatible with the isomorphism $B_h/h\simeq B_0$.} to $\ul{B_h\lmod}$ for some quantization $B_{h}$ of $(B_0,j^*\{,\})$.
    \end{itemize}
\end{df}
Equivalently, $\mathscr C$ is given by the datum $\{\mathscr C_m\}_{m\ge 0}$ of sheaves of $k[h]/h^{m+1}$-linear abelian categories, together with the choice of $k[h]/h^{m+1}$-linear equivalences 
$\mathscr C_{m}\xra{\sim} \mathscr C_{m+1}\otimes_{k[h]/h^{m+2}} k[h]/h^{m+1}$, $\mathscr C_0\xra{\sim} \mc O_X\lmod$ and satisfying a version of condition $(\vcenter{\hbox{\large $*$}})$.

We denote by $\mathsf{Quant^{Cat}}(X,\{,\})$ the 2-groupoid of categorical quantizations of $(X,\{,\})$. Any algebra quantization $\mc O_{X,h}$ of $(X,\{,\})$ defines a categorical quantization $\mc O_{X,h}\lmod$ by taking the datum of $\{\mc O_{X,h}/h^{m+1}\lmod\}_{m\ge 0}$. This gives a functor $p\colon \mathsf{Quant}(X,\{,\}) \to \mathsf{Quant^{Cat}}(X,\{,\})$. By definition, a general categorical quantization $\mathscr C=\{\mathscr C_m\}_{m\ge 0}$ comes from a quantization $\mc O_{X,h}$ on any affine chart in $X$. We note however that the groups of automorphisms of $\mc O_{X,h}$ in $\mathsf{Quant}(X,\{,\})$ and $\mc O_{X,h}\lmod$ in $\mathsf{Quant^{Cat}}(X,\{,\})$ are quite different; in particular, the latter is a group in groupoids rather than sets.

\begin{rem}
    Alternatively, one could replace condition ($\vcenter{\hbox{\large $*$}}$) in Definition \ref{def:categorical quantization char 0} by a perhaps more natural assumption that the sheaf $\mathscr C_m$ of $k[h]/h^{m+1}$-linear abelian categories is locally \textit{flat} over $k[h]/h^{m+1}$ in the sense of Lowen--Van der Bergh \cite{LowenVanDenBergh2006} (and is locally compatible with the Poisson bracket). One can show that this ultimately leads to the same notion because a flat deformation of $B_0\lmod$ is given by $B_m\lmod$ for some flat deformation $B_m$ of $B_0$.
\end{rem}
Given a symplectic scheme $(X,\omega)$, one can try to use the same strategy as in \cite{bk2} to classify categorical quantizations of $(X,\omega)$. Namely, given a categorical quantization $\mathscr C$, one can associate to it a torsor under the group stack $\Aut(W_{n,h}^\wedge\lmod)$ over $X_\dR$. The group stack $\Aut(W_{n,h}^\wedge\lmod)$ can be explicitly described as the following crossed module\footnote{Crossed modules serve as an explicit model of group-like monoids in 1-truncated spaces. One can think of crossed modules as a suitable notion of 2-term complexes of non-commutative groups. It is given by a (not necessarily injective) group homomorphism $d\colon H\ra G$ together with an action $a\colon G \ra \Aut(H)$ such that for $g\in G,h\in H$ one has $d(a(g)(h))= g\cdot d(h) \cdot g^{-1}$ and $a(d(h))(h')= h\cdot h'\cdot h^{-1}$. If one models group-like monoids in spaces by simplicial (non-commutative) groups then the simplicial group corresponding to $d\colon H\ra G$ looks as follows: $n$-simplices are $H^n\rtimes G$ and the face maps are given by $d_0(h_n,\ldots,h_1,g)=(h_n,\ldots,h_2,d(h_1)g)$, $d_i(h_n,\ldots,h_1,g)=(h_n,\ldots,h_{i+1}h_i, \ldots, g)$ for $0<i<n$ and $d_n(h_n,\ldots,h_1,g)=(h_{n-1},\ldots,g)$. These formulas endow the homotopy quotient $[G/H]$ with a group structure after passing to the geometric realization of the corresponding simplicial group.}  
\begin{equation}\label{eq:cross-module formula for automorphisms of Weyl algebra}
W_{n,h}^{\wedge,\times} \xrightarrow{x\mapsto \Ad(x)} \Aut_{k[[h]]}^{\mr{cont}}(W_{n,h}^\wedge)=:G,
\end{equation}
where $\Ad(x)$ denotes the conjugation automorphism $w\mapsto xwx^{-1}$, and we consider the tautological action of $G$ on $W_{n,h}^{\wedge,\times}$. The corresponding homotopy group sheaves are given by $\pi_0(\Aut(W_{n,h}^\wedge\lmod))\simeq \mr{Out}^{\mr{cont}}_{k[[h]]}(W_{n,h}^\wedge)$ and $\pi_1(\Aut(W_{n,h}^\wedge\lmod))\simeq Z(W_{n,h}^\wedge)^\times \simeq L^+\mathbb G_m$. Moreover, the map $G\ra G_0$ induces an isomorphism $\mr{Out}^{\mr{cont}}_{k[[h]]}(W_{n,h}^\wedge)\simeq G_0$; this gives a natural homomorphism 
\begin{equation}\label{eq:map from autoequivalences to G_0}
\Aut(W_{n,h}^\wedge\lmod) \twoheadrightarrow \pi_0(\Aut(W_{n,h}^\wedge\lmod))\simeq G_0.
\end{equation}
However, differently from $G=\mr{Aut}(W^\wedge_{n,h})\twoheadrightarrow G_0$, the surjection $\Aut(W_{n,h}^\wedge\lmod)\twoheadrightarrow G_0$ is in fact canonically split: there is a natural isomorphism 
\begin{equation}\label{eq:splitting of autoequivalences of W-mod}
    G_0 \times BL^+\mbb G_m  \xra{\sim} \Aut(W_{n,h}^\wedge\lmod).
\end{equation}

An isomorphism \eqref{eq:splitting of autoequivalences of W-mod} is uniquely defined by a splitting of the projection \eqref{eq:map from autoequivalences to G_0}.  In general, a map from a discrete group sheaf $H$ to a crossed module like \eqref{eq:cross-module formula for automorphisms of Weyl algebra} is described by the datum of a group sheaf extension\footnote{More invariantly, let $[K'/K]$ be the group object in spaces associated to a crossed module $K\ra K'$ and assume we have a homomorphism $H\ra [K'/K]$  from a discrete group. Then one obtains an extension $1\ra K\ra \widehat{H} \ra H\ra 1$ with a homomorphism $\widehat{H}\ra K'$ by taking the fiber product of $H$ over $[K'/K]$ with the natural projection $K'\ra [K'/K]$.} 
$$
1 \ra W_{n,h}^{\wedge,\times} \ra \widehat H \ra H \ra 1
$$
together with a homomorphism $\widehat H \ra G$, compatible with the action on/of $W_{n,h}^{\wedge,\times}$. Recall Fedosov's central extension \eqref{eq:central extension by loops}. By construction, it fits into a short exact sequence of the form 
\begin{equation}\label{eq:another extension formula for G-hat}
1\ra W_{n,h}^{\wedge,\times} \ra \widehat{G} \ra G_0 \ra 1
\end{equation}
and it comes with a natural map $\widehat{G} \ra G$, which one checks to be compatible with the action on/of $W_{n,h}^{\wedge,\times}$. The extension \eqref{eq:another extension formula for G-hat} exactly produces the splitting in \eqref{eq:splitting of autoequivalences of W-mod}.

 Returning to categorical quantizations, one can show that the 2-groupoid $\mathsf{Quant^{Cat}}(X,\omega)$ is equivalent to the 2-groupoid of liftings of the $G_0$-torsor $\mc M_{X,\omega} \ra X_\dR$ to an 
$\Aut(W_{n,h}^\wedge\lmod)$-torsor via the map \eqref{eq:map from autoequivalences to G_0} together with a small piece of extra structure (that corresponds to the trivialization $\alpha$ in Definition \ref{def:categorical quantization char 0}). 
Given the isomorphism \eqref{eq:splitting of autoequivalences of W-mod}, the space of liftings of the $G_0$-torsor $\mc M_{X,\omega} \ra X_\dR$ to an $\Aut(W_{n,h}^\wedge\lmod)$-torsor is identified with the space of $BL^+\mbb G_m$-torsors over $X_\dR$ or, in other words, the 2-groupoid of $L^+\mathbb G_m$-gerbes on $X_\dR$. This 2-groupoid is given by $\mr{Map}(X_\dR,B^2L^+\mathbb G_m)\simeq (\tau^{\le 2}R\Gamma_{\textup{\'et}}(X,C))[2]$, where $
C=(\mc O_X[[h]])^\times \xra{d\log} \Omega^1_X[[h]] \xra{\ d\ } \Omega^2_X[[h]] \xra{\ d\ }\ldots,$ and $\tau^{\le 2}R\Gamma_{\textup{\'et}}(X,C)[2]$ is interpreted as a 2-truncated space via Dold-Kan correspondence. There is a natural map from $\tau^{\le 2}R\Gamma_{\textup{\'et}}(X,C)[2]$ to the 2-groupoid of $\mathbb G_m$-gerbes over $X$; in terms of complex $C$ it is induced by the natural projection $p\colon C \ra (\mc O_X[[h]])^\times \twoheadrightarrow \mc O_X^\times$ after applying $\tau^{\le 2}R\Gamma_{\textup{\'et}}(X,-)[2]$. The 2-groupoid $\mathsf{Quant^{Cat}}(X,\omega)$ of categorical quantizations is then described by $\tau^{\le 2}R\Gamma_{\textup{\'et}}(X,F)[2]$ where $F$ is the fiber of $p\colon C\ra \mc O_X^\times$. Here, as in \eqref{eq:two formulas for fiber F}, we can identify $F$ with the subcomplex $h\cdot \mc O_X[[h]] \xra{d} \Omega^1_X[[h]] \xra{d}\Omega^2_X[[h]] \xra{d}\ldots$ of $\Omega^{\bullet}_{X,\dR}[[h]]$. Note that additively the complex \eqref{eq:two formulas for fiber F} splits as 
$
 \Omega^{\ge 1}_{X,\dR} \oplus h\cdot \Omega^{\bullet}_{X,\dR}[[h]].
$

We arrive at the following result:
\begin{Th}[Folklore]\label{thm:categorical quantizations in char 0}
   Assume $\mbb Q\subset k$. Let $(X,\omega)$ be a symplectic scheme over $k$. Then there is a natural equivalence of 2-groupoids
\begin{equation}\label{eq:categorical period map char 0}
\mathsf{Per}^\mathsf{Cat}\colon  \mathsf{Quant^{Cat}}(X,\omega) \simeq \tau^{\le 2}\!\left(\Fib\left(R\Gamma(X_\dR,L^+\mathbb G_m)\to R\Gamma(X,\mc O_X^\times)\right)\right)[2],     
\end{equation}
functorial in open embeddings of symplectic schemes. 
One can also identify $\Fib(R\Gamma(X_\dR,L^+\mathbb G_m)\to R\Gamma(X,\mc O_X^\times))$ with $R\Gamma_\dR(X)'[[h]]:= \mr{Fil}^{1}_\mr{H}(R\Gamma_\dR(X))\oplus h\cdot R\Gamma_\dR(X)[[h]]$, where $\mr{Fil}^{1}_\mr{H}(R\Gamma_\dR(X))$ is the first step of Hodge filtration on de Rham cohomology complex.
\end{Th}
Theorem \ref{thm:categorical quantizations in char 0} can be considered as a higher global version of Theorem \ref{thm:BK theorem} (see Remark \ref{rem:relation between categorical and non-categorical quantizations}). Note the central role of Fedosov's extension \eqref{eq:central extension by loops} in the proof of both theorems: it was used to construct the period map in Section \ref{ssec:formal quantizations of schemes in char 0}, but ultimately produces a full description of the 2-groupoid of categorical quantizations by splitting the group stack of autoequivalences of $W_{n,h}^\wedge\lmod$.

Note that the 2-groupoid on the right-hand side has a canonical object given by $0\in H^2_\dR(X)'[[h]]\simeq\pi_0(\tau^{\le 2}R\Gamma_\dR(X)'[[h]][2])$; let $\QCoh_h(X,\omega)\in \mathsf{Quant^{Cat}}(X,\omega)$ denote the corresponding object on the left. 
\begin{cor}\label{cor:canonical categorical quantization} There is a canonical categorical quantization 
$$
(X,\omega) \mapsto \QCoh_h(X,\omega),
$$
which locally agrees with modules over the distinguished quantization of Bezrukavnikov-Kaledin.
\end{cor}

One of the main goals of this paper is to generalize Theorem \ref{thm:categorical quantizations in char 0} and Corollary \ref{cor:canonical categorical quantization} to characteristic $p$. Before passing to this context, let us finish this section with some complementary remarks.

\begin{rem}\label{rem:reinterpreting period map}
    Note that by Theorem \ref{thm:categorical quantizations in char 0} the set of isomorphism classes $\pi_0(\mathsf{Quant^{Cat}}(X,\omega))$ is identified with $H^2(\Fib(R\Gamma(X_\dR,L^+\mathbb G_m)\to R\Gamma(X,\mc O_X^\times)))$. The period map \eqref{eq:period map char 0} considered in Section \ref{ssec:formal quantizations of schemes in char 0} has a simple interpretation as the map 
    $$
\pi_0(p)\colon \pi_0(\mathsf{Quant}(X,\omega)) \ra \pi_0(\mathsf{Quant^{Cat}}(X,\omega))\simeq H^2(\Fib(R\Gamma(X_\dR,L^+\mathbb G_m)\to R\Gamma(X,\mc O_X^\times))),
    $$
    sending a quantization $\mc O_{X,h}$ to $\{\mc O_{X,h}/h^{n+1}\lmod\}_{n\ge 0}$.
\end{rem}
\begin{rem} \label{rem:relation between categorical and non-categorical quantizations}
    Following Remark \ref{rem:reinterpreting period map}, we can analyze the fibers of the period map $\mathsf{Per}$ in \eqref{eq:period map char 0} as the map 
    $$p\colon \pi_0(\mathsf{Quant}(X,\omega)) \ra \pi_0(\mathsf{Quant^{Cat}}(X,\omega)).$$ The extra datum of a lifting of a categorical quantization $\mathscr C=\{\mathscr C_m\}_{m\ge 0}$ to a quantization $\mc O_{X,h}$ is given by a compatible system of $k[h]/h^{m+1}$-flat objects $\{X_m\in \mathscr C_m(X)\}_{n\ge 1}$ that are lifts of the structure sheaf $\mc O_X\in \mathscr C_0(X)\simeq \QCoh(X)$. Here, given a quantization $\mc O_{X,h}$, one just takes the system $\{\mc O_{X,h}/h^{m+1} \in O_h/h^{m+1}\lmod\}_{m\ge 0}$. Conversely, a compatible system of objects $\{X_m\in \mathscr C_m(X)\}_{m\ge 0}$ produces a compatible system of sheaves of algebras $\mc A_m:=\ul{\mr{End}}_{\mathscr C_m}(X_m)^{\op}$; the limit $\lim_{n\ge 1}\mc A_n$ then defines a quantization of $(X,\omega)$. Similar to \cite[Theorem A]{Lowen2005}, the obstruction to lifting $X_{m-1}$ to $X_m$ lies in $H^2(X,\mc O_X)$, and the set of such liftings is a torsor over $H^1(X,\mc O_X)$. If $X$ has the cohomology vanishing $H^{1}(X,\mc O_X)=H^{2}(X,\mc O_X)=0$ (for example $X$ is affine) one recovers from this a bijection 
    $$
\pi_0(\mathsf{Quant}(X,\omega)) \xra{\sim} \pi_0(\mathsf{Quant^{Cat}}(X,\omega))\simeq H^2_\dR(X)[[h]],
    $$
in Theorem \ref{thm:BK theorem}.
\end{rem}

\begin{rem}\label{rem:rozansky-witten}
    In general, the map $\pi_0(p)$ is not surjective. In particular, Bressler--Gorokhovsky--Nest--Tsygan show that in the proper case\footnote{Their argument is complex-analytic. In Section \ref{subsection:rw} we also establish a version of their result in characteristic $p$.} (see \cite[Theorem 5.3.1 and Remark 5.3.4]{BGNT2007}) the Rozansky-Witten class $\mr{RW}_{\Theta}(X,\omega)\in H^2(X,\mc O_X)$ gives an obstruction for the class $[\QCoh_h(X,\omega)]\in \pi_0(\mathsf{Quant^{Cat}}(X,\omega))$ of the canonical categorical quantization to be in the image of the map $\pi_0(p)$ modulo $h^2$. The class $\mr{RW}_{\Theta}(X,\omega)$ is non-zero for example in the case of a K3-surface (see \cite[Proposition 19]{SawonRW}). 
\end{rem}

\subsection{Restricted symplectic structures and Frobenius-constant quantizations.}\label{ssec:restricted symplectic structure and Frobenius-constant quantization} Let $k$ be an $\mathbb F_p$-algebra. In this case we do not know the answer even to the following basic question:
\begin{question}\label{question:quantization char p}
    Let $(X,\omega)$ be a symplectic affine scheme over $k$. Is there a quantization of $(X,\omega)$ in the sense of Definition \ref{def:quantization}?
\end{question}
A problem in approaching Question \ref{question:quantization char p} is the failure of the formal Darboux lemma in characteristic $p$: for example, the form $d\log x \wedge d\log y$ on $\mathbb G_m^2$ is not exact in the formal neighborhood of any point, and thus can't be equivalent to $dx\wedge dy$.  Thus, a Fedosov-type argument from Sections \ref{ssec:formal quantizations of schemes in char 0} and \ref{ssec:categorical quantizations in char 0} can not work in general, because the sheaf $\mc M_{X,\omega}\ra X_\dR$ does not need to be a torsor.

In \cite{bk}, Bezrukavnikov and Kaledin proved that the answer to Question \ref{question:quantization char p} is positive at least if we assume that $\omega$ is exact. In fact they proved a stronger statement; to formulate it we need to recall the notions of \textit{Frobenius-constant quantization} and \textit{restricted Poisson structure}.

In characteristic $p$ any Poisson bracket on a smooth scheme $X$ is not only linear over the scalars $k$, but also the subsheaf\footnote{Here $X^{(1)}$ denotes the Frobenius twist: the fiber product $X\times_{\Spec k, F_k} \Spec k$ where $F_k\colon \Spec k \to \Spec k$ is the absolute Frobenius on $k$. The absolute Frobenius $F_X\colon X\ra X$ factors through $X^{(1)}$ and defines the relative Frobenius map $F_{X/k}\colon X\ra X^{(1)}$. When $X$ is smooth, $F_{X/k}$ corresponds to the embedding of sheaves of algebras $k\cdot \mc O_X^p\hookrightarrow \mc O_X$.} $\mc O_{X^{(1)}}\simeq k\cdot \mc O_X^p\subset \mc O_X$, since functions of the form $c\cdot f^p$ are killed by any derivation. This automatically makes the Poisson center $\mc O_{X^{(1)}}\subset Z_{\mr{Pois}}(\mc O_X)\subset \mc O_X$ huge, and in the symplectic case, where the center is the smallest, one has an isomorphism $\mc O_{X^{(1)}}\simeq Z_{\mr{Pois}}(\mc O_X)$.

\begin{df}
A \textit{Frobenius-constant quantization} \cite{bk} of a Poisson variety $(X,\{-,-\})$ is a quantization $\mc O_{X,h}$ together with an isomorphism $\theta\colon Z_{\mr{Pois}}(\mc O_X)[[h]]\xra{\sim} Z(\mc O_{X,h})\subset \mc O_{X,h}$ that satisfies the following two conditions:
\begin{enumerate}
    \item The induced map $\theta_1\colon Z_{\mr{Pois}}(\mc O_X) \to \mc O_{X,h}/h\simeq \mc O_X$ is the natural embedding $Z_{\mr{Pois}}(\mc O_X) \hookrightarrow \mc O_X$;
    \item For any $f\in \mc O_X$ and any lift $\tilde f$ under $\mc O_{X,h} \twoheadrightarrow \mc O_X$, the expression $\theta(f^p)$ is equal to $\tilde{f}^p$ modulo $h^{p-1}$.
\end{enumerate}
\end{df}
In particular, $\mc O_{X,h}$ is a formal deformation of $\mc O_X$ over the Poisson center $Z_{\mr{Pois}}(\mc O_X)$. The model example of a Frobenius-constant quantization is given by asymptotic differential operators.
 \begin{example}\label{ex:differential operators char p}
      Let $Y=\Spec C$ be a smooth affine $k$-scheme and let $\mc D_{Y,h}$ be the algebra of asymptotic differential operators, which, following Example \ref{ex:examples of quantizations}, is a quantization of $(X,\omega)=(T^*Y,d\eta)$. One has $X\simeq \Spec (\Sym_{C}^\star T_Y)$, where $T_Y$ is the $C$-module of global vector fields on $Y$. Note that $C^{(1)}:=k\cdot C^p\subset C$ lies in the center $\mc D_{Y,h}$: indeed, $p$-th power of any function on $Y$ is killed by any derivation. One can extend the embedding $C^{(1)}\hookrightarrow Z(\mc D_{Y,h})$ to an isomorphism
$$
\theta\colon \mc O(X^{(1)})[[h]]\simeq (k\cdot (\Sym_C^\star T_Y)^p)[[h]]\simeq (\Sym_{C^{(1)}}^\star T_{Y^{(1)}})[[h]] \xra{\sim} Z(\mc D_{Y,h});
$$
as follows. Module $T_{Y}$ comes with the structure of a restricted Lie algebra over $k$: given $v\in T_Y$, its $p$-th power (considered as an element in $\mr{End}_k(C)$) is again a derivation, and this gives a restricted $p$-power operation $v\mapsto v^{[p]}$ on $T_Y$. The association 
$$v\mapsto v^p-h^{p-1}v^{[p]} \in \mc D_{Y,h}$$ 
lands in the center and defines a map of $C^{(1)}$-modules $\theta\colon T_{Y}^{(1)}\simeq T_{Y^{(1)}}\to Z(\mc D_{Y,h})$. It extends uniquely to a continuous $C^{(1)}[[h]]$-algebra map
$$\theta\colon (\Sym_{C^{(1)}}^\star T_{Y^{(1)}})[[h]]\xra{} Z(\mc D_{Y,h}),$$
which turns out to be an isomorphism. Using that $\theta(v) = v^p \pmod {h^{p-1}}$ one can see that $\theta$ makes $\mc D_{Y,h}$ a Frobenius-constant quantization of $X$.
 \end{example}
 \begin{example}
     Let $X=\mathbb G_m^2\simeq \Spec k[x^{\pm 1},y^{\pm 1}]$ with the Poisson structure given by $\omega=d\log x \wedge d\log y$. Then $X$ does not admit a Frobenius constant quantization. (This follows from Proposition \ref{prop:restricted structures as 1-forms} below).
 \end{example}

An important observation of \cite{bk} is that, similarly to how a formal deformation $\mc O_{X,h}$ endows $\mc O_X$ with a Poisson bracket, the isomorphism $\theta\colon \mc O_{X^{(1)}}[[h]]\xra{\sim} Z(\mc O_{X,h})$ endows $\mc O_X$ with a further $p$-power operation $(-)^{[p]_{\theta}}\colon \mc O_X \to \mc O_X$. Namely, given a local section $f\in \mc O_X$, the section $f^{[p]_{\theta}}\in \mc O_X$ is uniquely determined by the equation 
$$
h^{p-1} f^{[p]_{\theta}}=\tilde f^p - \theta(f^p) \mod h^p\mc O_{X,h},
$$
where $\tilde f$ is any local lift of $f$ to $\mc O_{X,h}$. It turns out that $\mc O_X$, endowed with the Poisson bracket $\{-,-\}_{\mc O_{X,h}}$ and the $p$-power operation $(-)^{[p]_{\theta}}$ forms a sheaf of restricted Lie $k$-algebras over $X$. Moreover, given local sections $f,g\in \mc O_X$, there is a universal formula for  $(fg)^{[p]_\theta}$ (see \eqref{eq:restricted Poisson algebra}) in terms of $f,g,f^{[p]_\theta},g^{[p]_\theta}$ and the Poisson bracket; if $p>2$ it is uniquely reconstructed from the condition $(f^2)^{[p]_{\theta}}=2f^pf^{[p]_\theta}$ for all $f\in \mc O_X$. These properties  can be turned into a definition:

\begin{df}[\cite{bk}] Let $p>2$. A \textit{restricted Poisson structure} on a $k$-scheme $X$ is given by \begin{enumerate}
\item a $k$-linear Poisson bracket $\{-,-\}$ on $\mc O_X$;
\item a restricted $p$-th power operation $-^{[p]}\colon \mc O_X\ra \mc O_X$ that turns $(\mc O_X,\{-,-\})$ into a sheaf of restricted Lie algebras over $k$;
\item satisfying an additional condition that for any $f\in\mc O_X$ one has $(f^2)^{[p]}= 2f^pf^{[p]}$.
\end{enumerate}
Given a Frobenius-constant quantization $\mc O_{X,h}$ of $X$, one says that it is a quantization of the  restricted Poisson structure $(\mc O_X, \{-,-\}, (-)^{[p]})$ if $\{-,-\}=\{-,-\}_\theta$ and $(-)^{[p]}=(-)^{[p]_\theta}$. 
\end{df}

In the case when the Poisson structure is symplectic, one can describe the sets of its lifts to a restricted Poisson structure in more explicit terms. Consider the quotient sheaf $\Omega^1\slash d(\mc O_X)$; one has a well-defined map of sheaves $d\colon \Omega^1_X\slash d(\mc O_X) \to \Omega^2_X$. 

\begin{pr}[{\cite[Theorem 1.12]{bk}}]\label{prop:restricted structures as 1-forms} Let $p>2$. Let $(X,\omega)$ be a symplectic $k$-scheme. Then there is a natural bijection
\medskip
	\begin{equation*}\begin{Bmatrix} \text{\emph{global sections}} \\ [\eta] \text{\emph{ of }} \Omega^1_X/d(\mc O_X) \\
		\text{\emph{s.t. }} d[\eta] = \omega
	\end{Bmatrix}
 \xymatrix{\ar[rr]^{1\text{ \emph{to} }1}&&\ar[ll]}	\begin{Bmatrix}\text{\emph{restricted Poisson}} \\
 	\text{\emph{structures on }} $X$ \\
 	\text{\emph{compatible with }} \omega
 \end{Bmatrix} .
	\end{equation*}
\end{pr}
\medskip 
See Proposition \ref{prop:description of restricted symplectic structure} for an explicit construction of this bijection. Note that Proposition \ref{prop:restricted structures as 1-forms} shows that not every symplectic form $\omega$ can underlie a restricted Poisson structure: namely, it needs to be locally exact (so $\mathsf{C}(\omega)=0$). Given the description in Proposition \ref{prop:restricted structures as 1-forms}, we will typically denote a restricted symplectic $k$-scheme by $(X,[\eta])$. In Example \ref{ex:differential operators char p}, the restricted Poisson structure on $X=T^*Y$ which is quantized by $\mc D_{Y,h}$, corresponds to the class in  $H^0(X,\Omega^1\slash d(\mc O_X))$ defined by the canonical 1-form $\eta\in H^0(X,\Omega^1_X)$.

Let $\mathsf{Quant}(X,[\eta])$ denote the 1-groupoid of Frobenius-constant quantizations of $(X,[\eta])$. In \cite{bk}, Bezrukavnikov and Kaledin proved the following characteristic $p$ version of Theorem \ref{thm:BK theorem}.

\begin{Th}\label{thm:BK theorem char p} Assume $p>2$. Let $(X,[\eta])$ be a smooth restricted symplectic affine $k$-scheme. There is a canonical bijection of sets
$$
  \begin{tikzcd}
 \mathsf{Per}\colon \pi_0(\mathsf{Quant}(X,[\eta]))\simeq \{\textup{Quantizations of }(X,[\eta])\}/_{\mr{iso}}
        \arrow[r,"\sim"] & H^1_{\textup{\'et}}(X, \mc O_X^\times/\mc O_{X^{(1)}}^\times) \end{tikzcd}  
$$
functorial with respect to \'etale maps of symplectic $k$-schemes. 
\end{Th}
We remark that if $k$ is perfect, the group $H^1_{\textup{\'et}}(X, \mc O_X^\times/\mc O_{X^{(1)}}^\times)$ sits in the short exact sequence
$$
0\to \Pic(X)\slash p\Pic(X) \to H^1_{\textup{\'et}}(X, \mc O_X^\times/\mc O_{X^{(1)}}^\times) \to \Br(X)[p] \to 0.
$$
\begin{example}
    Similarly to Example \ref{ex:examples of quantizations}, in the case of $X=T^*Y$ over a perfect $k$, one has  $$\mathsf{Per}(\mc D_{Y,h})=-\tfrac{1}{2}\cdot[q^*K_Y]\in \Pic(X)\slash p\Pic(X)\subset H^1_{\textup{\'et}}(X, \mc O_X^\times/\mc O_{X^{(1)}}^\times),$$ 
where $q\colon T^*Y\to Y$ is the projection. We also note that, differently from characteristic 0, there are many non-trivial Frobenius-constant quantizations of $\mathbb A^{2n}=T^*\mathbb A^n$, since $\mr{Br}(\mathbb A^{2n})[p]$ is huge\footnote{For example, if $k$ is a field, then $k$ as an additive group embeds into $\mr{Br}(\mathbb A^{2n})[p]$ by \cite[Proposition~6.4]{Negron2017}.}.
\end{example}

 A key input in proving Theorem \ref{thm:BK theorem char p} is the \textit{restricted} Darboux lemma. To formulate it, it is convenient to extend the notion of a restricted symplectic scheme from smooth to \textit{pseudo-smooth} setting: an $S$-scheme $X$ is called pseudo-smooth\footnote{Pseudo-smooth schemes were originally considered in \cite{bk} under the name of \textit{quasi-regular}; we develop some bits of the theory in a more methodical way in Section \ref{sec:pseudo-smooth algebras and differential geometry}. Any pseudo-smooth scheme Zariski locally is isomorphic to the Frobenius neighborhood of a closed (but not necessarily smooth) subscheme of a smooth scheme (Proposition \ref{prop:local description of pseudo-smooth}).}  if it is locally of finite type and $\Omega^1_{X/S}$ is a locally free sheaf of finite rank. From the point of view of differential geometry in characteristic $p$ pseudo-smooth schemes behave similarly to smooth ones; in particular, one has the notion of a \textit{(restricted) symplectic} pseudo-smooth scheme, and the role of ``dimension''  is played by $\rk (\Omega^1_{X/S})$. However, for the Cartier isomorphism or the description of the Poisson center of a symplectic scheme in the pseudo-smooth setting, one needs to replace the relative Frobenius by the \textit{reduced Frobenius} $F_X^{\lp}\colon X\to X^{\lp}$: the sheaf of functions on $X^{\lp}$ is given by $\mc O_S\cdot \mc O_X^p\subset \mc O_X$, which for non-reduced $X$ can be different from $\mc O_X\otimes_{\mc O_S,F_S}\mc O_S$ (see Section \ref{sec:pseudo-smooth algebras and differential geometry} for more details). A key property of pseudo-smooth schemes is that $F_X^{\lp}$ is a locally trivial fibration for flat topology with the fiber given by the Frobenius neighborhood of the origin in an affine space. In the case $X$ is smooth, $X^{\lp}\simeq X^{(1)}$. Moreover, any smooth $X$ is pseudo-smooth over $X^{(1)}$: in particular, given a  restricted symplectic smooth $k$-scheme $(X,[\eta])$ one can instead consider it as a restricted symplectic scheme over $X^{(1)}$ in a pseudo-smooth sense. The restricted Darboux lemma states that flat locally over $X^{\lp}$ any restricted symplectic $S$-scheme $(X,[\eta])$ with $\rk(\Omega^1_{X/S})=2n$ is isomorphic to the Frobenius neighborhood of 0 in $T^*\mathbb A^n$, endowed with the canonical restricted structure. The latter is explicitly described as $\Spec A_0$ with $A_0=\mathbb F_p[x_i,y_j]/(x_i^p,y_j^p)_{i,j=1,\ldots, n}$, and with restricted structure defined by the 1-form $\eta_{\mathsf{can}}:=\sum_{i=1}^n{y_idx_i}\in \Omega^1_{A_0/\mathbb F_p}$. Note that $Y=\Spec A_0$ is pseudo-smooth over $\mathbb F_p$, and that $Y^{\lp}\simeq \Spec \mathbb F_p$. 

Let $G_0$ denote the group sheaf of restricted symplectomorphisms of $(\Spec A_0,[\eta_{\mathsf{can}}])$: its $R$-points for an $\mathbb F_p$-algebra $R$ are given by the group of $R$-algebra automorphisms of $A_0(R):=A_0\otimes R$ that preserve the class of $\eta_{\mathsf{can}}$ in $\Omega^1_{A_0(R)/R}/d(A_0(R))$. Differently from characteristic 0, $G_0$ is represented by a \textit{finite type} group scheme; its structure is recalled in Remark \ref{rem:groups H and G_0}. To any restricted symplectic $S$-scheme $X$ with $\rk(\Omega^1_{X/S})=2n$ one can associate a canonical $G_0$-torsor $\mc M_{X,[\eta]}\to X^{\lp}$: an $R$-point of $\mc M_{X,[\eta]}$ is given by the datum of an $R$-point of $X^{\lp}$, together with a restricted isomorphism $(\Spec R\times_{X^{\lp}}X, p_X^*[\eta]) \simeq \Spec R\times (\Spec A_0,[\eta_{\mathsf{can}}])$, and the fact that it is a torsor is precisely the restricted Darboux lemma. The $G_0$-torsor $\mc M_{X,[\eta]}\to X^{\lp}$ is called the torsor of \textit{restricted Frobenius-frames}; it is encoded by a map $[\eta]\colon X^{\lp} \to BG_0$ to the classifying stack of $G_0$.

The restricted symplectic $\mathbb F_p$-scheme $(\Spec A_0,[\eta_{\mathsf{can}}])$ has a natural Frobenius-constant quantization given by the $\mathbb F_p[[h]]$-algebra $A_h$ generated by variables $x_i,y_j$ for $i,j=1,\ldots,n$ and with relations $[y_i,x_j]=h\delta_{ij}$, $[x_i,x_j]=[y_i,y_j]=x_i^p=y_j^p=0$. The center of $A_h$ is canonically isomorphic to $\mathbb F_p[[h]]$. One can define a group scheme $G$, whose $R$-points are given by the group of $R[[h]]$-algebra automorphisms of $A_h(R):= A_h\otimes_{\mathbb F_p[[h]]}R[[h]]$; reduction modulo $h$ induces a surjection $G\twoheadrightarrow G_0$. Then, similarly to characteristic 0, the groupoid of Frobenius-constant quantizations of $(X,[\eta])$ is identified with the groupoid of liftings of the $G_0$-torsor $\mc M_{X,[\eta]}\to X^{\lp}$ to a $G$-torsor. Bezrukavnikov and Kaledin prove\footnote{ In fact, they prove it for so-called admissible (restricted symplectic) schemes: those, for which the natural map $H^i(X^{\lp}, \mc O_{X^{\lp}}) \to H^i(X, \mc O_{X})$ is an isomorphism for $i=1,2,3$.} Theorem \ref{thm:BK theorem char p} by studying the deformation problem of lifting the map $[\eta]\colon X^{\lp}\to BG_0$ to a map $X^{\lp}\to BG$. More precisely, they filter $G$ by a sequence of groups $G\to \ldots \to G_i \to \ldots \to G_1\to G_0$, and show that liftings of $[\eta]$ to $BG_1$ are classified by $H^1_{\textup{\'et}}(X, \mc O_X^\times/\mc O_{X^{\lp}}^\times)$, while under the admissibility assumption (e.g. if $X$ is affine) there is a unique further lift to $BG_i$ for $i\ge 2$, and, consequently, to $BG$. However, for a general symplectic scheme $X$ there are non-trivial obstructions for such a lift to exist. 

The goal of our paper is to prove (and then improve) a categorical global version of Theorem \ref{thm:BK theorem char p}: in particular, to construct a canonical  global categorical Frobenius-constant quantization of a restricted symplectic scheme in characteristic $p$.
\subsection{Canonical categorical formal quantization in characteristic $p$}\label{ssec:canonical categorical quantization in char p}

One can define a natural categorical version of Frobenius-constant quantization (Definition \ref{def:categorical Frobenius-constant quantization}): it is a quasi-coherent sheaf of abelian categories over $X^{\lp}\times\Spf k[[h]]$ (in the sense of Section \ref{Sheavesofcategories}) which deforms $\QCoh(X)$ and is Zariski-locally given by modules over a Frobenius-constant quantization of $\mc O_X$ as a sheaf of algebras. A Frobenius-constant quantization $(\mc O_{X,h},\theta)$ defines a sheaf of algebras over $X^{\lp}\times\Spf k[[h]]$ which then produces a quasi-coherent sheaf of abelian categories by passing to the corresponding category of left modules. In view of Theorem \ref{thm:categorical quantizations in char 0}, it is natural to ask if any restricted symplectic $k$-scheme has a canonical such categorical Frobenius-constant quantization in characteristic $p$.  The answer turns out to be yes; moreover, we also obtain a full description of the 2-groupoid $\mathsf{Quant^{Cat}}(X,[\eta])$ of Frobenius-constant quantizations. 

\begin{Th}[Proposition \ref{pr:classification_of_formal_quantizations}] \label{thm: categorical formal quantizations in char p} Let $p>2$ and $(X,[\eta])$ be a restricted symplectic $k$-scheme. Then there is a natural equivalence 
\begin{equation}\label{eq:period map in characteristic p}
	\mathsf{Per^{Cat}}\colon \mathsf{Quant^{Cat}}(X,[\eta]) \simeq \tau^{\le 2}R\Gamma_{\textup{\'et}}\! \left(X, \Fib\left(\cO_{X^\lp} [[h]]^\times \to \mc O_X^\times\right)\right)[2]
\end{equation}
	functorial in \'etale maps of schemes. The map $\cO_{X^\lp} [[h]]^\times \to \mc O_X^\times$ is induced by reduction modulo $h$ and embedding $\cO_{X^\lp}\subset\mc O_X$. 
\end{Th}	
Note that the descriptions of the 2-groupoids of categorical quantizations in Theorems \ref{thm: categorical formal quantizations in char p} and \ref{thm:categorical quantizations in char 0} look rather uniform: the only difference is that in the Frobenius-constant context the infinitesimal stack $X_\dR$ is replaced by the Frobenius twist $X^{\lp}$.

As a practical consequence, Theorem \ref{thm: categorical formal quantizations in char p} allows to extend the BK-period map for quantizations of $\mc O_X$ to general restricted symplectic schemes by looking at the map $\pi_0(\mathsf{Quant}(X,[\eta])) \to \pi_0(\mathsf{Quant^{Cat}}(X,[\eta]))$:

\begin{cor} There is a canonical period map 
\begin{equation}\label{intro:algebra_per_map}
  \mathsf{Per}\colon \pi_0(\mathsf{Quant}(X,[\eta])) \to H^2_{\textup{\'et}}\! \left(X, \Fib\left(\cO_{X^\lp} [[h]]^\times \to \mc O_X^\times\right)\right),  
\end{equation}
  extending the one in Theorem \ref{thm:BK theorem char p} for affine $X$.  
\end{cor}

\begin{example}
If $X$ is affine (or, more generally, if $H^1(X,\mc O_{X^\lp})=H^2(X,\mc O_{X^\lp})=0$) then the group $H^i_{\textup{\'et}}(X,(1+h\mc O_{X^\lp}\![[h]])^\times)$ is zero for $i=1,2$, and we get a simpler description of the set of isomorphism classes 
$$
\pi_0(\mathsf{Quant^{Cat}}(X,[\eta])) \simeq H^1_{\textup{\'et}}(X,\mc O_X^\times/\mc O_{X^\lp}^\times),
$$
matching the description in Theorem \ref{thm:BK theorem char p}.
\end{example}

Let us explain the proof of Theorem \ref{thm: categorical formal quantizations in char p}. The 2-groupoid on the right-hand side of \eqref{eq:period map in characteristic p} classifies $\mathbb G_m$-gerbes over $X^{\lp}\times \Spf k[[h]]$ together with the datum of a trivialization of the pullback to $X$ via the composite map $X\to X^{\lp}\hookrightarrow X^{\lp}\times \Spf k[[h]]$. It naturally acts on $\mathsf{Quant^{Cat}}(X,[\eta])$ via twisting a categorical quantization by the corresponding gerbe, and this way Theorem \ref{thm: categorical formal quantizations in char p} reduces to constructing a canonical categorical quantization (on which the action would be free). We do this in Section \ref{canonicalquant.theaction}, as a part of a more general construction of a certain quasi-coherent sheaf of abelian categories over $X^{\lp}\times \mathbb P^1$.

As in Section \ref{ssec:categorical quantizations in char 0}, this part of the construction reduces to finding a group extension 
\begin{equation}\label{eq:extension for formal disk in char p}
1\to A_h^\times \to \widehat{G} \to G_0 \to 1,
\end{equation}
together with a homomorphism $\widehat G \ra G:=\Aut(A_h)$ (here $R$-points of $\Aut(A_h)$ are given by the group of $R[[h]]$-linear automorphisms of $A_h(R)$), which is compatible with the action on/of $A_h^\times$. Such an extension produces an isomorphism of group stacks $$
\Aut(A_h\lmod) \xra{\sim} G_0\times BL^+\mathbb G_m,
$$
using which one can induce a sheaf of categories on $X^\lp$ from the $G_0$-torsor $\mc M_{X,[\eta]} \to X^{\lp}$. 
Our construction of extension \eqref{eq:extension for formal disk in char p} is quite different from the one of Fedosov, and is specific to characteristic $p$. 

More precisely, let $A_h[h^{-1}]$ denote the algebra sheaf on $k$-algebras whose $R$-points are given by $A_h(R)[h^{-1}]\simeq A_h\otimes_{\mbb F_p[[h]]}R((h))$. It is a classical computation (see Remark \ref{rem:A_S is matrix algebra outside of 0}) that the algebra $A_h[h^{-1}]$ is in fact isomorphic to $\mr{Mat}_{p^n}(\mbb F_p((h)))$, and so the sheaf $\Aut(A_h[h^{-1}])$, whose $R$-points are given by $R((h))$-linear automorphisms of $A_h(R)[h^{-1}]$, is naturally isomorphic to the loop group\footnote{Given an affine $k$-group scheme $H$ the corresponding loop group sheaf $LH$ is defined to be the {\it fppf} sheafification of the group presheaf   $R\mapsto H(R((h)))$.} $L\mr{PGL}(p^n)$. In Proposition \ref{pr:surjectivity} we prove that the sequence of group sheaves
\begin{equation}\label{eq:extension of loop groups}
    1\ra L\mathbb G_m \ra L\mr{GL}(p^n) \ra L\mr{PGL}(p^n) \ra 1
\end{equation}
is exact. There is a natural embedding $G:=\Aut(A_h)\rightarrow \Aut(A_h[h^{-1}])$, and one can pull back the extension \eqref{eq:extension of loop groups} along this map, obtaining an extension 
\begin{equation}\label{eq:tilde G}
1\to L\mathbb G_m \to \widetilde G \to G \to 1.
\end{equation} Moreover, by the result of Contou-Carr\`ere \cite{cc} the extension 
$L^+\bG_m \ra L\bG_m\ra L\bG_m/L^+\bG_m =: \Gra_{\bG_m}$ splits. This defines an extension 
\begin{equation}\label{eq:extension for formal neighborhood intro}
1\ra L^+\mathbb G_m \to \widehat G \to G \to 1.
\end{equation}
On the other hand, (see Remark \ref{rem:3+1-term exact sequence}) one also has a 4-term exact sequence
$$
1\to L^+\mathbb G_m \to A_h^\times \to G\to G_0\to 1,
$$
where the map $G\to G_0$ is given by reduction modulo $h$, $A_h^\times \to G$ sends $x\mapsto \Ad(x)$, and the map $L^+\mathbb G_m \to A_h^\times$ is induced by the embedding $\mbb F_p[[h]] \to A_h$. This allows to identify the kernel of the composite projection $\widehat G\to G\to G_0$ with $A_h^\times$, and this way $\widehat G$ fits into an exact sequence of the form \eqref{eq:extension for formal disk in char p}, which one also checks to be compatible with the action on/of $A_h^\times$.

Let $\QCoh_h$ denote the sheaf of categories on $X^{\lp}\times \Spf k[[h]]$ corresponding to this canonical Frobenius-constant categorical formal quantization (it implicitly depends on the restricted structure $(X,[\eta])$).

\begin{rem} The period map \eqref{intro:algebra_per_map} is not surjective in general. In Proposition \ref{pr:quant_of_algebra_mod_h3}
we show that the image of the composition 
\[ \pi_0(\mathsf{Quant}(X,[\eta])) \to H^2_{\textup{\'et}}\! \left(X, \Fib\left(\cO_{X^\lp} [[h]]^\times \to \mc O_X^\times\right)\right)\to  H^1_{\textup{\'et}}(X,\mc O_X^\times/\mc O_{X^\lp}^\times)\]
is contained in the ``quadric'' 
\[
\left\{
L \in H^1_{\textup{\'et}}\bigl(X,\mc O_X^\times/\mc O_{X^\lp}^\times\bigr)
\,\middle|\,
\iota_{\omega^{-1}}\bigl(c_1(L)\cup c_1(L)\bigr)
=
\tfrac{-1}{4}\overline{\mathrm{RW}}_\Theta(X,\omega)
\in H^2(X,\cO_X/\mc O_{X^\lp})
\right\},
\]
where $\overline{\mathrm{RW}}_\Theta (X, \omega)$  is the image of the Rozansky-Witten invariant\footnote{When $p>3$, there is a simple explicit formula for this Rozansky-Witten class: namely, $\mathrm{RW}_\Theta (X, \omega) = \frac{1}{3} \iota_{\omega^{-1}} c_2(T_X),$ where $c_2(T_X)\in H^2(X,\Omega^2_X)$ is the Chern class of $T_X$ in Hodge cohomology; see Remark \ref{rem:relation_to_c}.} $\mathrm{RW}_\Theta (X, \omega)\in H^2(X,\cO_X)$ of the symplectic variety $(X, \omega)$, corresponding to the simplest trivalent graph $\Theta$ (if $p=3$ we use a slightly modified  version\footnote{The construction of ${\mathrm{RW}}_\Theta (X, [\eta])$ in characteristic 3 is original, and could be of independent interest. In particular, it requires a somewhat subtle representation-theoretic analysis to define.} ${\mathrm{RW}}_\Theta (X, [\eta])$ of the Rozansky-Witten invariant, which really depends on the restricted structure $[\eta]$) in 
$H^2(X,\cO_X/\mc O_{X^\lp})$. For a superspecial K3-surface the class $\overline{\mathrm{RW}}_\Theta (X, \omega)$ for $p>3$ (resp. ${\mathrm{RW}}_\Theta (X, [\eta])$ for $p=3$) is non-zero (Example \ref{ex:nonsurjectivity}), and in this case the canonical categorical quantization $\QCoh_h$ doesn't come from a Frobenius-constant quantization of the algebra.
\end{rem}
\begin{rem} As was mentioned above, one can think of Theorem \ref{thm: categorical formal quantizations in char p} as the statement that the natural action of the 2-groupoid of $\mathbb G_m$-gerbes on $X^\lp
\times \Spf k[[h]]$, together with the trivialization of their pull-back along the map $X \twoheadrightarrow X^{\lp} \hookrightarrow X^\lp
\times \Spf k[[h]]$, on $\mathsf{Quant^{Cat}}(X,[\eta])$ is free and transitive.
    There is a similar way to think about its characteristic 0 analogue Theorem \ref{thm:categorical quantizations in char 0}. Note that there, the right-hand side of the equivalence \eqref{eq:categorical period map char 0} is identified with the 2-groupoid of $\mathbb G_m$-gerbes on $X_\dR\times \Spf k[[h]]$ together with the trivialization of their pull-back along the map $X\twoheadrightarrow X_\dR \hookrightarrow X_\dR\times \Spf k[[h]]$. One can show that there is an action of this groupoid on $\mathsf{Quant^{Cat}}(X,\{,\})$ for any Poisson structure $\{,\}$, and reformulate Theorem \ref{thm:categorical quantizations in char 0} as the statement that in the symplectic case the action is free and transitive. However, in this setting the construction of the action is more involved: the problem is that a categorical quantization $\mathscr C$ does not literally descend to a quasi-coherent sheaf of \textit{abelian} categories over $X_\dR\times \Spf k[[h]]$ (at least in the sense of Section \ref{Sheavesofcategories}), the essence of the problem being that the map $X\to X_\dR$ is only representable in formal schemes rather than schemes. Nevertheless, if we replace $\mathscr C$ by the corresponding sheaf of derived categories $\widehat{D}(\mathscr C)$ (as in Section \ref{ssec:qcoherent sheaves of stable categories}) it turns out to descend to $X_\dR\times \Spf k[[h]]$ as a quasicoherent sheaf of \textit{stable} categories, and this is enough to construct the action above. Moreover, such an action exists over any base (e.g. $\mathbb Z$) if we consider gerbes on the relative infinitesimal stack $X_{\mr{inf}}$. We plan to discuss this construction in more detail in a future work.
\end{rem}
\begin{rem}[Reparametrization equivariance]
Let $\cR_X$ be the discrete group of
$\cO_{X^\lp}$-algebra automorphisms of $\cO_{X^\lp}[[h]]$
that induce the identity modulo $h^2$.
The group $\cR_X$ acts by pullback on
$\mathsf{Quant}^{\mathsf{Cat}}(X,[\eta])$ and on the
target of the categorical period equivalence
\eqref{eq:period map in characteristic p}.
In general, this equivalence is not equivariant for
these actions.

The generalized central reduction of twistor differential
operators constructed in
Section \ref{ssec:central reduction of twistor differential operators}
determines, after inverting $h$, an $L\bG_m$-gerbe
$[\mathcal D_{X,[\eta]/h}]$ over $X^\lp$.
Proposition~\ref{prop:reparametrization-equivariance}
shows that the modified period functor
\[
  \widetilde{\mathsf{Per}}^{\mathsf{Cat}}\colon
  \mathsf{Quant}^{\mathsf{Cat}}(X,[\eta])
  \rightarrow
  \Bigl(\tau^{\leq2}R\Gamma_{\mathrm{fppf}}
    (X^\lp,L\bG_m)\Bigr)[2],
\]
given by
\[
  \widetilde{\mathsf{Per}}^{\mathsf{Cat}}(\fC)
  =
  \mathsf{Per}^{\mathsf{Cat}}(\fC)_{\mathrm{gen}}
  +[\mathcal D_{X,[\eta]/h}],
\]
is $\cR_X$-equivariant for the pullback actions.
Here the subscript $\mathrm{gen}$ denotes forgetting
the Frobenius trivialization and extending the band
$L^+\bG_m\to L\bG_m$.
An intrinsic categorical construction of this functor
is given in
Construction~\ref{rem:intrinsic-modified-period}.
\end{rem}
\subsection{Categorical twistor quantization in characteristic $p$.}\label{ssec:categorical twistor quantization} Let $k:=\mathbb F_p$. The main novelty of our work is the observation that the formal categorical quantization $\QCoh_h$ is in fact not formal and naturally extends to $X^\lp \times \Spec k[h]$, and can even be naturally extended to a sheaf of categories $\QCoh_{ \mathbb P^\mathsf{tw}}$ on $X^\lp \times \mathbb P^1$ (where $h$ is the standard linear coordinate at $0\in \mathbb P^1$). We call the sheaf $\QCoh_{\mathbb P^\mathsf{tw}}$ the \textit{twistor categorical quantization} of $(X,[\eta])$, alluding to the familiar picture of Simpson's ``twistor space'' of $\lambda$-connections in the complex-analytic setting. One beneficial feature of this extension of $\QCoh_h$ across $\mathbb P^1$ given by $\QCoh_{\mathbb P^\mathsf{tw}}$ is that it satisfies a certain strong uniqueness property which makes checking many of its other properties completely formal. Before listing some of the properties below in Theorem \ref{thm:intro_main_thm}, let us explicitly describe $\QCoh_{\mathbb P^\mathsf{tw}}$ in the case of a cotangent bundle, to give the reader a rough idea for how it should look in general.

\begin{example}[Section \ref{app.cotangentbundle}]\label{ex:our sheaf for cotangent bundle}
    Let $X=T^*Y$. In that case $\QCoh_{\mathbb P^\mathsf{tw}}$ identifies with the quasi-coherent sheaf of abelian categories defined by the quasi-coherent sheaf of algebras $\mc D_{Y,\mathbb P^1}$. More precisely, for any $f\colon \Spec R \to X^\lp \times \mathbb P^1$ the sections $\QCoh_{\mathbb P^\mathsf{tw}}(\Spec R)$ are given by the $R$-linear abelian category $f^*\mc D_{Y,\mathbb P^1}\lmod$. The sheaf $\mc D_{Y,\mathbb P^1}$ comes from the sheaf of crystalline differential operators $\mc D_Y$ via a version of Rees construction. Namely, the center of $\mc D_Y$ is locally identified with functions on $(T^*Y)^\lp$, and there are two natural filtrations on $\mc D_Y$: increasing \textit{Hodge}, or, in other words, order filtration, and decreasing \textit{conjugate} filtration\footnote{These filtrations are Koszul-dual to the Hodge and conjugate filtrations on the de Rham complex in a certain sense.}, which is induced by the filtration by powers of the ideal of the zero section on the center (see Section \ref{ssec:twistor differential operators} for details). Via a version of Rees construction for a pair of filtrations, this datum defines a quasi-coherent sheaf of algebras on the quotient stack $[((T^*Y)^\lp \times \mbb P^1)/\mathbb G_m]$ (here the restriction of the filtration to the center is non-trivial, and $\mathbb G_m$ acts by rescaling along the fibers of $T^*Y\to Y$). The quasi-coherent sheaf of algebras $\mc D_{Y,\mathbb P^1}$ is the pull-back of this sheaf to $(T^*Y)^\lp \times \mbb P^1$. We call $\mc D_{Y,\mathbb P^1}$ the sheaf of \textit{twistor differential operators} on $Y$.
\end{example}

Now let us describe some properties of $\QCoh_{\mathbb P^\mathsf{tw}}$ in the general case. When various restricted symplectic schemes $(X,[\eta])$ are involved in a statement, we will denote $\QCoh_{\mathbb P^\mathsf{tw}}$  by $\QCoh_{X,[\eta],\mathbb P^\mathsf{tw}}$ to distinguish them.
\begin{Th}\label{thm:intro_main_thm} Let $S$ be a base scheme of characteristic $p>2$ and let $(X,[\eta])$ be a restricted symplectic $S$-scheme. Then there is a natural quasi-coherent sheaf of categories $\QCoh_{\mathbb P^\mathsf{tw}}:=\QCoh_{X,[\eta],\mathbb P^\mathsf{tw}}$ over $X^\lp\times \mathbb P^1$, which satisfies the following properties: 
\smallskip 
 \begin{enumerate}
 \item \textup{(\'Etale functoriality).} The association 
$$
(X,[\eta])\mapsto \QCoh_{X,[\eta],\mathbb P^\mathsf{tw}}
$$
is functorial with respect to \'etale maps of restricted symplectic $S$-schemes.
\smallskip 
     \item \textup{(Fiber over zero).} The category of global sections of $\QCoh_{\mathbb P^\mathsf{tw}}$ on $X^\lp\times \{0\} \hookrightarrow X^\lp \times \mathbb P^1$ is naturally identified with $\QCoh(X)$.
     \smallskip
     \item \textup{(Restriction to the formal disk around zero).} Restriction of $\QCoh_{\mathbb P^\mathsf{tw}}$ to $X^\lp \times \Spf k[[h]] \hookrightarrow X^\lp \times \mathbb P^1$ is naturally identified with the canonical categorical quantization $\QCoh_h$ (discussed in Section \ref{ssec:canonical categorical quantization in char p}).\smallskip 
     \item \textup{(Restriction to the complement to zero).} Restriction of $\QCoh_{\mathbb P^\mathsf{tw}}$ to the open subset  $X^\lp\times \mathbb P^1\backslash \{0\}\subset X^\lp \times \mathbb P^1$  is defined by the sheaf of Azumaya algebras $\mc D_{X,[\eta]/h}$, which is the central reduction\footnote{In the case the restricted structure on $X$ comes from a section $\eta$ of $\Omega^1_X$ rather than $\Omega^1_X/d\mc O_X$ this is the restriction of $\mc D_{X,\mathbb P^1}$ (which is a sheaf of algebras over $(T^*X)^\lp$) to the graph $\Gamma_{\eta}\colon X^{\lp} \times \mathbb P^1\backslash \{0\} \to (T^*X)^\lp \times \mathbb P^1$, given by $(x,h)\mapsto ((\eta_x,x), h)$.} of twistor differential operators $\mc D_{X,\mathbb P^1}$ on $X$ defined by $[\eta]/h$ (Construction \ref{constr:central reductions over S}).
          \smallskip
          \item \textup{(Case $X=T^*Y$).} For $(X,[\eta])=(T^*Y,[\eta_{\mathsf{can}}])$ one has a natural equivalence $\QCoh_{\mathbb P^\mathsf{tw}}\simeq \mc D_{Y,\mathbb P^1}\lmod$.
          \smallskip 
     \item \textup{(Fiber over infinity).} The category of global sections of $\QCoh_{\mathbb P^\mathsf{tw}}$ over $X^\lp\times \{\infty\} \hookrightarrow X^\lp \times \mathbb P^1$ is equivalent to $\QCoh(X^{\lp})$.
          \smallskip
     \item \textup{(\'Etale local model).} There is an \'etale cover $j\colon U\twoheadrightarrow X^{\lp}$ and a map $q\colon U \to T^*\mathbb A^n$ such that $j^*\QCoh_{X,[\eta],\mathbb P^\mathsf{tw}} \simeq q^*\mc D_{\mathbb A^n,\mathbb P^1}\lmod$; in particular, \'etale locally over $X^{\lp}$ the sheaf $\QCoh_{X,[\eta],\mathbb P^\mathsf{tw}}$ is given by modules over a quasi-coherent sheaf of algebras. This is a consequence of \'etale Darboux lemma, which we prove in Proposition \ref{prop:etale Darboux lemma}.
          \smallskip
     \item \textup{($\mathbb G_m$-equivariance).} Assume there is a $\mathbb G_m$-action on $X$ that rescales $[\eta]$: namely, for $t\in \mathbb G_m$ one has $t^*[\eta]=t^a\cdot [\eta]$. Then the sheaf $\QCoh_{X,[\eta],\mathbb P^\mathsf{tw}}$ descends to the quotient $[(X^\lp \times \mathbb P^1)/\mathbb G_m]$, where $\mathbb G_m$ acts on $\mathbb P^1$ by rescaling the coordinate $h$ by $t^{a}$.
          \smallskip
     \item \textup{(Multiplicativity).} Given $(X_1,[\eta_1])$ and $(X_2,[\eta_2])$, we have a natural equivalence
     $$
\QCoh_{X_1\times_S X_2,[\eta_1 +\eta_2], \mathbb P^\mathsf{tw}}\simeq p_1^*\QCoh_{X_1,[\eta_1], \mathbb P^\mathsf{tw}}\otimes p_2^*\QCoh_{X_2,[\eta_2], \mathbb P^\mathsf{tw}},
     $$
where $p_i\colon (X_1\times_S X_2)^\lp \times \mbb P^1 \to X_i^\lp \times \mbb P^1$ denotes the projection, and where the tensor product on the right is a version of Lurie's tensor product for quasi-coherent sheaves of abelian categories (Construction \ref{constr: tensor product of sheaves of categories}).
\smallskip 
\item \textup{(Duality).} For $(X,-[\eta])$ one has natural equivalences 
$$
\QCoh_{X,-[\eta], \mathbb P^\mathsf{tw}} \simeq \QCoh_{X,[\eta], \mathbb P^\mathsf{tw}}^\vee \simeq (\id\times \iota)^*\QCoh_{X,[\eta],\mbb P^{\mathsf{tw}}},
$$
where the middle term denotes the dual quasi-coherent sheaf of abelian categories in sense of Construction \ref{constr:dualizable sheaves of categories}, and $\iota$ in the rightmost term is the automorphism of $\mathbb P^1$ sending the coordinate $h$ to $-h$.
\smallskip
\item \textup{(Lagrangian subvarieties).}\label{item:lagrangian} Let $X$ be a smooth $S$-scheme. For a smooth Lagrangian subvariety\footnote{Such Lagrangian subvarieties are called \textit{restricted}.} $i\colon Y\hookrightarrow X$ with ${[\eta]|}_Y=0$ there exists a canonical object  $\cV_{Y^{(1)} \times {\mathbb P^1}}\in \QCoh_{\mathbb P^\mathsf{tw}}(Y\twt \times {\mathbb P^1})$ such that the fiber of $i_*^{(1)}\cV_{Y^{(1)}\times \mbb P^1}\in \QCoh_{\mathbb P^\mathsf{tw}}(X\twt \times {\mathbb P^1})$ over $X\twt \times \{0\}$ is given by $i_*(K_Y^{\otimes \frac{1-p}{2}})\in \QCoh(X)$.
 \end{enumerate}
\end{Th}
\begin{rem}
    When $X$ is quasi-compact and has affine diagonal the scheme $X^\lp\times \mathbb P^1$ is 1-affine (see Remark \ref{rem:1-affineness}). Consequently, the quasi-coherent sheaf of abelian categories $\QCoh_{X,[\eta],\mathbb P^\mathsf{tw}}$ is in fact completely determined by its category of global sections 
    $$\QCoh_{\mathbb P^\mathsf{tw}}(X):=\QCoh_{X,[\eta],\mathbb P^\mathsf{tw}}(X^\lp\times \mathbb P^1)$$
    as a $\QCoh(X^\lp\times \mathbb P^1)$-module. 
\end{rem}

Let us say a couple of words about the construction of $\QCoh_{\mbb P^{\mathsf{tw}}}$. Namely, it is obtained as a pull-back of a certain quasi-coherent sheaf of abelian categories $\cA_ {{\mathbb P^1}}\lmod^{G_0}$ on $BG_0\times \mathbb P^1$ along the natural map $[\eta]\times \id\colon X^{\lp} \times \mathbb P^1 \to BG_0\times \mathbb P^1$, where $[\eta] \colon X^{\lp} \to BG_0$ classifies the torsor $\mc M_{X,[\eta]}\to X^{\lp}$ of restricted Darboux frames. This sheaf is constructed in Section \ref{canonicalquant.theaction}. Here $\mc A_{\mathbb P^1}$ is a sheaf of algebras over $\mathbb P^1$, given by a central reduction of the sheaf $\mc D_{\Spec C,\mathbb P^1}$ where $C=k[x_i]/(x_i^p)_{i=1,\ldots,n}$ is the Frobenius neighborhood of zero in a vector space of half the dimension (see Section \ref{sec: algebra A over projective line} for more details); its restriction to $\Spf k[[h]]\hookrightarrow \mathbb P^1$ is given by $A_h$. In Section \ref{canonicalquant.theaction} we construct an action of $G_0$ on the corresponding quasi-coherent sheaf  of abelian categories $\mc A_{\mathbb P^1}\lmod$ over $\mathbb P^1$. This amounts to constructing an extension
$$
1\to \mc A_{\mathbb P^1}^\times \to \widehat{G}_{\mathbb P^1} \to G_{0,\mathbb P^1} \to 1
$$
of group sheaves over $\mathbb P^1$ together with a compatible map $\widehat{G}_{\mathbb P^1} \to \Aut(\mc A_{\mathbb P^1})$. We glue $\widehat{G}_{\mathbb P^1}$ from two extensions: one over the formal disk $\Spf k[[h]]$, and the other over the complement $\mathbb P^1\backslash\{0\}$ via a version of Beauville-Laszlo gluing. Both extensions are obtained from the extension $1\to L\mathbb G_m \to \widetilde G \to G \to 1$ \eqref{eq:tilde G}: namely, one can view it as an extension $1\to A_h^\times \times \mr{Gr}_{\mathbb G_m} \to \widetilde G \to G_0 \to 1$. This produces two extensions of $G_0$: one given by $1 \to A_h^\times \to \widehat G \to G_0 \to 1$, that we already considered in \eqref{eq:extension for formal neighborhood intro}, and another one:
\begin{equation}\label{eq:extension for the action outside of 0}
1\ra \mr{Gr}_{\mathbb G_m} \to \widetilde G/A_h^\times \to G_0 \ra 1.
\end{equation}
The restriction of $\mc A_{\mathbb P^1}$ to $\mathbb P^1\backslash \{0\}$ is a split Azumaya algebra; interpreting $\mr{Gr}_{\mathbb G_m}$ as the product of $\mathbb Z$ and the mapping sheaf $\underline{\mr{Mor}_*}({\mathbb P^1}\!\setminus \!\{0\},\mbb G_m)$ one can translate \eqref{eq:extension for the action outside of 0} to the datum of $G_0$-action on $\mc A_{\mathbb P^1}\lmod$ on $\mathbb P^1\backslash\{0\}$. Moreover the induced $A_h[h^{-1}]^\times$-extensions for $\widehat G$ and $\widetilde G/A_h^\times$ are naturally isomorphic (see \eqref{eq:gluing}): this produces the required gluing datum via Construction \ref{constr:Drinfeld}. In fact we prove a stronger statement: namely, Theorem \ref{appendix:maintheorem} shows that there is a unique such $G_0$-action on $\mc A_{\mathbb P^1}\lmod$ which is trivial at $\{\infty\}\hookrightarrow \mathbb P^1$ and whose action on the center of the fiber at $\{0\}$ (given by $A_0\lmod$) agrees with the natural action of $G_0$ on $A_0$. 

Let us point out that the above action of $G_0$ on $\mc A_{\mathbb P^1}\lmod$ is provably \textit{not} induced by an action on the sheaf of algebras $\mc A_{\mathbb P^1}$ (Remark \ref{rem:this is not an action on algebra}). On the other hand, the restriction of this action to $\mbb P^1\setminus \{0\}$ can be described as an action on a Morita-equivalent sheaf of algebras $\mc A^\flat_{\mathbb P^1\setminus \{0\}}$ (Section \ref{ssec:action outside of 0}): this description is responsible for property (4) in Theorem \ref{thm:intro_main_thm}.

In general, the properties of $\QCoh_{\mbb P^{\mathsf{tw}}}$ in Theorem \ref{thm:intro_main_thm} are deduced in various subsections of Section \ref{canonicalquant.ofcat}. In Proposition \ref{mainth:quantofthecategory} we also emphasize some extra structure that uniquely defines $\QCoh_{\mbb P^{\mathsf{tw}}}$; this is essentially a reformulation of the uniqueness statement of Theorem \ref{appendix:maintheorem} in terms of the geometry of $X$. This allows to get a description of $\QCoh_{\mbb P^{\mathsf{tw}}}$ in the case of $(T^*X,[\eta])$ (Example \ref{ex:our sheaf for cotangent bundle}, Section \ref{app.cotangentbundle}), as well as formally deduce most of the other properties listed in Theorem \ref{thm:intro_main_thm}. For the part about Lagrangian subvarieties we need to use a variant of Theorem \ref{appendix:maintheorem} for a certain subgroup $G_{0,J}\subset G_0$ (see Section \ref{s.lag}).
\begin{rem}
    By construction, given a map $p\colon S'\to S$ and its pull-back $p\colon (X',[\eta]'):=(X,[\eta])\times_S S'\to (X,[\eta])$ to $X$ one has a natural equivalence 
 $$
\QCoh_{X',[\eta]',\mathbb P^\mathsf{tw}}\simeq (p\times \id)^*\QCoh_{X,[\eta],\mathbb P^\mathsf{tw}}.
 $$
 In particular, considering $(X,[\eta])$ as a symplectic $X^\lp$-scheme and taking the map $q\colon \mc M_{X,[\eta]} \to X^\lp$ for the map $p$ above, one has an identification $(X',[\eta]')\simeq \mc M_{X,[\eta]}\times (\Spec A_0,[\eta_{\mathsf{can}}])$. From this one gets an identification 
 $$
(q\times \id)^*\QCoh_{X,[\eta],\mathbb P^\mathsf{tw}} \simeq (\mc O_{\mc M_{X,[\eta]}} \otimes \mc A_{\mathbb P^1})\lmod.
 $$
 So the torsor $q\colon \mc M_{X,[\eta]} \to X^\lp$ trivializes the quasi-coherent sheaf of categories $\QCoh_{X,[\eta],\mathbb P^\mathsf{tw}}$ (as it should).
\end{rem}

\begin{rem}\label{rem:Simpson}
In the case $X=T^*Y$ we get the following description of the categories of global sections of $\QCoh_{\mbb P^{\mathsf{tw}}}$ along the fibers of the projection $X^{(1)}\times \mathbb P^1 \to \mathbb P^1$:

\begin{itemize}
		\item[$\star$] $\QCoh_{\mbb P^{\mathsf{tw}}}(X\twt\times \{0\})$ is the category of quasi-coherent Higgs sheaves on $Y$;
		\item[$\star$] $\QCoh_{\mbb P^{\mathsf{tw}}}(X\twt\times \{1\})$ is identified\footnote{Via the Morita equivalence $\cD_{X,[\eta]}\sim \cD_Y$  (\cite[Proposition 3.11]{bb}).} with the category $\cD_Y\lmod$ of $D$-modules on $Y$;
		\item[$\star$]  $\QCoh_{\mbb P^{\mathsf{tw}}}(X\twt\times \{\infty\})$ is the category of quasi-coherent Higgs sheaves on the Frobenius twist $Y\twt$.
	\end{itemize}This resembles the construction of the twistor space due to Deligne and Simpson: namely, given a smooth projective variety $Y$ over $\mbb C$ they construct a complex-analytic space $W\rightarrow \bP^1$ whose fiber over $\lambda \in \bP^1(\bC)$ is the moduli space of $\lambda$-connections on $Y$. In particular, fibers $W_0$, $W_1$ and $W_\infty$ over $0$, $1$ and $\infty\in \bP^1(\bC)$ are given by\begin{itemize} \item[$\star$]  holomorphic Higgs fields on $Y$;
		\item[$\star$]  bundles with holomorphic flat connections on $Y$;
		\item[$\star$]  holomorphic Higgs fields on the complex conjugate $\overline Y$
	\end{itemize} correspondingly. This way the sheaf of categories $\QCoh_{\mbb P^{\mathsf{tw}}}$ can be considered as a categorified characteristic $p$ analogue of their construction.
	This motivates calling $\QCoh_{\mbb P^{\mathsf{tw}}}$ the \textit{twistor quantization} of $(X,[\eta])$.
\end{rem}
\begin{rem}[Action by smooth Lagrangian correspondences] One of the early ideas, proposed as an attempt of a mathematical formalization of the physical concept of quantization, is that the quantization procedure should ultimately produce a functor from the (partially defined) Weinstein's ``symplectic category'' \cite{Weinstein1982SymplecticCategory}: its objects are given by symplectic varieties $(X,\omega)$, morphisms between $(X_1,\omega_1)$ and $(X_2,\omega_2)$ are given by Lagrangian correspondences in $(X_1\times X_2,-p_1^*\omega_1+p_2^*\omega_2)$, and their composition is well-defined if the Lagrangians are sufficiently transversal.

We expect that our construction of twistor quantization $\QCoh_{\mbb P^{\mathsf{tw}}}$ is actually providing such a functor from the \textit{restricted} characteristic $p$ variant of Weinstein's symplectic category, at least if one passes to the corresponding derived categories: namely, our construction $(X,[\eta])\to \QCoh_{X,[\eta],\mbb P^{\mathsf{tw}}}$ produces an association
$$(X,[\eta])\mapsto \widehat{D}(\QCoh_{\mbb P^{\mathsf{tw}}})(X),$$ where right-hand side is the category of global sections of the corresponding sheaf of derived categories $\widehat{D}(\QCoh_{X,[\eta],\mbb P^{\mathsf{tw}}})$ over $X^{\lp}\times \mathbb P^1$ (see Section \ref{ssec:qcoherent sheaves of stable categories}). Then (see Remark \ref{rem:action by Lagrangian correspondences}) any \textit{restricted} Lagrangian subvariety $L$ in $(X_1\times X_2,-p_1^*[\eta_1]+p_2^*[\eta_2])$ naturally produces a corresponding Fourier-Mukai functor 
$$
\mr{FM}^{\mr{quant}}_L\colon \widehat{D}(\QCoh_{\mbb P^{\mathsf{tw}}})(X_1) \to \widehat{D}(\QCoh_{\mbb P^{\mathsf{tw}}})(X_2).
$$
We expect that the functors $\mr{FM}^{\mr{quant}}_{L_i}$ compose compatibly when the corresponding Lagrangian correspondences are transversal. 
\end{rem}

Here is another interesting potential application of our construction.
\begin{rem}[Microlocalization of mod $p$ $F$-gauges]\label{rem:microlocalization of F-gauges} Let the base ring $k$ be perfect. Call symplectic restricted $k$-scheme $(X,[\eta])$ with a $\mathbb G_m$-action \textit{homogeneous of weight $a$} if $t^*[\eta]=t^a\cdot [\eta]$ for some $a\in \mathbb Z$ and all $t\in \mathbb G_m$. The cotangent bundle $(T^*Y,[\eta_{\mathsf{can}}])$ with the standard scaling $\mathbb G_m$-action along the fibers is homogeneous of weight 1. We claim that there is a category of ``sheared mod $p$ $F$-gauges'' defined for any homogeneous symplectic restricted $k$-scheme $(X,[\eta])$: concretely, we expect there to be natural association
	$$
\text{$\mathbb G_m$-equivariant smooth homogeneous } (X,[\eta])\quad \mapsto \quad \text{F-Gauge}_{\Prism^{(1)}}^{\mr{sh},p=0}(X,[\eta]),
    $$
whose value on $(T^*Y,[\eta_{\mathsf{can}}])$ contains the unbounded derived category $\text{F-Gauge}_{\Prism}^{p=0}(Y^{(1)})$ of mod $p$ $F$-gauges of Drinfeld and Bhatt-Lurie \cite[\S 4]{BhattFGauges} on $Y^{(1)}$ as a full subcategory. (While the whole category $\text{F-Gauge}_{\Prism^{(1)}}^{\mr{sh},p=0}(T^*Y,[\eta_{\mathsf{can}}])$ should be equivalent to the sheared version, as provided by the upcoming work \cite{BKMVZ}).

Let us explain the construction of $\text{F-Gauge}_{\Prism^{(1)}}^{\mr{sh},p=0}(X,[\eta])$. By Theorem \ref{thm:intro_main_thm}, the sheaf of categories $\QCoh_{\mathbb P^\mathsf{tw}}$ descends to $[(X^{(1)}\times \mathbb P^1)/\mathbb G_m]$; let us denote by $\QCoh_{\mathbb P^\mathsf{tw}/\mathbb G_m}$ the result of descent, and let $\widehat{D}(\QCoh_{\mathbb P^\mathsf{tw}/\mathbb G_m})$ denote the corresponding sheaf of derived categories on $[(X^{(1)}\times \mathbb P^1)/\mathbb G_m]$ (see Section \ref{ssec:qcoherent sheaves of stable categories}). The category $\mc C:=\widehat{D}(\QCoh_{\mathbb P^\mathsf{tw}/\mathbb G_m})([(X^{(1)}\times \mathbb P^1)/\mathbb G_m])$ has pull-back functors
$$
\mr{ev}_0\colon \mc C \to D(\QCoh(X/\mathbb G_m)) \quad \text{ and } \quad \mr{ev}_\infty \colon \mc C \to D(\QCoh(X^{(1)}/\mathbb G_m))
$$ given by restriction of $\widehat{D}(\QCoh_{\mathbb P^\mathsf{tw}/\mathbb G_m})$ to $[(X^{(1)}\times \{0\})/\mathbb G_m]$ and $[(X^{(1)}\times \{\infty\})/\mathbb G_m]$; here the $\mathbb G_m$-action on $X^{(1)}$ is through the relative Frobenius $F_{\mathbb G_m/k}\colon \mathbb G_m \to \mathbb G_m$. One can view $[X^{(1)}/\mathbb G_m]$ as the Frobenius twist of $[X/\mathbb G_m]$ relative to $B\mathbb G_m$; in particular there is a natural pull-back functor $F_{B\mathbb G_m}^*\colon D(\QCoh(X/\mathbb G_m)) \to D(\QCoh(X^{(1)}/\mathbb G_m))$. Then $\text{F-Gauge}_{\Prism^{(1)}}^{\mr{sh},p=0}(X,[\eta])$ can be defined as the right $t$-completion of the equalizer
$$
\mr{Eq}\left(\!\!\begin{tikzcd}
\mc C
  \arrow[rr, shift right=0.5ex, swap,
    "\operatorname{ev}_{\infty}"]
  \arrow[rr, shift left=0.5ex, 
    "F_{B\mathbf G_m}^{*}\circ\operatorname{ev}_{0}"]
&&
D(\QCoh(X^{(1)}/\mathbb G_m))
\end{tikzcd} \! \!\right);
$$
in other words, we glue $\mc C$ with itself along fibers over $0$ and $\infty$ via the functor $F_{B\mathbb G_m}^*$ (and then right $t$-complete).
\end{rem}

Remark \ref{rem:microlocalization of F-gauges} also motivates the following question. 
\begin{question}
    Does the association $(T^*Y,[\eta_{\mathsf{can}}])\mapsto \text{F-Gauge}_{\Prism}^{\mr{sh}}(Y^{(1)})$ of all F-gauges on $Y^{(1)}$ also microlocalize onto smooth homogeneous restricted symplectic $k$-schemes?
\end{question}
\subsection{Acknowledgments} The authors are  grateful  to Michael Finkelberg and Dmitry Kaledin for their constant attention to this work. 
We are also grateful to  Alexander Efimov, Nick Rozenblyum, Alexander Petrov, Pavel Etingof, Roman Bezrukavnikov, Jacob Lurie, Maxim Kontsevich, Federico Scavia, Ilya Dumanski and  
Akhil Mathew for helpful discussions on the material of this paper.
 We would also like to thank Peter Scholze for bringing up the question about an additional $\mbb G_m$-equivariant structure on $\QCoh_{\mathbb P^{\mathsf{tw}}}$. The third author wishes to thank 
Boris Feigin for introducing him to the ideas of Gelfand-Kazhdan ``formal geometry'' back in the year of $2005$.

The work of Vologodsky was supported by an NSERC Discovery Grant
(RGPIN-2026-07347). Part of this work was carried out during Vologodsky's
visits to the University of Chicago and Princeton University, which were supported by the Simons Foundation (the Simons Collaboration on Perfection in Algebra, Geometry, and Topology and MPS-SIM-00622511). 
Vologodsky thanks both
institutions for their hospitality.  

Dmitry Kubrak is very grateful to Max Planck Institute of Mathematics, Intitut des Hautes \'Etudes Scientifiques, Institute for Advanced Studies and Laboratoire de Mathématiques d'Orsay for the hospitality and excellent working conditions during the time when he was working on this paper.

\subsection*{Tool and computational resource disclosure}
Some of the results in Section 6 were
discovered with the help of computer assistance, specifically ChatGPT. In particular, ChatGPT suggested to us Construction \ref{rem:intrinsic-modified-period} and the computation of the restricted Rozansky-Witten invariant in Example \ref{ex:nonsurjectivity}. 
Discussions with ChatGPT helped us simplify, correct, and streamline 
exposition in various places, in particular, in Appendix \ref{app:surjectivity}. We have reviewed and verified all the content generated by  ChatGPT to ensure its accuracy.

\section{Differential geometry of pseudo-smooth schemes}\label{sec:pseudo-smooth algebras and differential geometry}

\subsection{Pseudo-smooth schemes}\label{ssec:pseudo-smooth schemes}
Let $B$ be an $\mbb F_p$-algebra. Following \cite{bk} we consider the following notion.

\begin{df}\label{def:pseudo-smooth algebras}
	A finite-type $B$-algebra $A$ is called \textit{pseudo-smooth} if $\Omega^1_{A/B}$ is a locally free $A$-module. A morphism of schemes $X\ra S$ locally of finite type is called \textit{pseudo-smooth} if $\Omega^1_{X/S}$ is a locally free sheaf of finite type.
\end{df}

\begin{rem}
	In \cite{bk}, Bezrukavnikov and Kaledin call such algebras quasi-regular\footnote{They also assume that $A/B$ is flat and $B$ is Noetherian.}. However, since terms quasi-regular and quasi-smooth are quite broadly used with a different meaning, we decided to use the term pseudo-smooth.
\end{rem}

\begin{example}\label{ex:pseudo-smooth schemes}\begin{enumerate}
		\item Smooth schemes are pseudo-smooth.
		\item Let $A\coloneqq B[x]/x^p$. Then $\Omega^1_{A/B}=A\cdot dx$ and $A$ is pseudo-smooth.
		\item \label{item:closed embeddings are pseudo-smooth}Let $Z\hookrightarrow S$ be a closed subscheme. Then $\Omega^1_{Z/S}=0$ and $Z$ is pseudo-smooth over $S$.
		\item If $X_1, X_2$ are pseudo-smooth schemes over $S$, then $X_1\times_SX_2$ is also pseudo-smooth over $S$.
	\end{enumerate}
\end{example}
\begin{rem}
	One might think of Example \ref{ex:pseudo-smooth schemes}\eqref{item:closed embeddings are pseudo-smooth} as a pathology and impose a flatness condition in Definition \ref{def:pseudo-smooth algebras}. This would be relatively harmless for our goals: namely, pseudo-smooth schemes that will appear in main applications will all be flat. 
\end{rem}

\begin{rem}	Pseudo-smoothness is preserved under base change: namely, having a pseudo-smooth $S$-scheme $X\ra S$ and a morphism $S'\ra S$ the base change $X'=X\times_S S'$ is pseudo-smooth over $S'$. Indeed one has $\Omega^1_{X'/S'}=f^*\Omega^1_{X/S}$ where $f\colon X'\ra X$ is the natural map.
\end{rem}

\begin{constr}[Reduced Frobenius: algebras]
	
	Let $A$ be a $B$-algebra. The absolute Frobenius map $F_A\colon A\ra A$ factors naturally through the Frobenius twist $A\twt\coloneqq A\otimes_{B,F_B}B$:
	
	$$
	\xymatrix{A\ar[rr]^{F_A}\ar@{-->}[rd]_{W_{A/B}}& &  A\\
		& A\twt\ar@{-->}[ru]_{F_{A/B}}&}
	$$
	with $W_{A/B}\colon a\mapsto a\otimes 1\in A\twt$ and $F_{A/B}\colon a\otimes b \mapsto a^pb\in A$. The map $F_{A/B}$ is usually called the relative Frobenius; by construction it is $B$-linear (while $F_A$ and $W_{A/B}$ are $F_B$-linear). However, if $B$ or $A$ are not reduced one can factor $F_{A/B}$ even further. Namely, let $A^{\lp}\subset A$ be the image of $F_{A/B}$: this is the $B$-subalgebra in $A$ generated by $p$-th powers $A^p$, in other words $A^{\lp}\coloneqq B\cdot A^p\subset A$. By definition, one has a natural surjection  $F_A^{{int}}\colon A\twt \twoheadrightarrow A^\lp$. Moreover, we get a further factorization 
	$$
	\xymatrix{A\ar[rrr]^{F_A}\ar[rd]_{W_{A/B}}&&& A\\
		&A\twt \ar[rru]^{F_{A/B}}\ar@{-->}[r]_{F_A^{int}}&A^\lp\ar@{-->}[ru]_{F_{A}^\lp}&}
	$$
	We have $F_{A/B}=F_A^\lp\circ F_{A}^{int}$ and we can put $W_{A}^\lp\coloneqq F_A^{int} \circ W_{A/B}$. We will call $F_A^\lp$ and $F_A^{int}$ the \textit{reduced} and \textit{intermediate} Frobenii of $A$ respectively.
	
\end{constr}

In the case when $A^{(1)}$ is reduced, one can show that $F_A^{int}$ is an isomorphism, and so $F_{A/B}=F_A^\lp$. However, further we will be interested in the non-reduced examples, and there the reduced twist $A^{\lp}$ will turn out to be more relevant than $A\twt$ for our purposes. 

\begin{rem} When $A$ is pseudo-smooth, $F_A^\lp$ still enjoys some of the properties that the relative Frobenius $F_{A/B}$ has if we assume that $A$ is actually smooth. In particular (see Remark \ref{rem:degree of reduced Frobenius}) $F_A^\lp$ is still a finite faithfully flat map, which is no longer true for $F_{A/B}$. This is illustrated by the example of $A\coloneqq B[x]/x^p$, where $A\twt\simeq B[x]/x^p$ and the relative Frobenius $F_{A/B}\colon B[x]/x^p \ra B[x]/x^p$ is the unique $B$-algebra map that sends $x$ to 0. This map is clearly not flat. On the other hand, $A^{\lp}\subset A$ is identified with $B\subset B[x]/x^p$ and the reduced Frobenius $F_A^\lp\colon A^{\lp}\ra A$ is given by the above embedding which is indeed finite and faithfully flat.
\end{rem}

\begin{rem}
	In \cite{bk}, for a pseudo-smooth $A$ it is claimed that $A^\lp$ can be identified with the tensor product $A^p\otimes_{B^p}B$ via the natural map $a^p\otimes b \mapsto a^pb$. This is not always true. Consider $B\coloneqq \mbb F_p[x]$ and $A=B[t]/(t^p-x)$. One has $\Omega^1_{A/B}\simeq A\cdot dt$, and so $A$ is pseudo-smooth. On the other hand, $B^p\simeq \mbb F_p[x^p]$, $A^\lp\simeq A^p\simeq B$, and so the natural map $A^p\otimes_{B^p}B\ra A^\lp$ is identified with the multiplication map $\mbb F_p[x]\otimes_{\mbb F_p[x^p]} \mbb F_p[x]\ra \mbb F_p[x]$.
\end{rem}

\begin{rem}\label{rem:universal homeo}
	Let $X=\Spec A$. By \cite[Tag0BR8]{stacks}, $F_X^{\lp}\colon X\ra X^\lp\coloneqq \Spec A^\lp$ is a universal homeomorphism. In particular, it induces an isomorphism of the underlying topological spaces.  In fact all three schemes $X\ra X^\lp \hookrightarrow X\twt\coloneqq \Spec A\twt$ have the same underlying topological space and the difference between them is only seen on the level of sheaves of functions. 
\end{rem}
\begin{constr}[Reduced Frobenius: schemes]\label{constr:reduced Frobenius}
	Given a base scheme $S$ and an $S$-scheme $X$, the construction $A\ra A^\lp$ globalizes, producing an $S$-scheme $X^{\lp}$. This is a scheme which has the same underlying topological space as the original $X$ and its Frobenius twist $X\twt\coloneqq X\times_{S,F_S}S$, but with the sheaf of functions given by $\mc O_X^\lp\coloneqq \mc O_X^p\cdot \mc O_S\subset \mc O_X$. One has maps $\mc O_X^\lp\hookrightarrow \mc O_X$ and $\mc O_{X\twt}\twoheadrightarrow \mc O_X^\lp$ (given locally by $F_A^\lp$ and $F_{A}^{int}$) which induce a map $F_{X}^\lp\colon X\ra X^\lp$ and a closed embedding $F_X^{int}\colon X^\lp\hookrightarrow X\twt$. This way one can also view $X^{\lp}$ as the schematic image of the relative Frobenius $F_{X/S}\colon X \ra X\twt$ inside the Frobenius twist $X\twt\coloneqq X\times_{S,F_S}S$. We will call $F_{X}^\lp$ and $F_X^{int}$ the \textit{reduced} and \textit{intermediate} Frobenii of $X$; both are morphisms of $S$-schemes. We can also define a \textit{reduced twist} map $W_X^\lp\colon X^\lp\ra X$ (by globalizing $W_A^\lp$); this map is not a morphism of $S$-schemes, instead it is $F_S$-linear. 
	One has relations $F_X^{int}\circ F_X^\lp=F_{X/S}$ and $W_{X/S}\circ F_X^{int}=W_X^\lp$.
\end{constr}

\begin{example}\label{ex:reduced vs relative Frobenius}Let $S=\Spec B$.
	\begin{enumerate}
		\item \label{item:intermediate Frobenius for A^d} Put $P\coloneqq B[x_1,\ldots,x_d]$ with $\Spec P\simeq \mbb A^d_S$. Then $P^\lp\coloneqq P^p\cdot B\simeq B[x_1^p,\ldots,x_d^p]\subset P$. On the other hand the relative Frobenius twist $P\twt\coloneqq P\otimes_{B,F_B}B\simeq B[y_1,\ldots,y_d]$ (with $y_i\coloneqq x_i\otimes 1$), and the relative Frobenius $F_{P/B}\colon P\twt\ra P$ is the unique $B$-algebra map that sends $y_i$ to $x_i^p$. We see that in this case the intermediate Frobenius $F_{\mbb A^d_S}^{int}\colon (\mbb A^d_S)^\lp \ra (\mbb A^d_S)\twt$ is an isomorphism. 
		\item Consider the unique surjection $P\twoheadrightarrow B$ of $B$-algebras sending every $x_i$ to 0. It corresponds to the embedding $\{0\}_S\hookrightarrow \mbb A^d_S$. Put
		$$
		Q\coloneqq B\otimes_{P, F^\lp_{P}}P\simeq P/(x_1^p,\ldots,x_d^p).
		$$
		$X\coloneqq\Spec Q\hookrightarrow \mbb A^d_S$ is the Frobenius neighborhood of 0: namely, $X\simeq \{0\}_S\times_{\mbb A^d_S, F^{\lp}_{\mbb A^d}} \mbb A^d_S$. We have $Q^\lp\coloneqq Q^p\cdot B\simeq B\subset Q$. However, $Q\twt \coloneqq Q\otimes_{B,F_B}B\simeq B[y_1,\ldots,y_d]/(y_1^p,\ldots,y_d^p)$, and the intermediate Frobenius $F_{Q}^{int}\colon Q\twt\ra Q^\lp\simeq B$ sends $y_i$ to $x_i^p=0$.
	\end{enumerate}
	
\end{example}

\begin{lm}[Interaction with \'etale maps]\label{lem:twist vs etale maps} Let $f\colon X'\to X$ be an \'etale map of $S$-schemes. Then all squares in the commutative diagram 
$$
\begin{tikzcd}
    X' \arrow[d,"f"'] \arrow[r,"F_{X'}^\lp"]
    & (X')^{\lp}
      \arrow[d,"f^{\lp}"']
      \arrow[r,hook,"F_{X'}^{int}"]
    & (X')\twt \arrow[d,"f\twt"']
    \\
    X \arrow[r,"F_X^\lp"']
    & X^{\lp} \arrow[r,hook,"F_X^{int}"']
    & X\twt
\end{tikzcd}
$$
are Cartesian.    
\end{lm}
\begin{proof}
    The statement for the boundary square (where the horizontal maps are given by relative Frobenii $F_{X'/S}$ and $F_{X/S}$) is true since the relative Frobenius $F_{X'/X}\colon X'\to X'\times_{X,F_X}X\simeq (X')\twt\times_{X\twt,F_{X/S}} X$ is an isomorphism (see e.g. \cite[Tag 0EBS]{stacks}). Moreover, $f\twt\colon (X')\twt \to X\twt$ is \'etale, and in particular flat. By \cite[Tag 081I]{stacks}, the schematic images of quasi-compact morphisms commute with flat base change: applying this to the affine map $F_{X/S}$ and the flat map $f\twt$ we get the statement of the lemma.
\end{proof}

\begin{rem}
	Since  $F_X^{\lp}\colon X\ra X^\lp$ is a universal homeomorphism (Remark \ref{rem:universal homeo}), by \cite[Tag 04DZ]{stacks} it induces an equivalence of small \'etale sites $(X^\lp)_\textup{\'et}\simeq X_\textup{\'et}$, via $$(U\ra X^\lp)\in (X^\lp)_\textup{\'et}\ \ \mapsto \ \ (X\times_{X^\lp}U\ra X)\in X_\textup{\'et}.$$ 
    Via Lemma \ref{lem:twist vs etale maps} the inverse functor is given by 
    $$
(V\to X)\in X_\textup{\'et} \ \mapsto \ (V^\lp\to X^\lp)\in (X^\lp)_\textup{\'et}.
    $$
\end{rem}

\begin{rem}[Interaction with products]\label{rem:twist vs product}
    The endofunctor $X\mapsto X^\lp$ on $S$-schemes is oplax-monoidal, but in general not monoidal. Namely, given $S$-schemes $X_1,X_2$ there is a natural map 
    $$
\iota\colon (X_1\times_SX_2)^\lp \to X_1^{\lp}\times_S X_2^\lp,
    $$
    given invariantly by the embedding of the schematic image of $X_1\times_SX_2$ in $X_1^{\lp}\times_S X_2^\lp$.
    Locally, when $S=\Spec B$ and $X_i=\Spec A_i$ it is defined by the natural algebra map 
\begin{equation}\label{eq:map for twist of product to the product of twists}
(B\cdot A_1^p)\otimes_B (B\cdot A_2^p) \to B\cdot (A_1\otimes_BA_2)^p\subset A_1\otimes_B A_2.
\end{equation}
    This map is always surjective (thus $\iota$ is a closed embedding) but not necessarily injective. For example, take $B=k[t]$, $A_1=B/t=k$ and $A_2=B[x,y]/(x^p-ty)$. Then $B\cdot A_1^p=k$, $B\cdot A_2^p=B[ty,y^p]\subset A_2$, while $B(A_1\otimes_BA_2)^p=k[y^p]$, and \eqref{eq:map for twist of product to the product of twists} is given by the map
    $$
k[t,ty,y^p]/t \to k[y^p],
    $$
    which sends the non-zero element $ty$ to 0.
\end{rem}
\begin{rem}[Interaction with base change]\label{rem:twist vs basechange}
    Given an $S$-scheme $X$ and a morphism $S' \to S$, write $(X \times_S S')^\lp$ for the reduced Frobenius twist of $S'$-scheme
     $X \times_S S'$. We have a canonical morphism of $S'$-schemes
     \begin{equation}\label{eq:base_change_red_twist}
        (X \times_S S')^\lp \to X^\lp \times _S S', 
     \end{equation}
     uniquely characterized by the property that its precomposition with
$F_{X \times_S S'}^\lp\colon X \times_S S' \to (X \times_S S')^\lp $ 
is $F_X^\lp \times \Id$. In general, \eqref{eq:base_change_red_twist} is not an isomorphism. For instance, if $S' \to S$ is a closed embedding then  \eqref{eq:base_change_red_twist} is identified with the map  $(X \times_S S')^\lp \to X^\lp \times _S (S')^\lp $ from Remark \ref{rem:twist vs product} and
the example given there also shows that \eqref{eq:base_change_red_twist} need not be an isomorphism in general.
\end{rem}

\subsection{Restricted Lie algebras and Jacobson's formula}\label{sec:Jacobson's formula}
Let $A$ be an associative {ring} such that $p\cdot A=0$. For $a\in A$ denote by $\ad(a)\colon A\ra A$ a map that sends $b\mapsto [a,b]=ab-ba$. 

If $A$ is commutative, given two elements $a,b\in A$ we have $(a+b)^p=a^p+b^p$. This formula generalizes to the general non-commutative case as follows
\begin{lm}[Jacobson, {\cite[Chapter II, \S 7, Proposition 3.2]{gd}}]\label{lem:Jacobson's formula} Let $a,b\in A$. Then 
	$$(a+b)^p=a^p+b^p +\sum_{i=1}^{p-1} L_i(a,b)$$
	where $L_i(a,b)$ is a Lie polynomial described as {$i^{-1}$ times} the coefficient of $t^{i-1}$ in the expression $\ad(a+tb)^{p-1}(b)\in A[t]$.  
\end{lm}
\begin{rem}\label{rem:L_1}
	In particular, $L_1(a,b)=\ad(a)^{p-1}(b)$.
\end{rem}

\begin{rem}\label{rem:ad^p}
	For an element $a\in A$ denote by $\ad(a)\colon A\ra A$ the commutator map $x\mapsto [a,x]$. Note that $\ad(a)^p=\ad(a^p)$. Indeed, $\ad(a)=\ell(a) - r(a)$ where $\ell(a),r(a)\colon A\ra A$ are left and right multiplications by $a$ correspondingly.  Since $\ell(a)$ and $r(a)$ commute we have 
	$$
	\ad(a)^p= (\ell(a) - r(a))^p=\ell(a)^p-r(a)^p=\ell(a^p)-r(a^p)=\ad(a^p).
	$$ 
\end{rem}
\begin{df}\label{def:restricted Lie algebra}
	\textit{A restricted Lie algebra} over a commutative ring $B$ is given by a $B$-module $L$ with a $B$-linear Lie bracket $[-,-]$ and a restricted $p$-power operation $-^{[p]}$ that are compatible in the following way: 
	\begin{enumerate}
		\item $(bx)^{[p]}=b^px^{[p]}$ for any $b\in B$ and $x\in L$;
		\item  $[x^{[p]},y]=\ad(x)^p(y)$ for any $x,y\in L$;
		\item \label{item:Jacobson} $(x+y)^{[p]}=x^{[p]}+y^{[p]}+ \sum_{i=1}^{p-1} L_i(x,y)$, where $L_i(x,y)$ are the Lie polynomials from Lemma \ref{lem:Jacobson's formula}.
	\end{enumerate} 
\end{df}

The following is the most natural example of a restricted Lie algebra:
\begin{example}\label{ex:ass algebra as a restricted Lie algebra}
	Let $A$ be an associative $B$-algebra. Then by Lemma \ref{lem:Jacobson's formula} and  Remark \ref{rem:ad^p}, $A$ endowed with the Lie bracket given by commutator $[-,-]$ and the restricted structure $a^{[p]}\coloneqq a^p$ is a restricted Lie algebra. 
\end{example}

Another source of natural examples of restricted Lie algebras is differential geometry in char $p$ (see Remark \ref{rem:vector fields as restricted Lie algebra}). In particular, the Lie algebra of any group scheme $G$ is endowed with a natural restricted Lie algebra structure. 
\subsection{Crystalline differential operators}\label{ssec:definition of diff op}

We essentially follow \cite[Section 2]{bmr}, except we work in a relative and pseudo-smooth setting.

Let $X\ra S$ be a pseudo-smooth $S$-scheme. Let $T_{X}\coloneqq T_{X/S}$ be the sheaf of $\mc O_S$-linear derivations of $\mc O_X$. We will call (local) sections of $T_{X}$ vector fields. Denote $\Omega^1_X\coloneqq \Omega^1_{X/S}$. One has $T_{X}\simeq (\Omega^1_X)^\vee$, and, since $X$ is pseudo-smooth, we get that $T_{X}$ is a locally free $\mc O_X$-module of finite rank.

The sheaf $T_{X}$ has a natural ($\mc O_S$-linear) Lie bracket given by the commutator\footnote{Namely, the commutator of two derivations is again a derivation.} which endows it with the structure of a Lie algebroid over $X$. Let $\mc D_X$ be the universal enveloping algebra of $T_X$. Explicitly, $\mc D_X$ is generated over $\mc O_S$ by $\mc O_X$ and $T_X$ with the relations given by $f_1\cdot f_2-f_2\cdot f_1=0$ for $f_i\in\mc O_X$, $f\cdot v=(fv)\in T_X$ and $v\cdot f-f\cdot v=v(f)\in \mc O_X$ for $v\in T_X$, $f\in \mc O_X$, and $v_1\cdot v_2 - v_2\cdot v_1=[v_1,v_2]\in T_X$ for $v_i\in T_X$. Giving an $\mc O_X$-quasicoherent $\mc D_X$-module is equivalent to giving a quasicoherent sheaf $\mc E$ endowed with a flat connection $\nabla\colon \mc E\ra \mc E\otimes \Omega^1_{X}$.

\begin{rem}\label{rem:PBW-filtration} The algebra $\mc D_X$ comes with a natural increasing ``PBW-filtration'' which is the filtration by order of the differential operator. Namely, $\mc D_{X,\le n}\subset \mc D_X$ is 0 if $n<0$, $\mc D_{X,\le 0}=\mc O_X$, and for $n>0$ one has $\mc D_{X,\le n}:=\mc D_{X,\le n-1}+T_X\cdot \mc D_{X,\le n-1}\subset \mc D_X$. Note that since $[T_X,T_X]\subset T_X$ we have $[\mc D_{X,\le i}, \mc D_{X,\le j}]\subset \mc D_{X,\le i+j-1}$. By a version of PBW-theorem for Lie algebroids (see e.g. \cite{rinehart}) one has a natural isomorphism
	$$
	\gr_\star\mc D_X\simeq \Sym^\star_{\mc O_X}T_X.
	$$
	Consequently, if we have a trivialization $T_X\simeq \mc O_X^d$ given by a frame of vector fields $\partial_1,\ldots,\partial_d$ we can decompose 
	$$
	\mc D_X\simeq \oplus_{\alpha\in \mbb N^d}\mc O_X\cdot \partial^\alpha
	$$
	as a left $\mc O_X$-module. Here, for $\alpha\in \mbb N^d$, $\partial^\alpha$ denotes the monomial $\partial_1^{\alpha_1}\ldots \partial_d^{\alpha_d}$. In terms of this decomposition $\mc D_{X,\le n}\simeq \oplus_{\alpha\in \mbb N^d}\mc O_X\cdot \partial^\alpha$ with $$|\alpha|=\alpha_1+\ldots+\alpha_d\le n.$$
\end{rem}

\begin{rem}[Restricted Lie algebra structure on $T_X$]\label{rem:vector fields as restricted Lie algebra}
	In char $p$, $T_X$ is naturally a sheaf of restricted Lie algebras over $\mc O_S$. Namely, note that the $p$-th iterate of an $\mc O_S$-linear derivation $v\in T_X$ again defines an $\mc O_S$-linear derivation of $\mc O_X$, which we will denote $v^{[p]}\in T_X$. Since $v^{[p]}$ is equal to the image of $v^p$ in the $\mc O_S$-linear endomorphisms $\mc End_{\mc O_S}(\mc O_X)$ and the Lie bracket on $T_X$ is induced by the commutator in $\mc End_{\mc O_S}(\mc O_X)$, one sees from Example \ref{ex:ass algebra as a restricted Lie algebra} that the relations in Definition \ref{def:restricted Lie algebra} are satisfied.
\end{rem}

The algebra $\mc D_X$ naturally acts on $\mc O_X$ by $\mc O_S$-linear endomorphisms: namely, $\mc O_X$ acts by multiplication, while $T_X$ acts by the corresponding $\mc O_S$-linear derivations. 
Note that any $\mc O_S$-linear derivation acts by zero on all functions in $\mc O_X^\lp=\mc O_X^p\cdot \mc O_S\subset \mc O_X$. It follows that $\mc O_X^\lp\subset \mc O_X$ lies in the center $Z(\mc D_X)$, and so one can consider $\mc D_X$ as a sheaf of $\mc O_X^\lp$-algebras. By adjunction, from the identification $(F_X^\lp)^{-1}\mc O_{X^\lp}\simeq \mc O_X^\lp$, we get a map $\mc O_{X^{\!\lp}}\ra Z(F_{X*}^\lp\mc D_X)$ and this way we can consider $F_{X*}^\lp\mc D_X$ as a sheaf of algebras over $X^{\lp}$. 

\begin{rem}
	By construction, the map $F_X^\lp\colon X\ra X^\lp$ is affine. Consequently, one has a symmetric monoidal equivalence between the category of  quasi-coherent sheaves on $X$ and the category of quasi-coherent sheaves of $F_{X*}^\lp\mc O_X$-modules on $X^\lp$. This then gives a similar equivalence between quasi-coherent $\mc D_X$-modules on $X$ and quasi-coherent $F_{X*}^\lp\mc D_X$-modules on $X^\lp$.
\end{rem}

\begin{rem}\label{rem: reduced Frobenius twist might not be smooth}
	Note that by base change,  $\Omega^1_{X\twt}\simeq W_{X/S}^*\Omega^1_X\coloneqq \Omega^1_X\otimes_{\mc O_X}\mc O_{X\twt}$. In particular, if $X$ is pseudo-smooth, then so is $X\twt$. This is not necessarily true for $X^\lp$ (see Remark \ref{rem:reduced Frobenius in the case of Frobenius neighborhood}).
\end{rem}

Let us also define the appropriate variants of twisted tangent and cotangent bundles.

\begin{constr}[Twisted tangent and cotangent bundles]\label{constr:twisted tangent and cotangent bundles} For the usual Frobenius twist, one has $\Omega^1_{X\twt}\simeq W_{X/S}^*\Omega^1_X$ and $T_{X\twt}\simeq W_{X/S}^*T_X\coloneqq T_X\otimes_{\mc O_X}\mc O_{X\twt}$. We define the \textit{(reduced Frobenius) twisted tangent and cotangent bundles} as 
$$T_X^\lp\coloneqq W_{X}^{\lp*}T_X\simeq  T_X\otimes_{\mc O_X}\mc O_{X}^{\lp} \qquad \text{and} \qquad \Omega_X^{1,\lp}\coloneqq W_{X}^{\lp*}\Omega^1_X\simeq \Omega_X^1 \otimes_{\mc O_X}\mc O_{X}^{\lp}$$ (here the map $\mc O_X\ra \mc O_{X}^\lp$ sends $f\mapsto f^p$). One has $\Omega_X^{1,\lp}\simeq (T_X^\lp)^\vee$. We warn that it is not necessarily true that $T_X^\lp$ (resp. $\Omega_X^{1,\lp}$) is isomorphic to $T_{X^\lp}$ (resp. $\Omega^1_{X^\lp}$), since for example $X^\lp$ can fail to be smooth. Nevertheless, there is still some natural compatibility (see Corollary \ref{cor:total space of twist is the twist of original bundles}).
\end{constr}

For a vector bundle $\mc E$ on a scheme $X$ let $\mr{Tot}_X(\mc E)\coloneqq \Spec_X (\Sym^\star_{\mc O_X}\mc E^\vee)$ be the total space of $\mc E$. In the case $\mc E$ is $T_X$ or $\Omega^1_X$ we denote it by $TX$ and $T^*\! X$ correspondingly. We denote by $\mc E^\lp\coloneqq W_{X}^{\lp*}\mc E\simeq \mc E\otimes_{\mc O_X}\mc O_{X}^\lp$ the pull-back of $\mc E$ to $X^\lp$.

\begin{lm}\label{lem:twist commutes with taking global sections}
	One has a natural isomorphism 
	$$
	\mr{Tot}_{X^{\!\lp}}(\mc E^\lp)\simeq \left(\mr{Tot}_{X}(\mc E)\right)^\lp.
	$$
\end{lm}
\begin{proof}
	Note that $\mr{Tot}_{X^{\!\lp}}(\mc E^\lp)\simeq \mr{Tot}_{X}(\mc E)\times_X X^\lp$ under the map $W_{X}^\lp\colon X^\lp \ra X$. Indeed, $\Sym_{\mc O_X^\lp}^*((\mc E^\vee)^\lp)\simeq (\Sym_{\mc O_X}^*\!\mc E^\vee)\otimes_{\mc O_X}\mc O_X^\lp$ and the isomorphism is obtained by taking $\Spec_{X^\lp}$ of both sides. By taking $W_{\mr{Tot}_{X}(\mc E)}^\lp\colon \left(\mr{Tot}_{X}(\mc E)\right)^\lp \ra \mr{Tot}_{X}(\mc E)$ and the projection $\left(\mr{Tot}_{X}(\mc E)\right)^\lp\ra X^\lp$ we obtain a map 
	$$
	\left(\mr{Tot}_{X}(\mc E)\right)^\lp \ra \mr{Tot}_{X}(\mc E)\times_X X^\lp\simeq \mr{Tot}_{X^{\!\lp}}(\mc E^\lp).
	$$
	To check that it is an isomorphism we can work locally on $X$ and $S$. So we can assume $X=\Spec A$, $S=\Spec B$ and $\mc E$ is trivial. Then we need to check that the natural map 
	$$
	A[x_1,\ldots,x_n]\otimes_{A}A^\lp \ra (A[x_1,\ldots,x_n])^\lp
	$$
	given by $f\otimes s\mapsto f^ps$ is an isomorphism. But it is clear that the right-hand side is the polynomial ring on $x_1^p,\ldots,x_n^p$ over $A^\lp$ and so this map is indeed an isomorphism.
\end{proof}
Plugging $\mc E$ to be $T_X$ or $\Omega^1_X$, we get:
\begin{cor}\label{cor:total space of twist is the twist of original bundles}
	There are natural isomorphisms 
	$$
	\mr{Tot}_{X^{\!\lp}}(T_X^\lp)\simeq (TX)^\lp \text{ and } \mr{Tot}_{X^{\!\lp}}(\Omega_X^{1,\lp})\simeq (T^*\! X)^\lp.
	$$
\end{cor}


\subsection{Differential operators on smooth schemes}\label{ssec:differential operators smooth}
For this section we assume that $X$ is a smooth scheme over $S$. We begin by identifying the reduced Frobenius in this setting.

Assume that $X=\mbb A^d_S$ is the affine space over $S$ of dimension $d$. Locally over $S$, we have $S=\Spec B$ and $X=\Spec B[x_1,\ldots,x_d]$. Let $P\coloneqq B[x_1,\ldots,x_d]$. We have $P^\lp=B[x_1^p,\ldots,x_d^p]\subset P$ and it is clear that $P$ is a free $P^\lp$-module of rank $p^{d}$. In particular, we see that $F_X^\lp\colon X\ra X^\lp$ is a finite faithfully flat map. Moreover, (see Example \ref{ex:reduced vs relative Frobenius}), $F_{P}^{int}$ induces an isomorphism $P\twt\xra{\sim} P^\lp$. Equivalently, the map $F_X^{int}\colon X^\lp\xra{\sim} X\twt$ is an isomorphism and, under this identification, $F_X^\lp=F_{X/S}$.

In general smooth setting, Zariski locally on $X$ there is an \'etale map $q\colon X\ra \mbb A^d_S$ for some $d$. From Lemma \ref{lem:twist vs etale maps} it follows that $F_X^{int}\colon X^\lp\xra{\sim} X\twt$ is also an isomorphism, and $F_X^\lp=F_{X/S}$, as in the case of affine space. We also get that $F_X^\lp$ is a finite faithfully flat map of rank $p^{\dim_S X}$.

Given the above identification $X^\lp\xra{\sim} X\twt$, one can identify the vector bundle $T_X^\lp$ with $T_{X\twt}$ and, consequently, its total space $(T^{*}X)^\lp$ with $T^*\! X\twt$.

\begin{constr}[The $p$-curvature map]\label{constr:theta} Let $X$ be a smooth $S$-scheme.
	Consider the map $
	\theta\colon T_X \ra \mc D_X
	$ of sheaves on $X$ given by the following formula
	$$
	v\mapsto \theta(v)\coloneqq v^p-v^{[p]},
	$$
	where $-^{[p]}$ is the restricted $p$-th power operation (see Remark \ref{rem:vector fields as restricted Lie algebra}).
	We claim that 
	\begin{itemize}
		\item $\theta$ is additive: $\theta(v_1+v_2)-(\theta(v_1)+\theta(v_2))=0$;
		\item $\theta$ is Frobenius-semilinear: $\theta(fv)-f^p\cdot \theta(v)=0$;
		\item $\theta(v)$ lies in the center $Z(\mc D_X)$ for any $v\in T_X$: $[\theta(v),f]=[\theta(v),w]=0$ for any $f\in \mc O_X$ and $w\in T_X$.
	\end{itemize}
	Indeed, note that $\theta(v)\in \mc D_{X,\le p}$ and its image in $\gr_p\mc D_X\simeq \Sym^p_{\mc O_X}(T_X)$ is $v^p$. One can see\footnote{For all cases except $[\theta(v),w]$ this follows from the fact that the associated graded of $\mc D_X$ is commutative. For $[\theta(v),w]$ one can use that $[-,w]$ induces a derivation of $\gr_\star\mc D_X$ by sending $x\in \gr_i\mc D_X$ to the class of $[\tilde x,w]\in \mc D_{X,\le i}$ in $ \gr_i\mc D_X$, and that any derivation is 0 on the $p$-th power $v^p$.} that all expressions above project to 0 in $\gr_p\mc D_X$, and so belong to $\mc D_{X,\le p-1}$. Thus, by Lemma \ref{lm:embedding of diff operators of order p-1} below, they are zero if and only if their images in $\mc End_{\mc O_S}(\mc O_X)$ are. However, $\theta(v)$ acts by 0 on any function $f$ in $\mc O_X$ since by definition $v^{[p]}$ acts as $v^{p}$, and so the image of all of the expressions above in $\mc End_{\mc O_S}(\mc O_X)$ is 0.
	
	This way we obtain a map 
	$$
	\theta'\colon T_X\otimes_{\mc O_X}\mc O_X^p \ra Z(\mc D_X)
	$$
	of $\mc O_X^p$-modules. Since $Z(\mc D_X)$ is in fact an $\mc O_X^\lp\simeq \mc O_X^p\cdot \mc O_S$-algebra we can extend this to a map
	$
	T_X\otimes_{\mc O_X}\mc O_X^\lp\ra  Z(\mc D_X)
	$ of $\mc O_X^\lp$-modules or, in other words, a map of quasi-coherent sheaves on $X^\lp$
	\begin{equation}\label{eq:p-curvature map}
		\theta\colon T_X^\lp \ra Z(F_{X*}^\lp\mc D_X),
	\end{equation}
	(compare with \cite[(2.0.4)--(2.0.5)]{ov}) which we call the ``$p$-curvature map''. 

	
\end{constr}

\begin{lm}\label{lm:embedding of diff operators of order p-1}
	Let $X$ be smooth over $S$. Then the natural $\mc O_S$-linear action of $\mc D_X$ on $\mc O_X$ induces an embedding $\mc D_{X,\le p-1}\hookrightarrow \mc End_{\mc O_S}(\mc O_X)$.
\end{lm}
\begin{proof}
	The statement is Zariski-local, so we can assume that we have an \'etale map $f\colon X\ra \mbb A^n_S$ given by functions $f_1,\ldots,f_n$. Let $z_1,\ldots,z_n$ be the coordinates on $\mbb A^n$. We have an isomorphism $T_X\simeq f^*T_{\mbb A^n_S}$ and can consider the trivialization of $T_X$ given by $\partial_i\coloneqq f^{-1}(\partial_{z_i})$. We have $f_i=f^*(z_i)$ and this way $\partial_i(f_j)=\delta_{ij}$. $T_X$ is trivial with basis $\partial_i$, $i=1,\ldots,n$ where $n=\rk(\Omega^1_X)$. Let $D\in \mc D_{X,\le p-1}$; by Remark \ref{rem:PBW-filtration} it can be written as a sum 
	$$
	D=\sum_{\alpha\in \mbb N^n} g_\alpha\cdot \partial^\alpha
	$$
	where $g_{\alpha}\in \mc O_X$ and $\alpha=(\alpha_1,\ldots,\alpha_n)$ is such that $|\alpha|\le p-1$. Take a minimal $\alpha$ with respect to the lexicographic order on $\mbb N^n$ such that $g_{\alpha}\neq 0$; then it is easy to check that $D(f_1^{\alpha_1}\ldots f_n^{\alpha_n})=\alpha_1!\ldots \alpha_n!\cdot g_\alpha $, which is non-zero since all $\alpha_i<p$.
\end{proof}

Recall the twisted tangent and cotangent bundles  $T_X^\lp$, $\Omega^{1,\lp}_X$ (Construction \ref{constr:twisted tangent and cotangent bundles}). By Corollary \ref{cor:total space of twist is the twist of original bundles} we have
$$
(T^{*}X)^\lp\simeq \Spec_{X^{\!\lp}} (\Sym^\star_{\mc O_X^{\!\lp}}T_X^\lp),
$$ the total space of $\Omega^{1,\lp}_X$. The natural projection $(T^{*}X)^\lp\ra X^\lp$ is affine and so one has an equivalence of categories of quasi-coherent sheaves on $(T^*\! X)^\lp$ and quasi-coherent $\Sym^\star_{\mc O_{X}^{\!\lp}}T_{X\lp}$-modules on $X^\lp$. The map $\theta$ above extends to a map of sheaves of $\mc O_X^\lp$-algebras 
$$
\theta\colon \Sym^\star_{\mc O_X^{\!\lp}}\!T_X^{\lp}\ra Z(F_{X*}^\lp\mc D_X),
$$
and this way $F_{X*}^\lp\mc D_X$ defines a sheaf of algebras over $(T^{*}\!X)^\lp$.

Then, following \cite[Theorem~2.1]{ov} (see also \cite[Lemma~1.3.3 and Theorem~2.2.4]{bmr}) we have the following structural result for $F_{X*}^\lp\mc D_X$. 

\begin{pr}\label{prop:smooth Azumaya} Let $X$ be a smooth $S$-scheme. Then
	\begin{enumerate} 
		\item The $p$-curvature map
		$$
		\theta\colon \Sym^\star_{\mc O_{X}^\lp}\!T_{X^\lp}\ra Z(F_{X*}^\lp\mc D_X)
		$$
		is an isomorphism. 
		
		\item The sheaf of algebras over $(T^*\! X)^\lp$ defined by $F_{X*}^\lp\mc D_X$ is an Azumaya algebra of rank $p^{2\dim_S X}$. 
	\end{enumerate}
\end{pr}

\begin{example}\label{ex:D_A^n}
	Assume that $S=\Spec B$ is affine and that $X\coloneqq \mbb A^d_S$ with coordinates $z_1,\ldots ,z_d$. In this case global sections of $\mc D_X$ are given by the Weyl algebra $$W_d(B)\coloneqq B\langle z_1,\ldots,z_d, \partial_{z_1},\ldots,\partial_{z_d}\rangle/([z_i,z_j]=[\partial_{z_i},\partial_{z_j}]=0, [\partial_{z_j},z_i]=\delta_{ij}).$$ The $p$-th power $\partial_{z_i}^p$ acts by zero on $\mc O_X$, so $\partial_{z_i}^{[p]}=0$ and $\theta(\partial_{z_i})= \partial_{z_i}^p$. Proposition \ref{prop:smooth Azumaya} says that the center $Z(W_d(B))$ is isomorphic to the polynomial ring $B[z_i^p,\partial_{z_i}^p]\subset \mc D_X$ and that $\mc D_X$ is an Azumaya algebra of rank $p^{2d}$ over it.
	
\end{example}

We finish with some remarks that will be useful in the next section:
\begin{rem}\label{rem:splitting on the 0 section}
	Consider the homomorphism $\Sym^\star_{\mc O_{X}^\lp}\!T_{X^\lp}\twoheadrightarrow \mc O_{X}^\lp$ sending $T_X^\lp$ to 0. It corresponds to the embedding of the zero section $i_0\colon X^\lp \hookrightarrow (T^*\! X)^\lp$. Then the Azumaya algebra $\mc D_{X,0}\coloneqq i_0^*(F_{X*}^\lp\mc D_X)$ on $X^\lp$ is canonically split. Indeed, $F_{X*}^\lp\mc O_X$ is a vector bundle of rank $p^{\dim_S X}$ and $T_X^\lp\subset Z(F_{X*}^\lp\mc D_X)$ acts on it by 0: so the action of $F_{X*}^\lp\mc D_X$ on $F_{X*}^\lp\mc O_X$ factors through $\mc D_{X,0}$. Comparing the ranks we get that $F_{X*}^\lp\mc O_X$ is a splitting bundle for $\mc D_{X,0}$, and
	\begin{equation}\label{eq:central reduction of D_X at zero section}
	\mc D_{X,0}\simeq \mc End_{\mc O_X^\lp}(F_{X*}^\lp\mc O_X).
	\end{equation}
\end{rem}
\begin{rem}[Cartier's descent]\label{rem:equivalence for the 0 section}
	From \eqref{eq:central reduction of D_X at zero section} one recovers the classical Cartier descent (see \cite[Theorem~5.1]{katz70}): namely, the standard Morita-equivalence between the category of quasi-coherent sheaves of $\mc D_{X,0}$-modules and quasi-coherent sheaves on $X^\lp$ (quasi-coherent $\mc O_{X}^\lp$-modules). Functors establishing this equivalence are explicitly given by $\mc Hom_{\mc D_{X,0}}(F_{X*}^\lp\mc O_X,-)$ and $-\otimes_{\mc O_{X}^\lp}\! F_{X*}^\lp\mc O_X$. The first functor $\mc E\mapsto \mc Hom_{\mc D_{X,0}}(F_{X*}^\lp\mc O_X,\mc E)$ can also be viewed as taking ``flat sections'' of $\mc E$. Under this equivalence the $\mc D_{X,0}$-module $F_{X*}^\lp\mc O_X$ corresponds to $\mc O_X^\lp$. 
	
\end{rem}
\begin{rem}\label{rem:ideal invariant under derivations}
	Let $\mc I\subset \mc O_X$ be a quasicoherent sheaf of ideals that is invariant under the action of $\mc D_X$. 
    The pushforward $ F_{X*}^\lp(\mc I)$ gives a sub-$\mc D_{X,0}$-module of $ F_{X*}^\lp\mc O_X$ which by the equivalence in Remark \ref{rem:equivalence for the 0 section} is necessarily of the form $\mc J\otimes_{\mc O_X^\lp}  F_{X*}^\lp\mc O_X\subset  F_{X*}^\lp\mc O_X$ for some ideal $\mc J\subset \mc O_X^\lp$. Viewing quasi-coherent $ F_{X*}^\lp\mc O_X$-modules as quasi-coherent sheaves on $X$ the ideal $\mc J\otimes_{\mc O_X^\lp}  F_{X*}^\lp\mc O_X$ exactly corresponds to $(F_{X}^\lp)^*\mc J$, and so we get that $\mc I\simeq F_{X}^{\lp*}(\mc J)$ for some ideal $\mc J\subset \mc O_{X}^\lp$.
\end{rem}

\subsection{Local structure of a pseudo-smooth scheme}The goal of this section is to give a local description of a general pseudo-smooth scheme which will allow us to make a reduction to smooth case in the future. 
We start with the following lemma:

\begin{lm}\label{ex:Frobenius neighborhood is a pseudo-smooth scheme}
	Let $Y$ be a smooth $S$-scheme, and let $Z\subset Y^\lp$ be a subscheme. Then the preimage $(F_{Y}^\lp)^{-1}(Z)=Z\times_{Y^\lp} Y\subset Y$ is a pseudo-smooth scheme.
\end{lm}
\begin{proof}
	Let $\mc J\subset \mc O_Y^\lp$ be the sheaf of ideals defining $Z$. Then $\mc I\coloneqq F_{Y}^{\lp*}(\mc J)\subset \mc O_Y$ is the sheaf of ideals defining $X\coloneqq (F_{Y}^\lp)^{-1}(Z)$. Note that $\mc I$ is locally generated by a set of elements in $\mc O_Y^\lp\simeq (F_Y^{\lp})^{-1}\mc O_{Y\lp}$. Let $i\colon X\ra Y$ be the natural embedding. We have an exact sequence 
	$$
	\ldots \ra i^{-1}(\mc I/\mc I^2)\xra{d} i^*\Omega^1_Y \twoheadrightarrow \Omega^1_X \ra 0 
	$$
	of sheaves on $X$ where the map $d$ sends $[f]\in \mc I/\mc I^2$ to $df\in \Omega^1_Y$. However, for any $f\in \mc O_Y^\lp$ we have $df=0$. Since $\mc I$ is locally generated by elements in $\mc O_Y^\lp$, we get that the map $d$ is equal to 0, and so $\Omega^1_X\simeq i^*\Omega^1_Y$.
\end{proof}
\begin{rem}\label{rem:twisted forms}
	Since $i^*\Omega^1_Y\simeq \Omega^1_X$, we also get that $i^*\Omega^j_Y\simeq \Omega^j_X$ for all $j\ge 0$. By considering the commutative diagram 
	$$
	\xymatrix{X^\lp \ar[r]^{W_X^\lp} \ar[d]_{i^\lp}& X\ar[d]^i\\
		Y^\lp\ar[r]^{W_Y^\lp}& Y}
	$$
	it also follows that the natural map $i^{\lp*}\Omega^{j,\lp}_Y\ra \Omega^{j,\lp}_X$ induced by the identification $i^{\lp*}\circ W_Y^{\lp,*}\simeq W_X^{\lp,*}\circ i^*$ is an isomorphism for any $j\ge 0$.
\end{rem}

\begin{rem}\label{rem:reduced Frobenius in the case of Frobenius neighborhood}
	Note that for $X\coloneqq (F_{Y}^\lp)^{-1}(Z)$ the reduced Frobenius twist $X^\lp$ is given by  $Z$. Indeed, $\mc O_X$ (or rather $i_*\mc O_X$) is the reduction $\mc O_Y\otimes_{\mc O_Y^\lp} \mc O_Y^\lp\!/\mc J$, and $\mc O_X^{\lp}$ is the subalgebra in $\mc O_X$ generated by $\mc O_X^p$ and $\mc O_S$. The latter is easily seen to be isomorphic to $\mc O_Y^\lp\!/\mc J\simeq \mc O_Z$ and so the embedding $i^\lp\colon X^\lp\hookrightarrow Y^\lp$ is identified with $Z\hookrightarrow Y^\lp$. We also see that the commutative square
	$$
	\xymatrix{X \ar[r]^{F_X^\lp}\ar[d]_i & X^\lp\ar[d]^{i^\lp}\\
		Y\ar[r]^{F_Y^\lp} & Y^\lp
	}
	$$
	is in fact Cartesian. In particular, $F_X^\lp$ is a finite locally free map of degree $p^{\dim_S Y}$.
	
	Note also that $Z$ in Lemma \ref{ex:Frobenius neighborhood is a pseudo-smooth scheme} can be any closed subscheme of $Y^\lp$; in particular, singular or even non-reduced.
\end{rem}

In fact, all pseudo-smooth schemes look like Frobenius neighborhoods in smooth schemes locally:
\begin{pr}\label{prop:local description of pseudo-smooth}
	Let $X$	be a pseudo-smooth $S$-scheme. Then Zariski-locally on $X$ it is of the form $(F_{Y}^\lp)^{-1}(Z)=Z\times_{Y^\lp} Y\subset Y$ with $Y$ smooth over $S$ and $Z\hookrightarrow Y^\lp$ is a closed subscheme (see Lemma \ref{ex:Frobenius neighborhood is a pseudo-smooth scheme}).
\end{pr}
\begin{proof}
	Let $x\in X$ be a point. We will show that there is a neighborhood of $x$ of the form above.
	
	Since $X$ is locally of finite type, passing to an open we can assume that there is an embedding $i\colon X\hookrightarrow \mbb A^n:=\mathbb A^n_S$; the natural map $i^{*}\colon i^*\Omega^1_{\mbb A^n}\ra \Omega^1_X$ is a surjection. Let $\mc I\subset \mc O_{\mbb A^n}$ be the ideal defining $X$. As before, we have an exact sequence
	$$
	\ldots \ra i^{-1}(\mc I/\mc I^2)\xra{d} i^*\Omega^1_{\mbb A^n} \xra{i^{*}} \Omega^1_X \ra 0  .
	$$
	Since $X$ is pseudo-smooth, the kernel of $i^*$ is a vector bundle. Let $d\coloneqq \rk \Omega^1_X$. After passing to a neighborhood~{$U$} of $x$ in $\mbb A^n$ we can assume that $\Ker(i^*)$ is a trivial vector bundle with basis given by differentials $df_1,\ldots, df_{n-d}$ of $(n-d)$ functions $f_1,\ldots, f_{n-d}\in \mc I$. Consider the map $s\colon \mc O_U^{\oplus n-d}\ra \Omega^1_U$ defined by $(n-d)$ sections $df_1,\ldots, df_{n-d}$. By design this map is an embedding after pull-back to $x$, so shrinking $U$ we can assume that it is an embedding globally, and that the quotient is a trivial vector bundle. Let $df_1,\ldots, df_{n-d},dg_1,\ldots,dg_d$ (with $f_i$'s being the functions we considered before, and $g_i$'s some other functions) be a trivialization of $\Omega^1_U$. 
	
	Let $i_Y\colon Y\hookrightarrow U$ be the subscheme defined by the ideal $\mc I_Y\coloneqq (f_1,\ldots,f_{n-d})\subset  \mc I$. By the Jacobian criterion for \'etaleness the map $q\colon Y\ra \mbb A^d$ given by $g_1,\ldots,g_d$ is \'etale. 
	In particular, $Y$ is smooth over $S$. Let $\mc I'\subset \mc O_Y$ be the ideal defining $i_X\colon X\hookrightarrow Y$. We have an exact sequence
	$$
	\ldots \ra i^{-1}(\mc I'/\mc I'^2)\xra{d} i^*_X\Omega^1_{Y} \xra{i_X^{*}} \Omega^1_X \ra 0.
	$$
	However, note that the map $i_X^*$ is now an isomorphism. Indeed, $i_X^*\Omega^1_Y$ is exactly the quotient of $i^*\Omega^1_U$ by $df_i$, which are spanning the kernel of $i^*$. Thus, we get that the map $d$ is 0, or, in other words, that $d(\mc I')\subset \mc I'\Omega^1_{Y}$. This is equivalent to $\mc I'$ being invariant under the action of $\mc D_Y$, from which it follows (by Remark \ref{rem:ideal invariant under derivations}) that $\mc I'\simeq F_Y^{\lp*}\mc J$ for some $\mc J\subset \mc O_{Y^\lp}$. Then, setting $Z\subset Y^\lp$ to be the subscheme defined by $\mc J$ we exactly get that $X\simeq (F_Y^{\lp})^{-1}(Z)$. 
\end{proof}

\begin{rem}\label{rem:degree of reduced Frobenius}From the proof of Proposition \ref{prop:local description of pseudo-smooth} it follows that $F_X^\lp\colon X\ra X^\lp$ is a finite faithfully flat map of degree $p^{\rk \Omega^1_X}$. Indeed, it is locally obtained as the base change of relative Frobenius for a smooth scheme $Y$ of (relative) dimension $\rk \Omega^1_X$ (see Remark \ref{rem:reduced Frobenius in the case of Frobenius neighborhood}).
\end{rem}

We record the following stronger compatibility of the reduced Frobenius with base change in the pseudo-smooth setting (compare with Remark \ref{rem:twist vs basechange}):
\begin{lm}\label{lem:reduced Frobenius is respected by pull-back} Let $X/S$ be a pseudo-smooth $S$-scheme, let $f\colon S'\to S$ be a morphism of schemes, and let $q\colon X'/S':= X\times_S S'/S'\to X$ be the fiber product with $X$. Then all squares in the commutative diagram
$$
\begin{tikzcd}
    X' \arrow[d,"q"'] \arrow[r,"F_{X'}^\lp"]
    & (X')^{\lp}
      \arrow[d,"q^{\lp}"]
      \arrow[r]
    & S' \arrow[d,"f"']
    \\
    X \arrow[r,"F_X^\lp"]
    & X^{\lp} \arrow[r]
    & S
\end{tikzcd}
$$
are Cartesian.    
\end{lm}
\begin{proof}
    One can either reduce to the smooth case (and usual Frobenius twist) using Proposition \ref{prop:local description of pseudo-smooth}, or argue directly as follows. The outer square is Cartesian by definition. Since $F_X^\lp$ is a flat cover, it is enough to see that the left square is Cartesian. As $\Omega^1_{X'/S'}=q^*\Omega^1_{X/S}$, and their ranks are the same, by Remark \ref{rem:degree of reduced Frobenius} both horizontal arrows in this square are finite locally free maps of the same degree. Moreover, one checks locally that the map $X'\to (X')^{\lp}\times_{X^\lp}X$ is a closed embedding: given a $B$-algebra $A$ and a map $B\to B'$ the natural map 
    $$
A\otimes_{A^{\lp}}(A\otimes_B B')^{\lp}\to A\otimes_B B'
    $$
    is surjective. Since it is also a map of locally free $(A\otimes_B B')^{\lp}$-modules of the same rank, it is an isomorphism.
\end{proof}

We also have the following strong multiplicativity statement for the reduced Frobenius twist.

\begin{lm}\label{lem:multiplicativity}Let $X_1,X_2$ be two pseudo-smooth $S$-schemes. Then the natural map 
$$
\iota\colon (X_1\times_S X_2)^\lp \to X_1^\lp\times_S X_2^\lp
$$
from Remark \ref{rem:twist vs product} is an isomorphism.     
\end{lm}
\begin{proof}
    It is enough to assume that all schemes are affine, so we have two $B$-algebras $A_1, A_2$ such that $\Spec A_i$ is pseudo-smooth over $B$. By Remark \ref{rem:degree of reduced Frobenius} $A_i$ is a finite locally free module over $C_i:=A_i^p\cdot B\subset A_i$ for $i=1,2$. It follows that the embedding of $C_i$-modules (and, consequently, $B$-modules) $C_i\hookrightarrow A_i$ is split injective. Consequently, the map of $B$-modules $C_1\otimes_BC_2 \to A_1\otimes_BA_2$ is also injective. The map $\iota$ is given by the multiplication map
    $$
C_1\otimes_BC_2 \xra{\iota^*} (A_1\otimes_B A_2)^p\cdot B\subset A_1\otimes_BA_2,
    $$
    and it is clear $\iota^*$ is surjective (this part is true without the pseudo-smoothness assumption). As the composite map to $A_1\otimes_BA_2$ is injective we get that $\iota^*$ is a bijection.  
\end{proof}

Proposition \ref{prop:local description of pseudo-smooth} allows to easily extend the discussion in Section \ref{ssec:differential operators smooth} to the pseudo-smooth setting. The following generalizes Lemma \ref{lm:embedding of diff operators of order p-1}:

\begin{cor}\label{cor:embedding of diff operators of order p-1} Let $X$ be a pseudo-smooth scheme over $S$. Then the natural map 
	$$\mc D_{X,\le p-1}\rightarrow \mc End_{\mc O_S}(\mc O_X)
	$$
	induced by the action of $\mc D_X$ on $\mc O_X$ is an embedding.
\end{cor}
\begin{proof}
	The statement is Zariski-local, so we can assume that $X=(F_Y^\lp)^{-1}(Z)\hookrightarrow Y$ for a closed subscheme $Z\hookrightarrow Y^\lp$ and $Y$ is smooth. We can also assume that $Y$ has an \'etale map $f\colon Y\ra \mbb A^n$. Let $\partial_i$ and $f_i$ be as in the proof of Lemma \ref{lm:embedding of diff operators of order p-1}. Let $i\colon X\ra Y$ be the embedding. Recall (see the proof of Lemma \ref{ex:Frobenius neighborhood is a pseudo-smooth scheme}) that $i^*\Omega^1_Y\xra{\sim} \Omega^1_X$, and, consequently, $T_X\simeq i^*T_Y$. Let $\partial_i'\coloneqq {(\partial_i)}_{|X}\in T_X$ and let $f_i'\coloneqq (f_i)_{|X}\in\mc O_X$. We have $\partial_i'(f_j')=\delta_{ij}$. The same argument as in Lemma \ref{lm:embedding of diff operators of order p-1} then shows that for any non-zero $D\in \mc D_{X,\le p-1}$ there is a function $g\in \mc O_X$ such that $D(g)\neq 0$.
\end{proof}	

\begin{constr}[$p$-curvature map in the pseudo-smooth setting]As in the smooth case, consider the map $
	\theta\colon T_X \ra \mc D_X
	$ of sheaves on $X$ given by the following formula
	$$
	v\mapsto \theta(v)\coloneqq v^p-v^{[p]}.
	$$
	Using Corollary \ref{cor:embedding of diff operators of order p-1} in place of Lemma \ref{lm:embedding of diff operators of order p-1} as in Construction \ref{constr:theta} one shows that
	\begin{itemize}
		\item $\theta$ is additive: $\theta(v_1+v_2)-(\theta(v_1)+\theta(v_2))=0$;
		\item $\theta$ is Frobenius-semilinear: $\theta(fv)-f^p\cdot \theta(v)=0$;
		\item $\theta(v)$ lies in the center $Z(\mc D_X)$ for any $v\in T_X$: $[\theta(v),f]=[\theta(v),w]=0$ for any $f\in \mc O_X$ and $w\in T_X$.
	\end{itemize}
	
	This way we obtain a map 
	$$
	\theta'\colon T_X\otimes_{\mc O_X}\mc O_X^p \ra Z(\mc D_X)
	$$
	of $\mc O_X^p$-modules. Since $Z(\mc D_X)$ is in fact an $\mc O_X^\lp\simeq \mc O_X^p\cdot \mc O_S$-algebra we can extend this to a map
	$
	T_X\otimes_{\mc O_X}\mc O_X^\lp\ra  Z(\mc D_X)
	$ of quasi-coherent sheaves on $X^\lp$
	\begin{equation}\label{eq:p-curvature map 2}
		\theta\colon T_X^\lp \ra Z(F_{X*}^\lp\mc D_X),
	\end{equation}
	which we continue to call the ``$p$-curvature map''.
	
\end{constr}

\begin{cor}\label{cor:D_X is an Azumaya algebarain pseudo-smooth setting} Let $X$ be a pseudo-smooth scheme over $S$.
	As before, consider $(T^*\! X)^\lp\simeq \Spec_{X^\lp} (\Sym^\star_{\mc O_X^\lp}\!T_X^\lp)$. 
	
	\begin{enumerate}
		\item The map $\theta\colon T_X^\lp \ra Z(F_{X*}^\lp\mc D_X)$ induces an isomorphism
		$$
		\theta\colon \Sym^\star_{\mc O_X^\lp}\!T_X^\lp\xra{\sim} Z(F_{X*}^{\lp}\mc D_X).
		$$
		
		\item $F_{X*}^{\lp}\mc D_X$, considered as a sheaf of algebras over $(T^*\! X)^\lp$, is an Azumaya algebra of rank\footnote{In general this is a locally constant function on $X$.} $p^{2\rk \Omega^1_X}$.
	\end{enumerate}
	
\end{cor}
\begin{proof}
	Both statements are Zariski-local and so we can assume $X\simeq (F_Y^{\lp})^{-1}(Z)$ where $Y$ is smooth over $S$ and $Z\hookrightarrow Y^\lp$ is a closed subscheme. In this case, $\mc D_X$ is a central reduction of $\mc D_Y$ via an embedding $T^*Y\twt|_Z\hookrightarrow T^*Y\twt\simeq (T^*Y)^\lp$. More precisely, let $i\colon X\hookrightarrow Y$ be the embedding. Following the proof of Lemma \ref{ex:Frobenius neighborhood is a pseudo-smooth scheme} we have $\Omega^1_X\simeq i^*\Omega^1_Y$ and, consequently, $T_X\simeq i^*T_Y$.
	The latter is in fact an isomorphism of Lie algebroids, and induces an isomorphism $\mc D_X\simeq i^*\mc D_Y$. Also, $$
	\Sym^\star_{\mc O_X^\lp}\!T_X^\lp\simeq \Sym^\star_{\mc O_Y^\lp}\!T_Y^\lp\otimes_{\mc O_Y^\lp}\mc O_Z
	$$ and so $(T^*\! X)^\lp\ra X^\lp$ can be identified with the restriction of $(T^*Y)^\lp\ra Y^\lp$ along $X^\lp\simeq Z\hookrightarrow Y^\lp$. Moreover, if we denote by $i_Z\colon Z\hookrightarrow Y^\lp$ the corresponding embedding, applying $F_{X*}$ we get
	$$
	F_{X*}^\lp\mc D_X\simeq F_{X*}^\lp(i^*\mc D_Y)\simeq i_Z^{*}F_{Y*}^\lp\mc D_Y\simeq F_{Y*}^\lp\mc D_Y\otimes_{\mc O_Y^\lp}\mc O_Z\simeq F_{Y*}^\lp\mc D_Y\otimes_{\mc O_{(T^*Y)^\lp}}\mc O_{(T^*\! X)^\lp},
	$$
	where in the last isomorphism we view $F_{Y*}^\lp\mc D_Y$ as an Azumaya algebra over $(T^*Y)^\lp$. Thus we get that $F_{X*}^\lp\mc D_X$ as a sheaf of algebras over $(T^*\! X)^\lp$ is a restriction of an Azumaya algebra $F_{Y*}^\lp\mc D_Y$, and as such is itself an Azumaya algebra. Its rank is equal to $p^{2\dim_S Y}= p^{2\rk \Omega^1_X}$.
\end{proof}

\subsection{De Rham complex and Cartier operation}

Let $X$ be a pseudo-smooth scheme over $S$. Put $r\coloneqq \rk \Omega^1_X$; this is a locally constant function on $X$. Consider the de Rham complex 
$$ \dR_X\coloneqq \left[ 0\ra \mc O_X\ra \Omega^1_X \ra \ldots \ra \Omega^r_X\ra 0\right]{:} $$
this is a complex of sheaves of $\mc O_S$-modules on $X$ with differential given by the (relative) de Rham differential. Since the (relative) de Rham differential is $\mc O_X^p\cdot \mc O_S$-linear, the pushforward
$
F_{X*}^\lp \dR_X
$
defines a complex of quasi-coherent sheaves on $X^\lp$ where each term is finite locally free. 

Let $\mc H_{\dR,X}^\star\coloneqq \mc H^\star(F_{X*}^\lp\dR_X)$ denote the sheaf cohomology algebra of $F_{X*}^\lp\dR_X$. Recall the notation $\Omega^{1,\lp}_X\coloneqq W_X^{\lp*}\Omega^1_X$. One can define the inverse Cartier map
$$
C^{-1}\colon \wedge^\star_{\mc O_{X}^\lp}\!\! (\Omega^{1,\lp}_X)  \ra \mc H_{\dR,X}^\star
$$ as follows.  Namely, in degree 0 it is given by $(F_{X}^\lp)^{*}\colon \mc O_{X^\lp} \ra F_{X*}^\lp\mc O_X$, while in degree 1 it sends $(W_{X}^{\lp})^{-1}(df)\in \Omega^{1,\lp}_X$ to $[f^{p-1}df]\in \mc H_{\dR,X}^1$. By a standard computation the latter gives a well-defined $\mc O_{X^\lp}$-linear map to $\mc H^1_{\dR,X}$. In higher degrees, $C^{-1}$ uniquely extends by multiplicativity. 

The classical result of Cartier (see e.g. 
\cite[Theorem~7.2]{katz70}) states that $C^{-1}$ is an isomorphism if $X$ over $S$ is smooth. This remains true in the pseudo-smooth setting.
\begin{pr}\label{prop: Cartier isom}
	Let $X$ be a pseudo-smooth $S$-scheme. Then the inverse Cartier map is an isomorphism. 
\end{pr}
\begin{proof}
	We deduce the statement from the smooth case. The statement is Zariski local, so we can assume $X=(F_Y^\lp)^{-1}(Z)$, as in Lemma \ref{ex:Frobenius neighborhood is a pseudo-smooth scheme}. In this case $X^\lp\simeq Z$. Let $i\colon X\hookrightarrow Y$ and $i^\lp\colon Z\hookrightarrow \mc Y^\lp$.  By Remarks \ref{rem:reduced Frobenius in the case of Frobenius neighborhood}, \ref{rem:twisted forms} and base change, $F_{X*}^\lp\dR_X$ can be naturally identified with $i^{\lp*}F_{Y*}^\lp\dR_Y$. By Cartier isomorphism in the smooth case, the cohomology sheaves $\mc H_{\dR,Y}^\star$ are isomorphic to $\wedge^\star_{\mc O_{Y}^\lp}\Omega^{1,\lp}_Y$; in particular, they are vector bundles and thus are flat over $\mc O_Y^\lp$. From this and Remark \ref{rem:twisted forms} we get that 
	\[\mc H_{\dR,X}^\star\simeq \mc H^\star(i^{\lp*}F_{Y*}^\lp\dR_Y)\simeq i^{\lp*}\mc H_{\dR,Y}^\star\simeq i^{\lp*}(\wedge^\star_{\mc O_{Y}^\lp}\! \Omega^{1,\lp}_Y)\simeq \wedge^\star_{\mc O_{X}^\lp}\Omega^{1,\lp}_X.\qedhere\]
\end{proof}
Let $C$ be the inverse to $C^{-1}$. In each degree $i$, $C$ gives an isomorphism $\mc H^i_{\dR,X}\simeq \Omega^{i,\lp}_X$ which we can precompose with the projection $F_{X*}^\lp\Omega^i_{X,cl}\twoheadrightarrow \mc H^i_{\dR, X}$ from the closed $i$-forms. We will call the resulting map $\mathsf{C}^i\colon F_{X*}^\lp\Omega^i_{X,cl}\twoheadrightarrow \Omega^{i,\lp}_X$ the \textit{Cartier operation}.
For $i=1$ one gets a short exact sequence
\begin{equation}\label{eq:Cartier ses}
	0\ra F_{X*}^\lp\mc O_{X}/\mc O_{X}^\lp \xra{d} F_{X*}^\lp\Omega^1_{X,cl}\xra{\mathsf{C}} \Omega^{1,\lp}_X \ra 0,
\end{equation}
of sheaves on $X^\lp$ where the term on the left is identified with the sheaf of exact 1-forms via~$d$.
\begin{rem}\label{rem:pairing with Cartier}
	The map $\mathsf C\colon F_{X*}^\lp\Omega^1_{X,cl}\ra \Omega^{1,\lp}_X$ dually induces a map $T_X^\lp\otimes F_{X*}^\lp\Omega^1_{X,cl}\ra \mc O_X^\lp$ which lifts the pairing map $T_X^\lp\otimes \Omega^{1,\lp}_X\ra \mc O_X^\lp$. Substituting a local section $\xi\in T_X^\lp$ we obtain the corresponding ``contraction'' maps 
	$$
	\widetilde{\iota_\xi}\colon F_{X*}^\lp\Omega^1_{X,cl}\ra \mc O_X^\lp \ \ \text{ and } \ \ \iota_\xi\colon \Omega^{1,\lp}_X\ra \mc O_X^\lp.
	$$
	Following \cite[Lemma 2.1]{bk}, we have the following explicit formula for $\widetilde{\iota_\xi}$. Namely, given a closed 1-form $\alpha\in \Omega^1_{X,cl}$ and a vector field $\xi\in T_X$ (giving a section $\xi^{\lp}\coloneqq (W_X^\lp)^{-1}(\xi)\in T_X^{\lp}$) we have
	\begin{equation}\label{eq:formula for the Cartier operation}
		\widetilde{\iota_{\xi^\lp}}(\alpha)\coloneqq \iota_{\xi^\lp}(\mathsf C(\alpha))=\iota_{\xi^{[p]}}\alpha - \xi^{p-1}(\iota_\xi\alpha)\in \mc O_X^\lp\subset \mc O_X.
	\end{equation}
 
\end{rem}

\subsection{The relative dualizing sheaf of the reduced Frobenius morphism}\label{ss:relative dualizing sheaf of the Frobenius morphism}
As an application of the Cartier isomorphism we prove the following.
\begin{pr}\label{prop:relativedualizaingsheaf}
	Let $X$ be a pseudo-smooth $S$-scheme, $d= \rk \Omega^1_{X}$.  Then the relative dualizing sheaf for the Frobenius morphism $F_X^\lp\colon X \to X^\lp$  is given by 
	$(\Omega^d_{X})^{\otimes 1-p}$.
	\end{pr}
	\begin{proof} Since $F_X^\lp$ is finite flat it is enough to construct  an isomorphism of  $F_{X*}^\lp\cO_X$-modules
		 \begin{equation}\label{app.quantofcotangentandserregeneral}
 F_{X*}^\lp ((\Omega^d_{X})^{\otimes 1-p}) \iso \mathcal{H}\! \mathit{om}_{\cO_{X^\lp}}(F_{X*}^\lp\cO_X, \cO_{X^\lp}).
 \end{equation} 
	Define the trace map 
 $$\tr:   F_{X*}^\lp ((\Omega^d_{X})^{\otimes 1-p}) \to \cO_{X^\lp}$$
 to be the composition 
 $$F_{X*}^\lp ((\Omega^d_{X})^{\otimes 1-p}) \iso (\Omega^{d, \lp}_{X})^{\otimes -1}\otimes  F_{X*}^\lp \Omega^d_{X}\to (\Omega^{d, \lp}_{X})^{\otimes -1} \otimes \mc H_{\dR,X}^d \iso \cO_{X^\lp}.$$
 Here the first and the last isomorphisms  come from the projection formula (using the isomorphism $F_{X}^{\lp *} \Omega^{d, \lp}_{X}\iso (\Omega^{d}_{X})^{\otimes p}$)  and  Proposition \ref{prop: Cartier isom} respectively.   The morphism in  the middle is given by the projection from the sheaf of top degree forms to the de Rham cohomology sheaf. Define a morphism of $F_{X*}^\lp\cO_X$-modules
 	 \begin{equation}\label{app.quantofcotangentandserregeneralbis}
 F_{X*}^\lp ((\Omega^d_{X})^{\otimes 1-p}) \to \mathcal{H}\! \mathit{om}_{\cO_{X^\lp}}(F_{X*}^\lp\cO_X, \cO_{X^\lp})
 \end{equation} 
sending a local section $a$ of $F_{X*}^\lp ((\Omega^d_{X})^{\otimes 1-p}) $ and a local section $f$ of   $F_{X*}^\lp\cO_X$ to $\tr(fa)$. We claim that (\ref{app.quantofcotangentandserregeneralbis}) is an isomorphism. Indeed,
(\ref{app.quantofcotangentandserregeneralbis})  is a morphism of vector bundles. Thus, it suffices to check  the  isomorphism property on fibers.  Using
   Proposition \ref{prop:local description of pseudo-smooth} and base change we reduce the problem to $X=\Spec\bF_p[x_1, \cdots, x_d]/(x_i^p, 1\leq i \leq d)$, $S =\Spec \mbb F_p$, where it is verified directly (see also Example \ref{ex:trace map for Frobenius neighborhood} below). 
 \end{proof}

 \begin{rem}\label{rem:invariance of trace} For every section $\xi$ of $T_{X}$, we have that 
  \begin{equation}\label{eq:invariance of trace}
 \tr \circ L_\xi =0.
  \end{equation} 
  Indeed, for a local generator $\mu$ of  $\Omega^d_{X/S}$,
 $$(\tr \circ L_\xi)(f \mu^{\otimes 1-p})= \tr (\mu^{\otimes -p} d\iota_\xi(f \mu))=0$$
 because the Cartier morphism vanishes on exact forms.   Observe that the $\cO_X$-module $(\Omega^d_{X})^{\otimes 1-p}$ has a right $\cD_X$-module structure:
 the right action of  $\xi$  is given by $-L_\xi$. (More generally, a left $\cD_X$-module structure on $M$ yields a right $\cD_X$-module structure on $M\otimes_{\cO_X} \Omega^d_{X}$. In particular, this can be applied to 
 $M=(\Omega^d_{X})^{\otimes -p}\iso F_{X}^{\lp *}( (\Omega^{d, \lp}_{X})^{\otimes -1})$ with its Frobenius descent connection.) Also,  the left action of $\cD_X$ on $\cO_X$ makes 
 $ \mathcal{H}\! \mathit{om}_{\cO_{X^\lp}}(F_{X*}^\lp\cO_X, \cO_{X^\lp})$ a right $F_{X*}^\lp\cD_X$-module. Equation (\ref{eq:invariance of trace}) implies that 
 (\ref{app.quantofcotangentandserregeneral}) is an isomorphism of right $F_{X*}^\lp\cD_X$-modules.
  \end{rem}

   \begin{example}\label{ex:trace map for Frobenius neighborhood} Let $C:=\bF_p[x_1, \cdots, x_d]/(x_i^p, 1\leq i \leq d)$ and $X=\Spec C$, for $S =\Spec \mbb F_p$.  In this case the map 
 $$
\tr\colon (\Omega^d_C)^{\otimes 1-p} \to \mathbb F_p
 $$
 is described explicitly as follows. We have a trivialization $\Omega^d_C\simeq C\cdot (dx_1\wedge\ldots \wedge dx_d)$ and, consequently,  $(\Omega^d_C)^{\otimes 1-p}\simeq C\cdot (dx_1\wedge\ldots \wedge dx_d)^{\otimes 1-p} $. Then $\tr(f\cdot (dx_1\wedge\ldots \wedge dx_d)^{\otimes 1-p})$ is given by the coefficient of $f$ at the monomial $x_1^{p-1}\!\cdots x_d^{p-1}\in C$. 
 \end{example}

\subsection{Line bundles with connection and Milne's exact sequence}
Let $(\mc E,\nabla)$, with $\nabla\colon \mc E\ra \mc E\otimes \Omega^1_X$ be a quasi-coherent sheaf with flat connection on a pseudo-smooth $S$-scheme $X$. The connection $\nabla$ induces a map $\mc D_X\ra \mc End_{\mc O_S}(\mc E)$, whose restriction to $Z(\mc D_X)$ factors through the subalgebra of $\mc O_X$-linear endomorphisms  $\mc End_{\mc O_X}(\mc E)^{\nabla'=0}\subset \mc End_{\mc O_S}(\mc E)$ that are flat with respect to the natural connection\footnote{Here the connection $\nabla'$ on $\mc End_{\mc O_X}(\mc E)$ is given by $(\nabla'(\psi))(s)\coloneqq \nabla(\psi(s))-\psi(\nabla(s))$ for $s\in \mc E$ and $\psi\in \mc End_{\mc O_X}(\mc E)$. Thus an endomorphism $\psi\in \mc End_{\mc O_X}(\mc E)$ is flat with respect to $\nabla'$ if and only if it commutes with the action of $T_X$ induced by $\nabla$.} on $\mc End_{\mc O_X}(\mc E)$. Indeed, central elements commute both with multiplication on functions and action by vector fields.

\begin{constr}[{$p$-curvature map for line bundles}]\label{constr:p-curvature map line bundles} Let $\mc E=\mc L$ be a line bundle. In this case, the endomorphism sheaf $(\mc End_{\mc O_X}(\mc L),\nabla')$ is given by $(\mc O_X,d)$. In particular, $\mc End_{\mc O_X}(\mc L)^{\nabla'=0}\subset \mc O_X$ is given by $\ker d$, which is identified with $\mc O_X^\lp\subset \mc O_X$ by Proposition \ref{prop: Cartier isom}. Using the $p$-curvature map (\ref{eq:p-curvature map 2}) we obtain a map of sheaves $T_X^\lp\ra \mc O_X^\lp$, or, in other words, a section of $\Omega^{1,\lp}_X$, which we will denote $\mathsf{c}_p(\mc L,\nabla)\in \Omega^{1,\lp}_X$. It is not hard to see from the definition (and the way $T_X^\lp\subset Z(F_{X*}^\lp\mc D_X)$ acts on the tensor product), that 
	$$
	\mathsf{c}_p((\mc L_1,\nabla_1)\otimes (\mc L_2,\nabla_2))=\mathsf c_p(\mc L_1,\nabla_1)+\mathsf c_p(\mc L_2,\nabla_2). 
	$$
	
	If we are given a trivialization $\mc L\simeq \mc O_X$, $\nabla$ can be written in the form $d+\alpha$ for some closed 1-form $\alpha\in \Omega^1_{X,cl}$. Restricting to line bundles of this form we get a map of sheaves
	$$
	\mathsf c_p\colon F_{X*}^\lp \Omega^1_{X,cl}\ra \Omega^{1,\lp}_X, \ \ \ \ \mathsf c_p\colon \alpha \mapsto \mathsf c_p(\mc O_X,d+\alpha).
	$$
	Since $(\mc O_X,d+\alpha_1)\otimes (\mc O_X,d+\alpha_2)\simeq (\mc O_X,d+\alpha_1+\alpha_2)$ the above map is a map of sheaves of abelian groups.
\end{constr}

\begin{rem}\label{rem:katzformula}
	One has an explicit formula for the map $\mathsf c_p$ in terms of the Cartier operation. Namely, given $\alpha\in \Omega^1_{X,cl}$ let's compute the pairing $\langle \xi^\lp,\mathsf c_p(\alpha)\rangle$ for any $\xi\in T_X$, $\xi^\lp\coloneqq (W_X^\lp)^{-1}(\xi)\in T_X^\lp$: this defines $\mathsf c_p(\alpha)$ uniquely. By definition, we need to compute the action of $\theta(\xi)=\xi^p-\xi^{[p]}$ on line bundle $(\mc O_X,d+\alpha)$. This is given by
	$$
	(\xi + \langle \xi,\alpha\rangle)^p - \xi^{[p]}-\langle \xi^{[p]},\alpha\rangle = \langle \xi,\alpha\rangle^p-  \langle \xi^\lp,\mathsf C(\alpha)\rangle \in \mc End_{\mc O_S}(\mc O_X)
	$$ where for the equality we used that $\xi^p-\xi^{[p]}$ acts by 0 on $\mc O_X$, formula (\ref{eq:formula for the Cartier operation}) and also Jacobson's formula. Note that $\langle \xi,\alpha\rangle^p=(W_X^\lp)^{-1}(\langle \xi,\alpha\rangle)=\langle \xi^{\lp},\alpha^\lp\rangle$, where $\alpha^\lp\coloneqq (W_X^\lp)^{-1}\alpha$. This gives
	$$
	\mathsf c_p(\alpha)=\alpha^\lp -\mathsf C(\alpha).
	$$
\end{rem}


\begin{constr}[Central reduction of $\mc D_X$ associated to a 1-form]\label{constr:classical central reductions associated to 1-forms}
	Let $\alpha\in H^0(X^\lp, \Omega^{1,\lp}_X)$ be a global section. Let $i_\alpha\colon X^\lp \hookrightarrow (T^*\! X)^\lp$ be the graph of $\alpha$. Under the identification $(T^*\! X)^\lp\simeq \Spec_{X^\lp} (\Sym^\star_{\mc O_X^\lp}\! T_X^\lp)$, the image $\Gamma_\alpha=\im(i_\alpha)\subset (T^*\! X)^\lp$ of $i_\alpha$ is defined by the sheaf of ideals generated by linear expressions 
	\begin{equation}\label{eq:relations for central reduction}\mc I_\alpha\coloneqq (\xi- \langle\xi,\alpha\rangle)_{\xi\in T_X^\lp}\subset \Sym^\star_{\mc O_X^\lp} \!T_X^\lp.
    \end{equation}We denote by $\mc D_{X,\alpha}\coloneqq i_\alpha^*(F_{X*}^\lp\mc D_X)$ the corresponding pull-back of $F_{X*}^\lp\mc D_X$ considered as a sheaf of algebras over $(T^*\! X)^\lp$. For any $\alpha$ the sheaf $\mc D_{X,\alpha}$ defines an Azumaya algebra over $X^\lp$ of rank $p^{2\rk \Omega^1_X}$.

    The sheaf $F_{X*}^\lp\mc O_X\subset F_{X*}^\lp\mc D_X$ naturally acts (from the left) on $\mc D_{X,\alpha}$, and the latter is a locally free module of rank $p^{\rk \Omega^1_X}$. Indeed, the statement is local, and we can assume that $X$ is a Frobenius-neighborhood of a closed subscheme in a smooth $S$-scheme $Y$, and $\alpha$ is a restriction of a $1$-form on $Y^{(1)}$, this way reducing to the case when $X$ is smooth over $S$ of relative dimension $n\coloneqq \rk(\Omega^1_{X})$. Then, choosing coordinates $\{z_i\}$ \'etale locally, $\mc D_{X,\alpha}$ is a free $F_{X*}^\lp\mc O_X$-module with basis $\partial^\alpha=\prod_{i=1}^n\partial_{z_i}^{\alpha_i}$ where $0\le \alpha_i \le p-1$. Indeed, we have that $F_{X*}^\lp\mc D_X$ is a free $F_{X*}^\lp\mc O_X$-module with basis $\partial^\alpha$ for all $\alpha=(\alpha_1,\ldots,\alpha_n)$ with $\alpha_i\ge 0$, and the relations in $\cI_{\alpha}$ become $\partial_{z_i}^p-\langle \partial_{z_i}^\lp,\alpha\rangle $.
\end{constr}

\begin{constr}[Line bundles with connection vs splittings.]\label{constr:line bundles vs splittings} In the smooth case see \cite[\S3.10]{bb}. Suppose we are given a line bundle with flat connection $(\mc L,\nabla)$ and put $\alpha\coloneqq \mathsf c_p(\mc L,\nabla)$. Then by definition of $\mathsf c_p$ we have that $\xi\in T_{X}^\lp$ acts on $F_{X*}^\lp\mc L$ by $\langle\xi,\alpha\rangle\in \mc O_X^\lp$. It follows that the action of $F_{X*}^\lp\mc D_X$ on $F_{X*}^\lp\mc L$ induced by $\nabla$ factors through $\mc D_{X,\alpha}$. Moreover, by Remark \ref{rem:degree of reduced Frobenius} $F_{X*}^\lp\mc L$ is a vector bundle of rank $p^{\rk \Omega^1_X}$ on $X^\lp$ and thus is a splitting bundle for $\mc D_{X,\alpha}$: namely, $\nabla$ necessarily induces an isomorphism 
	$$
	\nabla\colon \mc D_{X,\alpha} \xra{\sim} \mc End_{\mc O_X^\lp}(F_{X*}^\lp\mc L).
	$$ Vice versa, let $\mc E$ be a bundle on $X^\lp$ that splits $\mc D_{X,\alpha}$. Then $\mc E$ is a vector bundle of rank $p^{\rk \Omega^1_X}$. However $\mc E$ is also a quasi-coherent $F_{X*}^{\lp}\mc O_X$-module (via the composition $
	F_{X*}^\lp \mc O_X\ra F_{X*}^\lp \mc D_X \twoheadrightarrow\mc D_{X,\alpha})$, and so comes as the pushforward $F_{X*}^\lp\mc F$ of a quasi-coherent sheaf $\mc F$ on $X$. Via a local rank computation one can show that $\mc F$ is a line bundle\footnote{Indeed, this can be checked Zariski locally, so we can assume vector bundle $\mc E$ is trivial. Then  $\mc D_{X,\alpha}\simeq \mc End(\mc E)\simeq \mc E^{\oplus p^{\rk \Omega^1_{X}}}$ as a left module over itself, and consequently as a quasicoherent $F_{X*}^\lp \mc O_X$-module. By the discussion in Construction \ref{constr:classical central reductions associated to 1-forms}, $\mc D_{X,\alpha}$ is a locally free module  over $F_{X*}^\lp \mc O_X$ of rank $p^{\rk \Omega^1_{X}}$. This forces each summand $\mc E$ to be locally free of rank 1.}.  The action of $F_{X*}^\lp\mc D_X$ on $F_{X*}^\lp\mc F$ induces a natural flat connection on $\mc F$, whose $p$-curvature is given by $\alpha$. Moreover, isomorphisms between two line bundles $(\mc L_1,\nabla_1)$ and $(\mc L_2,\nabla_2)$ with the same $p$-curvature $\alpha$ exactly correspond to isomorphisms between the corresponding $\mc D_{X,\alpha}$-modules $F_{X*}^\lp\mc L_1$ and $F_{X*}^\lp\mc L_2$. In particular, any two such line bundles are \'etale locally isomorphic.
	
	This gives an equivalence of two groupoids: of line bundles with flat connection on $X$ with $p$-curvature $\alpha$, and of splitting bundles\footnote{An object of the groupoid of splitting bundles is given by a vector bundle $\mc E$ on $X^\lp$ together with an isomorphism $\gamma\colon \mc D_{X,\alpha}\simeq \mc End_{\mc O_X^\lp}(\mc E)$; morphisms are isomorphisms of vector bundles that preserve this data.} for $\mc D_{X,\alpha}$ on $X^\lp$ correspondingly.
\end{constr}

Consider the map $d\log\colon \mc O_X^\times \ra \Omega^1_{X,cl}$ of \'etale sheaves on $X$ sending $f$ to $df/f$. It induces a map of \'etale sheaves $F_{X*}^\lp\mc O_X^\times \ra F_{X*}^\lp\Omega^1_{X,cl}$ on $X^\lp$.
\begin{pr}\label{prop:Milne's exact sequence}
	There is a short exact sequence of \'etale sheaves on $X^\lp$
	$$
	0\ra F_{X*}^\lp\mc O_X^\times/\mc O_{X^\lp}^\times\xra{d\log} F_{X*}^\lp\Omega^1_{X,cl} \xra{\mathsf c_p} \Omega^{1,\lp}_X \ra 0, 
	$$
	 where $\mathsf c_p$ is the map from Construction \ref{constr:p-curvature map line bundles}.
\end{pr}
\begin{proof} For the smooth case, see \cite[(4.1.1)]{ov},
which refers to \cite[4.14]{Milne}.
	We use the equivalence of groupoids described in Construction \ref{constr:line bundles vs splittings}. First of all, the kernel of $d\log \colon F_{X*}^\lp\mc O_X^\times \ra F_{X*}^\lp\Omega^1_{X,cl}$ is given by the intersection of  $F_{X*}^\lp\mc O_X^\times$ and $\ker d\subset F_{X*}^\lp\mc O_X$, which is exactly $\mc O_{X^\lp}^\times$. This shows exactness on the left. To show that $\mathsf c_p$ is surjective, for any local section $\alpha\in \Omega^{1,\lp}_X$ it is enough to show that \'etale locally on $X$ there is a line bundle with $p$-curvature given by (the corresponding pull-back {of)} $\alpha$. By the above this is equivalent to finding a splitting bundle for the Azumaya algebra $\mc D_{X,\alpha}$, which is always possible \'etale locally. It remains to show the exactness in the middle. Let $\alpha_1,\alpha_2$ be two closed forms on an \'etale open $U\ra X$ such that $\mathsf c_p(\alpha_1)=\mathsf c_p(\alpha_2)$. Then by the discussion in Construction \ref{constr:line bundles vs splittings} line bundles with connection $(\mc O_U,d+\alpha_1)$  and $(\mc O_U,d+\alpha_2)$ are locally isomorphic, which means\footnote{Indeed, if  $(\mc O_U,\nabla)$ and $(\mc O_U,\nabla')$ are isomorphic, the isomorphism is given by multiplication by an invertible function $f$, and then $\nabla'=\nabla-d\log f$.} that (after some cover) $\alpha_1=\alpha_2-d\log f$.
\end{proof}
\subsection{Restricted Poisson structures}\label{ssec:restricted Poisson structures}In this section we recall the definition of restricted Poisson structure due to \cite{bk}, but in the simplified form for $p>2$ (which was first observed in \cite{bao-ye-zhang}). We then motivate the definition and also extend it to $p=2$, by describing the free restricted Poisson algebra on a flat module as a free quantized algebra on the free restricted Lie algebra.

Assume $p>2$ and let $B$ be a commutative $\mathbb{F}_p$-algebra.
\begin{df}[{\cite{bk}, \cite{bao-ye-zhang}}]\label{def:restricted Poisson algebra p>2}
A \emph{restricted Poisson structure} on a commutative $B$-algebra $A$ is the datum of a $B$-linear Poisson bracket $\{-,-\}$ and a $p$-power operation  $-^{[p]}\colon A\to A$ such that:
\begin{enumerate}[label=(\arabic*)]
    \item $(A,\{-,-\},-^{[p]})$ is a restricted Lie algebra over $B$ (Definition \ref{def:restricted Lie algebra});
    \item \label{it:equation of restrictedness} for any $a\in A$, the following identity holds:
    \begin{equation}\label{eq:simple equation on restricted Poisson structure}
(a^2)^{[p]}=2a^pa^{[p]}.
    \end{equation}
\end{enumerate}
\end{df}
\begin{rem}\label{rem:multiplication is uniquely defined}
    Note that \eqref{eq:simple equation on restricted Poisson structure} also uniquely defines a formula for $(ab)^{[p]}$ for a pair of elements $a,b\in A$. Indeed, $ab= \frac{1}{2}((a+b)^2-a^2-b^2 )$, and one can get a formula for $(ab)^{[p]}$ using the Jacobson's formula (Lemma \ref{lem:Jacobson's formula}) together with \eqref{eq:simple equation on restricted Poisson structure}. Note also that
    $$1^{[p]}=(1^2)^{[p]}=2\cdot 1^{[p]} \quad \implies \quad 1^{[p]}=0. $$

\end{rem}

In principle, the definition in this form is simple enough to use right away, so the reader can freely skip the rest of this section up to the discussion of symplectic restricted structures. Nevertheless, we would like to also explain how to naturally arrive at the condition \eqref{eq:simple equation on restricted Poisson structure}, and how one should generalize it to the case $p=2$ (even though most of our results later on will have the $p>2$ assumption). 

\begin{rem}\label{rem:restricted Poisson algebras are monadic}
    Let $\mathsf{Pois}^{\mathsf{rest}}_B$ denote the category of restricted Poisson $B$-algebras. Since the structure of restricted Poisson $B$-algebra is given by the structure of a $B$-module endowed with a set of operations with finitely many inputs, and the relations on their composition that also involve finitely many inputs, the forgetful functor
    $$
\mr{oblv}\colon \mathsf{Pois}^{\mathsf{rest}}_B \to B\lmod
    $$
    is conservative, and has a left adjoint $P^{\mathsf{rest}}_B\colon B\lmod \to \mathsf{Pois}^{\mathsf{rest}}_B$ given by the ``free restricted Poisson algebra''. The composition $T:=\mr{oblv}\circ P^{\mathsf{rest}}_B \colon B\lmod \to B\lmod$ defines a monad on $B\lmod$. Moreover, the monad $T$ commutes with sifted colimits and the natural functor
    $$
\mathsf{Pois}^{\mathsf{rest}}_B \to \Mod_T(B\lmod)
    $$
    is an equivalence of categories. Indeed, commutation of $\mr{oblv}$ with sifted colimits and the existence of a left adjoint follows from \cite[Proposition 9.3]{ARV2011} applied to the algebraic theories corresponding to $B$-modules and restricted Poisson $B$-algebras correspondingly, while monadicity follows from Barr-Beck theorem (see e.g. \cite[Theorem C.3]{ARV2011}). Further we will mostly denote $T(V)$ by $P^{\mathsf{rest}}_B(V)$, implicitly referring to the underlying $B$-module of the latter.
\end{rem}

Below we will explicitly describe the restriction of the monad $T$ to the full subcategory ${\msf{Flat}}_B\subset B\lmod$ spanned by flat $B$-modules. There is then a unique sifted colimit preserving extension of that monad to $B\lmod$ (see Remark \ref{rem:general free restricted Poisson algebra}). Note that by Lazard's theorem \cite[Tag 0659]{kerodon}, there is an equivalence 
$${\msf{Flat}}_B\simeq \Ind(\msf{Free}_B^{\msf{fin}}),$$
where $\msf{Free}_B^{\msf{fin}}\subset B\lmod$ is the full subcategory spanned by finite free $B$-modules. In particular, any flat $B$-module is a filtered colimit of finite free $B$-modules.
\begin{rem}[Universal enveloping algebra and PBW-filtration]\label{rem:PBW} Let $\msf{Ass}_B$ be the category of associative $B$-algebras, and let $\msf{Lie}_B$ be the category of Lie algebras over $B$. There is a natural functor
$$
\msf{Ass}_B \to \msf{Lie}_B,
$$
sending an algebra $A$ to itself with the Lie algebra structure given by commutator. It has a left adjoint, given by the universal enveloping algebra 
$$L\mapsto U_B(L).$$ Explicitly, $U_B(L)$ is the quotient of the tensor algebra $T_B(L)=\oplus_{i\ge 0}L^{\otimes_B i}$ by the 2-sided ideal generated by elements $\ell_1\otimes \ell_2 - \ell_2\otimes\ell_1 -[\ell_1,\ell_2]$ for $\ell_1,\ell_2\in L$.
  
    If a Lie algebra $L$ is flat as a $B$-module, then so is $U_B(L)$: indeed, it has the exhaustive Poincar\'e-Birkhoff-Witt filtration $F_\star^{PBW}U_B(L)$ with the associated graded given by the free commutative algebra $\Sym^\star_{B}(L)$ (see \cite[Theorems 6, 7(i) and 8]{Higgins1969}), which is flat. We warn the reader that without flatness condition on $L$ the PBW-theorem can fail: for example, the natural map $L\to U_B(L)$ can fail to be an embedding (see \cite{Cohn1963}, \cite[Theorem 4.4]{BokutChenChen2011}).
\end{rem}
\begin{constr}[Free Poisson algebra]\label{constr:free Poisson algebra}
	Let $V\in {\msf{Flat}}_B$ be a flat $B$-module. Let $L_B(V)$ be the free Lie algebra on $V$ over $B$. It is flat over $B$: if $V\in \msf{Free}_B^{\msf{fin}}$, it is a free $B$-module (see e.g. \cite[Corollary 0.10]{Reutenauer1993}) and in general one reduces to this case by Lazard's theorem. By the universal property, the free associative algebra $T_B(V)$ is the universal enveloping algebra of $L_B(V)$. Viewed this way, $T_B(V)$ acquires a PBW-filtration $F_\star^{PBW}T_B(V)$, for which $\gr_\star^{PBW}T_B(V)\simeq \Sym^{\star}_{B} L_B(V)$ (see Remark \ref{rem:PBW}). One can also consider the corresponding Rees algebra\footnote{In \cite{bk} $Q(V)$ is called the algebra of \textit{quantized polynomials} in $V$, hence the notation. } $Q_B(V)\coloneqq  \oplus_i F_i^{PBW}T_B(V)\cdot h^i$ corresponding to $F_\star^{PBW}T_B(V)$; this is an $h$-torsion free $B[h]$-algebra such that $Q_B(V)[h^{-1}]\simeq T_B(V)[h,h^{-1}]$ and $Q_B(V)/h\simeq \gr_\star^{PBW}T_B(V)$. Since $[F_i^{PBW}T_B(V),F_j^{PBW}T_B(V)]\subset F_{i+j-1}^{PBW}T_B(V)$ one has $[Q_B(V),Q_B(V)]\subset h\cdot Q_B(V)$ and there is a natural Poisson bracket on $\gr_\star^{PBW}T_B(V)$ defined by
	$$
	h\cdot \{f,g\}\coloneqq [\tf,\tg] \mod h^2 \qquad \text{(or, in other words, $\{f,g\}\coloneqq \tfrac{1}{h}[\tf,\tg]\mod h)$},
	$$
	where $\tf$ and $\tg$ are any lifts of $f,g\in \gr_\star^{PBW}T_B(V)\simeq Q_B(V)/h$ to $Q_B(V)$.
	
	The resulting Poisson bracket on $\gr_\star^{PBW}T_B(V)\simeq \Sym^{\star}_{B} L_B(V)$ is identified with the Kostant-Kirillov bracket, namely the unique bracket that on linear functions $z,w\in L_B(V)\subset \Sym^{\star}_{B} L_B(V)$ is given by $\{z,w\}\coloneqq [z,w]\in L_B(V)$. We call $\Sym^{\star}_{B} L_B(V)$ with the above bracket the \textit{free Poisson $B$-algebra} on $V$ and denote it $P_B(V)$. 
	
	This is justified by the universal property: namely, any $B$-linear map $f\colon V\ra A$ to a Poisson $B$-algebra $A$ extends uniquely to a Poisson algebra homomorphism $P_B(V)\ra A$. Indeed, $f$ extends uniquely to a Lie algebra map $f'\colon L_B(V)\ra A$ with respect to the Poisson bracket on $A$, and then we can extend $f'$ to an algebra map $\Sym^{\star}_{B} L_B(V)\ra A$. Since the Poisson bracket on $\Sym^{\star}_{B} L_B(V)$ is uniquely determined by its restriction to $L_B(V)$ we get that this map is automatically Poisson. In particular, for any flat Poisson $B$-algebra $A$ there is a natural Poisson $B$-algebra map $P_B(A)\ra A$, and one can identify the category of flat Poisson $B$-algebras with modules over the monad defined by (underlying $B$-module of) $P_B(-)$.
\end{constr}

The idea of a restricted Poisson structure on $A$ is that together with a Poisson bracket one should include a restricted $p$-power operation $-^{[p]}\colon A\ra A$ that would turn $(A,\{-,-\})$ into a restricted Lie algebra. However, having only this in mind, it is not clear what relation on $-^{[p]}$ and multiplication we should impose. In the rest of this section we try to explain how to naturally arrive at the condition \ref{it:equation of restrictedness} in Definition \ref{def:restricted Poisson algebra p>2}, as well as its original version from \cite{bk} (Construction \ref{constr:restricted Poisson algebras for all p}). 


\begin{constr}[Free restricted Poisson algebra] \label{constr:free restricted Poisson algebra} Let $B$ be a ring of characteristic $p$. Let $V$ be a flat $B$-module and consider the free restricted Lie algebra $L^{\mathsf{rest}}_B(V)$ over $B$. The underlying $B$-module of $L^{\mathsf{rest}}_B(V)$ has a natural filtration $L^{\mathsf{rest}}_B(V)^{\le *}$ with $L^{\mathsf{rest}}_B(V)^{\le -1}=0$, $L^{\mathsf{rest}}_B(V)^{\le 0}\simeq L_B(V)$, and $L^{\mathsf{rest}}_B(V)^{\le i+1}$ is inductively defined as $L^{\mathsf{rest}}_B(V)^{\le i} + B\cdot(L^{\mathsf{rest}}_B(V)^{\le i})^{[p]}$, where by the Jacobson's formula the latter expression is a well-defined $B$-module. One has $\gr_\star(L^{\mathsf{rest}}_B(V))\simeq \oplus_{i\ge 0}L_B(V)^{(i)}$, where $L_B(V)^{(i)}\coloneqq L_B(V)\otimes_{B,F_B^i} B$ is the {$i$'th} Frobenius twist. Note that $L^{\mathsf{rest}}_B(V)$ is still flat as a $B$-module.
	
	Consider the corresponding universal enveloping $B$-algebra $U_B(L^{\mathsf{rest}}_B(V))$ and its Rees algebra $Q^{\mathsf{rest}}_B(V)\coloneqq \oplus_i F_i^{PBW}U_B(L^{\mathsf{rest}}_B(V))\cdot h^i$ with respect to the PBW-filtration. By analogy with Construction \ref{constr:free Poisson algebra} one can then define the \textit{free restricted Poisson algebra} $P^{\msf{rest}}_B(V)$ as $Q^{\mathsf{rest}}_B(V)/h\simeq \Sym^\star_B L^{\mathsf{rest}}_B(V)$. As in \textit{loc.cit} $P^{\msf{rest}}_B(V)$ is naturally a Poisson algebra via the Lie bracket on $L^{\mathsf{rest}}_B(V)$. However, it is not yet clear that $\Sym L^{\mathsf{rest}}_B(V)$ also has a natural restricted Lie algebra structure.
\end{constr}

Let $(L,-^{[p]})$ be a restricted Lie algebra over $B$ that is flat as a $B$-module (e.g. $L^{\mathsf{rest}}_B(V)$), and consider its universal enveloping $B$-algebra $U_B(L)$. Let us denote by $U_{h,B}(L)\coloneqq \oplus_n F_n^{PBW} U_B(L)\cdot h^n$ its Rees algebra with respect to the PBW-filtration. $L$ embeds into $U_{h,B}(L)$ as $h\cdot L\subset h\cdot F_1^{PBW}U_B(L)\subset U_{h,B}(L)$ and generates $U_{h,B}(L)$ as a $B[h]$-algebra; the relations are given by $\ell_1\ell_2-\ell_2\ell_1=h\cdot [\ell_1,\ell_2]_L$ for any $\ell_1,\ell_2\in L$. The associated graded $\gr_\star^{PBW}U_B(L)\simeq U_{h,B}(L)/h$ is canonically identified with $\Sym_B^\star(L)$ by the PBW-theorem.

Consider a map 
\begin{equation}\label{eq:p-central map for Lie algebras}
    \theta\colon L\ra U_{h,B}(L),\quad x\mapsto \theta(x):=x^p-h^{p-1}x^{[p]}\in U_{h,B}(L).
\end{equation} Note that $\theta(bx)=b^p\theta(x)$ for any $b\in B$.

\begin{lm}\label{lem:central map to U(g)}
	$\theta$ is additive and its image lands in the center $Z(U_{h,B}(L))\subset U_{h,B}(L)$.
\end{lm}
\begin{proof}
	We will denote by $[-,-]_L$ the Lie bracket in $L$ to distinguish it from the commutator $[-,-]$ in $U_{h,B}(L)$. Additivity of $\theta$ follows from Jacobson's formula applied to both $x\mapsto x^p$ and $x\mapsto x^{[p]}$ (see Lemma \ref{lem:Jacobson's formula} and Definition \ref{def:restricted Lie algebra}\eqref{item:Jacobson}): indeed, for $x,y\in L\subset U_{h,B}(L)$
    $$
    (x+y)^p-x^p-y^p= \sum_{i=1}^{p-1}L_i(x,y)_{U_{h,B}(L)}=h^{p-1}\cdot \sum_{i=1}^{p-1}L_i(x,y)_{L} =h^{p-1}\cdot ((x+y)^{[p]}-x^{[p]}-y^{[p]}).
    $$ Here $L_i(x,y)_{U_{h,B}(L)}$ (resp. $L_i(x,y)_{L}$) denotes the evaluation of the corresponding Lie polynomial in $U_{h,B}(L)$, endowed with the commutator Lie bracket (resp. $L$ with bracket $[-,-]_L$). The second equality follows from $[x,y]=h\cdot [x,y]_L\in U_{h,B}(L)$ for any $x,y\in L$. 
    
    For the statement about the center, it is enough to show that $[\theta(x),y]=0$ for $y\in L$. For $y\in L$ we have $[x^p,y]=\ad(x)^p(y)=h^p\ad_L(x)^p(y)$. On the other hand, $[x^{[p]},y]=h\ad_L(x)^p(y)$. We get that $[\theta(x),y]=0$.
\end{proof}

Consequently, the map \eqref{eq:p-central map for Lie algebras} is $F_B$-linear, and extends to a map of $B$-algebras
$$
\theta\colon \Sym^\star_B L^{(1)}\ra Z(U_{h,B}(L)). 
$$
Let $-^{(1)}\colon L\ra L^{(1)}$ be the natural map sending $x$ to $x\otimes 1$. Consider also the map $p\colon U_{h,B}(L)\twoheadrightarrow \Sym^\star_B L\simeq U_{h,B}(L)/h$ given by reduction modulo $h$. 

Note that given $x\in \Sym^\star_BL$ and any lift $\tx\in U_{h,B}(L)$ the class of $\tx^p$ in  $U_{h,B}(L)/h^p$ only depends on $x$ (and not on the lift $\tx$). Indeed, by Jacobson's formula for any $y\in U_{h,B}(L)$
$$
(\tx+hy)^p=\tx^p+(hy)^p+\sum_{i=1}^{p-1}{L_i}(\tx,hy),
$$
but each ${L_i}(\tx,hy)$ is a Lie polynomial of Lie degree\footnote{Meaning that it is expressed as a sum of {terms each involving a composition of }$(p-1)$ commutators.} $(p-1)$ and  $L_i(\tx,hy)=h^iL_i(\tx,{y})$. Thus each of the summands ${L_i}(\tx,hy)$ is divisible by $h^p$, and the expression above reduces to $
(\tx+hy)^p\equiv  \tx^p\mod{h^p}. 
$ 
The operation $x\mapsto \tx^p \!\mod{h^p}$ recovers the restricted structure on $L$ in terms of $\theta$, namely for $x\in L$
$$h^{p-1}x^{[p]}= \tilde{x}^p- \theta(x\twt) \pmod{h^p}.$$
Under the identification $h^{p-1}U_{h,B}(L)/h^p\simeq U_{h,B}(L)/h \simeq  \Sym^\star_B L$  this defines $x^{[p]}$ uniquely. 

We can now specialize to the free case $L=L^{\mathsf{rest}}_B(V)$, and define the restricted structure on the whole of $P^{\msf{rest}}_B(V)=\Sym_B^\star L^{\mathsf{rest}}_B(V)$ using the map $\theta$ in a similar way. 

\begin{lm}\label{lem:restricted structure}\begin{enumerate}
		\item \label{it:p-curvature for quantizations}For any $x\in P^{\msf{rest}}_B(V)$ and any lift $\tx\in Q^{\mathsf{rest}}_B(V)$ the element $\tx^p-\theta(x\twt)$ is divisible by $h^{p-1}$.
		\item The map $-^{[p]}\colon P^{\msf{rest}}_B(V)\ra P^{\msf{rest}}_B(V)$ defined by 
		$$
		h^{p-1}x^{[p]}= \tilde{x}^p- \theta(x\twt) \pmod{h^p}
		$$
		endows $P^{\msf{rest}}_B(V)\simeq \Sym^*_B L^{\mathsf{rest}}_B(V)$ with a restricted Lie algebra structure that agrees with the one on $ L^{\mathsf{rest}}_B(V)$.
        \item The equality 
        $$
(x^2)^{[p]}=2x^px^{[p]},
        $$
        holds for any $x\in P^{\msf{rest}}_B(V)$.
	\end{enumerate}
	
\end{lm}
\begin{proof}
	 (1) By the discussion above, the association $x\mapsto \tilde x^p$ produces a well-defined map $P^{\msf{rest}}_B(V) \ra Q^{\mathsf{rest}}_B(V)/h^{p}$ which we can reduce further to $Q^{\mathsf{rest}}_B(V)/h^{p-1}$. The resulting map $P^{\msf{rest}}_B(V)\ra Q^{\mathsf{rest}}_B(V)/h^{p-1}$ is in fact a central algebra map: this is \cite[Lemma 1.3]{bk}, but we include the proof for completeness. If $p=2$, then $Q^{\mathsf{rest}}_B(V)/h^{p-1}\simeq P^{\msf{rest}}_B(V)$, the map is given by Frobenius, and there is nothing to check. For $p>2$, additivity follows from Jacobson's formula, using that $[Q^{\mathsf{rest}}_B(V),Q^{\mathsf{rest}}_B(V)]\subset h\cdot Q^{\mathsf{rest}}_B(V)$, as in Lemma \ref{lem:central map to U(g)}. For any $y\in Q^{\mathsf{rest}}_B(V)$ we have $[\tilde x^p,y]=\ad(\tilde x)^p(y)\in h^p\cdot  Q^{\mathsf{rest}}_B(V)$, so $\tilde x^p$ is central modulo $h^{p-1}$. Finally, for multiplicativity for any $x,y\in Q^{\mathsf{rest}}_B(V)/h^{p-1}$ one has
     $$
     x^py^p=\tfrac{1}{2}((x^p+y^p)^2-x^{2p}-y^{2p})=\tfrac{1}{2}((x+y)^{2p} - x^{2p}-y^{2p})=\tfrac{1}{2} ((x^2+y^2 +2xy -[x,y])^p-x^{2p}-y^{2p})=(xy)^p.
     $$
     Now, the map $x\mapsto \tilde x^p\in Q^{\mathsf{rest}}_B(V)/h^{p-1}$ agrees with $\theta$ on $L^{\mathsf{rest}}_B(V)$ and since $P^{\msf{rest}}_B(V)\simeq \Sym^\star_B L^{\mathsf{rest}}_B(V)$ is a free commutative algebra it agrees with $\theta$ on the whole $P^{\msf{rest}}_B(V)$. We get that $\tx^p-\theta(x\twt)=0$ in $Q^{\mathsf{rest}}_B(V)/h^{p-1}$ for all $x\in P^{\msf{rest}}_B(V)$. 
	
	(2) We need to check that $x\mapsto x^{[p]}$ satisfies the conditions in Definition \ref{def:restricted Lie algebra}. First of all, $(bx)^{[p]}=b^px^{[p]}$. For any lifts $\tilde x,\tilde y$ to $Q^{\mathsf{rest}}_B(V)$ of elements $x,y\in P^{\msf{rest}}_B(V)$ one has $[\tilde x,\tilde y]=h\cdot \{x,y\}\mod h^2$, and so one has 
	$$
	\ad_{[-,-]}(\tilde x^p)(\tilde y)=\ad_{[-,-]}(\tilde x)^p(\tilde y)=h^p \cdot \ad_{\{-,-\}}(x)^p(y)\mod h^{p+1}. 
	$$ 
	Since $\theta(x^{(1)})$ is central, applying $[-,\tilde y]$ we also get that
	$$
	h^p\ad_{\{-,-\}}(x^{[p]})(y)=h^p \ad_{\{-,-\}}(x)^p(y)\mod h^{p+1}, 
	$$
    and so $\ad(x^{[p]})=\ad(x)^p$ in $P^{\msf{rest}}_B(V)$.
	Finally, the formula for $(x+y)^{[p]}$ follows from the Jacobson formula. Namely, $\theta$ is a homomorphism, so we have $\theta(x^{(1)}+y^{(1)})=\theta(x^{(1)})+\theta(y^{(1)})$, and 
	$$
	h^{p-1}\cdot ((x+y)^{[p]}-x^{[p]}-y^{[p]})= (\tilde x+\tilde y)^p - \tilde x^p-\tilde y^p \mod h^p
	$$
	where the right-hand side is given by $h^{p-1}\sum_i L_i(x,y)$ (since each $L_i$ is a Lie polynomial of degree $(p-1)$ we have $L_i(\tilde x,\tilde y)=h^{p-1}L_i(x,y)$ modulo $h^p$).

    (3) Let $x\in P^{\msf{rest}}_B(V)$, and let $\tilde x\in Q^{\mathsf{rest}}_B(V)$ be a lift of $x$; $\tilde x^2$ is a lift of $x^2$. Then in the quotient $Q^{\mathsf{rest}}_B(V)/h^p$ we have an equality
    $$
h^{p-1}\cdot (x^2)^{[p]}= (\tilde x^2)^p - \theta(x^{(1)})^2=\tilde x^{2p} - (\tilde x^p-h^{p-1}x^{[p]})^2=2\tilde x^px^{[p]}\cdot h^{p-1},
    $$
    from which it follows that $(x^2)^{[p]}=2x^px^{[p]}$ in $P^{\msf{rest}}_B(V)$.
\end{proof}

\begin{rem}[Universal property]\label{rem:universal property}
    Let $p>2$, from Lemma \ref{lem:restricted structure} we get that $P^{\msf{rest}}_B(V)$ is a restricted Poisson algebra in the sense of Definition \ref{def:restricted Poisson algebra p>2}. We claim that  $P^{\msf{rest}}_B(V)$ is in fact a free restricted Poisson $B$-algebra on $V$. Indeed, let $A$ be a restricted Poisson $B$-algebra, and let $V\to A$ be a $B$-module map. Then it extends uniquely to a restricted Lie $B$-algebra map $L^{\msf{rest}}_B(V)\to A$, which in turn uniquely extends to a map of Poisson $B$-algebras $P^{\msf{rest}}_B(V)=\Sym_B^\star(L^{\msf{rest}}_B(V))\to A$. By Remark \ref{rem:multiplication is uniquely defined} the restricted structure on the image of $\Sym_B^\star(L^{\msf{rest}}_B(V))$ is uniquely determined by the one on the image of $L^{\msf{rest}}_B(V)$. Thus the map $P^{\msf{rest}}_B(V)\to A$ is automatically a \textit{restricted} Poisson $B$-algebra map.

    The above universal property endows the endofunctor $V\in \msf{Flat}_B\mapsto P^{\msf{rest}}_B(V)\in \msf{Flat}_B$ with the structure of a monad (namely, the free restricted Poisson $B$-algebra monad). Following Remark \ref{rem:restricted Poisson algebras are monadic} the category of flat restricted Poisson $B$-algebras is equivalent to the category of modules over $P^{\msf{rest}}_B(-)$. Note that $P^{\msf{rest}}_B(-)$ commutes with filtered colimits (either via the reasoning in \textit{loc.cit} or by inspecting the explicit formula in Construction \ref{constr:free restricted Poisson algebra}).
\end{rem}

We now discuss the natural way to extend Definition \ref{def:restricted Poisson algebra p>2} to all primes.

\begin{rem}
[Formula for $(xy)^{[p]}$]\label{rem:polynomial P}
	Consider $V\simeq Bx\oplus By$: a free $B$-module of rank 2 with basis given by $x$ and $y$. Consider the (non-restricted) algebras $Q_B(V)$ and $P_B(V)$. As in the proof of Lemma \ref{lem:restricted structure}\eqref{it:p-curvature for quantizations} one sees that the map $x\mapsto \tilde x^p$ defines a (central) algebra homomorphism $P_B(V)\to Q_B(V)/h^{p-1}$. In particular, the expression $(xy)^p-x^py^p$ is divisible by $h^{p-1}$. One can then define a Poisson polynomial $P(x,y)\in P_B(x,y)\coloneqq P_B(V)$ by 
	$$
	h^{p-1}P(x,y)=(xy)^p-x^py^p \quad \text{in }Q_B(V)/h^{p} .
	$$
	Then in $Q^{\mathsf{rest}}_B(V)/h^p$ we have an equality
	\begin{align*}
		h^{p-1}(xy)^{[p]}\coloneqq (xy)^p&-\theta(xy)=x^py^p+h^{p-1}P(x,y)-(x^p-h^{p-1}x^{[p]})(y^p-h^{p-1}y^{[p]})=\\
		&=h^{p-1}(x^py^{[p]} +x^{[p]}y^p +P(x,y)),
	\end{align*}
	which shows that 
    	\begin{equation}\label{eq:restricted Poisson algebra}
		(xy)^{[p]}=x^py^{[p]} +x^{[p]}y^p +P(x,y) \in P_B^{\msf{rest}}(x,y).
	\end{equation} 
    
    Since $P_B^{\msf{rest}}(V)$ is a free restricted Poisson $B$-algebra, this equality also holds for any two given elements $a,b$ in any restricted Poisson $B$-algebra $(A,-^{[p]},\{-,-\})$. Note that $P(x,y)\in P_B(x,y)\subset P_B^{\msf{rest}}(x,y)$, and so the expression $P(a,b)\in A$ only depends on the Poisson bracket, and not the restricted structure on $A$. In particular, given two restricted structures $-^{[p]_1}, -^{[p]_2}$ extending a Poisson structure $\{-,-\}$ on a $B$-algebra $A$, their difference 
    $$
x\mapsto \delta(x):=x^{[p]_1}- x^{[p]_2}
    $$
satisfies the property
$$
\delta(xy)=x^p\delta(y)+y^p\delta(x);
$$
in other words, it is a Frobenius-linear derivation.

\end{rem}

\begin{constr}\label{constr:restricted Poisson algebras for all p} Following Remark \ref{rem:polynomial P},
    one can now reformulate Definition \ref{def:restricted Poisson algebra p>2}: namely the condition \ref{it:equation of restrictedness} in there can be replaced by a more general condition 
    $$
(xy)^{[p]}=x^py^{[p]} +x^{[p]}y^p +P(x,y),
    $$
    which also prescribes the formula for $(xy)^{[p]}$ when $p=2$.
     This is how restricted Poisson algebras are defined in \cite{bk} (namely, see \cite[Definition 1.8]{bk}). When $p>2$ this gives the same definition as in Definition \ref{def:restricted Poisson algebra p>2}. As in Remark \ref{rem:universal property}, given a flat $B$-module $V$, the algebra $P_B^\msf{rest}(V)$, as defined in Construction \ref{constr:free restricted Poisson algebra},  gives a free restricted Poisson algebra on $V$ in this sense. 
\end{constr}

We end the section with the general comment on how to get a formula for a free restricted Poisson algebra $P_B^\msf{rest}(V)$ for a not necessarily flat $B$-module $V$.

\begin{rem}\label{rem:general free restricted Poisson algebra}
    Let $V\in B\lmod$. Then $V$ can be realized as a reflexive coequalizer
    $$
\begin{tikzcd}[column sep=scriptsize]
Q_1
  \arrow[r, shift left=1ex, "d_0"]
  \arrow[r, shift right=1ex, "d_1"']
&
Q_0
  \arrow[l, "s" description]
  \arrow[r, "q"]
&
V
\end{tikzcd}
    $$
with $Q_0$ and $Q_1$ in $\msf{Flat}_B$: for example, one can take any free resolution $P_1\xra{f} P_0 \twoheadrightarrow V$ and put $Q_0:=P_0$, $Q_1:=P_1\oplus P_0$ with $d_0=(0,\id)$, $d_1=(f,\id)$ and $s=(0,\id)$. Since, following Remark \ref{rem:restricted Poisson algebras are monadic} the free restricted Poisson algebra monad $T=P_B^\msf{rest}(-)$ commutes with sifted colimits (e.g. reflexive coequalizers) we get that 
$$
\textup{coeq}\left( \!\begin{tikzcd}[column sep=scriptsize]P_B^\msf{rest}(Q_1)  \arrow[r, shift left=0.5ex, "d_0"]
  \arrow[r, shift right=0.5ex, "d_1"']& P_B^\msf{rest}(Q_0) \end{tikzcd}\!\right)\!\! \!\begin{tikzcd}
    \arrow[r,"\sim"]& P_B^\msf{rest}(V).
\end{tikzcd}
$$
In fact, using the universal properties of $B\lmod$ and $\msf{Flat}_B$ as the sifted and the ind-completion of $\msf{Free}_B^{\msf{fin}}$ respectively, one can show that there is an equivalence of categories between the category of filtered colimit preserving monads on $\msf{Flat}_B$ and the category of sifted colimit preserving monads on $B\lmod$ that preserve the subcategory $\msf{Flat}_B\subset B\lmod$.
\end{rem}

\subsection{Restricted symplectic geometry}\label{ssec:some symplectic}
We now extend the discussion in \S \ref{ssec:restricted Poisson structures} to the non-affine setting, but restrict ourselves to the situation when the scheme is pseudo-smooth and the Poisson bracket is non-degenerate (equivalently, comes from a symplectic form). 

\begin{df}
	A \textit{symplectic scheme} $(X,\omega)$ over $S$ is a pseudo-smooth $S$-scheme $X$ with a non-degenerate globally defined closed 2-form $\omega\in H^0(X,\Omega^2_X)$.
\end{df}

Since $\omega$ is non-degenerate it induces an identification $T_X\xra{\sim} \Omega^1_X$ via $\xi\mapsto \iota_\xi\omega$, where $\iota_\xi$ is the contraction with $\xi$. This way $\omega$ also defines a section $\omega^{-1}$ of $\wedge^2_{\mc O_X}T_X$, which then gives a Poisson bracket on $\mc O_X$: $\{f,g\}\coloneqq \langle\omega^{-1},df\wedge dg\rangle$. The property that $d\omega=0$ translates to the Jacobi identity for $\{-,-\}$: in particular $(\mc O_X,\{-,-\})$ defines a sheaf of Lie algebras over $\mc O_S$ on $X$. 

Having a Poisson bracket $\{-,-\}$, for any $f\in\mc O_X$ the bracket $\{f,-\}$ defines a derivation $H_f\colon \mc O_X\ra \mc O_X$. This gives a map of sheaves $H_-\colon \mc O_X\ra T_X$ locally sending $f\mapsto H_f$.
\begin{df}\label{def:hamiltonian vector fields}
	A vector field $\xi\in H^0(X,T_X)$ is called \textit{Hamiltonian} if $\xi=H_f$ for some $f\in H^0(X,\mc O_X)$.
\end{df} 
\begin{rem}\label{rem:Hamiltonian vector fields}
	Tracing through the definitions one can see that the equation $\xi=H_f$ is equivalent to $\iota_\xi\omega=-df$. Since $\omega$ is non-degenerate it follows that Hamiltonian vector fields generate $T_X$ as an $\mc O_X$-module Zariski locally on $X$.
\end{rem}

One has a global version of Definition \ref{def:restricted Poisson algebra p>2}. Assume $p>2$, and let $S$ be a base $\mathbb F_p$-scheme.

\begin{df}[{\cite{bk}}] A \textit{restricted Poisson structure} on an $S$-scheme $X$ is given by an $\mc O_S$-linear Poisson bracket $\{-,-\}$ and a restricted $p$-power operation $-^{[p]}$ on $\mc O_X$ such that \begin{enumerate}
\item $(\mc O_X,\{-,-\}, -^{[p]})$ is a restricted Lie algebra over $\mc O_S$;
\item for any $f\in\mc O_X$ one has
\begin{equation}\label{eq:simple formula for restricted Poisson}
	(f^2)^{[p]}=2f^pf^{[p]}.
\end{equation}

\end{enumerate}
As in Construction \ref{constr:restricted Poisson algebras for all p}, one can extend this to $p=2$ by asking that 
\begin{equation}\label{eq:restricted symplectic scheme}
 (fg)^{[p]}=f^pg^{[p]} +f^{[p]}g^p +P(f,g)
\end{equation}
for $f,g\in \mc O_X$, where $P(x,y)$ is the Poisson polynomial from Remark \ref{rem:polynomial P}. When $p>2$ this condition is equivalent to \eqref{eq:simple formula for restricted Poisson}.
\end{df}

It is natural to ask when a given Poisson structure lifts to a restricted Poisson structure, and how to describe all such lifts. 

First of all, not every Poisson structure extends to a restricted Poisson structure. Indeed, note that given a restricted Poisson scheme $X$, for any $f\in \mc O_X$ we have 
$$
H_{f^{[p]}}=H_{f}^{[p]}\in T_X,
$$
since both vector fields act as $H_f^p$ on $\mc O_X$. Thus, Hamiltonian vector fields on $X$ must be closed under the restricted operation $-^{[p]}$. This imposes some strong restrictions on the underlying Poisson bracket. For example, in \cite[Theorem 1.11]{bk} Bezrukavnikov and Kaledin show that, given a \textit{symplectic} $S$-scheme $(X,\omega)$, Hamiltonian vector fields are closed under $-^{[p]}$ if and only if\footnote{Here $C\colon F_{X*}^\lp\Omega^2_{X,cl}\ra \Omega^{2,\lp}_X$ is the Cartier operator and the condition $C(\omega)=0$ is equivalent to $\omega$ being exact locally in Zariski topology.} $C(\omega)=0$ (in other words, $\omega$ is exact Zariski locally on $X$). 

More generally, in \cite[Theorem 1.12]{bk}, for symplectic pseudo-smooth $(X,\omega)$, they describe all restricted Poisson structures compatible with $\omega$. Below we include a different proof of their result, also providing an explicit formula for the isomorphism. Consider the sheaf $\Omega^1_X/d(\mc O_X)$ and the map $\Omega^1_X/d(\mc O_X) \ra \Omega^2_X$ induced by $d$.

\begin{pr}\label{prop:description of restricted symplectic structure} Let $p>2$. Let $(X,\omega)$ be a pseudo-smooth symplectic $S$-scheme. Then there is a natural bijection 
	\medskip 
	\begin{equation}\label{eq:comparing restricted Poisson structures}\begin{Bmatrix} \text{\emph{global sections}} \\ [\eta] \text{\emph{ of }} \Omega^1_X/d(\mc O_X) \\
		\text{\emph{s.t. }} d[\eta] = \omega
	\end{Bmatrix}
 \xymatrix{\ar[rr]^{1\text{-\emph{to}-}1}&&\ar[ll]}	\begin{Bmatrix}\text{\emph{restricted Poisson}} \\
 	\text{\emph{structures on }} $X$ \\
 	\text{\emph{compatible with }} \omega
 \end{Bmatrix} .
	\end{equation}
	\medskip
	
	Locally on $X$, given a lift of $[\eta]\in H^0(X, \Omega^1_X/d(\mc O_X))$ to a 1-form $\eta\in H^0(X, \Omega^1_X)$, the corresponding restricted $p$-th power operation $-^{[p]_{\eta}}\colon \mc O_X \ra \mc O_X$ is given by the formula 
	\begin{equation}\label{eq:restricted structure associated to a 1-form}
	f^{[p]_{\eta}}:= \iota_{H_f^{[p]}}\eta- H_f^{p-1}(\iota_{H_f}\eta) .
\end{equation}	
\end{pr}
\begin{proof}
	First, we explain an abstract way to construct the map from the left to the right. The idea is to realize $\mc O_X$ as a Lie algebra of a certain group scheme over $X^{\{p\}}$ associated to $[\eta]$. 
	
	More precisely, for any $\eta\in H^0(X, \Omega^1_X)$ define a group sheaf $\widetilde G_\eta$ over $X^{\{p\}}$ as follows: for $T\ra X^{\{p\}}$ put\footnote{Here is another description of $\widetilde G_\eta$: define a $\mbb G_a$-torsor $\mathcal L_\eta$ on $X$ with a (non-flat) connection so that the underlying $\mathbb G_a$-torsor is trivial and the connection is given by $\eta$. Then for $T\ra X^{\{p\}}$, a $T$-point of $\widetilde G_\eta$ is a pair $(\phi,\tau)$ consisting of an automorphism $\phi$ of $X\!\times_{X^{\{p\}}} \! T$ as a $T$-scheme and an isomorphism $\tau\colon \phi^*\mc L_\eta \xra{\sim} \mc L_\eta$ of $\mbb G_a$-torsors with connection.}
	$$
	\widetilde G_{\eta}(T):=\{(\phi,f), \ \phi \in \Aut_T(X\!\times_{X^{\{p\}}} \! T), \ f\in \mc O(X\! \times_{X^{\{p\}}} \! T), \ \text{s.t. } \phi^*\eta -\eta =df\}.
	$$
The multiplication on $\widetilde G_{\eta}$ is given by $(\phi_1,f_1)\cdot (\phi_2,f_2)= (\phi_2\circ \phi_1, \phi_1^*(f_2)+f_1)$. Since $X\ra X^{\{p\}}$ is finite locally free and $d$ is $\mc O_X^{\{p\}}$-linear, it is clear that $\widetilde G_{\eta}$ is represented by an affine group scheme over $X^{\{p\}}$.

The Lie algebra of $\widetilde G_{\eta}$ is identified with the vector bundle
$$
\mr{Lie}(\widetilde G_{\eta})=\{(\theta, f),\ \theta \in T_X, \ f\in \mc O_X, \ \text{s.t. } L_\theta \eta =df\}
$$
where $L_\theta$ denotes the Lie derivative. Here $(\theta,f)$ corresponds to the infinitesimal pair $(\varphi,\varepsilon f)\in \widetilde{G}_{\eta}(\mc O_{X^\lp}[\varepsilon]/\varepsilon^2)$,
where
$\varphi^*=\mathrm{id}+\varepsilon\theta$. The Lie bracket on $\mr{Lie}(\widetilde G_{\eta})$ is given by 
$$
[(\theta_1,f_1),(\theta_2,f_2)]= ([\theta_1,\theta_2],\theta_1(f_2)-\theta_2(f_1)). 
$$ 
One has a natural embedding of Lie algebras
$$
\mr{Lie}(\widetilde G_{\eta}) \subset \mc D_{X,\le 1}, \quad (\theta,f) \mapsto \theta +f,
$$
and below we will occasionally consider elements of $\mr{Lie}(\widetilde G_{\eta})$ as differential operators of order $\le 1$.

Now let $\omega=d[\eta]$ be the corresponding symplectic form. We define a map 
$$
\alpha\colon \mc O_X \ra  \mr{Lie}(\widetilde G_{\eta})\subset \mc D_{X,\le 1}
$$
by the formula $f\mapsto \alpha(f):= H_f+(-f+\iota_{H_f}\eta)$. Since 
$$
L_{H_f}\eta= (\iota_{H_f}d+d\iota_{H_f})\eta=\iota_{H_f}\omega +d(\iota_{H_f}\eta)= d(-f+\iota_{H_f}\eta),
$$
$\alpha(f)$ indeed lands in $\mr{Lie}(\widetilde G_{\eta})$.
\begin{lm}
	The map $\alpha\colon  (\mc O_X,\{-,-\}) \ra  \mr{Lie}(\widetilde G_{\eta})$ is an isomorphism of Lie algebras.
\end{lm}	

\begin{proof} 
We have  
\begin{align*}
[\alpha(f)&,\alpha(g)]= [H_f-f+\iota_{H_f}\eta, H_g -g+\iota_{H_g}\eta]= [H_f,H_g]+ L_{H_f}(-g+\iota_{H_g}\eta)\  -\\
&- L_{H_g}(-f+\iota_{H_f}\eta)= H_{\{f,g\}} -2\{f,g\} +L_{H_f}\iota_{H_g}\eta - L_{H_g}\iota_{H_f}\eta.
\end{align*}
Since $L_\theta \iota_{\xi}=\iota_\xi L_\theta + \iota_{[\theta,\xi]}$ this further simplifies to 
\begin{align*}
& H_{\{f,g\}} -2\{f,g\} + \iota_{H_g}\iota_{H_f} d\eta +\iota_{H_g}d\iota_{H_f} \eta + \iota_{[H_f, H_g]}\eta - \iota_{H_g}d\iota_{H_f}\eta = \\ &=H_{\{f,g\}} -2 \{f,g\} -\{g,f\} +  \iota_{[H_f, H_g]}\eta = H_{\{f,g\}} -\{f,g\} +\iota_{[H_f, H_g]}\eta, 
\end{align*}
which is exactly $\alpha(\{f,g\})$. Thus $\alpha$ is a map of Lie algebras.

	It remains to see that $\alpha$ is an isomorphism. Let $T^\mr{Ham}_X\subset T_X$ denote the ($\mc O_{X}^{\{p\}}$-linear) Lie subalgebra of Hamiltonian vector fields. Note that if $(\theta,f)\in \mr{Lie}(\widetilde G_{\eta})$, then $\theta\in T^\mr{Ham}_X$. Indeed, 
	$$
	\iota_{\theta} \omega = (L_\theta - d\circ \iota_{\theta})(\eta)=d(f-\iota_{\theta}\eta),
	$$
	so $\theta=H_{\iota_{\theta}\eta-f}$ by Remark \ref{rem:Hamiltonian vector fields}. The map $\mr{Lie}(\widetilde G_{\eta}) \ra T^\mr{Ham}_X$ is surjective with $H_f$ being hit by $\alpha(f)$, and its kernel is given by pairs $(0,f)$ with $df=0$, which is $\mc O_X^{\{p\}}\subset \mc O_X$. We get a short exact sequence of Lie algebras
	$$
	0\ra \mc O_X^{\{p\}} \ra \mr{Lie}(\widetilde G_{\eta}) \ra   T^\mr{Ham}_X \ra 0
	$$
 (with the zero bracket on $\mc O_X^{\{p\}}$). There is a similar exact sequence for $(\mc O_X,\{-,-\})$, with the map $\mc O_X\ra T^\mr{Ham}_X$ given by $f\mapsto H_f$, and whose kernel is the Poisson center $\mc O_X^{\{p\}}$. The map $\alpha$ gives a commutative diagram
 $$
\begin{tikzcd}[column sep=large, row sep=large]
0 \arrow[r]
& \mc O_X^{\{p\}}
    \arrow[r, "f\mapsto f"]
    \arrow[d, equal]
& \mc O_X
    \arrow[r, "f\mapsto H_f"]
    \arrow[d, "\alpha"]
& T_X^\mr{Ham}
    \arrow[r]
    \arrow[d, equal]
& 0
\\
0 \arrow[r]
& \mc O_X^{\{p\}}
    \arrow[r, "{f\mapsto (0,-f)}"]
& \mr{Lie}(\widetilde G_{\eta})
    \arrow[r, "{(\theta,f)\mapsto\theta}"]
& T_X^\mr{Ham}
    \arrow[r]
& 0.
\end{tikzcd}
 $$
 Thus $\alpha$ is an isomorphism.
\end{proof}	

Recall that the Lie algebra of an affine group scheme carries a canonical structure of a restricted Lie algebra (see \cite[Chapter II, \S 7, $n^\circ\ \! 2$]{gd}). It is defined as follows. Let $R:=\mc O_X^{\{p\}}$ and consider the algebra  of dual numbers $R[\varepsilon]$ (so $\varepsilon^2=0$). The Lie algebra $\mr{Lie}(\widetilde G_{\eta})$ (as a sheaf of $R$-modules) is by definition the kernel of the map $\widetilde G_{\eta}(R[\varepsilon]) \ra \widetilde G_{\eta}(R)$ induced by the reduction map $R[\varepsilon] \ra R$. Now, one can consider an auxiliary ring $R[\varepsilon_1,\ldots, \varepsilon_p]$ (with all $\varepsilon_i^2=0$), and an $R$-subalgebra $S\subset R[\varepsilon_1,\ldots, \varepsilon_p]$ generated by the sum $\sigma=\varepsilon_1+\ldots+\varepsilon_p$ and the product $\pi=\varepsilon_1\ldots\varepsilon_p$; it is easy to see that $S$ is isomorphic to the subring of $\Sigma_p$-invariants in $R[\varepsilon_1,\ldots, \varepsilon_p]$, and that $S\simeq R[\sigma, \pi]/(\sigma^p,\pi^2,\sigma\pi)$. Note that we have a natural projection $S \twoheadrightarrow R[\pi]/\pi^2$ sending $\sigma \mapsto 0$. For $x\in \mr{Lie}(\widetilde G_{\eta})$ denote by $e^{\varepsilon x}$ the corresponding element of $\widetilde G_{\eta}(R[\varepsilon])$. The elements $e^{\varepsilon_ix}$ in $\widetilde G_{\eta}(R[\varepsilon_1,\ldots, \varepsilon_p])$ commute\footnote{One has $e^{\varepsilon_ix}e^{\varepsilon_jx}e^{-\varepsilon_ix}e^{-\varepsilon_jx}=e^{\varepsilon_i\varepsilon_j[x,x]}=1$.}, and so $e^{\varepsilon_1x}\ldots e^{\varepsilon_px}$ gives an $\Sigma_p$-invariant element of $\widetilde G_{\eta}(R[\varepsilon_1,\ldots, \varepsilon_p])$. Since $\widetilde G_{\eta}$ is representable, we have $\widetilde G_{\eta}(R[\varepsilon_1,\ldots, \varepsilon_p])^{\Sigma_p}\simeq \widetilde G_{\eta}(S)$; thus, $e^{\varepsilon_1x}\ldots e^{\varepsilon_px}$ gives an element in $\widetilde G_{\eta}(S)$ and we can take its image in $\widetilde G_{\eta}(R[\pi]/\pi^2)$, which is equal to $e^{\pi y}$ for some $y\in \mr{Lie}(\widetilde G_{\eta})$. Then the association 
$$
x\mapsto y=:x^{[p]}
$$
defines a structure of restricted Lie algebra (see \cite[Chapter II, \S 7, Proposition 3.4]{gd}) on $\mr{Lie}(\widetilde G_{\eta})$. Let $(\theta,f)\in \mr{Lie}(\widetilde G_{\eta})$. To compute $(\theta,f)^{[p]}$ via the above recipe we need to look at the image of 
$$
(1+\varepsilon_1\theta, \varepsilon_1f)\cdot \ldots \cdot (1+\varepsilon_p\theta, \varepsilon_p f)\in \widetilde G_{\eta}(R[\varepsilon_1,\ldots, \varepsilon_p])
$$
in $\widetilde G_{\eta}(R[\pi]/\pi^2)$. It is not hard to see that it is given by a pair $(1+\pi \theta^{[p]}, \pi \theta^{p-1}(f))$. In other words, we get that 
$$
(\theta,f)^{[p]}=(\theta^{[p]},\theta^{p-1}(f)).
$$

	Via the identification $\alpha\colon \mc O_X\xra{\sim} \mr{Lie}(\widetilde G_{\eta})$ we can now endow $\mc O_X$ with the structure of restricted Lie algebra; more precisely, it is given by $f\mapsto \alpha^{-1}(\alpha(f)^{[p]})$. This agrees with (\ref{eq:restricted structure associated to a 1-form}): indeed, we have
	$$
	\alpha(f)^{[p]}=(H_f, -f+\iota_{H_f}\eta)^{[p]}= (H_{f}^{[p]}, H_f^{p-1}(-f+\iota_{H_f}\eta))= (H_f^{[p]}, H_f^{p-1}(\iota_{H_f}\eta))
	$$
	and this way
	$$
	 f^{[p]_{\eta}}:=\alpha^{-1}(\alpha(f)^{[p]})= \iota_{H^{[p]}_f}\eta- H_f^{p-1}(\iota_{H_f}\eta).
	$$
	  Moreover, $f\mapsto f^{[p]_{\eta}}$ defines a restricted Poisson structure:
	  \begin{align*}
	  (f^2)^{[p]_{\eta}}=&\ \iota_{H^{[p]}_{f^2}}\eta - H_{f^2}^{p-1}(\iota_{H_{f^2}}\eta) =(2f)^p \iota_{H_f^{[p]}}\eta - (2f)^{p-1}H_f^{p-1}(2f\cdot \iota_{H_f}\eta)= \\
	  & = (2f)^p\cdot (\iota_{H^{[p]}_f}\eta-H_f^{p-1}\iota_{H_f}\eta) = 2 f^pf^{[p]_{\eta}},
	  \end{align*}
	where above we used that $H_{f^2}=2fH_f$. 
	
	Note that the association $\eta \mapsto (-)^{[p]_{\eta}}$ in fact only depends on the class $[\eta]$ in $\Omega^1_X/d(\mc O_X)$: indeed, if $\eta'=\eta + dg$, then for any $f\in \mc O_X$
	$$
	f^{[p]_{\eta'}}-f^{[p]_{\eta}}= \iota_{H^{[p]}_f}dg-H_f^{p-1}(\iota_{H_f}dg)= (\ad f)^p(g)- (\ad f)^p(g) =0. 
	$$
	This way we see that having a section $ [\eta] \in H^0(X, \Omega^1_X/d(\mc O_X))$ the restricted Poisson structure $(-)^{[p]_{[\eta]}}$ is well-defined (by lifting $[\eta]$ to an actual 1-form $\eta$ locally and gluing the resulting operations $(-)^{[p]_{\eta}}$ globally).
	
	It remains to show that the  association $[\eta] \mapsto (-)^{[p]_{[\eta]}}$ is bijective. Note that both sides of (\ref{eq:comparing restricted Poisson structures}), considered as Zariski sheaves on $X^{\{p\}}$, are torsors over $T_X^{\{p\}}$. Indeed, the sheaf of sections $[\eta]\in \Omega^1_X/d(\mc O_X)$ with $d[\eta]=\omega$ is naturally a torsor over $\mc H^1_{X,\dR}\simeq \Omega^{1,\lp}_X$, which we can identify with $T_X^{\{p\}}$ by pulling back the isomorphism $T_X \xra{\sim} \Omega^1_X$, $\xi \mapsto \iota_\xi \omega$ along $W_X^\lp$. For the restricted Poisson structures note that by (\ref{eq:restricted symplectic scheme}) the difference of two $p$-th restricted power operations $(-)^{[p]_1}$ and $(-)^{[p]_2}$ defines a Frobenius-derivation, which, moreover, takes values in the Poisson center, since both $\{f^{[p]_1},g\}$ and $\{f^{[p]_2},g\}$ are equal to $H_{f}^{[p]}(g)$. Since $\{-,-\}$ comes from a symplectic form, the Poisson center is identified with $\cO_X^\lp\subset \cO_X$ and such derivations are given by the sections of $T_{X}^{\{p\}}$. Thus it remains to see that the association $[\eta] \mapsto (-)^{[p]_{[\eta]}}$ respects the torsor structure. Let $\alpha$ be a local section of $\Omega^1_X$ with $d\alpha=0$, then 
	$$f^{[p]_{\eta+\alpha}}-f^{[p]_{\eta}}=\iota_{H^{[p]}_f}\alpha - H_f^{p-1}(\iota_{H_f}\alpha) ,$$
	which by Remark \ref{rem:pairing with Cartier} is the same as $\iota_{H_f^{\{p\}}} \mathsf C(\alpha)$. This shows that if we act on $[\eta]$ by a vector field $\xi^{\{p\}} \in T_{X}^{\{p\}}$ the corresponding $p$-th power operation will change by
	$$\iota_{H_f^{\{p\}}}\iota_{\xi^{\{p\}}} \omega^{\{p\}} = -(\iota_{\xi}\iota_{H_f}\omega)^p = \xi(f)^p,$$
	which is exactly what we need.
\end{proof}	

Proposition \ref{prop:description of restricted symplectic structure} motivates the following definition.
\begin{df}[{\cite{bk}}]\label{def:restricted symplectic scheme}
	A \textit{restricted symplectic $S$-scheme} $(X,[\eta])$ is a pair of a pseudo-smooth $S$-scheme $X$ and a section $[\eta]\in H^0(X,\Omega^1_X/d(\mc O_X))$ such that $\omega\coloneqq d[\eta]\in H^0(X,\Omega^2_X)$ is non-degenerate. 
\end{df}
By Proposition \ref{prop:description of restricted symplectic structure} this datum is equivalent to a symplectic structure (given by $\omega$) together with a compatible restricted Poisson structure.

\begin{example}\label{ex:restricted structure on T^*X}
\begin{enumerate}
    \item  Let $Y=\mbb A^d_S$ be the relative affine space with coordinates $x_1,\ldots,x_d$ and let $X=T^*Y\simeq \Spec_S \mc O_S[x_1,\ldots,x_d,y_1,\ldots,y_d]$. Then the choice of the differential 1-form
    $$
\eta_{\mathsf{can}} := \sum_{i=1}^d y_idx_i\in \Omega^1_X
    $$
    endows $X$ with the restricted symplectic $S$-scheme structure. Indeed,
    $
    \omega:=d\eta_{\mathsf{can}}=\sum_{i=1}^d dy_i\wedge dx_i
    $
    is clearly non-degenerate. 
    \item More generally, let $Y$ be a pseudo-smooth $S$-scheme, and consider the cotangent bundle $T^*Y:=\Spec_{Y}(\Sym^\star_{\mc O_Y} \!T_Y)$. Let $\pi\colon T^*Y\to Y$ denote the affine projection map. Note that there is a tautological section $s\in H^0(T^*Y, \pi^*\Omega^1_Y)$, given by the map of $\Sym^\star_{\mc O_Y} \!T_Y$-modules $\Sym^\star_{\mc O_Y} \!T_Y\to (\Sym^\star_{\mc O_Y} \!T_Y)\otimes_{\mc O_Y} \!\Omega^1_Y$ sending $1$ to the canonical section of $T_Y\otimes_{\mc O_Y} \!\Omega^1_Y $. Let $\eta_{\msf{can}}\in H^0(T^*Y, \Omega^1_{T^*Y})$ be the image of $s$ under the natural map $\pi^*\Omega^1_Y \to \Omega^1_{T^*Y}$; this 1-form is often called the \textit{Liouville form}. The corresponding class $[\eta_{\msf{can}}]$ endows $T^*Y$ with a canonical restricted symplectic $S$-scheme structure. 
    
    \item Let $(X,[\eta])$ be a restricted symplectic $S$-scheme. Let $q\colon U\ra X$ be an \textit{\'etale} map. Then $ \Omega^1_U \simeq q^{*} \Omega^1_X$ and the pull-back $q^*[\eta]$ defines a restricted symplectic structure on $U$. Indeed, $d(q^*[\eta])=q^*\omega$ is non-degenerate.
\end{enumerate}
   
\end{example}

The following remark will be useful. 
\begin{rem}[$X$ as a restricted symplectic scheme over $X^\lp$]\label{rem:restricted stucture over the Frobenius twist}
	Let $X$ be a pseudo-smooth $S$-scheme and let $X\ra X^\lp$ be the reduced Frobenius. Then we claim that providing $X$ with the structure of a restricted symplectic scheme over $S$ and over the reduced twist $X^\lp$ is the same thing. Indeed, we have a commutative diagram
	$$
	\xymatrix{X\ar[r]\ar[dr]&X^\lp\ar[d]\\
		&S}
	$$
	which gives an exact sequence 
	$$
	F_X^{\lp*}\Omega^1_{X^\lp}\xra{F_X^{\lp*}} \Omega^1_X\ra \Omega^1_{X/X^\lp}\ra 0.
	$$
	Since the pull-back map $F_X^{\lp*}\Omega^1_{X^\lp}\ra \Omega^1_X$ is 0 we get that $\Omega^1_{X/X^\lp}\simeq \Omega^1_X$, and so $X$ is pseudo-smooth over $X^\lp$. We also have isomorphisms  $\Omega^i_{X/X^\lp}\simeq \Omega^i_X$ and, more generally, $\dR_{X/X^\lp}\simeq \dR_X$. Lastly, essentially by the definition of reduced twist, $X^\lp$ for $X$ considered as an $X^\lp$-scheme or as an $S$-scheme coincide (because $\mc O_X^\lp$ is defined as the image of relative Frobenius). Thus symplectic and, more generally, restricted symplectic structures on $X$ considered as an $X^\lp$-scheme or an $S$-scheme are the same thing.
\end{rem}

\subsection{The group scheme $G_0$ and Darboux lemma}\label{ssec: group scheme G_0} For this and the next section we assume that $p>2$.
Another observation of Bezrukavnikov and Kaledin is that restricted Poisson structures on $X$ can be interpreted as the reductions of the canonical torsor of ``Frobenius-frames''. This follows from their version of the Darboux lemma (Proposition \ref{pr:BK's Darboux lemma}), which in particular states that any two restricted symplectic structures on $X$ are isomorphic flat locally on $X^\lp$. In this section we recall some necessary details of their construction. But first, we introduce one more piece of notation. 
\begin{constr}[Algebra $A_0$]\label{constr:A_0}
Fix a number $d\in \mathbb N$. 
Let
$$A_0\coloneqq \mbb F_p[x_i,y_j]_{1\le i,j\le d}/ x_i^p=y_j^p=0;$$
this is the Frobenius neighborhood of 0 in a symplectic vector space $\mbb A^{2d}=\Spec \mbb F_p[x_i,y_j]_{1\le i,j\le d}$. We have $\Omega^1_{A_0}\simeq (\oplus_{i=1}^d A_0\cdot dx_i)\oplus (\oplus_{i=1}^d A_0\cdot dy_i)$; in particular $A_0$ is pseudo-smooth. We endow $\Spec A_0$ with a natural restricted symplectic structure given by the 1-form 
	$$
	\eta_{\mathsf{can}}=\sum_{i=1}^d y_idx_i\in \Omega^1_{A_0}:
	$$
the 2-form $\omega\coloneqq d\eta_{\mathsf{can}}=\sum_i dy_i\wedge dx_i\in \Omega^2_{A_0}$ is clearly non-degenerate.

\end{constr}

Let us make the following simple observation:
\begin{rem}\label{lem:fiber product of reduced Frobenius with itself}
   Let $C=\mbb F_p[x_i]/(x_i^p)_{1\le i\le d}$ and let $X$ be a pseudo-smooth $S$-scheme with $\rk \Omega^1_X=d$. Then there exists an fppf cover $Z\ra X^\lp$ and an isomorphism 
   $$
X\times_{X^\lp}Z \simeq \Spec C \times Z.
   $$
   Indeed, if $X$ is smooth over $S$, $F_{X}^\lp\colon X\ra X^\lp$ is  the relative Frobenius and the statement is classical: Zariski locally one has an \'etale map $X\ra \mbb A^{d}_S$ and in the case of affine space the isomorphism is easy to construct explicitly. A general pseudo-smooth scheme Zariski locally is  given by the Frobenius neighborhood $(F_Y^\lp)^{-1}(Z)$ of a closed subscheme $Z\hookrightarrow Y^\lp$ in a smooth $Y$ (Proposition \ref{prop:local description of pseudo-smooth}). Following Remark \ref{rem:reduced Frobenius in the case of Frobenius neighborhood} the reduced Frobenius map $F_X^\lp\colon  X\ra X^\lp\simeq Z$ is the pull-back of $F_Y^\lp\colon Y\ra Y^\lp$. This way we reduce to the smooth case.
\end{rem}
Thus, any pseudo-smooth $S$-scheme flat locally looks like $\Spec C$ (as the dimension parameter $d$ varies). In \cite{bk}, Bezrukavnikov and Kaledin proved a version of this statement for restricted symplectic $S$-schemes, with $(\Spec A_0,[\eta_{\mathsf{can}}])$ serving as the local model. Following Remark \ref{rem:restricted stucture over the Frobenius twist}, recall that given a restricted symplectic $S$-scheme $(X,[\eta])$ one can equivalently consider it as a symplectic scheme over $X^\lp$. 
\begin{pr}[Darboux lemma, {\cite[Proposition 3.4]{bk}}]\label{pr:BK's Darboux lemma} Let $(X,[\eta])$ be a restricted symplectic $S$-scheme such that $\rk \Omega^1_X=2d$. Then there exists an fppf cover $Z\ra X^\lp$ and an isomorphism of restricted symplectic $Z$-schemes
    $$
(X,[\eta])\times_{X^\lp}Z \simeq (\Spec A_0,[\eta_{\mathsf{can}}])\times Z. 
    $$
\end{pr}
\begin{proof}
    By Remark \ref{lem:fiber product of reduced Frobenius with itself}, there is a natural faithfully flat cover $\sqcup_iZ_i\ra X^\lp$ by a disjoint union of affines so that $X\times_{X^\lp}Z_i\simeq \Spec A_0 \times Z_i$.  Then Lemma \ref{lem:two restricted structures on A_0 are iso} (based on \cite[Proposition 3.4]{bk}) says that any two restricted symplectic structures on $\Spec A_0 \times Z_i$ as a $Z_i$-scheme are isomorphic fppf-locally.  
\end{proof}
This allows to view restricted Poisson structures on $X$ as torsors over a certain affine group scheme over $X^\lp$ (namely, the automorphisms $(\Spec A_0,[\eta_{\mathsf{can}}])$). However, not every such torsor would correspond to a symplectic structure on $X$: for example the trivial torsor rather corresponds to the class $[\eta_{\mathsf{can}}]$ on $\Spec A_0 \times X^\lp$. This discrepancy is fixed by adding the datum of an isomorphism of the underlying $\Aut(\Spec A_0)$-torsor with the torsor of ``Frobenius-frames'' on $X$ (defined below).  

\begin{constr}[Group scheme $G_0$]\label{constr:G_0}
	Consider the functor on $\mbb F_p$-algebras given by
$$
	A_0\colon R\mapsto A_0(R)\coloneqq A_0\otimes_{\mbb F_p}R.$$
 For any $R$, we have $\Omega^1_{A_0(R)/R}\simeq \Omega^1_{A_0}\otimes_{\mbb F_p}R$ and the pair $$(\Spec A_0(R),[\eta_{\mathsf{can}}]):=(\Spec A_0(R),[\eta_{\mathsf{can}}\otimes 1])$$ defines a restricted symplectic $R$-scheme. 
 
 Let $H=\ul{\mr{Aut}}(A_0)$ be the functor sending an $\mbb F_p$-algebra $R$ to
	$$
	\mr{Aut}_{R}(A_0(R))\simeq \mr{Aut}_R(\Spec A_0(R))^{\op},
	$$
	the group of $R$-linear \textit{algebra automorphisms} of $A_0(R)$. The group-valued functor $H$ is represented by the affine subgroup scheme of $\GL(A_0)$ given by those linear automorphisms that preserve multiplication on $A_0$. 
    
    Let $G_0\subset \ul{\mr{Aut}}(A_0)$ denote the functor sending an $\mbb F_p$-algebra $R$ to
	$$
	\mr{Aut}_R^{\mathsf{rest}}((\Spec A_0(R),[\eta_{\mathsf{can}}]))^{\mathrm{op}}\subset \mr{Aut}_R(\Spec A_0(R))^{\mathrm{op}}, 
	$$
	the subgroup of $R$-algebra automorphisms of $A_0(R)$ that also preserve the restricted symplectic structure given by $\eta_{\mathsf{can}}$. It is represented by a group subscheme of $\ul{\mr{Aut}}(A_0)$: one adds equations on $\phi\in \ul{\mr{Aut}}(A_0)$ saying that the class $[\eta_{\mathsf{can}}]\in \Omega^1_{A_0(R)}/d(A_0(R))\simeq (\Omega^1_{A_0}/dA_0)\otimes R$ is preserved by $\phi$.
	
\end{constr}

\begin{constr}[Torsor of ``Frobenius frames'']\label{constr:torsor of Frobenius-frames} Let $X$ be a pseudo-smooth scheme over a base scheme $S$ such that $\rk {\Omega^1_{X/S}}=2d$ . Let $A_0$ be as above and consider the reduced Frobenius map $F_X^\lp\colon X\ra X^\lp$ (Construction \ref{constr:reduced Frobenius}).
	The pseudo-smooth scheme $X$ defines a canonical $\ul{\mr{Aut}}(A_0)$-torsor $\mc M_X$ on $X^\lp$ in the following way. Namely, for a map $T\ra X^\lp$ one puts
	$$
	\mc M_X(T)\coloneqq \mr{Iso}_{T}(\Spec A_0\times T, X\times_{X^\lp} T).
	$$
	This is a sheaf in the flat topology on the category of affine schemes over $X^\lp$, endowed with a natural left action\footnote{More precisely, for $a\in \ul{\mr{Aut}}(A_0)$, write
$\sigma_a\coloneqq \operatorname{Spec}(a)$, so that
$\sigma_{ab}=\sigma_{b}\circ\sigma_{a}$.
The left $\ul{\mr{Aut}}(A_0)$-action on frames is given by $
a\colon m\mapsto m\circ\sigma_a.$} $\ul{\mr{Aut}}(A_0)\curvearrowright \mc M_X$. Moreover, by Remark \ref{lem:fiber product of reduced Frobenius with itself}, there exists an fppf cover $Z\ra X^\lp$, so that $X\times_{X^\lp} Z$ is isomorphic to $\Spec A_0\times Z $. Thus, $\mc M_X$ is in fact an $\ul{\mr{Aut}}(A_0)$-torsor over $X^{\{p\}}$. In \cite{bk} it is called the \textit{torsor of Frobenius frames on $X$}. The torsor $\mc M_X$ is classified by a canonical map $X^\lp\ra B(\ul{\mr{Aut}}(A_0))$.
\end{constr}
There is a natural restricted symplectic version of $\cM_X$ given by the torsor of ``Darboux frames''. 
\begin{constr}[Torsor of ``Darboux frames'']\label{constr:torsor of Darboux frames} Let $(X,[\eta])$ be a restricted symplectic $S$-scheme such that $\rk {\Omega^1_{X/S}}=2d$. Consider a functor $\mc M_{X,[\eta]}$ on affine schemes over $X^\lp$, sending $T\ra X^\lp$ to 
	$$
	\mc M_{X,[\eta]}(T)\coloneqq \mr{Iso}_{T}^{\mathsf{rest}}((\Spec A_0,[\eta_{\mathsf{can}}])\times T, (X,[\eta])\times_{X^\lp}T),
	$$
    where $\mr{Iso}_{T}^{\mathsf{rest}}$ is the set of isomorphisms as restricted symplectic $T$-schemes. The functor $\mc M_{X,[\eta]}$ is a flat sheaf, and has a natural $G_0$-action. By Darboux lemma (Proposition \ref{pr:BK's Darboux lemma}) it is in fact a $G_0$-torsor, which is called the \textit{torsor of Darboux frames} associated to $(X,[\eta])$ in \cite{bk}. Similarly, the torsor $\mc M_{X,[\eta]}$ is classified by a map $X^\lp \ra BG_0$.
\end{constr}

\begin{constr}[Restricted symplectic structures as reductions of $\mc M_X$ to $G_0$]\label{constr:symplectic structures as reductions}
	
    We now recall the bijection between the restricted symplectic structures on $X$ and reductions of the $\ul{\mr{Aut}}(A_0)$-torsor $\mc M_X$ to the subgroup $G_0\subset \ul{\mr{Aut}}(A_0)$. 

 Recall that such a reduction is by definition a pair $(\mc M',\alpha)$ of a $G_0$-torsor $\mc M_X'$ together with an isomorphism
 $$
\alpha\colon \ul{\mr{Aut}}(A_0)\times^{G_0} \mc M'\simeq \mc M_X,
 $$
 where $\ul{\mr{Aut}}(A_0)\times^{G_0} \mc M'$ is the quotient of $\ul{\mr{Aut}}(A_0)\times \mc M'$ by the (free) action of $G_0$ via $(a,m)\mapsto (ag^{-1},gm)$. Given $(\mc M',\alpha)$,  we can take a flat cover $Z\ra X^\lp$, such that $X\times_{X^\lp}Z\simeq \Spec A_0\times Z$ and descend the restricted symplectic $Z$-scheme given by $(\Spec A_0,[\eta_{\mathsf{can}}])\times Z$ to $X^\lp$ using the descent data provided by $\mc M'$ and $\alpha$. The result will be a symplectic $X^\lp$-scheme, whose underlying scheme is $X$. 
 
    Conversely, given $(X,[\eta])$ there is a natural identification 
    $$
\alpha_{[\eta]} \colon \ul{\mr{Aut}}(A_0)\times^{G_0} \mc M_{X,[\eta]} \simeq \mc M_{X}.
    $$
 Indeed, given an isomorphism $m\colon (\Spec A_0,[\eta_{\mathsf{can}}])\times T \iso (X,[\eta])\times_{X^\lp}T$ and an automorphism $\sigma_a\coloneqq\Spec(a)\colon \Spec A_0 \times T \iso \Spec A_0 \times T$ induced by $a\in \ul{\mr{Aut}}(A_0)(T)$, there is a natural isomorphism $m\circ \sigma_a\colon  \Spec A_0 \times T \iso X \times_{X^{\lp}} T$, and the corresponding map $\ul{\mr{Aut}}(A_0)\times \mc M_{X,[\eta]} \ra \mc M_{X}$ factors through $\ul{\mr{Aut}}(A_0)\times^{G_0} \mc M_{X,[\eta]}$.
	
\end{constr}
\begin{rem}\label{rem:groups H and G_0}
    For the reader's convenience let us briefly describe the structure of the group schemes $H:=\ul{\mr{Aut}}(A_0)$ and $G_0$. Let $I_0=(x_1,\ldots,x_d,y_1,\ldots,y_d)\subset A_0$ be the maximal ideal and let $H^0\subset H$ be the stabilizer of $I_0$: namely, $R$-points $H^0(R)$ are given by those $\phi\in \mr{Aut}_{R}(A_0(R))$ that preserve $I_0(R):= I_0\otimes R$. The subgroup $H^0$ is smooth: indeed, one has a surjection $\theta\colon H^0 \twoheadrightarrow \GL(2d)$ given by sending $\phi$ as above to the induced $R$-linear automorphism of $I_0(R)/I_0(R)^2\simeq \Omega^1_{A_0}\otimes_{A_0}\mbb F_p\otimes_{\mbb F_p} R$. The kernel of $\theta$ is a smooth unipotent algebraic group; it has a filtration with the associated graded pieces given by vector group schemes of the form $\ul{\Hom}(I_0/I_0^2, I_0^r/I_0^{r+1})$ for $r>1$. Moreover, there is a natural subgroup $\alpha:=\alpha_p^{2d}\subset H$ that acts on $\Spec A_0\simeq \alpha$ by translations: given $(\epsilon_i,\gamma_j)_{1\le i,j\le d}\in \alpha_p(R)^{2d}$ we have an automorphism of $A_0(R)$ sending $x_i\mapsto x_i+\epsilon_i$ and $y_j\mapsto y_j+\gamma_j$. The multiplication map induces an isomorphism 
    $$
m\colon \alpha \times H^0 \simeq H,
    $$
    but only as schemes. In particular, $H^0$ is the maximal reduced subgroup $H^{\mr{red}}$ of $H$. The Lie algebra of $H$ (resp. $H^0$) is given by the Lie algebra of vector fields $T_{A_0}$ (resp. the subalgebra of vector fields that preserve $I_0$).

    Let $G^0_0:=G_0\cap H^0$; equivalently, $G_0^0\subset G_0$ is the stabilizer of the
origin of $\Spec A_0$. By \cite[Proposition 3.4]{bk} this intersection is still smooth. By \cite[Lemma 4.1]{bv} restriction of $\theta$ to $G_0^0\subset H^0$ induces a surjective map 
    $$
G_0^0 \twoheadrightarrow \Sp(2d) 
    $$
(with the symplectic form induced by $\omega=d\eta_{\mathsf{can}}$), and the kernel is a smooth unipotent group scheme whose structure is discussed in detail in \cite[Section 4]{bv}. The subgroup $\alpha$ belongs to $G_0$, and we still have an isomorphism $\alpha\times G_0^0\simeq G_0$ as schemes. Note that $G_0^0$,
and therefore $G_0$, is geometrically connected. By \cite[Lemma 3.3]{bk} the Lie algebra $\Lie(G_0)\subset \Lie(H)=T_{A_0}$ is identified with the subalgebra of Hamiltonian vector fields (in other words, the image of the map $A_0\ra T_{A_0}$ sending $f\mapsto H_f=\{f,-\}$).
\end{rem}

\begin{rem}Let $\upsilon\colon X^\lp \ra BH$ be the map classifying the torsor $\mc M_X$ of Frobenius frames. Many differential-geometric structures on $X$ can be viewed as pull-backs of their universal versions along $\upsilon$. For example:
\begin{enumerate}
    \item The ring $A_0$ is naturally a ring object in the category of $H$-representations. One can consider the corresponding affine scheme $\Spec_{BH}(A_0)$ over $BH$. Then, more or less by definition of $\mc M_X$, the pull-back of $\Spec_{BH}(A_0) \ra BH$ along $\upsilon$ is identified with the reduced Frobenius map $X\ra X^\lp$. Moreover, $\Spec_{BH}(A_0) \ra BH$ is naturally identified with the map $BH^0\ra BH$; indeed, $\Spec_{BH}(A_0)$ is the quotient of $\Spec A_0$ by the natural action of $H$, which, since the action of $\alpha\subset H$ on $\Spec A_0=\alpha$ is free, is given by the classifying stack of the stabilizer of $0\in \Spec A_0$, namely $H^0$.
    \item Vector spaces $A_0$, $T_{A_0}$, $\Omega^1_{A_0}$, $T_{A_0}^{\lp}$, $\Omega^{1,\lp}_{A_0}$ and $\mc D_{A_0}$ all have a natural $H$-representation structure. Pull-backs along $\upsilon$ of the corresponding quasi-coherent sheaves on $BH$ are given by $F_{X*}^\lp\mc O_X$, $F_{X*}^\lp T_X$, $F_{X*}^\lp\Omega^1_X$, $T_{X}^{\lp}$, $\Omega^{1,\lp}_X$ and $F_{X*}^\lp\mc D_X$ correspondingly.
    \item Similarly to (1), the affine map $BG^0_0\ra BG_0$ can be considered as the universal restricted symplectic scheme with $\rk \Omega^1=2d$. Indeed, $BG^0_0$ is identified with the quotient stack $[\Spec A_0/G_0]\simeq \Spec_{BG_0}(A_0)$, and since $G_0$ preserves $[\eta_{\mathsf{can}}]$ it is endowed with the structure\footnote{Meaning that the pull-back along any $f\colon \Spec R \ra BG_0$ has the structure of a restricted symplectic $R$-scheme functorially in $f$.} of a restricted symplectic scheme over $BG_0$. Given a restricted symplectic $S$-scheme $(X,[\eta])$ and the corresponding map $\xi\colon X^\lp \ra BG_0$, pull-back of $BG^0_0\ra BG_0$ along $\xi$ is identified with $(X,[\eta])$ as a restricted symplectic $X^\lp$-scheme. 
    \item Pulling back the map of stacks $\Spec \mbb F_p \to BG_0^0\to BG_0$ via the map $X^\lp \ra BG_0$ that classifies a symplectic restricted structure $[\eta]$ on $X$,
defines a morphism
\[
\mathcal M_{X,[\eta]}\rightarrow X,
\]
which makes $\mathcal M_{X,[\eta]}$ a $G_0^0$-torsor over
$X$.
\end{enumerate}
\end{rem}

\begin{rem}[Functoriality of $\mc M_{X,[\eta]}$]\label{rem:base change of the Darboux torsor}
  Let $(X,[\eta])$ be a restricted symplectic $S$-scheme, and let $f\colon S'\to S$ be a morphism of schemes. The pull-back $[\eta]':=f^*[\eta]$ for the projection $X':=X\times_S S'\to X$ endows the base change $X'$ with the structure of a restricted symplectic $S'$-scheme. Note that $\rk(\Omega^1_{X'/S'})=\rk(\Omega^1_{X/S})$, since $\Omega^1_{X'/S'}=f^*\Omega^1_{X/S}$. By Lemma \ref{lem:reduced Frobenius is respected by pull-back}, the natural map $X'\to (X')^{\lp}\times_{X^\lp}X$ is an isomorphism and this produces a natural $G_0$-equivariant map
  \begin{equation}\label{eq:pull-back of M}
  \mc M_{X',[\eta]'} \to \mc M_{X,[\eta]}\times_{X^{\lp}} (X')^{\lp};
  \end{equation}
  since both sides are $G_0$-torsors over $(X')^{\lp}$ it is automatically an isomorphism. 

Similarly, let $f\colon X'\to X$ be an \'etale map of $S$-schemes. Given a restricted symplectic structure $[\eta]$ on $X$ the pull-back $[\eta]':=f^*[\eta]$ defines a restricted symplectic $S$-scheme structure on $X'$ (here we use that $\Omega^1_{X'/S}=f^*\Omega^1_{X/S}$). By Lemma \ref{lem:twist vs etale maps}, there is again a $G_0$-equivariant morphism \eqref{eq:pull-back of M}, which is automatically an isomorphism.
  
In both cases we get that the map $[\eta]'\colon (X')^{\lp} \to BG_0$, classifying the torsor $\mc M_{X',[\eta]'}\to (X')^{\lp}$, is obtained by precomposing the map $[\eta]\colon X^\lp \to BG_0$ with $f^\lp\colon (X')^{\lp}\to X^\lp$. 
\end{rem}

Later we will need the following lemma:
\begin{lm}\label{lem:polarization stabilizer is connected}
    Let $H_{\mr{pol}}\subset G_0$ denote the subgroup scheme of restricted automorphisms of $A_0$ that carry $C_0= \bF_p[x_1, \cdots , x_d]/(x_i^p)_{1\leq i \leq d}\subset A_0$ to itself. Then $H_{\mr{pol}}$ is connected. 
\end{lm}
\begin{proof}
    First of all, note that $p\colon H_{\mr{pol}}\to \underline{\mr{Aut}}(C_0)$ is surjective: indeed $\Spec A_0$ is identified with the Frobenius-neighborhood of the zero section in $T^*\Spec C_0$, and so there is a natural action of $\mr{Aut}(C_0)$ on it, and it also maps $C_0\subset A_0$ to itself. By Remark \ref{rem:groups H and G_0} one sees that $\underline{\mr{Aut}}(C_0)$ is connected, so it is enough to show that $K:=\Ker(p)$ is connected. Let $\phi\in K(R)$, so it is a restricted $R$-linear automorphism of $A_0(R)$ such that $\phi(x_i)=x_i$ for all $i$. We have 
    $\phi^*\eta_{\mathsf{can}}-\eta_{\mathsf{can}}= df$; since $\phi(x_i)=x_i$, 
    $$\phi^*\eta_{\mathsf{can}}-\eta_{\mathsf{can}}= \sum_{i=1}^d(\phi^*(y_i) -y_i)dx_i,$$
    from which we get that $\partial_{y_i}f=0$ for all $i$ (so $f\in C_0(R)$) and $\phi^*(y_i)=y_i +\partial_{x_i}f$. We get that $\phi$ is uniquely determined by the class $\overline{f}\in C_0(R)/R$ of $f$ modulo constants, and the only condition on $\overline{f}$ that we need to check is that $\phi^*(y_i)^p=0$, which is equivalent to $(\partial_{x_i}f)^p=0$ for all $i$. This way $K(R)$ can be identified with the subgroup in $C_0(R)$ given by $f=\sum_{I}c_I x^I$ where constant term is 0 and all the coefficients in the linear part $\sum_i c_ix_i$ satisfy $c_i^p=0$. This identifies $K$ with the product of an affine space and $d$ copies of $\alpha_p$, and so $K$ is connected. 
\end{proof}

We end the section with the following \'etale version of the Darboux lemma, which is a partial strengthening of Proposition \ref{pr:BK's Darboux lemma}.
\begin{pr}[\'Etale Darboux lemma]\label{prop:etale Darboux lemma}
    Let $(X,[\eta])$ be a restricted symplectic $S$-scheme with $\rk(\Omega^1_X)=2d$. Then there exists an \'etale cover $q\colon U\twoheadrightarrow X$ and a map
    $$
p\colon U\ra T^*\mbb A^d_{\mbb F_p}=\mbb A^{2d}_{\mbb F_p},
    $$
    such that $q^*[\eta]=p^*[\eta_{\mathsf{can}}]$. In other words, there exists a set of functions $f_i,g_i\in \mc O(U)$, $1\le i\le d$ so that $q^*[\eta]= [\sum_{i=1}^d f_idg_i]$. 
\end{pr}
\begin{proof}

Let $\mbb A^{2d}:=\mbb A^{2d}_{\mbb F_p}$. Note that $\Spec A_0$ is identified with $(F_{\mbb A^{2d}})^{-1}(0)$: the Frobenius-neighbourhood of $0$ inside $\mbb A^{2d}$. An $R$-point of $\mc M_{X,[\eta]}$ lying over a map $\Spec R \ra X^\lp$ can be equivalently described as the datum of an embedding of $R$-schemes $\beta\colon X\times_{X^\lp}\Spec R\hookrightarrow \mbb A^{2d}\times \Spec R$, such that the schematic image is identified with $(F_{\mbb A^{2d}})^{-1}(0)\times \Spec R$, and such that $\mr{pr}_1^*[\eta]= \beta^*[\sum_i{y_id{x_i}}]$, where $\mbb A^{2d}=\Spec \mbb F_p[x_i,y_i]_{1\le i\le d}$. Consider the contracted product $Z:=\mbb A^{2d}\times^{\alpha}\mc M_{X,[\eta]}$, where $\alpha=\alpha_p^{2d}\subset G_0$ is the subgroup of translations  (see Remark \ref{rem:groups H and G_0}), and $\alpha$ acts by translations on $\mbb A^{2d}$ via the embedding $\alpha\hookrightarrow \mbb G_a^{2d}\simeq \mbb A^{2d}$. We have a natural map $Z\ra [\alpha \backslash \mc M_{X,[\eta]}]\ra X^\lp$, which is a smooth cover: indeed the first map is a (fppf) locally trivial fibration with fiber an affine space, and the second is a locally trivial fibration with the smooth fiber $G_0/\alpha\simeq G_0^{\mr{red}}$ (see Remark \ref{rem:groups H and G_0}). We claim that an $R$-point of $Z$ lifting a map $\Spec R \ra X^\lp$ is given by a closed embedding $\beta\colon  X\times_{X^\lp}\Spec R\hookrightarrow \mbb A^{2d}\times \Spec R:=\mbb A^{2d}_R$ such that $\mr{pr}_1^*[\eta]=\beta^*[\sum_i{y_id{x_i}}]$. Indeed, let $T$ be a functor encoding the latter data. There is a natural map $\mbb A^{2d}\times \mc M_{X,[\eta]}\ra T$ sending a pair $(a,\beta_0)$ to the parallel transport $\beta_0+a$ of $\beta_0$, and which identifies $T$ with the quotient of $\mbb A^{2d}\times \mc M_{X,[\eta]}$ by the free action of $\alpha$ that acts as $(a,\beta_0)\mapsto (a-t, \beta_0+t)$; the latter quotient is exactly $Z$. Here note that the parallel transport map $+a \colon \mbb A^{2d}_R \ra \mbb A^{2d}_R$ indeed preserves the class $[\eta]$, so the map to $T$ is well-defined.

Since $\pi\colon Z\ra X^\lp$ is smooth, there exists an \'etale covering $W=\sqcup_i\Spec R_i\ra X^\lp$ with a section $s\colon W\rightarrow Z$. By the above discussion it produces a map 
$$
\beta\colon U:= X\times_{X^\lp} W \ra \mbb A^{2d}\times W \ra \mbb A^{2d}. 
$$
The projection $q\colon U\ra X$ is an \'etale cover, and $q^*[\eta]= [\sum_{i=1}^d f_idg_i]$ with $f_i=\beta^*(y_i)$ and $g_i=\beta^*(x_i)$.
\end{proof}

The following lemma was used in the proof of Proposition \ref{pr:BK's Darboux lemma}.
\begin{lm}\label{lem:two restricted structures on A_0 are iso}
    Given a commutative $\mbb F_p$-algebra $R$, any two restricted symplectic $R$-scheme structures on $\Spec A_0(R)$ are isomorphic fppf locally on $R$. 
\end{lm}

\begin{proof}
    For an $\mathbb F_p$-vector space $V$ denote by $\ul{V}:=\Spec (\Sym_{\mbb F_p}V^\vee)$ the corresponding vector scheme. Let $X$ denote the functor on commutative $\mbb F_p$-algebras sending a commutative ring $R$ to the set of  restricted symplectic structures on the $R$-scheme $\Spec A_0(R)$.  It is naturally represented by a finite type $\mbb F_p$-scheme: namely it is the preimage of the locus $(\ul {\Omega^{2}_{A_0,cl})}^{non-deg}\subset \ul{\Omega^2_{A_0, \mr{cl}}}$ of non-degenerate closed 2-forms under the map $d\colon\ul{\Omega^1_{A_0}/d(A_0)} \ra 
    \ul {\Omega^2_{A_0, \mr{cl}}}$. Note that $X$ is an affine bundle over an open inside an affine space, and so is regular. We have a natural action $H=\ul{\mr{Aut}}(A_0)\curvearrowright X$ and it is enough to see that the natural map $\alpha\colon H\times X\ra X\times X$, $(h,x)\mapsto (h(x),x)$ is an fppf surjection. Let $(\eta_1,\eta_2)\colon \Spec k\rightarrow X\times X$ be a geometric point. By \cite[proof of Proposition 3.4]{bk} (applied to $B=k$) the fiber $\alpha^{-1}(\eta_1,\eta_2)$ is a $G_0$-torsor over $k$. In particular, the map $\alpha$ is surjective, and the dimension of a fiber over any geometric point  does not depend on the point, and is equal to $\dim(H\times X) - \dim(X\times X)$. Since $H\times X$ is a local complete intersection\footnote{Note that any finite type group scheme over a field is lci, see \cite[Chapter III, \S 3, $n^\circ\ \! 6$]{gd}} (in particular Cohen-Macaulay) and $X\times X$ is regular, $\alpha$ is also flat by the ``miracle flatness'' theorem  \cite[Tag 00R4]{stacks}.
\end{proof}

\subsection{Restricted Lagrangian subschemes of restricted symplectic schemes}
Let $(X,[\eta])$ be a restricted symplectic $S$-scheme with $\rk(\Omega^1_{X/S})=2n$.
\begin{df}\label{def:restricted Lagrangian} A pseudo-smooth $S$-subscheme $i\colon Y\hookrightarrow X$ is a \textit{restricted Lagrangian subscheme} if 
\begin{enumerate}[label=(\arabic*)]
    \item $\rk(\Omega^1_{Y/S})=n$;
    \item $i^*[\eta]=0$.
\end{enumerate}
    
\end{df}
\begin{rem}Let $\omega=d[\eta]\in H^0(\Omega^2_{X/S})$.
    Note that from condition (2) it also follows that $i^*\omega=0$, so $Y$ is also a Lagrangian in a more conventional sense.
\end{rem} 
\begin{example}\label{ex:Lagrangian in Frobenius-disk}
    Let $A_0:= \mbb F_p[x_i,y_j]_{1\le i,j\le d}/ x_i^p=y_j^p=0$ with the canonical restricted symplectic structure $[\eta_{\msf{can}}]$, and consider a surjection $A_0\twoheadrightarrow A_0/J:=\mbb F_p[x_i]_{1\le i\le d}/x_i^p=0$, where $J\subset A_0$ is the ideal generated by $y_i$'s. Then $\Spec A_0/J\hookrightarrow \Spec A_0$ is an example of a restricted Lagrangian subscheme. 
\end{example}

Further we will need a variant of Darboux lemma for pairs consisting of a restricted symplectic scheme and a restricted Lagrangian subscheme. In the smooth case over a field this result is due to Mundinger \cite[Theorem 2.10]{Mu}, and the proof in the general pseudo-smooth setting is similar.

\begin{pr}[Darboux lemma for restricted Lagrangians]\label{prop:Darboux lemma for Lagrangians} Let $i\colon Y\hookrightarrow X$ be a restricted Lagrangian subscheme of a restricted symplectic $S$-scheme $(X,[\eta])$. Then, there exists an fppf cover $T\to Y^{\lp}$, together with an isomorphism of restricted symplectic $T$-schemes
$$
\alpha\colon T\times_{X^\lp} (X,[\eta])\xra{\sim}  T\times (\Spec A_0,[\eta_{\msf{can}}]),
$$
 which identifies the closed subscheme $T\times_{Y^{\lp}}Y\hookrightarrow T\times_{X^{\lp}}X$ with    $T\times \Spec A_0/J\hookrightarrow T\times \Spec A_0$ (with notation from Example \ref{ex:Lagrangian in Frobenius-disk}).
\end{pr}
\begin{proof}
    By taking the common refinement of covers from Remark \ref{lem:fiber product of reduced Frobenius with itself} and the Darboux lemma (Proposition \ref{pr:BK's Darboux lemma}), we can find an fppf cover by disjoint union of affine schemes $\sqcup_{i\in I}\Spec R_i\to Y^\lp$ such that $\Spec R_i\times_{X^{\lp}} (X,[\eta])\simeq \Spec A_0(R_i)$, and $\Spec R_i\times_{Y^{\lp}} Y$ is abstractly isomorphic to $\Spec C_0(R_i)$, where $C_0:=\mbb F_p[t_1,\ldots,t_n]/(t_1^p,\ldots,t_n^p)$ and $C_0(R_i):=C_0\otimes_{\mbb F_p}R_i$. Thus, it is enough to show that, given a ring $R$ and a surjection $q\colon A_0(R)\twoheadrightarrow C_0(R)$ defining a restricted Lagrangian subscheme, there is an element $g\in G_0(R)$ such that $\Ker(q\circ g)$ is identified with the ideal $J=(y_1,\ldots, y_n)$. Let $I_0(R)\subset A_0(R)$ and $N_0(R)\subset C_0(R)$ be the ideals of the zero sections in the respective schemes (so $I_0(R)$ is generated by all $x_i$'s and $y_j$'s and $N_0(R)$ by $t_i$'s). Composing with a suitable translation $x_i\mapsto x_i + \epsilon_i$, $y_j\mapsto y_j+\gamma_j$ in $G_0(R)$ (see Remark \ref{rem:groups H and G_0}) we can assume that $q(I_0(R))\subset N_0(R)$ or, in other words, that the map $\Spec C_0(R)\hookrightarrow \Spec A_0(R)$ sends zero to zero. Since the map $q$ defines a Lagrangian subscheme, the induced map $dq^*_{|\{0\}}\colon R^{2n}\simeq (I_0/I_0^2)(R)\to (N_0/N_0^2)(R)\simeq R^n$ on the conormal bundles gives a point of the Lagrangian Grassmanian for $(I_0/I_0^2)(R)$ endowed with the symplectic form $\omega^{-1}_{|\{0\}}$. Then there exists an element $\sigma\in \Sp_{2n}(R)\subset G_0^0(R)$ which makes it into a standard projection $R\cdot dx_1\oplus\ldots\oplus R\cdot dx_n\oplus R\cdot dy_1\oplus\ldots\oplus R\cdot dy_n \twoheadrightarrow R\cdot dx_1\oplus \ldots \oplus R\cdot dx_n $; in other words, after precomposing with $\sigma$ we have $q(x_i)\equiv t_i \! \mod N_0(R)^2$ and $q(y_i)\in N_0(R)^2$ for all $i$. Applying an endomorphism of $A_0(R)$ induced by an action on $R[x_i]_{1\le i\le n}/(x_i^p)\subset A_0(R)$ (as in Lemma \ref{lem:polarization stabilizer is connected}), we can actually assume that $q(x_i)= t_i$ on the nose. It remains to modify $y_i$ so that $q(y_i)$ is actually 0. Since $q$ defines a \textit{restricted} Lagrangian there exists an element $f(t):=f(t_1,\ldots,t_n)\in C_0(R)$ such that $q^*\eta=df$; in other words,
    $$
\sum_{i=1}^n q^*(y_i)\cdot dt_i = \sum_{i=1}^n \partial_{t_i}\! (f(t))\cdot dt_i.
    $$
   Now consider $f(x)$ as an element of $A_0(R)$, and take an $R$-linear automorphism $\phi$ of $A_0(R)$ sending $x_i\mapsto x_i$ and $y_i\to y_i-\partial_{x_i}(f(x))$. Note that $\phi\in G_0(R)$ since $\phi^*(\eta)-\eta =-df(x)$. On the other hand, precomposing $q$ with $\phi$ now sends $y_i$ to 0, as we wanted. 
\end{proof}

\begin{constr}\label{constr:stabilizer in G_0 of a Lagrangian} Let $G_{0,J}\subset G_0$ denote the group subscheme, for which the subgroup $G_{0,J}(R)\subset G_0(R)$ is given by those restricted symplectomorphisms $\phi$ of $(\Spec A_0(R),[\eta_{\msf{can}}])$ for which $\phi(J)=J$. By Proposition \ref{prop:Darboux lemma for Lagrangians}, given a restricted Lagrangian $Y\subset X$ the scheme 
$$\mc M_{X,Y,[\eta]}\to Y^{\lp},$$ whose $R$-points are given by the datum of a point $\Spec R\to Y^\lp$ and an isomorphism of restricted $R$-schemes
$$
\Spec R\times_{X^\lp} (X,[\eta])\xra{\sim}  (\Spec A_0(R),[\eta_{\msf{can}}])
$$
that identifies $\Spec R\times_{Y^\lp}Y\subset \Spec R\times_{X^\lp} (X,[\eta])$ with $\Spec (A_0/J\otimes_{\mbb F_p}R)\hookrightarrow \Spec A_0(R)$, defines a $G_{0,J}$-torsor over $Y^\lp$.
\end{constr}

The following will be important later:

 \begin{lm}\label{app.connectedness}
 The group $G_{0, J}$  is connected.
 \end{lm}
 \begin{proof}
 Let $G_{0, J}^0\subset G_{0, J}$ be the stabilizer of $0\in \Spec A_0$.  Since   $(G_{0, J}^0)_{red} = (G_{0, J})_{red}$ it suffices to prove that  $G_{0, J}^0$ is connected.  The subgroup $\bG_m \subset \underline{\Aut}( A_0)$ of homotheties (namely, automorphisms $s_\lambda\in  \underline{\Aut}( A_0)$ of the form
$s_\lambda(x_i)=\lambda \cdot x_i$,
$s_\lambda(y_i)=\lambda \cdot y_i$) normalizes $G_{0, J}^0$. 
Moreover, the map $
(\lambda,g)\longmapsto
s_\lambda\circ g\circ s_\lambda^{-1}
$
extends to a morphism
$\Phi\colon \bA^1 \times G_{0, J}^0 \to G_{0, J}^0$. The restriction $\Phi_{|0\times G_{0, J}^0}$ carries  $G_{0, J}^0$ to the subgroup $P\subset \Sp(2n)\subset  G_{0, J}^0$ of linear symplectic transformations preserving $J$, which is connected. The lemma follows. 
 \end{proof}

\section{Sheaves of algebras $\mc A_{{\mathbb P^1}}$ and $\mc A_{{\mathbb P^1}}^\flat$ and their categories of modules}\label{sec: algebra A over projective line}

 For each dimension parameter $n$, we will define a quasi-coherent sheaf of algebras on ${\mathbb P^1}$: that is, a $\mbb P^1$-version  $\mc A_{{\mathbb P^1}}:=\mc A_{{\mathbb P^1},n}$ of the Rees construction of the \textit{reduced Weyl algebra} on $n$ pairs of generators, which encodes the Hodge and conjugate filtrations on the latter. Following \cite{bv} we will construct an $\mc O_{\mathbb P^1}$-linear action of $G_0$ on $\mc A_{{\mathbb P^1}}^\flat-\Mod$, where $\mc A_{{\mathbb P^1}}^\flat:=\mc A_{{\mathbb P^1},2n}$. This will be used later to construct the canonical action of $G_0$ on the category $\mc A_{{\mathbb P^1}}-\Mod$.

\subsection{Rees construction over $\mbb P^1$}\label{ssec:Rees construction}  First let us recall the usual Rees construction and its properties.

Let $\mc B$ be a quasicoherent sheaf of associative (but not necessarily commutative) algebras over a base scheme $Y$. Let $\mc B_{\le *}=\ldots \subset \mc B_{\le n} \subset \mc B_{\le n+1}\subset \ldots$ be an increasing filtration on $\mc B$ by quasi-coherent sheaves, indexed by integers. We assume that the filtration is unital ($1\in \mc B_{\le 0}$), multiplicative ($\mc B_{\le i}\cdot \mc B_{\le j}\subset \mc B_{\le i+j}$) and exhaustive (namely $\mc B=\colim_{n} \mc B_{\le n}$). 

\begin{constr}[{Rees construction over $\mbb A^1$}]\label{constr:Rees construction over A^1} Let $h$ be a formal variable. Let $\mc B_{\le *}$ be as above.

Define \textit{the Rees construction} $\mc B_{\mbb A^1}$ of the filtered algebra $\mc B_{\le *}$ as a subsheaf of algebras
$$
\mc B_{\mbb A^1}\coloneqq \oplus_i \mc B_{\le i}\cdot h^i \subset \mc B[h,h^{-1}]\coloneqq \mc B\otimes_{\mbb Z}\mbb Z[h,h^{-1}].
$$
Since $\mc B_{\le i}\cdot \mc B_{\le j}\subset \mc B_{\le i+j}$ and $\mc B_{\le i}\subset \mc B_{\le i+1}$, the multiplication on $\mc B_{\mbb A^1}$ is well-defined and $\mc B_{\mbb A^1}$ defines a (quasicoherent) subsheaf of $\mc O_Y[h]$-algebras in $\mc B[h,h^{-1}]$. One also has isomorphisms 
\begin{equation}\label{eq:comparison Rees algebras}
\mc B_{\mbb A^1}/h\simeq \gr_\star\mc B \quad \text{and} \quad \mc B_{\mbb A^1}[h^{-1}]\simeq \mc B[h,h^{-1}].
\end{equation}
Here for the second isomorphism we used the exhaustiveness of the filtration.

Using a monoidal equivalence of categories of quasicoherent $\mc O_Y[h]$-modules and quasicoherent sheaves on $\mbb A^1\times Y$ we can view the $\mc O_Y[h]$-algebra $\mc B_{\mbb A^1}$ as a quasicoherent sheaf of algebras $\mc B_{\mbb A^1}$ over $\mbb A^1\times Y$. Isomorphisms in \eqref{eq:comparison Rees algebras} then give
$$
{{\mc B_{\mbb A^1}|}_{\{0\}\times Y}}\simeq \gr_\star\mc B \quad \text{and} \quad {\mc B_{\mbb A^1}|}_{\mbb G_m\times Y}\simeq \mc O_{\mbb G_m}\!\boxtimes \mc B.
$$

\end{constr}

\begin{constr}[A variant over the center]\label{constr:} Let $Z(\mc B)\subset \mc B$ be the center and assume that it is also a quasi-coherent sheaf. Define the filtration $Z(\mc B)_{\le *}\coloneqq Z(\mc B)\cap \mc B_{\le *}$. This filtration is also multiplicative and exhaustive and we can consider the corresponding Rees construction $Z(\mc B)_{\mbb A^1}$. Then $\mc B_{\mbb A^1}$ is naturally a sheaf of algebras over $Z(\mc B)_{\mbb A^1}$ and as such also defines a quasi-coherent sheaf of algebras over the relative spectrum $\Spec_{\mbb A^1\times Y}(Z(\mc B)_{\mbb A^1})$.
\end{constr}

\begin{constr}[{Descent to $[\mbb A^1/\mbb G_m]$}]\label{constr:Rees algebra on A^1/G_m}
Consider the standard action of $\mbb G_m$ on $\mbb A^1$ (such that $t\circ h=th$ for $t\in \mbb G_m$). It induces a $\mbb G_m$-action on $Y\times \mbb A^1$ and endows $\mc O_Y[h]$ with a $\mbb Z$-grading such that $\mc O_Y$ has grading 0 and the variable $h$ has grading 1. Consider the corresponding quotient stack $[\mbb A^1/\mbb G_m]\times Y$. Using faithfully flat descent the abelian category of quasicoherent sheaves on $[\mbb A^1/\mbb G_m]\times Y$ can be identified with the category of quasicoherent sheaves of graded $\mc O_Y[h]$-modules on $Y$.

The algebra $\mc B_{\mbb A^1}$ comes with a natural grading $(\mc B_{\mbb A^1})^i\coloneqq \mc B_{\le i}\cdot h^i$ that makes it into a graded $\mc O_Y[h]$-algebra. This way $\mc B_{\mbb A^1}$ defines a quasi-coherent sheaf of algebras $\mc B_{[\mbb A^1/\mbb G_m]}$  on $[\mbb A^1/\mbb G_m]\times Y$.

The descent above also has a variant over the center. Namely, $Z(\mc B)_{\mbb A^1}$ is also a graded $\mc O_Y[h]$-algebra, which endows $\Spec_{\mbb A^1\times Y}(Z(\mc B)_{\mbb A^1})$ with a natural $\mbb G_m$-action. Moreover, $\mc B_{\mbb A^1}$ is a graded $Z(\mc B)_{\mbb A^1}$-algebra and as such defines a quasi-coherent sheaf of algebras on the quotient stack $[\Spec_{\mbb A^1\times Y}(Z(\mc B)_{\mbb A^1})/\mbb G_m]$. 
\end{constr}

Now, having two multiplicative exhaustive filtrations on $\mc B$ we can glue the two corresponding Rees constructions into a sheaf over $\mbb P^1 \times Y$:

\begin{constr}[Rees construction over $\mbb P^1$]\label{constr:Rees construction over P^1} Let $\mc B$ be a quasi-coherent sheaf of algebras over $Y$ endowed with an exhaustive increasing filtration $\mc B_{\le *}$ and an exhaustive decreasing  filtration $\mc B^{\ge *}$, both unital and multiplicative. We will denote such data by $(\mc B, \mc B_{\le*},\mc B^{\ge *})$.

Cover $\mbb P^1$ by two charts given by affine lines $\mbb P^1\backslash \{\infty\}\simeq \Spec \mbb Z[h]$ and $\mbb P^1\backslash \{0\}\simeq \Spec \mbb Z[h^{-1}]$, with intersection $\mbb P^1\backslash (\{0\}\sqcup \{\infty\})\simeq \Spec \mbb Z[h,h^{-1}]$. On these charts take Rees constructions\footnote{Here, one can view $\mc B_{\mbb P^1\backslash \{0\}}$ as an instance of \ref{constr:Rees construction over A^1} for the increasing filtration given by $F_i\mc B:= \mc B^{\ge -i}$, and where the coordinate on $\mbb A^1$ is $h^{-1}$.} associated to the two filtrations:
$$
\mc B_{\mbb P^1\backslash \{\infty\}}\coloneqq \oplus_{i}\mc B_{\le i}\cdot h^i \subset \mc B[h,h^{-1}] \supset \mc B_{\mbb P^1\backslash \{0\}}\coloneqq \oplus_{i}\mc B^{\ge -i}\cdot h^{-i}.
$$
$\mc B_{\mbb P^1\backslash \{\infty\}}$ and $\mc B_{\mbb P^1\backslash \{0\}}$ are the $\mc O_Y[h]$ and $\mc O_Y[h^{-1}]$ subalgebras of $\mc B[h,h^{-1}]$, and as such define quasicoherent sheaves of algebras on $(\mbb P^1\backslash \{\infty\})\times Y$ and $(\mbb P^1\backslash \{0\})\times Y$ correspondingly. Moreover, we have natural isomorphisms 
$$
{\mc B_{\mbb P^1\backslash \{\infty\}|}}_{\mbb G_m\times Y}\simeq \mc O_{\mbb G_m}\boxtimes\mc B \simeq {\mc B_{\mbb P^1\backslash \{0\}|}}_{\mbb G_m\times Y}.
$$
Using the composite isomorphism as the gluing data we obtain a quasi-coherent sheaf of algebras on $\mbb P^1\times Y$ which we will denote $\mc B_{\mbb P^1}$ and call \textit{Rees construction over $\mbb P^1$}.

The above construction also localizes over the center of $\mc B$. Namely, $\mc B_{\mbb P^1}$ naturally defines a sheaf of algebras  on $\Spec_{\ \!\!\mbb P^1\times Y}Z(\mc B)_{\mbb P^1}$ which we will continue to denote $\mc B_{{\mathbb P^1}}$. Here $Z(\mc B)$ denotes the center of $\mc B$.
\end{constr}
\begin{constr}[{Descent to $[\mbb P^1/\mbb G_m]$}]\label{constr:Rees construction over P^1/G_m} Analogously to Construction \ref{constr:Rees algebra on A^1/G_m} algebras $\mc B_{\mbb P^1\backslash \{\infty\}}$ and $\mc B_{\mbb P^1\backslash \{0\}}$ are endowed with natural gradings $(\mc B_{\mbb P^1\backslash \{\infty\}})_i\coloneqq \mc B_{\le i}\cdot h^i$ and $(\mc B_{\mbb P^1\backslash \{0\}})_{i}\coloneqq \mc B^{\ge i}\cdot h^{i}$. This way they define quasi-coherent sheaves of algebras on $[\mbb P^1\backslash \{\infty\}/\mbb G_m]\times Y$ and $[\mbb P^1\backslash \{0\}/\mbb G_m]\times Y$. Moreover, the restrictions of both $\mc B_{[\mbb P^1\backslash \{\infty\}/\mbb G_m]}$ and $\mc B_{[\mbb P^1\backslash \{0\}/\mbb G_m]}$ to $[\mbb G_m/\mbb G_m]\times Y\simeq Y$ are naturally identified with $\mc B=\colim_{i\ra \infty} \mc B_{\le i} =\colim_{i\ra\infty} \mc B^{\ge -i}$. This way they glue to a quasi-coherent sheaf of algebras $\mc B_{[\mbb P^1/\mbb G_m]}$ on $[\mbb P^1/\mbb G_m] \times Y$.

This construction also localizes over the center of $\mc B$. Namely $\mc B_{[\mbb P^1/\mbb G_m]}$ naturally defines a quasi-coherent sheaf of algebras over the relative spectrum $\Spec_{[\mbb P^1/\mbb G_m]\times Y}(Z(\mc B)_{[\mbb P^1/\mbb G_m]})$.
\end{constr}

\begin{example}[Split case]\label{ex: split case} Assume $\mc B$ is commutative (so $\mc B\simeq Z(\mc B)$).
Let $\mc B$ be graded $\mc B\simeq \oplus_{i} \mc B_i$ and consider the split exhausting filtrations associated to it: $$
\mc B_{\le n}\coloneqq \oplus_{i\le n}\mc B_i \text{ and }
\mc B^{\ge n}\coloneqq \oplus_{i\ge n}\mc B_i.
$$
Then there is a natural isomorphism  
$$\Spec_{\mbb P^1\times Y}(\mc B_{\mbb P^1})\simeq \Spec_Y (\mc B)\times \mbb P^1.$$ Moreover, this isomorphism is $\mbb G_m$-equivariant, where the $\mbb G_m$-action on $\Spec_Y (\mc B)$ is defined by the grading on $\mc B$. This way we also get an isomorphism
 $$
 \Spec_{ [\mbb P^1/\mbb G_m]\times Y}(\mc B_{[\mbb P^1/\mbb G_m]}) \simeq [(\Spec_Y (\mc B)\times \mbb P^1)/\mbb G_m].
 $$
 To see these isomorphisms explicitly, note that 
 $$
 (\mc B[h])_n\simeq \oplus_{i+j=n}\mc B_{i}\cdot h^j \simeq \oplus_{i\le n} \mc B_i \simeq \mc B_{\le n}
 $$
and, this way $\mc B[h]\simeq \mc B_{\mbb P^1\backslash \{\infty\}}$ as a graded $\mc O_Y$-algebra. Similarly, 
$$
(\mc B[h^{-1}])_n\simeq \oplus_{i-j=n}\mc B_{i}\cdot h^{-j} \simeq \oplus_{i\ge n}\mc B_{i}\simeq \mc B^{\ge  n}
$$
and so $\mc B[h^{-1}]\simeq \mc B_{\mbb P^1\backslash \{0\}}$. Moreover, under these identifications, the gluing data over $\mathbb G_m\times Y$ is just given by the isomorphism $(\mc B[h])[h^{-1}]\simeq (\mc B[h^{-1}])[h]$. This way we can identify $\mc B_{\mbb P^1}$ with the pushforward $p_*\mc O$ for the projection $p\colon \mbb P^1\times \Spec_{Y}(\mc B)\ra \mbb P^1\times Y$, hence the statement above.
\end{example}

\subsection{Twistor differential operators $\mc D_{X,\mbb P^1}$}\label{ssec:twistor differential operators}
We now discuss a particular case of the above construction in the case of the sheaf of differential operators on a pseudo-smooth scheme. Let us clarify the setup. 

\begin{constr}[{Hodge and conjugate filtrations on $\mc D_{X}$}]
Namely, let $X$ be a pseudo-smooth $S$-scheme and let $\mc D_X$ be the sheaf of differential operators that we have defined in Section \ref{ssec:definition of diff op}. The pushforward $F_{X*}^\lp\mc D_X$ defines a quasi-coherent sheaf of algebras over $X^\lp$. 
We suppress $F_{X*}^\lp$ from the notation and implicitly regard
$\mc D_X$ as a sheaf of $\mc O_X^\lp$-algebras. Recall the isomorphism
$\theta^{-1}\colon Z(\mc D_X)\simeq
\Sym^\star_{\mc O_X^\lp}\!T_X^\lp$
(see Corollary \ref{cor:D_X is an Azumaya algebarain pseudo-smooth setting}).
For brevity, we temporarily omit the subscript $\mc O_X^\lp$.

We endow $\mc D_X$ with two filtrations: 
\begin{itemize}
\item (Hodge filtration). Consider filtration $\mr{Fil}^{\mr{H}}_{\le *}\coloneqq \mc D_{X,\le *}$ by the order of differential operator (see Remark \ref{rem:PBW-filtration}). One has $\mc D_{X,\le n}=0$ if $n<0$ and $\mc D_{X,\le 0}\simeq \mc O_X\subset \mc D_X$:
$$
\ldots = 0 = 0\subset \mc O_X \subset \mc D_{X,\le 1} \subset \mc D_{X,\le 2}\subset \ldots
$$
\item (Conjugate filtration). This is a filtration $\mr{Fil}_{\mr{cnj}}^{\ge *}$ defined as follows. Namely one has $\mr{Fil}_{\mr{cnj}}^{\ge n}\simeq \mc D_{X}$ if $n\le 0$:
$$
\ldots =  \mc D_{X}= \mc D_{X} \supset\mr{Fil}_{\mr{cnj}}^{\ge 1}\supset \mr{Fil}_{\mr{cnj}}^{\ge 2}\supset \ldots
$$
and $\mr{Fil}_{\mr{cnj}}^{\ge i}$ is defined as $\Sym^{\lceil\frac{i}{p}\rceil}T_X^\lp\cdot \mc D_{X}\subset \mc D_{X}$. In particular, it is a $p$-step filtration (one has $\mr{Fil}_{\mr{cnj}}^{\ge i+1}= \mr{Fil}_{\mr{cnj}}^{\ge i}$ unless $p|i$). 
\end{itemize}
\end{constr} 
\begin{rem}
The conjugate filtration $\mr{Fil}_{\mr{cnj}}^{\ge *}$ is a reindexed variant of a filtration that appeared in \cite{ov}.
\end{rem}

\begin{rem}[Induced filtrations on the center]\label{rem:Rees construction for the center} Filtrations induced on $Z(\mc D_{X})$ by the Hodge and conjugate filtrations on $\mc D_X$ above are in fact split. Namely, we claim that the two filtrations $Z(\mc D_{X})\cap \mr{Fil}^{\mr{H}}_{\le *}$ and $Z(\mc D_{X})\cap \mr{Fil}_{\mr{cnj}}^{\ge *}$ are obtained by construction in Example \ref{ex: split case} from the single grading on $Z(\mc D_{X})$ given as follows:
$$
Z(\mc D_{X})_n\coloneqq \begin{cases}\Sym^kT_{X}^\lp & \text{if $n=p\cdot k$}\\
0 & \text{if $p\nmid n$}
\end{cases}
$$
This is clear for the conjugate filtration: since $Z(\mc D_{X})\cap \mr{Fil}_{\mr{cnj}}^{\ge n}$ is given by $\Sym^{\ge\lceil\frac{n}{p}\rceil}T_X^\lp$ which is exactly the direct sum $\oplus_i Z(\mc D_{X})_i$ with $i\ge n$. For the Hodge filtration this requires a slight argument. 

More precisely, we need to show that $\mr{Fil}^{\mr{H}}_{\le n}\cap Z(\mc D_{X})\simeq \Sym^{\le \lfloor\frac{n}{p}\rfloor}T_X^\lp$ (where the latter is isomorphic to $\oplus_i Z(\mc D_{X})_i$ with $i\le n$). Recall that the map $\theta\colon \Sym^\star_{\mc O_X^\lp}\!T_X^\lp \ra Z(\mc D_X)$ is induced by $v^\lp\mapsto \theta(v^\lp)\colon v^p-v^{[p]}$ and so $\theta(T_X^\lp)\subset \mr{Fil}^{\mr{H}}_{\le p}$. Consequently, $\theta(\Sym^{\le k}_{\mc O_X^\lp}\!T_X^{\lp})\subset \mr{Fil}^{\mr{H}}_{\le pk}$ and to show the statement it is enough to show that the map $\gr_{pk}\theta\colon \Sym^k_{\mc O_X^\lp}\!T_X^\lp\ra \mr{gr}^{\mr{H}}_{pk}$ between the associated graded pieces is an embedding. We have $
\theta(v^{\lp})= v^p \in \gr_p^{\mr{H}}$ and this way we see that the map of $\mc O_X^\lp$-algebras $\gr_\star(\theta)\colon \Sym^\star_{\mc O_X^\lp}T_X^\lp \ra \Sym^\star_{\mc O_X}\!\!T_X$ is the map given by reduced Frobenius $F_{T^*\! X}^\lp\colon T^*\! X\ra (T^*\! X)^\lp$ (as affine schemes over $X^\lp$), and is an embedding on the level of sheaves of functions essentially by definition.

Following the discussion in Example \ref{ex: split case} we then get the following description of the Rees construction for $Z(\mc D_{X})$. Namely, the grading on $Z(\mc D_{X})$ that we considered above corresponds to the $\mbb G_m$-action on $(T^*\! X)^\lp\simeq \Spec_{X^{\!\lp}}Z(\mc D_{X})$ where $t\in\mbb G_m$ acts by rescaling the fibers of the vector bundle $(T^*\! X)^\lp\ra X^\lp$ by $t^p$. 
We get a $\mbb G_m$-equivariant isomorphism $$\Spec_{X^{\!\lp}\!\times {\mathbb P^1}} Z(\mc D_{X})_{\ \!\!{\mathbb P^1}}\simeq (T^*\! X)^\lp \times {\mathbb P^1},$$
where the action on the right is diagonal.

It then also descends to an isomorphism 
$$
\Spec_{X^{\!\lp}\!\times [{\mathbb P^1}/\mbb G_m]} Z(\mc D_{X})_{\ \!\![{\mathbb P^1}/\mbb G_m]}\simeq [((T^*\! X)^\lp \times {\mathbb P^1})/\mbb G_m]
$$
for the corresponding quotients by $\mbb G_m$.
\end{rem}

\begin{df}[Twistor differential operators]\label{def:tdo}
We define the sheaf $\mc D_{X,{\mathbb P^1}}$ of \textit{twistor differential operators} on $X$ as the Rees construction (\ref{constr:Rees construction over P^1}) over ${\mathbb P^1}$ associated to $(\mc D_X,\mr{Fil}^{\mr{H}}_{\le *},\mr{Fil}_{\mr{cnj}}^{\ge *})$. By Remark \ref{rem:Rees construction for the center} and discussion in Example \ref{ex: split case}, $\mc D_{X,{\mathbb P^1}}$ defines a $\mbb G_m$-equivariant quasi-coherent sheaf of algebras over $(T^*\! X)^\lp \!\times {\mathbb P^1}$ (with $\mbb G_m$-action as in Remark \ref{rem:Rees construction for the center}). 
\end{df}

\begin{rem}\label{rem:twistor differential operators explicitly}Let us describe $\mc D_{X,{\mathbb P^1}}$ more explicitly. Namely, the restriction $\mc D_{X,{\mathbb P^1}\backslash \{\infty\}}$ is the Rees algebra associated to PBW-filtration on $\mc D_X$ and as such can be identified with the subalgebra in $\mc D_X[h]$ generated by $\mc O_X$ and $h\cdot T_X$. One can also describe it as a sheaf of algebras over $\mc O_X^\lp[h]$ generated by $\mc O_X$ and $T_X$ with the relations given by $f_1\cdot f_2-f_2\cdot f_1=0$, $v\cdot f-f\cdot v=h\cdot v(f)$ and $v_1\cdot v_2-v_2\cdot v_1=h\cdot [v_1,v_2]$ (here we identify $T_X$ with its image in $\mc D_{X,\le 1}\cdot h$ in the Rees construction).

 Following Remark \ref{rem:Rees construction for the center} and Example \ref{ex: split case} we have an isomorphism ${Z(\mc D_{X})}_{{\mathbb P^1}\backslash \{\infty\}}\simeq Z(\mc D_{X})[h]\simeq (\Sym^\star_{\mc O_X^\lp}\!\!T_X^\lp)[h]$. The corresponding map 
${Z(\mc D_{X})}_{{\mathbb P^1}\backslash \{\infty\}}\ra \mc D_{X, {\mathbb P^1}\backslash \{\infty\}}$ sends $\mc O_X^\lp$ to $\mc O_X$ via $F_X^\lp$, while $v^\lp\in T_X^\lp$ maps to\footnote{To see why this is the formula: we have $\theta(v)\in \mr{Fil}_{\le p}^{\mr{H}}$ and $\theta(v)\cdot h^p=v^ph^p-v^{[p]}h^p=(vh)^p-h^{p-1}(v^{[p]}h)$.} $\theta_h(v)=v^p-h^{p-1}\cdot v^{[p]}$. 

For  the other chart, note that since $\mr{Fil}_{\mr{cnj}}^{\ge 0}=\mc D_X$ we have a natural embedding $Z(\mc D_X)[h^{-1}]\subset {Z(\mc D_{X})}_{{\mathbb P^1}\backslash \{0\}}$. Moreover, by the way how $\mr{Fil}_{\mr{cnj}}^{\ge *}$ is defined, $\mc D_{X,{\mathbb P^1}\backslash \{0\}}$ is identified with the tensor product $\mc D_X[h^{-1}]\otimes_{Z(\mc D_X)[h^{-1}]} {Z(\mc D_{X})}_{{\mathbb P^1}\backslash \{0\}}$. Using Remark \ref{rem:Rees construction for the center} and Example \ref{ex: split case}, one can further identify $\mc D_{X,{\mathbb P^1}\backslash \{0\}}$ with the pullback of $\mc D_X\boxtimes \mc O_{{\mathbb P^1}\backslash\{0\}}$ along the endomorphism $\lambda_{h^{-p}}$ of $(T^*X)^\lp \times ({\mathbb P^1}\backslash\{0\})$ given by the fiberwise dilation by $h^{-p}$.
In particular, over $\{\infty\}\in {\mathbb P^1}\backslash \{0\}$, $\lambda_{h^{-p}}$ is the zero map
and thus the fiber $\mc D_{X,\{\infty\}}\coloneqq \mc D_{X,{\mathbb P^1}\backslash \{0\}}/h^{-1}$ is identified with $\mc D_{X,0}\otimes_{\mc O_{X}^\lp}\Sym^\star_{\mc O_X^\lp}\!T_X^\lp$, where $\mc D_{X,0}$ is the central reduction from Remark \ref{rem:splitting on the 0 section}. Recall that $\mc D_{X,0}\simeq \mc End_{\mc O_X^\lp}(\mc O_X)$ is a split Azumaya algebra; this way $\mc D_{X,\{\infty\}}$ is Morita equivalent to $\Sym^\star_{\mc O_X^\lp}\!T_X^\lp$. In particular, the category of quasi-coherent sheaves of $\mc D_{X,\{\infty\}}$-modules is equivalent to the category of quasi-coherent sheaves over $(T^*\! X)^\lp$. 

The restriction of  $\mc D_{X,{\mathbb P^1}}$ to ${\mathbb P^1}\backslash (\{0\}\sqcup \{\infty\})$ is identified with  $\mc D_X[h,h^{-1}]$.
\end{rem}
\subsection{Central reduction of twistor differential operators}\label{ssec:central reduction of twistor differential operators}

We extend Construction \ref{constr:classical central reductions associated to 1-forms} to the algebra of twistor differential operators. Let $\alpha\in H^0(X^\lp, \Omega^{1,\lp}_X)$ be a global section, $i_\alpha\colon X^\lp \hookrightarrow (T^*\! X)^\lp$  the graph of $\alpha$. Denote by
$$\mc I_{\alpha, {\mathbb P^1}} \subset  Z(\mc D_{X})_{\ \!\!{\mathbb P^1}} \subset \mc D_{X,{\mathbb P^1}}$$
the sheaf of ideals determined by the closed embedding $i_\alpha \times \Id_{{\mathbb P^1}} \colon X^\lp \times {\mathbb P^1} \hookrightarrow (T^*\! X)^\lp \times {\mathbb P^1}$.  Let $\mc D_{X, \frac{\alpha}{h^p}, {\mathbb P^1}}$ be the quotient $\mc D_{X,{\mathbb P^1}}/\mc I_{\alpha, {\mathbb P^1}} \mc D_{X,{\mathbb P^1}}$.

Recall that we have the natural map ${-}^{\{p\}}\colon H^0(X,\Omega^1_X)\ra H^0(X^\lp, \Omega^{1,\lp}_X)$, $\alpha\mapsto \alpha^{\{p\}}$. It turns out that the association $\alpha \mapsto \mc D_{X, \frac{\alpha^{\{p\}}}{h^p}, {\mathbb P^1}}$ for $\alpha\in H^0(X,\Omega^1_X)$ only depends on the class $[\alpha]\in H_{Zar}^{0}(X, \coker(\mathcal{O}_{X} \xra{d} \Omega_{X/S}^{1}))$ in a canonical way, as the discussion below shows.
\begin{constr}
Let $\mu \in H^0(X, \Omega^{1}_{X})$ be a closed $1$-form on $X$. Define an automorphism {$\phi_{\frac\mu h}$} of algebra $\mc D_{X,{\mathbb P^1}\backslash \{\infty\}}$ (viewed  as the  subalgebra of $\mc D_X[h]$ generated by $\mc O_X$ and $h\cdot T_X$)
by setting $\phi_{\frac{\mu}{h}}(f)=f$ and $\phi_{\frac{\mu}{h}}(hv)=hv + \iota_{v}\mu$,  
for every function $f$ and every vector field $v$. Informally, this automorphism is the conjugation by $e^{-\frac{1}{h} \int \mu }$: though this expression does not make sense as a function on $X \times {\mathbb P^1}$,  the conjugation automorphism is well defined. For two closed 1-forms $\mu_1, \mu_2  \in H^0(X, \Omega^{1}_{X/S})$, we have that
\begin{equation}\label{eq:transitivity}
\phi_{\frac{\mu_1}{h}} \circ \phi_{\frac{\mu_2}{h}}= \phi_{\frac{\mu_1 +\mu_2}{h}}.
\end{equation}
The action of $\phi_{\frac{\mu}{h}}$ on the center $(\Sym^\star_{\mc O_X^\lp}\! T_X^\lp)[h]$ of the algebra $\mc D_{X,{\mathbb P^1}\backslash \{\infty\}}$ is given by the Katz $p$-curvature formula (see Remark \ref{rem:katzformula}): $\phi_{\frac{\mu}{h}}$ acts trivially on $\mc O_X^\lp$ and, for $v^\lp\in T_X^\lp$, one has that 
$\phi_{\frac{\mu}{h}} (\theta_h(v^\lp))= \theta_h(v^\lp) + \iota_{v^\lp} (\mu^\lp - h^{p-1} C(\mu)).$ In particular, if $C(\mu)=0$ {\it i.e.}, $\mu$ is locally exact, then  $\phi_{\frac{\mu}{h}}$ extends to an automorphism of $\mc D_{X,{\mathbb P^1}}$ and acts on the spectrum of its centre  $(T^*\! X)^\lp \times {\mathbb P^1}$ as the translation by $\mu^\lp$. Thus, for an exact $1$-form $\mu$ on $X$ and a section $\alpha\in H^0(X^\lp, \Omega^{1,\lp}_X)$, one has that $\phi_{\frac{\mu}{h}}(\mc I_{\alpha, {\mathbb P^1}})=\mc I_{\alpha - \mu^\lp, {\mathbb P^1}}$. It follows that $\phi_{\frac{\mu}{h}}$ descends to an isomorphism:
$$\phi_{\frac{\mu}{h}}\colon \mc D_{X, \frac{\alpha }{h^p}, {\mathbb P^1}} \iso \mc D_{X, \frac{\alpha -\mu^\lp}{h^p}, {\mathbb P^1}}.$$

\end{constr}
\begin{constr}[Generalized central reduction $\mathcal{D}_{X, \frac{[\eta]}{h}, {\mathbb P^1}}$]\label{constr:central reductions over S}
A class 
$$[\eta] \in H_{Zar}^{0}\left(X, \operatorname{coker}\left(\mathcal{O}_{X} \stackrel{d}{\longrightarrow} \Omega_{X}^{1}\right)\right)$$
gives rise to a certain sheaf of algebras $\mathcal{D}_{X, \frac{[\eta]}{h}, {\mathbb P^1}}$ locally given by a central reduction of $\mathcal{D}_{X, {\mathbb P^1}}$.  This is a sheaf of quasi-coherent $\mc O_{X^\lp\times \mathbb P^1}$-algebras  that comes together with the following structure. 
For any open subset $U$ together  with a $1$-form $\eta \in \Omega^1_{X/S}(U)$ representing  $[\eta]$,  we have an isomorphism of $\mc O_U^\lp$-algebras
\begin{equation}\label{eq:local_description_of_central_reduction}
\gamma_{U, \eta}\colon (\mathcal{D}_{X, \frac{[\eta]}{h}, {\mathbb P^1}})_{|U} \iso \mc D_{U, \frac{(\eta_{|U})^\lp}{h^p}, {\mathbb P^1}}
\end{equation}
such that, for any two $1$-forms $\eta_1$, $\eta_2$ on $U$ representing the class $[\eta]$, one has that
 $\gamma_{U, \eta_2} \circ \gamma_{U, \eta_1}^{-1}  =  \phi_{{\frac{\eta_1 -\eta_2}h}}.$
Formula (\ref{eq:transitivity}) shows that the algebra $\mathcal{D}_{X, \frac{[\eta]}{h}, {\mathbb P^1}}$ together with the collections of $\gamma_{U, \eta}$ exists and is unique up to a unique isomorphism. Moreover, the construction is functorial, meaning that 
the sheaf of algebras  $\mathcal{D}_{X, \frac{[\eta]}{h}, {\mathbb P^1}}$  is equivariant with respect to the group $\Aut_S(X,[\eta])$ of automorphisms of $X$ over $S$  preserving the class $[\eta]$. 
 For future reference we display the induced action  $\Aut_S(X,[\eta])$  on the direct image $\mathcal{D}_{X, \frac{[\eta]}{h}, {\mathbb P^1}}$ along the projection $\pro_{{\mathbb P^1}}\colon X^\lp \times {\mathbb P^1} \to S \times {\mathbb P^1}$ 
 by $\cO_{S\times {\mathbb P^1}}$-algebra 
automorphisms
\begin{equation}\label{eq:action on Central reduction}
\psi\colon \Aut_{S}(X,[\eta]) \to \Aut (\pro_{{\mathbb P^1} *}\mathcal{D}_{X, \frac{[\eta]}{h}, {\mathbb P^1}}).
\end{equation}
 \end{constr}
 The algebra $\mathcal{D}_{X, \frac{[\eta]}{h}, {\mathbb P^1}}$ has further symmetries: simultaneous multiplication of $h$ and $[\eta]$ by a scalar $\lambda$ doesn't change the algebra. More precisely, endow ${\mathbb P^1}$ with an action $\chi$ of $\bG_m$ by homotheties:  $\chi(h)= z h$, where $z$ is the coordinate on $\bG_m$.
 Then pullback of $\mathcal{D}_{X, \frac{[\eta]}{h}, {\mathbb P^1}}$ with respect to the morphism  $\bG_m \times X^\lp \times {\mathbb P^1} \to X^\lp \times {\mathbb P^1}$ given by the $\bG_m$-action on the last factor is naturally isomorphic to 
 $\mathcal{D}_{\bG_m \times X,\, ([z^{-1}\pro_X^*\eta]) / h, \,{\mathbb P^1}}$, where $\mbb G_m\times X$ is considered as a scheme over $\mbb G_m\times S$ and $\pro_X\colon \bG_m \times X \to X$ is the projection.
 \begin{example}[The sheaf $\mc D_{X,0,{\mathbb P^1}}$]\label{ex:D_X0S}
Let's consider a particular case of the above construction when we have $[\eta]=0$. The corresponding sheaf of algebras $\mc D_{X,0,{\mathbb P^1}}$ is simply the restriction of $\mc D_{X,{\mathbb P^1}}$ to the zero section $X^\lp\!\times{\mathbb P^1} \ra (T^*\! X)^\lp\!\times{\mathbb P^1}$, or in terms of sheaves of algebras we have $\mc D_{X,0,{\mathbb P^1}}\simeq \mc D_{X,{\mathbb P^1}}\otimes_{Z(\mc D_X)_{{\mathbb P^1}}}\mc O_X^\lp.$ 

Let us describe more explicitly the gluing data defining $\mc D_{X,0,{\mathbb P^1}}$. Namely, $\mc D_{X,0,{\mathbb P^1}\backslash \{\infty\}}$ is the quotient of $\mc D_{X,{\mathbb P^1}\backslash \{\infty\}}$ (see Remark \ref{rem:twistor differential operators explicitly}) by the two-sided ideal generated by $v^p-h^{p-1}v^{[p]}$, $v\in T_X$. On the other hand, we have $\mc D_{X,0,{\mathbb P^1}\backslash \{0\}}\simeq \mc D_{X,0}[h^{-1}]$. Indeed, by Remark \ref{rem:twistor differential operators explicitly} $\mc D_{X,{\mathbb P^1}\backslash \{0\}}\simeq \mc D_X[h^{-1}]\otimes_{Z(\mc D_X)[h^{-1}]} {Z(\mc D_{X})}_{{\mathbb P^1}\backslash \{0\}}$ from which we get that 
$$
\mc D_{X,0,{\mathbb P^1}\backslash \{0\}}\simeq \mc D_X[h^{-1}]\otimes_{Z(\mc D_X)[h^{-1}]} \mc O_X^{\lp}[h^{-1}]\simeq \mc D_{X,0}[h^{-1}].
$$
\end{example}
Observe that for every class $[\eta]$ the restriction of  $\mathcal{D}_{X, \frac{[\eta]}{h}, {\mathbb P^1}}$ to $X^\lp \times \{\infty\} \mono X^\lp \times {\mathbb P^1}$ is naturally isomorphic to that of  $\mc D_{X,0,{\mathbb P^1}}$. In particular, the former is a canonically split 
Azumaya algebra over $X^{\lp} \times \{\infty\} $. 
 \begin{constr}[Tensor product of modules over the ring of twisted differential operators and their central reductions]\label{const:tesnsor_product}
 For quasi-coherent left $\mc D_{X,\mbb P^1}$-modules\footnote{A $\mc D_{X,\mbb P^1}$-module 
 is said to be quasi-coherent if the underlying $\cO_{X\times \bP^1}$-module is quasi-coherent.}  $M_1, M_2$, the tensor product $M_1 \otimes_{\cO_{X\times \bP^1}} M_2$ has a natural  structure of left $\mc D_{X,\mbb P^1}$-module: over the $\bP^1\backslash \{\infty\}$ chart, the action of $h \cdot v\in h\cdot T_X$ is given by the Leibniz formula. Over the $\bP^1\backslash \{0\}$ chart, the action of the additional central generators 
    \[
h^p(v^p-v^{[p]}), \qquad v\in T_X
\]
is also given by  the Leibniz formula and the compatibility amounts to the additivity of $p$-curvature. 

Next, for any $[\eta]_1, [\eta]_2 \in H_{Zar}^{0}\left(X, \operatorname{coker}\left(\mathcal{O}_{X} \stackrel{d}{\longrightarrow} \Omega_{X}^{1}\right)\right)$ and  $\mathcal{D}_{X, \frac{[\eta]_i}{h}, {\mathbb P^1}}$-modules
$M_i$, $i=1,2$, $M_1 \otimes_{\cO_{X\times \bP^1}} M_2$ has a natural  structure of module over $\mathcal{D}_{X, \frac{[\eta]_1+[\eta]_2}{h}, {\mathbb P^1}}$. This again follows from 
the additivity of $p$-curvature.
 \end{constr}
 \begin{constr}[The multiplicative $\bG_m$-gerbe associated to $\mathcal{D}_{X, \frac{[\eta]}{h}, {\mathbb P^1}}$]\label{const:multiplicative_gerbe}
 Corollary \ref{cor:D_X is an Azumaya algebarain pseudo-smooth setting} implies that
 $\mathcal{D}_{X, \frac{[\eta]}{h}, {\mathbb P^1}}$ restricted to $X^\lp\times \bP^1\backslash \{0\}$ is an Azumaya algebra. Denote by $t=h^{-1}$ the coordinate on $\bP^1\backslash \{0\}=\bG_a$ and by $\mathcal{D}_{X, t[\eta], \bG_a}$ the above restriction. 
 For an $\bF_p$-algebra $R$ and  $t_1, t_2\in \bG_a(R)=R$, there is a natural Morita equivalence  
   \begin{equation}
       \mathcal{D}_{X, t_1[\eta]} \otimes _{\cO_{X^\lp \times \Spec R}} \mathcal{D}_{X, t_2[\eta]} 
       \sim \mathcal{D}_{X, (t_1 + t_2)[\eta]}
       \end{equation}
       of Azumaya algebras over $X^\lp \times \Spec R$. The equivalence is given
       by \[
       \mathcal{D}_{X, t_1[\eta]} \otimes _{\cO_{X \times \Spec R}} \mathcal{D}_{X, t_2[\eta]},
       \]
     which has the structure of a left module over   $\mathcal{D}_{X, (t_1 + t_2)[\eta]}$ and a right module over $\mathcal{D}_{X, t_1[\eta]} \otimes _{\cO_{X^\lp \times \Spec R}} \mathcal{D}_{X, t_2[\eta]}$. This gives  the $\bG_m$-gerbe 
     \begin{equation}\label{eq:multstr}
     [\mathcal{D}_{X, t[\eta], \bG_a}]\to X^\lp\times \bG_a :=\bG_{a, X^\lp}
     \end{equation}
a multiplicative structure, which exhibits  \eqref{eq:multstr} as an extension of 
   $\bG_{a, X^\lp}$ by $B \bG_{m, X^\lp}$ in the category of strictly commutative group stacks over $X^\lp$.  
 \end{constr}
 For a test scheme $Z\to X^\lp $ let $\cQ_{X, [\eta]}(Z)$ be the groupoid of splittings of 
 the  strictly commutative group stack extension  $\eqref{eq:multstr}$ pulled back to $Z$. 
 Note that $\cQ_{X, [\eta]}(Z)$ is equipped with a natural action of the Picard groupoid of 
 extensions of $\bG_{a, Z}$ by $\bG_{m, Z}$ in the category of commutative group schemes over $Z$. This defines a stack $\cQ_{X, [\eta]}$ over $X^\lp$ equipped with an action of 
 $B \bG_{a}^{\hat\sharp}$, where  $ \bG_{a}^{\hat\sharp}$ is a group valued {\it fpqc} sheaf 
 sending a test scheme $Z$ to $\Hom(\bG_{a, Z}, \bG_{m, Z})$.
 \begin{pr}\label{pr:miltiplicative_property}
     The map $\cQ_{X, [\eta]} \to X^\lp$ exhibits the source as a gerbe over $X^\lp$ banded by
    $ \bG_{a}^{\hat\sharp}$.  
 \end{pr}
 \begin{proof}
    We just need to check that any two objects of $\cQ_{X, [\eta]}(Z)$ are {\it fpqc}  locally
    isomorphic and that there exists an {\it fpqc} cover $Z\to  X^\lp$ such that  $\cQ_{X, [\eta]}(Z)$ is not empty. 
    For the first assertion use \cite[Cor. 3.4, Cor. 3.10, Prop. 2.30]{bmv}, which imply that
     any commutative group extension of $\bG_{a, Z}$ by $ \bG_{m, Z}$ splits {\it fppf} locally on $Z$ (for $Z$ over $\bF_p$).
For the second assertion, we shall construct a multiplicative splitting $M$ of
$[\mathcal{D}_{X, t[\eta], \bG_a}]$ pulled back along 
$X \times \bG_a \xrightarrow{F_{X}^\lp \times \Id_{\bG_a}}X^\lp \times \bG_a$.
For an $X$-scheme $Y$, a $Y$-point of the pullback $X \times_{X^\lp} [\mathcal{D}_{X, t[\eta], \bG_a}]$ consists of a morphism of $Y$-schemes  $f\colon Y\to \bG_{a, Y}$ and a line bundle $\cL$ over $X \times_{X^\lp} Y$ equipped with an additional structure of $\mathcal{D}_{X \times_{X^\lp} Y, f[\eta]}$-module. The group operation on $[\mathcal{D}_{X, t[\eta], \bG_a}](Y)$ is given by $(f_1, \cL_1)\cdot (f_2, \cL_2)=(f_1+f_2, \cL_1\otimes \cL_2)$. 
Consider the sheaf $\mathcal{D}_{X, t[\eta], \bG_a}$. It carries
an action of the algebra $\mathcal{D}_{X, t[\eta], \bG_a}$ given by the left multiplication  and a commuting action of $\cO_X$ given by the right multiplication.  Set
\[ M:=\mathcal{D}_{X, t[\eta], \bG_a} \otimes_{\cO_X} (\Omega^d_{X})^{\otimes p-1}, \]
where $d= \rk \Omega^1_{X}$, and the tensor product is formed with respect to the right action of $\cO_X$ on the first factor. 
We view $M$ as a module over   $(F_{X}^\lp \times \Id_{\bG_a})^{*}(\mathcal{D}_{X, t[\eta], \bG_a})$. Then $M$
is a splitting of the gerbe $X \times_{X^\lp} [\mathcal{D}_{X, t[\eta], \bG_a}] \to \bG_{a, X}$.
 Let us construct a multiplicative structure on $M$. Let  $(f, \cL)$ be a $Y$-point of $X \times_{X^\lp} [\mathcal{D}_{X, t[\eta], \bG_a}]$.
Let $q:X\times_{X^\lp}Y\to Y$ be the projection, and let
$i:Y\to X\times_{X^\lp}Y$ be the section induced by $Y\to X$.
Put $E=q_*\cL$, and denote the corresponding Azumaya algebra
on $Y$ by $A_f$.
Finite flat duality and
Proposition~\ref{prop:relativedualizaingsheaf} give
\[
  f^*M\simeq
  E\otimes_{\cO_Y}(i^*\cL)^{-1},
  \qquad
  \underline{\Hom}_{A_f}(f^*M,E)\simeq i^*\cL.
\]
The functor $(f,\cL)\mapsto i^*\cL$ is symmetric monoidal
and restricts to the identity on the kernel $B\bG_m$.
It therefore determines a multiplicative splitting,
whose underlying splitting module is $M$.
\end{proof}

\subsection{Twistor reduced Weyl algebra $\mc A_{{\mathbb P^1}}$}\label{ssec:twistor reduced Weyl algebra}
We will now be interested in a very particular case of Construction \ref{constr:central reductions over S}. Namely, let us fix a number $n\in \mathbb N$. Consider the algebra $C\coloneqq \mbb F_p[x_1,\ldots,x_n]/(x_1^p,\ldots,x_n^p)$: its spectrum is the Frobenius neighborhood of $\{0\}$ in an $n$-dimensional affine space over $\mbb F_p$. Put $Y\coloneqq \Spec C$ considered as a scheme over $\mbb F_p$. This is a pseudo-smooth $\mbb F_p$-scheme (with $\Omega^1_C\simeq \oplus_{i=1}^nC\cdot dx_i$) and one has $Y^\lp\simeq \Spec (C^p)\simeq \Spec \mbb F_p$. 

\begin{df}\label{def:A_S}
We define the \textit{twistor reduced Weyl algebra} $\mc A_{{\mathbb P^1}}$ as $\mc D_{Y,0,{\mathbb P^1}}$. This is a quasi-coherent sheaf of algebras on ${\mathbb P^1}$.
\end{df}

\begin{rem}
Note that by Remark \ref{lem:fiber product of reduced Frobenius with itself} any pseudo-smooth $X$ considered as a scheme over $X^\lp$ flat locally looks like $C$ with $n=\rk (\Omega^1_X)$. This way $\mc A_{{\mathbb P^1}}$ can serve as a model example of a ``twistor quantization'' (in this context of $T^*\! X$).
\end{rem}

Let us describe $\mc A_{{\mathbb P^1}}$ in terms of its restrictions to ${\mathbb P^1}\backslash \{\infty \}$ and ${\mathbb P^1}\backslash \{0\}$ and the gluing data. First, let us describe the algebra $D_{Y}:=\Gamma(Y,\mc D_Y)$. We have $T_Y\simeq \oplus_{i=1}^nC\cdot \partial_{i}$ with $\partial_{i}(x_j)=\delta_{ij}$. Thus 
$$
D_Y\simeq \mbb F_p\langle x_i,y_j\rangle_{1\le i,j\le n}/([x_i,x_j]=[y_i,y_j]=x_i^p=0, [y_j,x_i]=\delta_{ij})
$$
with $y_i$ corresponding to $\partial_i$. One sees easily that $\partial_i^{[p]}=0$ and so the center $Z(D_Y)$ is described as 
$$
Z(D_Y)\simeq \mbb F_p[y_1^p,\ldots,y_n^p]\subset D_Y.
$$
Hodge filtration on $D_Y$ is given by the total degree in $y_i$'s. Following Remark \ref{rem:twistor differential operators explicitly} the corresponding Rees algebra $D_{Y,{\mathbb P^1}\backslash \{\infty \}}$ is the $\mbb F_p[h]$-algebra given by 
$$
D_{Y,{\mathbb P^1}\backslash \{\infty \}}\simeq  \mbb F_p[h]\langle x_i,y_j\rangle_{1\le i,j\le n}/([x_i,x_j]=[y_i,y_j]=x_i^p=0, [y_j,x_i]=h\cdot \delta_{ij})
$$
The center $Z(D_{Y,{\mathbb P^1}\backslash \{\infty \}})$ is identified with 
$$
Z(D_{Y,{\mathbb P^1}\backslash \{\infty \}})\simeq \mbb F_p[h,y_1^p,\ldots,y_n^p]\subset D_{Y,{\mathbb P^1}\backslash \{\infty \}}.
$$
We have 
$
D_{Y,0, {\mathbb P^1}\backslash \{\infty \}}\coloneqq D_{Y,{\mathbb P^1}\backslash \{\infty \}}\otimes_{Z(D_{Y,{\mathbb P^1}\backslash \infty})}\mbb F_p[h]
$ and so 
$$
D_{Y,0, {\mathbb P^1}\backslash \{\infty \}}\simeq \mbb F_p[h]\langle x_i,y_j\rangle_{1\le i,j\le n}/([x_i,x_j]=[y_i,y_j]=y_j^p=x_i^p=0, [y_j,x_i]=h\cdot \delta_{ij}).
$$

Finally, $D_{Y,0}\coloneqq D_Y\otimes_{Z(D_Y)}\mc O(Y)^\lp$ is described as 
$$
D_{Y,0}\coloneqq \mbb F_p\langle x_i,y_j\rangle_{1\le i,j\le n}/([x_i,x_j]=[y_i,y_j]=x_i^p=y_j^p=0, [y_j,x_i]=\delta_{ij}).
$$
We have 
$$
D_{Y,0}\simeq \mr{End}_{\mbb F_p}(C)\simeq \Mat_{p^n}(\mbb F_p).
$$

\begin{rem}[Explicit description of $\mc A_{{\mathbb P^1}}$]\label{rem:description of A_S} Now, following Example \ref{ex:D_X0S} we can explicitly describe $\mc A_{{\mathbb P^1}}$. Let $\mc A_{{\mathbb P^1}\backslash \{\infty \}}$ and $\mc A_{{\mathbb P^1}\backslash \{0 \}}$ be the restriction of $\mc A_{{\mathbb P^1}}$ to the corresponding charts. We have $\mc A_{{\mathbb P^1}\backslash \{\infty \}}\simeq \mc D_{Y,0, {\mathbb P^1}\backslash \{\infty \}}$, and so the $\mbb F_p[h]$-algebra $A_{{\mathbb P^1}\backslash \{\infty\}}\coloneqq \Gamma({\mathbb P^1}\backslash \{\infty\},\mc A_{{\mathbb P^1}\backslash \{\infty \}})$ is described as
 $$
 A_{{\mathbb P^1}\backslash \{\infty\}}\simeq \mbb F_p[h]\langle x_i,y_j\rangle_{1\le i,j\le n}/([x_i,x_j]=[y_i,y_j]=y_j^p=x_i^p=0, [y_j,x_i]=h\cdot \delta_{ij}).
 $$ 
 Note that $A_{{\mathbb P^1}\backslash \{\infty\}}$ is a free $\mbb F_p[h]$-module of rank $p^{2n}$ with a basis given by ordered monomials $\{x^Iy^J\ | \ 0\le I_r,J_s\le p-1\}$.

On the other hand, $\mc A_{{\mathbb P^1}\backslash \{0 \}}\simeq \mc D_{Y,0, {\mathbb P^1}\backslash \{0\}}\simeq \mc D_{Y,0}[h^{-1}]$, and so by the above the $\mbb F_p[h^{-1}]$-algebra $A_{{\mathbb P^1}\backslash \{0\}}\coloneqq \Gamma({\mathbb P^1}\backslash \{0\},\mc A_{{\mathbb P^1}\backslash \{0\}})$ is given by
$$
A_{{\mathbb P^1}\backslash \{0 \}}\simeq \mbb F_p[h^{-1}]\langle x'_i,y_j'\rangle_{1\le i,j\le n}/([x'_i,x_j']=[y_i',y_j']=x_i'^p=y_j'^p=0, [y_j',x_i']=\delta_{ij}).
$$
Similarly, this is a free $\mbb F_p[h^{-1}]$-module of rank $p^{2n}$. It follows that $\mc A_{{\mathbb P^1}}$ is locally free as an $\mc O_{{\mathbb P^1}}$-module of rank $p^{2n}$.
The identification $A_{{\mathbb P^1}\backslash \{\infty\}}[h^{-1}]\simeq A_{{\mathbb P^1}\backslash \{0 \}}[h]$ is given by $x_i=x'_i$ and $y_i=h\cdot y_i'$ (since $y_i$ and $y_i'$ correspond respectively to $h\cdot \partial_i$ and $\partial_i$). 
\end{rem} 

\begin{rem}\label{rem:A_S is matrix algebra outside of 0}
Since $D_{Y,0}\simeq \mr{End}_{\mbb F_p}(C) \simeq \Mat_{p^n}(\mbb F_p)$ we get that $A_{{\mathbb P^1}\backslash \{0 \}}\simeq \Mat_{p^n}(\mbb F_p)[h^{-1}]$. Also, $A_{{\mathbb P^1}\backslash \{\infty\}}$ embeds into $A_{{\mathbb P^1}\backslash \{\infty\}}[h^{-1}]\simeq A_{{\mathbb P^1}\backslash \{0 \}}[h]\simeq \Mat_{p^n}(\mbb F_p)[h,h^{-1}]$, and so $\mc A_{{\mathbb P^1}}$ is an order in the matrix algebra of rank $p^{2n}$ over ${\mathbb P^1}$. 
\end{rem}
\begin{constr}[A module $\cV_{{\mathbb P^1}}$ over $\cA_{{\mathbb P^1}}$]\label{constr:A module over A}
We define a sheaf $\cV_{{\mathbb P^1}}$ of $\cA_{{\mathbb P^1}}$-modules as follows. As an $\cO_{{\mathbb P^1}}$-module $\cV_{{\mathbb P^1}}$ is free:
  $$\cV_{{\mathbb P^1}} = \cO_{{\mathbb P^1}} \otimes \bF_p[x_1, \cdots ,x_n]/(x_i^p)_{1\leq i\leq n}.$$
  The action of $x_i\in \Gamma({{\mathbb P^1}}, \cA_{{\mathbb P^1}})$ on  $\cV_{{\mathbb P^1}}$ is given by the multiplication on the right factor; over ${\mathbb P^1}\backslash \{\infty \}$ the action of  $y_i\in \Gamma({\mathbb P^1}\backslash \{\infty \}, \cA_{{\mathbb P^1}})$ is given by the formula $y_i (1\otimes f) = h\otimes \frac{\partial f}{\partial x_i}$, for $f\in \bF_p[x_1, \cdots ,x_n]/(x_i^p)_{1\leq i\leq n} $, while on ${\mathbb P^1}\backslash \{0 \}$, $y_i'\in \Gamma({\mathbb P^1}\backslash \{0 \}, \cA_{{\mathbb P^1}\backslash \{0 \}})$ acts by $y'_i(1\otimes f)=1\otimes \frac{\partial f}{\partial x_i}$.
  \end{constr}
  
\subsection{Algebra $A_0$}
Let $i_0\colon \{0\}\simeq\Spec \mbb F_p \ra {\mathbb P^1}$ be an embedding of 0. 

\begin{df}
We define the $\mbb F_p$-algebra $A_0$ to be $i_0^*\mc A_{{\mathbb P^1}}$. Equivalently, $A_0\simeq  A_{{\mathbb P^1}\backslash \{\infty\}}/h$ and so
$$
A_0\simeq \mbb F_p[x_i,y_j]/(x_i^p,y_j^p)_{1\le i,j\le n}.
$$
Let $A_h$ be the $h$-adic completion of the $\mathbb F_p[h]$-algebra $A_{{\mathbb P^1}\backslash \{\infty\}}$. Since $A_{{\mathbb P^1}\backslash \{\infty\}}$ is finitely-generated as an $\mathbb F_p[h]$-module one has $A_h\simeq A_{{\mathbb P^1}\backslash \{\infty\}}\otimes_{\mathbb F_p[h]}\mathbb F_p[[h]].$
\end{df}

For an $\mbb F_p$-algebra $R$ we denote by $A_0(R)$ the $R$-algebra $A_0\otimes_{\mbb F_p}R$. Note that $\Spec A_0(R)$ is pseudo-smooth over $R$ for any $R$.

\begin{constr} 
We denote by $\ev_0\colon A_0\twoheadrightarrow \mbb F_p$ the ``evaluation at 0''-map that sends all $x_i$ and $y_j$ to 0. We let $I_0\coloneqq \Ker(\ev_0)\subset A_0$ be the ideal given by its kernel. Obviously, $I_0=(x_1,\ldots,x_n,y_1,\ldots,y_n)$. For any $R$ we put $I_0(R)\coloneqq I_0\otimes_{\mbb F_p}R\subset A_0(R)$. Note that $I_0(R)\subset \alpha_p(A_0(R))$.
\end{constr}

Let us now apply some of the results of Section \ref{sec:pseudo-smooth algebras and differential geometry} to $S\coloneqq \Spec R$ and $X\coloneqq \Spec A_0(R)$. Following Example \ref{ex:reduced vs relative Frobenius}(2) we have $A_0(R)^\lp\simeq R\subset A_0(R)$; this way the reduced Frobenius map $F_X^\lp\colon X\ra X^\lp$ is identified with the projection $X\ra S$. On the other hand, since $x_i^p=y_j^p=0$ the reduced twist map $W_{A_0(R)}^\lp\colon A_0(R)\ra R$ induced by $a\mapsto a^p\in R\simeq A_0(R)^\lp$ is identified with the composition $A_0(R)\xra{\ev_0}R\xra{(-)^p}R$, $a\mapsto \ev_0(a)^p$.

 Consider the $A_0(R)$-module $\Omega^1_{A_0(R)}\coloneqq \Omega^1_{A_0(R)/R}$ of $R$-relative differential 1-forms. $\Omega^{1}_{A_0(R)}$ is a free $A_0(R)$-module of rank $2n$ spanned by $dx_i$ and $dy_i$ for $i=1,\dots,n$. Consider also the $R$-module $\Omega^{1,\lp}_{A_0(R)}\coloneqq \Omega^1_{A_0(R)}\otimes_{A_0(R),\ev_0^p} R$. We let $\Omega^{1,cl}_{A_0(R)}\subset \Omega^{1}_{A_0(R)}$ be the $R$-module of closed 1-forms. We have the de Rham differential $A_0(R)\xra{d}\Omega^{1,cl}_{A_0(R)}$ and the Cartier operation $\mathsf{C}\colon \Omega^{1,cl}_{A_0(R)} \ra \Omega^{1,\lp}_{A_0(R)}$.

\begin{lm}\label{lem:sequence for closed 1-forms}There is a short exact sequence of $R$-modules
$$
0\ra A_0(R)/R\xra{d}  \Omega^{1,cl}_{A_0(R)} \xra{\mathsf{C}} \Omega^{1,\lp}_{A_0(R)}\ra 0
$$
\end{lm}
\begin{proof}
This follows by applying the global sections functor to the Cartier short exact sequence (\ref{eq:Cartier ses}) on the affine $R$-scheme $X=\Spec A_0(R)$.
\end{proof}	

\begin{rem}
Under an isomorphism $$
\Omega^{1,\lp}_{A_0(R)}\simeq \left(\oplus_{i=1}^{n} R\cdot dx_i\right) \oplus \left(\oplus_{i=1}^{n} R\cdot d{y_i}\right),
$$ the Cartier operation has the following description. Namely, given $\alpha\in \Omega^1_{X/R,cl}$ we can write it as a linear combination of monomials in $x_i,y_j$ times $dx_k,dy_l$ with coefficients in $R$ and look at the coefficients $a_i\in R$ and $b_j\in R$ in front of summands of the form $x_i^{p-1}dx_i$ and $y_j^{p-1}dy_j$. Then 
	$
	\mathsf C(\alpha)=\sum_i a_i\cdot {dx_i} + \sum_j b_j\cdot {dy_j}\in \Omega^{1,\lp}_{A_0(R)}.
	$
\end{rem}
We can also consider maps $d\log\colon A_0(R)^\times\ra \Omega^{1,cl}_{A_0(R)}$ and the map $-^\lp\colon \Omega^{1,cl}_{A_0(R)}\ra \Omega^{1,\lp}_{A_0(R)}$ sending $\alpha$ to $\alpha\otimes 1\in \Omega^{1,\lp}_{A_0(R)}$.

\begin{lm}\label{lem:Milne's sequence}
There is a short exact sequence 
\begin{equation}\label{eq:Milne}
0\ra A_0(R)^\times/R^\times\xra{d\log}  \Omega^{1,cl}_{A_0(R)} \xra{\alpha\mapsto \mathsf{C}(\alpha)-\alpha^\lp} \Omega^{1,\lp}_{A_0(R)}\ra 0.
\end{equation}

\end{lm}
\begin{proof}
Let $X\coloneqq \Spec A_0(R)$. By Example \ref{ex:reduced vs relative Frobenius}, $X^\lp\simeq S\coloneqq \Spec R$. Recall Milne's exact sequence (Proposition \ref{prop:Milne's exact sequence}). We claim that $H^1_\textup{\'et}(S,F_{X*}^\lp\mc O_X^\times/\mc O_S^\times)=0$. Indeed, by Remark \ref{rem:structure of units in A_0} below, $F_{X*}^\lp\mc O_X^\times$ splits as $\mc O_S^\times \times (1+I_0)^\times$ and $H^1_\textup{\'et}(S,(1+I_0)^\times)$ is 0. The other terms in  Milne's exact sequence are coherent sheaves, and so also don't have higher cohomology. Thus, applying global sections to \ref{prop:Milne's exact sequence} we exactly get (\ref{eq:Milne}) above.
\end{proof}

\begin{rem}\label{rem:structure of units in A_0}
Recall the notation $I_0(R)\coloneqq \Ker(\ev_0)(R)$. Consider the subgroup $(1+I_0(R))^\times\subset A_0(R)^\times$: it is well-defined since $I_0(R)$ lies in the nilradical of $A_0(R)$, and identifies with the kernel of the map $A_0(R)^\times\twoheadrightarrow R^\times$ induced by $\ev_0$. Together with the map $R^\times\ra A_0(R)^\times$ this gives a splitting $A_0(R)^\times\simeq R^\times \times(1+I_0(R))$ functorially in $R$. Moreover, $(1+I_0(R))^\times$ has a (functorial in $R$) finite filtration by 
$$
(1+I_0(R))^\times\supset (1+I_0(R)^2)^\times \supset (1+I_0(R)^3)^\times \supset \ldots  
$$
with $(1+I_0(R)^n)/(1+I_0(R)^{n+1})\simeq I_0(R)^n/I_0(R)^{n+1}$, which is a free $R$-module spanned by monomials in $x_i$'s and $y_j$'s of degree $n$.

This shows that the \'etale sheaf $(1+I_0)^\times\colon R\mapsto (1+I_0(R))^\times$ on the big site of affine $\mbb F_p$-schemes has a finite filtration with the associated graded given by a coherent sheaf. In particular, for any specific $\Spec R$ the restriction of this sheaf to the small 
\'etale site $(\Spec R)_\textup{\'et}$ doesn't have higher cohomology: $H^{>0}_\textup{\'et}(\Spec R, (1+I_0)^\times)=0$.
\end{rem}

\begin{constr}[Restricted structure on the Frobenius neighborhood]\label{constr:restricted structure on A_0}
We endow $X=\Spec A_0(R)$ with the structure of a restricted symplectic $R$-scheme as in Construction \ref{constr:A_0}. Namely, we consider $(X,[\eta_{\mathsf{can}}])$ with  
$$
\eta_{\mathsf{can}}\coloneqq \sum_{i=1}^n y_idx_i\in \Omega^{1}_{A_0(R)}.
$$ 
The corresponding symplectic form is given by $\omega\coloneqq d\eta_{\mathsf{can}}=\sum_i dy_i\wedge dx_i\in \Omega^2_{A_0(R)}$.
\end{constr}

For the lemma that will follow shortly we will need the following generalization of Definition \ref{def:hamiltonian vector fields}.
\begin{df}
	Let $X$ be a symplectic $S$-scheme. A derivation $\xi\in H^0(X,T_X)$ is called $\log$-\textit{Hamiltonian} if $\xi=f^{-1}\cdot H_f$ for some $f\in H^0(X,\mc O_X^\times)$.
\end{df}

\begin{rem}
 Similarly to the Hamiltonian case, the condition $\xi=f^{-1}\cdot H_f$ is equivalent to $\iota_\xi\omega=-d\log f$.
\end{rem}

Given a vector field $\xi\in H^0(X,T_X)$ we denote by $L_\xi\colon \Omega^i_X\ra \Omega^i_X$ the Lie derivative along $\xi$. We recall the Cartan formula:
$L_\xi=\iota_\xi\circ d+d\circ\iota_\xi$ in terms of the de Rham differential and contraction.  A vector field $\xi$ is called \textit{Poisson} if $L_\xi\omega=0$.

\begin{lm}\label{lem:description of log and usual Hamiltonian vector fields}
	\begin{enumerate}
	\item A Poisson vector field $\xi\in T_{A_0(R)}$ is Hamiltonian if and only if for all $f\in A_0(R)$ we have
	$$
	\xi(f^{[p]})=H_f^{p-1}(\xi(f)).
	$$
	\item A Poisson derivation $\xi\in T_{A_0(R)}$ is $\log$-Hamiltonian if and only if for all $f\in A_0(R)$ we have
	\begin{equation}\label{eq:log-Ham}
			\xi(f^{[p]})=H_f^{p-1}(\xi(f))\  +\ \xi(f)^p.
	\end{equation}
	
\end{enumerate}
\end{lm}
\begin{proof}
The proof essentially follows \cite[Lemma 2.5]{bk}. By Cartan formula $L_\xi=d\iota_\xi+\iota_\xi d$ and the assumption $L_\xi \omega=0$ we get $d(\iota_\xi \omega)=0$, and so $\beta\coloneqq \iota_\xi\omega\in \Omega^1_{A_0(R)}$ is closed. By Remark \ref{rem:Hamiltonian vector fields} and Lemma \ref{lem:sequence for closed 1-forms} $\xi$ is Hamiltonian if and only if $\mathsf C(\beta)=0$. Similarly, by Lemma \ref{lem:Milne's sequence} $\xi$ is $\log$-Hamiltonian if and only if $\mathsf C(\beta)=\beta^\lp$. Since $\omega$ is non-degenerate, Hamiltonian vector fields generate $T_{A_0(R)}$. Then vector fields of the form $H_f^\lp$ generate $T_{A_0(R)}^\lp$, and so it is enough to understand the pairing of $\mathsf C(\beta)\in \Omega^{1,\lp}_{A_0(R)}$ with (the twists of) Hamiltonian vector fields. By Remark \ref{rem:pairing with Cartier} $\iota_{H_f^\lp}(\mathsf C (\beta))=\iota_{H_f^{[p]}}\beta  - H_f^{p-1}(\iota_{H_f}\beta)$. Substituting $\beta=\iota_\xi\omega$ we get $\iota_{H_f^{\lp}}(\mathsf C(\beta))=\xi(f^{[p]}) - H_f^{p-1}(\xi(f))\in R\subset A_0(R)$. Thus, $\mathsf C(\beta)=0$ if and only if $\xi(f^{[p]}) - H_f^{p-1}(\xi(f))=0$ for all $f\in A_0(R)$. This gives (1). (2) is then obtained similarly, by noting that $\iota_{H_f^\lp}(\beta^\lp)=(\iota_{H_f}\beta)^\lp=(\xi(f))^p$.
\end{proof}
\subsection{$h$-separable algebras}
Let $\Sch_{\!/{\mathbb P^1}}$ be the category of schemes over ${\mathbb P^1}$. To construct a group action on the category of $\mc A_{{\mathbb P^1}}$-modules it will be crucial to have some understanding of the algebra automorphisms of $\mc A_{{\mathbb P^1}}$.  This will become easier if we restrict to a certain subcategory in $\Sch_{\!/{\mathbb P^1}}$, where automorphisms of $\mc A_{{\mathbb P^1}}$ are closely related to  $G_0$. The goal of this section is to define this subcategory.

\begin{constr}
	Let $R\in \CAlg_{\mbb F_p[h]/}$. For an element $b\in R/h$ let $\tb\in R$ be a lift. Then the operation $b\mapsto \tb^p$ gives a well-defined map $R/h\ra R/h^p$. Indeed, given another lift $\tb'=\tb+hc$ we have $(\tb')^p=\tb^p+h^pc^p$. 
	
	Consider $\alpha_p(R/h)=\{a\in R/h\ |\  a^p=0\}\subset R/h$. Then restricting the above map to $\alpha_p(R/h)$ gives a map $\delta\colon \alpha_p(R/h) \ra hR/h^pR$.  
\end{constr}
\begin{df}\label{def:h-separable algebras}An $\mbb F_p[h]$-algebra $R$ is called \textit{$h$-separable} if 
	\begin{enumerate}
		\item  it is $h$-torsion free (i.e., multiplication by~$h$ is injective on~$R$); 
		\item the map $\delta_R\colon \alpha_p(R/h) \ra hR/h^pR$ is equal to 0.
	\end{enumerate}		
\end{df}
\begin{rem}
	The last condition is equivalent to the following:  for any $a\in R$ one has
	$$a^p\in hR \quad \Leftrightarrow \quad a^p\in h^pR.$$
\end{rem}
We denote by $\CAlg_{\mbb F_p[h]/}^{h-\sep}$ the category of $h$-separable $\mbb F_p[h]$-algebras.

\begin{example}\label{ex:separable algebras}
\begin{enumerate}
\item Let $C$ be an $\mbb F_p$-algebra. Then $R\coloneqq C\otimes_{\mbb F_p} \mbb F_p[h]$ is $h$-separable. 

\item Any  $\mbb F_p[h,h^{-1}]$-algebra $R$ is $h$-separable. Indeed, it is obviously $h$-torsion free and the second condition is vacuous since $R/h=0$.

\item $R\coloneqq \mbb F_p[h^{1/p}]$ is not $h$-separable. Indeed, $R/h\simeq \mbb F_p[x]/x^p$ where $x$ is the class of $h^{1/p}$, and the map $\delta_R$ sends $x$ to $h\in hR/h^pR$.

\item More generally, let $K$ be any finite field extension of $\mbb F_p(h)$ that is ramified at $0$. Let $R\subset K$ be the integral closure of $\mbb F_p[h]$ inside $K$. Then $R$ is not $h$-separable.
\end{enumerate}
\end{example}

Note that $\alpha_p(R/h)$ has a natural structure of $R$-module such that $\delta_R$ is a Frobenius-linear map of $R$-modules. Let $\delta_R^\ell\colon \alpha_p(R/h)\otimes_{R,F_R}R\ra hR/h^pR$ be its linearization. Obviously, $\delta_R=0 \Leftrightarrow \delta_R^\ell =0$.

The property of $h$-separability is local with respect to \'etale topology:
\begin{lm}\label{lem:h-separable survives etale covers}
	Let $R$ be an $\mbb F_p[h]$-algebra and let $R\ra R'$ be an \'etale cover. Then $R$ is $h$-separable if and only if $R'$ is.
\end{lm}
\begin{proof}
	$R'$ is faithfully flat over $R$, and so $R$ is $h$-torsion free if and only if so is $R'$. It remains to relate $\delta_R$ and $\delta_{R'}$. We have $R'/h\simeq R/h\otimes_RR'$, and the map $R/h \to R'/h$ is also \'etale. By flatness, we also have $hR'/h^pR'\simeq hR/h^pR\otimes_R R'$. We have a factorization 
	$$
	\xymatrix{R\ar[rr]^{F_R} \ar[d] && R \ar[rr]^{\id_R} \ar[d] && R\ar[d] \\
	R' \ar@/_1pc/[rrrr]_{F_{R'}}\ar[rr]^{W_{R'/R}}&& R'\otimes_{R, F_R}R \ar[rr]^{F_{R'/R}} && R'} 
	$$	
	of $F_{R'}$ as the composition $F_{R'/R}\circ W_{R'/R}$ (recall that  by definition $W_{R'/R}$ is the pull-back of $F_R$ to $R'$) and $F_{R'/R}$ is an equivalence since $R\to R'$ is \'etale. The same holds for $R$ replaced by $R/h$ and $R'$ by $R'/h$. We get that the natural map $\alpha_p(R/h)\otimes_R R'\to \alpha_p(R'/h)$ is an isomorphism, and the linearized map $\delta_{R'}^\ell\colon \alpha_p(R'/h)\otimes_{R',F_{R'}} R'\ra hR'/h^pR'$ is in fact identified with
	$$
	\delta_R^\ell\otimes R'\colon \alpha_p(R/h)\otimes_{R,F_R}R\otimes_R R'\ra hR/h^pR\otimes_RR'\simeq hR'/h^p R'.
	$$
	By faithful flatness we then get $\delta_{R'}^\ell=0\Leftrightarrow \delta_R=0$.
\end{proof}

\begin{rem}\label{rem: p-power of a lifting}
	Let $R$ be an $h$-separable $\mathbb F_p[h]$-algebra. Then for any $a\in R/h$ the image $\ta^p\in R/h^p$ only depends on the $p$-th power $a^p\in R/h$. Indeed, if $a,b\in R/h$ are such that $a^p=b^p$, then $a-b\in \alpha_p(R/h)$ and $
	\ta^p-\tb^p=\delta_R(a-b)=0$, so $\ta^p=\tb^p$. In other words, the map $(\widetilde{-})^p\colon R/h \ra R/h^p$ factors through the image of the absolute Frobenius on $R/h$. This gives a well-defined map 
	$$\mathsf s\colon (R/h)^p\simeq (R/h)/\alpha_p(R/h) \ra R/h^p.$$
\end{rem}

\begin{df}
A ${\mathbb P^1}$-scheme $X\ra {\mathbb P^1}$ is called \textit{$h$-separable} if its base change to ${\mathbb P^1}\backslash \{\infty\}$ has a Zariski cover by $\{U_i\coloneqq \Spec R_i\}$ where each $R_i$ is an $h$-separable $\mbb F_p[h]$-algebra.
\end{df}

\begin{example}\label{ex:h-separable schemes}
\begin{enumerate}
\item If the map $X\ra {\mathbb P^1}$ factors through ${\mathbb P^1}\backslash \{0\}$ then $X$ is $h$-separable. Indeed, in this case $X_{{\mathbb P^1}\backslash \{\infty\}}$ can be covered by affine $\mbb F_p[h,h^{-1}]$-schemes. 
\item Let $R$ be an $h$-separable $\mathbb F_p[h]$-algebra. Then the composition $\Spec R\ra {\mathbb P^1}\backslash \{\infty\} \ra {\mathbb P^1}$ defines an $h$-separable ${\mathbb P^1}$-scheme.
\end{enumerate}
\end{example}

We consider the big site $\Sch^{h-\sep{}}_{\!/{\mathbb P^1}}\subset \Sch_{\!/{\mathbb P^1}}$ of all {$h$-separable} schemes over ${\mathbb P^1}$ endowed with the \'etale topology\footnote{In fact for all the computations we will only need to use Zariski topology, but the \'etale topology also make sense for this setting.}. 

\begin{rem}\label{rem:subcategories of h-separable schemes} Any ${\mathbb P^1}$-scheme that lives over ${\mathbb P^1}\backslash \{0\}$ is $h$-separable, and so we have a natural embedding $\Sch_{\!/{\mathbb P^1}\backslash \{0\}}\subset \Sch^{h-\sep{}}_{\!/{\mathbb P^1}}$. We also have a fully faithful embedding $(\CAlg_{\mbb F_p[h]/}^{h-\sep})^{\mr{op}}\subset \Sch^{h-\sep{}}_{\!/{\mathbb P^1}}$ with the image given by those $h$-separable schemes that live over ${\mathbb P^1}\backslash \{\infty\}$ and are affine. Moreover, any $(X\ra {\mathbb P^1})\in \Sch^{h-\sep{}}_{\!/{\mathbb P^1}}$ has a Zariski cover $U_i\ra X$ such that each $U_i\ra {\mathbb P^1}$ belongs to one of the two subcategories.
\end{rem}

The functor ${\mc A}_{{\mathbb P^1}}\colon (f\colon X\ra {\mathbb P^1}) \mapsto \Gamma(X, f^*{\mc A}_{{\mathbb P^1}})$ defines a sheaf on $\Aff^{h-\sep{}}_{\!\!/{\mathbb P^1}}$. Its restriction to $(\CAlg_{\mbb F_p[h]/}^{h-\sep})^{\mr{op}}\subset \Aff^{h-\sep{}}_{\!\!/{\mathbb P^1}}$ is given by
$$
R\mapsto {\mc A}_{{\mathbb P^1}}(R)\simeq A_{{\mathbb P^1}\backslash \{\infty\}}\otimes_{\mathbb F_p[h]}R,
$$
(see Remark \ref{rem:description of A_S} for the definition of $A_{{\mathbb P^1}\backslash \{\infty\}}$). Explicitly, $$
{\mc A}_{{\mathbb P^1}}(R)\simeq R\langle x_i,y_j\rangle_{1\le i,j\le n}/([x_i,x_j]=[y_i,y_j]=y_j^p=x_i^p=0, [y_j,x_i]=h\cdot \delta_{ij}).
 $$ 
 Note that the underlying $R$-module of ${\mc A}_{{\mathbb P^1}}(R)$ is free of rank $p^{2n}$.

Taking the fiber $(X\ra {\mathbb P^1})\mapsto X_{\{0\}}$ over $\{0\}\in{\mathbb P^1}$ defines a map of sites $i_0^{-1}\colon\Sch^{h-\sep{}}_{\!/{\mathbb P^1}}\ra \Sch_{\!/\mbb F_p}$. On $(\CAlg_{\mbb F_p[h]/}^{h-\sep})^{\mr{op}}\subset \Sch^{h-\sep{}}_{\!/{\mathbb P^1}}$ this functor is given by sending an $\mbb F_p[h]$-algebra $R$ to $\Spec R/h$. Consider the sheaf $A_0$ on $\Sch_{\!/\mbb F_p}$ sending $X\mapsto \Gamma(X, A_0\otimes_{\mbb F_p}\mc O_X)$ and the corresponding pushforward $i_{0*}A_0$. Its restriction to $\Sch_{\!/{\mathbb P^1}\backslash \{0\}}\subset \Sch^{h-\sep{}}_{\!/{\mathbb P^1}}$ is 0, while on $(\CAlg_{\mbb F_p[h]/}^{h-\sep})^{\mr{op}}\subset \Aff^{h-\sep{}}_{\!\!/{\mathbb P^1}}$ it is given by $R\mapsto A_0(R/h)$.

We have a natural map $\mc A_{{\mathbb P^1}}\ra i_{0*}A_0$ (induced by the isomorphism $A_0\simeq i_0^*\mc A_{{\mathbb P^1}}$ as a quasi-coherent sheaf on ${\mathbb P^1}$). Its restriction to $\Sch_{\!/{\mathbb P^1}\backslash \{0\}}\subset \Sch^{h-\sep{}}_{\!/{\mathbb P^1}}$ is the zero map $\mc A_{{\mathbb P^1}}\ra 0$, while on $(\CAlg_{\mbb F_p[h]/}^{h-\sep})^{\mr{op}}\subset \Aff^{h-\sep{}}_{\!\!/{\mathbb P^1}}$ it is given by reduction modulo $h$
$$
{\mc A}_{{\mathbb P^1}}(R) \ra {\mc A}_{{\mathbb P^1}}(R)/h\simeq A_0(R/h).
$$

\subsection{Automorphisms of $\mc A_{{\mathbb P^1}}$}\label{ssec: automorphisms of A_S}
We let ${\ul{\Aut}}(\mc A_{{\mathbb P^1}})$ be the sheaf on the big fpqc site of ${\mathbb P^1}$ which sends $f\colon X\ra {\mathbb P^1}$ to algebra automorphisms of $f^*\mc A_{{\mathbb P^1}}$. Given the explicit description of $\mc A_{{\mathbb P^1}}$ in Remark \ref{rem:description of A_S} one sees that ${\ul{\Aut}}(\mc A_{{\mathbb P^1}})$ is represented by a finite type group scheme over ${\mathbb P^1}$. Since by definition $A_0\simeq i_0^*\mc A_{{\mathbb P^1}}$ (as quasi-coherent sheaves) we get a natural ``reduction mod $h$''-map $i_0^*\colon {\ul{\Aut}}(\mc A_{{\mathbb P^1}})\ra i_{0*}(\ul{\Aut}(A_0))$.

Recall the group subscheme $G_0\subset {\ul{\Aut}}(A_0)$ of restricted Poisson automorphisms of $(\Spec A_0,[\eta_{\mathsf{can}}])$ (Construction \ref{constr:G_0}).
The first goal of this section is to show that if we restrict to $\Sch^{h-\sep{}}_{\!/{\mathbb P^1}}\subset \Sch_{\!/{\mathbb P^1}}$ the map $i_0^*\colon {\ul{\Aut}}(\mc A_{{\mathbb P^1}})\ra i_{0*}(\ul{\Aut}(A_0))$ above factors through $i_{0*}G_0$. 

The key is to reinterpret the restricted Poisson structure on $A_{0}(R/h)$ in terms of ${\mc A}_{{\mathbb P^1}}(R)$.  Note that (since $A_0\coloneqq i_0^*\mc A_{{\mathbb P^1}}$) we have $A_0(R/h)\simeq {\mc A}_{{\mathbb P^1}}(R)/h$.
 
 \begin{rem}[Poisson structure]\label{rem:Poisson structure in terms of A_h}
 Let $R$ be  $h$-torsion free. Then ${\mc A}_{{\mathbb P^1}}(R)$ is $h$-torsion free as well, and we can identify $h^i{\mc A}_{{\mathbb P^1}}(R)/h^{i+1}{\mc A}_{{\mathbb P^1}}(R)\simeq A_{0}(R/h)$ (via dividing by $h^i$) for any $i\ge 0$. For an element $f\in A_{0}(R/h)$ let $\tf\in {\mc A}_{{\mathbb P^1}}(R)$ denote a lift under the projection ${\mc A}_{{\mathbb P^1}}(R)\epi A_{0}(R/h)$. One can define a Poisson bracket $\{-,-\}$ on $A_{0}(R/h)$ by the equation $h\{f,g\}=[\tf,\tg] \pmod{h^2}$. One sees immediately from the relations defining ${\mc A}_{{\mathbb P^1}}(R)$ that $\{x_i,x_j\}=\{y_i,y_j\}=0$, while $\{y_j,x_i\}=\delta_{ij}$ and so this Poisson structure agrees with the one defined by $\omega$ (see Construction \ref{constr:restricted structure on A_0}). 
 \end{rem}
 
 The restricted $p$-th power operation $-^{[p]}$ (corresponding to the restricted Poisson structure on $A_0(R/h)$) is more subtle to identify, and this is where $h$-separability of $R$ is important. 
 
 \begin{constr}[Restricted structure]\label{constr:restricted structure on A_0 via A_h}
 Note that for any $f\in  A_{0}(R/h)$ and any lift $\tf\in {\mc A}_{{\mathbb P^1}}(R)$ the expression $\tf^p$ is well-defined mod $h^p$. Indeed, by the Jacobson's formula (see Section \ref{sec:Jacobson's formula}) we have
 	$$
 	(\tf + h\tg)^p=\tf^p + h^p \tg^p + \sum_{i=1}^{p-1} h^iL_i(\tf,\tg),$$
 	where each $L_i(\tf,\tg)$ is a Lie polynomial of total Lie degree $p-1$. Since the commutator $[{\mc A}_{{\mathbb P^1}}(R),{\mc A}_{{\mathbb P^1}}(R)]\subset h {\mc A}_{{\mathbb P^1}}(R)$ we get that $h^iL_i(\tf,\tg)\in h^{p-1+i}{\mc A}_{{\mathbb P^1}}(R)$, and, consequently, that $(\tf + h\tg)^p=\tf^p \pmod{h^p}$.

 	 Recall the map $\mathsf s\colon (R/h)^p\ra R/h^p$ from Remark \ref{rem: p-power of a lifting}. We will denote by $\mathsf s$ its composition with the embedding $R/h^p\hookrightarrow {\mc A}_{{\mathbb P^1}}(R)/h^p$. The $p$-restricted structure then appears as follows, namely for any $f\in A_{0}(R/h)$ and a lift $\tf\in {\mc A}_{{\mathbb P^1}}(R)$ one can show as in Lemma \ref{lem:restricted structure} that the difference $\tf^p- \mathsf s(\ev_0(f)^p)$ is divisible by $h^{p-1}$ and define
 	$$
 	f^{[p]}\coloneqq \frac{\tf^p- \mathsf s(\ev_0(f)^p)}{h^{p-1}} \pmod{h}.
 	$$
 	Here we implicitly identified $h^{p-1}{\mc A}_{{\mathbb P^1}}(R)/h^p{\mc A}_{{\mathbb P^1}}(R)\simeq A_0(R/h)$ and considered $s(\ev_0(f)^p)$ as an element of ${\mc A}_{{\mathbb P^1}}(R)$ which is well-defined modulo $h^p$.
 	By the same argument as in Lemma \ref{lem:restricted structure}, one sees that this defines a restricted Poisson structure on $A_0(R/h)$.
 \end{constr}

\begin{rem}
	It is important to note that the map $\mathsf s(\ev_0(-)^p)\colon A_0(R/h)\ra R/h^p$ given by $f\mapsto \mathsf{s}(\ev_0(f)^p)\in R/h^p$ is preserved\footnote{Meaning for any $\phi\in \Aut_{R/h}(A_0(R/h))$ and $f\in A_0(R/h)$ we have $\mathsf s(\ev_0(f)^p)=\mathsf s(\ev_0(\phi(f))^p)$.} by any $R/h$-linear automorphism of $A_0(R/h)$. Indeed, $\Aut_{R/h}(A_0(R/h))$ acts trivially on $A_0(R/h)^\lp\simeq R/h$, and $\ev_0(f)^p$ is identified with the reduced twist map $W_{A_0(R/h)}^\lp\colon A_0(R/h)\ra A_0(R/h)^\lp$ which commutes with the action by automorphisms. 
\end{rem}
 \begin{lm}\label{lem:two restricted structures}
 	The restricted structures on $A_0(R/h)$ from Constructions \ref{constr:restricted structure on A_0 via A_h} and \ref{constr:restricted structure on A_0} coincide.
 \end{lm}
\begin{proof}
	Once we know that both constructions give restricted Poisson structures it is enough to check that they act in the same way on the algebra generators $x_i$ and $y_i$. We claim that in both cases $x_i^{[p]}=y_i^{[p]}=0$. Indeed, for Construction \ref{constr:restricted structure on A_0 via A_h} this follows from $\ev_0(x_i)=\ev_0(y_j)=0$. For Construction \ref{constr:restricted structure on A_0}, note that $H_{y_j}=\partial_{x_j}$ and $H_{x_i}=-\partial_{y_i}$. In both cases $H_f^{[p]}=0$, and so by Proposition \ref{prop:description of restricted symplectic structure} we need to check that $$H_{x_i}^{p-1}(\iota_{H_{x_i}}\eta_{\mathsf{can}})=H_{y_i}^{p-1}(\iota_{H_{y_i}}\eta_{\mathsf{can}})=0.$$ We have $\iota_{H_{x_i}}\eta_{\mathsf{can}}=0$ and $\iota_{H_{y_i}}\eta_{\mathsf{can}}=y_i$, and the vanishing follows.
\end{proof}
\begin{lm}\label{lem:the reduction lies in G_0 for separable guys}
	Let $R$ be an $h$-separable $\mbb F_p[h]$-algebra. If $\phi$ is an $R$-linear algebra  automorphism of ${\mc A}_{{\mathbb P^1}}(R)$, then the reduction of $\phi$ modulo $h$ preserves the restricted Poisson structure on $A_{0}(R/h)$. In other words, $\phi_0\in G_0(R/h)$. 
\end{lm}
\begin{proof}
	By Remark \ref{rem:Poisson structure in terms of A_h} we have
	$$
	h\cdot \{\phi_0(f),\phi_0(g)\}=[\phi(\tf),\phi(\tg)]=\phi([\tf,\tg])=h\cdot \phi_0(\{f,g\}) \pmod{h^2},
	$$
	thus $\phi_0$ preserves the Poisson structure. 
	Similarly, $$h^{p-1}(\phi_0(f))^{[p]}=(\widetilde{\phi_0(f)})^p - \mathsf s(\ev_0({\phi_0(f)})^p) = \phi\left(\tf^p- \mathsf s(\ev_0({f})^p)\right) = h^{p-1}\cdot \phi_0(f^{[p]}) \pmod{h^p},$$ 
	so $\phi_0$ also preserves the restricted structure.
\end{proof}

\begin{cor}
The map $i_0^*\colon {\ul{\Aut}}(\mc A_{{\mathbb P^1}})\ra i_{0*}{\ul{\Aut}}(A_0)$ of sheaves on $
\Sch^{h-\sep}_{\!/{\mathbb P^1}}$ factors through $i_{0*}G_0$.
\end{cor}

\begin{proof}
Indeed, by Remark \ref{rem:subcategories of h-separable schemes} it is enough to check on the subcategories $\Sch_{\!/{\mathbb P^1}\backslash \{0\}}\subset 
\Sch^{h-\sep}_{\!/{\mathbb P^1}}$ and $(\CAlg_{\mbb F_p[h]/}^{h-\sep})^{\mr{op}}\subset 
\Sch^{h-\sep}_{\!/{\mathbb P^1}}$. The restriction of both $i_{0*}{\ul{\Aut}}(A_0)$ and $i_{0*}G_0$ to the first subcategory is trivial and so the statement is trivial as well. For $R\in (\CAlg_{\mbb F_p[h]/}^{h-\sep})^{\mr{op}}\subset 
\Sch^{h-\sep}_{\!/{\mathbb P^1}}$ this is exactly the statement of Lemma \ref{lem:the reduction lies in G_0 for separable guys}.
\end{proof}
Let $\mathsf r\colon {\ul{\Aut}}(\mc A_{{\mathbb P^1}})\ra  i_{0*}G_0$ denote the resulting map of sheaves on $\Sch^{h-\sep{}}_{\!/{\mathbb P^1}}$. We are going to describe the kernel of $\mathsf r$. Note that $\mc A_{{\mathbb P^1}}^\times$ acts on $\mc A_{{\mathbb P^1}}$ by conjugation and that its image lands in $\Ker(\mathsf r)$. Indeed, it is enough to see separately for objects over ${\mathbb P^1}\backslash \{0\}$ or for spectra of $h$-separable $\mathbb F_p[h]$-algebras. The restriction of $\Ker(\mathsf r)$ to ${\mathbb P^1}\backslash \{0\}$ is the whole ${\ul{\Aut}}(\mc A_{{\mathbb P^1}})$, so there is nothing to prove. For an $h$-separable $\mathbb F_p[h]$-algebra $R$ the action of ${\mc A}_{{\mathbb P^1}}(R)^\times$ on ${\mc A}_{{\mathbb P^1}}(R)$ by conjugation is trivial modulo $h$ since $[{\mc A}_{{\mathbb P^1}}(R),{\mc A}_{{\mathbb P^1}}(R)]\subset h\cdot {\mc A}_{{\mathbb P^1}}(R)$. Thus we have a natural map $$\Ad\colon {\mc A}_{{\mathbb P^1}}^\times \ra  \Ker(\mathsf r)\subset {\ul{\Aut}}(\mc A_{{\mathbb P^1}})$$ of Zariski sheaves on $\Sch^{h-\sep{}}_{\!/{\mathbb P^1}}$. The map of algebras $\mc O_{{\mathbb P^1}}\ra \mc A_{{\mathbb P^1}}$ induces a homomorphism $\mbb G_m \ra \mc A_{{\mathbb P^1}}^\times$ which factors through the invertible elements in the center $Z(\mc A_{{\mathbb P^1}})^\times$.

\begin{pr}\label{prop:automorphisms trivial modulo h are inner locally}
	The map $\Ad\colon {\mc A}_{{\mathbb P^1}}^\times \rightarrow{} \Ker(\mathsf r)$ is a surjection of Zariski sheaves on $\Sch^{h-\sep{}}_{\!\!/{\mathbb P^1}}$. The kernel of $\Ad(-)$ is canonically identified with $\bG_m\subset {\mc A}_{{\mathbb P^1}}^\times$.
\end{pr}	

\begin{proof} It is enough to prove the statement separately for the restrictions to $\Sch_{\!/{\mathbb P^1}\backslash \{0\}}\subset 
\Sch^{h-\sep}_{\!/{\mathbb P^1}}$ and $(\CAlg_{\mbb F_p[h]/}^{h-\sep})^{\mr{op}}\subset 
\Sch^{h-\sep}_{\!/{\mathbb P^1}}$. 

For $\Sch_{\!/{\mathbb P^1}\backslash \{0\}}\subset 
\Sch^{h-\sep}_{\!/{\mathbb P^1}}$ we have $\Ker(\mathsf r)\simeq \ul{\Aut}(\mc A_{{\mathbb P^1}\backslash \{0\}})$ and we want to show that there is a short exact sequence 
$$
1\ra \mbb G_m \ra \mc A_{{\mathbb P^1}\backslash \{0\}}^\times \ra \Aut(\mc A_{{\mathbb P^1}\backslash \{0\}}) \ra 1.
$$
By Remark \ref{rem:description of A_S} we can identify $\mc A_{{\mathbb P^1}\backslash \{0\}}$ with the matrix algebra $\Mat_{p^n}$, and the sequence above turns into 
$$
1\ra \mbb G_m \ra \mr{GL}_{p^n} \ra \mr{PGL}_{p^n} \ra 1,
$$
which is exact in Zariski topology.

 The case of $(\CAlg_{\mbb F_p[h]/}^{h-\sep})^{\mr{op}}\subset 
\Sch^{h-\sep}_{\!/{\mathbb P^1}}$ is more interesting. Let us first identify the kernel of $\Ad$. By definition, it consists of the units $Z({\mc A}_{{\mathbb P^1}})^\times$ in the center. Every element $a\in {\mc A}_{{\mathbb P^1}}(R)$ can be uniquely written as $\sum_{I,J}r_{I,J}x^Iy^J$ with $r_{I,J}\in R$ (with $I=(i_1,\ldots,i_n)$, $J=(j_1,\ldots,j_n)$ and $0 \le i_k, j_l\le p-1$). Since $R$ is $h$-torsion free, from the relations defining ${\mc A}_{{\mathbb P^1}}(R)$ it is clear that if $r_{I,J}\neq 0$ for some non-zero $I$ or $J$ then there is $x_i$, or $y_j$ such that one of the commutators $[a,x_i]$ or $[a,y_j]$ is non-zero. Thus $Z({\mc A}_{{\mathbb P^1}}(R))^\times\simeq R^\times$ and $\Ker(\Ad(-))\simeq \bG_m$.
	
	It remains to show that any automorphism of ${\mc A}_{{\mathbb P^1}}(R)$ that is trivial modulo $h$ becomes inner after a Zariski localization $R\ra R'$. We will first prove this after inverting $h$, then for $h$-completion of $R$, and then will obtain the result for $R$ using Beauville-Laszlo gluing. Note that if $h\in R$ is invertible then $\Spec R$ lies over $\Sch_{\!/{\mathbb P^1}\backslash \{0\}}$, and we are already done.
	
	If $R$ is $h$-complete, then ${\mc A}_{{\mathbb P^1}}(R)$ is also $h$-complete since it is a finite free $R$-module. We will approximate ${\mc A}_{{\mathbb P^1}}(R)$ by ${\mc A}_{{\mathbb P^1}}(R)/h^n$ for all $n\ge 0$. To start the induction on $n$, let us show that any $\phi\in \Ker(\mathsf r)$ is equal to an inner automorphism modulo $h^2$. In other words, we would like to find $x\in {\mc A}_{{\mathbb P^1}}(R)^\times$ such that $\phi=\Ad(x)\mod h^2$. Note that the map $\phi-\id \mod h^2 \colon {\mc A}_{{\mathbb P^1}}(R)\rightarrow {\mc A}_{{\mathbb P^1}}(R)/h^2$ factors through $h\cdot {\mc A}_{{\mathbb P^1}}(R)$ and gives a well defined map
	$\xi_\phi \colon A_0(R/h)\simeq {\mc A}_{{\mathbb P^1}}(R)/h \rightarrow h{\mc A}_{{\mathbb P^1}}(R)/h^2{\mc A}_{{\mathbb P^1}}(R) \simeq A_0(R/h)$, which is an $R/h$-linear derivation\footnote{Indeed, $$h\cdot \xi_\phi(xy)=\phi(\tx\cdot \ty)-\tx\cdot \ty=\phi(\tx)\cdot \phi(\ty)-\tx\cdot \ty=h\cdot (\xi_\phi(x)\cdot y+x\cdot \xi_\phi(y)) \pmod{h^2}.$$} of $A_0(R/h)$. For $\phi=\Ad(x)$ with $x=x_0\mod h$ we have 
	$$
	h\cdot \xi_{\Ad(x)}(a)=xax^{-1}-a= [x,a] \cdot x^{-1}= h\cdot \tfrac{1}{x_0}\{x_0,a\} \pmod{h^2},
	$$ and so one sees that $\xi_{\Ad(x)}=\frac{1}{x_0}H_{x_0}$. Thus it is enough to show that for $\phi\in \Ker(\mathsf r)$ the corresponding vector field $\xi_\phi$ is log-Hamiltonian.

	Let $f,g\in A_0(R/h)$ and let $\widetilde{f},\widetilde{g}$ be their lifts to ${\mc A}_{{\mathbb P^1}}(R)$. Since  $\phi([\widetilde{f},\widetilde{g}])=[\phi(\widetilde{f}),\phi(\widetilde{g})]$, one deduces that
    $$
    \xi_\phi(\{f,g\})=\{\xi_\phi(f),g\}+\{f,\xi_\phi(g)\}.
    $$In other words, $\xi_\phi$ is Poisson (namely, $L_{\xi_\phi}\omega=0$). By Lemma \ref{lem:description of log and usual Hamiltonian vector fields} it remains to show that $\xi_\phi$ satisfies equation (\ref{eq:log-Ham}).  Moreover, it is enough to check (\ref{eq:log-Ham}) on $f\in I_0(R/h)=\{f\in A_0(R/h)\ |\ \ev_0(f)=0\}$: indeed, the constants $R/h$ and $I_0(R/h)$ span all $A_0(R/h)$. By Construction \ref{constr:restricted structure on A_0 via A_h}, for $f\in I_0(R/h)$ and any lift $\tf\in {\mc A}_{{\mathbb P^1}}(R)$ we have 
	  $$
	  \tf^p=h^{p-1}\cdot f^{[p]}\pmod{h^p}.
	  $$ 
	  Let $\widetilde{f^{[p]}}$ be a lift of $f^{[p]}$. We then have 
	  $$
	  h^{p}\cdot \xi_\phi(f^{[p]})=h^{p-1}\cdot(\phi(\widetilde{f^{[p]}})-\widetilde{f^{[p]}})= \phi(\tilde{f}^p)- \tilde{f}^p=(\tf+h\cdot \widetilde{\xi_\phi(f)})^p-\tf^p\pmod{h^{p+1}}.
	  $$
	  Using Jacobson's formula (Lemma \ref{lem:Jacobson's formula}), the right-hand side is identified with $h^p(\xi_\phi(f)^p)+h\cdot L_1(\tf,\widetilde{\xi_\phi(f)}))\pmod{h^{p+1}}$. From this, using that $L_1(x,y)=\ad(x)^{p-1}(y)$ (Remark \ref{rem:L_1}) and identifying $A_0(R/h)\simeq h^p{\mc A}_{{\mathbb P^1}}(R)/h^{p+1}{\mc A}_{{\mathbb P^1}}(R)$, we get that
	  $$
	  \xi_\phi(f^{[p]})=\xi_\phi(f)^p +H_f^{p-1}(\xi_\phi(f)),
	  $$
	  as desired.

	For $n\ge 2$ one shows analogously that if $\phi\colon {\mc A}_{{\mathbb P^1}}(R) \rightarrow {\mc A}_{{\mathbb P^1}}(R)$ is an automorphism which is identity modulo $h^n$ then there exists an element $x\in {\mc A}_{{\mathbb P^1}}(R)$ such that $\phi = \Ad(1+h^{n-1}x)$ modulo $h^{n+1}$. Namely, such $\phi$ also defines a derivation $\xi_\phi$ of $A_0(R/h)$ by considering the map  
	$
	\phi -\id \mod h^{n+1}\colon {\mc A}_{{\mathbb P^1}}(R) \rightarrow {\mc A}_{{\mathbb P^1}}(R)/h^{n+1}
	$ which factors through ${\mc A}_{{\mathbb P^1}}(R)/h\simeq A_0(R/h)$ and whose image lies in $h^n{\mc A}_{{\mathbb P^1}}(R)/h^{n+1}{\mc A}_{{\mathbb P^1}}(R)\simeq A_0(R/h)$. As in the case of $n=1$, looking at the action on a commutator, one sees that $\xi_\phi$ is Poisson. We also have 
	$$
	h^n\cdot \xi_{\Ad(1+h^{n-1}x)}(a)=(1+h^{n-1}x)a(1-h^{n-1}x+h^{2n-2}x^2 -\ldots)-a=h^{n-1}\cdot [x,a]=h^n\cdot \widetilde{\{x_0,a\}}\pmod{h^{n+1}};
	$$ 
 in other words, $\xi_{\Ad(1+h^{n-1}x)}=H_{x_0}$. Let $f\in I_0(R/h)$, with a lift $\tilde f\in \mc A_{\mbb P^1}(R)$ and also a chosen lift $\widetilde{f^{[p]}}\in \mc A_{\mbb P^1}(R)$ of $f^{[p]}\in A_0(R/h)$. Then $\tf^p - h^{p-1}\cdot \widetilde{f^{[p]}}=h^pc$ for some $c\in \mc A_{\mbb P^1}(R)$. Given $\phi$ that is trivial modulo $h^n$ with $n\ge 2$, we get
	$$
	h^{p+n-1}\cdot \xi_\phi(f^{[p]})=h^{p-1}\cdot(\phi(\widetilde{f^{[p]}})-\widetilde{f^{[p]}})=\phi(\tilde f^p)-\tilde f^p -h^p(\phi(c)-c)=(\tf+h^n\cdot \widetilde{\xi_\phi(f)})^p-\tf^p\pmod{h^{p+n}},
	$$
	which (using Jacobson's formula again) gives
	$$
	\xi_\phi(f^{[p]})=H_f^{p-1}(\xi_\phi(f))
	$$
	So by Lemma \ref{lem:description of log and usual Hamiltonian vector fields} $\xi_\phi$ is Hamiltonian and $\phi$ agrees with $\Ad(1+h^{n-1}x)$ for some $x$.
	
	This way, starting with an automorphism $\phi\in \Ker(\mathsf{r})$, we can construct a sequence of elements $x_1,x_2,x_3,\ldots\in {\mc A}_{{\mathbb P^1}}(R)$ such that 
	$$
	\phi=\ldots\circ \Ad(1+h^2x_3)\circ \Ad(1+hx_2)\circ \Ad(x_1)=\Ad(\ldots\cdot  (1+h^2x_3)(1+hx_2)x_1),
	$$
	where the infinite product on the right-hand side is well-defined since $R$ is $h$-complete.

	Let now $R$ be an arbitrary $h$-separable $\mathbb F_p[h]$-algebra. We will show that the map of (Zariski) sheaves ${\mc A_{{\mathbb P^1}}^\times} \rightarrow \Ker(\mathsf r)$ on $\Sch^{h-\sep{}}_{\!/{\mathbb P^1}}$ is surjective. It is enough to check this on stalks, and so we can assume that $R$ is local (and, in particular, that $\Pic(R)=0$). Let $\phi\in \Ker(\mathsf{r})(R)$: in other words, it is an $R$-algebra automorphism of ${\mc A}_{{\mathbb P^1}}(R)$, which reduces to the identity modulo $h$. Let $R^\wedge_h$ denote the $h$-adic completion of $R$. Since ${\mc A}_{{\mathbb P^1}}(R)$ is a free $R$-module of finite rank, by Beauville-Laszlo theorem \cite{bl}, $\phi$ is uniquely encoded by a pair of automorphisms $\phi_1\in \Aut_{R^\wedge_h}({\mc A}_{{\mathbb P^1}}(R^\wedge_h))$, $\phi_2\in \Aut_{R[h^{-1}]}({\mc A}_{{\mathbb P^1}}(R[h^{-1}]))$ such that their images in $\Aut_{R^\wedge_h[h^{-1}]}({\mc A}_{{\mathbb P^1}}(R^\wedge_h[h^{-1}]))$ agree; moreover $\phi_1$ is the identity modulo $h$. By the $h$-complete case we can find a unit $a_1\in {\mc A}_{{\mathbb P^1}}^\times(R^\wedge_h)$, such that $\phi_1=\Ad(a_1)$; the image of $\phi_1$ in $\Aut_{R^\wedge_h[h^{-1}]}({\mc A}_{{\mathbb P^1}}(R^\wedge_h[h^{-1}]))$ then is also given by conjugation by the image of $a_1$ in ${\mc A}_{{\mathbb P^1}}^\times(R^\wedge_h[h^{-1}])$. For $\phi_2$, let $L_{\phi_2}$ denote the line bundle on $\Spec R[h^{-1}]$ whose values on $U$ are given by $a\in {\mc A}_{{\mathbb P^1}}(U)$ such that $ax=\phi_2(x)a$ for any $x\in {\mc A}_{{\mathbb P^1}}(U)$; we have a natural embedding of vector bundles $L_{\phi_2}\hookrightarrow (\mc A_{\mbb P^1})_{|\Spec R[h^{-1}]}$, which we can glue with the line subbundle $ \mc O_{\Spec R_h^\wedge}\subset (\mc A_{\mbb P^1})_{|\Spec R^\wedge_h}$, defined by $a_1\in {\mc A}_{{\mathbb P^1}}(R^\wedge_h)$, via Beauville-Laszlo gluing. The result is a line bundle $L\subset \mc A_{\mbb P^1}$ on $\Spec R$. Since $\Pic(R)=0$ we can trivialize $L$, getting an element $a\in \mc A_{\mbb P^1}(R)^\times$. By construction, one has $\phi(x)a=ax$ on each of the gluing charts for the Beauville-Laszlo gluing, and so $\phi=\Ad(a)$ in $\Aut_R({\mc A}_{{\mathbb P^1}}(R))$. 
	
\end{proof}

\begin{rem}\label{rem:3+1-term exact sequence}
	From Proposition \ref{prop:automorphisms trivial modulo h are inner locally} we get an exact sequence of sheaves on $\Sch^{h-\sep{}}_{\!/{\mathbb P^1}}$:
	$$
	1\rightarrow \bG_m \rightarrow \mc {A}_{{\mathbb P^1}}^\times \rightarrow \ul{\Aut}(\mc A_{{\mathbb P^1}}) \xrightarrow{\mathsf r} i_{0*}G_0.
	$$	
	In fact, if we restrict to even smaller subcategory of locally constant schemes (see the discussion just above Lemma \ref{M_Sequivariantstructure}) one can also show that the morphism $\mathsf r$ becomes surjective, and so the sequence above becomes a 4-term exact sequence.
\end{rem}

Further we will need the following corollary:

\begin{cor}\label{app.lemma.inneraut} Let $T$ be an $\mbb F_p$-scheme and let $\pi\colon T\times{\mathbb P^1}\ra {\mathbb P^1}$ be the projection. Let $R$ be the stalk of $\cO_{T\times {\mathbb P^1}}$ at a point $x\in T\times {\mathbb P^1}$ and consider the corresponding map $\Spec R\ra {\mathbb P^1}$. Then every $R$-linear automorphism of $\mc A_{{\mathbb P^1}}(R)$ that lies in $\Ker(\mathsf r)(R)$ is inner. 
\end{cor}
\begin{proof}
 Note that $\Spec R\ra {\mathbb P^1}$ is $h$-separable. Indeed, either $\pi(x)=\{\infty\}$ and then $\Spec R$ lies over ${\mathbb P^1}\backslash \{0\}$ or $\pi(x)\in {\mathbb P^1}\backslash \{\infty\}$ and then $R$ is a localization of a $\mathbb F_p[h]$-algebra of the form Example \ref{ex:separable algebras}(1) (and thus is also $h$-separable).
\end{proof}
\subsection{The algebra $\mc A_{{\mathbb P^1}}^\flat$ and the $G_0$-action on $\mc A_{{\mathbb P^1}}^\flat-\Mod$}\label{ssec:A_Sflat}
We implicitly fix a number $n\in \mbb N$. As in Construction \ref{constr:A_0} we endow $X\coloneqq \Spec A_0$ with a restricted symplectic structure given by the 1-form $\eta_{\mathsf{can}}= \sum_i{y_idx_i}$.

\begin{df} The sheaf of algebras $\mc A_{{\mathbb P^1}}^\flat$ is defined as $\mc D_{\Spec A_0,\frac{[\eta_{\mathsf{can}}]}{h},{\mathbb P^1}}$.
\end{df}
\begin{constr}[$G_0$-action on $\mc A_{{\mathbb P^1}}^\flat$]\label{constr:G_0action on the cental reduction} This is a variant of the construction in \cite[Section 3.1]{bv}.
Let $S$ be a scheme over $\bF_p$. Denote by $\pro_{{\mathbb P^1}}\colon S \times {\mathbb P^1} \to {\mathbb P^1}$ the projection. Applying Construction \ref{constr:central reductions over S} 
(with $X=  S \times \Spec A_0$ and $[\eta]$ being the pullback of $[\eta_{\mathsf{can}}]$ to $X$) we have a homomorphism
$$G_0(S) \to \Aut (\pro_{{\mathbb P^1}}^*\mc A_{{\mathbb P^1}}^\flat ),$$
where the group displayed on the right-hand side consists of all $\cO_{S \times {\mathbb P^1}}$-linear algebra automorphisms of $\pro_{{\mathbb P^1}}^*\mc A_{{\mathbb P^1}}^\flat $. This defines a homomorphism of group schemes over ${\mathbb P^1}$: 
\begin{equation}\label{eq:action_of_G_0_on_A_flat_psi}
  \psi\colon G_0\times {\mathbb P^1}\to \underline{\Aut} (\cA_{{\mathbb P^1}}^\flat).  
\end{equation}
\end{constr}
\begin{rem} \label{rem:splitting of bemol algebra}
Note that $A_0$ is a particular example of  the algebra $C$ considered in Section
\ref{ssec:twistor reduced Weyl algebra} where we take the number of parameters to be $2n$. We have $A_0^\lp\simeq \mbb F_p$ and since $y_i^p=0\in A_0$ it follows that 
$$
(\eta_{\mathsf{can}})^{\lp} = \sum_iy_i^pdx_i=0.
$$
Using \eqref{eq:local_description_of_central_reduction}
 we get that in fact 
\begin{equation}\label{eq:iso of sharp with non sharp}
	\mc D_{\Spec A_0,\frac{[\eta_{\mathsf{can}}]}{h},{\mathbb P^1}} \simeq \mc D_{\Spec A_0,0,{\mathbb P^1}}.
	\end{equation} Thus, the discussion in \S \ref{ssec:twistor reduced Weyl algebra} applies to $\mc D_{\Spec A_0,\frac{[\eta_{\mathsf{can}}]}{h},{\mathbb P^1}}$ as well. In particular,  $\mc A_{{\mathbb P^1}\!\setminus \!\{0\}}^\flat$ is a split Azumaya algebra over
${\mathbb P^1}\!\setminus\!\{0\}$   (see Example \ref{ex:D_X0S}).
\end{rem}	
\begin{warning}
  The algebra  $ \mc D_{\Spec A_0,0,{\mathbb P^1}}$ carries a natural action of the group scheme $\underline{\Aut}(A_0)$  and, in particular,
  its subgroup $G_0$. However, the isomorphism \eqref{eq:iso of sharp with non sharp} is not $G_0$-equivariant (with respect to $G_0$-action on $\mc D_{\Spec A_0,\frac{[\eta_{\mathsf{can}}]}{h},{\mathbb P^1}}$ introduced in Construction \ref{constr:G_0action on the cental reduction}). 
\end{warning}

Let $(X,[\eta])$ be a restricted symplectic $S$-scheme. Let $\mc M_{X,[\eta]}\ra X^\lp$ be the $G_0$-torsor of Darboux frames (see Construction \ref{constr:torsor of Darboux frames}). Let $\pro_{{\mathbb P^1}}
\colon \mc M_{X,[\eta]}\times {\mathbb P^1} \ra {\mathbb P^1}$ be the projection. The $G_0$-action on $\mc A_{{\mathbb P^1}}^\flat$ endows $\pro_{{\mathbb P^1}}^*\mc A_{{\mathbb P^1}}^\flat$ with the structure of $G_0$-equivariant sheaf of algebras on $\mc M_{X,[\eta]}\times {\mathbb P^1}$ (where the action of $G_0$ on the product is given by its natural action on the first factor). Denote the corresponding sheaf on the quotient scheme $X^\lp\times {\mathbb P^1}$ by 
$$
\mc M_{X,[\eta]}\times^{G_0} \mc A_{{\mathbb P^1}}^\flat. 
$$
\begin{pr}\label{prop:localization}
Let $(X,[\eta])$ be a restricted symplectic $S$-scheme. Then one has a natural isomorphism 
$$
\mc M_{X,[\eta]}\times^{G_0} \mc A_{{\mathbb P^1}}^\flat \simeq \mc D_{X,\frac{[\eta]}{h},{\mathbb P^1}}.
$$ 
\end{pr}
\begin{proof}
    It suffices to construct, for every morphism $s: Z\to X^\lp$ from an affine scheme $Z$ and a lift $\tilde s: Z\to \mc M_{X,[\eta]}$ 
    a natural isomorphism of sheaves of $\cO(Z)\otimes \cO_{{\mathbb P^1}}$-algebras over ${\mathbb P^1}$
    \begin{equation}\label{pf:localization_for_central_red}
     \gamma_{\tilde s}\colon  s^*  \mc D_{X,\frac{[\eta]}{h},{\mathbb P^1}} \iso A_{{\mathbb P^1}}^\flat \otimes \cO(Z),
    \end{equation}
    such that, for any two lifts  $\tilde s_i: Z\to \mc M_{X,[\eta]}$, $i=1,2$, of $s$ with $\tilde s_1 = g\circ \tilde s_2$, for some $g\in G_0(Z)$, we have
     \begin{equation}\label{pf:localization_for_central_red_pr}
     \gamma_{\tilde s_1} = g\circ \gamma_{\tilde s_2}. 
    \end{equation}
 In fact, $s^*  \mc D_{X,\frac{[\eta]}{h}, {\mathbb P^1}}$ is identified with the central reduction  $\mc D_{X\times_{X^\lp} Z,\frac{s^*[\eta]}{h}, {\mathbb P^1}}$  of the algebra of twistor differential operators on $X\times_{X^\lp} Z$ relative to $Z$. The lift $\tilde s$ determines an isomorphism of $Z$-schemes $X\times_{X^\lp} Z \iso \Spec A_0  \times Z$ such that  the pullback of $[\eta_{\mathsf{can}}]$ is 
 $[\eta]$. This gives \eqref{pf:localization_for_central_red}. The formula \eqref{pf:localization_for_central_red_pr} follows from the functoriality of the formation of $\mc D_{X,\frac{[\eta]}{h},{\mathbb P^1}}$.
\end{proof}

\section{Quasi-coherent sheaves of abelian categories}\label{twisting.stacks} 
\subsection{Quasicoherent sheaves of abelian categories.}\label{Sheavesofcategories} The standard references for the material in this section are \cite[Section 5]{Stefanich} and \cite[Section 10.6]{LurieSAG}.

For a (classical) additive category $\cC$, let $\Center (\cC)$ denote  the \textit{center} of $\cC$, that is the ring of endomorphisms 
of the identity functor $\Id\colon \cC \to \cC$. Given a commutative ring $R$, an $R$-linear structure on a presentable additive category $\cC$ is given by a ring homomorphism $R  \to \Center (\cC)$. Note that by \cite[Remark D.1.3.4]{LurieSAG} this is equivalent to providing a continuous action of the monoidal category $R\lmod$ on $\cC$. In particular, for any $M\in R\lmod$ we have a natural endofunctor of $\cC$ given by tensor product with $M$: $X\mapsto X\otimes_R M$. Let $\mr{LinCat}_R^{\mr{Add}}$ denote the $(2,1)$-category of (classical) $R$-linear presentable additive categories (considered as an $(\infty,1)$-category).

Following \cite[Definition 1.3.5.1]{LurieHA}, by a \textit{Grothendieck abelian category} we will mean a presentable abelian category $\cC$ for which the collection of monomorphisms is closed under filtered colimits. In particular, $\cC$ is cocomplete. We will denote by $\mr{LinCat}_R^{\mr{Ab}}\subset \mr{LinCat}_R^{\mr{Add}}$ the full subcategory spanned by those $R$-linear additive categories whose underlying category is equivalent to a Grothendieck abelian category.  We will call an object of $\mr{LinCat}_R^{\mr{Ab}}$ \textit{an abelian $R$-linear category} (or an abelian category over $R$), see \cite[Definition D.1.4.1.(3)]{LurieSAG}. Note that morphisms in $\mr{LinCat}_R^{\mr{Ab}}$ are given by continuous $R$-linear functors, and they are not assumed to be exact.                                                                           

 Given a homomorphism $R\to R'$  there is a natural \textit{base change} functor $\mr{LinCat}_R^{\mr{Ab}} \ra \mr{LinCat}_{R'}^{\mr{Ab}}$, $\cC \mapsto \cC \otimes_R R'$ which is left adjoint to the forgetful functor  $\mr{LinCat}_{R'}^{\mr{Ab}} \ra \mr{LinCat}_{R}^{\mr{Ab}}$ (\cite[Construction D.2.4.1]{LurieSAG}). The category $\cC \otimes_R R'\in  \mr{LinCat}_{R'}^{\mr{Ab}}$ can be explicitly described as the $R'$-linear category formed by pairs $(M, \alpha)$, where $M$ is an object of $\cC$  and $\alpha\colon R' \to \End(M)$ is a homomorphism of $R$-algebras. The unit of adjunction $ \cC \ra \cC \otimes_R R'$ is given by the functor that sends an object $X\in \cC$ to a pair given by $X\otimes_R R'$ with the natural action of $R'$ on the right. Note that if $R'$ is flat over $R$ this functor is exact.

A key useful result that we are going to consistently use below is a form of flat descent for $R$-linear abelian categories. More precisely, by \cite[Corollary D.6.8.4]{LurieSAG} the functor
$$
\mr{LinCat}_{-}^{\mr{Ab}}\colon \CAlg_{\mbb Z} \ra \widehat{\mr{Cat}}_\infty
$$
sending a commutative ring $R$ to $\mr{LinCat}_{R}^{\mr{Ab}}$ forms an fpqc sheaf.

\begin{constr}\label{constr:quasi-coherent sheaves of categories}  
Let $S$ be a base affine scheme and consider the category $\mr{PStk}_{/S}$ of prestacks over $S$, by which we mean accessible functors $\cX\colon \Aff_{S}^{\op} \ra \mr{Spc}$ to the category of spaces.	Given a prestack\footnote{By $\mr{Spc}$ here we mean the $\infty$-category of spaces, also known as $\infty$-groupoids.} $\cX\colon \Aff_{S}^{\op} \ra \mr{Spc}$ we define the category $\mr{LinCat}_{\cX}^{\mr{Ab}}$ as the value on $\cX$ of the right Kan extension of the functor $\mr{LinCat}_{-}^{\mr{Ab}}$; explicitly, one has
$$
\mr{LinCat}_{\cX}^{\mr{Ab}}:= \lim_{\Spec R \ra \cX } \mr{LinCat}_{R}^{\mr{Ab}},
$$
where the limit is taken over the slice $(\Aff_{S}/\mc X)^{\op}$.  
Note that, as a limit of $(2,1)$-categories, $\mr{LinCat}_{\cX}^{\mr{Ab}}$ is still a $(2,1)$-category. Given a map of prestacks $f\colon \cX \ra \cY$ we will denote by
$$
f^*\colon \mr{LinCat}_{\cY}^{\mr{Ab}}\ra \mr{LinCat}_{\cX}^{\mr{Ab}}
$$ 
the natural pull-back functor. The association $\cX \mapsto \mr{LinCat}_{\cX}^{\mr{Ab}}$ sends colimits to limits. Since $\mr{LinCat}_{-}^{\mr{Ab}}$ was an fpqc sheaf, it also factors through the fpqc-sheafification, namely the pull-back 
$$
\iota^*\colon \mr{LinCat}_{\mr L_{\mr{fpqc}}^{\mr{acc}}(\cX)}^{\mr{Ab}} \ra \mr{LinCat}_{\cX}^{\mr{Ab}},
$$
for the natural map $\iota\colon \cX \ra \mr L_{\mr{fpqc}}^{\mr{acc}}(\cX)$ to the (accessible) fpqc sheafification is an equivalence.

\end{constr}

\begin{df}\label{def:sheaf of abelian categories} Let $\cX\colon \Aff_{S}^{\op} \ra \mr{Spc}$ be a prestack. 
	\begin{itemize}
	\item[(1)] We will call an object $\fC$ of $\mr{LinCat}_{\cX}^{\mr{Ab}}$ a \textit{quasi-coherent sheaf of abelian categories on $\cX$}. It amounts to the following piece of data:
	\begin{itemize}
		\item[(i)]  For every affine scheme $Z\ra \cX$ an abelian category  $\fC(Z)$ over $\mc O(Z)$.
		\item[(ii)] For every morphism $u\colon Z' \to Z $ of affine schemes over $\cX$, an equivalence 
		$$u^*\colon  \fC(Z) \otimes_{\cO(Z)} \cO(Z') \xra{\sim}  \fC(Z'), $$
		of categories over $\cO(Z')$.
		\item[(iii)] For every composable pair of morphisms $u_1$, $u_2$ of affine schemes over $\cX$,  an isomorphism $(u_1\circ u_2)^* \xra{\sim} u_2^* \circ u_1^*$, such that the natural compatibility axiom for $3$-fold compositions holds.
	\end{itemize}
	
\item[(2)]	For a quasi-coherent sheaf of categories $\fC\in \mr{LinCat}_{\cX}^{\mr{Ab}}$ we define its category of global sections to be the additive category given by the limit
	$$
	\fC(\cX):=\lim_{Z\ra \cX} \fC(Z) 
	$$
	over affine schemes mapping to $\cX$. We warn the reader that in general $\fC(\cX)$ does \textit{not} need to be an abelian category (since functors in the limit above are typically not exact). However, it is abelian if $\cX$ is a geometric stack (Remark \ref{rem:sheaves on categories for algebraic stacks}).  
	\end{itemize}
\end{df}	

\begin{example}\label{ex:examples of sheaves of categories}
	\begin{enumerate}[label=(\arabic*),leftmargin=*, rightmargin=0pt]
		\item
	Let $\cX$ be represented by an affine scheme $\Spec R$. Then the category of affine schemes over $\cX$ has a final object given by the identity map $\Spec R\ra \cX$, and so 
	$$
	\mr{LinCat}_{\cX}^{\mr{Ab}}\simeq \mr{LinCat}_{R}^{\mr{Ab}}.
	$$
	In other words, a quasi-coherent sheaf of abelian categories over $\Spec R$ is just an $R$-linear abelian category.
	\item\label{item:modules over a sheaf of algebras} Let $\cX\colon \Aff_{S}^{\op} \ra \mr{Spc}$ be a prestack. Define the \textit{naive} category of quasi-coherent sheaves on $\cX$ to be the limit 
	$$\QCoh^{\mr{naive}}(\cX)\coloneqq \lim_{Z\ra \cX} \mc O(Z)\lmod$$
	over affine schemes over $\cX$; this is an additive monoidal 1-category. Let $A\in \QCoh^{\mr{naive}}(\cX)$ be an algebra object; for any map $f\colon Z\ra \cX$ it defines an algebra $A(Z)$ over $\mc O(Z)$ and this association is functorial under pull-back. Then the association to $f\colon Z \ra \cX$ the abelian $\mc O(Z)$-linear category $A(Z)\lmod$ defines an object $A\lmod$ of $\lim_{Z \ra \cX } \mr{LinCat}_{\mc O(Z)}^{\mr{Ab}}$, or, in other words,  a quasi-coherent sheaf of abelian categories over $\cX$.
	\end{enumerate}
\end{example}

As a standard consequence of fpqc descent for $\mr{LinCat}_{-}^{\mr{Ab}}$ one can use \v Cech resolutions to compute  $\mr{LinCat}_{\cX}^{\mr{Ab}}$:

\begin{rem}\label{rem:Cech descent for sheaves of abelian categories} Let $f\colon \cX \ra \cY$ be an fpqc surjection of prestacks and let $\cX^{(\bullet)}$ denote the corresponding \v Cech nerve. Then by effective descent \cite[\S 6.2.3]{LurieHTT} the natural map $\colim_{[\bullet]\in \Delta^{\op}} \cX^{(\bullet)}\ra \cY$ becomes an equivalence after fpqc sheafification. Applying $\mr{LinCat}_{-}^{\mr{Ab}}$, we get an equivalence of categories
	$$
	\mr{LinCat}_{\cY}^{\mr{Ab}} \xra{\sim} \lim_{[\bullet]\in \Delta}  \mr{LinCat}_{\cX^{(\bullet)}}^{\mr{Ab}}. 
	$$
	As usual, in the limit one can restrict just to boundary maps in $\Delta$. Since all terms in the limit are $(2,1)$-categories, it can in fact be replaced by a similar limit over $\Delta_{\le 3}\subset \Delta$.
\end{rem}

Let us apply this to the situation when $\cX$ is an algebraic stack. 
\begin{rem}[Sheaves of categories on geometric stacks]\label{rem:sheaves on categories for algebraic stacks} Let $\cX\colon \Aff_{S}^{\op} \ra \mr{Grpd}$ be a classical\footnote{We consider it as a prestack via the nerve functor $N\colon \mr{Grpd} \ra \mr{Spc}$ whose essential image is given by 1-truncated spaces.}(=1-truncated) stack in the fpqc topology. 
We assume that $\cX$ is a \textit{geometric} stack in the conventions of \cite{Stefanich}, for example it can be a classical Artin stack. Then: \begin{enumerate}[leftmargin=*, rightmargin=0pt]
	\item By \cite[Theorem 4.7]{Pridham2013}, there is a small diagram $X_-\colon I^{\op}\to \Aff_{S}$ of affine $S$-schemes with \textit{flat} transition maps, such that $\mc X\simeq \colim_{i\in I^{\op}}X_i$, with $X_i=\Spec R_i$ in fpqc sheaves. One has an equivalence of (2,1)-categories 
\begin{equation}\label{eq:geometric stack as a colimit of affines}
\mr{LinCat}_{\cX}^{\mr{Ab}} \xra{\sim} \lim_{i\in I}\mr{LinCat}_{R_i}^{\mr{Ab}}. 
\end{equation}

In the case $\mc X$ is a quasi-compact Artin stack with affine diagonal, one can take a smooth cover $p\colon U=\Spec R \ra \cX$ by an affine scheme, and all terms $U_{(\bullet)}=\Spec R^{(\bullet)}$ of the corresponding \v Cech nerve are also affine. Then as $I$ one can take the subcategory $\hat \Delta$ of the simplex category $\Delta$ with only boundary maps as maps; namely, by effective descent (Remark \ref{rem:Cech descent for sheaves of abelian categories}) one has an equivalence:
$$
\colim_{\hat\Delta^{\op}}U_{(\bullet)}\xra{\sim}\mc X
$$ and, consequently, an equivalence
\begin{equation}\label{eq:descent for a stack}
\mr{LinCat}_{\cX}^{\mr{Ab}} \xra{\sim} \lim_{[\bullet]\in \hat\Delta^{\op}}  \mr{LinCat}_{R^{(\bullet)}}^{\mr{Ab}}. 
\end{equation}
Since all terms in the limit are $(2,1)$-categories, one can also replace $\hat\Delta^{\op}$ in the limit by $\hat\Delta_{\le 3}^{\op}$.
This way an object $\fC \in \mr{LinCat}_{\cX}^{\mr{Ab}}$ amounts to an explicit piece of data similar to the one in Construction \ref{constr:quasi-coherent sheaves of categories}: an $R$-linear abelian category $\cC$, together with an equivalence of categories $u\colon \cC \otimes_{R,\phi_1} R^{(1)} \simeq \cC \otimes_{R,\phi_2} R^{(1)}$ for the two different maps $\phi_1,\phi_2\colon R=R^{(0)}\ra R^{(1)}$, with a choice of an isomorphism of functors $\alpha\colon \delta_1^*(u)\xra{\sim} \delta^*_0(u)\circ \delta_2^*(u)$ for the pull-back to $R^{(2)}$, whose various pull-backs to $R^{(3)}$ are strictly compatible.

\item\label{item:global sections on stacks} Let $\fC\in \mr{LinCat}_{\cX}^{\mr{Ab}} $. Via \eqref{eq:geometric stack as a colimit of affines}, one has a formula for global sections
$$
\fC(\cX)\simeq \lim_{i\in I}\fC(X_i)
$$
where all functors in the limit are \textit{exact} (they are of the form $\mc C\to \mc C\otimes_{R_i}R_j$ where $\mc C\in \mr{LinCat}_{R_i}$ and $R_i\to R_j$ is flat). In particular, it follows that $\cC(\cX)$ is an abelian category\footnote{Abelian categories are closed under limits with respect to exact functors in $\mr{LinCat}_R^{\mr{Add}}$ \cite[Remark D.1.6.5]{LurieSAG}. }: kernels and cokernels can be computed termwise in the limit.

If $\mc X$ is an Artin stack with affine diagonal, with an atlas $U=\Spec R\to \mc X$, from \eqref{eq:descent for a stack} we also get a more compact formula:
$$
\fC(\cX) \simeq \lim_{[\bullet]\in \hat \Delta_{\le 3}}  \fC(U^{(\bullet)}).
$$

\end{enumerate}
\end{rem}	
In \cite{Stefanich} Stefanich proves the following strong 1-affineness statement for quasi-coherent sheaves of abelian categories (extending \cite[Theorem 10.6.5.1]{LurieSAG}). 

\begin{rem}[1-affineness for geometric stacks]\label{rem:1-affineness}
  Let $\mc X$ be a geometric stack (in the conventions of \cite{Stefanich}, e.g. an Artin stack). Let $\QCoh(\mc X)\coloneqq \QCoh^{\mathrm{naive}}(\mc X)$ denote the \textit{abelian} category of quasi-coherent sheaves on $\mc X$. Given $\mathfrak C\in \mr{LinCat}_{\cX}^{\mr{Ab}}$, and any point $g\colon \Spec R\to \mc X$, we have a natural monoidal pull-back functor $g^*\colon \QCoh(\mc X)\to R\lmod$, which endows $\mathfrak C(\Spec R)\in \mr{LinCat}_{R}^{\mr{Ab}}$ with a natural $\QCoh(\mc X)$-module structure. This way the abelian category $\mathfrak C(\mc X)=\lim_{\Spec R\to \mc X}\mathfrak C(\Spec R)$ can be endowed with the structure of a $\QCoh(\mc X)$-module. This construction is functorial in $\mathfrak C$ and extends to a functor 
  $$
\Gamma^{\mr{enh}}\colon \mr{LinCat}_{\cX}^{\mr{Ab}} \to \Mod_{\QCoh(\mc X)}(\mr{Groth}), \quad \mathfrak C\mapsto \mathfrak C(\mc X)
  $$
  where $\mr{Groth}:=\mr{LinCat}_{\mbb Z}^{\mr{Ab}}$ is the category of all Grothendieck abelian categories (see \cite[Section 5.4]{Stefanich}). By \cite[Corollary 5.4.4]{Stefanich}, if $\mc X$ is quasi-compact and has affine diagonal, the functor $\Gamma^{\mr{enh}}$ is an equivalence. So, any $\mathfrak C\in \mr{LinCat}_{\cX}^{\mr{Ab}}$ is in fact completely reconstructed from the category of its global sections $\mathfrak C(\mc X)$ together with the induced $\QCoh(\mc X)$-module structure.

  Moreover, given a map $f\colon \mc X \to \mc Y$ of quasi-compact geometric stacks with affine diagonal, then the pull-back functor $f^*\colon \mr{LinCat}_{\cY}^{\mr{Ab}}\to \mr{LinCat}_{\cX}^{\mr{Ab}}$ is described as the base change 
  $$
-\otimes_{\QCoh(\mc Y)}\QCoh(\mc X)\colon \Mod_{\QCoh(\mc Y)}(\mr{Groth}) \to \Mod_{\QCoh(\mc X)}(\mr{Groth}).
  $$
  This functor has a right adjoint $f_*\colon \mr{LinCat}_{\cX}^{\mr{Ab}}\to \mr{LinCat}_{\cY}^{\mr{Ab}}$ given by the forgetful functor, which satisfies base change with respect to maps of geometric stacks with affine diagonal. 

\end{rem}

\begin{rem}[Pull-back and push-forward on the global sections for affine maps]\label{rem:pull-back and pushforward for global sections}
    Assume $\mc X\to \mc Y$ is a map of geometric stacks that is representable in affine schemes: namely, for any $\Spec R\to \mc Y$ the fiber product $\mc X_R=\Spec R\times_{\mc Y}\mc X$ is again an affine scheme. In this case one has an equivalence $\QCoh(\mc X)\simeq \Mod_{f_*\mc O_{\mc X}}(\QCoh(\mc Y))$. Let $\mathfrak C$ be a quasi-coherent sheaf of abelian categories over $\mc Y$. Then the natural functor
    $$
    \fC(\mc Y)\otimes_{\QCoh(\mc Y)}\QCoh(\mc X) \to \fC(\mc X):=(f^*\fC)(\mc X)
    $$
    is an equivalence, and identifies the target with $\Mod_{f_*\mc O_{\mc X}}(\fC(\mc Y))$ (here we use the $\QCoh(\mc Y)$-linear structure on $\fC(\mc Y)$ from Remark \ref{rem:1-affineness}). Indeed, by \cite[Theorem 4.8.4.6]{LurieHA} for any presentable $\QCoh(\mc Y)$-module category $\mc C$ one has $\mc C\otimes_{\QCoh(\mc Y)}\QCoh(\mc X) \simeq \Mod_{f_*\mc O_X}(\mc C)$, and the association $\mc C\mapsto \mc C\otimes_{\QCoh(\mc Y)}\QCoh(\mc X)$ commutes with limits in $\mc C$. Realizing $\mc Y$ as $\colim_{i\in I}Y_i$ with $Y_i$ affine as in Remark \ref{rem:sheaves on categories for algebraic stacks}, we reduce to the case when both $\mc X=\Spec A$ and $\mc Y=\Spec B$ are affine, where we get the natural isomorphism 
    $$
\fC(\Spec B)\otimes_{B\lmod}A\lmod\simeq \fC(\Spec A). 
    $$

We have the natural pull-back functor: $f^*_{\fC}\colon \fC(\mc Y)\to \fC(\mc X)$ which via the above is identified with the free $f_*\mc O_{\mc X}$-module functor $-\otimes_{\mc O_{\mc Y}}f_*\mc O_{\mc X}\colon \fC(\mc Y)\to \Mod_{f_*\mc O_X}(\fC(Y))$. This functor has a right adjoint $f_{\fC,*}\colon \fC(\mc X)\to \fC(\mc Y)$ given by the forgetful functor $\Mod_{f_*\mc O_{\mc X}}(\fC(\mc Y)) \to \fC(\mc Y).$
    Note that if $f=i$ is a closed embedding, then $i^*_{\mc C}i_{\mc C*}$ is the identity functor since $i_*\mc O_{\mc X}\otimes_{\mc O_{\mc Y}}i_*\mc O_{\mc X} \simeq i_*\mc O_{\mc X}$. Consequently, the functor
    $$
i_{\fC,*}\colon \mathfrak C(\cX):=(i^*\mathfrak C)(\cX) \to \mathfrak C(\mc Y)
    $$
    is fully faithful. 
\end{rem}

We now discuss the notions of tensor product and duals in the setting of quasi-coherent sheaves of abelian categories.
\begin{constr}[Tensor product of sheaves of categories]\label{constr: tensor product of sheaves of categories}
By \cite[Variant D.2.3.2]{LurieSAG}, the category
$\mr{LinCat}_R^{\mr{Add}}$ of ordinary presentable
$R$-linear additive categories carries a natural symmetric
monoidal structure given by \textit{Lurie's tensor product}.
Given $\mc{C}, \mc{D}\in \mr{LinCat}_{R}^{\mr{Ab}}$ their tensor product $\mc{C}\otimes_R \mc{D}$ is defined by the universal property 
$$
\Fun^{L}_R(\mc{C}\otimes_R \mc{D},\mc E) \simeq \Fun^{L, L}_{R,R}(\mc{C}\times \mc{D},\mc E)
$$
for any $\mc E\in \mr{LinCat}_{R}^{\mr{Ab}}$, and where $\Fun^{L}_R$ (resp. $\Fun^{L, L}_{R,R}$) denotes the space of continuous $R$-linear functors (resp. functors that are continuous and $R$-linear in each variable). The monoidal unit for $-\otimes_R-$ is given by $R\lmod$. By \cite[Proposition D.2.2.1(d)]{LurieSAG}, the subcategory $\mr{LinCat}_R^{\mr{Ab}}\subset \mr{LinCat}_R^{\mr{Add}}$ of Grothendieck abelian $R$-linear categories is closed under Lurie's tensor product, and thus also acquires a symmetric monoidal structure. In the case $\mc C_1=A\lmod$ and $\mc C_2 =B\lmod$, for $R$-algebras $A$ and $B$, the tensor product $\mc{C}_1\otimes_R \mc{C}_2$ is naturally equivalent to $A\otimes_RB\lmod$.

Symmetric monoidal structure on $\mr{LinCat}_{R}^{\mr{Ab}}$ is functorial in $R$ and, given a prestack $\cX\colon \Aff_{S}^{\op} \ra \mr{Spc}$, induces a symmetric monoidal structure on $\mr{LinCat}_{\mc X}^{\mr{Ab}}:= \lim_{\Spec R \ra \cX } \mr{LinCat}_{R}^{\mr{Ab}}$. Given $\mathfrak C, \mathfrak D\in \mr{LinCat}_{\mc X}^{\mr{Ab}}$ we will denote the corresponding tensor product by $\mathfrak C\otimes_{\mc O_{\cX}} \mathfrak D$ or, when $\cX$ is clear from the context, simply $\mathfrak C\otimes\mathfrak D$. If $\mathfrak C$ and $\mathfrak D$ are given by diagrams $\{\mathfrak C_f\}$ and $\{\mathfrak D_f\}$ indexed by $f\colon \Spec R \ra X$ then $\mathfrak C\otimes_{\mc O_X} \mathfrak D$ is explicitly given by the diagram $\{\mathfrak C_f\otimes_R\mathfrak D_f\}$.
\end{constr}

\begin{constr}[Dualizable sheaves of abelian categories] \label{constr:dualizable sheaves of categories}
    Note that $R\lmod$ is the unit for the symmetric monoidal structure on $\mr{LinCat}_{R}^{\mr{Ab}}$. For $\mc C\in \mr{LinCat}_{R}^{\mr{Ab}}$ denote by $\mc C^\vee:=\Fun^L_R(\mc C, R\lmod)$ the category of $R$-linear continuous functors from $\mc C$ to $R\lmod$. This is an additive $R$-linear category, but in general  \textit{not} a Grothendieck abelian category. We call $\mc C$ \textit{dualizable} if the natural functor
    $$
\mc C^\vee \otimes_R \mc C \ra \Fun^L_R(\mc C,\mc C)
    $$
    is an equivalence. For a dualizable $\mc C\in \mr{LinCat}_{R}^{\mr{Ab}}$ it is known that $\mc C^\vee$ is again Grothendieck abelian\footnote{More generally, $\mc C$ is dualizable if and only if products are exact and Grothendieck's condition (Ab6) holds (see \cite{kanda}).} (see \cite{kanda}). For an associative $R$-algebra $A$ the category $\mc C=A\lmod$ of left modules is dualizable, and the dual is given by $\mc C^\vee=A^{\op}\lmod$. It is formal that for a dualizable $\mc C$ one has $\mc C\iso (\mc C^\vee)^\vee$ and the map $\mc C^\vee \otimes_R \mc D \ra \Fun^L_R(\mc C,\mc D)$ is an equivalence for any $\mc D \in \mr{LinCat}_{R}^{\mr{Ab}}$ (in particular, the target is again in $\mr{LinCat}_{R}^{\mr{Ab}}$). Let us denote by $\mr{DualCat}_R^{\mr{Ab}}\subset \mr{LinCat}_{R}^{\mr{Ab}}$ the full subcategory spanned by dualizable objects.

We can also right Kan extend the functor $R\mapsto \mr{DualCat}_R^{\mr{Ab}}$ from affine schemes to prestacks; as usual, for a given prestack $\cX\colon \Aff_{S}^{\op} \ra \mr{Spc}$ we have \begin{equation}\label{eq:dualizable objects as limit of dualizables}
    \mr{DualCat}_\cX^{\mr{Ab}}\simeq \lim_{\Spec R \ra \cX } \mr{DualCat}_{R}^{\mr{Ab}}.
\end{equation} By \cite[Proposition 4.6.1.11]{LurieHA}, $\mr{DualCat}_\cX^{\mr{Ab}}\subset \mr{LinCat}_{\mc X}^{\mr{Ab}}$ is exactly the full subcategory spanned by dualizable objects. Given $\mathfrak C \in \mr{DualCat}_\cX^{\mr{Ab}}$ we have the dual sheaf of categories $\mathfrak C^\vee\in \mr{DualCat}_\cX^{\mr{Ab}}$ and for any $\mathfrak D\in \mr{LinCat}_{\mc X}^{\mr{Ab}}$ we have an equivalence $\mathfrak C^\vee\otimes \mathfrak D\simeq \ul{\Fun}^L(\mathfrak C,\mathfrak D)$ where the target is defined as the quasicoherent sheaf of categories sending $f\colon \Spec R \ra \cX$ to $\Fun^L_R(\mathfrak C_f,\mathfrak D_f)$ (note that this indeed defines a quasicoherent sheaf under the assumption that $\mathfrak C$ is dualizable).
\end{constr}

\begin{example}\label{ex:dualizable categories}\begin{enumerate}
    \item Let $\cX$ be a prestack and $A\in \QCoh^{\mr{naive}}(\cX)$ be an algebra object. Then the quasi-coherent sheaf of categories $$A\lmod=\{A(R)\lmod\}_{\Spec R \ra \cX }\in \lim_{\Spec R \ra \cX }\mr{LinCat}_R^{\mr{Ab}}=\mr{LinCat}_{\mc X}^{\mr{Ab}}$$ from Example \ref{ex:examples of sheaves of categories} is dualizable. Its dual is given by $A^{\mr{op}}\lmod$.
    \item \label{item:locally given by modules is dualizable} Let $\mc X$ be a geometric stack, and assume that for any map $q\colon \Spec R\to \mc X$ the pull-back $q^*\mathfrak C\in\mr{LinCat}_{R}^{\mr{Ab}}$ is equivalent to $A\lmod$ for some $R$-algebra $A$ flat locally on $R$. Then, refining the diagram $I$ \eqref{eq:geometric stack as a colimit of affines} realizing $\mc X$ as a colimit of affines along flat maps, so that pull-back of $\fC$ to each term is given by $A\lmod$ for some $A$, we get that $\mathfrak C$  is dualizable (by \eqref{eq:descent for a stack}).
    \item Let $\mc X,\mc Y$ be two prestacks, let $\mathfrak C \in \mr{LinCat}_{\mc X}^{\mr{Ab}}$ and $\mathfrak D \in \mr{LinCat}_{\mc Y}^{\mr{Ab}}$. Then we can consider the exterior tensor product 
    $$\mathfrak C \boxtimes \mathfrak D:= \mr{pr}_{\mc X}^*\mathfrak C\otimes_{\mc O_{\mc X\times \mc Y}}\mr{pr}_{\mc Y}^*\mathfrak D,$$
    which is a quasi-coherent sheaf of categories over $\mc X\times \mc Y$. By construction its sections on $(f_1,f_2)\colon \Spec R\ra \mc X\times \mc Y$ are given by the tensor product $\mathfrak C_{f_1}\otimes_R \mathfrak D_{f_2}$. 
\end{enumerate}
 
\end{example}

We end the section with the following remark which addresses quasi-coherent sheaves of categories on formal stacks.
  \begin{rem}[Sheaves of  categories on formal prestacks]\label{cat.over.f.s} Let $S:=\Spec \mbb Z[h]$. By an $h$-adic formal prestack we will mean a functor $\widehat T\colon \Aff_{\mbb Z[h]}^{\mr{\op}}\ra \mr{Spc}$ such that $\widehat T(Z)=\varnothing$ unless $h$ is nilpotent on $\mc O(Z)$. A datum of such $\widehat T$ is given by a collection of  prestacks $T_m:=\widehat{T}\times_S S_m$ over $S_m= \Spec (\bZ[h]/h^{m+1})$, $m\geq 0$, together with isomorphisms $T_{m+1}\times_{S_{m+1}} S_m \iso T_m$. The prestack $\widehat T$ is reconstructed from $\{T_m\}_{m\in \mbb N}$ as $\colim_m\ \! T_m$ in prestacks over $\Spec \mbb Z[h]$. We get an equivalence of categories
  	$$
  	\mr{LinCat}_{\widehat T}^{\mr{Ab}} \xra{\sim } \lim_m \mr{LinCat}_{T_m}^{\mr{Ab}}.
  	$$
  	In particular, a quasi-coherent  sheaf of  categories over  $\widehat T$ is given by a collection $\{\fC_m, i^*_m \fC_{m+1} \iso \fC_{m}, m\geq 0\}$,
	where  $\fC_m$  is a quasi-coherent  sheaf of  categories over 
	$T_m$ and $i^*_m \fC_{m+1} $
	is the pullback of $\fC_{m+1}$ to $T_m$. 
\end{rem}

\subsection{Digression: quasicoherent sheaves of stable categories}\label{ssec:qcoherent sheaves of stable categories} The material in this section is quite technical and will only be used later to define an action by restricted Lagrangians, so can be freely skipped on the first reading. 

Along with Definition \ref{def:sheaf of abelian categories}, there is also another notion of a quasicoherent sheaf of categories, which is more flexible and robust in many aspects, and is perhaps more broadly used than the Construction \ref{constr:quasi-coherent sheaves of categories}, especially in the context of derived algebraic geometry where in fact it is the only reasonable notion.

Let $\Pr_{\mr{st}}^{\mr L}$ denote the $\infty$-category of presentable stable categories; it has a natural monoidal structure given by Lurie's tensor product. One has a functor $\CAlg_{\mbb Z} \ra \Pr_{\mr{st}}^{\mr L}$ sending $R$ to the corresponding unbounded derived category $D(R)$. Each $D(R)$ is a symmetric monoidal category, and the functor $D(-)$ naturally lifts to the category of commutative algebras in $\Pr_{\mr{st}}^{\mr L}$. We let $\mr{LinCat}^{\mr{st}}_{R}$ denote the category $D(R)$-modules in $\Pr_{\mr{st}}^{\mr L}$.
\begin{constr}\label{constr:sheaves of stable cats}Consider the functor 
 $$\mr{LinCat}^{\mr{st}}_{-}\colon \CAlg_{\mbb Z} \ra \mr{Cat}_\infty$$   
 sending $R\mapsto \mr{LinCat}^{\mr{st}}_{R}$, and a morphism $R\to R'$ to the Lurie tensor product $-\otimes_{D(R)}D(R')\colon \mr{LinCat}^{\mr{st}}_{R} \to \mr{LinCat}^{\mr{st}}_{R'}$.

Let $S$ be a base affine scheme and, as in Construction \ref{constr:quasi-coherent sheaves of categories}, consider the category $\mr{PStk}_{/S}$ of prestacks over $S$. Given a prestack $\mc X\colon \mr{Aff}_S^{\op}\to \mr{Spc}$ one defines the $\infty$-category $\mr{LinCat}_{\mc X}^{\mr{st}}$ of \textit{quasicoherent sheaves of stable categories} as the value of right Kan extension of $D(-)\lmod$ along the embedding $\mr{Aff}_S^{\op}\hookrightarrow \mr{PStk}_{/S}^{\op}$. Explicitly, one has
$$
\mr{LinCat}^{\mr{st}}_{\mc X}\simeq \lim_{\Spec R\to \mc X} \mr{LinCat}^{\mr{st}}_{R},
$$
where the limit is taken over the slice category $(\mr{Aff}_S/\mc X)^{\op}$.

For a map of prestacks $f\colon \mc X\to \mc Y$ we have a natural pull-back functor $\mr{LinCat}^{\mr{st}}_{\mc Y}\to \mr{LinCat}^{\mr{st}}_{\mc X}$. Given $\mathscr C\in \mr{LinCat}^{\mr{st}}_{\mc Y}$ we have $\mc O(S)$-linear category of global sections $\mathscr C(\mc Y)=\lim_{\Spec R\to \mc Y} \mathscr C_R \in \mr{LinCat}^{\mr{st}}_{\mc O(S)}$. For $f\colon \mc X\to \mc Y$, there is a natural functor
$$
f^*_{\mathscr{C}}\colon \mathscr C(\mc Y) \to (f^*\mathscr C)(\mc X),
$$
induced by the functor $(\mr{Aff}_S/\mc X)^{\op}\to (\mr{Aff}_S/\mc Y)^{\op}$. Since the functors in the limit $\mathscr C(\mc Y)=\lim_{\Spec R\to \mc Y} \mathscr C_R$ are all colimit commuting, colimits in $\mathscr C(\mc Y)$ are computed termwise in the limit. 
Consequently, the functor $f^*_{\mathscr{C}}$ also commutes with colimits, and so has a right adjoint $f_{*,\mathscr C}$ by the adjoint functor theorem.

\end{constr}

 Similarly to Constructions \ref{constr: tensor product of sheaves of categories} and \ref{constr:dualizable sheaves of categories}, one can define a natural (Lurie's) tensor product $-\otimes_{\mc O_{\mc X}}-$ on $\mr{LinCat}_{\mc X}^{\mr{st}}$, as well as the corresponding notion of a dualizable quasi-coherent sheaf of stable categories. Monoidal unit is given by the sheaf $D(\QCoh)_{\mc X}\in \mr{LinCat}_{\mc X}^{\mr{st}}$ whose sections on $\Spec R\to \mc X$ are given by $D(R)\in \mr{LinCat}_{R}^{\mr{st}}$. Let us denote by $\mr{DualCat}_{\mc X}^{\mr{st}}\subset \mr{LinCat}_{\mc X}^{\mr{st}}$ the full subcategory spanned by dualizable sheaves.

 \begin{constr}[Fourier-Mukai transform I]\label{constr:FM-transform 1}
     
     We will use the following construction later. Let $\mc X\ra \mc Z \leftarrow \mc Y$ be prestacks, and let $\mathscr C\in \mr{DualCat}_{\mc X}^{\mr{st}}$ and $\mathscr D\in \mr{LinCat}_{\mc Y}^{\mr{st}}$. Let $\mathscr C^\vee\in \mr{DualCat}_{\mc X}^{\mr{st}}$ be the dual of $\mathscr C$. Consider the fiber product $\mc W:=\mc X\times_{\mc Z} \mc Y$ and define
     $$
\mathscr C^\vee\boxtimes \mathscr D:= p_{\mc X}^*\mathscr C^\vee \otimes_{\mc O_{\mc W}} p_{\mc Y}^*\mathscr D \in \mr{LinCat}_{\mc W}^{\mr{st}},
     $$
     where $p_{\mc X}$ and $p_{\mc Y}$ are projections. 
Then, given an object $\mc F\in (\mathscr C^\vee\boxtimes \mathscr D)(\mc W)$ there is a natural Fourier-Mukai transform functor 
$$
\mr{FM}_{\mc F}\colon \mathscr C(\mc X) \to \mc D(\mc Y),
$$
defined as the composition of the following functors: pull-back $p_{\mc X,\mathscr C}^*\colon \mathscr C(\mc X) \to (p_{\mc X}^*\mathscr C)(\mc W)$, functor $-\otimes \mc F\colon (p_{\mc X}^*\mathscr C)(\mc W)\to (p_{\mc X}^*\mathscr C\otimes p_{\mc X}^*\mathscr C^\vee \otimes p_{\mc Y}^*\mathscr D)(\mc W)$ induced by tensoring with $\mc F$, evaluation $\ev_{p_X^*\mathscr C}\otimes \id_{\mathscr D} \colon (p_{\mc X}^*\mathscr C\otimes p_{\mc X}^*\mathscr C^\vee \otimes p_{\mc Y}^*\mathscr D)(\mc W) \to (p_{\mc Y}^*\mathscr D)(\mc W)$, and, finally the functor $p_{\mc Y,\mathscr D,*}\colon (p_{\mc Y}^*\mathscr D)(\mc W) \to \mathscr D(\mc Y)$. This construction is functorial in $\mc F$ and produces a functor 
$$
\mathrm{FM}_{-}\colon (\mathscr C^\vee\boxtimes \mathscr D)(\mc W) \to \Fun(\mathscr C(\mc X),\mc D(\mc Y)).
$$
\end{constr}

We now relate quasicoherent sheaves of abelian and stable categories. 
 \begin{constr}\label{constr:derived category}
  Let $\mr{LinCat}_\mc X^{\mr{Ab,ex}}\subset \mr{LinCat}_\mc X^{\mr{Ab}}$ denote the wide subcategory with morphisms given by only those transformations $\mathfrak C\to \mathfrak D$, whose pull-back along any $\Spec R\to \mc X$, given by a functor $\mathfrak C_R\to \mathfrak D_R$ in $\mr{LinCat}_R^{\mr{Ab}}$, is assumed to be not only continuous but also left exact.
 
   Let $\mc X$ be a geometric stack. Then we claim that there is a natural functor
\begin{equation}\label{eq:completed derived category sheaf}
\widehat{D}(-)\colon \mr{LinCat}_{\mc X}^{\mr{Ab,ex}} \to \mr{LinCat}^{\mr{st}}_{\mc X},
\end{equation}
    whose restriction along any flat $\Spec R\to \mc X$ is given by the functor $\mathfrak C_R\mapsto \widehat{D}(\mathfrak C_R)$ from $\mr{LinCat}_{R}^{\mr{Ab,ex}} \to \mr{LinCat}_{R}^{\mr{st}}$. Here $\widehat{D}(\mc C)$ denotes the left $t$-completion of the unbounded derived category of $\mc C$.

 Let us sketch the construction of $\widehat{D}(-)$ following \cite{LurieSAG}. Let $\operatorname{LinCat}_{R}^{\mr{pst}}$ denote the category of Grothendieck prestable $R$-linear $\infty$-categories, and denote by $\mr{LinCat}_{\mc X}^{\mr{pst}}$ the value on $\mc X$ of right Kan extension of the functor $\mr{LinCat}_{-}^{\mr{pst}}\colon R\mapsto \operatorname{LinCat}_{R}^{\mr{pst}}$. By \cite[Construction 10.1.2.8]{LurieSAG} there is a natural transformation $(-)^{\heartsuit}\colon \mr{LinCat}_{-}^{\mr{pst}} \to \mr{LinCat}_{-}^{\mr{Ab}}$ associating to a prestable category $\mc D$ its abelian heart $\mc D^\heartsuit$. Let $\CAlg_{\mbb Z}^{\mr{flat}}\subset \CAlg_{\mbb Z}$ denote the wide subcategory with morphisms given by flat ring homomorphisms. Let $\operatorname{LinCat}_{R}^{(0), \mr{comp},\mr{pst}, \mr{lex}}\subset \mr{LinCat}_{R}^{\mr{pst}}$ denote the 1-faithful subcategory spanned by \textit{weakly 0-complicial and complete} Grothendieck prestable $\infty$-categories and left exact (cocontinuous $R$-linear) functors. By \cite[Propositions D.5.1.3(a) and D.5.7.1(b)]{LurieSAG}, $R\mapsto \operatorname{LinCat}_{R}^{(0), \mr{comp},\mr{pst},\mr{lex}}$ defines a subfunctor of $R\mapsto \mr{LinCat}_{R}^{\mr{pst}}$ restricted to $\CAlg_{\mbb Z}^{\mr{flat}}$. Moreover, as in the proof of \cite[Corollary 10.6.6.6]{LurieSAG}, the functor $(-)^\heartsuit$ restricts to a transformation 
 $$
(-)^\heartsuit\colon \operatorname{LinCat}_{-}^{(0), \mr{comp},\mr{pst}, \mr{lex}} \to \mr{LinCat}_{-}^{\mr{Ab,ex}}
 $$
 of functors on $\CAlg_{\mbb Z}^{\mr{flat}}$. Moreover, by  \cite[Corollary 10.6.6.4]{LurieSAG} (applied to affine schemes) this transformation is in fact an equivalence, with the inverse given termwise by the left t-completed connective derived category functor $\widehat D(-)_{\ge 0}$. This produces a transformation 
 $$
\widehat D(-)_{\ge 0}\colon \mr{LinCat}_{-}^{\mr{Ab,ex}} \to \operatorname{LinCat}_{-}^{(0), \mr{comp},\mr{pst},\mr{lex}}\to \operatorname{LinCat}_{-}^{\mr{comp},\mr{pst}}
 $$
 of functors $\CAlg_{\mbb Z}^{\mr{flat}}\to \widehat{\operatorname{Cat}}_{\infty}$. Here the final target is a functor associating to $R$ the category $\operatorname{LinCat}_{R}^{\mr{comp},\mr{pst}}$ of all Grothendieck prestable complete $\infty$-categories. Realizing a geometric stack $\mc X$ as a colimit $\colim_{i\in I}X_i$ of affine schemes with flat transition maps, as in Remark \ref{rem:sheaves on categories for algebraic stacks}, and using flat descent for both sides (see \cite[Corollary D.6.8.4 and Theorem D.6.8.1]{LurieSAG}, produces a functor 
 $$
\widehat D(-)_{\ge 0}\colon \mr{LinCat}_{\mc X}^{\mr{Ab,ex}} \to \operatorname{LinCat}_{\mc X}^{\mr{comp},\mr{pst}}.
 $$
 One can compose it further with the embedding $\operatorname{LinCat}_{\mc X}^{\mr{comp},\mr{pst}}\to \operatorname{LinCat}_{\mc X}^{\mr{pst}}$ (see \cite[Notation 10.4.1.2.]{LurieSAG}) ans then the functor \cite[Construction 10.1.2.7]{LurieSAG} $\mr{Sp}(-)\colon \operatorname{LinCat}_{\mc X}^{\mr{pst}}\to \operatorname{LinCat}_{\mc X}^{\mr{st}}$ given by stabilization, to obtain the desired transformation \eqref{eq:completed derived category sheaf}.

 Note that for any algebra $A$ one has $\widehat{D}(A\lmod)\simeq D(A)$. If $f\colon \mc X \to \mc Y$ is a flat map of geometric stacks and $\fC\in \mr{LinCat}_{\mc Y}^{\mr{Ab}}$ then there is a natural equivalence $f^*\widehat{D}(\fC)\simeq \widehat{D}(f^*\fC)$.
 \end{constr}
\begin{rem}
    Let $\mathfrak C\in \mr{LinCat}_{\mc X}^{\mr{Ab,ex}}$. Then there is a natural functor $\iota_{\mathfrak C}\colon \mathfrak C(\mc X) \to (\widehat{D}(\mathfrak C))(\mc X)$ between the corresponding global sections, given as follows. Since $\mc X$ is geometric, one has natural equivalences 
    $$
    \mr{LinCat}^{\mr{Ab},\mr{ex}}_{\mc X}\simeq \lim_{\text{flat }\!\Spec R\to \mc X} \mr{LinCat}^{\mr{Ab},\mr{ex}}_{R} \qquad \text{ and } \qquad \mr{LinCat}_{\mc X}^{\mr{st}}\simeq \lim_{\text{flat }\!\Spec R\to \mc X} \mr{LinCat}_{R}^{\mr{st}}
    $$
(where the indexing diagram is given by flat maps $\Spec R\to \mc X$ with maps between affine schemes also assumed to be flat), and the functor $\iota_{\mathfrak C}$ is induced by the usual functors $-[0]\colon \mathfrak C_R \to \widehat{D}(\mathfrak C_R)$, sending $A\in \mathfrak C_R\mapsto A[0]\in \widehat{D}(\mathfrak C_R)$ for any flat $\Spec R\to \mc X$. Assume additionally that $X$ is geometric in the sense of
\cite[Definition 9.3.0.1]{LurieSAG}; for example, $\mc X$ may be
a quasi-compact classical Artin stack with affine diagonal. Then $\iota_{\mathfrak C}$ extends to an equivalence of categories 
\begin{equation}\label{eq:iso between d of global sections}
\iota_{\mathfrak C}\colon \widehat{D}(\mathfrak C(\mc X)) \xra{\sim} (\widehat{D}(\mathfrak C))(\mc X).
\end{equation}
Indeed, in the terminology of \cite{LurieSAG} $\widehat{D}(\mathfrak C)_{\ge 0}$ is a complete, weakly 0-complicial prestable quasi-coherent stack by \cite[Corollary 10.6.6.6]{LurieSAG}, and by \cite[Proposition 10.4.6.6]{LurieSAG} its global sections $(\widehat{D}(\mathfrak C)_{\ge 0})(\mc X)$ are also weakly 0-complicial. Its heart is identified with $\mathfrak C(X)$ by \cite[Corollary 10.6.6.6]{LurieSAG}, and so $(\widehat{D}(\mathfrak C)_{\ge 0})(\mc X)\simeq \widehat{D}(\mathfrak C(X))_{\ge 0}$ by \cite[Corollary C.5.9.7]{LurieSAG}. One then gets \eqref{eq:iso between d of global sections} by stabilizing both sides.
\end{rem}

\begin{lm}
        Let $\mc X$ be a geometric stack. Let $\mathfrak C\in \mr{DualCat}_{\mc X}^{\mr{Ab}}$. Then $\widehat{D}(\mathfrak C)\in \mr{DualCat}_{\mc X}^{\mr{st}}$.
\end{lm}
\begin{proof} Realizing $\mc X$ as a colimit of a diagram of affine schemes along flat maps, one reduces to showing that if $\mc C\in \mr{DualCat}_{R}^{\mr{Ab}}$ then $\widehat{D}(\mc C)\in \mr{DualCat}_{R}^{\mr{st}}$ for a commutative ring $R$. By \cite[Theorem 3.1.17]{Stefanich2} $\mc C$ is 1-strongly compactly assembled. Then by \cite[Proposition 3.1.22]{Stefanich2} it follows that $D_{\ge 0}(\mc C)$ is a strongly compactly assembled $\infty$-category, and so $D(\mc C)=\mr{Sp}(D_{\ge 0}(\mc C))$ is compactly assembled by \cite[Remark 3.1.23]{Stefanich2}. By \cite[Theorem D.7.0.7]{LurieSAG} it follows that $D(\mc C)\in \mr{DualCat}_{R}^{\mr{st}}$. On the other hand, by \cite[Proposition 3.1.15]{Stefanich2} one also has that the $t$-structure on $D(\mc C)$ is separated and products are $t$-exact. Thus, by \cite[Proposition 1.2.1.19]{LurieHA} $D(\mc C)\simeq \widehat{D}(\mc C)$, and so $\widehat{D}(\mc C)\in \mr{DualCat}_{R}^{\mr{st}}$, as desired.
\end{proof}
Using the classification result\footnote{Namely, any object of $\mr{DualCat}_{R}^{\mr{Ab}}$ is in fact the Serre quotient $A\lmod/(A/I)\lmod$ for some $R$-algebra $A$ and an idempotent ideal $I\subset A$.} of Kanda \cite{kanda} one can in fact canonically identify the dual of  ${D}(\mc C)$ with $D(\mc C^{\vee})$ (and thus $\widehat {D}(\mathfrak C)^\vee$ with $\widehat D(\mathfrak C^{\vee})$). Unfortunately, we do not know a direct reference for the proof of this fact, though it is known to experts. Instead, for simplicity, we will make an assumption on the dualizable sheaf of abelian categories, which in principle should not be necessary. 
\begin{rem}
    Let $A,B$ be classical $R$-algebras. Then there is
a natural $R$-linear additive functor
    \begin{equation}\label{eq:classical morita vs derived}
\Fun^L_R(A\lmod, B\lmod) \to \Fun^L_R(D(A), {D}(B)).
  \end{equation}    
  Here the source is an ordinary additive category,
while the target is a stable $\infty$-category.
    Indeed, via evaluation on $A\in A\lmod \subset D(A)$, left-hand side is given by the category of $B\otimes_R A^{\op}$-modules in $R\lmod$, while the right-hand side is the category of $B\otimes_R^{\mathbb L} A^{\op}$-modules in $D(R)$. There is a natural projection $B\otimes_R^{\mathbb L} A^{\op}\twoheadrightarrow B\otimes_R A^{\op}$ which together with the functor $R\lmod\to D(R)$ produces the natural functor $\Mod_{B\otimes_R A^{\op}}(R\lmod) \to \Mod_{B\otimes_R^{\mathbb L} A^{\op}}(D(R))$, giving \eqref{eq:classical morita vs derived}. This functor is (contravariantly) functorial in the source $A\lmod$ (resp. covariantly in the target $B\lmod$) with respect to equivalences $A\lmod\xra{\sim} A'\lmod$ (resp. $B\lmod\xra{\sim} B'\lmod$). Moreover, this construction is functorial with respect to flat maps $R\to R'$ (since in this case tensor product $-\otimes_R R'$ agrees with the derived one).

    Now, let $\mc X$ be a geometric stack. Assume we have two quasicoherent sheaves $\mathfrak C,\mathfrak D\in \mr{LinCat}_{\mc X}^{\mr{Ab}}$, that are locally given by the category of modules over an algebra, as in Example \ref{ex:dualizable categories}\eqref{item:locally given by modules is dualizable}. Realizing $\mc X$ as a colimit $\colim_{i\in I}\Spec R_i$ of affine schemes with flat transition maps, one can refine $I$, so that pull-backs of $\mathfrak C$ and $\mathfrak D$ to any term in the diagram are given by categories of modules over an algebra. Applying the functorial transformation \eqref{eq:classical morita vs derived} termwise in the limits 
    $$
    (\fC^\vee\otimes\mathfrak D)(\cX) \xra{\sim } \lim_{i\in I}   (\fC^\vee\otimes\mathfrak D)(X_i) \qquad \text{and} \qquad (\widehat{D}(\fC)^\vee\otimes\widehat{D}(\mathfrak D))(\cX) \xra{\sim } \lim_{i\in I}   (\widehat{D}(\fC)^\vee\otimes\widehat{D}(\mathfrak D))(X_i)
    $$
    gives a natural additive functor
    \begin{equation}\label{eq:map from abelain functors to derived}
(\mathfrak C^\vee \otimes\mathfrak D)(\mc X) \to (\widehat D(\mathfrak C)^\vee \otimes\widehat{D}(\mathfrak D))(\mc X),
     \end{equation}
  between the corresponding categories of global sections.  
\end{rem}
\begin{constr}[Fourier-Mukai transform II]\label{constr:FM-transform ii} Let $\mc X\ra \mc Z \leftarrow \mc Y$ be a diagram of geometric stacks where maps are flat. Consider the fiber product $\mc W:=\mc X\times_{\mc Z} \mc Y$. Let $\mathfrak C\in \mr{DualCat}_{\mc X}^{\mr{Ab}}$  and $\mathfrak D\in \mr{LinCat}_{\mc Y}^{\mr{Ab}} $, and assume that they are locally given by the category of modules over an algebra. Consider the exterior tensor product 
$\mathfrak C^\vee\boxtimes \mathfrak D:= p_{\mc X}^*\mathfrak C^\vee \otimes_{\mc O_{\mc W}} p_{\mc Y}^*\mathfrak D \in \mr{LinCat}_{\mc W}^{\mr{Ab}}.$ Applying \eqref{eq:map from abelain functors to derived} to $p_{\mc X}^*\mathfrak C$ and $p_{\mc Y}^*\mathfrak D$ we get a natural functor
$$
(\mathfrak C^\vee\boxtimes \mathfrak D)(\mc W) \to (\widehat{D}(\mathfrak C)^\vee\boxtimes\widehat{D}(\mathfrak D))(\mc W).
$$
Then, via Construction \ref{constr:FM-transform 1} we get a natural additive functor
$$
\mr{FM}_{-}\colon (\mathfrak C^\vee\boxtimes \mathfrak D)(\mc W) \to \Fun(\widehat{D}(\mathfrak C)(\mc X), \widehat{D}(\mathfrak D)(\mc Y)).
$$
\end{constr}
\subsection{Group scheme action on an abelian category.}\label{Groupsactingonacategory} Let us fix a base scheme $S$. Given two fpqc-stacks $X,Y$ over $S$ we will denote $X\times Y:= X\times_S Y$.

For a flat affine group scheme $H$  over $S$ acting on a fpqc-stack $X$ and a quasi-coherent sheaf of categories $\fC$ over $X$, an $H$-equivariant structure on  $\fC$ consists of a  quasi-coherent  sheaf of categories  $\fC_{X/H}$ over the quotient stack $X/H$ together with an equivalence   of  quasi-coherent  sheaves of categories between the pullback of $\fC_{X/H}$ to $X$ and $\fC$.

In particular, if $H$ acts trivially on $X$ we shall refer to the above data as an action of $H$ on $\fC$. Note that, indeed, for every affine scheme $Z\to X$, the group $H(Z)$ acts on the abelian category $\fC(Z)$ (via the natural action of $H(Z)$ on the point $Z\to X \to X/H\simeq X\times BH$). For $h\in H(Z)$ we will denote by $\phi_h$ the corresponding autoequivalence of $\fC(Z)$. 

The following construction will play an important role in what follows.
\begin{constr}[A version of the Grothendieck construction]\label{constr:A version of the Grothendieck construction}
 Let  $A$ be an algebra in $\QCoh^{\mr{naive}}(X)$ whose underlying quasi-coherent sheaf is locally free of finite type\footnote{Meaning that for any affine scheme $f\colon Z\ra X$ the pull-back $f^*A$ is locally free of finite type over $Z$.} and let 
$ \underline{A}^\times$ denote the relatively affine group scheme\footnote{That is, a group stack $\ul A^\times \ra X$ over $X$ that is representable in affine schemes.}  of its invertible elements. Let 
\begin{equation}\label{twistext.eq}
  1 \to \underline{A}^\times   \to  \widehat H_X \to  H \times X\to 1 
\end{equation}
be an  extension of group schemes over $X$ together with an action 
\begin{equation}\label{twistext2.eq}
\alpha\colon \widehat H_X \to \underline{\Aut} (A)
\end{equation}
of  $\widehat H_X$ on $A$ by algebra automorphisms such that  
\begin{itemize}
\item[(T)] the action of  $\widehat H_X$ on  $\underline{A}^\times$ induced by (\ref{twistext2.eq}) is equal to the adjoint action arising from exact sequence (\ref{twistext.eq}). 
\end{itemize}
Since $\ul A^\times\subset \ul A$ is dense, it follows that the restriction of $\alpha$ to the subgroup  $\underline{A}^\times \subset \widehat H_X$ is equal to $\Ad$.  
Extension (\ref{twistext.eq}) gives rise to an action of the group scheme  $H$ on  $A\lmod$.   More precisely, (\ref{twistext2.eq}) defines an action of $\widehat H_X $  on $A\lmod$ and we claim that the latter factors canonically through $H$. Indeed, for an affine  scheme  over $X$, $u\colon Z \to X$, and an element $a\in  \underline{A}^\times (Z)\subset \widehat H(Z)$,  the corresponding endofunctor $\phi_{a}$ of   $ A\lmod(Z)$ is induced by the inner automorphism $\Ad_a$ of $u^*A(Z)$. Consequently, the multiplication by $a$
induces an isomorphism between the identity endofunctor and $\phi_{a}$. The above construction determines descent data for $A\lmod$ 
 along the morphism $X \to  X/H= X\times BH$, where $BH$ is the classifying stack of $H$.  We denote the corresponding quasi-coherent sheaf of categories over $X/H$ by  $A\lmod^H$. 

\end{constr}
\begin{rem}\label{rem:H-action via an extension}
    Equivalently, one can view the action of $H$ on $A\lmod$ (as a sheaf of abelian categories on $X$) given in Construction \ref{constr:A version of the Grothendieck construction} in the following way. For a stack $Y\ra X$ let us still denote by $A$ the pull-back of $A$ to $Y$.  As we discuss below, the extension \eqref{twistext.eq} together with the map \eqref{twistext2.eq} produces an $A$-bimodule $M$ over $H\times X$ together with an isomorphism
    $$
\gamma\colon \mr{pr}_1^*M \otimes_{A} \mr{pr}_2^*M \iso m^*M 
    $$
     of $A$-bimodules (where $m,\mr{pr}_i\colon H\times H\times X \ra H\times X$ denote the multiplication and projections
    ). Moreover, the bimodule $M$ is invertible\footnote{Meaning there is a bimodule $M'$ and natural isomorphisms of $A$-bimodules $M\otimes_{A}M'\simeq M'\otimes_{A}M\simeq A$.}, and the isomorphism $\gamma$ satisfies a cocycle condition for its various pull-backs to the triple product $H\times H\times H\times X$. Viewing $M$ as an autoequivalence of $A\lmod$ (given by $N\mapsto M\otimes_{A}N$) then gives a descent datum for $A\lmod$ to $X\times BH$.

    Let us construct the pair $(M,\gamma)$, given the data of \eqref{twistext.eq} and \eqref{twistext2.eq}. First we define $M$ as a left $A$-module. 
    The group scheme $\ul{A}^\times$ acts on $A$ by multiplication on the right. This action preserves the left $A$-module structure and so the formula $M:= A \times^{\ul{A}^\times} \widehat H_X$ defines a sheaf of left $A$-modules locally isomorphic to $A$.  Now, the map $\alpha\colon \widehat{H}_X \ra \underline{\Aut} (A)$, $h\mapsto \alpha_h$ defines a (commuting) right action of $A$ on $M$: locally, it is given by the formula $(a,h)\circ a':= (a\cdot \alpha_h(a'),h)$. (We leave it to the reader to check that this action is well-defined). This defines an object $M\in A\bimod$. 

Since  $\widehat H_X \twoheadrightarrow H\times X$ is an fpqc-surjection, to see that $M$ is invertible as an $A$-bimodule, it is enough to see that $f^*M$ is invertible for those maps $f\colon Z \ra H\times X$ from an affine $Z$ which factor as  $f\colon Z \xra{\widehat h}\widehat H_X \twoheadrightarrow H\times X$. Such a factorization gives a map $\alpha_{\widehat h}:= \alpha\circ \widehat h\colon Z \ra \underline{\Aut} (A)$ and the $A(Z)$-bimodule $f^*M$ is given explicitly by $A(Z)$ with the standard left $A(Z)$-action and right action given by $(a,a') \mapsto a\cdot \alpha_{\widehat{h}}(a')$. This bimodule corresponds to the endofunctor of $A(Z)\lmod$ induced by the automorphism $\alpha_{\widehat{h}}\in \underline{\Aut} (A)(Z)\simeq \Aut (A(Z))$, and so is invertible\footnote{Explicit inverse to such a bimodule is given by $A(Z)$ where right $A(Z)$-action is standard and left action is given by $(a,a') \mapsto  (\alpha_{\widehat{h}}(a)\cdot a')$.}.

    It remains to construct $\gamma$. Locally, given a map $$(h_1,h_2)\colon Z\ra (H\times X)\times_X (H \times X)\simeq H \times H \times X$$ from an affine $Z$, we define a map $\gamma'_{(h_1,h_2)}\colon$ 
    $$
    \begin{tikzcd}M_{h_1}\times M_{h_2} \simeq \left(A(Z) \times^{\ul{A}^\times(Z)} (\widehat H_X)_{h_1}\right) \times \left(A(Z) \times^{\ul{A}^\times(Z)} (\widehat H_X)_{h_2}\right) 
        \arrow[r] & A(Z) \times^{\ul{A}^\times(Z)} (\widehat H_X)_{h_1h_2}\simeq M_{h_1h_2},
    \end{tikzcd} 
    $$
    by the formula $(a_1,\widehat h_1)\times (a_2,\widehat h_2)\mapsto (a_1\cdot \alpha_{\widehat h_1}(a_2), \widehat h_1\widehat h_2)$. One can check that $\gamma'_{(h_1,h_2)}$ is biadditive, respects left and right $A(Z)$-actions, and factors through the tensor product $M_{h_1}\otimes_{A(Z)} M_{h_2}$.  This defines the desired map of $A(Z)$-bimodules
    $$
    \gamma_{(h_1,h_2)}\colon M_{h_1}\otimes_{A(Z)} M_{h_2} \ra M_{h_1h_2}.
    $$
    To show that it's an isomorphism, it is enough to do so for those points $(h_1,h_2)\colon Z\ra (H\times X)\times_X (H \times X)$ which factor as 
    $$
Z\xra{(\widehat h_1,\widehat h_2)} \widehat H_X \times_X \widehat H_X \twoheadrightarrow (H\times X)\times_X (H \times X).
    $$
    Fixing such a factorization identifies $M_{h_1}$, $M_{h_2}$ and $M_{h_1h_2}$ with the bimodules corresponding to the autoequivalences of $A(Z)\lmod$ induced by the automorphisms $\alpha_{\widehat h_1}$, $\alpha_{\widehat h_2}$ and $\alpha_{\widehat h_1\widehat h_2}$ of $A(Z)$ correspondingly. One can then see explicitly that the map $\gamma_{(h_1,h_2)}$ is an isomorphism. 
\end{rem}
\begin{example}\label{ex:action on A-mod induced by action on A}
   \begin{enumerate}[leftmargin=*, rightmargin=0pt]
       \item 
   It can be illustrative for the reader to see the data of \eqref{twistext.eq} and \eqref{twistext2.eq} that correspond to the $H$-action on $A\lmod$ defined by an $H$-action $\phi\colon H \ra \ul{\Aut}(A)$ on $A$ itself. Namely, the corresponding extension $\widehat{H}_X$ is given by the semidirect product (with respect to $\phi$)
    $$
1 \to \underline{A}^\times   \to (H\times X) \ltimes_X \underline{A}^\times  \to  H \times X\to 1 
    $$
    with $\alpha$ given by the map $(H\times X) \ltimes_X \underline{A}^\times \ra \ul{\Aut}(A)$ sending $(h,a)\mapsto h\circ \Ad_a$. Condition $(T)$ guarantees that $\alpha$ is well-defined. Indeed, e.g. in terms of Remark \ref{rem:H-action via an extension}, one checks that the corresponding bimodule $M$ on $H\times X$ is given by $A$ with the standard left $A$-action and right $A$-action twisted by the automorphism $\phi$, which is exactly the bimodule inducing the autoequivalence $\phi^*\colon A\lmod \ra A\lmod$ over $H\times X$.
    \item Let $A_1, A_2$ be two algebras in $\QCoh^{\mr{naive}}(X)$ that are locally free of finite type over $X$. Then given $H$-actions on $A_1\lmod$ and $A_2\lmod$ we have a natural action on $(A_1\otimes_{\mc O_X} \! A_2)\lmod$: in terms of sheaves of categories on $BH_X$, it is given by the tensor product $A_1\lmod^H\otimes_{\mc O_{BH_X}}A_2\lmod^H$ (Construction \ref{constr: tensor product of sheaves of categories}). In the case the actions on  $A_i\lmod$ for $i=1,2$ come from Construction \ref{constr:A version of the Grothendieck construction}, one can also explicitly describe the corresponding data \eqref{twistext.eq} and \eqref{twistext2.eq}. Namely, let 
  	$$
  	  1 \to \underline{A}^\times_i  \to  \widehat H_{i,X} \to  H_X\to 1 
  	$$
  	and $\alpha_i\colon \widehat H_{i,X} \ra \ul \Aut(A_i)$ define an action of $H$ on $A_i\lmod$. Then the action on $(A_1\otimes_{\mc O_X} \! A_2)\lmod$ is given by the extension 
  	$$
  	1 \to \ul{(A_1\otimes A_2)}^\times \to \widehat H_{1\otimes 2,X} \to H_X \to 1
  	$$ 
  	where $\widehat H_{1\otimes 2,X}$ is defined as the quotient of $(\widehat H_{1,X}\times_{H_X} \widehat H_{2,X}) \ltimes \ul{(A_1\otimes A_2)}^\times $ by the subgroup $\underline{A}^\times_1\times \underline{A}^\times_2$, embedded antidiagonally via $(a_1,a_2)\mapsto ((a_1,a_2),(a_1\otimes a_2)^{-1})$. The map $\alpha_{1\otimes 2}\colon  \widehat H_{1\otimes 2,X} \ra \ul \Aut(A_1\otimes A_2)$ is defined as follows: the components $\widehat{H}_{i,X}$ act separately on the $A_i$, while $\ul{(A_1\otimes A_2)}^\times$ acts by conjugation.
    \end{enumerate} 
\end{example}

\begin{rem}[Explicit local description of $A\lmod^H$] \label{rem:description of sections of AXH-mod}
	 Let $Z$ be an fpqc-stack and let $f\colon Z \to X\times BH$ be a morphism. Let us give an explicit description of the category of global sections $ A\lmod^H(Z)$.
	Let $u\colon \cM \to Z$ be the $H$-torsor corresponding to $f$.
	The group $\widehat H_X $ acts on $\cM$ via the projection $\widehat H_X \to H \times_S X$.
	Action (\ref{twistext2.eq}) of $\widehat H_X$ on $A$ makes the pullback $A_\cM$  of $A$ to $\cM$ an $\widehat H_X $-equivariant sheaf of algebras on $\cM$.
	Then  $ A\lmod^H(Z)$  is the category of $\widehat H_X $-equivariant (naive) quasi-coherent\footnote{So, objects of the category of $A_\cM$-modules in $\QCoh^{\mr{naive}}(\mc M/ \widehat H_X)$ (for the latter see Example \ref{ex:examples of sheaves of categories}\ref{item:modules over a sheaf of algebras}).}
	$A_\cM$-modules $\cF$ on  $\cM$, such that the action of $\underline{A}^\times $ on $\cF$ induced by the   $A_\cM$-module structure coincides with the action of the subgroup 
	$\underline{A}^\times \mono \widehat H_X $. 
	
	In the case $Z=X\times BH$ and $f$ is the identity map we get the following explicit description of the category of global sections of $A\lmod^H$; namely an object in $A\lmod^H(X\times BH)$ is given by a quasi-coherent sheaf $M\in \QCoh^{\mr{naive}}(X)$ together with an action of $A$  and an action $\beta\colon \widehat H_X \ra \ul\Aut_X(M)$ such that
	\begin{enumerate}[label=(\arabic*)]
		\item $\beta_h(am)=\alpha_h(a)\beta_h(m)$ for $a\in A$, $m\in M$, $h\in \widehat H_X$;
		\item $\beta_a(m) =am$ for $a\in A^\times\subset \widehat H_X$ and $m\in M$.
	\end{enumerate}	
	\end{rem}

 
 Assume that,  for $f\colon Z \to X\times BH$ as above, we are given an $\widehat H_X$-torsor $\widehat \cM \to Z$ with $\widehat \cM/\underline{A}^\times = \cM$. 
  Using the action of $\widehat H_X$ on $A$ define 
 a quasi-coherent sheaf of $\cO_Z$-algebras $O= \widehat \cM  \times ^{\widehat H_X}  A_Z$.
 Applying the construction from Example \ref{ex:examples of sheaves of categories}\ref{item:modules over a sheaf of algebras}  we consider the quasi-coherent sheaf of categories  $O\lmod$  over $Z$.
  By definition, the category $O\lmod(Z)$  is equivalent to the category of $\widehat H_X $-equivariant quasi-coherent 
modules over $A_{\widehat \cM }$.  An object $\cF \in  A\lmod^H(Z)$  determines via the pullback along $\widehat \cM  \to \cM $ 
  an object of $O\lmod(Z)$.  Repeating  this construction for every affine scheme over $Z$ we get a functor between quasi-coherent sheaves of categories over $Z$:
 \begin{equation}\label{twist.morph.ofstacks}
 f^*(A\lmod^H) \to O\lmod.
\end{equation}
\begin{pr}\label{pr.twistedmodules} 
 Morphism (\ref{twist.morph.ofstacks}) is an equivalence. In particular, it induces an equivalence of categories  $A\lmod^H(Z) \iso O\lmod(Z)$.
 \end{pr}
 \begin{proof} Thanks to the faithfully flat descent for quasi-coherent sheaves of categories it suffices to exhibit  an {\it fpqc}  cover $Z'\to Z$ such that functor (\ref{twist.morph.ofstacks}) induces  an equivalence of categories  $A\lmod^H(Z')   \iso O\lmod(Z')$. We take $u\colon Z'=  \widehat \cM  \to Z$.
Then both categories are identified with  the category of quasi-coherent  $A_{\widehat \cM}$-modules and   morphism (\ref{twist.morph.ofstacks}) evaluated at $\widehat \cM$ is isomorphic to the identity functor.
  \end{proof}
 \begin{example} \label{example: O-mod when there is a section} Assume that the group homomorphism $\widehat H_X \to  H\times_S X $ has a section $s\colon H\times_S X \to \widehat H_X$. Then the action of $H$ on  $A\lmod$ lifts to an action of $H$ on $A$. Moreover, the section defines a lift $\widehat \cM \to Z$ of  $\cM \to Z$ such that the corresponding sheaf of algebras is given by the formula 
 $O=  \cM  \times ^{H}  A$. Note that when the image of the section $s$ commutes with $\ul{A}^\times \subset \widehat H_X$ the action of $H$ on $A$ is automatically trivial (because $\ul A^\times\subset \ul A$ is dense).
 \end{example}

 \begin{rem}\label{rem:informalaction}
Let $\phi$ be an action of $H$ on the category  $A\lmod$. For a scheme $Z$ over $X$ and a point $h\in H(Z)$  denote by  $\phi_{h}$ the corresponding autoequivalence  of the category of $A_Z$-modules. Assume  that,
for every $Z\to X$ and $h\in H(Z)$,  {\it fpqc}  locally on $Z$ the autoequivalence $\phi_h$ preserves the isomorphism class of the free $A_Z$-module $A_Z$.
Then $\phi$
 arises from data (\ref{twistext.eq}) and  (\ref{twistext2.eq}) for an appropriate $ \widehat H_X$.  
 Namely, one sets $ \widehat H_X(Z)$ to be the group of pairs $(h, \alpha_h)$, where $h\in H(Z)$ and $\alpha_h\colon  \phi_{h}(A_Z) \iso A_Z$ is an isomorphism of $A_Z$-modules. If $X$, $H$ are schemes of finite type over a field then it is enough to check the above condition on $\phi$ for schemes $Z$ which are also of finite type.
  \end{rem}

   \section{Action of $G_0$  on $\cA_{{\mathbb P^1}}\lmod$.}\label{canonicalquant.theaction}
  \subsection{Main Theorem.}\label{subs.m.th} The action of  $G_0$ on $A_0$ defines its action on  $A_0\lmod$. 
 The goal of this subsection is to extend it to an  action of $G_0$  on $\cA_{{\mathbb P^1}}\lmod$. 

  \begin{Th}\label{appendix:maintheorem} Let $H$ be an affine group scheme over $\bF_p$ equipped with a homomorphism $\upsilon\colon H\to G_0$.
  There exists a unique (up to a unique isomorphism) action of the group scheme $H$  on $\cA_{{\mathbb P^1}}\lmod$ equipped with a trivialization at $\{\infty\}\hookrightarrow \mathbb P^1$, such that 
   the induced action of $H $ on the center of $A_0\lmod $  (identified with $A_0$) equals the composition   $H\rar{\upsilon} G_0 \mono  \underline{\Aut} (A_0)$.
  \end{Th}

\begin{proof} Let's start with some preliminary remarks. 
\begin{enumerate}[label=\alph*)]
\item \label{enum: every autoequivalence is locally an automorphism} Set $H_{{\mathbb P^1}} =H \times {\mathbb P^1}$. 
  We claim that giving an action of $H$ on  $\cA_{{\mathbb P^1}}\lmod$ amounts to specifying an extension   
  \begin{equation}\label{diag.f.extany}
	\begin{tikzcd}
	1 \arrow[r, ""] & \underline{\cA}_{{\mathbb P^1}}^\times  \arrow[r, ""]\arrow[dr, "Ad"]& \widehat H_{{\mathbb P^1}}  \arrow[r, ""]\arrow[d, "\alpha"]& H_{{\mathbb P^1}} \arrow[r, ""] &1  \\
	&& \underline{\Aut} (\cA_{{\mathbb P^1}}), &  &
	\end{tikzcd}
	\end{equation}
	satisfying property (T) in Construction \ref{constr:A version of the Grothendieck construction}. Indeed, using Remark \ref{rem:informalaction} it suffices to check that, for any affine scheme $Z$ over ${\mathbb P^1}$, every autoequivalence  $\phi$ of the category  $\cA_{{\mathbb P^1}}(Z)\lmod$ over $\cO(Z)$ {\it fppf}  locally on $Z$ preserves the isomorphism class of the free $\cA_{{\mathbb P^1}}(Z)$-module $\cA_{{\mathbb P^1}}(Z)$. We shall see that this is true even Zariski locally on $Z$. Choose a point $z\in Z$ and let $Z_{(z)}:=\Spec \mc O_{Z,z}$ be the spectrum of the local ring. 
 If the image of $z$ in ${\mathbb P^1}$ is not $0$ then $\cA_{{\mathbb P^1}}(Z)$ is  an Azumaya algebra in some Zariski neighborhood $U$ of $z$. Hence, the restriction of $\phi$ to $U$ is given by the tensor product with a line bundle $\mc L$. Replacing $U$ with a smaller neighbourhood over which $\mc L$ is trivial, we get that $\phi$ restricted to that neighbourhood is isomorphic to the identity functor. If the image of $z$ in ${\mathbb P^1}$ is $0$, then $\cA_{{\mathbb P^1}}(Z_{(z)})$ is a local ring: indeed, its reduction modulo the maximal ideal $\mathfrak{m}_z\subset \mc O_{Z,z}$ is $A_0(\kappa(z))$, which is local. Hence, by Kaplansky's theorem, every projective $\cA_{{\mathbb P^1}}(Z_{(z)})$-module is free. In particular, this can be applied to a projective module  $\phi( \cA_{{\mathbb P^1}}(Z_{(z)}))$.  Using that the ring of endomorphisms of  $\phi( \cA_{{\mathbb P^1}}(Z_{(z)}))$ is a local ring (namely, $\cA_{{\mathbb P^1}}(Z_{(z)})^{op}$) we conclude that $\phi( \cA_{{\mathbb P^1}}(Z_{(z)}))$ is indecomposable, and therefore isomorphic to $\cA_{{\mathbb P^1}}(Z_{(z)})$. Since $\phi(\cA_{\mathbb P^1}(Z))$ is a finitely generated
projective $\cA_{\mathbb P^1}(Z)$-module, the isomorphism
over $Z_{(z)}$, together with its inverse, extends to a
Zariski neighbourhood of $z$.
	
  \item \label{enum:triviality of A-star torsor} By Lemma \ref{lem:local triviality of A-module}, $\widehat H_{{\mathbb P^1}}$ viewed as an
 $\underline{\cA}_{{\mathbb P^1}}^\times$-torsor over $H_{{\mathbb P^1}}$ is locally trivial for the Zariski topology on $H_{{\mathbb P^1}}$. 
	
\item  Under the identification in part a)  a trivialization of the action at $h= \infty$  amounts to the choice of a section $H=H \times  \{\infty\} \to \widehat H_{{\mathbb P^1}}   \times_{{\mathbb P^1}} \{\infty\} $ whose image commutes with the subgroup $ \underline{\cA}_{{\mathbb P^1}}^\times  \times _{{\mathbb P^1}} \{\infty\} \subset  \widehat H_{{\mathbb P^1}}   \times _{{\mathbb P^1}} \{\infty\} $ (see Example \ref{example: O-mod when there is a section}).
	The action of $H= H \times  \{0\}$ on the center of $A_0\lmod $  is induced by $\alpha_0\colon \widehat H_{{\mathbb P^1}}  \times _{{\mathbb P^1}} \{0\} \to \underline{\Aut} (A_0)$. The uniqueness part of Theorem \ref{appendix:maintheorem}  asserts that, for any two extensions (\ref{diag.f.extany}), such that $\alpha_0$ equals the composition 
	$$\widehat H_{{\mathbb P^1}}  \times _{{\mathbb P^1}} \{0\} \to H_{{\mathbb P^1}}  \times _{{\mathbb P^1}} \{0\} \rar{\upsilon} G_0\subset \underline{\Aut} (A_0),$$
	together with chosen sections at $h= \infty$,
	 there exists a unique isomorphism between them that carries one section to the other.  
\end{enumerate}

   We will first prove the uniqueness assertion in the theorem.  Let $T$ be an affine scheme over $\bF_p$, $\pi\colon T\times {{\mathbb P^1}} \to {\mathbb P^1}$ the projection. The tensor product with a line bundle defines a functor
  \begin{equation}\label{twist.morph.ofgroupoids}
  \Pic(T\times {\mathbb P^1}) \to \Aut(\pi^*\cA_{{\mathbb P^1}}\lmod)
  \end{equation}
  from the Picard groupoid   of line bundles over $T\times {\mathbb P^1}$ to the groupoid of autoequivalences\footnote{The composition of autoequivalences makes $\Aut(\pi^*\cA_{{\mathbb P^1}}\lmod)$ into a group object in the category of groupoids. 
  Functor (\ref{twist.morph.ofgroupoids})  is a homomorphism of groups in groupoids.} of the quasi-coherent sheaf of categories $\pi^*\cA_{{\mathbb P^1}}\lmod$.
   \begin{lm}\label{app.lemma.autoequivalences} 
  Functor  (\ref{twist.morph.ofgroupoids})  is fully faithful and its essential image  consists of the subgroup stack  $\fH \subset \Aut(\pi^*\cA_{{\mathbb P^1}}\lmod)$ spanned by autoequivalences that act identically on the center  of  $(\pi^*\cA_{{\mathbb P^1}})_{|T\times 0} \lmod$ (identified with  $A_0\otimes \cO(T)$). 
   \end{lm}
  \begin{proof} 
   The only part that requires a proof is the essential surjectivity of the functor  $ \Pic(T\times {\mathbb P^1})  \to \fH$. 
 Let $\phi$ be an autoequivalence in  $\fH$. We claim that  $\phi$,   locally for the Zariski topology on $T\times {\mathbb P^1}$,  is isomorphic to $\Id$. Indeed, as explained in remarks \ref{enum: every autoequivalence is locally an automorphism} and \ref{enum:triviality of A-star torsor}, every autoequivalence of $\pi^*\cA_{{\mathbb P^1}}\lmod $ is, Zariski locally on $T\times {\mathbb P^1}$, induced by an automorphism of $\pi^*\cA_{{\mathbb P^1}}$. Since $\phi \in \fH$,
  the latter automorphism is equal to $\Id$ modulo $h$.  
 Now the claim follows from Corollary \ref{app.lemma.inneraut}. Assigning to every open subset $U\subset T\times {\mathbb P^1}$ the set of isomorphisms $\phi_{|U} \simeq \Id$ we define a torsor (for the Zariski  topology) under the sheaf of central elements in $(\pi^*\cA_{{\mathbb P^1}})^\times$.
   The latter sheaf is equal to $\cO^\times_{{\mathbb P^1}\times T}$. This defines a line bundle $L_\phi$ over $T\times {\mathbb P^1}$ such that $\phi$ is  the tensor product with $L_\phi$.  
   \end{proof}
  Next, consider the groupoid $\fH'$ formed by pairs $(\phi, \alpha)$, where $\phi\in \fH$ and $\alpha\colon  \phi_{|T\times \infty}\iso \Id$. 
  The equivalence  $\Pic(T\times {\mathbb P^1}) \iso \fH$ induces an equivalence  between $\fH'$  and the groupoid of line bundles over $T\times {\mathbb P^1}$  equipped with a trivialization at $T\times \{\infty\}$. Since the Picard stack of ${\mathbb P^1}=\bP^1_{\bF_p}$ is isomorphic to $\bZ\times B\bG_m$ it follows that $\fH'$
  is discrete and its $\pi_0$ is the group $H^0(T, \bZ)$.

Returning to the proof of uniqueness, recall that giving  an action $\psi$ of $H$  on $\cA_{{\mathbb P^1}}\lmod$ is equivalent to giving  descent data along the morphism ${\mathbb P^1} \to  {\mathbb P^1} \times  BH $. In particular, any action gives rise to an autoequivalence 
  $$\phi\colon  \pi^*\cA_{{\mathbb P^1}}\lmod \iso \pi^*\cA_{{\mathbb P^1}}\lmod ,$$
  where $\pi\colon H\times {\mathbb P^1} \to {\mathbb P^1}$ is the projection. If $\phi _1$, $\phi _2$ are equipped with trivializations at $h= \infty$ and induce the same action on the center of $A_0\lmod$ then $\phi_1 \circ \phi_2^{-1}$ is an object of $\fH'$ corresponding to some element  $[\phi_1 \circ \phi_2^{-1}] \in H^0(H, \bZ)$.  Using the cocycle {\it isomorphisms} related to $\phi_i$ it follows 
  that $[\phi_1 \circ \phi_2^{-1}]$ is, in fact, a group homomorphism $H\to \bZ$. The latter must be trivial because $H$ is assumed to be affine and, in particular, quasi-compact.  We conclude that there is a unique isomorphism $\phi_1 \iso \phi_2$ that commutes with trivializations at $H\times \{\infty\}$. The discreteness of  $\fH'$ (applied to  $T=H\times H$) implies
  that $\phi_1 \iso \phi_2$ commutes with the cocycle \textit{isomorphisms}. This completes the proof of the uniqueness assertion of  the Theorem.
  \begin{rem}\label{app.uniquenessoverabase} 
  The argument above proves a more general assertion:  if $T$ is any scheme over $\bF_p$ and $\pro_{{\mathbb P^1}}\colon {\mathbb P^1} \times T \to {\mathbb P^1}$ is the projection, then there exists at most one (up to a unique isomorphism) action 
  of $H$  on $\pro_{{\mathbb P^1}}^* \cA_{{\mathbb P^1}}\lmod$ equipped with a trivialization over $\infty \times T$ such that 
   the induced action of $H$ on the center of $(\pro^*_{{\mathbb P^1}} A_{{\mathbb P^1}})_{0\times T} \lmod $   is given by $\upsilon$.
 \end{rem}
  
  Let us prove the existence.  Without loss of generality we may assume that   $H =G_0$, $\upsilon=\Id.$ Let $A_h$  be the $h$-adic completion of $\cA_{{\mathbb P^1}\backslash\{\infty \} }$. 
  For a group scheme $H$ over $\mbb F_p$, let $L^+H$ and $LH$ be the {\it fppf} sheafifications of the group presheaves  sending an $\bF_p$-algebra $R$  to $H(R[[h]])$ and $H(R((h)))$.
  
 Recall from Remark \ref{rem:A_S is matrix algebra outside of 0} that there is an isomorphism $A_h(h^{-1}) \iso \Mat_{p^n}(\mbb F_p((h)))$.  This defines an embedding $G = \ul\Aut(A_h)\mono  \ul\Aut(A_h[h^{-1}]) \simeq L\!\PGL(p^n)$, and a central extension
 \begin{equation}\label{fundcentralextensionbisbis}
1 \to L\bG_m \to \widetilde G  \to   G \to 1, \qquad \widetilde G=L\!\GL(p^n)\times_{L\!\PGL(p^n)} G
\end{equation} 
together with an embedding $\underline{A}_h^\times \mono  \widetilde G$  induced by the maps $\underline{A}_h^\times \rightarrow \underline{A_h[h^{-1}]}^\times\simeq L\!\GL(p^n)$ and $\underline{A}_h^\times \rightarrow G$ (see \cite[\S 3.2]{bv} or Proposition \ref{pr:surjectivity} for a stronger and cleaner statement).
By a result of Contou-Carrere (see for instance \cite[Lemma A.1]{bv}), the extension 
$L^+\bG_m \ra L\bG_m\ra L\bG_m/L^+\bG_m =: \Gra_{\bG_m}$ splits: 
we have an isomorphism of group-valued sheaves
\begin{equation}\label{eq:contou}
    L\bG_m\iso L^+\bG_m \times \underline{\bZ} \times  \underline{\mr{Mor}_*}({\mathbb P^1}\!\setminus \!\{0\},\mbb G_m), 
\end{equation}
where $ \underline{\mr{Mor}_*}({\mathbb P^1}\!\setminus \!\{0\},\mbb G_m) $ 
is the subsheaf of $\underline{\mr{Mor}}({\mathbb P^1}\!\setminus \!\{0\},\mbb G_m)$ formed by maps sending the point $\{\infty\}\in {\mathbb P^1}\!\setminus \!\{0\}$ to $\{1\}\in \mbb G_m$ (that is, for any ring $R$, $\underline{\mr{Mor}_*}({\mathbb P^1}\!\setminus \!\{0\},\mbb G_m)(R)= (1+h^{-1}R[h^{-1}])^\times$).
Thus we have the commutative diagram
\begin{equation}\label{diag.f.ext.bisnew}
	\begin{tikzcd}
    &  & L^+\bG_m \times  \Gra_{\bG_m} \arrow[d, hook, ""]&  &  \\
	1 \arrow[r, ""] & \underline{A}_h^\times \times  \Gra_{\bG_m} \arrow[r, "{(a,b)}"]\arrow[dr, "{Ad\times e}"']& \widetilde G \arrow[r, ""]\arrow[d, two heads, ""]& G_0\arrow[r, ""] &1  \\
	&& G,  &  &
	\end{tikzcd}
	\end{equation} 
  where $e\colon  \Gra_{\bG_m}\to   G$ is the trivial homomorphism. Consider the following two extensions extracted from diagram \eqref{diag.f.ext.bisnew}

    \begin{equation}\label{eq:actionontheformaldisk}
	\begin{tikzcd}
    & & \ \ \widetilde G/\! \Gra_{\bG_m} \arrow[d, equals] & & \\
	1 \arrow[r, ""] & \underline{A}_h^\times  \arrow[r, "a"]\arrow[dr, "Ad"']& \widehat G  \arrow[r, ""]\arrow[d, ""]& G_0\arrow[r, ""] &1  \\
	&& \underline{\Aut} (A_h)=G,  &  &
	\end{tikzcd}
	\end{equation} 

      \begin{equation}\label{eq:actiononthecomplementtotheformaldisk}
      \begin{tikzcd}
      & & \ \   \widetilde G/\!  \underline{A}_h^\times \arrow[d, equals] & & \\
          1 \arrow[r, ""]& \Gra_{\bG_m} \rar{-b}   \arrow[r, ""]& \widetilde {G}_0  \arrow[r, ""] & G_0 \arrow[r, ""] & 1.  
          \end{tikzcd}
      \end{equation}
Let
$
j\colon \widetilde G\hookrightarrow\underline{A_h[h^{-1}]}^\times
$
be the natural inclusion, and set
$
c\coloneqq j\circ b\colon \mathrm{Gr}_{\mathbb G_m}
\longrightarrow\underline{A_h[h^{-1}]}^\times.
$
The image of $c$ is central. We have the following isomorphism of extensions of $G_0$ by $\underline{A_h[h^{-1}]}^\times$:
\begin{equation}\label{eq:gluing}
    \coker \left(\underline{A}_h^\times\rar{\delta} (\widetilde G/\! \Gra_{\bG_m}) \ltimes \underline{A_h[h^{-1}]}^\times \right)\iso 
    \coker \left(\Gra_{\bG_m} \rar{(-b,c^{-1})} (\widetilde G/\!  \underline{A}_h^\times) \times \underline{A_h[h^{-1}]}^\times \right). 
\end{equation}
The semidirect product at the left-hand side of \eqref{eq:gluing} is taken with respect to the action $\widetilde G/\! \Gra_{\bG_m} \to G= \ul\Aut(A_h)\to \ul\Aut((A_h[h^{-1}])^\times)$, and $\delta(a)\coloneqq (a, a^{-1})$. One checks directly that $\operatorname{Im} \delta$ (resp. $\operatorname{Im} (-b,c^{-1})$) is a normal subgroup of $\widetilde G/\! \Gra_{\bG_m} \ltimes \underline{A_h[h^{-1}]}^\times $
(resp. $\widetilde G/\!  \underline{A}_h^\times \times \underline{A_h[h^{-1}]}^\times$).
The isomorphism  \eqref{eq:gluing} is induced by $j\colon \tilde G\mono \underline{A_h[h^{-1}]}^\times$, and given by the formula $[(g,a)]\mapsto [(g, j(g)\cdot a)]$.

We will explain below how  \eqref{eq:actionontheformaldisk} determines an action of $G_0$ on the restriction $\cA_{\widehat{\mbb D}}\lmod$ of $\cA_{{\mathbb P^1}}\lmod$ to the formal completion $\widehat{\mbb D}:=\Spf \bF_p[[h]]$ of ${{\mathbb P^1}}$ at $\{0\}$. The diagram \eqref{eq:actiononthecomplementtotheformaldisk} determines an action of $G_0$ on $\cO_{{{\mathbb P^1}}\backslash \{0\}}\lmod \simeq\!\footnote{Recall that $\cA_{{{\mathbb P^1}}}$ is a split Azumaya algebra away from zero.}  \cA_{{{\mathbb P^1}}\backslash \{0\}}\lmod$.  We will then use the Beauville-Laszlo theorem and \eqref{eq:gluing} to patch the two actions together.

For a non-negative integer $r$, denote by $(\underline{A}_h^\times)^{\geq r}$ the subgroup of    $\underline{A}_h^\times$ of elements equal to $1$ modulo $h^{r+1}$.  Diagram \eqref{eq:actionontheformaldisk} yields a compatible family of extensions
\begin{equation}\label{fundextruncold}
  1 \to \underline{A}_h^\times/(\underline{A}_h^\times)^{\geq r}    \to  \widehat G_r  \to  G_0 \to 1,
  \end{equation} 
  where\footnote{It follows from \eqref{eq:actionontheformaldisk} that $(\underline{A}_h^\times)^{\geq r}\subset \widehat G$ is a normal subgroup.} $\widehat G_r: = \widehat G / (\underline{A}_h^\times)^{\geq r}$.
Observe that the group scheme $\underline{A}_h^\times/(\underline{A}_h^\times)^{\geq r} $ over $\Spec \bF_p$  is obtained from the group scheme  $\underline{\cA}_{\mathbb D_{r}}^\times$ over the $r$-th infinitesimal neighborhood  $\mathbb D_{r}:=\Spec \bF_p[h]/h^{r+1}$ of $\{0\}\in {\mathbb P}^1$
by applying the Weil restriction of scalars functor $\operatorname{Res}_{\mathbb D_{r}/\bF_p}$, that is, for any $\bF_p$-algebra $R$, we have \[\underline{A}_h^\times/(\underline{A}_h^\times)^{\geq r}(R) = \mr{Mor}_{\mathbb D_{r}}\!(\mathbb D_{r} \times \Spec R, \underline{\cA}_{\mathbb D_{r}}^\times).\]
 Therefore, by adjunction, sequence (\ref{fundextruncold}) gives rise to an extension of group schemes over $\mathbb D_{r}$
 \begin{equation}\label{fundextrunc}
	\begin{tikzcd}
  1 \arrow[r, ""] &  \underline{\cA}_{\mathbb D_{r}}^\times    \arrow[r, ""]\arrow[dr, "Ad"] &  \widehat G_{\mathbb D_{r}}  \arrow[r, ""]\arrow[d, ""]&  G_0 \times \mathbb D_{r} \arrow[r, ""] &1  \\
&& \underline{\Aut} (\cA_{\mathbb D_{r}}).  &  &
	\end{tikzcd}
	\end{equation} 
Explicitly,  extension (\ref{fundextrunc}) is obtained by pulling back (\ref{fundextruncold}) to $\mathbb D_{r}$ and then taking the pushout with respect to the homomorphism $ \operatorname{Res}_{\mathbb D_{r}/\bF_p}( \underline{\cA}_{\mathbb D_{r}}^\times) \times \mathbb D_{r} \to \underline{\cA}_{\mathbb D_{r}}^\times$.  Using Construction \ref{constr:A version of the Grothendieck construction}, diagram (\ref{fundextrunc})  yields
the promised action of $G_0 \times \mathbb D_{r}$  on $\cA_{\mathbb D_{r}}\lmod$. It is clear from the construction that the actions are compatible as $r$ varies. Following Construction \ref{constr:A version of the Grothendieck construction} and Remark \ref{cat.over.f.s}, we denote by $\cA_{ \widehat{\mbb D}}\lmod^{G_0}$ the corresponding quasi-coherent sheaf of categories over $\widehat{\mbb D}/G_0$.
 
To construct the action over ${\mathbb P^1}\!\setminus \!\{0\}$ observe that 
the extension \eqref{eq:actiononthecomplementtotheformaldisk} reduces uniquely to the subgroup   $\Gra_{\bG_m}^\circ :=\underline{\mr{Mor}_*}({\mathbb P^1}\!\setminus \!\{0\},\mbb G_m)\subset \underline{\bZ}\times \underline{\mr{Mor}_*}({\mathbb P^1}\!\setminus \!\{0\},\mbb G_m)= \Gra_{\bG_m}$, producing an extension
    \begin{equation}\label{eq:actiononthecomplementtotheformaldisk2}
          1 \to \underline{\mr{Mor}_*}({\mathbb P^1}\!\setminus \!\{0\},\mbb G_m) \rar{-b} (\widetilde G_0)^\circ  
          \to G_0\to 1;
      \end{equation}
     see  \cite[Corollary 4.2]{bv}.
Applying the Weil adjunction  to \eqref{eq:actiononthecomplementtotheformaldisk2} we obtain an extension of group schemes over ${\mathbb P^1}\!\setminus \!\{0\}$
\begin{equation}\label{eq:actiononthecomplementtotheformaldisk2Weil}
1\to \bG_m \to \widehat G_{{{\mathbb P^1}}\backslash \{0\}}^{(1)} \to G_0 \times {{\mathbb P^1}}\backslash \{0\}\to 1,
\end{equation}
which splits over $h=\infty$.
This determines an action of $G_0$ on  $\cO_{{{\mathbb P^1}}\backslash \{0\}}\lmod$ with a trivialization over $h=\infty$. Using the Morita equivalence 
  $\cO_{{{\mathbb P^1}}\backslash \{0\}}\lmod \simeq\! \cA_{{{\mathbb P^1}}\backslash \{0\}}\lmod$ we reinterpret  \eqref{eq:actiononthecomplementtotheformaldisk2Weil}
as an action on $\cA_{{{\mathbb P^1}}\backslash \{0\}}\lmod $, explicitly given by the diagram
 \begin{equation}\label{eq:actiononthecomplementtotheformaldisk3Weil}
	\begin{tikzcd}
    & & \ \  (\widehat G_{{{\mathbb P^1}}\backslash \{0\}}^{(1)} \times \underline{\cA}_{{{\mathbb P^1}}\backslash \{0\}}^\times)/\bG_m   \arrow[d, equals] & & \\
	1 \arrow[r, ""] & \underline{\cA}_{{{\mathbb P^1}}\backslash \{0\}}^\times  \arrow[r, ""]\arrow[dr, "Ad"']& \widehat G_{{{\mathbb P^1}}\backslash \{0\} }  \arrow[r, ""]\arrow[d, ""]& G_0\times {{\mathbb P^1}}\backslash \{0\}\arrow[r, ""] &1  \\
	&& \underline{\Aut} (\cA_{{{\mathbb P^1}}\backslash \{0\}}).  &  &
	\end{tikzcd}
	\end{equation}
Our goal is to glue \eqref{eq:actiononthecomplementtotheformaldisk3Weil} and \eqref{fundextrunc} to an action of $G_0$ on $\cA_{{\mathbb P^1}}\lmod$ given by a diagram of the form
 \begin{equation}\label{diag.f.extany2}
	\begin{tikzcd}
	1 \arrow[r, ""] & \underline{\cA}_{{\mathbb P^1}}^\times  \arrow[r, ""]\arrow[dr, "Ad"]& \widehat G_{{\mathbb P^1}}  \arrow[r, ""]\arrow[d, "\alpha"]& G_0 \times {\mathbb P^1} \arrow[r, ""] &1  \\
	&& \underline{\Aut} (\cA_{{\mathbb P^1}}). &  &
	\end{tikzcd}
	\end{equation}
Recall from Remark \ref{rem:H-action via an extension} that the datum of \eqref{diag.f.extany2} is equivalent to that of an  $\cA_{{\mathbb P^1}}$-bimodule $M_{{\mathbb P^1}}$ on $G_0 \times {\mathbb P^1}$ together with an isomorphism of  ${\cA}_{{\mathbb P^1}}$-bimodules over 
$G_0 \times G_0\times {\mathbb P^1}$:
\begin{equation}\label{eq:bimodule}
   \gamma\colon \mathrm{pr}_1^* M_{{\mathbb P^1}} \otimes_{{\cA}_{{\mathbb P^1}}} \mathrm{pr}_2^* M_{{\mathbb P^1}} \iso m^* M_{{\mathbb P^1}},
\end{equation}
satisfying certain properties that define an action of $G_0$ on ${\cA}_{{\mathbb P^1}}\lmod$. 
 We first construct $M_{{\mathbb P^1}}$. 
Twisting the bimodule $\cA_{\mathbb D_r}$ by the $ \underline{\cA}_{\mathbb D_{r}}^\times$-torsor $\widehat G_{\mathbb D_{r}}$, for $r=1,2,\cdots$, we get a left 
$\cA_{\widehat{\mbb D}}$-module $M_{\widehat{\mbb D}}$  over $\Spf \cO(G_0)[[h]]$. Similarly, \eqref{eq:actiononthecomplementtotheformaldisk3Weil} determines an $\cA_{{{\mathbb P^1}}\backslash \{0\}}$-bimodule $M_{{{\mathbb P^1}}\backslash \{0\}}$ over $\Spec \mc O(G_0)[h^{-1}]\simeq G_0\times ({{\mathbb P^1}}\backslash \{0\})$. Let us construct an isomorphism of bimodules
\begin{equation}\label{eq:gluing1a}
    M_{\widehat{\mbb D}}[h^{-1}]\iso M_{{{\mathbb P^1}}\backslash \{0\}} \otimes _{\mc O(G_0)[h^{-1}]} \mc O(G_0)((h)).
\end{equation}

\begin{constr}\label{constr:Drinfeld}
 1. 
    Let us define a functor 
    from the category of $\underline{A_h[h^{-1}]}^\times$-torsors over $G_0$ to the category of left projective  modules over the ring
    $\underline{A_h[h^{-1}]}(G_0)=\mc O(G_0)((h)) \otimes_{\bF_p((h))} \! A_h[h^{-1}]$: if $\cF$ is a torsor as above, pick a faithfully flat
    map $\cO(G_0) \to R$ such that $\cF(R)\ne \varnothing$ and set
$$M_{\mc F}= \lim (\underline{M}(R)  \rightrightarrows \underline{M}(R\otimes_{\cO(G_0)} R)),$$
where, for any ring map $\cO(G_0) \to B$, we set $\underline{M}(B)= \underline{A_h[h^{-1}]}(B) \times ^{\underline{A_h[h^{-1}]}^\times(B)} \cF(B)$.
By a result of Drinfeld (\cite[Theorem 1.6]{mathew}) $M_\mc F$ is a projective $\mc O(G_0)((h))$-module. Since $A_h[h^{-1}]$ is a matrix algebra over 
$\mbb F_p((h))$ it follows that $M_{\mc F}$ is projective as a module over $\underline{A_h[h^{-1}]}(G_0)$.

2. Assume that the torsor $\mc F$ arises from a group extension of the form:
  \begin{equation}\label{eq:actiononthepunformaldisk}
	\begin{tikzcd}   
1 \arrow[r, ""] & \underline{A_h[h^{-1}]}^\times \arrow[r, "a"]\arrow[dr, "Ad"']& \cF  \arrow[r, ""]\arrow[d, ""]& G_0\arrow[r, ""] &1  \\
	&& \underline{\Aut} (A_h[h^{-1}]),  &  &
	\end{tikzcd}
	\end{equation} 
Then $M_{\mc F}$ comes with the structure of a bimodule over $\underline{A_h[h^{-1}]}(G_0)$ 
together with an isomorphism of bimodules over  $\underline{A_h[h^{-1}]}(G_0\times G_0)$
$$\gamma\colon \mr{pr}_1^*M_{\mc F} \otimes_{\underline{A_h[h^{-1}]}(G_0\times G_0)} \mr{pr}_2^*M_{\mc F}  \iso m^*M_{\mc F}, $$
as in Remark \ref{rem:H-action via an extension}.

3. In the case $\mc F$ is induced from an $\underline{A_h}^\times$-torsor $\mc G$ (resp. $\underline{\mr{Mor}_*}({\mathbb P^1}\!\setminus \!\{0\},\mbb G_m)$-torsor $\mc H$) one also gets natural identifications 
$$M_{\mc F}\simeq M_{\mc G}[h^{-1}] \quad \text{(resp. $ M_{\mc F}\simeq N_{\mc H}\otimes_{\mc O(G_0)[h^{-1}]}  \underline{A_h[h^{-1}]}(G_0) \ )$},$$  
where the left $ \underline{A_h}(G_0)$-module $M_{\mc G}$ (resp. $\mc O(G_0)[h^{-1}])$-module $N_{\mc H}$) is defined by descent (analogously to the construction of $\mc M_{\mc F}$ above).

\end{constr}

The left-hand side of \eqref{eq:gluing1a} (as a left $\cA_{\widehat{\mbb D}}[h^{-1}]$-module) is given by applying Construction \ref{constr:Drinfeld} to the extension 
$$1\to \underline{A_h[h^{-1}]}^\times \to \coker \left(\underline{A}_h^\times\rar{\delta} (\widetilde G/\! \Gra_{\bG_m}) \ltimes \underline{A_h[h^{-1}]}^\times \right) \to G_0\to 1.$$
Likewise the right-hand side of \eqref{eq:gluing1a} comes from 
 $$1\to \underline{A_h[h^{-1}]}^\times \to \coker \left(\Gra_{\bG_m} \rar{(-b,c^{-1})} (\widetilde G/\!  \underline{A}_h^\times) \times \underline{A_h[h^{-1}]}^\times \right)\to G_0 \to 1$$ 
 The isomorphism \eqref{eq:gluing}   between these extensions defines \eqref{eq:gluing1a}. 

Applying the Beauville-Laszlo theorem  \cite{bl} to $M_{\widehat{\mbb D}}$, $M_{{{\mathbb P^1}}\backslash \{0\}}$ and \eqref{eq:gluing1a} we obtain 
the desired bimodule $M_{{\mathbb P^1}}$. Isomorphism \eqref{eq:bimodule} is glued from the ones over $\widehat{\mbb D}$ and ${{\mathbb P^1}}\backslash \{0\}$ 
using Part 2 of Construction \ref{constr:Drinfeld}.
This completes the proof of Theorem \ref{appendix:maintheorem}. 
 \end{proof}
 We shall denote by $\cA_ {{\mathbb P^1}}\lmod^{G_0}$ the quasi-coherent sheaf of categories over  ${{\mathbb P^1}}/G_0\simeq {\mathbb P^1}\times BG_0$  obtained by applying Theorem \ref{appendix:maintheorem} to $H=G_0$ and $\upsilon=\Id$. 
 
 We end the section with a proof of the result that we used in Remark \eqref{enum:triviality of A-star torsor} (in the proof of Theorem \ref{appendix:maintheorem}).
 \begin{lm} \label{lem:local triviality of A-module} Let $\cA$ be a sheaf of $\cO_X$-algebras on a scheme $X$ which is finite locally free as an $\cO_X$-module. Then  $$H^1_{\mathrm{fppf}}(X, \cA^\times)\simeq H^1_{\mathrm{Zar}}(X, \cA^\times).$$ 
 \end{lm}
 \begin{proof}
 	It suffices to show that any $\cA$--module $\cL$ that is fppf locally isomorphic to $\cA$ is also Zariski locally isomorphic to $\cA$. Note that the $\cO_X$-module underlying $\cL$ is Zariski locally free by flat descent. Let $\pi\colon \mathrm{Tot}_X(\cL)\ra X$ denote the projection. Consider the open subscheme $U\subset \mathrm{Tot}_X(\cL)$ classifying points $v$ of $\mathrm{Tot}_X(\cL)$ such that the map $\cA_{|\pi(v)} \to \cL_{|\pi(v)}$ defined by $a\mapsto a\cdot v$ is an isomorphism. Since $\mc L$ is fppf locally trivial the projection $U\to X$ is surjective. Now, pick a point $x\in X$; we claim that there exists a $k(x)$-valued point of $U$ over $x$. If the residue field $k(x)$ is infinite this is obvious.  If the residue field of $x$ is finite, use the Lang theorem: $H^1_{\mr{fppf}}(x,\mc A^\times)\simeq H^1_{\textup{\'et}}(x,\mc A^\times)=0$. In both cases this point extends to a section of $U\ra X$  over a Zariski neighbourhood of $x$ which then defines a trivialization of $\mc L$ over this neighbourhood. 
 \end{proof}
 
 \subsection{Restriction of $\cA_ {{\mathbb P^1}}\lmod^{G_0}$  to $\{0\}/G_0$.}\label{subsectionrestrictionto0}
In this subsection we shall identify the action of $G_0$ on $(\cA_ {{\mathbb P^1}}\lmod)_{|\{0\}} \iso A_ {0}\lmod$ from Theorem \ref{appendix:maintheorem} with the one induced by the natural action of $G_0$ on $A_0$.  

For applications in \S \ref{s.lag} we shall work in a more general setting of Theorem \ref{appendix:maintheorem}. Let $H$ be an affine group scheme over $\bF_p$ equipped with a homomorphism $\upsilon\colon H\to G_0$.
 Let $\iota\colon {{\mathbb P^1}} \to {{\mathbb P^1}}$ be the involution sending $h$ to $-h$. There exists a unique isomorphism of algebras over ${{\mathbb P^1}}$
   \begin{equation}\label{app.alpha}
 \alpha\colon  \iota^* \cA_{{\mathbb P^1}}^{op}\iso \cA_{{\mathbb P^1}}
 \end{equation}
 such that, for every $i$,  $\alpha(x_i) = x_i$, $\alpha(y_i)=y_i$. Observe that, for any algebra $\cB_{{\mathbb P^1}}$, the group stack of autoequivalences of  $\iota^* \cB_{{\mathbb P^1}}^{op}\lmod$ is that of  $\cB_{{\mathbb P^1}}\lmod$ pulled back by $\iota$.
  It follows that there is an equivalence of categories between the groupoid of actions of $H$ on  $\cB_{{\mathbb P^1}}\lmod$  and on $\iota^* \cB_{{\mathbb P^1}}^{op}\lmod$.
 Consequently, $\alpha$ defines an autoequivalence of the groupoid of actions of $H$ on  $\cA_{{\mathbb P^1}}\lmod$. We shall refer to the latter as twisting by $\alpha$. Applying this autoequivalence to an action from  Theorem \ref{appendix:maintheorem} (equipped with a trivialization at $\infty$) we derive that its twist by $\alpha$ is uniquely isomorphic to the former one.  We rephrase this using 
 the language of group extensions: if the action is given by a diagram as in (\ref{diag.f.extany}), then there exists a unique isomorphism $\widehat \tau\colon \iota^* \widehat H_{{\mathbb P^1}}  \to  \widehat H_{{\mathbb P^1}} $ fitting in the diagram
  \begin{equation}\label{diag.f.z2}
	\begin{tikzcd}
	1 \arrow[r, ""] &\iota^* \underline{\cA}_{{\mathbb P^1}}^\times  \arrow[r, ""]\arrow[d, "\tau"]& \iota^* \widehat H_{{\mathbb P^1}}  \arrow[r, ""]\arrow[d, "\widehat \tau "]& H_{{\mathbb P^1}}\arrow[r, ""]\arrow[d, "Id"] &1  \\
	1 \arrow[r, ""] &\underline{\cA}_{{\mathbb P^1}}^\times  \arrow[r, ""]&  \widehat H_{{\mathbb P^1}}  \arrow[r, ""] & H_{{\mathbb P^1}}\arrow[r, ""] &1 
	\end{tikzcd}
	\end{equation}
 in which the first row is the pullback of the second one via $\iota$,  $\tau(a) =\alpha(a)^{-1}$,  that commutes with the  homomorphisms to $\underline{\Aut} (\cA_{{\mathbb P^1}}) = \underline{\Aut} (\cA_{{\mathbb P^1}}^{op}) \stackrel{\alpha}{\iso} \iota^*\underline{\Aut} (\cA_{{\mathbb P^1}}) $ and respects the sections at $\infty$.
 In particular, restricting $\widehat \tau$ to $h=0$, we find an involution $\widehat \tau_0$ of  $\widehat H_{{\mathbb P^1}} \times_{{\mathbb P^1}} \{0\}$ that carries $a\in  \underline{A}_0^\times$ to $a^{-1}$ and equals the identity on the quotient $H$. 
 \begin{pr}\label{pr:restrictionto0forgroups} Assume that $H$ is connected. There exists a unique $\widehat \tau_0$-invariant group scheme theoretic section 
 \begin{equation}\label{app.section}
 H \to  \widehat H_{{{\mathbb P^1}}} \times_{{\mathbb P^1}} \{0\}.
 \end{equation}
 of the projection $\widehat H_{{{\mathbb P^1}}} \times_{{\mathbb P^1}} \{0\} \to H$.   In particular, this applies to $H=G_0$.
 \end{pr}
 \begin{proof}
 Let  $(\widehat H_{{{\mathbb P^1}}} \times_{{\mathbb P^1}} \{0\})^{\widehat \tau_0} \subset \widehat H_{{{\mathbb P^1}}} \times_{{\mathbb P^1}} \{0\} $ be the subgroup of $\widehat \tau_0$-invariants. 
 The homomorphism  $\underline{A}_0^\times \rar{a\to a^2} \underline{A}_0^\times$ is surjective with kernel  isomorphic to $\mu_2$. This gives an extension 
 $$1 \to \mu_2 \to   (\widehat H_{{{\mathbb P^1}}} \times_{{\mathbb P^1}} \{0\})^{\widehat \tau_0}  \to H \to 1.$$
 We wish to show that the above has a unique splitting. The uniqueness follows from  connectedness of $H$. For the existence, it suffices to prove the claim for   
 $H=G_0$, $\upsilon= \Id$. In this case the result follows from
 \begin{lm}\label{lm:vanishingofcohomology} We have $H^1_{\textup{\'et}}(G_0\times \Spec \overline {\bF}_p, \mu_2)=0$.
 \end{lm}
  \begin{proof} By nil-invariance of \'etale cohomology, we can replace $G_0$ by the reduced subgroup $(G_0)_{\mr{red}}$. According to \cite[Lemma 4.1]{bv},  $(G_0)_{\mr{red}}$ is an extension of  $\Sp(2n)$ by a (connected) unipotent group. By \cite{fp}, $H^1_{\textup{\'et}}(\Sp(2n)\times \Spec \overline {\bF}_p, \mu_2)\simeq H^1_{\mr{sing}}(\Sp(2n)(\mbb C), \mbb F_2)$ (recall that $p\ne 2$); the latter group is 0 since the topological group $\Sp(2n)(\mbb C)$ is simply connected. 
\end{proof}
\noindent Indeed, this gives a section $s\colon G_0 \rightarrow \widehat G_{{\mathbb P^1}}\times_{{\mathbb P^1}}\{0\}$ of the corresponding $\mu_2$-torsor over $\overline{\mbb F}_p$. Changing the section by an element of $\mu_2(G_0)\simeq \mbb Z/2$ so that $s(1)=1$, one sees from connectedness of $(G_0)_{\overline{\mbb F}_p}$ that $s$ is in fact a group splitting. The uniqueness of the group splitting then guarantees that it descends to $\mbb F_p$.
\end{proof}
  The section in (\ref{app.section}) identifies the $G_0$-action on $A_0\lmod$ with the one induced by the tautological action of $G_0$ on $A_0$. 

\subsection{Restriction of $\cA_ {{\mathbb P^1}}\lmod^{G_0}$ to $({\mathbb P^1} \backslash \{0\})/G_0$} \label{ssec:action outside of 0}
In this subsection we will see that the action of $G_0$ on  $\cA_{{{\mathbb P^1}}\backslash \{0\}} \lmod$ 
lifts to an action of this group on the Morita equivalent algebra $\cA_{{{\mathbb P^1}}\backslash \{0\}}^\flat$.

Recall from Construction \ref{constr:G_0action on the cental reduction}  the action 
$\psi\colon G_0\times {\mathbb P^1}\to \underline{\Aut} (\cA_{{\mathbb P^1}}^\flat)$. 
Since  $\cA_{{\mathbb P^1}} $ and $\cA_{{\mathbb P^1}}^\flat $  restricted to ${{\mathbb P^1}}\backslash \{0\}$ are split Azumaya algebras (see Remarks \ref{rem:A_S is matrix algebra outside of 0} and \ref{rem:splitting of bemol algebra}) their categories of modules are equivalent: 
 $$\cA_{{{\mathbb P^1}}\backslash \{0\}}\lmod \iso  \cO_{{{\mathbb P^1}}\backslash \{0\}}\lmod  \iso \cA_{{{\mathbb P^1}}\backslash \{0\}}^\flat \lmod .$$
 Thus $\psi_{|{{\mathbb P^1}}\backslash \{0\}}$ defines an action of $G_0$ on $\cA_{{{\mathbb P^1}}\backslash \{0\}}\lmod$.
  \begin{Th}\label{th:restriction_to_P-0}
The restriction to ${\mathbb P^1} \backslash \{0\}$ of the action of $G_0$ on  $\cA_{{{\mathbb P^1}}}\lmod$ from Theorem \ref{appendix:maintheorem}      is isomorphic to the one defined by $\psi_{|{{\mathbb P^1}}\backslash \{0\}}$.
  \end{Th}
  \begin{proof}
  Let $\ul{A}^\flat_{{\mathbb P^1}\backslash\{0\}}$ (resp. $\ul{A}^{\flat,\times}_{{\mathbb P^1}\backslash\{0\}}$) denote a functor on $\mbb F_p$-algebras sending $R$ to $\ul{\cA}^\flat_{{{\mathbb P^1}}\backslash \{0\}}\big(\Spec R\times ({{\mathbb P^1}}\backslash \{0\})\big)$ (resp. $\ul{\cA}^{\flat,\times}_{{{\mathbb P^1}}\backslash \{0\}}\big(\Spec R\times ({{\mathbb P^1}}\backslash \{0\})\big)$).

  Following Example \ref{ex:action on A-mod induced by action on A}, and applying Weil restriction, the assertion of the theorem amounts to the existence of a commutative diagram 
     \begin{equation}\label{eq:twoactions}
	\begin{tikzcd}
   1 \arrow[r, ""] & 
   \underline{\mr{Mor}_*}({\mathbb P^1}\!\setminus \!\{0\},\mbb G_m)  \arrow[r, "-b"]\arrow[d, hook, ""'] & (\widetilde G_0)^\circ   \arrow[d, "i"'] \arrow[r, ""]&\arrow[d, "\id"'] G_0 \arrow[r, ""] &1 \\  
   1 \arrow[r, ""] & \underline{A}^{\flat,\times}_{{{\mathbb P^1}}\backslash \{0\}} \arrow[r, ""]\arrow[dr, "Ad"']&   G_0 \ltimes \underline{A}^{\flat,\times}_{{{\mathbb P^1}}\backslash \{0\}}  \arrow[r, ""]\arrow[d, "\alpha^\flat"]& G_0 \arrow[r, ""] &1  \\
	&& \underline{\Aut} (\ul{A}^\flat_{{\mathbb P^1}\backslash\{0\}} ),  &  &
	\end{tikzcd}
	\end{equation}  
where $\underline{\Aut} (\ul{A}^\flat_{{\mathbb P^1}\backslash\{0\}} )$ denotes a functor on $\mbb F_p$-algebras sending $R$ to the group $\Aut_{R[h^{-1}]}(\ul{A}^\flat_{{\mathbb P^1}\backslash\{0\}}(R))$ of $R[h^{-1}]$-linear algebra automorphisms.
Here the top row, which encodes the action of $G_0$ on  $\cA_{{{\mathbb P^1}}\backslash \{0\}}\lmod \iso  \cO_{{{\mathbb P^1}}\backslash \{0\}}\lmod$,  is copied from formula \eqref{eq:actiononthecomplementtotheformaldisk2}. The two bottom rows encode the action of 
$G_0$ on $\cA_{{{\mathbb P^1}}\backslash \{0\}}^\flat \lmod$ induced by its action $\psi$ on $\cA_{{{\mathbb P^1}}\backslash \{0\}}^\flat$.
The  composition $\alpha^\flat \circ i$ is required to be the trivial homomorphism.


To produce the above diagram we shall use the following result borrowed from \cite[Proposition 3.3.]{bv}:
there is a unique homomorphism $\tilde \psi_0$ making the diagram below commutative.
   \begin{equation}\label{res:bv}
\def\normalbaselines{\baselineskip20pt
\lineskip3pt  \lineskiplimit3pt}
\def\mapright#1{\smash{
\mathop{\to}\limits^{#1}}}
\def\mapdown#1{\Big\downarrow\rlap
{$\vcenter{\hbox{$\scriptstyle#1$}}$}}
\begin{matrix}
 1 &  \to   &   \Gra_{\mathbb{G}_m} & \to  & \widetilde{G}_0 & \to &  G_0 & \to & 1\cr
&& \mapdown{Id}  &  &\mapdown{\tilde\psi_0}  &&  \mapdown{\psi_0} && \cr
1 &  \to   &   \Gra_{\mathbb{G}_m} & \to  & \widetilde G_0^\flat & \to &  G_0^\flat & \to & 1
\end{matrix}
 \end{equation}
Here the bottom sequence is the analogue of the top one, where $\cA_{{\mathbb P^1}}$ is replaced by $\cA_{{\mathbb P^1}}^\flat$. The map $\psi_0$ is the composition
$G_0 \rar{\psi}  \underline{\Aut} (A_h^\flat)\to G_0^\flat$.
In particular, $\widetilde{G}_0\iso \widetilde G_0^\flat \times _{ G_0^\flat, \psi_0} G_0$.

Recall (cf. \eqref{fundcentralextensionbisbis}) that $\widetilde G_0^\flat$ is defined to be the  quotient of $\widetilde G^\flat\subset \underline{A_h^{\flat}[h^{-1}]^\times}$ 
\begin{equation}\label{fundcentralextensionbisbisbis}
1 \to L\bG_m \to \widetilde G^\flat  \to   G^\flat \to 1, \qquad \widetilde G^\flat =L\!\GL(p^{2n})\times_{L\!\PGL(p^{2n})} G^\flat
\end{equation} 
by the subgroup $\underline{A_h^{\flat,\times}}$. Thus, $\widetilde G_0^\flat \times _{ G_0^\flat, \psi_0} G_0$ is the quotient of $\tilde G^\flat \times _{ G_0^\flat, \psi_0} G_0$ by the same subgroup. Let us construct a section 
$\widetilde G_0^\flat \times _{ G_0^\flat, \psi_0} G_0 \mono \widetilde G^\flat \times _{ G_0^\flat, \psi_0} G_0$ of the quotient map.
Consider the subgroup $H':=\widetilde G^\flat \times _{ G^\flat, \psi} G_0  \subset \widetilde G^\flat \times _{ G_0^\flat, \psi_0} G_0$,
and let $H\subset H'$ be the fiber product $H'\times _{\underline{A_h^{\flat}[h^{-1}]^\times}} \underline{A}^{\flat,\times}_{{{\mathbb P^1}}\backslash \{0\}}$.
Thus, an $R$-valued point of $H$ is a pair $(g, a)$, where
$a\in \underline{A}^{\flat,\times}_{{{\mathbb P^1}}\backslash \{0\}}(R)$, $g\in G_0(R)$, such that $\Ad_a= \psi (g)$.
The action $\psi_\infty$ on  the fiber $A_{\infty}^\flat$ of $\cA^\flat_{{\mathbb P^1}}$ at $h=\infty$ lifts along the homomorphism
$\ul{A}^{\flat, \times}_{\infty} \rar{Ad}\underline{\Aut}(A_\infty^\flat)$: under the identification 
$A_\infty^\flat= \End_{\bF_p}(A_0)$ the lift is given by the action $\tilde \psi_\infty\colon G_0 \mono \GL(A_0) \iso \underline{A}^{\flat, \times}_{\infty}$ of $G_0$ on $A_0$. Let $H_* \subset H$ be the subgroup formed by $(g,a)\in H$ such that $\ev_{\infty} (a) \in \underline{A}^{\flat, \times}_{\infty}$ is equal to 
$\tilde \psi_\infty(g)$. Write $\Gra_{\mathbb{G}_m}^\circ$ for the neutral connected component of the  affine Grassmannian.
Thus, we have $\Gra_{\mathbb{G}_m}\iso \Gra_{\mathbb{G}_m}^\circ \times \underline{\bZ}$. 
\begin{lm}\label{lm:easylm} The composite\footnote{The map from $\underline{\bZ} $ to the fiber product is defined to be the composition  $\underline{\bZ} \mono L\bG_m\mono  \widetilde G^\flat \times _{ G_0^\flat, \psi_0} G_0 $.} $H_*  \times \underline{\bZ}\mono \widetilde G^\flat \times _{ G_0^\flat, \psi_0} G_0 \epi \widetilde G_0^\flat \times _{ G_0^\flat, \psi_0} G_0$
is an isomorphism. 
    \end{lm} 
\begin{proof}
The kernel of the map $H_* \to G_0$ is $ \Gra_{\mathbb{G}_m}^\circ = \underline{\mr{Mor}_*}(\bP^1\backslash \{0\}, \bG_m)\subset \underline{\mr{Mor}}(\bP^1\backslash \{0\}, \bG_m)$. We have the commutative diagram
 $$
\begin{tikzcd}
1 \arrow[r]
  & \operatorname{Gr}_{\mathbb G_m} \arrow[r] \arrow[d, equal]
  & H_* \times \underline{\bZ}\arrow[r] \arrow[d, ""]
  & G_0 \arrow[d, equal] \\
1 \arrow[r]
  & \operatorname{Gr}_{\mathbb G_m} \arrow[r]
  & \widetilde G_0^\flat
      \times_{G_0^\flat,\,\psi_0} G_0 \arrow[r]
  & G_0 \arrow[r]
  & 1
\end{tikzcd}
 $$
 Thus it suffices to prove the surjectivity of the map $H_* \to G_0$.  The embedding $\bG_m \mono \underline{\mr{Mor}}(\bP^1\backslash \{0\}, \bG_m)$ determines an isomorphism $\bG_m \times H_* \iso H$ such that the map $H\to G_0$ is trivial on the first factor.
Thus, it suffices to show the surjectivity of 
 $H\to G_0$. Concretely, it is enough to check that, for any $R$-point $g$ of $G_0$, {\it fppf} locally on $\Spec R$ the automorphism $\psi(g)$ of a matrix algebra over $R[h^{-1}]$, 
$$A^{\flat}_{{{\mathbb P^1}}\backslash \{0\}}(R[h^{-1}])\iso R[h^{-1}]\otimes \End_{\bF_p}(A_0),$$
is inner. In fact, it suffices to check this in the universal case $R=\cO(G_0)$, $G_0\rar{g=\Id} G_0$. The automorphism $\psi (g)$ can be regarded as a map 
\begin{equation}\label{decompleqbis}
 \Psi\colon G_0 \times \bA^1\to \PGL(A_0).
 \end{equation}
It is enough to verify  that there exists an open cover, $G_0=\cup _i U_i$, such that, for every $i$, the map $\Psi|_{U_i\times \bA^1}$ can be lifted to a map $U_i \times \bA^1\to \GL(A_0)$.
 Since $G_0$ is a group scheme of finite type over a perfect field, the reduced subscheme
 $G_{0, \mathrm{red}}\subset G_0$ is smooth. It follows that, for every affine open subset $U\subset G_0$, we have that 
$$ \Pic (U \times \Spec \bF_p[h^{-1}] ) \iso  \Pic (U_{\mathrm{red}} \times \Spec \bF_p[h^{-1}] )\iso   \Pic (U_{\mathrm{red}}).$$
Thus, if $G_0=\cup _i U_i$ is an open cover, such that $\Psi|_{U_i\times \{0\}}$ lifts to a map $U_i \times \{0\} \to \GL(A_0)$, then the same is true  for $\Psi|_{U_i\times \bA^1}$.
\end{proof}
 We can now  finish the proof of the theorem by producing  diagram  \eqref{eq:twoactions}. Define  map $i$
 to be the composition
 $$(\widetilde G_0)^\circ \mono \widetilde G_0  \iso \widetilde G_0^\flat \times _{ G_0^\flat, \psi_0} G_0 \iso H_*  \times \underline{\bZ} \rar{\delta} 
 G_0 \ltimes \underline{A}^{\flat,\times}_{{{\mathbb P^1}}\backslash \{0\}}.$$
 Here the first arrow comes from the factorization $(\widetilde G_0)^\circ \times \underline{\bZ}\iso \widetilde G_0$, 
 the second map is given by diagram \eqref{res:bv}, the third one is taken from Lemma \ref{lm:easylm},
 and $\delta$ is given by the formula $\delta(g,a, m)= (g, a^{-1})$ (see the description of $H\supset H_*$  preceding the statement of Lemma \ref{lm:easylm}, where the notation $(g,a)\in H$ is explained). 
 \end{proof}

\begin{constr}\label{const:mult_structure_on_the_gerbe_over_BG_0xG_a}
    An action of $G_0$ on $\cO_{{{\mathbb P^1}}\backslash \{0\}}\lmod$ can be viewed  as a 
$\bG_m$-gerbe over ${\mathbb P^1}\backslash \{0\} \times B G_0$ together with a trivialization 
of its pullback along the map ${\mathbb P^1}\backslash \{0\} \to {\mathbb P^1}\backslash \{0\} \times B G_0$ classifying the trivial torsor, or, equivalently, a homomorphism of group stacks over ${\mathbb P^1}\backslash \{0\}$:
\begin{equation}\label{eq:g_m-gerbe_related_to_action_b}
    {\mathbb P^1}\backslash \{0\} \times G_0 \to {\mathbb P^1}\backslash \{0\} \times B \bG_{m}.
\end{equation}
Identify ${\mathbb P^1}\backslash \{0\}$ with the scheme underlying $\bG_a$ and set $t=h^{-1}$. Denote by 
\begin{equation}\label{eq:g_m-gerbe_related_to_action}
  [\cA^\flat_{\bG_a}] \to \bG_a \times B G_0  
\end{equation}
 the $\bG_m$-gerbe corresponding to the action $\psi_{|\bG_a}$  of $G_0$ on $\cA_{\bG_a}^\flat\lmod\iso \cO_{\bG_a}\lmod$.
In Construction \ref{const:multiplicative_gerbe} we introduced, for any pseudo-smooth $X$ over $S$ and a class $[\eta]$, a multiplicative structure on  the $\bG_m$-gerbe
\([\mathcal{D}_{X, t[\eta], \bG_a}]\to X^\lp\times \bG_a \)
    equivariant with respect to automorphisms of $X/S$ preserving $[\eta]$. Applying this to $X= \Spec A_0 \times S$, $[\eta]=[\eta_{\mathsf{can}}]$ and varying $S$, we obtain a   multiplicative structure on $ [\cA^\flat_{\bG_a}]$ {\it i.e.}, a lift of
    \eqref{eq:g_m-gerbe_related_to_action} to a homomorphism of strictly commutative group stacks over $BG_0$, whose fiber is $B\bG_m$. Equivalently, \eqref{eq:g_m-gerbe_related_to_action_b} lifts to homomorphism
    \begin{equation}\label{eq:g_m-gerbe_related_to_action_mult}
  \rho\colon G_0 \to \underline{\Ext} (\bG_a, \bG_{m}), 
\end{equation}
   where $\underline{\Ext} (\bG_a, \bG_{m})$ denotes the group stack whose $S$-points are the Picard groupoid of extensions of 
   $\bG_{a, S}$ by $\bG_{m, S}$.
\end{constr}
For a test $\bF_p$-scheme $S$ denote $\cQ(S)$ the group consisting of pairs $(g, \nu_g)$, where $g\in G_0(S)$ and $\nu_g$ is a splitting of the corresponding extension of $\bG_{a, S}$ by $\bG_{m, S}$. The kernel of the projection 
 $\cQ(S)\to G_0(S)$ is  $ \bG_{a}^{\hat\sharp}(S):= \Hom(\bG_{a, S}, \bG_{m, S})$. Set $\hat W_{\mathrm {big}}= \underline{\mr{Mor}_*}({\mathbb P^1}\!\setminus \!\{0\},\mbb G_m)$.
   \begin{pr}\label{pr:rediction_og_BV_to_hatW} 
   The extension
       \eqref{eq:actiononthecomplementtotheformaldisk2} \( 1 \to \hat W_{\mathrm {big}}\to (\widetilde G_0)^\circ \to G_0 \to 1 \)
        admits a unique reduction to $\bG_{a}^{\hat\sharp} \subset \hat W_{\mathrm {big}}$ given by
        \begin{equation}
            1 \to \bG_{a}^{\hat\sharp} \to \cQ \to G_0 \to 1.
        \end{equation}
        Moreover, the uniqueness assertion holds for the restriction of the extension to any subgroup $H\subset G_0$.
  \end{pr}
\begin{proof}
    For the existence assertion it suffices to check that, for any $\bF_p$-scheme $S$ and 
    $g\in G_0(S)$ there exists an {\it fpqc} cover $T \to S$ such that the pullback of $\rho(g)$ to $T$ splits.  In fact, by  \cite[Cor. 3.4, Cor. 3.10, Prop. 2.30]{bmv} any commutative 
      group extension of $\bG_{a, S}$ by $ \bG_{m, S}$ splits  {\it fppf} locally on $S$ ({\it cf.} proof of Proposition \ref{pr:miltiplicative_property}). For the uniqueness part we 
      shall use the following.
      \begin{lm}
    There is an isomorphism of group fppf sheaves over $\bF_p$
    \[\hat W_{\mathrm {big}} /\bG_{a}^{\hat\sharp} \iso  \hat W_{\mathrm {big}}.\]
    \end{lm}
      \begin{proof}
        The group $\hat W_{\mathrm {big}}$ is a module over $W_{\mathrm {big}}$.
        The standard $p$-typical decomposition  of $W_{\mathrm {big}}$ yields 
        a direct \emph{sum} decomposition of  $\hat W_{\mathrm {big}}$:
        \[\hat W_{\mathrm {big}}\simeq  \bigoplus_{\substack{m\geq1\\p\nmid m}} \hat W, \]
        where the first summand  is \[\bigcap _{(m, p)=1, m>1} \ker F_m.\] The map 
        $\bG_{a}^{\hat\sharp} \mono \hat W_{\mathrm {big}}$ lands in this summand and its image is identified with the Frobenius kernel on $\hat W$.  Now the lemma follows from the surjectivity of $F: \hat W \to \hat W$.
      \end{proof}
By \cite[Cor. 6.3]{bv} any homomorphism from an affine group scheme over a field to $\hat W_{\mathrm {big}}$ is trivial. In particular, for any subgroup $H\subset G_0$,
$\Hom(H, \hat W_{\mathrm {big}}/ \bG_{a}^{\hat\sharp})\simeq \Hom(H, \hat W_{\mathrm {big}})=0$.
This proves the uniqueness.
      \end{proof}
\subsection{The dual of $\cA_{{\mathbb P^1}}\lmod^{G_0}$}\label{ssec:dual action}

Denote by $(\cA_{{\mathbb P^1}}\lmod^{G_0})^\vee$ the quasi-coherent sheaf of categories on $BG_0\times {\mathbb P^1}$ that is dual to $\cA_{{\mathbb P^1}}\lmod^{G_0}$ in the sense of Construction \ref{constr:dualizable sheaves of categories}; note that $\cA_{{\mathbb P^1}}\lmod^{G_0}$ is a particular case of Example \ref{ex:dualizable categories}\eqref{item:locally given by modules is dualizable}, and thus is dualizable. Further, it will be useful to have its description in terms of the original category $\cA_{{\mathbb P^1}}\lmod^{G_0}$.

Consider the isomorphism $\beta\colon \cA_{{\mathbb P^1}}^{\op}\iso \cA_{{\mathbb P^1}}$ that preserves all $x_i$ and sends all $y_i$ to $-y_i$; denote by $\beta_0\colon A_0 \iso A_0$ the fiber of $\beta$ at $\{0\}$. Isomorphism $\beta_0$ sends the restricted structure $[\eta_{\mr{can}}]$ to $[-\eta_{\mr{can}}]$ and thus normalizes  $G_0:= \ul{\mr{Aut}}(A_0,[\eta_{\mathsf{can}}])$, consequently defining an involutive automorphism of $G_0$ which we denote $\beta_0$. It further gives an automorphism $\beta_0\times \id\colon BG_0 \times {\mathbb P^1} \iso BG_0 \times {\mathbb P^1}$.

\begin{pr}\label{prop: dual action on A_S-mod}
    One has a natural equivalence
    $$
(\cA_{{\mathbb P^1}}\lmod^{G_0})^\vee\simeq (\beta_0\times \id)^*\cA_{{\mathbb P^1}}\lmod^{G_0}
    $$
     of sheaves of categories on $BG_0\times {\mathbb P^1}$.
\end{pr}
\begin{proof}
    The pull-back $(\cA_{{\mathbb P^1}}\lmod^{G_0})^\vee$ along ${\mathbb P^1} \ra BG_0\times {\mathbb P^1}$ is the dual category to $\mc A_{\mathbb P^1}\lmod$. The latter is given by $\mc A_{{\mathbb P^1}}^{\op}\lmod$, and so $(\cA_{{\mathbb P^1}}\lmod^{G_0})^\vee$ is defined by a certain action of $G_0$ on $\cA_{\mathbb P^1}^{\op}\lmod$. Now, twist it by the isomorphism $\beta\colon \cA_{{\mathbb P^1}}^{\op}\iso \cA_{{\mathbb P^1}}$. The result is an action of $G_0$ on $\cA_{{\mathbb P^1}}\lmod$ which satisfies Theorem \ref{appendix:maintheorem} for the homomorphism $\beta_0\colon G_0 \ra G_0$.
\end{proof}
\begin{rem}\label{rem:another formula for the dual}
The argument from   \S \ref{subsectionrestrictionto0}, \eqref{diag.f.z2}, also produces an equivalence   
$$
(\cA_{{\mathbb P^1}}\lmod^{G_0})^\vee\simeq (\id\times \iota)^*\cA_{{\mathbb P^1}}\lmod^{G_0}
    $$
     of sheaves of categories on $BG_0\times {\mathbb P^1}$.
\end{rem}
 \subsection{The action of $\Sp(2n) \subset G_0$ on $\cA_{{\mathbb P^1}}\lmod$.}\label{subs:action_of_symplectic_group} 
We shall describe explicitly the restriction of the $G_0$-action on  $\cA_{{\mathbb P^1}}\lmod$ to the subgroup 
$\Sp(2n) \subset G_0$. First we observe that, for any open subset $U\subset {\mathbb P^1}\backslash\{0\}$, any action of 
$\Sp(2n)$ on $\cA_{U}\lmod$ admits a unique trivialization. In fact, since, for $U$ as above, $\cA_{U}\lmod\cong \cO_{U}\lmod$, an action of a group scheme on $\cA_{U}\lmod$ amounts to specifying an extension of the group scheme by $\bG_m$. Since every extension of $\Sp(2n)$ by $\bG_m$ has a unique splitting the claim follows. Secondly, observe that the action of $\Sp(2n)$ on $A_0$ lifts to an action  on 
$\cA_{{\mathbb A^1}}$. This is evident from the description of $A_{\mathbb A^1} :=\Gamma(\mathbb A^1, \cA_{{\mathbb A^1}})$
 as the quotient of the free associative $\bF_p[h]$-algebra on a symplectic $2n$-dimensional vector space $(V, \omega)$ modulo the relations: $v\cdot w - w\cdot v = h\,  \omega(v,w)$, $v^p=0$, for all $v,w\in V$. Using the first observation,
 the induced action of $\Sp(2n)$ on $\cA_{{\mathbb A^1}}\lmod$ extends uniquely to
 an $\Sp(2n)$-action on $\cA_{{\mathbb P^1}}\lmod$. The  uniqueness claim in  Theorem \ref{appendix:maintheorem} implies that this is the restriction of the $G_0$-action from the same theorem. 
 
 Also, note that the action of the maximal parabolic subgroup $P\subset \Sp(2n)$, which consists of all $f\in \Sp(2n)$ such that $f(\Span\{ x_1, \cdots, x_n\})=\Span\{ x_1, \cdots, x_n\}$, on $\cA_{{\mathbb A^1}}$
 extends to an action on $\cA_{{\mathbb P^1}}$. To see this explicitly, put $t=h^{-1}$ and use the generators $x_i'=x_i$ and $y_i'=h^{-1}y_i$ for
$\mc A_{\mathbb P^1\setminus\{0\}}$. For $g\in P$ we can write $
g(x)=Ax,$ $ g(y)=Cx+A^{-T}y.$
Consequently, $
g(x')=Ax',
$ $
g(y')=tCx'+A^{-T}y',
$
so $g$ preserves
$A_{\mathbb P^1\setminus\{0\}}$. By uniqueness, the action of $P$ on $\cA_{{\mathbb P^1}}\lmod$ is induced by the above action on the algebra.
 \subsection{A $\bG_m \ltimes G_0$-equivariant structure on $\cA_{{\mathbb P^1}}\lmod$.}\label{subs.addingmultiplicativegroup} 
 The next proposition shows that the action of $G_0$ on the quasi-coherent sheaf of categories  $\cA_{{\mathbb P^1}}\lmod$ constructed in Theorem \ref{appendix:maintheorem} extends to  an equivariant structure with respect to a larger group acting on ${\mathbb P^1}$.  
  Define an action $\lambda\colon A_0 \to A_0 \otimes \cO(\bG_m)$ of $\bG_m$ on $A_0$ by the  formulae 
  \begin{equation}\label{app.g_mactionona_s}
 \lambda(x_i)=x_i, \quad \lambda(y_i) = z y_i, \quad 1\leq i\leq n.
  \end{equation}
  Here $z$ denotes the coordinate on $\bG_m$.
 We have that $\lambda ([\eta])= z  [\eta]$. Hence, the subgroup $\bG_m \subset  \ul{\Aut} (A_0)$ normalizes $G_0$.  Denote by 
  \begin{equation}\label{app.semidirect}
 \bG_m \ltimes G_0:=\bG_m \ltimes_{\Ad_\lambda} G_0    \subset \Aut (A_0)
  \end{equation}
  the subgroup generated by $G_0$ and $\bG_m$. 
 Endow $\mathbb P^1$ with the $\mbb G_m$-action $\chi$
by homotheties, so that
$
\chi_z^*(h):=z^{-1}h.
$ The projection  to the first factor defines an action of $\bG_m \ltimes G_0$ on ${{\mathbb P^1}}$. 
 Formulae (\ref{app.g_mactionona_s}) yield a  $\bG_m$-equivariant structure on sheaf of algebras $\cA_{{\mathbb P^1}}$. We denote by $(\cA_{{\mathbb P^1}}\lmod)_{{{\mathbb P^1}}/\bG_m}$ the corresponding quasi-coherent sheaf of categories on the stack ${{\mathbb P^1}}/\bG_m$. Since the $\bG_m$-action on the fiber of $\cA_{{\mathbb P^1}}$ over the fixed point $\{\infty\}\in {{\mathbb P^1}}$ is trivial, we have that 
 the restriction of $(\cA_{{\mathbb P^1}}\lmod)_{{{\mathbb P^1}}/\bG_m}$ to $ \{\infty\}/\bG_m$ is equivalent to $ \cO_{\{\infty\}/ \bG_m}\lmod$.
 Finally, let $q\colon {{\mathbb P^1}}/\bG_m \to {{\mathbb P^1}}/(\bG_m \ltimes G_0 )$ be the morphism of stacks induced by the inclusion $\bG_m \mono \bG_m \ltimes G_0 $.  
  \begin{pr}\label{app.G_mequivarinceprop}
   There exists a unique (up to a unique equivalence)  quasi-coherent sheaf of categories   $(\cA_{{\mathbb P^1}}\lmod)_{{{\mathbb P^1}}/(\bG_m \ltimes G_0) }$ on the stack ${{\mathbb P^1}}/(\bG_m \ltimes G_0)$ equipped with equivalences
   $$\Theta\colon q^*((\cA_{{\mathbb P^1}}\lmod)_{{{\mathbb P^1}}/(\bG_m \ltimes G_0) }) \iso (\cA_{{\mathbb P^1}}\lmod)_{{{\mathbb P^1}}/\bG_m},$$
   $$\Xi_\infty\colon ((\cA_{{\mathbb P^1}}\lmod)_{{{\mathbb P^1}}/(\bG_m \ltimes G_0) })_{|\{\infty\}/(\bG_m \ltimes G_0 )}\iso \cO_{\{\infty\}/(\bG_m \ltimes G_0)}\lmod,$$
   and an isomorphism  $\varpi $ between the pullbacks  of $\Xi_\infty$ and $\Theta$ to $ \{\infty\}/\bG_m$, such that the action of   $\bG_m \ltimes G_0$ on the center of $A_0\lmod$ induced by $\Theta$ comes from the embedding  (\ref{app.semidirect}).
   Moreover, the $\bG_m \ltimes G_0$-action on $A_0\lmod$  is isomorphic to the one given by (\ref{app.semidirect}).
   \end{pr}
   \begin{rem}\label{rephrasingg_m}
  Consider the diagram corresponding to the $G_0$-action on $\cA_{{\mathbb P^1}}\lmod$ (see Remark \ref{enum: every autoequivalence is locally an automorphism} in the proof of Theorem \ref{appendix:maintheorem}).
  \begin{equation}\label{diag.f.extanyg_m}
	\begin{tikzcd}
	1 \arrow[r, ""] & \underline{\cA}_{{\mathbb P^1}}^\times  \arrow[r, ""]\arrow[dr, "Ad"]& \widehat G_{{\mathbb P^1}}  \arrow[r, ""]\arrow[d, "\alpha"]& G_0 \times {{\mathbb P^1}}\arrow[r, ""] &1  \\
	&& \underline{\Aut} (\cA_{{\mathbb P^1}}). &  &
	\end{tikzcd}
	\end{equation}
Morphisms $\lambda$, $\chi$ and $\Ad_\lambda$ determine a $\bG_m$-equivariant structure on all the group schemes displayed above except for $\widehat G_{{\mathbb P^1}}$. The proposition asserts  that there exists a unique $\bG_m$-equivariant structure on $\widehat G_{{\mathbb P^1}}$ that makes the whole diagram $\bG_m$-equivariant and such that 
the section $G_0 \times \{\infty\} \to \widehat G_{{\mathbb P^1}}   \times_{\mbb P^1} \!\{\infty\} $ commutes with the $\bG_m$-action.
\end{rem}
\begin{proof} We start with the existence assertion. Let $(\cA_{{\mathbb P^1}}\lmod)_{{{\mathbb P^1}}/G_0}$ be  the quasi-coherent sheaf of categories  on ${{\mathbb P^1}}/G_0$ constructed in Theorem \ref{appendix:maintheorem}. We construct the 
   quadruple  $((\cA_{{\mathbb P^1}}\lmod)_{{{\mathbb P^1}}/(\bG_m \ltimes G_0) }, \Theta, \Xi_\infty, \varpi)$ by specifying the descent data  for $(\cA_{{\mathbb P^1}}\lmod)_{{{\mathbb P^1}}/G_0}$ along the morphism $$\pro\colon {{\mathbb P^1}}/G_0 \to {{\mathbb P^1}}/(\bG_m \ltimes G_0).$$
    Let $\pro_{{{\mathbb P^1}}/G_0},  \chi\colon ({{\mathbb P^1}}/G_0) \times \bG_m \to   {{\mathbb P^1}}/G_0$ be the projection and the action morphisms respectively. We have to construct an equivalence
    \begin{equation}\label{app.descentagrumentg_m}
    \pro_{{{\mathbb P^1}}/G_0}^*((\cA_{{\mathbb P^1}}\lmod)_{{{\mathbb P^1}}/G_0})\iso   \chi^*((\cA_{{\mathbb P^1}}\lmod)_{{{\mathbb P^1}}/G_0})
    \end{equation}
    together with the cocycle datum. Note the $\bG_m$-equivariant structure on $\cA_{{\mathbb P^1}}$ identifies  the pullback of the quasi-coherent sheaves in question  to ${{\mathbb P^1}}\times \bG_m$  with $\cA_{{\mathbb P^1}}\lmod$ lifted along  $\pro_{{\mathbb P^1}}\colon {{\mathbb P^1}} \times \bG_m \to {{\mathbb P^1}}$.
    Thus, each of the corresponding quasi-coherent sheaves of categories determines an action of $G_0$ on $\pro_{{\mathbb P^1}}^*\cA_{{\mathbb P^1}}\lmod$, each of which comes with a trivialization over $\{\infty\} \times \bG_m$. Using Remark \ref{app.uniquenessoverabase} there exists a unique isomorphism between the two actions compatible with trivializations. This gives equivalence (\ref{app.descentagrumentg_m}). The cocycle datum for this equivalence is defined using the ``unique up to a unique isomorphism'' part of Remark \ref{app.uniquenessoverabase}.
    
    For the uniqueness assertion, observe that
    if $((\cA_{{\mathbb P^1}}\lmod)_{{{\mathbb P^1}}/(\bG_m \ltimes G_0) }, \Theta, \Xi_\infty, \varpi)$ is a quadruple from the Proposition,  then, by Theorem \ref{appendix:maintheorem}, $\pro^*((\cA_{{\mathbb P^1}}\lmod)_{{{\mathbb P^1}}/(\bG_m \ltimes G_0 )})$  is equivalent to $(\cA_{{\mathbb P^1}}\lmod)_{{{\mathbb P^1}}/G_0}$. 
   Hence $(\cA_{{\mathbb P^1}}\lmod)_{{{\mathbb P^1}}/(\bG_m \ltimes G_0 )}$ arises from a descent datum along  $\pro $ as above and  our assertion also follows from the ``unique up to a unique isomorphism'' part  of Remark \ref{app.uniquenessoverabase}.  
   
   As in \S \ref{subsectionrestrictionto0}, to prove the last assertion of the Proposition it suffices to construct a $\bG_m$-invariant section  $G_0 \times \{0\} \to \widehat G_{{\mathbb P^1}}   \times _{{\mathbb P^1}} \{0\} $. In fact, the uniqueness of $\widehat \tau$ in (\ref{diag.f.z2})  implies that the latter is $\bG_m$-equivariant. It follows that the same is true for section (\ref{app.section}).
   \end{proof}

   \begin{rem}
       Note that the action of $-1\in \mbb G_m$ on ${\mathbb P^1}/G_0\simeq BG_0\times {\mathbb P^1}$ is given by the involution $\beta_0\times \iota$ (in the notations of \S\ref{subsectionrestrictionto0} and \S\ref{ssec:dual action}). Thus the restriction of \eqref{app.descentagrumentg_m} to ${\mathbb P^1}/G_0 \times \{-1\}$ produces an equivalence 
       $$
\cA_{{\mathbb P^1}}\lmod^{G_0}\simeq (\beta_0\times \iota)^*\cA_{{\mathbb P^1}}\lmod^{G_0}
       $$
       and, since $\beta_0$ is an involution, also an equivalence
       $$
(\beta_0\times\id)^*\cA_{{\mathbb P^1}}\lmod^{G_0}\simeq (\id\times \iota)^*\cA_{{\mathbb P^1}}\lmod^{G_0}.
       $$
      The latter can be seen as the composition
      $$
(\beta_0\times \id)^*\cA_{{\mathbb P^1}}\lmod^{G_0}\simeq (\cA_{{\mathbb P^1}}\lmod^{G_0})^\vee\simeq (\id\times \iota)^*\cA_{{\mathbb P^1}}\lmod^{G_0}
      $$
       of the equivalences from Proposition \ref{prop: dual action on A_S-mod} and Remark \ref{rem:another formula for the dual}.
       \end{rem}
        \begin{rem}\label{rem:another_action_of_multiplicative_group}
        Consider another $\bG_m$-equivariant structure on sheaf of algebras $\cA_{{\mathbb P^1}}$ over $\mathbb P^1$ given by the formulas
        \[
x_i \mapsto z x_i, \quad y_i \mapsto  z y_i, \quad h\mapsto z^2 h, \quad 1\leq i\leq n.
        \]
        This induces an action $\mu$ of $\bG_m$ on $G_0$. We claim that $\cA_{{\mathbb P^1}}\lmod$ descends to 
        a quasi-coherent sheaf of categories on the quotient stack ${{\mathbb P^1}}/(\bG_m \ltimes_{\mu} G_0)$ with respect to the new action. This descent can be constructed using the strategy analogous to the one we used in Proposition  \ref{app.G_mequivarinceprop}. 
        
        Alternatively, let $\lambda$ denote the $\bG_m$-action used in
Proposition \ref{app.G_mequivarinceprop}, so that
$
\lambda(z)(x_i)=x_i,$ 
$\lambda(z)(y_i)=z y_i,$ 
$\lambda(z)(h)=zh.
$
Let $\rho\colon \bG_m\to P\subset \Sp(2n)$ be the one-parameter
subgroup given by
$
\rho(z)(x_i)=z x_i,$ 
$\rho(z)(y_i)=z^{-1}y_i.
$
The $\bG_m$-action $\mu$ displayed above is given by $
z\longmapsto \rho(z)\circ\lambda(z^2)\in \mathbb G_m\ltimes G_0$ and thus there is a homomorphism $\alpha\colon \mathbb G_m\ltimes_{\mu} G_0 \to \mathbb G_m\ltimes_{\lambda} G_0$, which is identity on $G_0$ and sends $(z,1)\mapsto (z^2, \rho(z))$. The desired descent for $\cA_{{\mathbb P^1}}\lmod$ is then obtained by pulling back the sheaf in Proposition  \ref{app.G_mequivarinceprop} via $\alpha$. Via \S\ref{subs:action_of_symplectic_group}, since $\rho$ lands in the subgroup $P$, the pull-back of $\cA_{\mathbb P^1}\lmod_{\mathbb G_m\ltimes_{\mu} G_0}$ to $\mathbb P^1/\mathbb G_m$ is exactly given by modules over the $\mathbb G_m$-equivariant sheaf of algebras $\cA_{\mathbb P^1}$ over $\mathbb P^1$ that we started with.
        \end{rem}
        
        \subsection{The restriction of the $G_0$-action on $\cA_{{\mathbb P^1}}\lmod$ to $\Spec \bF_p[h]/(h^2)\mono  {\mathbb P^1}$.}\label{subsec:the_action_modulo_the_square_of_h}
 Fix an integer $r\geq 0$. The restriction of the $G_0$-action on $\cA_{{\mathbb P^1}}\lmod$ to $\Spec \bF_p[h]/(h^{r+1})\mono  {\mathbb P^1}$ is described by a group scheme  extension
\begin{equation}\label{fundextruncoldbis}
  1 \to \big(\underline{A}_h/(h^{r+1})\big)^\times   \to  \widehat G_{r}  \to  G_0 \to 1, 
\end{equation} 
together with a homomorphism $ \widehat G_{r} \to \underline{\Aut}(\underline{A}_h/(h^{r+1}))$ (which, for $r\geq 1$, factors through
$\widehat G_{r-1}$); see formula \eqref{fundextruncold}. By Proposition \ref{pr:restrictionto0forgroups},
$\widehat G_{0}\cong G_0 \ltimes \underline{A}_0^\times$. Thus, for $r=1$, the above extension admits a reduction to an extension of $G_0$ by the (commutative) subgroup
$\underline{A}_0 \cong \ker \big( \big(\underline{A}_h/(h^{2})\big)^\times \to  \underline{A}_0^\times\big)$, as shown in the first row of the diagram below.
\begin{equation}\label{eq:extension-diagram}
\vcenter{\hbox{$
\xymatrix@C=2.5em@R=1.6em{
1 \ar[r] &
\underline{A}_0 \ar[r] &
H_1 \ar[r] &
G_0 \ar[r] &
1 \\
1 \ar[r] &
\underline{A}_0 \ar@{=}[u] \ar[d] \ar[r] &
H_2 \ar[u] \ar[d] \ar[r] &
G_0^0 \ar[u] \ar@{=}[d] \ar[r] &
1 \\
1 \ar[r] &
\mathbb{G}_a \ar[r] &
H_3 \ar[r] &
G_0^0 \ar[r] &
1
}
$}}
\end{equation}
Here $H_1$ is the preimage of $G_0$ under the projection $\widehat G_{1}\to \widehat G_{0}\cong G_0 \ltimes \underline{A}_0^\times$. Conversely, $\widehat G_{1}$ is the quotient of $H_1 \ltimes \big(\underline{A}_h/(h^{2})\big)^\times$ by the (normal) subgroup $\underline{A}_0$ embedded anti-diagonally. The second row is the pullback of the top extension along   $G^0_0\mono G_0$, see  Remark \ref{rem:groups H and G_0}. Finally, the bottom row is obtained as the pushout of the latter along   $\underline{A}_0 \to \bG_a$  (derived from the algebra map $A_0\to \bF_p$). 

Let
$$
\pi\colon BG_0^0\rightarrow BG_0
$$
be the morphism induced by the inclusion $G_0^0\hookrightarrow G_0$.
By Remark \ref{rem:groups H and G_0}, the natural action of $G_0$ on
$\Spec A_0$ is transitive with stabilizer $G_0^0$. Hence the isomorphism of $G_0$-schemes
$
G_0/G_0^0\simeq\Spec A_0
$
 and, consequently,
$
A_0
\simeq
\Gamma(G_0/G_0^0,\mc O)
\simeq
\operatorname{Ind}_{G_0^0}^{G_0}(\bF_p)
$
as $G_0$-representations, where $\operatorname{Ind}$ denotes the
right adjoint to restriction. Equivalently, one has $
\pi_*\bG_a\simeq\underline A_0.$ Since $\pi$ is affine, Shapiro's lemma gives a natural equivalence
$
R\Gamma(BG_0,\underline A_0)
\simeq
R\Gamma(BG_0^0,\bG_a),
$
and hence an isomorphism
\begin{equation}\label{eq:bijection of H^2}
H^2(BG_0,\underline A_0)
\xrightarrow{\sim}
H^2(BG_0^0,\bG_a).
\end{equation}
Explicitly, this isomorphism sends an extension of $G_0$ by
$\underline A_0$ to the extension obtained by first restricting it
to $G_0^0$ and then pushing it out along
$\ev_0\colon\underline A_0\to\bG_a$. Thus, under the
isomorphism \eqref{eq:bijection of H^2}, the class of the top row of
\eqref{eq:extension-diagram} is carried to the class of the bottom
row.

To describe the  extension of $G_0^0$ by $\bG_a$ appearing in the diagram, we shall use the $\bG_m$-action introduced in Remark
\ref{rem:another_action_of_multiplicative_group}, which makes the above extension $\bG_m$-equivariant. With respect to this action $\bG_a$ has weight  $2$: the map $\bG_m \times \bG_a\to \bG_a$ carries $(z,t)$ to $z^2t$. Given
a finite-dimensional $\bF_p$-vector space $V$ and $m\in \bZ$, denote by ${\bf V}\{m\}$ the corresponding vector group scheme endowed with the weight $m$ action of $\bG_m$. Writing $\fm\subset A_0$ for the maximal ideal, recall that the group $G_0^0$ has a finite filtration by normal subgroups $G^i_0\subset G_0^0$ such that 
\begin{equation}\label{eq:description_of_gr}
  G^i_0/G^{i+1}_0\iso {\bf \fm^{i+2}/\fm^{i+3}}\{i\}, \text{ for } i>0, 
\end{equation}
 and  $G^0_0/G^{1}_0\iso \Sp(2n)$
(with the trivial action of $\bG_m $), $G^0_0\cong \Sp(2n)\ltimes G^1_0$ (\cite[Proof of Lemma 4.1]{bv}).  In fact, for an $\bF_p$-algebra $R$, $G^i_0(R)$ consists of all $g\in G^0_0(R)\subset \Aut_R(A_0\otimes R)$,  which carry $\fm\otimes R$ to itself and equal $\Id$ modulo $\fm^{i+1}\otimes R$. Let us sketch a construction of  isomorphisms \eqref{eq:description_of_gr}. If $g\in G^i_0(R)$  
then $g - \Id\colon A_0\otimes R\to A_0\otimes R$ carries $\fm^{j}\otimes R$ to $\fm^{i+j}\otimes R$ and, thus, defines
a degree $i$ map $\theta\colon \oplus_{j\geq 0} \fm^{j}/\fm^{j+1} \otimes R\to \oplus_{j\geq 0} \fm^{j}/\fm^{j+1} \otimes R$. Viewing $\theta$ 
as a degree $i$ endomorphism
of the graded $R$-module $A_0\otimes R$,
one verifies that $\theta$ is a Hamiltonian derivation, $\theta=\{f, \cdot\}$, for a unique $f\in A_0 \otimes R$
of degree $i+2$. Map \eqref{eq:description_of_gr} carries $g$ to $f$.

Define a $\bG_m$-invariant $1$-cocycle:
\begin{equation}\label{eq:1-cocycle_of_G}
    c: G^0_0\to {\bf \fm^{3}/\fm^{4}}\{1\}
\end{equation}
to be the composition $G^0_0\to G^0_0/G^2_0\cong  \Sp(2n)\ltimes {\bf \fm^{3}/\fm^{4}}\{1\} \to {\bf \fm^{3}/\fm^{4}}\{1\},$
where the last map is the projection to the second factor. The cocycle condition means that $c(g_1 g_2)= c(g_1) + g_1 c(g_2)$,
where the action of $G^0_0$ on ${\bf \fm^{3}/\fm^{4}}\{1\}$ factors through $\Sp(2n)$. 
Next, we define a $\Sp(2n) \times \bG_m$-invariant bilinear map of schemes 
\begin{equation}\label{eq:paringonsym}
   {\bf B}\colon {\bf \fm^{3}/\fm^{4}}\{1\} \times {\bf \fm^{3}/\fm^{4}}\{1\} \to \mathbb{G}_a\{2\}
\end{equation}
as follows. Write $V:= \fm^{}/\fm^{2}$ and $\omega(\cdot, \cdot)$ for the symplectic form on $V$ (see \S \ref{subs:action_of_symplectic_group}). If $p>3$ the multiplication in $A_0$ defines  $\mathrm{Sym}^3V\iso \fm^{3}/\fm^{4}$;
for $p=3$, one has $\coker (V' \xrightarrow{v\mapsto v^{\otimes 3}} \mathrm{Sym}^3V)\iso \fm^{3}/\fm^{4}$, where $V'$ is the Frobenius twist of the $\Sp(V)$-representation  $V$ ({\it i.e.}, $V'=V$ as vector spaces over $\bF_p$ 
with the action of $\Sp(V)$ obtained by precomposing the standard action with the Frobenius on $\Sp(V)$). The symplectic form on $V$  determines an alternating bilinear form $B$ on  $\mathrm{Sym}^3V$:
\begin{equation}\label{eq:form_on_sym_three}
    B(v_1\otimes v_2\otimes v_3, u_1\otimes u_2\otimes u_3)= \sum_{\sigma \in S_3} \omega(v_1, u_{\sigma(1)}) \omega(v_2, u_{\sigma(2)})\omega(v_3, u_{\sigma(3)}) ,
\end{equation}
which descends to a form on $\coker (V' \to \mathrm{Sym}^3V)$
if $p=3$. This yields \eqref{eq:paringonsym}. Finally, define a $2$-cocycle:
\begin{equation}\label{eq:2-cocycle}
   {\bf B}\circ(c\cup c)\colon G^0_0\times G^0_0 \to  \mathbb{G}_a\{2\}, \quad  \big ({\bf B}\circ(c\cup c)\big)(g_1, g_2)= 
   {\bf B}(c(g_1), g_1c(g_2)).
   \end{equation}
\begin{pr}\label{pr:action_mod_h^2} The extension of $G_0^0$ by $\mathbb{G}_a$ corresponding to the cocycle\footnote{That is, the extension, whose underlying scheme is $G_0^0 \times \mathbb{G}_a$ and the group operation is given by $(g_1, t_1)(g_2,t_2) = \big(g_1g_2, t_1+t_2 + \big(\frac{1}{8}{\bf B}\circ(c\cup c)\big)(g_1,g_2)\big)$.}
$ \frac{1}{8}{\bf B}\circ(c\cup c)$ is isomorphic to the extension in the bottom row of
\eqref{eq:extension-diagram}. Moreover, if $p>3$, the Picard groupoid  of
$\bG_m$-equivariant central extensions
\begin{equation}\label{eq:central_ext_sequence}
  1\to \mathbb{G}_a\{2\} \to Q \to G_0^0 \to 1  
\end{equation}
is discrete and its $\pi_0$ is isomorphic to the $1$-dimensional vector space of $\Sp(2n)$-invariant alternating bilinear forms on  $\fm^{3}/\fm^{4}$, which is   spanned by $B$. 
\end{pr}
\begin{proof}  Assume, first, that $p>3$.  We shall start with the second assertion of the Proposition. In particular, we intend to  show that the isomorphism class of a $\bG_m$-equivariant central extension \eqref{eq:central_ext_sequence} is determined by the isomorphism class of the corresponding extension of the Lie algebras. Then we describe the Lie algebras appearing in \eqref{eq:extension-diagram}.

 Observe that  every $\bG_m$-equivariant central extension of a group scheme, endowed with the trivial action of $\bG_m$, by $\mathbb{G}_a\{2\}$ admits a unique $\bG_m$-invariant  splitting. This applies to the subgroup $\Sp(2n) \subset G^0_0$. It follows that the groupoid of
$\bG_m$-equivariant central extensions \eqref{eq:central_ext_sequence} is equivalent to the groupoid of
$\Sp(2n) \times \bG_m$-equivariant central extensions
\begin{equation}\label{eq:central_ext_sequence_bis}
  1\to \mathbb{G}_a\{2\} \to Q' \to G_0^1 \to 1.  
\end{equation}
Using \eqref{eq:description_of_gr} and a result of Demazure and Gabriel \cite[\rm{II}, \S 3, n.4]{gd}, which describes the Picard groupoid of extensions of a vector group by $\bG_a$ (see, also, \cite[\rm{I}, \S 4.22]{j}, for a modern exposition of this result and the computation of higher cohomology), every $\bG_m$-equivariant central extension \eqref{eq:central_ext_sequence_bis} admits a unique $\bG_m$-invariant splitting over $G_0^3$. 
Such a splitting is automatically $Q'$-invariant. It also follows from {\it loc. cit.} that \eqref{eq:central_ext_sequence_bis} admits  a $\bG_m$-invariant splitting over $G_0^2$; such splittings form a 
torsor under  $(\fm^{4}/\fm^{5})^\vee\iso (\mathrm{Sym}^4 V)^\vee$. By \cite[II, \S 2.17]{j}, $(\mathrm{Sym}^4 V)^\vee\simeq \mathrm{Sym}^4 V$ is an induced module, and so group cohomology of
$\Sp(2n)$ with coefficients in $(\mathrm{Sym}^4 V)^\vee$ vanishes. Using this vanishing in degrees $0$ and $1$, we conclude that
\eqref{eq:central_ext_sequence_bis} uniquely descends to a $\Sp(2n) \times \bG_m$-equivariant extension $\bar{Q'}$ of 
$G^1_0/G^2_0\cong  {\bf \fm^{3}/\fm^{4}}\{1\}$ by $\mathbb{G}_a\{2\}$.
 Using the Demazure-Gabriel result again, we infer that the Picard groupoid of $\bG_m$-equivariant extensions 
 \begin{equation}\label{eq:what_we_classify}
     1\to \mathbb{G}_a\{2\} \to \bar{Q'} \to {\bf \fm^{3}/\fm^{4}}\{1\} \to 1
 \end{equation}
  is discrete, and its $\pi_0$  is the vector space of alternating bilinear forms 
on $\fm^{3}/\fm^{4}$: given $\bar{Q'}$ as above the commutator map $\bar{Q'}\times \bar{Q'} \to \mathbb{G}_a\{2\}$, $(g_1, g_2)\mapsto
g_1g_2g_1^{-1}g_2^{-1}$, 
descends to a map ${\bf \fm^{3}/\fm^{4}}\{1\}\times {\bf \fm^{3}/\fm^{4}}\{1\} \to \mathbb{G}_a\{2\}$, which must
be induced from an alternating bilinear form $(\cdot, \cdot)_Q$ 
on $\fm^{3}/\fm^{4}$ by $\bG_m$-invariance. (In particular, any such extension is determined by the  corresponding extension of Lie algebras.)  Consequently,  
the Picard groupoid of $\Sp(2n) \times \bG_m$-equivariant extensions of $ {\bf \fm^{3}/\fm^{4}}\{1\}$ by  $\mathbb{G}_a\{2\}$ is as described 
in the Proposition. Moreover, since  the representation of $\Sp(V)$ on $\mathrm{Sym}^3 V$ is irreducible (see, for instance, \cite[Lemma 6.6]{bv}),
 it admits a unique, up to a scalar multiple, invariant bilinear form. 

It remains to verify the first assertion of the Proposition. Assuming that $p>3$ and using the remark above it suffices to compare the corresponding extensions of Lie algebras. The Lie algebra of $\widehat G$ is isomorphic to the quotient of $h^{-1} A_h$ (which is a Lie subalgebra of the associative algebra $A_h[h^{-1}]$) by the $1$-dimensional ideal $h^{-1}\bF_p$ and $\Lie G_0$
(which is the Lie algebra of Hamiltonian vector fields on $\Spec A_0$ by \cite[Lemma 3.3.]{bk})
is obtained by taking the further quotient by $A_h$. 
The quotient of $\Lie \widehat G$ by $h^2A_h$ has a natural grading given by
$\deg x_i=\deg y_i =1$, $1\leq i\leq n$, $\deg h=2$; elements of this quotient of positive grading degree form a Lie subalgebra $\mathfrak{i}$ and elements of degree greater than $2$ form a subideal $\mathfrak{i'}\subset \mathfrak{i}$. 
The associative algebra anti-involution  $\alpha: A_h^{op} \to A_h$, $\alpha(h)=-h$, $\alpha(x_i)=x_i$, $\alpha(y_i)=y_i$, $1\leq i\leq n$, 
descends to a Lie algebra  anti-involution  of the quotient $\frak{i}/\frak{i'}$,
which we also denote by $\alpha$. Set
\[\fg := \{ w\in \mathfrak{i}/\mathfrak{i'} \;|\; \alpha(w)=-w\}.\]
This is a Lie algebra that fits into the following diagram
\begin{equation}\label{diag:extension_of_Lie}
\xymatrix{
 &  h\bF_p \ar[d]&  \\
 & \fg \ar[d]&\\
h^{-1} \fm^{4}/\fm^{5} \ar@{-->}[ru]^s \ar[r] &  
\Lie G^1_0/G^3_0  \ar[r]& \Lie G^1_0/G^2_0
} 
\end{equation}
Here $s$ is the unique $\Sp(2n)$-invariant map of Lie algebras, given by the Weyl symmetrization \eqref{eq:weyl}:
\begin{equation}\label{pr:explicit_formula_for_s}
   h^{-1} \fm^{4}/\fm^{5}\stackrel{h}{\iso} \mathrm{Sym}^4 V \xrightarrow{\text{symmetrization}}  V^{\otimes 4}\xrightarrow{h^{-1}\circ \text{ product}} h^{-1} A_h.
\end{equation}
The alternating bilinear form $(\cdot, \cdot)_{H_3}$ on $ \Lie G^1_0/G^2_0 \cong   h^{-1} \fm^{3}/\fm^{4} \stackrel{h}{\iso} \mathrm{Sym}^3 V$
corresponding to the bottom row of
\eqref{eq:extension-diagram} is given by the formula
\begin{equation}\label{pr:explicit_formula_for_Lie}
    (f_1, f_2)_{H_3} = [\tilde f_1, \tilde f_2]_\fg -s([\hat{f}_1, \hat{f}_2]) \in h\bF_p.
\end{equation}
Here $\hat{f}_i$, $i=1,2$, are arbitrary lifts of $f_i$ to $\Lie G^1_0/G^3_0 $ and
$\tilde f_i$ are arbitrary lifts of $\hat{f}_i$ to $\fg$. The quotient of $\fg$ by 
the image of $s$ is the Lie algebra of the extension of \eqref{eq:what_we_classify} derived from \eqref{eq:extension-diagram}. 

Under the identification of the central kernel with $h\bF_p$, $t\mapsto ht$, the commutator form
 corresponding to the cocycle $ \frac{1}{8}{\bf B}\circ(c\cup c)$ is $ \frac{h}{4} B$. We need to verify that $(\cdot, \cdot)_{H_3}= \frac{h}{4} B $. We compute 
\[(x_1^3, y^3_1)_{H_3}= [\tfrac{x_1^3}{h}, \tfrac{y_1^3}{h}]_\fg - s(-9x^2_1y^2_1) = -\frac{3h}{2} = \tfrac{h}{4} B(x^3_1,y^3_1).\]
Since the space of $\Sp(2n)$-invariant forms on $\mathrm{Sym}^3 V$ is $1$-dimensional, this implies the result. 

Let us explain the required modifications for $p=3$. 
The Picard groupoid of
$\bG_m$-equivariant central extensions \eqref{eq:central_ext_sequence} is equivalent to the Picard groupoid  of
$\Sp(2n) \times \bG_m$-equivariant central extensions
\eqref{eq:central_ext_sequence_bis}, and, using the Demazure-Gabriel result, every such extension descends uniquely to a $\Sp(2n) \times \bG_m$-equivariant central extension of
$G^1_0/G_0^3$ by $\mathbb{G}_a\{2\}$.
Let $\cP_1$ (resp. $\cP_2$) 
be the Picard groupoid of $\Sp(2n) \times \bG_m$-equivariant central extensions of
$G^1_0/G_0^3$ (resp. $G^1_0/G_0^2$) by $\mathbb{G}_a\{2\}$. The map $G^1_0/G_0^3\to G^1_0/G_0^2$ induces 
\begin{equation}\label{eq:map-P_1toP_2}
    \cP_2 \to \cP_1.
\end{equation}
The groupoid $\cP_2$ is discrete and its $\pi_0$
is the vector space of $\Sp(2n)$-invariant alternating bilinear forms on  $\fm^{3}/\fm^{4}$. By the irreducibility of the representation of $\Sp(V)$
on $\fm^{3}/\fm^{4}\iso \coker (V' \to \mathrm{Sym}^3V)$  (\cite[Lemma 6.6]{bv}), the latter is a $1$-dimensional vector space spanned by $B$. 
Using this classification it also follows that $\cP_2$
is equivalent to the groupoid of $\Sp(2n) \times \bG_m$-equivariant central extensions of $\Lie G^1_0/G_0^2$ by $h\bF_p$. On the other hand, applying \cite[\rm{II}, \S 3, n.4]{gd} again, every
extension in $\cP_1$ admits a $\bG_m$-equivariant section over $G^2_0/G_0^3$, and set of such sections is a torsor under $\Hom(G^2_0/G_0^3, \mathbb{G}_a\{2\})^{\bG_m}\iso (\fm^{4}/\fm^{5})^\vee$ equipped with an action of $\Sp(2n)$. An extension in $\cP_1$ lies in the essential image $\cP_1'\subset \cP_1$ of \eqref{eq:map-P_1toP_2} if and only if the corresponding torsor has
an element fixed by $\Sp(2n)$\footnote{Here we use that the image of any section over $G^2_0/G_0^3$ is a normal subgroup. In fact, for any $\bG_m$-equivariant central extension of
$G^1_0/G_0^3$ by $\mathbb{G}_a\{2\}$ the preimage of $G^2_0/G_0^3$ is a central subgroup.}. Moreover, the functor \eqref{eq:map-P_1toP_2} exhibits $\cP_1'$ as the quotient of $\cP_2$ by an action of $H^0(\Sp(2n), (\fm^{4}/\fm^{5})^\vee)$. It is shown by Suprunenko and Zalesskii \cite[Proposition 2.2]{SuprunenkoZalesskii} that
$H^0(\Sp(2n), (\fm^{4}/\fm^{5})^\vee)$ is $0$ if $n>1$ and $\bF_p$ otherwise.
In the first case, the map  $\bF_p\iso \cP_2 \to \cP_1'$ is an equivalence. If $n=1$ then, using commutativity of  
$G_0^1/G_0^3$,  the action of $H^0(\Sp(2n), (\fm^{4}/\fm^{5})^\vee)$ on $\cP_2$ is trivial. Thus, $\cP_1'\iso \bF_p\times B\bF_p$.

We expect that $\cP_1' = \cP_1$ but won't prove it here\footnote{In fact, we expect that $H^1(\Sp(2n), (\fm^{4}/\fm^{5})^\vee)=0$ for all $p>2$.}; 
instead we shall prove that the extension, which arises from diagram  \eqref{eq:extension-diagram} is in $\cP_1'$. Since $\Hom(G^2_0/G_0^3, \mathbb{G}_a\{2\})^{\bG_m}\iso 
\Hom(\Lie G^2_0/G_0^3, h\bF_p)$, the $\Hom(G^2_0/G_0^3, \mathbb{G}_a\{2\})^{\bG_m}$-torsor in question is identified with the set of all sections $s$ in the diagram  \eqref{diag:extension_of_Lie}. 
We construct an $\Sp(2n)$-invariant section $s$ using the denominator-free Weyl symmetrization map  (see  Proposition \ref{pr:denominator-free_weyl}):
\begin{equation}\label{pr:explicit_formula_for_s_bis}
   h^{-1} \fm^{4}/\fm^{5}\stackrel{h}{\iso} \coker(V'\otimes V \to \mathrm{Sym}^4 V) \xrightarrow{\frac{1}{h}W}   h^{-1} A_h, \quad s(h^{-1}f)= h^{-1}W(f). 
\end{equation}
Here 
\begin{equation}
W(v_1\otimes v_2 \otimes v_3 \otimes v_4)= v_1 v_2 v_3 v_4 -\frac{h}{2}\sum_{i<j} c_{ij} v_t v_q + 4h^2(c_{12}c_{34} + c_{13}c_{24} + c_{14}c_{23}),
\end{equation}
where $c_{ij}$ is the Poisson bracket $\{v_i,v_j\}\in \bF_p$, and $(t,q)$ is the complementary pair with $t<q$.  

To complete the proof of Proposition we need to compare the bilinear form\footnote{As we pointed out above, if $n=1$, $\Sp(2n)$-equivariant section $s$ is not unique, but form \eqref{pr:explicit_formula_for_Lie} does not depend on its choice.}  \eqref{pr:explicit_formula_for_Lie} with $\frac{h}{4}B$. We 
 compute
\[(x_1^2y_1, x_1 y^2_1)_{H_3}= [\tfrac{x_1^2y_1}{h}+x_1, \tfrac{x_1 y^2_1}{h}+y_1]_\fg 
-s(\{ x_1^2y_1, x_1 y^2_1\})= -h -s(0)= -h = \tfrac{h}{4} B(x_1^2y_1, x_1 y^2_1).\]
The irreducibility of $\fm^{3}/\fm^{4}$ implies the result.
\end{proof}

 \subsection{Further remarks.}
 The following corollary of Theorem \ref{appendix:maintheorem} will be used  in \S \ref{s.lag}.
 \begin{pr}\label{app.autAlifting}
  Let $f\colon T \to G_0 \times {{\mathbb P^1}} $ be a morphism of schemes, $T_0\mono T$ the scheme theoretic fiber of $f$ over $G_0 \times \{0\}$. Then Zariski locally on $T$ there exists a morphism of ${\mathbb P^1}$-schemes $g\colon  T \to \underline \Aut (\cA_{{\mathbb P^1}})$ such that $f_{|T_0} = g_{|T_0}$.
 \end{pr}	
 \begin{proof}
 Consider the action of $G_0 $ on $\cA_{\mbb P^1}\lmod$ from Theorem \ref{appendix:maintheorem}  and let 
$$ 1 \to \underline{\cA}_{{\mathbb P^1}}^\times  \to \widehat G_{{\mathbb P^1}}  \to  G_0 \times {{\mathbb P^1}}\to 1$$
be the corresponding extension \eqref{diag.f.extany}. As explained in Remark \ref{enum:triviality of A-star torsor} in the proof of Theorem \ref{appendix:maintheorem}, Zariski locally  on $T$  the map $f$ admits a lifting $\tilde f\colon T \to  \widehat G_{{\mathbb P^1}}$. The composition of $\tilde f$ with $\alpha\colon \widehat G_{{\mathbb P^1}} \to \underline \Aut (\cA_{{\mathbb P^1}})$ does the job.
  \end{proof}

\begin{rem}\label{rem:this is not an action on algebra} We claim that the action of $G_0$ on $\cA_{{\mathbb P^1}}\lmod$ constructed above does not lift to an action on the algebra $\cA_{{\mathbb P^1}}$ even if ${{\mathbb P^1}}$ is replaced by the formal completion $\widehat{\mbb D}$. 
To see  this we shall prove that the surjection $r\colon G= \underline{\Aut}(A_h) \epi G_0$ has no sections {\it i.e.}, group homomorphisms 
$s\colon G_0 \to G$
with $r\circ s= \Id$. 
Assuming the contrary, consider the induced morphisms of the restricted Lie algebras: 
\begin{equation}
    \operatorname{d}r_{1_G}\colon \Lie G \to  \Lie G_0, \quad \operatorname{d}s_{1_{G_0}}\colon \Lie G_0 \to  \Lie G.
\end{equation}
 Observe that $\Lie G$, being the Lie algebra of $\bF_p[[h]]$-linear derivations of the associative algebra $A_h$,  has an additional structure of restricted Lie algebra over $\bF_p[[h]]$\footnote{In particular, this means that $(f(h)D)^{[p]}= f(h^p) D^{[p]}$, for any $f(h)\in \bF_p[[h]]$ and $D\in \Lie G$.}. Thus, the second map extends uniquely to a morphism of restricted Lie algebra over $\bF_p[[h]]$
\begin{equation}\label{eq:noexist}
     \Lie G_0 \otimes \bF_p[[h]]\to  \Lie G.
\end{equation}
We intend to show that this map must be an isomorphism and arrive to a contradiction by inverting $h$. 
To this end, recall that
the restricted Lie algebra $\Lie G_0$ is the algebra of Hamiltonian vector fields on $\Spec A_0$  (\cite[Lemma 3.3.]{bk}).
The Lie algebra $\Lie G= \Der A_h$ of $\bF_p[[h]]$-linear derivations of $A_h$ fits into an exact sequence
\[0\to \tfrac{1}{h} \bF_p[[h]] \to \tfrac{1}{h} A_h \rar{\ad} \Der A_h\to 0.\]
This follows from the above description of  $\Lie G_0$ 
and Proposition \ref{prop:automorphisms trivial modulo h are inner locally} which asserts that every derivation of $A_h$ which is trivial modulo $h$ is inner. In particular, $\Lie G$ is a finite free $\bF_p[[h]]$-module and  $\operatorname{d}r_{1_G}$ identifies the target with $\Lie G/h\Lie G$.
This implies \eqref{eq:noexist} is an isomorphism. 

By inverting $h$ we derive an isomorphism of restricted Lie algebras over 
$\bF_p((h))$.
\[\Lie G_0 \otimes_{\bF_p}  \bF_p((h)) \iso  \Lie G \otimes _{\bF_p[[h]]} \bF_p((h)).\]
 The isomorphism $A_h(h^{-1}) =\Mat_{p^n}(\bF_p((h)))$ identifies $\Lie G \otimes _{\bF_p[[h]]} \bF_p((h))$
with  $\mathfrak{pgl}_{p^n} \otimes _{\bF_p} \bF_p((h))$, the quotient of 
$\mathfrak{gl}_{p^n} \otimes _{\bF_p} \bF_p((h))$ by its center. 
It is shown in \cite[Lemma 4.4]{bv}  that there exists  a split surjection $G_0 \to \bG_a$. This yields a nonzero homomorphism  of restricted Lie algebras $  \Lie G_0 \otimes_{\bF_p}  \bF_p((h)) \to \bF_p((h))$,
where $\bF_p((h))$ is equipped with the zero  $p$-th power operation. We claim that $\mathfrak{pgl}_{p^n} \otimes _{\bF_p} \bF_p((h))$ 
 does not admit such homomorphisms. Indeed, the quotient $\mathfrak{gl}_{p^n}/[\mathfrak{gl}_{p^n},\mathfrak{gl}_{p^n}]$ 
is the restricted Lie algebra of the multiplicative group. In particular, the $p$-th power operation on this quotient is non-trivial.  This contradiction completes the proof.
\end{rem}
\begin{rem}[The action of $\Lie G_0$ on $\cA_{{\mathbb P^1}}\lmod$]\label{rem:Lie_algebra} The diagram of the Lie algebras 
\begin{equation}\label{diag.f.extany2.Lie}
	\begin{tikzcd}
	0 \arrow[r, ""] & \underline{\cA}_{{\mathbb P^1}}  \arrow[r, ""]\arrow[dr, "ad"]& \Lie \widehat G_{{\mathbb P^1}}  \arrow[r, ""]\arrow[d, "D(\alpha)"]& \Lie G_0 \otimes \cO_{{\mathbb P^1}} \arrow[r, ""] &0  \\
	&& \Der(\cA_{{\mathbb P^1}}) &  &
	\end{tikzcd}
	\end{equation}
determined by the action of $G_0$ on $\cA_{{\mathbb P^1}}\lmod$ (see Theorem \ref{appendix:maintheorem}) admits the following explicit description. Denote 
$\cO_{{\mathbb P^1}}([0])$ the invertible $\cO_{{\mathbb P^1}}$-module corresponding to the divisor $[0]$ on ${\mathbb P^1}$. Since Lie bracket on the fiber of $\cA_{{\mathbb P^1}}$
over $0$ is trivial, the vector bundle $\cA_{{\mathbb P^1}}([0]):= \cA_{{\mathbb P^1}}\otimes \cO_{{\mathbb P^1}}([0])$ has a unique Lie algebra structure such that the map 
$\cA_{{\mathbb P^1}}\rar{\bf{1}} \cA_{{\mathbb P^1}}([0])$ (inducing an isomorphism over ${\mathbb P^1}\backslash \{0\}$)  makes $\cA_{{\mathbb P^1}}$  a Lie ideal in $\cA_{{\mathbb P^1}}([0])$.
 The cokernel of $\bf{1}$
 is the skyscraper sheaf  $i_{0*}  (A_0, \{-,-\})$, whose fiber over $0$ is $A_0$ with its Poisson bracket $\{-,-\}$.
Consider the pullback diagram of Lie algebras
\begin{equation}\label{diag.f.extany2.Lie.c}
	\begin{tikzcd}
	\tilde \fg \arrow[r, ""] \arrow[d, ""] &  \cA_{{\mathbb P^1}}([0]) \arrow[d] \\
	 \cO_{\mathbb P^1} \otimes (A_0, \{-,-\}) \arrow[r, ""] & i_{0*}  (A_0, \{-,-\}).
	\end{tikzcd}
	\end{equation}
\end{rem}
The Lie algebra $\tilde \fg$ is an extension of $\cO_{\mathbb P^1} \otimes (A_0, \{-,-\})$ by 
$\cA_{{\mathbb P^1}}$. The map  $\cO_{\mathbb P^1} \mono \cO_{\mathbb P^1} \otimes (A_0, \{-,-\})$ 
given by the algebra morphism $\bF_p\to A_0$ lifts to a homomorphism 
of Lie algebras  $\cO_{\mathbb P^1} \to \tilde \fg$. The lift is uniquely characterized by its composition $\cO_{\mathbb P^1} \to  \cA_{{\mathbb P^1}}([0])$ with $\tilde \fg  \to \cA_{{\mathbb P^1}}([0])$; the latter is set to carry $f\in \cO_{\mathbb P^1} $ to $h^{-1} f\in \cA_{{\mathbb P^1}}([0])$.
We claim that there is an isomorphism $\Lie \widehat G_{{\mathbb P^1}} \iso \coker(\cO_{\mathbb P^1} \to \tilde \fg)$ that carries the homomorphism 
$D(\alpha): \Lie \widehat G_{{\mathbb P^1}} \to  \Der(\cA_{{\mathbb P^1}})$ to the map
$ \coker(\cO_{\mathbb P^1} \to \tilde \fg) \rar{\ad} \Der(\cA_{{\mathbb P^1}}).$
We do not prove this claim because we will not use it in this paper.

\section{Canonical twistor categorical quantization of  $(X,[\eta])$.}\label{canonicalquant.ofcat}
\subsection{Construction of the canonical quantization and its uniqueness.}\label{ssec:construction of the canonical quantization}
For the duration of this subsection let  $S$ be the base scheme over $\bF_p$ with $p>2$, and let $(X, [\eta])$  be a pseudo-smooth  scheme over $S$  equipped with a restricted symplectic structure 
$$ [\eta]\in H^0(X, \coker(\cO_X \rar{d} \Omega^1_{X/S})).$$
  Denote by  $\omega := d[\eta]$ the corresponding  symplectic form on $X/S$. 
 Also, denote by $2n$ the rank of the vector bundle $\Omega^1_{X/S}$, which we implicitly assume to be constant, {\it i.e.}  $n\in \bZ$.
 Let $\pi\colon \cM_{X, [\eta]}\to X^\lp$ be the $G_0$-torsor of Darboux frames (see Construction \ref{constr:torsor of Darboux frames}).  Denote by $\xi\colon  X^\lp \to BG_0$ the morphism corresponding to the latter. In Theorem \ref{appendix:maintheorem} we constructed an action of $G_0$ on 
$ \cA_{{\mathbb P^1}}\lmod$ and the corresponding quasi-coherent sheaf of categories $(\cA_{{\mathbb P^1}}\lmod)_{{{\mathbb P^1}}/G_0}$ on ${{\mathbb P^1}}/G_0$. Let $\QCoh_{X,[\eta],\mathbb P^\mathsf{tw}}$  
 be its pullback via the morphism $\xi \times \Id\colon X^\lp \times {{\mathbb P^1}} \to BG_0 \times {{\mathbb P^1}}= {{\mathbb P^1}}/G_0$:
$$ \QCoh_{X,[\eta],\mathbb P^\mathsf{tw}} := (\xi \times \Id)^*(\cA_{{\mathbb P^1}}\lmod)_{{{\mathbb P^1}}/G_0}.$$
When there is no risk of confusion we shall write $\QCoh_{\mathbb P^\mathsf{tw}}$ for $\QCoh_{X,[\eta],\mathbb P^\mathsf{tw}}$.

In \S \ref{subsectionrestrictionto0} we identified the action of $G_0$ restricted to $A_0\lmod$ with the one induced by the tautological $G_0$-action on $A_0$. In particular, the
section $G_0\ra \widehat{G}_{{\mathbb P^1}}\times_{{\mathbb P^1}} \{0\}$ (see Proposition \ref{pr:restrictionto0forgroups}) defines a lift of the $G_0$-torsor $\cM_{X, [\eta]} \to X^\lp$  to a $\hat G_{{\mathbb P^1}} \times_{{\mathbb P^1}} \{0\}$-torsor. 
Using  Proposition \ref{pr.twistedmodules} and the isomorphism $F_*^\lp \cO_X \simeq \cM_{X, [\eta]}\times ^{G_0} A_0$ we derive an equivalence
 \begin{equation}\label{eq:resttotheorigin}
 \Xi_0\colon (\QCoh_{\mathbb P^\mathsf{tw}})_{|X^\lp \times \{0\}} \iso F_*^\lp\cO_X\lmod.
 \end{equation}
   Observe that  (\ref{eq:resttotheorigin}) yields, in particular, an equivalence between the category  $\QCoh_{\mathbb P^\mathsf{tw}}(X^\lp \times \{0\})$ and the category of quasi-coherent sheaves on $X$. 
   We shall refer to $\QCoh_{\mathbb P^\mathsf{tw}}$   as {\it the canonical twistor categorical quantization of $X$}.

The trivialization of the $G_0$-action on the fiber of $\cA_{{\mathbb P^1}}\lmod $ at $h=\infty$ yields an equivalence  
$$\Xi_\infty\colon (\QCoh_{\mathbb P^\mathsf{tw}})_{| X^\lp\times \{\infty\}} \iso \cO_{X^\lp}\lmod.$$ 
Moreover, the restriction of $\QCoh_{\mathbb P^\mathsf{tw}}$ to ${\mathbb P^1} \backslash \{0\}$ can be described as modules over a central reduction of the algebra of twistor differential operators (see Construction \ref{constr:central reductions over S}): 
 \begin{equation}\label{eq:resttothecomplement}
(\QCoh_{\mathbb P^\mathsf{tw}})_{|X^\lp \times {{\mathbb P^1}}\backslash \{0\}} \iso  \mc D_{X,h^{-1}{[\eta]},{\mathbb P^1}\backslash \{0\}}\lmod,
\end{equation}
where  $\mc D_{X,h^{-1}{[\eta]},{\mathbb P^1}\backslash \{0\}}$ is the restriction of $\mc D_{X,h^{-1}{[\eta]},{\mathbb P^1}}$ to $X^\lp \times ({{\mathbb P^1}}\backslash \{0\})$.
To see this recall from Theorem \ref{th:restriction_to_P-0}  that the restriction of   $(\cA_{{\mathbb P^1}}\lmod)_{{{\mathbb P^1}}/G_0}$ to $({{\mathbb P^1}}\backslash \{0\})/G_0$ is equivalent to the restriction of 
$(\cA_{{\mathbb P^1}}^\flat\lmod)_{{{\mathbb P^1}}/G_0}$, where the action of $G_0$ on $\cA_{{\mathbb P^1}}^\flat\lmod$ arises from a $G_0$-action on $\cA_{{\mathbb P^1}}^\flat$. Proposition \ref{pr.twistedmodules} and the isomorphism
of algebras $\mc M_{X,[\eta]}\times^{G_0} \mc A_{{\mathbb P^1}}^\flat \simeq \mc D_{X,h^{-1}{[\eta]},{\mathbb P^1}}$ from Proposition  
\ref{prop:localization} yield (\ref{eq:resttothecomplement}).

Recall from Corollary \ref{cor:D_X is an Azumaya algebarain pseudo-smooth setting}
and Example \ref{ex:D_X0S} that $\mc D_{X,h^{-1}{[\eta]},{\mathbb P^1}\backslash \{0\}}$ is
an Azumaya algebra over $X^\lp \times {\mathbb P^1}\backslash \{0\} $ equipped with a canonical splitting over $X^\lp \times  \{\infty\}$. In particular, the category 
$\QCoh_{X,[\eta],\mathbb P^\mathsf{tw}}(X^{\lp}\times {\mathbb P^1}\backslash \{0\})$ is equivalent to the category of quasi-coherent sheaves on $X^{\lp}\times {\mathbb P^1}\backslash \{0\}$ twisted by a $\bG_m$-gerbe $[\mc D_{X,h^{-1}{[\eta]},{\mathbb P^1}\backslash \{0\}}]\to 
X^{\lp}\times {\mathbb P^1}\backslash \{0\}$. By Constructions \ref{const:multiplicative_gerbe} or \ref{const:mult_structure_on_the_gerbe_over_BG_0xG_a} this gerbe
has a multiplicative structure with respect to the group structure on ${\mathbb P^1}\backslash \{0\}\simeq \bG_a$. In particular, the class of $[\mc D_{X,h^{-1}{[\eta]},{\mathbb P^1}\backslash \{0\}}]$ in the Brauer group is killed by $p$.
\smallskip 

Note that the formation $(X,[\eta])\mapsto \QCoh_{X,[\eta],\mathbb P^\mathsf{tw}}$ is functorial in the following sense:
\begin{pr} \begin{enumerate}
    \item For $S' \to S$, let $[\eta']$ be the restricted symplectic structure on $X \times_S S'/S'$ obtained by 
     pulling back  $[\eta]$.  
     Write $(X \times_S S')^\lp$ for the reduced Frobenius twist of $S'$-scheme
     $X \times_S S'$.
     Then
     $\QCoh_{X \times_S S',[\eta'],\mathbb P^\mathsf{tw}}$ is naturally equivalent to the pullback of $\QCoh_{X,[\eta],\mathbb P^\mathsf{tw}}$ along the morphism $(X \times_S S')^\lp \times \bP^1 \to X^\lp \times \bP^1$ (see Remark \ref{rem:twist vs basechange}).
     \item For an \'etale morphism $X'\to X$, let $[\eta']$ be the restricted symplectic structure on $X'/S$
     obtained by 
     pulling back  $[\eta]$.   Then $\QCoh_{X',[\eta'],\mathbb P^\mathsf{tw}}$ is naturally equivalent to the pullback of $\QCoh_{X,[\eta],\mathbb P^\mathsf{tw}}$ along the morphism $(X')^\lp \times \bP^1 \to X^{\lp} \times \bP^1$.
     
\end{enumerate}
    
\end{pr}
 \begin{proof}
  This follows directly from Remark \ref{rem:base change of the Darboux torsor}:  the $G_0$-torsor of  Darboux frames for $X \times_S S'/S'$ 
     (resp. for $X'/S$) is obtained
     by pulling back the corresponding torsor for $X/S$ along the map $(X \times_S S')^\lp \to X^\lp $ (resp. $(X')^\lp  \to X^{\lp}$). Since $\QCoh_{\mathbb P^\mathsf{tw}}$ is by definition pulled back from $BG_0\times \mathbb P^1$ via the map classifying the torsor of Darboux frames, the proposition follows.
    \end{proof} 
 In particular, $\QCoh_{\mathbb P^\mathsf{tw}}$ is an equivariant quasi-coherent sheaf of categories with respect to the group of automorphisms of $(X/S, [\eta])$ that acts naturally on $X^\lp \times {\mathbb P^1}$. 

In the remaining part of this subsection we explain in what sense the canonical quantization $\QCoh_{\mathbb P^\mathsf{tw}}$  is unique. Let $H$ be a connected affine group scheme over $\bF_p$ equipped with a homomorphism $\upsilon\colon H\to G_0$ and let $\pi_H\colon \cM_{X, [\eta], H}\to X^\lp$ be an $H$-torsor equipped with an $H$-equivariant morphism  $\cM_{X, [\eta], H} \to \cM_{X, [\eta]}$, where the action of $H$ on  $\cM_{X, [\eta]}$ is given by $\upsilon$. For brevity let $\mc M\coloneqq \cM_{X, [\eta], H}$ and denote by $\pro_{\mc M}\colon \cM \times \mbb P^1 \to \mbb P^1$ the projection map.
 \begin{pr}\label{mainth:quantofthecategory} There exists 
a unique (up to an equivalence that is itself unique up to a unique isomorphism) quadruple $(\QCoh_{\mathbb P^\mathsf{tw}},  \Theta_H, \Xi_{\infty}, \varpi_H)$, consisting of a quasi-coherent sheaf of categories $\QCoh_{\mathbb P^\mathsf{tw}}$ on $X^\lp \times {{\mathbb P^1}}$
equipped with the following pieces of structure:
\begin{itemize}
\item[(i)] An equivalence $$\Theta_H\colon (\pi_H  \times \Id)^* \QCoh_{\mathbb P^\mathsf{tw}} \iso \pro_{\cM}^* \cA_{{\mathbb P^1}}\lmod $$  of quasi-coherent sheaves of abelian categories over 
$\cM\times {{\mathbb P^1}}$.
\item[(ii)] An equivalence  $$\Xi_\infty\colon (\QCoh_{\mathbb P^\mathsf{tw}})_{| X^\lp\times \{\infty\}} \iso \cO_{X^\lp}\lmod$$ of quasi-coherent sheaves of categories over $X^\lp$. 

\item[(iii)] An isomorphism 
$$\varpi_H\colon \pi_H^* (\Xi_\infty)\xra{\sim } \pro_\cM^*(\psi)\circ {\Theta_H}_{|\cM\times \{\infty\}} $$ of  equivalences $
\pi^*_H ((\QCoh_{\mathbb P^\mathsf{tw}})_{| X^\lp \times \{\infty\}}) \xra{\sim } \cO_{\cM}\lmod
$ of sheaves of categories on $\cM\times \{\infty\}.$  Here $\psi\colon \mc A_{\mbb P^1}\lmod_{|\{\infty\}}\simeq \mbb F_p\lmod$ is an equivalence from  
Remark \ref{rem:A_S is matrix algebra outside of 0}, whose pull-back $\pro_{\cM}^*(\psi)$ yields an equivalence $(\pro_{\cM}^* \cA_{{\mathbb P^1}}\lmod)_{|\cM\times \{\infty\}} \iso \cO_{\cM}\lmod$.
\end{itemize}
It should satisfy the following property 
\begin{itemize}
\item[(P)] The isomorphism 
\begin{equation}\label{eq:what_we_want_to_descent}
  \Center (\pi^*_H ((\QCoh_{\mathbb P^\mathsf{tw}})_{| X^\lp\times \{0\}}))\simeq A_0\otimes _{\bF_p} \cO_{ \cM},  
\end{equation}
 determined by $\Theta_H$,  descends\footnote{For a quasi-coherent sheaf of categories $\fS$ over a stack $Y$, we denote by $\Center (\fS)$ the presheaf of $\cO_Y$-algebras on $Y$  assigning to each affine scheme $Z$ over $Y$ the $\cO(Z)$-algebra  $\Center (\fS(Z))$. Property (P) asserts the existence of an isomorphism \eqref{eq:descent_to_X} 
 of presheaves of $\cO$-algebras on affine schemes over $X^\lp$, whose restriction to the category of affine schemes over $\cM$ is
 given by \eqref{eq:what_we_want_to_descent}. This makes sense because the pullback of 
 $F_*^\lp \cO_X$ to $\cM$ is isomorphic to $A_0\otimes _{\bF_p} \cO_{ \cM}$.
} to an isomorphism
\begin{equation}\label{eq:descent_to_X}
  \Center((\QCoh_{\mathbb P^\mathsf{tw}})_{| X^\lp \times \{0\}}) \simeq F_*^\lp \cO_X.  
\end{equation} 
\end{itemize}
\end{pr}
\begin{proof}  The existence assertion follows from the construction of $\QCoh_{\mathbb P^\mathsf{tw}}$  explained above: $ \Theta_H$ is the pullback of the equivalence $(\pi \times \Id)^* \QCoh_{\mathbb P^\mathsf{tw}} \iso \pro_{\cM_{X, [\eta]}}^* \cA_{{\mathbb P^1}}\lmod $ along 
the morphism  $\cM \to \cM_{X, [\eta]}$ and $\varpi_H$ comes from the construction of $\Xi_{\infty}$ as a descent of the equivalence on $\cM_{X, [\eta]}\times \{\infty\}$.
Property (P) follows from (\ref{eq:resttotheorigin}).

While proving uniqueness, we may assume that $X$ is connected. Let $(\fS_i,  \Theta_i, \Xi_{\infty, i}, \varpi _i)$, $i=1,2$, be  quadruples with property (P). We shall prove that there exists a unique equivalence (up to a unique isomorphism) connecting the two. The argument is similar to the proof of the uniqueness part of Theorem  \ref{appendix:maintheorem}.  Giving a pair  $(\fS_i,  \Theta_i)$ is equivalent to specifying the descent data for $\pro_{\cM}^* \cA_{{\mathbb P^1}}\lmod $ along the morphism $\cM\times {{\mathbb P^1}} \to \cM/H\times {{\mathbb P^1}}=X^\lp \times {{\mathbb P^1}}$. In particular, 
$(\fS_i,  \Theta_i)$ determines an autoequivalence $\phi_i$ of $ \cA_{{\mathbb P^1}}\lmod $ lifted to $H \times \cM\times {{\mathbb P^1}}$.  Property (P) implies that the composition $\phi_1\circ \phi_2^{-1}$ acts as $\Id$ on the center of $ \cA_{{\mathbb P^1}}\lmod $  pulled back to $H \times \cM\times \{0\}$.
 Applying  Lemma \ref{app.lemma.autoequivalences} we conclude that $\phi_1\circ \phi_2^{-1}$ is the tensor product with a line bundle $L$ over $H \times \cM\times {{\mathbb P^1}}$.  The equivalence $\Xi_{\infty, i}$ yields an isomorphism between $\phi_i$ restricted to $H \times \cM\times \{\infty\}$ and the identity functor.  In turn, the latter determines a trivialization of $L$ restricted to $H \times \cM\times \{\infty\}$. The groupoid of line bundles over  $H \times \cM\times {{\mathbb P^1}}$ equipped with a trivialization  over $H \times \cM\times \{\infty\}$ 
 is discrete and its $\pi_0$ is the group $H^0(H \times \cM, \bZ)=\bZ$ (since $H$ is connected). Finally, the cocycle  constraint implies  that $\phi_i$ and, hence, $\phi_1\circ \phi_2^{-1}$ is isomorphic to $\Id$ when restricted to $\{1_{H}\} \times \cM\times {{\mathbb P^1}}$. This defines a trivialization of $L$.  
\end{proof}
Let $(\QCoh_{\mathbb P^\mathsf{tw}},  \Theta_H, \Xi_{\infty}, \varpi_H)$ be a quadruple  from Proposition \ref{mainth:quantofthecategory}.  Note that the equivalence (\ref{eq:resttotheorigin}) is actually not a part of the defining properties of $(\QCoh_{\mathbb P^\mathsf{tw}},  \Theta_H, \Xi_{\infty}, \varpi_H)$. We shall now explain how the uniqueness statement of the Proposition \ref{mainth:quantofthecategory} formally produces a preferred equivalence of the form (\ref{eq:resttotheorigin}). 

First, note that, 
by property (P),  $\Theta_H$ determines an $ F_*^\lp \cO_X$-linear structure on $(\QCoh_{\mathbb P^\mathsf{tw}})_{|X^\lp \times \{0\}}$ ({\it i.e.}, $(\QCoh_{\mathbb P^\mathsf{tw}})_{|X^\lp \times \{0\}}$ is the direct image of a quasi-coherent sheaf of categories on $X$). 
Secondly,
observe that, for any ring $A$, the groupoid of autoequivalences of $A\lmod$ is identified with the groupoid of autoequivalences of $A^{op}\lmod$: indeed, $A^{{\op}}\lmod$ is naturally equivalent to the category of colimit-preserving functors from $A\lmod$ to the category of abelian groups. In particular, the descent data for $\pro_{\cM}^* \cA_{{\mathbb P^1}}\lmod $ along the morphism $\cM\times {{\mathbb P^1}} \to X^\lp \times {{\mathbb P^1}}$ specified by  
$(\QCoh_{\mathbb P^\mathsf{tw}},  \Theta_H)$ determines a descent datum\footnote{In the terminology of Construction \ref{constr:dualizable sheaves of categories}, it is the descent data for the dual sheaf $\QCoh_{\mathbb P^\mathsf{tw}}^\vee$.} for $\pro_{\cM}^* \cA_{{\mathbb P^1}}^{\op}\lmod $. We denote by $\QCoh_{\mathbb P^\mathsf{tw}}^\circ$ the corresponding sheaf of categories over $X^\lp \times {{\mathbb P^1}}$ and by $\Theta_H^\circ\colon (\pi_H  \times \Id)^* \QCoh_{\mathbb P^\mathsf{tw}}^\circ \iso \pro_{\cM}^* \cA_{{\mathbb P^1}}^{\op}\lmod $ the natural equivalence defined by $\Theta_H$. 
Explicitly,  $\QCoh_{\mathbb P^\mathsf{tw}}^\circ$ assigns to each affine $Z\rar{u} X^\lp \times {{\mathbb P^1}}$ 
the category of all colimit-preserving $\cO(Z)$-linear functors  $ u^* \QCoh_{\mathbb P^\mathsf{tw}} \to \cO(Z)\lmod$. The equivalence $\Xi_\infty$  determines an equivalence 
$\Xi^\circ_\infty\colon (\QCoh_{\mathbb P^\mathsf{tw}}^\circ)_{| X^\lp \times \{\infty\}} \iso \cO_{X^\lp}\lmod$.
 Next, recall from 
(\ref{app.alpha}) an isomorphism $ \alpha\colon  \iota^* \cA_{{\mathbb P^1}}^{{\op}}\iso \cA_{{\mathbb P^1}}$, where $\iota\colon \mbb P^1 \to \mbb P^1$ sends $h$ to $-h$.
Let us consider the sheaf of categories $({\Id}\times \iota)^*\QCoh_{\mathbb P^\mathsf{tw}}^\circ$ with the equivalences
$$
\Theta_H^{\mr{dual}}\colon (\pi_H\times \iota)^*\QCoh_{\mathbb P^\mathsf{tw}}^\circ \xra{(\Id\times \iota)^*(\Theta_H^\circ)} \pro_\cM^*(\iota^*\mc A_{\mbb P^1}^{\op})\lmod \xra{\pro_{\mc M}^*(\alpha)} \pro_\cM^*\mc A_{\mbb P^1}\lmod
$$
and 
$$
\Xi_\infty^{\mr{dual}}\coloneqq\Xi_\infty^{\circ}\colon (\QCoh_{\mathbb P^\mathsf{tw}})_{| X^\lp\times \{\infty\}}^\circ \iso \cO_{X^\lp}\lmod.
$$
The isomorphism $(\varpi_H)^\circ$ defines an isomorphism $\pi_H^* (\Xi_\infty^\circ)\xra{\sim } \pro_\cM^*(\psi^\circ) \circ {\Theta_H^\circ}_{|\cM\times \{\infty\}} $, while the restriction ${\Theta_H^{\mr{dual}}}_{|\mc M\times \{\infty\}}$ is given by $\pro_{\mc M}^*(\alpha_{|\{\infty\}})\circ {\Theta_H^\circ}_{|\cM\times \{\infty\}}$. To complete $(({\Id}\times \iota)^*\QCoh_{\mathbb P^\mathsf{tw}}^\circ,\Theta_H^{\mr{dual}},\Xi_\infty^{\mr{dual}})$ to a quadruple as in Proposition \ref{mainth:quantofthecategory} we need to relate $\psi$ and $\psi^{\circ}$.

\begin{rem}\label{rem:transformation at the infinity} Let $A_{\infty}:={\mc A_{\mbb P^1}}_{|\{\infty\}}$.
    There is a canonical up to scalar\footnote{The scalar will not be important: by Proposition \ref{mainth:quantofthecategory} the quadruples $(\QCoh_{\mathbb P^\mathsf{tw}},  \Theta_H, \Xi_{\infty}, \varpi_H)$ and $(\QCoh_{\mathbb P^\mathsf{tw}},  \Theta_H, \Xi_{\infty}, c\cdot \varpi_H)$ are canonically isomorphic for any $c\in \mbb F_p^\times$. } isomorphism $\beta$ filling the diagram $$
\begin{tikzcd}
	A_\infty^{\op}\lmod && A_\infty\lmod \\
	& \mbb F_p\lmod
	\arrow[from=1-1, to=1-3, "\alpha_{\infty}^*"]
	\arrow[""{name=0, anchor=center, inner sep=0}, from=1-1, to=2-2, "\psi^\circ"']
	\arrow[from=1-3, to=2-2, "\psi"]
	\arrow["{\beta}"{description, inner sep=0.5pt}, between={0.2}{0.9},Rightarrow, nfold, from=1-3, to=0]
\end{tikzcd}
    $$
Indeed, let $C=\mbb F_p[x_i]/(x_i^p)_{1\le i\le n}$, and recall that by Remark \ref{rem:A_S is matrix algebra outside of 0} there is a natural isomorphism $A_\infty\simeq \End_{\mbb F_p}(C)$; here $x_i$ act on $C$ by multiplication, while $y_i$ by $\partial_{x_i}$. The isomorphism $\alpha_{\infty}\colon A_{\infty}^{\op}\to A_{\infty}$ is given by $x_i\mapsto x_i$ and $y_i\mapsto -y_i$. Let $C^\vee:=\Hom_{\mbb F_p}(C,\mbb F_p)$; this is naturally a left $A_\infty^{\mr{op}}$-module. Equivalence $\psi$ is given by $\Hom_{A_\infty}(C,-)$, while $\psi^{\circ}$ by $\Hom_{A_\infty^{\op}}(C^\vee,-)$; thus a datum of $\beta$ is equivalent to an isomorphism of right $A_\infty$-modules
 \begin{equation}\label{eq:pairing on C and C dual}
\beta\colon \alpha_{\infty}^{*,-1}(C) \xra{\sim} C^\vee.
\end{equation}
 Since $C^\vee$ is the unique up to isomorphism simple $A_\infty^{\op}$-module, by Schur's lemma there is a unique such isomorphism up to scalar.

 Note that the equivalence $\alpha_{\infty}^*$ induced by  $\alpha_{\infty}\colon A_\infty^{\op}\to A_\infty$ is symmetric: $\alpha_\infty^{\op}\circ\alpha_\infty\xra{\sim }\id$. Moreover, $\beta$ satisfies the following compatibility property: namely, the composition 
 $$
\psi^\circ \xra{\beta}\psi\circ \alpha_{\infty}^*\xra{(\beta^{\circ})^{-1}\circ\alpha_{\infty}^*} \psi^\circ\circ (\alpha_{\infty}^*)^{\op} \circ\alpha_{\infty}^*\simeq \psi^\circ,
 $$
 where $\beta^\circ\colon \psi^\circ\circ (\alpha_{\infty}^*)^{\op} \to (\psi^{\circ})^\circ=\psi $ is obtained from $\beta$ by applying opposites, is the identity map.
Indeed, this is equivalent to the condition that the non-degenerate bilinear form $(-,-)_\beta$ on $C$ defined by \eqref{eq:pairing on C and C dual} is symmetric, and in our situation this is automatic: in general $(-,-)_\beta$ can be either symmetric or skew-symmetric, but since $\dim_{\mbb F_p}\!C=p^n$ one excludes the skew-symmetric case by non-degeneracy. In fact, identifying $C^\vee$ with $(\Omega^n_{C})^{\otimes 1-p}$ via Proposition \ref{prop:relativedualizaingsheaf}, one can explicitly describe the symmetric form $(-,-)_\beta\colon C\otimes C\to \mbb F_p$ as 
 $$
(f,g)_\beta=\tr(fg\cdot (dx_1\wedge \ldots\wedge dx_n)^{\otimes 1-p})\in \mbb F_p
 $$
(in the notation of Example \ref{ex:trace map for Frobenius neighborhood}).
\end{rem}

Now the transformation $\beta$ gives an isomorphism $\pro_\mc M^*(\psi^\circ)\simeq \pro_\cM^*(\psi)\circ \pro_\cM^*(\alpha_{\infty}^*)$. This way $\varpi_H^{\circ}$ produces an isomorphism
$\varpi_H^{\mr{dual}} \colon \pi_H^* (\Xi_\infty^{\mr{dual}})\xra{\sim } \pro_\cM^*(\psi) \circ {\Theta_H^{\mr{dual}}}_{|\cM\times \{\infty\}},$
as wanted. By Proposition \ref{mainth:quantofthecategory} there exists a unique equivalence of quadruples 
\begin{equation}\label{app.symmetrysigma}
\Sigma\colon (\QCoh_{\mathbb P^\mathsf{tw}},  \Theta_H, \Xi_\infty, \varpi_H)\iso (({\Id}\times \iota)^*\QCoh_{\mathbb P^\mathsf{tw}}^\circ, \Theta^{\mr{dual}}_H, \Xi^{\mr{dual}}_\infty, \varpi^{\mr{dual}}_H).
\end{equation} 
It follows from the discussion in Remark \ref{rem:transformation at the infinity} and Proposition \ref{mainth:quantofthecategory} that the restriction
$$\Sigma_0\colon (\QCoh_{\mathbb P^\mathsf{tw}})_{|X^\lp \times \{0\}} \iso (\QCoh_{\mathbb P^\mathsf{tw}}^\circ)_{|X^\lp \times \{0\}}$$
 of $\Sigma$ to $X^\lp \times \{0\}$ is a \emph{symmetric}\footnote{The datum of a dualizable sheaf of abelian categories (see Construction \ref{constr:dualizable sheaves of categories}) $\cC$ with a \emph{symmetric} equivalence $\cC \iso \cC^\vee$ is, by definition,  the datum of a $\bZ/2\bZ$-equivariant object
 in the $2$-groupoid of dualizable sheaves of abelian categories (equipped with the $\bZ/2\bZ$-action sending
 $\cC$ to $\cC^\vee$). Explicitly, it consists of a dualizable sheaf of abelian categories $\cC$, an equivalence $\Sigma_0\colon\cC \iso \cC^\vee$, and an isomorphism between the composite  $\cC \rar{\Sigma_0} \cC^\vee \rar{\Sigma_0^\vee} (\cC^\vee)^\vee\iso \cC$ and $\Id$ satisfying a certain \emph{property}.} equivalence, whose pullback  to $\cM\times \{0\}$ is naturally isomorphic to the identity functor 
 $A_0 \otimes \cO_{\cM} \lmod \simeq A_0^{{\op}}\otimes \cO_{\cM} \lmod$.
 Every $ F_*^\lp \cO_X$-linear  equivalence $\Xi_0\colon  (\QCoh_{\mathbb P^\mathsf{tw}})_{|X^\lp \times \{0\}}\iso   F_*^\lp \cO_X\lmod $ induces an equivalence $(\QCoh_{\mathbb P^\mathsf{tw}}^\circ)_{|X^\lp \times \{0\}} \iso  F_*^\lp \cO_X\lmod$ that we will denote by $\Xi_0^\circ$.  
\begin{pr}\label{mainth:xi0}
There exists a unique triple $(\Xi_0, \kappa, \upsilon)$, where 
\begin{equation}\label{app.quantofcatmoduloh}
\Xi_0\colon (\QCoh_{\mathbb P^\mathsf{tw}})_{|X^\lp \times \{0\}} \iso F_*^\lp\cO_X\lmod
\end{equation} 
 is a   $ F_*^\lp \cO_X$-linear equivalence of quasi-coherent sheaves of categories over $X^\lp$, $\kappa$ is an isomorphism between $\pi^*_H(\Xi_0)$ and the composite 
 $$\pi^*_H((\QCoh_{\mathbb P^\mathsf{tw}})_{|X^\lp \times \{0\}}) \stackrel{\Theta_H}{\iso} A_0\otimes \cO_{\cM} \lmod  \iso \pi^*_H( F_*^\lp \cO_X\lmod),$$ 
 and $\upsilon$ is an isomorphism $\Xi_0^\circ  \circ \Sigma_0 \iso \Xi_0$ compatible with $\kappa$. 
\end{pr}
\begin{proof}
Locally, for the  {\it fpqc} topology,  $(\QCoh_{\mathbb P^\mathsf{tw}})_{|X^\lp \times \{0\}}$,  as a quasi-coherent sheaf of categories with a $ F_*^\lp \cO_X$-linear structure,  is equivalent  to $ F_*^\lp \cO_X\lmod$. Assigning to a scheme over  $X^\lp$ the groupoid of such equivalences we define a gerbe on  $X^\lp $  banded by the group scheme of invertible elements in  $ F_*^\lp \cO_X$ or, equivalently,
a gerbe $\cE$ on $X$  banded by $\bG_m$.    The symmetric equivalence $\Sigma_0$ amounts to specifying a symmetric  isomorphism between $\cE$  and the opposite gerbe, {\it i.e.}, the reduction of $\cE$ to a gerbe $\cE'$ on $X$  banded by $\mu_2 \subset \bG_m$.
 By construction the pullback of $\cE'$ to $\cM\times_{X^\lp} X$ is trivialized.  We wish to show that the latter descends uniquely to a trivialization of $\cE'$ on $X$. 
 
 For the uniqueness part  it suffices to show that every $\mu_2$-torsor over $X$ equipped with a trivialization over  $\cM\times_{X^\lp} X$ admits a unique trivialization on $X$ compatible with the given trivialization over 
 $\cM\times_{X^\lp} X$. This  follows from connectedness of $H$. 

 For the existence part we may assume that $H=G_0$. The $G_0$-torsor $\pi \times \Id\colon \cM_{X, [\eta]}\times_{X^\lp} X \to X$ is  locally trivial for the \'etale topology (because it is the pushout of the $G^0_0$-torsor  $\cM_{X, [\eta]} \to X$ along $G^0_0\hookrightarrow G_0$ and $G_0^0$ is smooth). 
  Using that $H^0_{\textup{\'et}}(G_0\times \Spec \overline {\bF}_p, \mu_2)\simeq \mu_2$ ($G_0$ is connected) and  $H^1_{\textup{\'et}}(G_0\times \Spec \overline {\bF}_p, \mu_2)=0$ (see Lemma \ref{lm:vanishingofcohomology}), we conclude that $\mu_2 \iso \tau^{<2} R  (\pi \times \Id)_* \mu_2$. Via the Leray spectral sequence this implies that
 the pullback map
\begin{equation}\label{pullbacktrivatzero}
 H^i_{\text{\'et}}(X, \mu_2) \to H^i_{\text{\'et}}(\cM_{X, [\eta]}\times_{X^\lp} X, \mu_2) 
 \end{equation} 
  is an isomorphism for $i=0, 1$ and injective for $i=2$ as desired. 
 \end{proof}

\subsection{$\bG_m$-equivariant quantizations.}\label{ssec:G_m-equivariant quantizations} Let $S$ be a scheme over $\bF_p$ with $p>2$, $X$ be an $S$-scheme. Let $\mbb G_{m,S}\coloneqq \mbb G_m \times_{\mbb F_p} S$, and assume there is a $\mbb G_{m}$-action on $X$ with the action map of $S$-schemes $\gamma\colon \mathbb G_{m}\times X \to X$. Denote $Y:=\mathbb G_{m}\times X$, and consider it as a $\mathbb G_{m,S}$-scheme via $\mathbb G_{m}\times X\to \mathbb G_m\times S$. There is a natural pull-back map $$
\gamma^*\colon H^0(X,\Omega^1_{X/S}/d(\mc O_X))\to H^0(Y, \Omega^1_{Y/ \mbb G_{m,S}}/d(\mc O_{Y})).
$$
 Let $z\in \mc O(\mbb G_{m,S})$ be the standard coordinate function and let $p_X\colon Y\to X$ be the projection.  

\begin{df}\label{def:homogeneous restricted structure}
    A restricted structure $[\eta]$ on an $S$-scheme $X$ with a $\mbb G_{m,S}$-action  is called \textit{homogeneous of weight $w\in \mbb Z$} if it satisfies the following equation in $H^0(Y, \Omega^1_{Y/ \mbb G_{m,S}}/d(\mc O_{Y}))$:
    $$
\gamma^*[\eta]=z^{w}\cdot p_X^*[\eta].
    $$
    Equivalently, for any $S$-scheme $f\colon T\to S$ and an element $t\in \mathbb G_m(T)$, one has 
    $$\gamma_t^*[\eta_T]= t^w\cdot [\eta_T]$$ in $H^0(X_T, \Omega^1_{X_T/T}/d(\mc O_{X_T}))$, where $X_T\coloneqq X\times_S T$ and $[\eta_T]$ is the pull-back of $[\eta]$ under $X_T\to X$ (considered as a section of $\Omega^1_{X_T/T}/d(\mc O_{X_T})$).
\end{df}

\begin{example}
Assume that instead of a restricted structure we are given a symplectic $2$-form $\omega\in \Omega^2_{X/S}(X)$
such that 
\begin{equation}\label{multgractionomega}
\gamma^* \omega  =  z^w \cdot p_X^* \omega \in  H^0(Y,\Omega^2_{Y/\bG_{m,S}})
\end{equation}
for some integer $w$ that is \textit{invertible} in $\bF_p$. The $\bG_m$-action  on $X$ defines a homomorphism from the Lie algebra of $\bG_m$ to the Lie algebra $T_{X/S}$ of vector fields on $X$. Denote by $\theta$ the image of the generator $z\frac{\partial}{\partial z}$ of $\Lie \bG_m$. 
Formula (\ref{multgractionomega}) together with the identity $d\omega=0$  imply
that 
$$d \iota_{\theta} \omega = w \omega. $$
The $1$-form $\eta = \frac{1}{w} \iota_{\theta} \omega $ then defines a restricted Poisson structure on $X$, and one can check that it is homogeneous of weight $w$ in the sense of Definition \ref{def:homogeneous restricted structure}. The Liouville form $\eta_{\mathsf{can}}$ on $T^*Y$ (for a pseudo-smooth $S$-scheme $Y$, see Example \ref{ex:restricted structure on T^*X}) is an example of this construction; here the $\mathbb G_m$-action on $T^*Y$ is given by standard rescaling along the fibers of the projection $\pi\colon T^*Y\to Y$, and $\eta_{\mathsf{can}}$ is homogeneous of weight 1.
\end{example}

Let $(X,[\eta])$ be a homogeneous $\mbb G_m$-equivariant restricted symplectic $S$-scheme. The $\mbb G_m$-action on $X$ defines a natural action on $X^{\lp}$ via the absolute Frobenius $F_{\mbb G_m}\colon \mbb G_m \to \mbb G_m$. Indeed, the composition 
$$\bG_m \times {X\twt}  \xra{(F_{\bG_m}, \id)} \bG_m \times {X\twt}\simeq (\mbb G_m\times X)^{(1)}  \xra{\gamma\twt} {X\twt},$$
defines a $\mathbb G_m$-action on $X\twt$. Moreover the relative Frobenius $F_{X/S}\colon X\to X^{(1)}$ is $\mathbb G_m$-equivariant, so the action restricts to the schematic image of $F_{X/S}$, given by $X^\lp \hookrightarrow X\twt$. Denote the corresponding action map by $\gamma_p\colon \mbb G_m\times X^{\lp} \to X^\lp$. Consider also an action $\chi_w\colon \bG_m \times {{\mathbb P^1}}\to {{\mathbb P^1}}$ given by the formula $\chi_w^*(h)=z^w h$. 

\begin{pr}Let $(X,[\eta])$ be a homogeneous $\mbb G_m$-equivariant restricted symplectic $S$-scheme of weight $w$. Consider the $\mathbb G_m$-action on $X^{\lp}\times \mbb P^1$ given by $(\gamma_p,\chi_w)$. Then there exists a quasi-coherent sheaf of abelian categories
$\QCoh_{\mathbb P^\mathsf{tw}\!/\mbb G_m}$ over the quotient stack $(X^\lp \times {{\mathbb P^1}})/_{\!\gamma_p, \chi_w} \bG_m$ with a natural identification 
$$
q^*\QCoh_{\mathbb P^\mathsf{tw}\!/\mbb G_m}\simeq \QCoh_{\mathbb P^\mathsf{tw}},
$$
where $q\colon X^{\lp}\times \mbb P^1 \to (X^\lp \times {{\mathbb P^1}})/_{\!\gamma_p, \chi_w} \bG_m$ denotes the quotient map.

\end{pr}
\begin{rem}
    In other words, the quasi-coherent sheaf of abelian categories $\QCoh_{\mathbb P^\mathsf{tw}}$ descends from $X^{\lp}\times \mbb P^1$ to the quotient stack $(X^\lp \times {{\mathbb P^1}})/_{\!\gamma_p, \chi_w} \bG_m$.
\end{rem}
\begin{proof}
     We explain the construction of $\QCoh_{\mathbb P^\mathsf{tw}\!/\mbb G_m}$. Recall from \S \ref{subs.addingmultiplicativegroup}  the homomorphism  $\lambda\colon \bG_m\to \Aut (A_0)$ and denote by $\lambda_{w}$ its pre-composition with the homomorphism $t_{w}\colon \bG_m \to \bG_m$, $t^*_{w}(z)=z^{-w}$.   This action of $\bG_m$ on $A_0$ normalizes $G_0\subset \Aut (A_0)$. Consider the  semidirect product $\bG_m \ltimes_{\Ad_{\lambda_{w}}} G_0$. Let $\pi\colon \cM_{X, [\eta]}\to {X^\lp}$ be the $G_0$-torsor of  Darboux frames. Then, using $\lambda_w$ and $\gamma$,  the action of $G_0$ on $\cM_{X, [\eta]}$ extends to an action of $\bG_m \ltimes_{\Ad_{\lambda_w}} G_0$ making $\pi$ equivariant with respect to this larger group. This action is defined as follows: if $(f\colon T\ra X^\lp, g\colon \Spec A_0 \times T \xra{\sim} X\times_{X^\lp, F^\lp_X, f} T)$ is a $T$-point of $\mc M_{X,[\eta]}$ (see Construction \ref{constr:torsor of Darboux frames}), then the action of $z\in \mathbb G_m(T)$ sends $(f,g)$ to $(\gamma_p(z)\circ f, g')$, where the isomorphism $g'\colon  \Spec A_0 \times T \xra{\sim }X\times_{X^\lp, F^\lp_X,\gamma_p(z)\circ f} T$ is given by the following composition
	\begin{equation}\label{eq:action on Darboux torsor in the homogemeous case}
	    \Spec A_0 \times T \xra{\lambda_w(z)} 	\Spec A_0 \times T \xra{g}  X\times_{X^\lp, F^\lp_X, f} T \xra{\gamma(z)}  X\times_{X^\lp, F^\lp_X,\gamma_p(z)\circ f} T .
	\end{equation}
    It is an isomorphism of restricted symplectic $T$-schemes: $\gamma(z)$ multiplies $[\eta_T]$ by $z^w$ since $[\eta]$ has weight $w$, while $\lambda_w(z)$ multiplies it by $z^{-w}$.

The product   $\cM_{X, [\eta]} \times  {{\mathbb P^1}}$  equipped with diagonal action of $\bG_m \ltimes_{\Ad_{\lambda_w}} G_0$ yields a morphism  of stacks
\begin{equation}\label{app.g_mequiveq}
(X^\lp \times  {{\mathbb P^1}})/_{\!\gamma_p, \chi_w} \bG_m \xra{\sim}  (\cM_{X, [\eta]} \times  {{\mathbb P^1}})/(\bG_m \ltimes_{\Ad_{\lambda_w}} G_0) \to  {{\mathbb P^1}}/_{\!\chi_w}(\bG_m \ltimes_{\Ad_{\lambda_w}} G_0).
\end{equation}
The homomorphism 
$$\bG_m \ltimes_{\Ad_{\lambda_w}} G_0 \xra{(t_w,\id)} \bG_m \ltimes_{\Ad_{\lambda}} G_0$$
and the identity  map on $ {{\mathbb P^1}}$ induce a morphism
$${\mathbb P^1}/_{\!\chi_w}(\bG_m \ltimes_{\Ad_{\lambda_w}} G_0) \to  {{\mathbb P^1}}/_{\!\chi}(\bG_m \ltimes_{\Ad_{\lambda}} G_0),$$
where $\chi:= \chi_{-1}$ is the action in \S \ref{subs.addingmultiplicativegroup}. Let $$\overline{\xi}_{\mathbb P^1}\colon (X^\lp \times  {{\mathbb P^1}})/_{\!\gamma_p, \chi_w} \bG_m \to   {{\mathbb P^1}}/(\bG_m \ltimes_{\Ad_{\lambda}} G_0)$$
be the pre-composition of the latter with the map 
\eqref{app.g_mequiveq}.
Set $$\QCoh_{\mathbb P^\mathsf{tw}\!/\mbb G_m}\coloneqq  \overline{\xi}_ {{\mathbb P^1}}^*((\cA_ {{\mathbb P^1}}\lmod)_{{{\mathbb P^1}}\!/(\bG_m \ltimes G_0 )}),$$ where  $(\cA_ {{\mathbb P^1}}\lmod)_{{{\mathbb P^1}}\!/(\bG_m \ltimes G_0 )}$ is the quasi-coherent sheaf of abelian categories provided by Proposition \ref{app.G_mequivarinceprop}. By construction, the pullback of $(\QCoh_{\mathbb P^\mathsf{tw}})_{(X^\lp \times {{\mathbb P^1}})/_{\!\gamma_p, \chi_w} \bG_m}$ to $X^\lp \times {{\mathbb P^1}}$ is  $\QCoh_{\mathbb P^\mathsf{tw}}$.

\end{proof}

\begin{rem} The fiber over $\{0\}/(\bG_m\ltimes G_0)\hookrightarrow \mbb P^1/(\bG_m\ltimes G_0)$ of $\cA_{\mathbb P^1}\lmod_{\mbb P^1/(\bG_m\ltimes G_0)}$ is given by $A_0\lmod$ where $A_0$ is endowed with the $\bG_m\ltimes G_0$-action considered above. One checks using the formula for the action \eqref{eq:action on Darboux torsor in the homogemeous case} that the pull-back of $A_0$ via the map $\overline{\xi}_{\mathbb P^1}$ is given by the push-forward $F_{X*}^{\lp}\mc O_{X/\mbb G_m}$ with respect to the natural map $F_X^{\lp}\colon X/_{\!\gamma}\mbb G_m\to X^\lp/_{\!\gamma_p}\mbb G_m$.
    This map is affine, and so the category of quasicoherent sheaves of $F_{X*}^{\lp}\mc O_{X/\mbb G_m}$-modules is identified with the category of quasi-coherent sheaves on $X/_{\!\gamma}\bG_m$; consequently, the category of global sections of  $(\QCoh_{\mathbb P^\mathsf{tw}})_{(X^\lp \times {{\mathbb P^1}})/_{\!\gamma_p, \chi_w} \bG_m}$  over $(X^\lp \times \{0\})/_{\!\gamma_p} \bG_m\to (X^\lp \times {{\mathbb P^1}})/_{\!\gamma_p, \chi_w} \bG_m$ 
 is identified with $\QCoh(X/_{\!\gamma}\bG_m)$.  
\end{rem}

\subsection{The canonical quantization of the product.} Given $(X,[\eta])$ let us denote by $\QCoh_{X,[\eta],\mathbb P^\mathsf{tw}}$ the canonical twistor categorical quantization: namely, the quasi-coherent sheaf of categories on $X^\lp\times {\mathbb P^1}$ which we denoted $\QCoh_{\mathbb P^\mathsf{tw}}$ in Section \ref{ssec:construction of the canonical quantization}. For brevity let us employ the following notation for global sections: 
$$\QCoh_{\mathbb P^\mathsf{tw}}(X,[\eta]):=\QCoh_{X,[\eta],\mathbb P^\mathsf{tw}}(X^{\lp}\times \mathbb P^1).$$ Let $(X_1,[\eta_1])$ and $(X_2,[\eta_2])$ be restricted pseudo-smooth $S$-schemes as in Section \ref{ssec:construction of the canonical quantization}. The relative product $X_1\times_S X_2$ is again pseudo-smooth, and if $\rk(\Omega^1_{X_i/S})=2n_i$ then $\rk(\Omega^1_{X_1\times_S X_2/S})=2(n_1+n_2)$. We can endow it with the restricted structure given by $p_1^*[\eta_1]+p_2^*[\eta_2]$. Recall that by Lemma \ref{lem:multiplicativity} the natural map  
$$\iota\colon (X_1\times_S X_2)^{\lp}\to X_1^{\lp}\times_S X_2^{\lp}$$
is an isomorphism.

\begin{pr}[K\"unneth formula]\label{prop:QCoh_h of the product}
    There is a natural equivalence 
    $$
    \QCoh_{X_1\times_S X_2,p_1^*[\eta_1]+p_2^*[\eta_2],\mathbb P^\mathsf{tw}}\simeq p_1^*\QCoh_{X_1,[\eta_1],\mathbb P^\mathsf{tw}}\otimes p_2^*\QCoh_{X_2,[\eta_2],\mathbb P^\mathsf{tw}}
    $$
     of quasicoherent sheaves of abelian categories over $(X_1\times_S X_2)^{\lp} \times {\mathbb P^1}$. Here, on the right, $p_i\colon (X_1 \times_S X_2)^\lp \times {\mathbb P^1} \ra X_i^\lp\times {\mathbb P^1}$ denote the projections relative to ${\mathbb P^1}$.
\end{pr}
\begin{proof}
    Let $2n_i$ denote the rank of $\Omega^1_{X_i}$. Let $\mc A_{n,{\mathbb P^1}}$ denote the twistor reduced Weyl algebra from  \S\ref{ssec:twistor reduced Weyl algebra} for dimension parameter equal to $n$. Similarly, let $G_{n,0}$ be the group scheme from Construction \ref{constr:G_0} for dimension parameter $d=n$. By construction, $\QCoh_{X_1\times_S X_2,p_1^*[\eta_1]+p_2^*[\eta_2],\mathbb P^\mathsf{tw}}$ is obtained as the pull-back of the sheaf of categories $\mc A_{n_1+n_2,{\mathbb P^1}}\lmod^{G_{n_1+n_2,0}}$ on $BG_{n_1+n_2,0}\times {\mathbb P^1}$ under the map $(X_1\times_S X_2)^\lp\times {\mathbb P^1}\ra BG_{n_1+n_2,0}\times {\mathbb P^1}$ classifying the torsor $\mc M_{X_1\times_SX_2, p_1^*[\eta_1]+p_2^*[\eta_2]}$ of Darboux-frames for $p_1^*[\eta_1]+p_2^*[\eta_2]$. Consider the subgroup $i\colon H:=G_{n_1,0}\times G_{n_2,0}\hookrightarrow G_{n_1+n_2,0}$. There is a natural $H$-torsor over $(X_1\times_S X_2)^{\lp}\simeq X_1^\lp\times_S X_2^\lp$ given by $\mc M_{X_1,[\eta_1]}\times_S \mc M_{X_2,[\eta_2]}$. There is a natural $H$-equivariant morphism $\mc M_{X_1,[\eta_1]}\times_S \mc M_{X_2,[\eta_2]}\to \mc M_{X_1\times_SX_2, p_1^*[\eta_1]+p_2^*[\eta_2]}$ over $(X_1\times_S X_2)^{\lp}$: indeed, given two points $s_i\colon T\to X_i^\lp$ with framing $\alpha_i\colon X_i\times_{X_i^\lp} T\simeq \Spec A_{n_i,0}\times T$, then $\alpha_1\times_T \alpha_2$ provides a framing 
    $$(X_1\times_SX_2)\times_{(X_1\times_S X_2)^{\lp}} T\simeq (\Spec A_{n_1,0}\times T)\times_T (\Spec A_{n_2,0}\times T)\simeq \Spec A_{n_1+n_2,0} \times T.$$
    It follows that the map $p_1^*[\eta_1]+p_2^*[\eta_2]\colon (X_1\times_S X_2)^{\lp} \to BG_{n_1+n_2,0}$ factors as the composition of $([\eta_1],[\eta_2])\circ \iota \colon (X_1\times_S X_2)^{\lp} \to BG_{n_1,0}\times BG_{n_2,0}$ and $i\colon BG_{n_1,0}\times BG_{n_2,0}\to BG_{n_1+n_2,0}$.
    
This way, for the statement of the proposition it will be enough to construct a natural equivalence 
    $$
i^*\mc A_{n_1+n_2,{\mathbb P^1}}\lmod^{G_{n_1+n_2,0}} \iso p_1^*\mc A_{n_1,{\mathbb P^1}}\lmod^{G_{n_1,0}}\otimes p_2^*\mc A_{n_2,{\mathbb P^1}}\lmod^{G_{n_2,0}}
    $$
    of sheaves of categories on $BG_{n_1,0}\times BG_{n_2,0} \times {\mathbb P^1}$. Note that there is a natural isomorphism $\mc A_{n_1,{\mathbb P^1}}\otimes_{\mc O_{{\mathbb P^1}}} \mc A_{n_2,{\mathbb P^1}}\simeq \mc A_{n_1+n_2,{\mathbb P^1}}$, and so both sides above are defined by some actions of the subgroup $G_{n_1,0}\times G_{n_2,0}\hookrightarrow G_{n_1+n_2,0}$ on $\mc A_{n_1+n_2,{\mathbb P^1}}\lmod$. They also naturally come with the data required to apply Theorem \ref{appendix:maintheorem}, the uniqueness part of which then provides us with the above equivalence.
\end{proof}

Recall the notion of a dualizable quasicoherent sheaf of abelian categories (Construction \ref{constr:dualizable sheaves of categories}). Given a restricted symplectic $S$-scheme $(X,[\eta])$, one can consider its ``dual'' $(X,[-\eta])$. The association $X,[\eta]\mapsto \QCoh_{X,[\eta],\mathbb P^\mathsf{tw}}$ intertwines these two notions of duality.
\begin{pr}[Duality]\label{prop:objects on the product as functors} The quasi-coherent sheaf of categories $\QCoh_{X,[\eta],\mathbb P^\mathsf{tw}}$ over $X^\lp\times {\mathbb P^1}$ is dualizable. There is a natural equivalence 
\begin{equation}\label{eq:continuous dual of the quantization}
    \QCoh_{X,[\eta],\mathbb P^\mathsf{tw}}^\vee\simeq\QCoh_{X,[-\eta],\mathbb P^\mathsf{tw}}.
\end{equation}

\end{pr}
\begin{proof}
    The sheaf $\QCoh_{X,[\eta],\mathbb P^\mathsf{tw}}$ is as in Example \ref{ex:dualizable categories}\eqref{item:locally given by modules is dualizable}, and thus is dualizable. It remains to construct the equivalence $\eqref{eq:continuous dual of the quantization}$.

    Recall the automorphism $\beta_0$ of $A_0$ from the discussion in \S\ref{ssec:dual action}. Since by construction $\QCoh_{X,[\eta],\mathbb P^\mathsf{tw}}\simeq [\eta]^* \mc A_{{\mathbb P^1}}\lmod^{G_0}$ where $[\eta]\times \id\colon X^\lp\times {\mathbb P^1}\ra BG_0\times {\mathbb P^1}$, we have 
    $$\QCoh_{X,[\eta],\mathbb P^\mathsf{tw}}^\vee\simeq \left(([\eta]\times \id)^* \mc A_{{\mathbb P^1}}\lmod^{G_0}\right)^\vee \simeq ([\eta]\times \id)^* \left(\mc A_{{\mathbb P^1}}\lmod^{G_0}\right)^\vee\simeq (\beta_0 \circ [\eta]\times \id)^* \mc A_{{\mathbb P^1}}\lmod^{G_0},$$ 
    where the last equivalence is given by Proposition \ref{prop: dual action on A_S-mod}. Note that since $\beta_0^*([\eta_{\mr{can}}])=[-\eta_{\mr{can}}]$ the composition 
    $$
    X^\lp\xra{[\eta]} BG_0 \xra{\beta_0} BG_0
    $$
    classifies the restricted structure on $X$ given by $[-\eta]$. Together with the above this gives an equivalence
    $$
\QCoh_{X,[\eta],\mathbb P^\mathsf{tw}}^\vee \simeq ([-\eta]\times \id)^* \mc A_{{\mathbb P^1}}\lmod^{G_0}\simeq \QCoh_{X,[-\eta],\mathbb P^\mathsf{tw}}.
    $$
\end{proof}

\begin{rem}[{Fourier-Mukai action by objects in $\QCoh_{\mathbb P^\mathsf{tw}}(X_1\times_S X_2,-p_1^*[\eta_1]+p_2^*[\eta_2])$}]\label{rem:action by Lagrangian correspondences} Let $X_1$ and $X_2$ be smooth over $S$. 
By Proposition \ref{prop:QCoh_h of the product} applied to $(X_1,[-\eta_1])$ and $(X_2,[\eta_2])$, and Proposition \ref{prop:objects on the product as functors} we have an equivalence of quasicoherent sheaves of abelian categories on $X_1^{\lp}\times_S X_2^{\lp} \times \mathbb P^1$
$$
\QCoh_{X_1\times_S X_2,-p_1^*[\eta_1]+p_2^*[\eta_2],\mathbb P^\mathsf{tw}}\simeq \QCoh_{X_1,[\eta_1],\mathbb P^\mathsf{tw}}^\vee \boxtimes \QCoh_{X_2,[\eta_2],\mathbb P^\mathsf{tw}}.
$$ Let $\widehat{D}(\QCoh_{X_1,[\eta_1],\mathbb P^\mathsf{tw}})$ and $\widehat{D}(\QCoh_{X_2,[\eta_2],\mathbb P^\mathsf{tw}})$ be the corresponding quasicoherent sheaves of stable categories (Construction \ref{constr:derived category}). Then via the Fourier-Mukai transform (Construction \ref{constr:FM-transform ii}), given an object 
$$
\mc M\in \QCoh_{\mathbb P^\mathsf{tw}}(X_1\times_S X_2,-p_1^*[\eta_1]+p_2^*[\eta_2])
$$
in the category of global sections, one naturally associates 
a functor $$\mr{FM}_{\mc M}\colon \widehat{D}(\QCoh_{\mathbb P^\mathsf{tw}})(X_1,[\eta_1])\to \widehat{D}(\QCoh_{\mathbb P^\mathsf{tw}})(X_2,[\eta_2]).$$ Unravelling Construction \ref{constr:FM-transform 1}, the functor $\mathrm{FM}_{\mc M}$ morally takes an object of $\widehat{D}(\QCoh_{\mathbb P^\mathsf{tw}})(X_1,[\eta_1])$, pulls it back to $X_1^{\lp}\times_S X_2^{\lp}\times \mbb P^1$, takes derived tensor with $\mc M$ and then takes derived pushforward to $X_2^{\lp}\times \mbb P^1$.

As an example, by Theorem \ref{app.thquantoflarg}, a restricted Lagrangian $L$ in $X_1\times_S X_2$ with respect to the restricted structure $-p_1^*[\eta_1]+p_2^*[\eta_2])$ can be quantized to a canonical object $i_*\mc V_{L^\lp\times \mbb P^1}$ (here $i$ is the embedding $L^\lp\times \mbb P^1\to (X_1\times_S X_2)^\lp\times \mbb P^1$) in $\QCoh_{\mathbb P^\mathsf{tw}}(X_1\times_S X_2,-p_1^*[\eta_1]+p_2^*[\eta_2])$. This way any such $L$ defines a functor
$$
\mr{FM}_{L}^{\mr{quant}}\colon \widehat{D}(\QCoh_{\mathbb P^\mathsf{tw}})(X_1,[\eta_1])\to \widehat{D}(\QCoh_{\mathbb P^\mathsf{tw}})(X_2,[\eta_2]).
$$
In the case when $L$ is the graph of an \'etale map $f\colon X_2\to X_1$ (and $[\eta_2]=f^*[\eta_1]$), one has an identification $\QCoh_{X_2,[\eta_2],\mathbb P^\mathsf{tw}} \simeq f^*\QCoh_{X_1,[\eta_1],\mathbb P^\mathsf{tw}}$ and one can show that $\mr{FM}_{L}^{\mr{quant}}$ is given by the natural pull-back functor 
$$
f^*_{\widehat{D}(\QCoh_{X_1,[\eta_1],\mathbb P^\mathsf{tw}})}\colon \widehat{D}(\QCoh_{\mathbb P^\mathsf{tw}})(X_1,[\eta_1]) \to \widehat{D}(\QCoh_{\mathbb P^\mathsf{tw}})(X_2,[\eta_2])
$$
in the notation of Construction \ref{constr:sheaves of stable cats}.
\end{rem}
\subsection{The canonical twistor categorical quantization of the cotangent bundle.}\label{app.cotangentbundle} As before,  assume $p>2$ and let $S$ be a base $\mathbb F_p$-scheme. The goal of this section is to describe $\QCoh_{X,[\eta],\mathbb P^\mathsf{tw}}$ in the case of the cotangent bundle.

\begin{pr}\label{prop:twistor quantization for cotangent bundle}
    Let $Y$ be a pseudo-smooth scheme over $S$. Let $X\coloneqq T^*Y$ be the cotangent bundle, endowed with the restricted symplectic $S$-scheme structure given by  the Liouville form $\eta:=\eta_\mathsf{can}\in \Omega^1_{X/S}(X)$ (Example \ref{ex:restricted structure on T^*X}). Then there is a natural equivalence
    $$
\QCoh_{X,[\eta],\mathbb P^\mathsf{tw}}\simeq \mc D_{Y,\mbb P^1}\lmod
    $$
    of quasi-coherent sheaves of abelian categories on $ X^{\lp}\times \mbb P^1$. Here $\mc D_{Y,\mbb P^1}$ denotes the sheaf of twistor differential operators (Definition \ref{def:tdo}), considered as a quasi-coherent sheaf of algebras on $X^\lp \!\times {\mathbb P^1}$.
\end{pr}
\begin{proof}

Let $H:=H_{\mr{pol}}\subset G_0$ (in the notation of Lemma \ref{lem:polarization stabilizer is connected}) be the subgroup of restricted automorphisms  of $A_0$ that carry the subalgebra $$C_0= \bF_p[x_1, \cdots , x_n]/(x_i^p)_{1\leq i \leq n}\subset A_0$$ to itself and let $\cM_{X, [\eta], H} \mono \cM_{X, [\eta]}$ be a 
closed subscheme of Darboux \textit{polarized} frames: in other words, frames compatible with the projection $X\to Y$. That is, a $T$-point of  $\cM_{X, [\eta], H}$ is a pair $(i_\nu, \alpha)$, where   $i_\nu: T \to X^\lp$ is a morphism
(that itself can be thought as a pair consisting of a morphism  $T \to Y^\lp$ and global section $\nu$ of the sheaf $ \Omega_{Y_T/T}^{1, \lp}$ for $Y_T\coloneqq Y \times_{Y^\lp} T$)  and $\alpha$ is an isomorphism 
$X_T\coloneqq X \times_{X^\lp} T \iso \Spec A_0 \times T$ of restricted symplectic schemes over $T$ that  fits into a commutative diagram
\begin{equation}\label{diag:reductionoftorsor}
\def\normalbaselines{\baselineskip20pt
	\lineskip3pt  \lineskiplimit3pt}
\def\mapright#1{\smash{
		\mathop{\to}\limits^{#1}}}
\def\mapdown#1{\Big\downarrow\rlap
	{$\vcenter{\hbox{$\scriptstyle#1$}}$}}
\begin{tikzcd}
 X \times_{X^\lp} T\arrow[r,"\alpha","\sim"']\arrow[d] & \Spec A_0 \times T \arrow[d] \\
Y \times_{Y^\lp} T  \arrow[r,"\zeta","\sim"']  &\Spec C_0 \times T,
\end{tikzcd}
\end{equation} 
for some $\zeta$ (uniquely determined by $\alpha$). The group $H$ acts on $\cM_{X, [\eta], H}$ and we claim that   $\pi_H\colon \cM_{X, [\eta], H} \to X^\lp$ is an $H$-torsor. Indeed, by construction it is a pseudo-torsor, and we only need to show it is locally non-empty. Let $i_\nu\colon T=\Spec R\to X^\lp$ be a point, thought of as a pair of a map $s\colon T\to Y^{\lp}$ and a section $\nu\in \Omega_{Y_T/T}^{1, \lp}$.
Fppf-locally on $T$, we can find a Frobenius coordinate frame
\[
  \zeta\colon Y  \times_{Y^\lp} T \xrightarrow{\sim}
  \operatorname{Spec}(C_0)\times T.
\]
If $\nu=0$, then $X_T$ is a Frobenius-neighborhood of zero section in $X=T^*Y$, which via $\zeta$ has a natural identification with $\Spec A_0\times T$. Otherwise, after a further fppf covering of $T$, we can also find a one-form $\mu\in\Gamma(Y_T,\Omega^1_{Y_T/T})$ such that
$\mu^{\lp}=\nu$ and $[\mu]=0$: indeed, via the isomorphism $\zeta$ we have $\Omega_{Y_T/T}^{1, \lp}\simeq \oplus_{i=1}^nR\cdot dx_i$ and $\nu=\sum_{i=1}^n a_idx_i$ for some $a_i\in R$. Adjoining to $R$ elements $u_i$ such that $u^p_i=a_i$ we can take $\mu=\sum_{i=1}^n u_idx_i\in \Gamma(Y_T,\Omega^1_{Y_T/T})$; one has $\mu^{\lp}=\sum_{i=1}^n u_i^pdx_i=\nu$ and $\mu=d(\sum_{i=1}^n u_ix_i)$ is exact. 
Then the composite isomorphism 
$$
X\times_{X^\lp\!,\iota_\nu}T \xra{\text{translation by $-\mu$}} X\times_{X^\lp\!,\iota_0}T\begin{tikzcd}
 \arrow[r,"\zeta","\sim"'] & \Spec A_0\times T
\end{tikzcd}
$$
gives an isomorphism $\alpha$ (it is restricted since $[\mu]=0$).

We wish to construct an isomorphism of algebras
\begin{equation}\label{eq:pullbackoftwistordiffoperators}
 (\pi_H  \times \Id)^* \mc D_{Y,{\mathbb P^1}}  \iso \pro_{\cM_{X, [\eta], H} }^* \cA_{{\mathbb P^1}}.
 \end{equation}
 To do this, observe that a pair $(i_\nu, \alpha)$ as above uniquely determines a  $1$-form $\mu$ on $Y_T$ relative to $T$ with $\mu^\lp = \nu$:  the graph $i_\mu\colon  Y_T\mono X_T$ of $\mu$
is the preimage of the closed subscheme of $\Spec A_0 \times T$ specified by equations $y_i=0$, $1\leq i\leq n$. 
 The image $[\mu]$ in
 $$H^0(Y_T, \coker (\cO_{ Y_T} \rar{d}  \Omega^1_{Y_T}))$$ is $0$, because $\alpha$ is restricted.
(Conversely, as we saw earlier, $(\nu, \zeta, \mu)$ as above determines $\alpha$.) The pullback of $\mc D_{Y,{\mathbb P^1}} $ along the map $T\times {\mathbb P^1} \to X^\lp  \times {\mathbb P^1} $
 is the central reduction of twistor differential operators $$\gamma_\mu^{-1}\colon \mc D_{Y_T/T, \frac{\nu}{h^p}, {\mathbb P^1}} \xra{\sim}\mc D_{Y_T/T, \frac{[\mu]}{h}, {\mathbb P^1}} =  \mc D_{Y_T/T, [0], {\mathbb P^1}}$$
(see Construction \ref{constr:central reductions over S}). Map $\zeta$ determines an isomorphism $ \mc D_{Y_T/T, [0], {\mathbb P^1}}\iso  \cA_{{\mathbb P^1}} \boxtimes \cO_T$. This gives (\ref{eq:pullbackoftwistordiffoperators}) and, thus,  an equivalence  
 $$\Theta_H\colon (\pi_H  \times \Id)^*\mc D_{Y,{\mathbb P^1}} \lmod \iso  \pro_{\cM_{X, [\eta], H} }^*\cA_{\mbb P^1}\lmod .$$ 
  In addition, as explained in Remark \ref{rem:twistor differential operators explicitly} the restriction of  $\mc D_{Y,{\mathbb P^1}}$ to $X^\lp \times \{\infty\}$, being isomorphic to the pullback of   $\End_{\cO_{Y^\lp}} (F_{Y*}^\lp\cO_Y)$ 
via the projection $\pro^\lp: X^\lp \to Y^\lp$, is a canonically split
Azumaya algebra. 
 The splitting yields  an equivalence
 $$\Xi_\infty\colon (\mc D_{Y,{\mathbb P^1}} \lmod)_{| X^\lp\times \{\infty\}} \iso \cO_{X^\lp}\lmod.$$
Furthermore, the natural trivialization of the pull-back of the vector bundle $F_*^\lp\cO_Y$   to $\cM_{X, [\eta], H}$ yields an isomorphism $\varpi_H $ making $(\mc D_{Y,{\mathbb P^1}}\lmod,  \Theta_H, \Xi_\infty, \varpi_H)$ the canonical quantization\footnote{To apply Proposition \ref{mainth:quantofthecategory} we need to verify that $H$ is connected. This is Lemma \ref{lem:polarization stabilizer is connected}.} of 
$\QCoh(X)$.  In particular, we get an equivalence 
\begin{equation}\label{eq:iso between QCoh and differential operators}
\mc D_{Y,{\mathbb P^1}}\lmod \xra{\sim}  \QCoh_{X,[\eta],\mathbb P^\mathsf{tw}}
\end{equation}
of quasi-coherent sheaves of abelian categories on $X^\lp\times \mathbb P^1$.

\end{proof}

Next, we wish to compute the equivalence $
\Xi_0\colon (\QCoh_{\mathbb P^\mathsf{tw}})_{|X^\lp \times \{0\}} \iso  F_*^\lp \cO_X\lmod$ given by Proposition \ref{mainth:xi0}. Let us clarify what we mean by ``compute'' here. Namely, (see Remark \ref{rem:twistor differential operators explicitly}) the description of $\mc D_{Y,{\mathbb P^1}\backslash \{\infty\}}$ as the Rees construction applied to  $\mc D_{Y}$ endowed with the order filtration (also called the Hodge filtration) yields an (a priori unrelated) isomorphism
\begin{equation}\label{app.cotangentfiberat0}
(\mc D_{Y,{\mathbb P^1}})_{|X^\lp \times \{0\}} \iso F_{X *}^\lp \cO_X;
\end{equation}
this way $\Xi_0$ can be viewed as an $F_{X *}^\lp \mc O_X$-linear equivalence
$$
F_{X *}^\lp \cO_X\lmod \xra{\text{\eqref{app.cotangentfiberat0}}^{-1} } (\mc D_{Y,{\mathbb P^1}})_{|X^\lp \times \{0\}}\lmod \xra{\text{\eqref{eq:iso between QCoh and differential operators}}} (\QCoh_{\mathbb P^\mathsf{tw}})_{|X^\lp \times \{0\}} \xra{\Xi_0} F_{X *}^\lp \cO_X\lmod.
$$

\begin{pr}
The above composition is given by the formula
 $$M\mapsto M \otimes_{F_{X*}^\lp\!\cO_X} \!\! F_{X *}^\lp \mathrm{pr}^* K_Y^{\otimes\frac{1-p}{2}},$$
 where $K_Y: =\Omega^n_{Y/S}$ and $\mathrm{pr}\colon X \to Y$ is the projection. 
 \end{pr}
 \begin{proof}
 Recall that if $M$ is a left $\cD_Y$-module  then $M\otimes _{\cO_Y} \!\!K_Y$ has the structure of a right $\cD_Y$-module. This extends to an equivalence $\mc D_{Y,{\mathbb P^1}}\lmod \iso  ({\Id}\times \iota)^*\mc D_{Y,{\mathbb P^1}}^{{\op}}\lmod \cong ({\Id}\times \iota)^*(\mc D_{Y,{\mathbb P^1}}\lmod)^\circ$. The same is true for the functor $\Sigma_{h=1}$
 that carries a  left $\cD_Y$-module  $M$ to  the right $\cD_Y$-module  $M\otimes _{\cO_Y} K_Y^{\otimes 1-p} \iso (M\otimes _{\cO_Y} K_Y^{\otimes-p}) \otimes _{\cO_Y}  K_Y$, where  $K_Y^{\otimes-p}$ is equipped with the Frobenius descent connection. We wish to promote the corresponding equivalence  $\Sigma\colon \mc D_{Y,{\mathbb P^1}}\lmod \iso ({\Id}\times \iota)^*(\mc D_{Y,{\mathbb P^1}}\lmod)^\circ$, $M\mapsto M \otimes _{\cO_Y} K_Y^{\otimes 1-p}$ to an equivalence of quadruples
 $$(\mc D_{Y,{\mathbb P^1}}\lmod,  \Theta_H, \Xi_\infty, \varpi_H)\iso (({\Id}\times \iota)^*\mc D_{Y,{\mathbb P^1}}\lmod ^\circ,  \Theta^{\mr{dual}}_H, \Xi^{\mr{dual}}_\infty, \varpi^{\mr{dual}}_H)$$
 from \S\ref{ssec:construction of the canonical quantization}.

 We will indicate this identification. First, let us explain how to make the following diagram of functors $2$-commutative:
 \begin{equation}\label{app.ctwocomm}
	\begin{tikzcd}
	(\mc D_{Y,{\mathbb P^1}})_{| X^\lp\times \{\infty\}}\lmod  \arrow[d, "\Xi_\infty"']    \arrow[r, "\Sigma"] & (\mc D_{Y,{\mathbb P^1}}^{{\op}})_{| X^\lp\times \{\infty\}}\lmod   \arrow[d, "\Xi_\infty^{{\op}}"] \\
	\QCoh_{X^\lp}  \arrow[r, "\Id"] & \QCoh_{X^\lp}. 
	\end{tikzcd}
	\end{equation} 
All functors in the diagram are equivalences of $\QCoh_{X^\lp}$-linear categories and, since $\QCoh_{X^\lp}$ is generated over itself by $\mc O_{X^\lp}$, it is enough to construct a natural isomorphism $(\Sigma\circ\Xi_\infty^{-1})(\mc O_{X^\lp}) \xra{\sim} \Xi_\infty^{{\op},-1}(\mc O_{X^\lp})$. The algebra  $(\mc D_{Y,{\mathbb P^1}})_{| X^\lp\times \infty}$ is identified with that of $\mc O_{X^\lp}$-linear endomorphisms of $\mathrm{pr}^{\lp *}\! F_{Y*}^\lp\cO_Y$. In these terms the corresponding functors $\Xi_\infty^{-1}$ and $\Xi_\infty^{{\op},-1}$ are given by $-\otimes_{\cO_{X^\lp}} \mathrm{pr}^{\lp *}\! F_{Y*}^\lp\cO_Y$ and $\mathcal{H}\! \mathit{om}_{\cO_{X^\lp}}(\mathrm{pr}^{\lp *}\!F_{Y*}^\lp \cO_Y, -)$ correspondingly. Recall that in Remark \ref{rem:invariance of trace} we constructed a natural isomorphism of right $F_{Y*}^{\lp}\mc D_Y$-modules 
  \begin{equation}\label{app.quantofcotangentandserre}
 F_{Y*}^\lp K_Y^{1-p}  \iso \mathcal{H}\! \mathit{om}_{\cO_{Y^\lp}}\!(F_{Y*}^\lp\cO_Y, \cO_{Y^\lp}).
 \end{equation} 
Applying $\mathrm{pr}^{\lp*}$ to \eqref{app.quantofcotangentandserre} then gives a desired isomorphism.

Let $\pi_H\colon \mc M:=\mc M_{X,[\eta],H}\to X^\lp$ be the Darboux torsor of polarized frames. For compatibility of $\Theta_H$, we need to make the diagram
 $$
 	\begin{tikzcd}
	(\pi_H\times \id)^*\mc D_{Y,{\mathbb P^1}}\lmod  \arrow[d, "\Theta_H"']    \arrow[rr, "(\pi_H\times \id)^*\Sigma"] && (\pi_H\times \iota)^*\mc D_{Y,{\mathbb P^1}}^{\op}\lmod   \arrow[d, "(\id \times \iota)^*(\Theta_H^{{\mr{op}}})"] \\
	\pro_\cM^*\cA_{\mbb P^1}\lmod  \arrow[rr, "\pro_\cM^*(\alpha_*)"] && \pro_\cM^*\iota^*\cA_{\mbb P^1}^{\mr{op}}\lmod
	\end{tikzcd} 
 $$ (with the notation as in \S\ref{ssec:construction of the canonical quantization}) 2-commutative. The functor $(\pi_H\times \id)^*\Sigma$ is given by sending a module $M$ to $M\otimes_{\mc O_{Y_\mc M}} \!\pi_H^*\pro^{\lp,*}
\bigl(F_{Y*}^{\lp}K_Y^{\otimes(1-p)}\bigr)$ (here $\mc O_{Y_\mc M}$ is the pushforward of the structure sheaf under $Y_\cM:=Y\times_{Y^\lp}\mc M\to \mc M$, and $\mc M\to Y^\lp$ is the composition of $\pi_H$ and $\pro^{\lp}$). The universal isomorphism 
 $$
u\colon Y\times_{Y^\lp} \mc M \simeq \Spec C_0\times \mc M
 $$
gives an identification of $\pi_H^*\pro^{\lp,*}
\bigl(F_{Y*}^{\lp}K_Y^{\otimes(1-p)}\bigr)$ with  $(\Omega^n_{C_0})^{\otimes{1-p}}\otimes_{\mbb F_p}\mc O_\cM$. A trivialization of $(\Omega^n_{C_0})^{\otimes{1-p}}$ given by $\rho:=(dx_1\wedge\ldots \wedge dx_n)^{\otimes {1-p}}$ induces an identification $$
M\xrightarrow{\sim} M\otimes_{\mc O_{Y_\mc M}} \!\pi_H^*\pro^{\lp,*}
\bigl(F_{Y*}^{\lp}K_Y^{\otimes(1-p)}\bigr), \quad m\mapsto m\otimes \rho .
$$
It is then a direct check that the $\pro_\cM^*\cA_{\mbb P^1}$-action via $\Theta_H^{-1}$ on the left-hand side and the $\pro_\cM^*\iota^*\cA_{\mbb P^1}^{\mr{op}}$-action on the right-hand side via $(\id \times \iota)^*(\Theta_H^{{\mr{op}}})^{-1}$ are exactly intertwined by $\pro_\cM^*(\alpha)$. The compatibility of these isomorphisms with $\varpi_H$ and $\varpi_H^{\mr{dual}}$ is checked similarly.

 Let us return to the proof of the proposition. We have the following diagram of $F_{X*}^\lp\cO_X$-linear equivalences of quasicoherent sheaves of abelian categories over $X^\lp$:
 $$
 \xymatrix{ F_{X *}^\lp \cO_X\lmod \ar[r]^{-\otimes F_{X *}^\lp\mathrm{pr}^* \!\!K_Y^{1-p}}\ar[d]_{\sim}^{\text{\eqref{eq:iso between QCoh and differential operators}}}& F_{X *}^\lp \cO_X\lmod \ar@{=}[r]&F_{X *}^\lp \cO_X\lmod^\circ \\
(\QCoh_{\mathbb P^\mathsf{tw}})_{{|X^\lp \times \{0\}}} \ar[rr]_{\sim}^{\Sigma_0} \ar[d]^{\Xi_0}&& (\QCoh_{\mathbb P^\mathsf{tw}}^\circ)_{{|X^\lp \times \{0\}}}\ar[u]_{\sim}^{(\text{\eqref{eq:iso between QCoh and differential operators}}^\circ)^{-1}} \\
F_{X *}^\lp \cO_X\lmod \ar@{=}[rr]&&F_{X *}^\lp \cO_X\lmod^\circ \ar[u]^{(\Xi_0^\circ)^{-1}}.
}
 $$
 We need to show that the composition of the vertical arrows on the left sends $ F_{X *}^\lp\mc O_X$ to $ F_{X *}^\lp \mathrm{pr}^* K_Y^{\frac{1-p}{2}}$. Let $L$ be the line bundle given by the image of $F_{X *}^\lp\mc O_X$ under this composition.  Note that the composition of the right vertical arrows is also given by tensor product with $L$. By definition of $\Xi_0$ (see Proposition \ref{mainth:xi0})  $L$ is equipped with a trivialization
 of its pullback to $\cM_{X, [\eta],  H}$
 and a compatible isomorphism $L \otimes_{F_{X*}^\lp\cO_X} \! L  \iso  F_{X *}^\lp  \mathrm{pr}^* K_Y^{1-p}$ of $ F_{X*}^\lp \cO_X$-modules (see the diagram above). The module $ F_{X *}^\lp \mathrm{pr}^* K_Y^{\frac{1-p}{2}}$ comes equipped with the above structure, thus the uniqueness part of  Proposition \ref{mainth:xi0} completes the proof.
 
\end{proof}

\subsection{\'Etale local description of $\QCoh_{\mathbb P^\mathsf{tw}}$}

The \'Etale Darboux lemma (Proposition \ref{prop:etale Darboux lemma}) shows that \'etale locally on $X^\lp$ the sheaf of categories $\QCoh_{X,[\eta],\mathbb P^\mathsf{tw}}$ is given by modules over a sheaf of algebras.

\begin{pr}
    Let $(X,[\eta])$ be a restricted symplectic $S$-scheme with $\rk(\Omega^1_X)=2n$. Then there is an \'etale cover $q\colon U\ra X$ such that $(q^{\lp}\times \id)^*\QCoh_{X,[\eta],\mathbb P^\mathsf{tw}}$ is equivalent to $\mc B\lmod$ for some sheaf of algebras $\mc B$ over $U^\lp\times \mbb P^1$. More precisely, choosing the Darboux coordinates $p\colon U\ra T^*\mbb A^n$ as in Proposition \ref{prop:etale Darboux lemma} we get an equivalence
    $$
(q^{\lp}\times \id)^*\QCoh_{X,[\eta],\mathbb P^\mathsf{tw}} \simeq ((p^{\lp}\times \id)^* \mc D_{\mbb A^n,\mbb P^1})\lmod
    $$
    of sheaves of categories over $U^\lp\times \mbb P^1$, where $\mc D_{\mbb A^n,\mbb P^1}\in \QCoh((T^*\mbb A^n)\twt \times \mbb P^1)$ is the sheaf of twistor differential operators on $\mbb A^n$ (Definition \ref{def:tdo}).
\end{pr}
\begin{proof}
    By Proposition \ref{prop:etale Darboux lemma}, the maps $q\colon U\to X$ and $p\colon U\to T^*\mbb A^n$ induce a commutative square
    $$
    \begin{tikzcd}
        U^{\lp}\arrow[r,"q^{\lp}"]\arrow[d,"p^{\lp}"']& X^\lp \arrow[d,"{[\eta]}"]\\
        (T^*\mbb A^n)\twt \arrow[r,"{[\eta_{\mathsf{can}}]}"] & BG_0
    \end{tikzcd}
    $$
    which gives an equivalence 
    $$
(q^\lp\times \id)^*\QCoh_{X,[\eta],\mathbb P^\mathsf{tw}} \simeq (p^\lp\times \id)^*(([\eta_{\mathsf{can}}]\times \id)^* (\mc A_{\mbb P^1}\lmod^{G_0}))
    $$
    of sheaves of categories over $U^{\lp}\times \mbb P^1$. The natural equivalence $$
    ([\eta_{\mathsf{can}}]\times \id)^* (\mc A_{\mbb P^1}\lmod^{G_0})\simeq \mc D_{\mbb A^n,\mbb P^1}\lmod
    $$
    from Section \ref{app.cotangentbundle} then provides the required equivalence.
\end{proof}

 \subsection{The first order obstruction as a Rozansky-Witten invariant.}\label{subsection:rw}
Let $(X, [\eta])$ be a pseudo-smooth  scheme over $S$ equipped with a  restricted symplectic structure. The obstruction to lifting a line bundle 
$\cL \in \QCoh(X)\iso \QCoh_{\mathbb P^\mathsf{tw}}(X^\lp \times \{0\})$ to a $\mbb F_p[h]/(h^2)$-flat object of $\QCoh_{\mathbb P^\mathsf{tw}}(X^\lp \times \Spec \bF_p[h]/(h^2))$ is a class (\cite[Theorem A]{Lowen2005})
\begin{equation}\label{eq:obstruction_class_ssrw}
    \mathrm{Ob}_{\cL}\in \Ext^2_{\QCoh_{\mathbb P^\mathsf{tw}}(X^\lp \times \{0\})}(\cL, \cL) \iso H^2(X, \cO_X).
\end{equation}
 In this subsection we relate  $ \mathrm{Ob}_\cL$, for $p>3$, to the Rozansky-Witten invariant of the symplectic variety $(X, \omega)$, corresponding to the simplest trivalent graph $\Theta$, and the Chern class of $\cL$. 
 When $p=3$, we define a
 Rozansky-Witten invariant of a \textit{restricted} symplectic variety $(X, [\eta])$ and prove an analogous comparison result.  In the complex-analytic context  this relation has been first observed by Bressler--Gorokhovsky--Nest--Tsygan (see \cite[Theorem 5.3.1 and Remark 5.3.4]{BGNT2007}).

 We shall use the interpretation of the Rozansky-Witten 
theory proposed (in characteristic $0$) by Kapranov \cite{kap} and Kontsevich \cite{kon}.
Recall that a torsion-free symplectic connection $\nabla$ is a connection on  the tangent bundle $T_X$, such that, for every $\theta_i\in T_X$, $1\leq i\leq 3$,
 \[\nabla _{\theta_1} \theta_2 - \nabla _{\theta_2} \theta_1 =[\theta_1, \theta_2], \quad \omega(\nabla _{\theta_1} \theta_2, \theta_3) +  \omega(\theta_2, \nabla _{\theta_1} \theta_3)=L_{\theta_1}(\omega(\theta_2, \theta_3)).\]
If $\nabla$ and $\nabla'$ are two torsion-free symplectic connections, then the formula
\[ \theta_1 \otimes \theta_2 \otimes \theta_3\mapsto \omega(\nabla _{\theta_1} \theta_2, \theta_3)- \omega(\nabla _{\theta_1}' \theta_2, \theta_3) \]
determines an $\cO_X$-linear map $\Sym^3 T_X\to \cO_X$. Moreover, the sheaf of torsion-free symplectic connections is a torsor under $(\Sym^3 T_X)^\vee\iso \Gamma^3 \Omega_X^1$. We refer to the corresponding cohomology class $\mathrm{At}_{X,\omega } \in H^1(X, \Gamma^3 \Omega_X^1)$ as the \emph{symplectic  Atiyah class}. For $p>3$, using $\Gamma^3 \Omega_X^1 \iso \Sym^3 \Omega_X^1$ and the pairing $B\colon \Sym^3 \Omega_X^1\otimes \Sym^3 \Omega_X^1 \to \cO_X$ defined in \eqref{eq:form_on_sym_three}, we obtain a class 
\[\mathrm{RW}_\Theta (X, \omega):= B\circ (\mathrm{At}_{X,\omega } \cup \mathrm{At}_{X,\omega })\in  H^2(X, \cO_X). \]
This is the Rozansky-Witten invariant of  $(X, \omega)$ corresponding to the $\Theta$-graph \cite{kap}\footnote{In \cite[\S 5.3]{kap}, Kapranov proposed a construction of Rozansky-Witten invariants, that associates to a holomorphic symplectic manifold $(X, \omega)$ and a trivalent graph $\Gamma$ with $l$ vertices a cohomology class 
$\mathrm{RW}_\Gamma (X, \omega)\in H^l(X, \cO_X)$ (denoted in {\it loc. cit.} by $c_\Gamma(X)$ ). We only use his construction applied to the trivalent graph  with two vertices connected by three edges (that has the shape of letter $\Theta$).}. For $p=3$, we modify the above construction as follows. If $\nabla$ is a torsion-free symplectic connection then the map
\begin{equation}\label{eq:definition_of_gamma}
 \gamma_{\nabla}(\theta):=   \omega(\nabla _{\theta} \theta, \theta) - (L^{p-1}_{\theta} \circ \iota_\theta  - \iota_{\theta ^{[p]}} )([\eta ] )
\end{equation}
determines a section $\gamma_{\nabla}\in F_{X}^{\lp *} \Omega^{1, \lp}_{X} \iso \operatorname{F}_{\mathrm{abs}}^* \Omega^{1}_{X}$.  Here we  use that the map 
$L^{p-1}_\theta \circ \iota_\theta  - \iota_{\theta ^{[p]}}: \Omega^{\cdot}_X \to \Omega^{\cdot -1}_X$ commutes with the de Rham differential and, thus, yields a map  (of sheaves of sets) $\Omega^{1}_X/d(\cO_X) \to \cO_X$ (in particular, the expression $(L^{p-1}_\theta \circ \iota_\theta  - \iota_{\theta ^{[p]}} )([\eta ])$ is well-defined). Also, we use the following properties of  $\gamma_{\nabla}$:
\[\gamma_{\nabla}(\theta_1 +\theta_2)=
\gamma_{\nabla}(\theta_1)+\gamma_{\nabla}(\theta_2), \quad \gamma_{\nabla}(f \theta_1)= f^p \gamma_{\nabla}(\theta_1),\]
for every $\theta_1, \theta_2\in T_X$
and $f\in \cO_X$\footnote{Proof: using Darboux Lemma \ref{pr:BK's Darboux lemma}, we may assume that $X=\Spec (A_0\otimes R)$, for some $\bF_p$-algebra $R$, $X\to S$ is the projection $\Spec (A_0\otimes R) \to \Spec R$, and $[\eta]$ is given by $\eta_{\mathrm{can}}= \sum_i y_i dx_i\in \Omega^1_{A_0}$. If $\nabla$ is the standard flat connection, $\theta= \sum_i (a_i \partial_{x_i} +b_i\partial_{y_i})$, $a_i,b_i \in A_0\otimes R$, one computes  
\[\omega(\nabla _{\theta} \theta, \theta) - (L^{2}_{\theta} \circ \iota_\theta  - \iota_{\theta ^{[3]}} )([\eta ] ) = \sum_i\big(a_i L_{\theta}(b_i) - b_i L_{\theta}(a_i) \big)-
\Big(\sum_i \big(y_i L_{\theta}^2(a_i) + 2b_i  L_{\theta}(a_i) +  a_i L_{\theta}(b_i)\big)-  \sum_i y_i L_{\theta}^2(a_i) \Big)=0.\] If $\nabla'$ is another torsion-free symplectic connection, then 
\(\gamma_{\nabla'}(\theta) - \gamma_{\nabla}(\theta)\) is obtained by evaluating the symmetric form $\nabla' - \nabla\in  \Gamma^3 \Omega_X^1\iso (\mathrm{Sym}^3 T_X)^*$ at 
$\theta\otimes \theta \otimes \theta$. This is enough.}. Note that,  for any two torsion-free symplectic connections, $\nabla$, $\nabla'$,
the section  $\gamma_{\nabla} - \gamma_{\nabla'}\in F_{X}^{\lp *} \Omega^{1, \lp}_{X}$ is obtained from 
$\nabla - \nabla'\in  \Gamma^3 \Omega_X^1$ via the map $ \Gamma^3 \Omega_X^1 \to  F_{X}^{\lp *} \Omega^{1, \lp}_{X}$,
$v^{\otimes 3}\mapsto v$. We shall say that a torsion-free symplectic connection $\nabla$ is restricted if
$\gamma_{\nabla}=0$. It follows that torsion-free restricted symplectic  connections form a torsor under
$\ker( \Gamma^3 \Omega_X^1 \to  F_{X}^{\lp *} \Omega^{1, \lp}_{X})$. Consequently, for $p=3$,
the symplectic  Atiyah class  $\mathrm{At}_{X,\omega } \in H^1(X, \Gamma^3 \Omega_X^1)$
admits a refinement  to a class
 $\mathrm{At}_{X,[\eta]} \in H^1(X, \ker( \Gamma^3 \Omega_X^1 \to  F_{X}^{\lp *} \Omega^{1, \lp}_{X}))$. 
 \begin{rem}\label{rem:interpretation_of_Atiyah_class}
    The torsor of torsion-free symplectic connections (resp. torsion-free restricted symplectic  connections)
    can be visualized as follows. Recall from \eqref{eq:1-cocycle_of_G}  a $1$-cocycle $c$ of the group $G_0^0$ with values in  ${\bf \fm^{3}/\fm^{4}}$. For $p>3$ (resp. $p=3$), twisting this cocycle with the $G_0^0$-torsor of ``Darboux frames" $\mc M_{X,[\eta]}\to X$ (see Construction \ref{constr:torsor of Darboux frames}), one obtains a $\Sym^3 \Omega^1_X\iso \Gamma^3 \Omega^1_X$-torsor\footnote{The isomorphism $\Sym^3 \Omega^1_X\iso \Gamma^3 \Omega^1_X$ carries $\eta_1 \otimes \eta_2 \otimes \eta_3$ to $\sum_{\sigma\in S_3} \eta_{\sigma(1)} \otimes \eta_{\sigma(2)} \otimes \eta_{\sigma(3)}$.} over $X$ (resp. a torsor under $\coker ( F_{X}^{\lp *} \Omega^{1, \lp}_{X} \to\Sym^3 \Omega^1_X)\iso \ker( \Gamma^3 \Omega_X^1 \to  F_{X}^{\lp *} \Omega^{1, \lp}_{X})$), which is identified with the torsor 
 of  torsion-free  symplectic  connections (resp. the torsor of torsion-free restricted symplectic  connections); {\it cf.} \cite[\S 5.1.2]{kap}.
\end{rem}
\begin{df}\label{def:restricted RW-class}
    Let $p=3$. The  bilinear form $\omega^{-1}: \Omega_X^1 \otimes \Omega_X^1 \to \cO_X$ determines a non-degenerate alternating bilinear form $B$ on
$\ker( \Gamma^3 \Omega_X^1 \to  F_{X}^{\lp *} \Omega^{1, \lp}_{X})$, {\it cf.} \eqref{eq:form_on_sym_three}. The \textit{restricted Rozansky-Witten} invariant of  $(X, [\eta])$ is defined by the formula:
\[\mathrm{RW}_\Theta (X, [\eta]):= B\circ (\mathrm{At}_{X,[\eta] } \cup \mathrm{At}_{X,[\eta] })\in  H^2(X, \cO_X). \]
\end{df}

For a line bundle $\cL$ over $X$, we denote by $c_1(\cL)\in H^1(X, \Omega^1_X)$ its Chern class, and let $\iota_{\omega^{-1}}(c_1(\cL) \cup c_1(\cL))\in H^2(X, \cO_X)$ be the image
 of $c_1(\cL) \cup c_1(\cL)$ under the map $H^2(X, \Omega^2_X)\rar{\omega^{-1}} H^2(X, \cO_X)$ (induced by the contraction with the bivector field $\omega^{-1}\in \bigwedge^2 T_X$).
 \begin{rem}\label{rem:relation_to_c}
    For $p>3$, we have \[\mathrm{RW}_\Theta (X, \omega) = \tfrac{1}{3} \iota_{\omega^{-1}} c_2(T_X),\]
    where $c_2(T_X)\in H^2(X, \Omega^2_X)$ is the Chern class of $T_X$. In fact, if $\mathrm{At}_{X, T_X}\in H^1(X, \Omega^1_X \otimes \End(T_X))$
    is the Atiyah class  of $T_X$ then $\mathrm{RW}_\Theta (X, \omega)=
   -\frac{1}{6}\iota_{\omega^{-1}} \tr(\mathrm{At}_{X, T_X} \cup \mathrm{At}_{X, T_X}) $. This follows from the fact that the torsor of all connections on 
   $T_X$ is the pushout of the torsor of torsion-free symplectic 
   connections
   along the map $\Gamma^3\Omega^1_X \to \Omega^1_X \otimes \End(T_X)$ given by $\omega$. If $p=3$, then $\iota_{\omega^{-1}} c_2(T_X)=0$.
\end{rem}
\begin{pr}\label{pr:obstruction_main}  
    For $p>3$ (resp. $p=3$), the obstruction class \eqref{eq:obstruction_class_ssrw} is equal to $ \frac{1}{8}\mathrm{RW}_\Theta (X, \omega)+ \frac{1}{2}\iota_{\omega^{-1}}(c_1(\cL) \cup c_1(\cL))$ (resp. $ \frac{1}{8}\mathrm{RW}_\Theta (X, [\eta])+ \frac{1}{2}\iota_{\omega^{-1}}(c_1(\cL) \cup c_1(\cL))$).
\end{pr}
\begin{proof}  Let $\mc M_{X,[\eta]}\to X^\lp$ be the $G_0$-torsor of ``Darboux frames''; see Construction \ref{constr:torsor of Darboux frames}. Giving a line bundle $\cL$ over $X$ is equivalent to specifying a lift of this torsor to a $\widehat{G}_0\iso G_0\ltimes \underline{A}_0^\times$-torsor  $\mc M_{X,[\eta], \cL} \to X^\lp$. Moreover, by Proposition \ref{pr.twistedmodules}, for every $r\geq 0$, 
lifting $\mc M_{X,[\eta], \cL} \to X^\lp$ to a $\widehat{G}_r$-torsor is equivalent to specifying a flat over $\mathbb F_p[h]/(h^{r+1})$ object of $\QCoh_{\mathbb P^\mathsf{tw}}(X^\lp \times \Spec \mathbb F_p[h]/(h^{r+1}))$, together with an isomorphism between its reduction  modulo $h$ and $\cL$.
In particular,   \eqref{eq:obstruction_class_ssrw} is the obstruction class to lifting $\mc M_{X,[\eta], \cL} \to X^\lp$ to a $\widehat{G}_1$-torsor. To prove the Proposition we shall use the explicit description of the cocycle, corresponding to the extension $\widehat{G}_1\epi \widehat{G}_0$,  explained in  \S  \ref{subsec:the_action_modulo_the_square_of_h}.
Let 
\begin{equation}\label{eq:anotherext}
   1 \to \bG_a \to H \to \widehat{G}_0^0\to 1  
\end{equation}
be the extension obtained from $1\to \underline{A}_0 \to \widehat{G}_1 \to \widehat{G}_0\to 1$ by, first, 
 pulling it back along $\widehat{G}_0^0:= G_0^0\ltimes \underline{A}_0^\times\subset  \widehat{G}_0 $, and then taking the pushout along 
 $\underline{A}_0\to \bG_a$ (induced by $\mathrm{red}\colon A_0\to A_0/\fm=\bF_p$).
Since  $A_0$, as a representation of
$\widehat{G}_0$, is induced from the trivial representation of 
$\widehat{G}_0^0$, $\mathrm{Ob}_{\cL}$ is also equal to the obstruction class to lifting the
$\widehat{G}_0^0$-torsor  $\mc M_{X,[\eta], \cL} \to X$ to an $H$-torsor. Equivalently, 
$\mathrm{Ob}_{\cL}$ is the pullback along the map $X \to B \widehat{G}_0^0$, classifying the torsor $\mc M_{X,[\eta], \cL} \to X$, of the class $[H]\in H^2(B \widehat{G}_0^0, \bG_a)$
corresponding to central extension \eqref{eq:anotherext}.
 We write $[H]= [H_{\mathrm{RW}}]+ [H_{\bG_m}]$, where $[H_{\mathrm{RW}}], [H_{\bG_m}]\in H^2(B \widehat{G}_0^0, \bG_a)$ are defined as follows. The extension $1\to \bG_a \to H_{\mathrm{RW}}\to \widehat{G}_0^0 \to 1$ is  obtained by  pulling  \eqref{eq:anotherext}  back  along 
$\widehat{G}_0^0 \epi G^0_0 \mono G_0^0\ltimes \underline{A}_0^\times\cong\widehat{G}_0^0 $. The second extension is  $ 1 \to \bG_a \to  G_0^0\ltimes K \to G_0^0\ltimes \underline{A}_0^\times\to 1 $, where $K$ is the group scheme of units in the algebra
\begin{equation}
  A_h/(h^2, h\fm)\cong \{ f +\lambda h| f\in A_0, \lambda\in \bF_p\}, \quad   (f_1 + \lambda_1 h)(f_2+\lambda_2 h)= f_1f_2 +  \big(\frac{1}{2}\mathrm{red}(\{f_1, f_2\}) +\lambda_1 \mathrm{red}(f_2) +\lambda_2\mathrm{red}(f_1)\big)h.
\end{equation}
with the action of $G_0^0$ given by $g(f +\lambda h)= g(f) + \lambda h.$ 

By construction, $[H_{\mathrm{RW}}]$ is the pullback of a class in $ H^2(B G_0^0, \bG_a)$
described in  Proposition \ref{pr:action_mod_h^2}. Namely, 
after the identification $\QCoh(B G_0^0)\cong \mathrm{Rep}(G_0^0)$,
Proposition \ref{pr:action_mod_h^2} expresses this class as the cup product
 $[c]\cup [c]$, where $[c]\in H^1(B \widehat{G}_0^0, \fm^{3}/\fm^{4})$ is represented by $1$-cocycle \eqref{eq:1-cocycle_of_G}, followed by the
 homomorphism of $\widehat{G}_0^0$-representations $\frac{1}{8}B: \fm^{3}/\fm^{4} \otimes \fm^{3}/\fm^{4} \to \bF_p$.
 Consequently, using Remark \ref{rem:interpretation_of_Atiyah_class}, the pullback of $[H_{\mathrm{RW}}]$ along $X \to B \widehat{G}_0^0$ is equal to $\frac{1}{8}\mathrm{RW}_\Theta (X, \omega)$ (resp. 
 $\frac{1}{8}\mathrm{RW}_\Theta (X, [\eta])$).

On the other hand, the pullback of $[H_{\bG_m}]$ is the obstruction to lifting 
$\cO_X$-module $\cL$ to a flat module over the canonical quantization $O$ of the Poisson algebra $(\cO_X, \{\cdot, \cdot\})$ modulo $h^2$:
\begin{equation}\label{eq:quant_mod_square}
 O:=  \{f_0 + f_1 h| f_i\in \cO_X\}, \quad   (f_0 + f_1 h)(q_0 + q_1 h)= f_0q_0
 +  (\frac{1}{2}\{f_0, q_0\} +f_1q_0+ f_0q_1)h,
\end{equation}
and it is given by $\frac{1}{2}\iota_{\omega^{-1}}(c_1(\cL) \cup c_1(\cL))$ as desired. 
\end{proof}

\subsection{Quantizations of $\QCoh(X)$ vs quantizations of  $\cO_X$}\label{quantizationofalgebrasandcategories}
For the duration of this subsection, let $(X, [\eta])$ be a pseudo-smooth  scheme over $S$ equipped with a  restricted symplectic structure. In \cite{bk}, Bezrukavnikov and Kaledin introduced the notion of 
a  Frobenius-constant  quantization $O_{X,h}$ of $\cO_X$. By definition, $O_{X,h}$ is a
sheaf of $\cO_{{X^\lp}}[[h]]$-algebras whose underlying module is finite locally free, and equipped with an isomorphism of  $\cO_{{X^\lp}}$-algebras $O_{X,h}/(h)\iso F_*^\lp \cO_X$ subject to certain conditions (see \S \ref{ssec:restricted symplectic structure and Frobenius-constant quantization}). 
 It is shown in \cite[Lemma 4.3]{bk} that the groupoid of Frobenius-constant quantizations is equivalent to the groupoid of $G$-torsors $\cM_{X, [\eta], h}\to X^\lp$ together with an isomorphism $\cM_{X, [\eta], h} \times ^{ G} G_0 \iso \cM_{X, [\eta]}$ of $G_0$-torsors.
For a $G$-torsor $\cM_{X, [\eta], h}\to X^\lp$ as above the corresponding Frobenius-constant  quantization $O_{X,h}$ is given by the formula $\cM_{X, [\eta], h} \times ^{ G} A_h$. It is proven in \cite[Theorem 1.23]{bk} that
if $X$ is \emph{admissible}, meaning that $F^{\lp *}_X\colon  H^i(X^\lp, \cO_{{X^\lp}})\to H^i(X, \cO_{{X}})$ is an isomorphism for $i=1,2,3$, then 
the set $\pi_0(\mathsf{Quant}(X,[\eta]))$ of isomorphism classes of Frobenius-constant quantizations of $\cO_X$ is in a natural bijection with $H^1_{\textup{\'et}}(X, \cO_X^\times/\cO_{X^\lp}^{\times})$. For instance, this applies to affine $X$.
In particular, every two Frobenius-constant quantizations of $\cO_X$ are isomorphic \'etale locally on $X$.
\begin{df}\label{def:categorical Frobenius-constant quantization}
    A formal Frobenius-constant  quantization of $\QCoh(X)$ is a quasi-coherent sheaf of abelian categories $\fC$ over the formal scheme $X^\lp [[h]]$ together with an equivalence $\Xi_0\colon\fC|_{X^\lp }\iso F_{X *}^\lp \cO_X\lmod$ of quasi-coherent sheaves of categories over $X^\lp$ satisfying the following property: 
    there is a Zariski cover $X=\bigcup _{i\in I} U_i$ and Frobenius-constant  quantizations $O_{U_i, h}$ of $\cO_{U_i}$,  such that, for every $i\in I$,
    the pair $(\fC|_{U^\lp _i[[h]] },   (\Xi_0)|_{U_i^\lp})$  is equivalent to 
    $(O_{U_i, h}\lmod, \Psi_i)$, where  $\Psi_i\colon O_{U_i, h}\lmod|_{U_i^\lp} \iso F_*^\lp \cO_{U_i}\lmod$ is the equivalence given by 
    $O_{U_i, h}/(h)\iso F_*^\lp \cO_{U_i}$.   We denote by $\mathsf{Quant^{Cat}}(X,[\eta])$ the $2$-groupoid of formal Frobenius-constant  quantizations of $\QCoh(X)$.
\end{df}
For example, the canonical quantization $\QCoh_{\mathbb P^\mathsf{tw}}$ restricted to $X^\lp[[h]]$ is a formal Frobenius-constant quantization, denoted by $\mathfrak C_{\mathrm{can}}$.
To see this, recall from \eqref{eq:actionontheformaldisk} the extension $1\to L^+\bG_m\to  \widehat G  \to G\to 1$
that we used to construct the action of $G_0$ on $\cA_{\widehat {{\mathbb P^1}}}\lmod$. We have a surjection $\widehat G \epi G_0$ whose kernel is $\underline{A}_h^\times$. Zariski locally we can lift the $G_0$-torsor   $\cM_{X, [\eta]}$ to a $\widehat G$-torsor. The latter gives rise to a $G$-torsor $\cM_{X, [\eta], h}$ and the corresponding quantization $O_h$. Now the claim follows from Proposition \ref{pr.twistedmodules}.

Consequently,  a formal Frobenius-constant  quantization of $\QCoh(X)$ is a pair consisting of a quasi-coherent sheaf of abelian categories $\fC$ over the formal scheme $X^\lp [[h]]$ and an equivalence $\Xi_0\colon\fC|_{X^\lp }\iso F_{X *}^\lp \cO_X\lmod$ of sheaves of categories over $X^\lp$, which is \'etale locally equivalent to 
$(\mathfrak C_{\mathrm{can}}, \Xi_0)$.

Observe that the Picard 2-groupoid $$\cC_X= \tau^{\le 2}R\Gamma_{\textup{\'et}}\! \left(X, \Fib\left(\cO_{X^\lp} [[h]]^\times \to \mc O_X^\times\right)\right)[2],$$
classifying $\bG_m$-gerbes over  $X^\lp [[h]]$ equipped with a trivialization of the pullback along $X \rar{F^\lp}X^\lp \mono X^\lp[[h]]$, acts on the 2-groupoid of formal Frobenius-constant  quantization of $\QCoh(X)$: the action of $\cP \in \cC_X$ is given by  the tensor product (see Construction \ref{constr: tensor product of sheaves of categories}) with the category $\QCoh^\cP(X^\lp [[h]])$ of quasi-coherent sheaves twisted by $\cP$; the equivalence $\Xi_0\colon\fC|_{X^\lp }\iso F_{X *}^\lp \cO_X\lmod$
together with the trivialization of $\cP|_{X}$ yields $\big(\fC \otimes \QCoh^\cP(X^\lp [[h]])\big) |_{X^\lp }\iso F_{X *}^\lp \cO_X\lmod$.
\begin{pr}\label{pr:classification_of_formal_quantizations}
    The action of $\cC_X$ on $\mathfrak C_{\mathrm{can}}$  determines an equivalence of 2-groupoids 
    \begin{equation}\label{eq:description_of_the_groupoid_of_formal_quant}
       \cC_X \iso  \mathsf{Quant^{Cat}}(X,[\eta]).
    \end{equation}
    Consequently, we have an equivalence:
     \begin{equation}
      \mathsf{Per}^\mathsf{Cat}\colon  \mathsf{Quant^{Cat}}(X,[\eta]) \iso \tau^{\le 2}R\Gamma_{\textup{\'et}}\! \left(X, \Fib\left(\cO_{X^\lp} [[h]]^\times \to \mc O_X^\times\right)\right)[2].
    \end{equation}
    \end{pr} 
\begin{proof}
Consider the \'etale sheaf  $\cC'$ of 2-groupoids on $X$, whose 2-groupoid of sections over an \'etale $U\to X$ consists of pairs
$(\fC, \Xi_0')$, where 
 $\fC$ is a quasi-coherent sheaf of abelian categories over $U^\lp [[h]]$ together with an isomorphism $\Xi_0'$ of $\cO_{U^\lp}$-algebras between the center of $\fC|_{U^\lp }$ and $F_{U *}^\lp \cO_U$ such that \'etale locally the pair $(\fC, \Xi_0')$ is equivalent to the pair corresponding to a Frobenius-constant  quantization $O_{U,h}$. Since every two Frobenius-constant  quantizations of $\cO_U$ are \'etale locally isomorphic we infer that same is true for objects of $\cC'(U)$. Moreover, since Zariski locally every $\cO_{X^\lp}[[h]]$-linear automorphism of $O_{U,h}$, which is equal to $\Id$ modulo $h$, is inner, we conclude that any two $1$-morphisms in $\cC'$ are locally 2-isomorphic. The sheaf of
2-automorphisms of any 1-morphism is the sheaf of central
units
$$
Z(O_{U,h})^\times
\iso
\cO_{U^\lp}[[h]]^\times
=
L^+\bG_m(U^\lp).
$$
Therefore, \(\cC'\) is an \(L^+\bG_m\)-banded 2-gerbe, which is trivial, with a section provided by $\mathfrak C_{\mathrm{can}}$.  It follows that $\cC'\simeq B^2 \cO_{X^\lp}[[h]]^\times$,
with the canonical quantization $\mathfrak C_{\mathrm{can}}$ corresponding to
the identity element in $B^2 \cO_{X^\lp}[[h]]^\times$. 
Observe that, for an object $(\fC, \Xi_0')$ of $\cC'(X)$, the map $\Xi_0'$ yields an action of $F_{X *}^\lp \cO_X\lmod$ on $\fC|_{X^\lp }$ and that \'etale locally $\fC|_{X^\lp }$ (as a category with an action of $F_{X *}^\lp \cO_X\lmod$) is equivalent to $F_{X *}^\lp \cO_X\lmod$. Thus   $(\fC|_{X^\lp }, \Xi_0')$ determines a $\bG_m$-gerbe over $X$. If $X \to B^2 \cO_{X^\lp}[[h]]^\times$
classifies $(\fC, \Xi_0')$ then the above $\bG_m$-gerbe is classified by the composite 
$X \to B^2 \cO_{X^\lp}[[h]]^\times\to B^2 \cO_{X^\lp}^\times\rar{F^{\lp *}}B^2\cO_{X}^\times$. To complete the proof, it remains to observe that
giving a  formal Frobenius-constant  quantization of $\QCoh(X)$
is equivalent to giving an object of $\cC'(X)$ together with a trivialization of the corresponding $\bG_m$-gerbe.
\end{proof}
Let $\cR_X$ be the discrete group of 
$\cO_{X^\lp}$-algebra automorphisms of $\cO_{X^\lp}[[h]]$ that induce the identity modulo
$h^2$. 
The group $\cR_X$ acts by pullback on
$\mathsf{Quant}^{\mathsf{Cat}}(X,[\eta])$ and on $\cC_X$, and
these actions are compatible with the action of $\cC_X$ on
$\mathsf{Quant}^{\mathsf{Cat}}(X,[\eta])$.
The equivalence \eqref{eq:description_of_the_groupoid_of_formal_quant},
however, uses the distinguished object $\mathfrak C_{\mathrm{can}}$,
which need not be preserved by reparametrization.
We describe the resulting correction. 

Put $t=h^{-1}$ and identify $\mathbb P^1\setminus\{0\}$ with
$\bG_a$. Let
\[
  [\mathcal D_{X,t[\eta], \bG_a}]
  \rightarrow X^\lp\times \bG_a
\]
be the multiplicative $\bG_m$-gerbe of
Construction~\ref{const:multiplicative_gerbe}.
The regular bimodule gives a multiplicative splitting of its pullback along
$F_X^\lp\times\mathrm{id}$; see the proof of Proposition \ref{pr:miltiplicative_property} and note that in the symplectic situation considered here, the symplectic form trivializes the relative
dualizing line
$(\Omega_X^d)^{\otimes(1-p)}$
of Proposition~\ref{prop:relativedualizaingsheaf}, where
$d=\operatorname{rk}\Omega^1_{X/S}$.

For $r\in\cR_X$, set
\[
  \delta_r:=\frac{1}{r(h)}-\frac{1}{h}
  \in\Gamma(X^\lp,\cO_{X^\lp})[[h]].
\]
The same notation denotes the morphism
$\delta_r\colon X^\lp[[h]]\to X^\lp\times\bG_a$
defined by $t\mapsto\delta_r$. Set
\[
  \mathscr T_r:=\delta_r^* [\mathcal D_{X,t[\eta], \bG_a}]
  =
  [\mathcal D_{X,[\eta]/r(h)-[\eta]/h}].
\]
The Frobenius splitting above supplies the trivialization of
$\mathscr T_r$ along $X\to X^\lp\hookrightarrow X^\lp[[h]]$,
so $\mathscr T_r$ is an object of $\cC_X$.
We define the twisted pullback by
\begin{equation}\label{eq:twisted_action}
  r*\cP:=r^*\cP\otimes\mathscr T_r.
\end{equation}
The identity
\[
  \delta_{rs}=\delta_r+r(\delta_s)
\]
and the multiplicative structure on $[\mathcal D_{X,t[\eta], \bG_a}]$ give the
coherent composition isomorphisms for this action.
\begin{pr}[Reparametrization equivariance]
\label{prop:reparametrization-equivariance}
The equivalence
\eqref{eq:description_of_the_groupoid_of_formal_quant}
intertwines the twisted action \eqref{eq:twisted_action} on
$\cC_X$ with the pullback action on
$\mathsf{Quant}^{\mathsf{Cat}}(X,[\eta])$.
Consequently, the modified period functor
\begin{equation}\label{eq:modified-period-map}
  \widetilde{\mathsf{Per}}^{\mathsf{Cat}}\colon
  \mathsf{Quant}^{\mathsf{Cat}}(X,[\eta])
  \rightarrow
  \Bigl(\tau^{\leq2}R\Gamma_{\mathrm{fppf}}
    (X^\lp,L\bG_m)\Bigr)[2],
\end{equation}
given by
\[
  \widetilde{\mathsf{Per}}^{\mathsf{Cat}}(\mathfrak C)
  =
  \mathsf{Per}^{\mathsf{Cat}}(\mathfrak C)_{\mathrm{gen}}
  +[\mathcal D_{X,[\eta]/h}],
\]
is equivariant for the pullback actions.
Here the subscript $\mathrm{gen}$ means that we forget the
Frobenius trivialization and extend the band
$L^+\bG_m\to L\bG_m$.
\end{pr}
\begin{proof}
    It is enough to construct coherent equivalences
\begin{equation}\label{eq:canonical-reparametrization}
  r^*\mathfrak C_{\mathrm{can}}
  \simeq
  \mathfrak C_{\mathrm{can}}\otimes\mathscr T_r
\end{equation}
that respect the specified identifications at $h=0$.
We shall do this in the universal case over $B G_0 \times X^\lp [[h]]$.

All groups in the following calculation are base changed to $X^{\lp}$;
we will omit this subscript. Write $A=A_{h,X^\lp}$ for the $h$-adic completion of $A_h\otimes \cO_{X^\lp}$. For a coefficient ring $B=R[[h]]$ we use
the convention
\[
  r^*A=B\otimes_{B,r}A,\qquad
  b\,r(c)\otimes a=b\otimes ca.
\]
Fix the isomorphism
\begin{equation}\label{eq:effect_of_rep_on_A}
  r^*A\xrightarrow{\sim}A,\qquad
  x_i\mapsto x_i,\qquad
  y_i\mapsto\tfrac{r(h)}h\cdot y_i.
\end{equation}
The associated semilinear automorphism $\rho_r$ of $A$ acts
on scalar Laurent series by substitution $h\mapsto r(h)$.
It induces automorphisms of $G$ and $\widetilde G$, preserves
$\underline A^\times$, and induces the identity on $G_0$.
Moreover, $\rho_{rs}=\rho_r\rho_s$ and
$\rho_r\equiv\mathrm{id}\pmod h$.

Recall the fundamental extension
\begin{equation}\label{eq:fundamental_extension_over_base}
  1\rightarrow L\bG_m\rightarrow\widetilde G
  \rightarrow G\rightarrow1
\end{equation}
and the Contou--Carrere decomposition
\begin{equation}\label{eq:Contou-Carrere}
  L\bG_m\simeq L^+ \bG_m\times \Gra_{\bG_m}.
\end{equation}
Let $c_-\colon \Gra_{\bG_m} \to L\bG_m$ be its section, and let $p_+$ and $p_-$
be its projections. As in \eqref{eq:actionontheformaldisk},
the original $G_0$-action is encoded by
$\widehat G=\widetilde G/c_-(\Gra_{\bG_m})$.

Let $\bar r$ be the automorphism induced by $r$ on $\Gra_{\bG_m}$.
The transported action is encoded by
\[
  \widehat G_r=\widetilde G/c_r(\Gra_{\bG_m}),\qquad
  c_r:=r\circ c_-\circ\bar r^{-1}.
\]
Both $c_-$ and $c_r$ are sections of the same projection $p_-$.
Their ratio is therefore a homomorphism
\[
  a_r:=p_+\circ c_r\colon \Gra_{\bG_m}\rightarrow L^+\bG_m,\qquad
  c_r(\gamma)=c_-(\gamma)\,a_r(\gamma).
\]

To compare the resulting quotients, put
$E:=\widetilde G/\underline A^\times$.
Recall the band convention in
\eqref{eq:actiononthecomplementtotheformaldisk}:
the inclusion of $\Gra_{\bG_m}$ in $E$ is
\begin{equation}\label{eq:negative-band-convention}
  \jmath(\gamma)=(\overline{c_-(\gamma)})^{-1}.
\end{equation}
Thus it is the map denoted by $-b$ there.
Let $\mathcal E_r:=(a_r)_*E$ be the pushout of this
central extension along $a_r$. There is an isomorphism
\begin{equation}\label{eq:change-of-complement}
  \widehat G_r
  \simeq
  \bigl(\widehat G\times_{G_0}\mathcal E_r\bigr)/
       \{(z,z^{-1}):z\in L^+\bG_m\}.
\end{equation}
Indeed, the map from $\widetilde G$ to the right-hand side
sends $g$ to the pair consisting of its images in
$\widehat G$ and $\mathcal E_r$.
By \eqref{eq:negative-band-convention}, it sends
$c_r(\gamma)$ to
$(a_r(\gamma),a_r(\gamma)^{-1})$.
Its kernel is $c_r(\Gra_{\bG_m})$, and it is locally surjective.
It also sends $a\in\underline A^\times$ to $(a,1)$ and
preserves the map to $G$.
Hence \eqref{eq:change-of-complement} compares the crossed
extensions encoding the two actions, including their
distinguished copies of $\underline A^\times$.

By Proposition~\ref{pr:rediction_og_BV_to_hatW},
the extension $E$ is induced from
\[
  1\rightarrow \bG_{a}^{\hat\sharp}\rightarrow\cQ
   \rightarrow G_0\rightarrow1,\qquad
  \bG_{a}^{\hat\sharp}:=\underline{\Hom}(\bG_a,\bG_m).
\]
It therefore suffices to compute $a_r$ on $\bG_{a}^{\hat\sharp}$.
For a test-scheme character $\chi\in \bG_{a}^{\hat\sharp}$, the section $c_-$
is given by $c_-(\chi)=\chi(h^{-1})$, and
$$
  r\bigl(c_-(\chi)\bigr)=\chi\!\left(\tfrac{1}{r(h)}\right)=\chi\!\left(\tfrac{1}{h}+\delta_r\right)=c_-(\chi)\,\chi(\delta_r).
$$
Since $\chi(\delta_r)\in L^+\bG_m$, this also proves
$\bar r|_{\bG_{a}^{\hat\sharp}}=\mathrm{id}$. Consequently,
\[
  a_r|_{\bG_{a}^{\hat\sharp}}=\operatorname{ev}_{\delta_r},\qquad
  \mathcal E_r\simeq
  (\operatorname{ev}_{\delta_r})_*\cQ.
\]
By Construction~\ref{const:mult_structure_on_the_gerbe_over_BG_0xG_a},
the latter extension corresponds to the pullback of
$[\cA^\flat_{\bG_a, X^\lp}]$ along $t\mapsto\delta_r$.
If $\mathfrak U$ denotes the universal sheaf of categories
on $BG_0\times X^\lp[[h]]$, we have therefore obtained
\begin{equation}\label{proof:twist}
  r^*\mathfrak U
  \simeq
  \mathfrak U\otimes
  \delta_r^*[\cA^\flat_{\bG_a,X^\lp}].
\end{equation}
Tensoring by a gerbe here means the usual action through its
category of twisted quasi-coherent sheaves.
Pulling back along the Darboux-frame classifying map
$X^\lp\to BG_0$ gives \eqref{eq:canonical-reparametrization}
for the underlying sheaves of categories.

It remains to check that \eqref{proof:twist} respects the
markings at $h=0$. Denote this equivalence by $\Phi_r$,
and put $H:=G_0^0$. We continue to work universally,
writing $T$ for the coefficient scheme.

The $G_0$-equivariant multiplicative splitting of
$[\mathcal D_{\Spec A_0,t[\eta_{\mathrm{can}}],\bG_a}]$
over $\Spec A_0\times\bG_a$, constructed in the proof of
Proposition~\ref{pr:miltiplicative_property}, restricts
at the origin to an $H$-equivariant multiplicative
splitting. It therefore defines a section
\[
  s\colon H\rightarrow\cQ|_H.
\]
Evaluating this splitting at $t=\delta_r(0)$ gives a
trivialization of the correction gerbe over $BH_T$
at $h=0$. Its pullback along the torsor of pointed
Darboux frames is the specified Frobenius
trivialization of $\mathscr T_r$ along $X$.
Indeed, a pointed Darboux frame identifies the origin
with the diagonal section, and hence identifies the
corresponding evaluation functors in the proof of
Proposition~\ref{pr:miltiplicative_property}.

On the other hand, $\Phi_r$, together with the
canonical markings, induces another trivialization
of the correction gerbe at $h=0$. The difference
between these two trivializations is an invertible
sheaf $\mathcal L_r$ on
\[
  [\Spec A_0/G_0]\times T\simeq BH_T.
\]
Its pullback to the standard atlas $T\to BH_T$
has a canonical trivialization. Indeed,
$\rho_r\equiv\mathrm{id}\pmod h$, the canonical
marking is the identity on the underlying category
of $A_0$-modules, and the comparison
\eqref{eq:change-of-complement} preserves this
underlying local functor, as expressed by
$a\mapsto(a,1)$.
Thus the remaining equivariant structure on
$\mathcal L_r$ is given by a character
\begin{equation}\label{eq:marking_character}
  \lambda_r\colon H_T\rightarrow\bG_{m,T}.
\end{equation}

All these constructions commute with base change
in the coefficient scheme. It therefore suffices
to consider the universal reparametrization
\[
  T_{\mathrm{univ}}
  :=\Spec\mathbb F_p[a_2,a_3,\ldots],
  \qquad
  r_{\mathrm{univ}}(h)
  :=h+a_2h^2+a_3h^3+\cdots.
\]
By Remark~\ref{rem:groups H and G_0}, the group $H$
is smooth and is an extension of a symplectic group
by a smooth connected unipotent group. In particular,
it has no nontrivial characters over any field
extension of $\mathbb F_p$.
Hence $\lambda_{r_{\mathrm{univ}}}$ is trivial
on every geometric fiber.
Since $H_{T_{\mathrm{univ}}}$ is reduced, it follows
that $\lambda_{r_{\mathrm{univ}}}=1$.
Specializing the coefficients gives $\lambda_r=1$
for every reparametrization $r$.

The line bundle $\mathcal L_r$ therefore has a unique
trivialization extending its specified trivialization
on the atlas. This identifies the two markings.
These identifications are compatible with base change
and composition, by uniqueness, together with
$\rho_{rs}=\rho_r\rho_s$ and the multiplicative
coherence of the correction gerbe.
This proves the first assertion.

Finally, the multiplicative structure gives
\[
  (\mathscr T_r)_{\mathrm{gen}}
      \otimes[\mathcal D_{X,[\eta]/h}]
  \simeq
  [\mathcal D_{X,[\eta]/r(h)}]
  =
  r^*[\mathcal D_{X,[\eta]/h}].
\]
Combining this with the first assertion proves the
equivariance of \eqref{eq:modified-period-map}.
\end{proof}

\begin{constr}[An intrinsic construction of the modified period map]
\label{rem:intrinsic-modified-period}
The modified period map \eqref{eq:modified-period-map} also admits a categorical interpretation
that makes its compatibility with reparametrizations formal.

\smallskip
\noindent\emph{The neutral stack.}
On the category of affine schemes over $X^\lp$, put
\[
  \mathcal B(\Spec R)=R((h)),\qquad
  \mathscr V(\Spec R)
    =\operatorname{Proj}_{\mathrm{fg}}(R((h))).
\]
The restriction functors are extension of scalars.
By Drinfeld's descent theorem
\cite[Theorem~1.6]{mathew},
$\mathscr V$ is an {\it fpqc} stack of $\mathcal B$-linear
additive categories.
For a faithfully flat $R\to R'$, its descent data use the
ring $(R'\otimes_RR')((h))$.
The stack of $\mathcal B$-linear autoequivalences of
$\mathscr V$ is the Picard stack of invertible
$\mathcal B$-modules.
By Remark~\ref{rem:the_case_of_F_p_and_Q-algebras},
these modules are {\it fppf} locally trivial on the coefficient
scheme. Hence, on the {\it fppf} site,
\[
  \underline{\Aut}_{\mathcal B}(\mathscr V)
  \simeq B L\bG_m.
\]

\smallskip
\noindent\emph{The generic projective category.}
For $\mathfrak C\in\mathsf{Quant}^{\mathsf{Cat}}(X,[\eta])$
and $U=\Spec R\to X^\lp$, set
\[
  \mathfrak C_n(U)
    :=\mathfrak C\bigl(\Spec(R[h]/(h^{n+1}))\bigr),\qquad
  \mathscr P_{\mathfrak C}^+(U)
    :=\varprojlim_{n\geq0}\mathfrak C_n(U)^{\mathrm{cp}}.
\]
Here $\mathrm{cp}$ denotes the full subcategory of compact
projective objects, and the limit is the $2$-limit of
categories: its objects are compatible systems.
Base change preserves compact projectives, since the
corresponding forgetful right adjoint is exact and preserves
coproducts.

Let $\mathscr P_{\mathfrak C}^{\mathrm{gen}}$ be the {\it fppf}
sheafification of the category-valued prestack
\[
  U\mapsto
  \operatorname{Kar}
    \bigl(\mathscr P_{\mathfrak C}^+(U)[h^{-1}]\bigr).
\]
Localization leaves the objects unchanged and replaces
each morphism module by its localization at $h$:
\[
  \operatorname{Hom}_{\mathscr P^+[h^{-1}]}(P,Q)
  =\operatorname{Hom}_{\mathscr P^+}(P,Q)[h^{-1}].
\]
The symbol $\operatorname{Kar}$ denotes idempotent completion.
Fppf locally on $X^\lp$, the formal category is given by the
reduced Weyl algebra, so this construction gives
$\operatorname{Proj}_{\mathrm{fg}}(A_{h}(R)[h^{-1}])$.
The algebra $A_{h}(R)[h^{-1}]$ is a split matrix algebra;
see Remark~\ref{rem:A_S is matrix algebra outside of 0}.
Thus $\mathscr P_{\mathfrak C}^{\mathrm{gen}}$ is {\it fppf}
locally equivalent, as a stack, to $\mathscr V$.

\smallskip
\noindent\emph{The gerbe of trivializations.}
Define
\[
  \mathcal G_{\mathfrak C}(U)
  :=
  \operatorname{Eq}_{\mathcal B|_U}
  \bigl(\mathscr V|_U,\,
        \mathscr P_{\mathfrak C}^{\mathrm{gen}}|_U\bigr).
\]
The objects are equivalences of linear stacks on the category
of affine schemes over $U$; the arrows are invertible natural
transformations. Restriction along $V\to U$ defines the
base-change functors.
Precomposition makes $\mathcal G_{\mathfrak C}$ a torsor
under $\underline{\Aut}_{\mathcal B}(\mathscr V)\simeq
B L\bG_m$. Equivalently, it is an $L\bG_m$-gerbe on $X^\lp$.
Its band is given by scalar natural automorphisms.

This construction is compatible with twisting:
for $\mathcal T\in\cC_X$ there is a natural equivalence
\[
  \mathcal G_{\mathcal T\otimes\mathfrak C}
  \simeq
  \mathcal T_{\mathrm{gen}}\otimes\mathcal G_{\mathfrak C}.
\]
It therefore suffices to identify it on
$\mathfrak C_{\mathrm{can}}$.
By \eqref{eq:resttothecomplement}, the generic projective category of
the canonical quantization is the category of finite
projective modules over $\mathcal D_{X,[\eta]/h}$.
With $[A]$ denoting the gerbe of left splitting modules of
an Azumaya algebra $A$, its gerbe of trivializations is
$[\mathcal D_{X,[\eta]/h}]$.
Consequently,
\[
  \mathcal G_{\mathfrak C}
  \simeq
  \mathsf{Per}^{\mathsf{Cat}}(\mathfrak C)_{\mathrm{gen}}
      \otimes[\mathcal D_{X,[\eta]/h}].
\]

Finally, $r(h)/h$ is a unit for every $r\in\cR_X$.
Thus reparametrization preserves the $h$-adic filtration
and commutes with the compatible-projective construction,
localization, idempotent completion, and stackification.
We obtain coherent equivalences
\[
  \mathcal G_{r^*\mathfrak C}\simeq r^*\mathcal G_{\mathfrak C},
\]
which give an intrinsic proof of the second assertion of
Proposition~\ref{prop:reparametrization-equivariance}.
The first assertion additionally uses the compatibility
with the Frobenius trivialization established above.

\end{constr}

Using Proposition \ref{pr:classification_of_formal_quantizations} we obtain a \emph{period map}
\begin{equation}\label{eq:period_map}
  \mathsf{Per}\colon  \pi_0(\mathsf{Quant}(X,[\eta])) \to \pi_0(\mathsf{Quant^{Cat}}(X,[\eta]))
    \iso 
    H^1_{\textup{\'et}}(X, \cO_X^\times/\cO_{X^\lp}^\times)\times H^2(X^\lp, (1+h\cO_{X^\lp}[[h]])^\times) 
\end{equation}
 from the set of isomorphism classes of Frobenius-constant quantizations of $\cO_X$ to the set of equivalence classes of
 formal Frobenius-constant  quantizations of $\QCoh(X)$, that sends $O_{X,h}$ to the equivalence class 
 of $O_{X, h}\lmod$.
This map can be explicitly described as follows. Consider the commutative diagram:
  \begin{equation}\label{eq:period_const}
	\begin{tikzcd}
    1 \arrow[r, ""] & L^+\bG_m  \arrow[r, ""]\arrow[d, ""']& \widehat G  \arrow[r, ""]\arrow[d, ""]& G\arrow[r, ""] \arrow[d, ""']&1  \\
	1 \arrow[r, ""] &\underline{A}_0^\times \arrow[r, ""] &  G_0 \ltimes \underline{A}_0^\times  \arrow[r, ""]  &G_0\arrow[r, ""]  & 1.
	\end{tikzcd}
	\end{equation} 
The top row is defined in \eqref{eq:actionontheformaldisk}, the map $\widehat G\to  G_0 \ltimes \underline{A}_0^\times $
is defined using 
$\widehat G / (1+h\underline{A}_h)^{\times}\iso  G_0 \ltimes \underline{A}_0^\times$ 
constructed in Proposition \ref{pr:restrictionto0forgroups}. Using \eqref{eq:period_const}, a $G$-torsor $\cM_{X, [\eta], h}\to X^\lp$ gives
rise to an $L^+\bG_m$-gerbe over $ X^\lp$ (which classifies lifts of $\cM_{X, [\eta], h}\to X^\lp$ to  a $\widehat G$-torsor)
together with a trivialization of the associated $\bG_m$-gerbe on $X$. This gives \eqref{eq:period_map}.
One can check that \eqref{eq:period_map} refines   $\pi_0(\mathsf{Quant}(X,[\eta]))\to H^1_{\textup{\'et}}(X, \cO_X^\times/\cO_{X^\lp}^\times)$ from \cite[Proposition 1.19]{bk}.

\begin{rem}\label{rem:quant_of_alg_vs_cat} For the formal  quantization $\fC=O_{X, h}\lmod$ of $\QCoh(X)$, corresponding to a  Frobenius-constant quantization of $\cO_X$, the category $\fC(X^\lp[[h]])$ contains a distinguished object: the free $O_{X, h}$-module of rank 1. Conversely,
for any formal Frobenius-constant quantization $(\fC, \Xi_0)$ of $\QCoh(X)$ and an object $M\in \fC(X^\lp[[h]])$ together with an isomorphism
$e: \Xi_0(M/h) \iso F^{\lp }_*\cO_{X}$ such that $M$ is flat over $\bF_p[[h]]$ and $h$-complete, 
the sheaf of algebras $$O_{X, h}:=\underline{\End}_{\fC}(M)^{{\op}}$$ over $X^\lp\times \Spf \mathbb F_p[[h]]$ is a  Frobenius-constant quantization of $\cO_X$ such that $\fC=O_{X, h}\lmod$.
This construction exhibits an equivalence of categories between the groupoid of quadruples $(\fC, \Xi_0, M, e)$, consisting of 
a formal Frobenius-constant quantization  of $\QCoh(X)$ and a formal deformation $(M,e)$ of   $\Xi_0^{-1}(F^{\lp}_*\cO_{X})\in \fC(X^\lp)$ as above, and the groupoid of  Frobenius-constant quantization of $\cO_X$. 
In particular, if 
$$H^1(X, \cO_{{X}})= H^2(X, \cO_{{X}})=0,$$ 
then any equivalence  class 
 of formal Frobenius-constant  quantizations  of $\QCoh(X)$ arises from a unique element of  $\pi_0(\mathsf{Quant}(X,[\eta]))$. Indeed, under the above vanishing assumption $\Xi_0^{-1}(F^{\lp}_*\cO_{X})$ admits a unique (up to a non-canonical isomorphism) formal deformation 
 $M\in \fC(X^\lp[[h]])$.
    \end{rem}
In general, the period map \eqref{eq:period_map} is not surjective. For instance, using results of Section \ref{subsection:rw} we  explain below  that \eqref{eq:period_map} is not surjective for K3 surfaces admitting a restricted symplectic structure.  

 For an integer $r\geq 0$, let $G_r$ be the image of $G = \ul\Aut(A_h)$ in $\ul\Aut(A_h/(h^{r+1}))$. For $r>0$,   we have
$G_r\cong \coker((\underline{A}_h^\times )^{\geq r-1} \rar{\Ad} G)$,
where $(\underline{A}_h^\times)^{\geq r-1}$ is the subgroup of    $\underline{A}_h^\times$ of elements equal to $1$ modulo $h^{r}$;  {\it cf.} \eqref{fundextruncold}. Say a level $r$ Frobenius-constant quantization of $(X, [\eta])$ is a lifting of $G_0$-torsor  $\mc M_{X,[\eta]}\to X^\lp$ to a $G_r$-torsor $\mc M_{X,[\eta], h, r}\to X^\lp$. Given a lift as above,
the formula $\cM_{X, [\eta], h,r} \times ^{ G_r} A_h/(h^{r+1})$ 
determines a sheaf of algebras flat over $\bF_p[h]/(h^{r+1})$ that deforms $\cO_X$. For instance, using the identification $G_1\iso G_0\ltimes (\underline{A}_0^\times/\bG_m)$, the set of isomorphism classes of  level $1$ Frobenius-constant quantizations of $(X, [\eta])$ is in bijection with $H^1_{\textup{\'et}}(X, \cO_X^\times/\cO_{X^\lp}^\times)$; the quantization corresponding to the zero element is given by 
\eqref{eq:quant_mod_square}. Given a level $r$ quantization, for $r>0$, the obstruction to its lifting to a level $r+1$    quantization is a class in $H^2(X, \cO_X/\cO_{X^\lp})$.
\begin{pr}\label{pr:quant_of_algebra_mod_h3} Let $O_L$ be the level $1$ Frobenius-constant quantization of $(X, [\eta])$ corresponding to some  $L\in H^1_{\textup{\'et}}(X, \cO_X^\times/\cO_{X^\lp}^\times)$. Write
$c_1(L)\in H^1(X, \Omega_X^1)$ for the image of $L$  under the $d\log$ map (see Proposition \ref{prop:Milne's exact sequence}).
Then, for $p>3$ (resp. $p=3$), the obstruction to extending $O_L$ to a level $2$  Frobenius-constant  quantization is  the image of $ \frac{1}{8}\mathrm{RW}_\Theta (X, \omega)+ \frac{1}{2}\iota_{\omega^{-1}}(c_1(L) \cup c_1(L))\in H^2(X, \cO_X)$ (resp. $ \frac{1}{8}\mathrm{RW}_\Theta (X, [\eta])+ \frac{1}{2}\iota_{\omega^{-1}}(c_1(L) \cup c_1(L))$) in $H^2(X, \cO_X/\cO_{X^\lp})$.
\end{pr}
\begin{proof}
If $L$ is the image of some $\cL\in H^1(X, \cO_X^\times)$ then the assertion follows directly from Proposition \ref{pr:obstruction_main}. To see this, observe that for any  $r\geq 0$, the map $\widehat{G} \to G$ descends to a homomorphism:
\(\widehat{G}_r \to G_{r+1}\)
   fitting into  the commutative diagram 
\[
	\begin{tikzcd}
	 \widehat{G}_{r+1}  \arrow[r, ""] \arrow[d, ""]& G_{r+2} \arrow[d, ""] \\
     \widehat{G}_r  \arrow[r, ""] & G_{r+1} .
	\end{tikzcd}
	\]
    Take $r=0$ and apply Proposition \ref{pr:obstruction_main}.

    In general, we modify the argument as follows. Consider the composition $\bG_m \mono L^+ \bG_m \mono \widehat{G} \epi \widehat{G}_1$. This is an injective homomorphism, whose image $\bG_m$ is a central subgroup of $\widehat{G}_1$. We have the following commutative diagram:
    \[
	\begin{tikzcd}
	1 \arrow[r, ""]  & \underline{A}_0 \arrow[r, ""] \arrow[d, ""] & \widehat{G}_{1}/\bG_m  \arrow[r, ""] \arrow[d, ""]& G_{1} \arrow[d, "\Id"] \arrow[r, ""] & 1 \\
    1 \arrow[r, ""]  & \underline{A}_0/\bG_a  \arrow[r, ""]& G_2  \arrow[r, ""] & G_{1} \arrow[r, ""] & 1.
	\end{tikzcd}
	\]
    The level $1$ Frobenius-constant quantization $O_L$ determines a ${G}_1$-torsor $\mc M_{X,[\eta], h, 1}\to X^\lp$ and hence a map $f: X^\lp\to B{G}_1$. On the other hand, the bottom row of the diagram above yields a class $[G_2]\in H^2(B{G}_1, \underline{A}_0/\bG_a)$, whose pullback along $f$ is the obstruction to extending $O_L$ to a level $2$  Frobenius-constant  quantization.  Using the top row of the diagram $[G_2]$ lifts to a class
    $[\widehat{G}_{1}/\bG_m]\in H^2(B{G}_1, \underline{A}_0)$. Let $G_1^0\subset G_1$ be the preimage of $G_0^0\subset G^0$ under the projection $G_1\to G_0$. The isomorphism   $H^2(B{G}_1, \underline{A}_0) \iso H^2(B{G}_1^0, \bG_a)$ carries the class $[\widehat{G}_{1}/\bG_m]$ to $[H/\bG_m]$, where $H$ is defined in \eqref{eq:anotherext}. Using notation from the proof of Proposition \ref{pr:obstruction_main}, we have  
    $[H/\bG_m]= [H_{\mathrm{RW}}/\bG_m]+ [H_{\bG_m}/\bG_m]$. The rest  is parallel to the proof of  Proposition \ref{pr:obstruction_main}.
  
\end{proof}
\begin{example}[Nonsurjectivity of the period map]\label{ex:nonsurjectivity} We are going to show that the period map $\msf{Per}$ is not surjective in the case of the superspecial K3 surface over an algebraically closed field $k$ of characteristic $p$. 

  A  symplectic structure $\omega$ on a smooth variety $X$ over a perfect field $k$ can be promoted to a restricted symplectic structure $[\eta]$, with $d[\eta]=\omega$, if and only if the de Rham cohomology class of $\omega$ lives in the $0$-th term of conjugate filtration on $H^2_{\mathrm{dR}}(X)$. Indeed, using 
$H^0(X, \Omega^1_X/d(\cO_X))\iso H^2 \big(X,\Fib(\cO_X^p \oplus \Omega_X^{2, \mathrm{cl}}[-2] \to (\Omega_X^\bullet, d))\big)$, we have a long exact sequence:
\[H^1_{\mathrm{dR}}(X)\to H^0(X, \Omega^1_X/d(\cO_X)) \to H^2(X, \cO_X^p) \oplus H^0(X, \Omega_X^{2, \mathrm{cl}}) \to H^2_{\mathrm{dR}}(X), \]
which implies the claim. In particular, 
 a K3 surface $X$ admits a restricted structure if  and only if  $X$ is superspecial\footnote{For instance, the Kummer K3 associated to a product of two supersingular elliptic curves is superspecial. In fact, by a result of Ogus, \cite[Corollary 7.14]{ogus}, if $k=\bar k$, there is exactly one superspecial K3 surface over $k$.}, that is, 
 \[ \mathrm{Fil}^2_{\mathrm{Hodge} }H^2_{\mathrm{dR}}(X) =
 \mathrm{Fil}^{\mathrm{conj}}_0 H^2_{\mathrm{dR}}(X). \]
In fact, using vanishing of $H^1_{\mathrm{dR}}(X)$ and the exact sequence above, we have $H^0(X, \Omega^1_X/d(\cO_X))\iso H^0(X, \Omega_X^2)$, which is a $1$-dimensional vector space over $k$. Also, recall that for any supersingular and, in particular, superspecial K3 $X$, the Frobenius endomorphism of $H^2(X, \cO_X)$
is $0$ and, thus, $H^2(X, \cO_X)\iso H^2(X, \cO_X/\cO_X^p)$.

Fix a restricted symplectic structure $[\eta]$ on a superspecial K3 surface $X$ over $k$. 
Via Serre duality we have a canonical isomorphism $H^2(X, \Omega^2_X)\xra{\sim}k$; contraction with $\omega^{-1}$ produces a class $\iota_{\omega^{-1}}(1)\in H^2(X,\mc O_X)$. 
We claim that the image of the period map $$\pi_0(\mathsf{Quant}(X,[\eta])) \xra{\msf{Per}} H^1(X,\mc O_X^\times/\mc O_{X^\lp}^\times)\times H^2(X^\lp, (1+h\cO_{X^\lp}[[h]])^\times)$$\eqref{eq:period_map}, is contained in the ``quadric''
\begin{equation}\label{eq:period quadric}
\left\{
(L,\alpha) \in H^1_{\textup{\'et}}\bigl(X,\mc O_X^\times/\mc O_{X^\lp}^\times\bigr)\times H^2(X^\lp, (1+h\cO_{X^\lp}[[h]])^\times)
\,\middle|\,
c_1(L)^2 =-2\in H^2(X,\Omega^2_X)
\right\},
\end{equation}
where $c_1(L)\in H^1(X,\Omega^1_X)$ is the image of $L\in H^1(X,\mc O_X^\times/\mc O_{X^\lp}^\times)$ under the map $d\log$.
By Proposition \ref{pr:quant_of_algebra_mod_h3}, it suffices to show that, for $p>3$ (resp. $p=3$) the  Rozansky-Witten invariant $\mathrm{RW}_\Theta (X, \omega)$ (resp. the restricted  Rozansky-Witten invariant $\mathrm{RW}_\Theta (X, [\eta])$)
is equal to $8\cdot\iota_{\omega^{-1}}(1)\in H^2(X,\mc O_X)$. For $p>3$, 
using Remark \ref{rem:relation_to_c} and $c_2(T_X)= \chi(X)=24\in H^2(X, \Omega^2_X)=k$,
we find that $\mathrm{RW}_\Theta (X, \omega)$ is the image of $8=\frac{1}{3}\cdot 24$ under the isomorphism $\iota_{\omega^{-1}}\colon H^2(X, \Omega^2_X)\iso H^2(X, \cO_X).$

For $p=3$, the following argument proving  the non-vanishing of $\mathrm{RW}_\Theta (X, [\eta])$ was suggested to us by ChatGPT. Since  formation of $\mathrm{RW}_\Theta (X, [\eta])$ commutes with 
extensions of the ground field, we may (and will) replace $k$ by an
algebraic closure. First, observe
that, for every rank-two vector bundle $E$ in characteristic 3,
there is a
canonical exact sequence
\begin{equation}\label{eq:canonical short exact sequence for Gamma3}
  0\rightarrow E\otimes\det(E)
   \xrightarrow{}
   \Gamma^3(E)
   \xrightarrow{}
   \operatorname{F}_{\mathrm{abs}}^* E
   \rightarrow 0.
\end{equation}
Indeed, dually one has a universal derivation $d\colon \Sym^3(E)\to \Sym^2(E)\otimes E$, sending $uvw\in \Sym^3(E)$ to $uv\otimes w + uw\otimes v + vw\otimes u$; in char 3 it factors through $\Sym^3(E)/\operatorname{F}_{\mathrm{abs}}^* E$ and the latter injects into the image of the Koszul map $\kappa\colon E\otimes \Lambda^2(E)\to \Sym^2(E)\otimes E$, $\kappa(u\otimes (v\wedge w))\coloneqq uv\otimes w-uw\otimes v$. Explicitly, one has $d(uvw)=\kappa(-u\otimes(v\wedge w)-v\otimes(u\wedge w))$. The resulting map $\Sym^3(E)/\operatorname{F}_{\mathrm{abs}}^* E\to E\otimes \Lambda^2(E)$ is then an isomorphism (e.g. by comparing the dimensions). Dually, we get an isomorphism $E\otimes\det(E)\xra{\sim} \ker(\Gamma^3(E)\to \operatorname{F}_{\mathrm{abs}}^* E)$, giving \eqref{eq:canonical short exact sequence for Gamma3}.

Applying \eqref{eq:canonical short exact sequence for Gamma3} to $E=\Omega_X^1$, and using
$\omega$ to trivialize $\det(E)$, we obtain an isomorphism
$
s_\omega\colon \Omega_X^1 \xrightarrow{\sim}
K:=\ker\!\left(
\Gamma^3(\Omega_X^1) \longrightarrow
F_{\mathrm{abs}}^*\Omega_X^1
\right).
$
One checks that the pairing  on $E$ defined by the formula \eqref{eq:form_on_sym_three} satisfies $B(s_\omega(\alpha_1),s_\omega(\alpha_2))=-\iota_{\omega}^{-1}(\alpha_1\wedge\alpha_2)$.
Recall the restricted Atiyah class $\mathrm{At}_{X,[\eta]}\in H^1(X,K)$ (Definition \ref{def:restricted RW-class}), and put $\mathrm{at}_{X,[\eta]}:=s_{\omega^{-1}}(\mathrm{At}_{X,[\eta]})\in H^1(X, \Omega^1_X)$; by the definition and above discussion, we get the formula
\[\mathrm{RW}_\Theta (X, [\eta])= -\iota_{\omega^{-1}}(\mathrm{at}_{X,[\eta]} \cup \mathrm{at}_{X,[\eta]}). \]
  By a result of Ogus \cite[Corollary 7.14]{ogus}, all superspecial K3 surfaces over the algebraic closure $\bar k$ are isomorphic. 
  In fact, using \cite[Ex. 5.2]{Shioda1979}, we may take for $X$ the Fermat quartic 
\begin{equation}\label{eq:Fermatchat}
 X=V(\Phi)\subset\bP^3,
 \qquad
 \Phi=x_0^4+x_1^4+x_2^4+x_3^4.
\end{equation}
Let \(D\) be the subgroup of projective automorphisms of  $X$ generated by diagonal matrices 
\[
 \operatorname{diag}(\zeta_0,\zeta_1,\zeta_2,\zeta_3),
 \qquad
 \zeta_i^4=1,
 \]
and permutations of the coordinates, and let  \(D_0\subset D\) be the (index $4$) subgroup of  automorphisms preserving $\omega\in H^0(X, \Omega_X^2)$. One verifies ({\it cf}. \cite[Proof of Lemma 3.7]{ogus1978griffiths}) that the space of $D_0$-invariant classes
$H^1(X, \Omega_X^1)^{D_0}$ is a $1$-dimensional space generated by the class of hyperplane section $[H]$. Using the isomorphism $H^0(X, \Omega^1_X/d(\cO_X))\iso H^0(X, \Omega_X^2)$ observed above, the restricted structure $[\eta]$ is $D_0$-invariant; consequently  $\mathrm{at}_{X,[\eta]}\in H^1(X, \Omega_X^1)^{D_0}$, and so $\mathrm{at}_{X,[\eta]}=c\cdot [H]$ for some $c\in k$. We have $[H]\cup [H]=4=1 \in H^2(X, \Omega_X^2)=k$, and thus 
$$\mathrm{RW}_\Theta (X, [\eta])= -\iota_{\omega^{-1}}(\mathrm{at}_{X,[\eta]} \cup \mathrm{at}_{X,[\eta]})=-\iota_{\omega^{-1}}(c^2). $$ Thus it remains to show that $-c^2=8$ or, in other words, that $c^2=1$. Note that for any $t\in k^\times$, one has\footnote{Indeed, $\mathrm{At}_{X,t[\eta]}=t\cdot\mathrm{At}_{X,[\eta]}\in H^1(X,K)$, so $\mathrm{at}_{X,t[\eta]}=s_{t^{-1}\omega}(\mathrm{At}_{X,t[\eta]})=t^{-1}s_{\omega}(t\cdot \mathrm{At}_{X,[\eta]})=\mathrm{at}_{X,[\eta]}$.} $\mathrm{at}_{X,[\eta]}=\mathrm{at}_{X,t[\eta]}$; thus we can assume that $[\eta]$ is defined over $\mbb F_3$, and so $c\in \mbb F_3$. It then remains to show that $c\neq 0$.

Suppose that $c=0$; then $\mathrm{At}_{X,[\eta]}=0$, and so there exists a torsion-free restricted symplectic connection on $T_X$. However, in fact $T_X$ doesn't admit \textit{any} connection. To see this, choose \(a,b\in k\) with \(a^4=b^4=-1\). The Fermat quartic contains the line
\[
 L=\{x_0=ax_1,\ x_2=bx_3\}\simeq\bP^1.
\]
By adjunction, \(N_{L/X}\simeq\cO_{\bP^1}(-2)\), and the tangent-normal
sequence
\[
 0\rightarrow\cO_{\bP^1}(2)
 \rightarrow T_X|_L
 \rightarrow\cO_{\bP^1}(-2)
 \rightarrow 0
\]
splits because
\[
 \operatorname{Ext}^1(\cO(-2),\cO(2))
 =H^1(\bP^1,\cO(4))=0.
\]
A connection on \(T_X\) induces a connection on any direct summand of $(T_X)|_L$.
But \(\cO_{\bP^1}(2)\) admits no connection: its Atiyah class is
\[
 c_1(\cO_{\bP^1}(2))
 =2c_1(\cO_{\bP^1}(1))\neq0
 \quad\text{in}\quad H^1(\bP^1,\Omega_{\bP^1}^1).
\]
\end{example}
\begin{question}
Let $X$ be a superspecial K3 surface over an algebraically closed field of characteristic $p>2$. Is the ``quadric'' \eqref{eq:period quadric} exactly the image of the period map $\msf{Per}$?
\end{question}

 \subsection{Canonical quantization of Lagrangian subschemes.}\label{s.lag}
For the duration of this subsection let 
$ (X, [\eta])$ be a restricted symplectic $S$-scheme with $\rk(\Omega^1_{X})=2n$  where $S$ is the base scheme of characteristic $p>2$.
Let $i\colon Y\hookrightarrow X$ be a restricted Lagrangian $S$-scheme (Definition \ref{def:restricted Lagrangian}). 
 Let $J\subset A_0$ be the ideal generated by $y_i$, $1\leq i \leq n$, and let $G_{0, J} \subset G_0$ be the subgroup of automorphisms preserving $J$. Let $\pi_Y\colon  \cM_{X, Y, [\eta]}\to {Y^\lp}$ be the $G_{0,J}$-torsor of Lagrangian Darboux frames (Construction \ref{constr:stabilizer in G_0 of a Lagrangian}). It fits into the commutative diagram
\begin{equation}\label{diag.Mundinger}
\def\normalbaselines{\baselineskip20pt
	\lineskip3pt  \lineskiplimit3pt}
\def\mapright#1{\smash{
		\mathop{\to}\limits^{#1}}}
\def\mapdown#1{\Big\downarrow\rlap
	{$\vcenter{\hbox{$\scriptstyle#1$}}$}}
\begin{matrix}
  \cM_{X, Y, [\eta]} & \mono  &   \cM_{X, [\eta]}  \cr
 \mapdown{\pi_Y}  &  &\mapdown{\pi} \cr
  {Y^\lp} & \mono  & X^\lp,
\end{matrix}
\end{equation} 
where $\pi\colon \cM_{X, [\eta]}\to X^\lp$ is the $G_0$-torsor of Darboux frames.
Let $(\QCoh_{\mathbb P^\mathsf{tw}},  \Theta, \Xi_\infty, \varpi)$ be the canonical categorical twistor quantization from \S\ref{ssec:categorical twistor quantization}. 
The equivalence $\Theta$ induces an equivalence 
$$\QCoh_{\mbb P^{\msf{tw}}}(\cM_{X, [\eta]}\times  {{\mathbb P^1}}) \iso \cA_{\mathbb P^1}\lmod (\cM_{X, [\eta]}\times  {{\mathbb P^1}})$$ which we will also denote by $\Theta$.
Let $\pro_{ {{\mathbb P^1}},Y}\colon \cM_{X, Y, [\eta]}  \times  {{\mathbb P^1}} \to  {{\mathbb P^1}}$  be the projection.  Recall from Construction  \ref{constr:A module over A} the $\cA_{{\mathbb P^1}}$-module $\cV_{{\mathbb P^1}}$. The induced closed embedding of schemes $i^\lp\times \id\colon {Y^\lp}\times  {{\mathbb P^1}}\to X^\lp\times  {{\mathbb P^1}}$ induces a fully faithful embedding $(i^\lp\times \id)_*\colon \QCoh_{\mbb P^{\msf{tw}}}({Y^\lp}\times  {{\mathbb P^1}})\subset \QCoh_{\mbb P^{\msf{tw}}}(X^\lp\times  {{\mathbb P^1}})$ (see the end of Remark \ref{rem:pull-back and pushforward for global sections}).

\begin{Th}\label{app.thquantoflarg} There exists a unique (up to a unique isomorphism) pair $(\cV_{{Y^\lp}\times  {{\mathbb P^1}}}, \vartheta )$ consisting of an object
$\cV_{{Y^\lp}\times  {{\mathbb P^1}}} \in \QCoh_{\mbb P^{\msf{tw}}}({Y^\lp}\times  {{\mathbb P^1}})\subset \QCoh_{\mbb P^{\msf{tw}}}(X^\lp\times  {{\mathbb P^1}})$  and an isomorphism
$$\vartheta\colon \Xi_\infty (\cV_{{Y^\lp}\times  {{\mathbb P^1}}})_{|{Y^\lp} \times \{\infty\}} \iso \cO_{{Y^\lp}},$$
such that there exists\footnote{We emphasize that $ \varsigma$ is not a part of the data $(\cV_{{Y^\lp}\times  {{\mathbb P^1}}}, \vartheta )$ but rather its \emph{property}.} an isomorphism in $ \cA_{\mathbb P^1}\lmod (\cM_{X,Y, [\eta]}\times  {{\mathbb P^1}})\subset \cA_{\mathbb P^1}\lmod (\cM_{X, [\eta]}\times  {{\mathbb P^1}})$
 \begin{equation}\label{app.thquantoflarg.eq}
 \varsigma\colon \Theta( (\pi_Y \times \Id)^* \cV_{{Y^\lp}\times  {{\mathbb P^1}}}) \iso \pro_{ {{\mathbb P^1}},Y}^* \cV_ {{\mathbb P^1}}.
 \end{equation}
Moreover, one has an isomorphism 
\begin{equation}\label{app.squareroot}
\Xi_0 (\cV_{{Y^\lp}\times  {{\mathbb P^1}}})_{|{Y^\lp} \times \{0\}} \iso F_{Y,*}^\lp K_{Y}^{\frac{1-p}{2}},
\end{equation}
where $K_Y= \Omega_Y^n$ is the canonical sheaf. 
\end{Th}
\begin{rem}
Recall that by Proposition \ref{prop:relativedualizaingsheaf} the line bundle $K_{Y}^{1-p}$ is the relative dualizing sheaf $F^{\lp,!}_Y \cO_{{Y^\lp}}$ of the reduced Frobenius morphism $F^\lp_Y\colon Y \to {Y^\lp}$.
Consequently, there is a non-degenerate pairing 
$$ F_{Y,*}^\lp K_{Y}^{\frac{1-p}{2}} \!\otimes_{\mc O_{Y^\lp}} \! \!F_{Y,*}^\lp K_{Y}^{\frac{1-p}{2}}\to  F_{Y,*}^\lp K_{Y}^{\frac{1-p}{2}} \!\otimes_{F_{Y,*}^\lp\cO_Y}  \!F_{Y,*}^\lp K_{Y}^{\frac{1-p}{2}} \xra{\sim} F_{Y,*}^\lp K_{Y}^{1-p} \xra{\tr} \cO_{{Y^\lp}}$$
on a vector bundle $F_{Y,*}^\lp K_{Y}^{\frac{1-p}{2}}\in \QCoh(Y^\lp)$.
\end{rem}
\begin{proof} We start with the existence part. 
In \S \ref{canonicalquant.theaction} we defined a quasi-coherent sheaf of categories $(\cA_ {{\mathbb P^1}}\lmod)_{ {{\mathbb P^1}}/G_0}$ on $BG_0\times  {{\mathbb P^1}}$ whose pullback to $ {{\mathbb P^1}}$ is $\cA_ {{\mathbb P^1}}\lmod$. We shall construct a global object $\cV_{ {{\mathbb P^1}}/G_{0,J}} \in (\cA_ {{\mathbb P^1}}\lmod)_{ {{\mathbb P^1}}/G_0}(BG_{0, J}\times {\mathbb P^1} )$ 
with identification of the pullback to ${{\mathbb P^1}}$ with $\cV_{ {{\mathbb P^1}}}$ and then define $\cV_{{Y^\lp}\times  {{\mathbb P^1}} }$ to be the pullback of $\cV_{ {{\mathbb P^1}}/G_{0,J}}$
along the map ${Y^\lp} \times  {{\mathbb P^1}} \to BG_{0, J}\times  {{\mathbb P^1}}$ determined by $\cM_{X, Y, [\eta]}$. Explicitly, using Remark \ref{enum: every autoequivalence is locally an automorphism} in the beginning of the proof of Theorem \ref{appendix:maintheorem}, let 
\begin{equation}\label{diag.f.extany.lag}
	\begin{tikzcd}
	1 \arrow[r, ""] & \underline{\cA}_{{\mathbb P^1}}^\times  \arrow[r, ""]\arrow[dr, "Ad"]& \widehat {G}_{{{\mathbb P^1}}, J}  \arrow[r, ""]\arrow[d, "\alpha"]&  G_{0, J}  \times {{\mathbb P^1}}\arrow[r, ""] &1  \\
	&& \underline{\Aut} (\cA_{{\mathbb P^1}}). &  &
	\end{tikzcd}
	\end{equation}
be the diagram corresponding to the $G_{0, J}\times {{\mathbb P^1}}$-action on $\cA_{{\mathbb P^1}}\lmod$. Also, let $\underline{\Aut}(\cA_{{\mathbb P^1}},\!\cV_{{\mathbb P^1}})$ be the  group scheme parametrizing  pairs $(\phi_{\cA_{{\mathbb P^1}}}, \phi_{\cV_{{\mathbb P^1}}})$, where $\phi_{\cA_{{\mathbb P^1}}}\!\in \underline{\Aut}(\cA_{{\mathbb P^1}})$, $\phi_{\cV_{{\mathbb P^1}}}\!\in \underline{\Aut}_{\cO_{{\mathbb P^1}}}(\cV_{{\mathbb P^1}})$ with 
$\phi_{\cV_{{\mathbb P^1}}}(a v)= \phi_{\cA_{{\mathbb P^1}}}\!(a) \cdot \phi_{\cV_{{\mathbb P^1}}}\!(v)$, for all $a\in \cA_{{\mathbb P^1}}$, $v\in \cV_{{\mathbb P^1}}$.
Via Remark \ref{rem:description of sections of AXH-mod}, giving an object $\cV_{ {{\mathbb P^1}}/G_{0,J}} $ as above is equivalent to giving a homomorphism
\begin{equation}\label{app.lagreqstr}
 \widehat {G}_{{{\mathbb P^1}}, J}  \to \underline{\Aut}(\cA_{{\mathbb P^1}}, \!\cV_{{\mathbb P^1}}),
 \end{equation}
 whose composition with the projection $\underline{\Aut}(\cA_{{\mathbb P^1}}, \!\cV_{{\mathbb P^1}})\to \underline{\Aut}(\cA_{{\mathbb P^1}})$ is $\alpha$ and whose restriction to $\underline{\cA}_{{\mathbb P^1}}^\times$ carries $a \in  \underline{\cA}_{{\mathbb P^1}}^\times$ to $(\Ad_a, a)$. The uniqueness part of Theorem \ref{appendix:maintheorem} asserts that (\ref{diag.f.extany.lag}) together with the splitting at $h=\infty$ is uniquely characterized by
 the requirement that  the restriction of $\alpha$ to $\widehat{G}_{{{\mathbb P^1}}, J}  \times_{{\mathbb P^1}} \{0 \} $ factors through the natural inclusion  $G_{0, J} \subset \underline{\Aut}(A_0)$. To produce (\ref{app.lagreqstr}) we shall explain  a convenient construction of (\ref{diag.f.extany.lag}).

Call a scheme $T$ over ${{\mathbb P^1}}$ {\it locally constant} if, for every point $x \in T$, there exists an open neighborhood $x\in U\subset T$ and an $\mbb F_p$-scheme $P$ such that the map $U \rightarrow {{\mathbb P^1}}$ factors as an \'etale map $U \rightarrow {{\mathbb P^1}} \times P$ followed by the projection  ${{\mathbb P^1}} \times P \rightarrow {{\mathbb P^1}}$. 
 Observe that a smooth scheme over a locally constant scheme is also 
locally constant and that the fiber product $T_1 \times_{{\mathbb P^1}}  T_2$  of locally constant  schemes is again locally constant. 
Denote by $\cC$ the site whose underlying category is formed  by locally constant schemes over ${{\mathbb P^1}}$  equipped with the Zariski topology.  Define a sheaf  of groups $\widehat{G}_{{{\mathbb P^1}}, J}$ sending $T \in \cC$ to 
$$\widehat{G}_{{{\mathbb P^1}}, J}(T) = \{ (\phi_{\cA_{{\mathbb P^1}}}, \phi_{\cV_{{\mathbb P^1}}}) \in \underline{\Aut}(\cA_{{\mathbb P^1}}, \!\cV_{{\mathbb P^1}})(T), \psi \in G_{0, J}(T)|  (\phi_{\cA_{{\mathbb P^1}}})_{| T_0} = \psi_{|T_0}\},$$
where $T_0$ is the scheme theoretic fiber of the structure morphism $T\to {{\mathbb P^1}}$ over $0\in {{\mathbb P^1}}$. 
Consider the sequence of sheaves on $\cC$.
\begin{equation}\label{app.seqquantlagr}
1 \rightarrow \underline{\cA}_{{\mathbb P^1}}^\times\rightarrow \widehat{G}_{{{\mathbb P^1}}, J} \rightarrow G_{0, J} \times {{\mathbb P^1}} \rightarrow 1,
\end{equation}
where the second map sends $a \in  \underline{\cA}_{{\mathbb P^1}}^\times$ to $\{(\Ad_a, a), \Id \}$ and the third one carries   $(\phi_{\cA_{{\mathbb P^1}}}, \phi_{\cV_{{\mathbb P^1}}}, \psi) \in  \widehat{G}_{{{\mathbb P^1}}, J}$ to $\psi$.
\begin{lm}\label{M_Sequivariantstructure}
	The sequence  (\ref{app.seqquantlagr}) is exact for the Zariski topology on $\cC$. 
	\end{lm}

\begin{proof}
	Every scheme $T\in \cC$ is flat over ${{\mathbb P^1}}$. In particular, the restriction map   $ \underline{\cA}_{{\mathbb P^1}}^\times(T)  \to   \underline{\cA}_{{\mathbb P^1}}^\times(T\backslash T_0) $ is injective. Hence,
	the injectivity of  the second morphism  in (\ref{app.seqquantlagr}) follows from the diagram 
	\[
	\begin{tikzcd}
	  \underline{\cA}_{{\mathbb P^1}}^\times(T) \arrow[r, ""] \arrow[hookrightarrow]{d}{} & \underline{\Aut}_{\cO_{{\mathbb P^1}}}(\cV_{{\mathbb P^1}}) (T) \arrow[hookrightarrow]{d}{} \\
	  \underline{\cA}_{{\mathbb P^1}}^\times(T\backslash T_0) \arrow[r, "\cong"] & \underline{\Aut}_{\cO_{{\mathbb P^1}}}(\cV_{{\mathbb P^1}}) (T\backslash T_0).
	\end{tikzcd}
	\]
Let us show exactness at the middle term. We have to check that Zariski locally on $T$ every $ (\phi_{\cA_{{\mathbb P^1}}}, \phi_{\cV_{{\mathbb P^1}}}) \in \underline{\Aut}(\cA_{{\mathbb P^1}}, \!\cV_{{\mathbb P^1}})(T)$ with   $ (\phi_{\cA_{{\mathbb P^1}}})_{|T_0} = \Id$ comes from a section of $ \underline{\cA}_{{\mathbb P^1}}^\times$. Every $T\in \mc C$ is $h$-separable (by Example \ref{ex:separable algebras} and Lemma \ref{lem:h-separable survives etale covers}), so by Proposition \ref{prop:automorphisms trivial modulo h are inner locally} we may assume that $\phi_{\cA_{{\mathbb P^1}}} =  \Ad_{a'}$, for some $a' \in \cA_{{\mathbb P^1}}^\times(T)$. 
Since $\underline {\Aut}_{\cA_{{\mathbb P^1}}}\!(\cV_{{\mathbb P^1}})(T) \cong \cO^\times(T)$ we conclude that $\phi_{\cV_{{\mathbb P^1}}}= f a'$,  for some $f \in  \cO(T)^\times$. Hence, the pair $(\phi_{\cA_{{\mathbb P^1}}}, \phi_{\cV_{{\mathbb P^1}}})$ is the image of $ f a'\in \cA_{{\mathbb P^1}}^\times(T)$.

Lastly, let us check the surjectivity of  the map to $G_{0, J} \times {{\mathbb P^1}}$. Let $\psi \in G_{0, J}(T)$, for some $T\in \cC$. 
Using Proposition \ref{app.autAlifting}, Zariski locally on $T$, there exists  $\phi_{\cA_{{\mathbb P^1}}} \in \underline{\Aut}(\cA_{{\mathbb P^1}})(T)$ with $(\phi_{\cA_{{\mathbb P^1}}})_{| T_0} = \psi_{|T_0}$. 
  To construct $\phi_{\cV_{{\mathbb P^1}}}$  consider the embedding  $\cA_{{\mathbb P^1}}   \subset \End_{\cO_{{\mathbb P^1}}}(\cV_{{\mathbb P^1}})$ given by the action of $\cA_{{\mathbb P^1}}$ on $\cV_{{\mathbb P^1}}$. 
  Note that as an $\cO_{{\mathbb P^1}}$-algebra $\End_{\cO_{{\mathbb P^1}}}(\cV_{{\mathbb P^1}})$ is generated by global sections $x_i$, $\frac{y_i}{h}$, $1\leq i \leq n$ that are rational sections of $\cA_{{\mathbb P^1}}$.
  It follows that every automorphism  $\phi_{\cA_{{\mathbb P^1}}} \in \underline{\Aut}(\cA_{{\mathbb P^1}})(T)$ whose restriction to $T_0$  preserves the ideal  generated by $y_i$'s   extends to an automorphism of  the algebra $\End_{\cO_{{\mathbb P^1}}}\!(\cV_{{\mathbb P^1}})$ pulled back  to $T$. Using that,
  Zariski locally, every automorphism of an Azumaya algebra is inner, we infer, locally on $T$, the existence of $\phi_{\cV_{{\mathbb P^1}}}\in \underline{\Aut}_{\cO_{{\mathbb P^1}}}(\cV_{{\mathbb P^1}})(T)$ with $\Ad_{\phi_{\cV_{{\mathbb P^1}}}} = \phi_{\cA_{{\mathbb P^1}}}$ in $\underline{\Aut}(\End_{\cO_{{\mathbb P^1}}}(\cV_{{\mathbb P^1}}))(T)$ as desired. 
  \end{proof}
Next, we claim that the sheaf  $\widehat{G}_{{{\mathbb P^1}}, J} $ is representable by a locally constant scheme over $G_{0, J} \times {{\mathbb P^1}}$. Indeed, using Lemma \ref{M_Sequivariantstructure} we may view $\widehat{G}_{{{\mathbb P^1}}, J} $ as an   $\underline{\cA}_{{\mathbb P^1}}^\times$-torsor over  $G_{0, J} \times {{\mathbb P^1}}$ locally trivial for the Zariski topology. The total space of this torsor is locally constant since $\cA_{{\mathbb P^1}}^\times$ is smooth over ${{\mathbb P^1}}$. By the Yoneda Lemma  $\widehat{G}_{{{\mathbb P^1}}, J} $ acquires a group scheme structure over ${{\mathbb P^1}}$ making (\ref{app.seqquantlagr}) a group scheme extension. 
 Together with the natural homomorphism $\widehat{G}_{{{\mathbb P^1}}, J} \to \underline{\Aut}(\cA_{{\mathbb P^1}}, \!\cV_{{\mathbb P^1}}) \to \underline{\Aut}(\cA_{{\mathbb P^1}})$ it defines an action of $G_{0, J}\times {{\mathbb P^1}}$ on $\cA_{{\mathbb P^1}}\lmod$ and makes  $\cV_{{\mathbb P^1}}$ a $G_{0, J}\times {{\mathbb P^1}}$-equivariant $\cA_{{\mathbb P^1}}$-module with respect to this action.
 To complete the proof of the existence, it remains to observe that (\ref{app.seqquantlagr}) splits over $\mbb P^1\backslash \{0\}$ and, in particular, over $h=\infty$.
 
 By Lemma \ref{app.connectedness} the group $G_{0,J}$ is connected. To produce the identification (\ref{app.squareroot}) we shall use the equivalence $\Xi_0$ \eqref{eq:resttotheorigin}. 
  The isomorphism $\alpha\colon  \iota^* \cA_{{\mathbb P^1}}^{{\op}}\iso \cA_{{\mathbb P^1}}$  makes the left $\iota^* \cA_{{\mathbb P^1}}$-module $ \iota^* \cV_{{\mathbb P^1}}$ into a right  $\cA_{{\mathbb P^1}}$-module.
 Define a perfect pairing 
  \begin{equation}\label{app.duality1bis}
B\colon   \iota^* \cV_{\mbb P^1}\otimes_{\cO_{{\mathbb P^1}}} \!\cV_{{\mathbb P^1}} \to   \iota^* \cV_{{\mathbb P^1}}\otimes_{\cA_{{\mathbb P^1}}} \!\cV_{{\mathbb P^1}} \to \cO_{{\mathbb P^1}}
 \end{equation}
  as follows. The dualizing sheaf on $\Spec A_0/J$ is  $(\Omega^n_{\Spec A_0/J})^{1-p}$. 
  Hence, the trace map defines a non-degenerate bilinear form:
 \begin{equation}\label{app.duality1}
 (\Omega^n_{\Spec A_0/J})^{\frac{1-p}{2}} \otimes _{\mbb F_p} (\Omega^n_{\Spec A_0/J})^{\frac{1-p}{2}} \to (\Omega^n_{\Spec A_0/J})^{1-p}  \rar{\tr} \mbb F_p.
 \end{equation}
Using the identification   
  \begin{equation}\label{app.iden}
  A_0/J  \iso (\Omega^n_{\Spec A_0/J})^{\frac{1-p}{2}}, \quad f\mapsto f (dx_1\cdots dx_n)^{\frac{1-p}{2}}
  \end{equation}
  the latter yields a map $A_0/J  \otimes _{\mbb F_p} A_0/J\to \mbb F_p$ which by the $\cO_{{\mathbb P^1}}$-linearity extends to a bilinear form $B\colon   \iota^* \cV_{{\mathbb P^1}}\otimes_{\cO_{{\mathbb P^1}}} \cV_{{\mathbb P^1}} \to \cO_{{\mathbb P^1}}$. We need to check that $B$ factors through  $  \iota^* \cV_{{\mathbb P^1}}\otimes_{\cA_{{\mathbb P^1}}} \cV_{{\mathbb P^1}}$. This amounts to the following identity:
  $$B(\frac{\partial f_1}{\partial x_i}, f_2) + B(f_1, \frac{\partial f_2}{\partial x_i})=0,$$ 
  for $f_j\in A_0/J$, $1\leq i\leq n$, which in turn follows from the invariance property of the trace map (Remark \ref{rem:invariance of trace})
  $$ {\tr}  \circ L_{ \frac{\partial }{\partial x_i}}=0. $$
  
  For $\phi_{\cV_{{\mathbb P^1}}}\in  \Aut_{\cO_{{\mathbb P^1}}}(\cV_{{\mathbb P^1}})$, let $\phi_{\cV_{{\mathbb P^1}}}^t\in \Aut_{\cO_{{\mathbb P^1}}}( \iota^* \cV_{{\mathbb P^1}})$ be the adjoint automorphism with respect to the bilinear form  (\ref{app.duality1bis}) and $\nu(\phi_{\cV_{{\mathbb P^1}}})= (\phi_{\cV_{{\mathbb P^1}}}^t)^{-1}$.
 Define an isomorphism  
 $\widehat \tau $
   in the diagram below
  \begin{equation}\label{diag.f.z3}
	\begin{tikzcd}
	1 \arrow[r, ""] &\iota^* \underline{\cA}_{{\mathbb P^1}}^\times  \arrow[r, ""]\arrow[d, "\tau"]& \iota^* \widehat G_{{{\mathbb P^1}}, J}  \arrow[r, ""]\arrow[d, "\widehat \tau "]& G_{0, J}\times {{\mathbb P^1}}\arrow[r, ""]\arrow[d, "Id"] &1  \\
	1 \arrow[r, ""] &\underline{\cA}_{{\mathbb P^1}}^\times  \arrow[r, ""]&  \widehat G_{{{\mathbb P^1}}, J}  \arrow[r, ""] & G_{0, J}\times {{\mathbb P^1}}\arrow[r, ""] &1 
	\end{tikzcd}
	\end{equation}
	sending  $(\phi_{\cA_{{\mathbb P^1}}}, \phi_{\cV_{{\mathbb P^1}}}, \psi ) \in  \iota^* \widehat{G}_{{{\mathbb P^1}}, J}(T)$, for $T\in \cC$, to
$ (\alpha \phi_{\cA_{{\mathbb P^1}}} \!\alpha^{-1}, \nu(\phi_{\cV_{{\mathbb P^1}}}), \psi)$. Also, recall that $\tau(a)= \alpha(a)^{-1}$.  
We need to check that $ (\alpha \phi_{\cA_{{\mathbb P^1}}} \! \alpha^{-1},  \nu(\phi_{\cV_{{\mathbb P^1}}}))\in  \underline{\Aut}(\cA_{{\mathbb P^1}}, \!\cV_{{\mathbb P^1}})(T)$ for any $T\in \cC$. Since $T$ is flat over ${{\mathbb P^1}}$ and $\cV_{{\mathbb P^1}}$ is locally free it is enough to check this after replacing $T$ by $T\backslash T_0$. Over ${{\mathbb P^1}}\backslash \{0\}$ the morphism 
$\underline{\cA}_{{\mathbb P^1}}^\times  \to \underline{\Aut}(\cA_{{\mathbb P^1}}, \!\cV_{{\mathbb P^1}})$ is an isomorphism. Thus, we may assume that $(\phi_{\cA_{{\mathbb P^1}}}, \phi_{\cV_{{\mathbb P^1}}})= (\Ad_a, a)$, for some $a\in  \iota^* \underline{\cA}_{{\mathbb P^1}}^\times(T)$. Using that $B(a v, v')= B(v, \alpha(a) v') $,  we have that $ (\alpha \Ad_a \alpha^{-1},  \nu(a))= (\Ad_{\tau (a)}, \tau(a))$, as desired. 

By construction, the following diagram is commutative.  
 \begin{equation}\label{diag.f.z3}
	\begin{tikzcd}
	 \iota^* \widehat G_{{{\mathbb P^1}}, J}  \arrow[r, ""]\arrow[d, "\widehat \tau "]& \Aut_{\cO_{{\mathbb P^1}}}( \iota^* \cV_{{\mathbb P^1}})\arrow[d, "\nu"]   \\
	  \widehat G_{{{\mathbb P^1}}, J}  \arrow[r, ""] &\Aut_{\cO_{{\mathbb P^1}}}(\cV_{{\mathbb P^1}}). 
	\end{tikzcd}
	\end{equation}
Using Lemma \ref{app.connectedness} and Proposition \ref{pr:restrictionto0forgroups} there exists a unique section  $G_{0, J} \to  \widehat G_{{{\mathbb P^1}}, J}  \times_{{\mathbb P^1}}  \{0\}$ invariant under $\widehat \tau $. 
This section defines an action $\rho$  of  $G_{0, J}$   on the fiber $V_0=A_0/J$ of $\cV_{{\mathbb P^1}}$ over $0$,
 satisfying the following properties:
\begin{itemize}
\item[(i)] $\rho_g (a v)= g(a) \rho_g(v)$,
\item[(ii)] $B(\rho_g(v)\otimes \rho_g(v'))= B( v \otimes v').$
\end{itemize}
Using the description of $\Xi_0$ (see discussion before \eqref{eq:resttotheorigin}), we also have an isomorphism $\Xi_0 (\cV_{{Y^\lp}\times {{\mathbb P^1}}})_{|{Y^\lp} \times \{0\}} \iso  \cM_{X, Y, [\eta]} \times ^{G_{0,J}} V_0$.
Now  isomorphism  (\ref{app.squareroot}) is a consequence of the following.
\begin{lm}
Isomorphism (\ref{app.iden}) carries  the $G_{0, J}$-action $\rho$ on $V_0$ to the action of $G_{0, J}$ on  $ (\Omega^n_{\Spec A_0/J})^{\frac{1-p}{2}}$ through the natural homomorphism $G_{0, J}\to \underline{\Aut}( A_0/J)$
\end{lm}
\begin{proof}
Since the trace map $\tr\colon (\Omega^n_{\Spec A_0/J})^{1-p}  \to \mbb F_p$ is $\underline{\mr{Aut}}( A_0/J)$-invariant the second $G_{0, J}$-action on $V_0$ also satisfies properties (i) and (ii). We prove the Lemma by showing that there exists at most one $G_{0, J}$-action on $V_0$ satisfying (i) and (ii). Indeed, using (i), any two such actions $\rho$, $\rho'$ differ by the $1$-cocycle 
$c_g :=\frac{\rho_g (1)}{\rho'_g (1)}$  of  $G_{0, J}$ with values in $(A_0/J)^\times$. Using (ii) we see that $c_g^2\equiv 1$, that is $c_g$ is a homomorphism from $G_{0, J}$ to $\mu_2$. Since  $G_{0, J}$ is connected it follows that $c_g\equiv 1$.
\end{proof}
It remains to prove the uniqueness assertion of Theorem \ref{app.thquantoflarg}. Observe that, for any  flat morphism $f\colon W\to {{\mathbb P^1}}$  the group of automorphisms of $f^*\cV_{{\mathbb P^1}} \in \cA_{{\mathbb P^1}}\lmod(W)$ is $\cO^\times(W)$. 
Let  $(\cV_{{Y^\lp}\times {{\mathbb P^1}}}, \vartheta)$, $(\cV_{{Y^\lp}\times {{\mathbb P^1}}}^{ \prime}, \vartheta ')$ be two pairs as in Theorem \ref{app.thquantoflarg}.  By  the observation above, assigning to a flat morphism
 $q\colon T \to {Y^\lp}\times {{\mathbb P^1}}$ the set $L(T)$ of isomorphisms  $q^* \cV_{{Y^\lp}\times {{\mathbb P^1}}} \iso q^* \cV_{{Y^\lp}\times {{\mathbb P^1}}}^{\prime}$, we get a $\bG_m$-torsor $L$ over  ${Y^\lp}\times {{\mathbb P^1}}$. Isomorphisms $\vartheta$, $\vartheta'$ define a trivialization $\gamma$ of $L$ over ${Y^\lp}\times \{\infty\} $. The groupoid of $\bG_m$-torsors over  ${Y^\lp}\times {{\mathbb P^1}}$ equipped with a trivialization over ${Y^\lp}\times \{\infty\} $
 is discrete and its $\pi_0$ is the group   $H^0({Y^\lp}\times {{\mathbb P^1}}, \bZ)$ ({\it cf.} proof of the uniqueness part of Theorem  \ref{appendix:maintheorem}).  Using (\ref{app.thquantoflarg.eq}) and the inclusion $$(\pi_Y\times \Id)^*\colon H^0({Y^\lp}\times {{\mathbb P^1}}, \bZ) \mono H^0(\cM_{X, Y, [\eta]} \times {{\mathbb P^1}}, \bZ)$$ 
  we conclude that $(L, \gamma)$ is trivial. 

\end{proof}

\appendix

\section{Auxiliary lemmas on categorical quantization}
\subsection{Categorical sheaf property}\label{app:categorical sheaf property}

Let $k$ be a base commutative ring, and put $R_n=k[h]/(h^{n+1})$. Let $B_0$ be a flat commutative $k$-algebra.  We write
$X=\Spec(B_0)$ and denote by $\mr{Prin}(X)$ the basis of principal open subsets of $X$: it defines a site with a Grothendieck topology given by Zariski covers. For an algebra $A$ we denote the category of left $A$-modules by $A\lmod$. 

Let $B_n$ be an $R_n$-flat algebra together with an isomorphism
$B_n/hB_n\simeq B_0$. For $b\in B_0$, set
  $S_b=\{s\in B_n\mid s\equiv b^r \bmod{hB_n} \text{ for some }r\geq 0\}$. As we will recall below, $S_b$ is a two-sided denominator set\footnote{Namely, it satisfies both left and right Ore and denominator conditions. This shows that both left $S_b^{-1}B_n$ and right $B_nS_b^{-1}$ localizations exist and are isomorphic. We denote the thus unique 2-sided localization $B_n[S_b^{-1}]$.}. We write $B_{n,b}:=B_n[S_b^{-1}]$ for the resulting localization. Given this, we then consider the presheaf $\mr{Prin}(X)^{\op} \to \mr{LinCat}_{R_n}^{\mr{Ab}}$ sending $U=\Spec B_0[b^{-1}]$ to $B_{n,b}\lmod$.

  \begin{pr}\label{prop:Zariski descent for modules over deformation}
      The functor $U=\Spec B_0[b^{-1}] \mapsto B_n[S_b^{-1}]\lmod$ forms a Zariski sheaf on $\mr{Prin}(X)$ with values in $\mr{LinCat}_{R_n}^{\mr{Ab}}$.
  \end{pr}
  \begin{rem}
      For $n=0$ this is the usual Zariski descent for the category of modules.
  \end{rem}
  \begin{proof} Let us show that $S_b$ is a two-sided denominator set. Replacing $B_n$ by $B_n^{\op}$ it is enough to show that it satisfies left Ore and denominator conditions.  Since $B_n/h=B_0$ is commutative, for every $s\in S_b$ one has $\operatorname{ad}(s)(h^qB_n)\subset h^{q+1}B_n$ and so $\operatorname{ad}(s)^{n+1}=0$.  
  Given $b\in B_n$, equality $\operatorname{ad}(s)^{n+1}(b)=0$ implies the identity $$
          s^{n+1} b=\sum_{j=0}^{n+1}\binom {n+1}j\operatorname{ad}(s)^j(b)s^{n+1-j}.
$$
      This shows the left Ore condition: namely, that $s^{n+1}b=b's$ for some $b'\in B_n$. For the left denominator condition if $bs=0$, then $\operatorname{ad}(s)^{j}(b)=s^jb$, but since $\operatorname{ad}(s)^{n+1}=0$ we get $s^{n+1}b=0$.

  By \cite[Proposition II.3.5]{Stenstrom} $B_{n,b}:= B_n[S_b^{-1}]$ is flat both as a left and right $B_n$-module; in particular, it is flat over $R_n$. Also, by the universal property of localization we have a canonical isomorphism $B_{n,b}/h\simeq B_{0,b}:=B_0[b^{-1}]$. Thus $B_{n,b}$ is a flat deformation of $B_{0,b}$.

  Note that given two elements $b_1,b_2\in B_0$, multiplication produces a natural map
  $$
B_{n,b_1}\otimes_{B_n}B_{n,b_2}\to B_{n,b_1b_2}
  $$
of $(B_{n,b_1}, B_{n,b_2})$-bimodules. Both sides are flat modules over $R_n$, and the map is an isomorphism modulo $h$, thus the map above is an isomorphism. Iterating, the argument produces a canonical identification 
\begin{equation}\label{eq:ore-overlap}
    B_{n,b_1}\otimes_{B_n}\ldots \otimes_{B_n}B_{n,b_k}\xra{\sim}B_{n,b_1\ldots b_k}.
\end{equation}

We now check Zariski descent. Since every Zariski cover of an affine scheme by principal opens has a finite refinement, it is enough to consider a finite cover $X=\bigcup_{i=1}^m U_i$, where $U_i=\Spec B_0[b_i^{-1}]$. Set $C=\prod_i B_{n,b_i}$. It is flat as a right $B_n$-module. It is also faithful: indeed, if $M$ is a left $B_n$-module and $C\otimes_{B_n}M=0$, reduction modulo $h$ gives
$$
          (\prod_i B_0[b_i^{-1}])\otimes_{B_0}(M/hM)=0.
$$
     The product $\prod_iB_0[b_i^{-1}]$ is faithfully flat over $B_0$, so $M/hM=0$, and, consequently, $M=0$.

 Cipolla's faithfully flat descent theorem for associative rings, in the form of \cite[Theorem 3.8]{Nuss1997}, now gives an equivalence between the category $B_n\lmod$ and the category of descent data for the faithfully flat morphism $B_n\to C$. Under the decomposition $C=\prod_i B_{n,b_i}$
and the overlap isomorphisms \eqref{eq:ore-overlap},
this descent category is precisely the category
whose objects are collections of modules
$\{M_i\in B_{n,b_i}\lmod\}_{i=1,\ldots,m}$,
equipped with isomorphisms $\alpha_{ij}\colon B_{n,b_ib_j}\otimes_{B_{n,b_i}}M_i
          \xra{\sim}
          B_{n,b_ib_j}\otimes_{B_{n,b_j}}M_j$
      and satisfying the usual cocycle condition on triple intersections. This is exactly the category given by $
      \lim\left( \prod_{i} B_{n,b_i}\lmod \rightrightarrows \prod_{i_1,i_2} B_{n,b_{i_1}b_{i_2}}\lmod \mathrel{\substack{\rightarrow\\[-0.6ex]\rightarrow\\[-0.6ex]\rightarrow}} \prod_{i_1,i_2,i_3} B_{n, b_{i_1}b_{i_2}b_{i_3}}\lmod \right)$. The descent follows.
  \end{proof}
We denote this $\mr{LinCat}_{R_n}^{\mr{Ab}}$-valued sheaf by $\ul{B_n\lmod}$. It uniquely extends to a Zariski sheaf with values in $\mr{LinCat}_{R_n}^{\mr{Ab}}$ on the entire Zariski site of $X$.
\subsection{Local-to-global recognition principle}
In this subsection we show that the property ($\vcenter{\hbox{\large $*$}}$) in Definition \ref{def:categorical quantization char 0} is self-consistent: namely, any categorical quantization of $\Spec B_0$ is of the form $\ul{B_h\lmod}$ for some quantization $B_h$ of $(B_0,\{,\})$. (However, we warn the reader once again that the endomorphisms of $B_h$ and $\ul{B_h\lmod}$ are quite different.)

Let as before $X=\Spec B_0$. For any $n\ge 0$ consider the category $\mr{Shv}_{\mr{Zar}}(\mr{Prin}(X),\mr{LinCat}_{R_n}^{\mr{Ab}})$. By Proposition \ref{prop:Zariski descent for modules over deformation}, we have $\ul{B_n\lmod}\in \mr{Shv}_{\mr{Zar}}(\mr{Prin}(X),\mr{LinCat}_{R_n}^{\mr{Ab}})$; in particular, for $n=0$ we have $\ul{B_0\lmod}\in \mr{Shv}_{\mr{Zar}}(\mr{Prin}(X),\mr{LinCat}_{k}^{\mr{Ab}})$. 

We have a functor $\mr{Shv}_{\mr{Zar}}(\mr{Prin}(X),\mr{LinCat}_{R_n}^{\mr{Ab}})\to \mr{PShv}(\mr{Prin}(X),\mr{LinCat}_{k}^{\mr{Ab}})$ defined by sending $\mathscr C(U)\mapsto \mathscr C(U)\otimes_{R_n}k$ for each $U\in \mr{Prin}(X)$. Let us consider the category $$\mr{CatDef}_n(X):=\mr{Shv}_{\mr{Zar}}(\mr{Prin}(X),\mr{LinCat}_{R_n}^{\mr{Ab}})\times_{\mr{PShv}(\mr{Prin}(X),\mr{LinCat}_{k}^{\mr{Ab}})} \{\ul{B_0\lmod}\},$$
this is the (2,1)-category of categorical deformations of $\ul{B_0\lmod}$ along $R_n\twoheadrightarrow R_0=k$. An object of $\mr{CatDef}_n(X)$ is given by a sheaf $\mathscr C$ of Grothendieck $R_n$-linear abelian categories with a datum of an equivalence $\alpha\colon \mathscr C\otimes_{R_n}k\xra{\sim} \ul{B_0\lmod}$ as presheaves of Grothendieck $k$-linear abelian categories (in particular, $\mathscr C\otimes_{R_n}k$ is in fact a sheaf). For a deformation $B_n$, $\iota_{B_n}\colon B_n\otimes_{R_n}k\xra{\sim} B_0$ of $B_0$ along $R_n\twoheadrightarrow R_0=k$, the sheaf $\ul{B_n\lmod}$ together with the natural equivalence $\iota_{B_n}\colon \ul{B_n\lmod}\otimes_{R_n}k\xra{\sim} \ul{B_0\lmod}$ defines an object of $\mr{CatDef}_n(X)$.

\begin{pr}
    Let $(\mathscr C,\alpha)\in \mr{CatDef}_n(X)$. Assume that for a Zariski cover $\sqcup_{i=1}^m j_i\colon \sqcup_{i=1}^mU_i \to X$, with $U_i=\Spec B_{i,0}$, the pull-back $j^*_i\mathscr C\in \mr{CatDef}_n(U_i)$ is identified with $(\ul{B_{i,n}\lmod},\iota_{B_{i,n}})$ for some flat associative algebra deformation $B_{i,n}$ of $B_{i,0}$ over $R_n$. Then there exists a flat deformation $B_n$ of $B_0$ and an equivalence $(\mathscr C,\alpha)\simeq (\ul{B_n\lmod},\iota_{B_n})$.
\end{pr}
\begin{proof}
We can refine the cover so that each $B_{i,0}$ is a localization $B_0[b_i^{-1}]$ for some $b_i\in B_0$. The proof is analogous to \cite[Theorem A]{Lowen2005} and \cite[Theorems 8.15 and 8.16]{LowenVanDenBergh2006}. Put $\mathcal C:=\mathscr C(X)$ to be the category given by global sections; since it can be realized as a finite limit (given by truncated \v Cech diagram for the cover $\sqcup_{i=1}^mU_i \to X$) of categories of modules under exact functors (given by localizations) it is an abelian $R_n$-linear category. For each $i$ we have an exact restriction functor $\mc C\to B_{i,n}\lmod$. Let us call an object $M\in \mc C$ \textit{flat} over $R_n$ if its image in $B_{i,n}\lmod$ is flat over $R_n$ for all $i=1,\ldots, m$.

The equivalence $\alpha(X)$ identifies the reduction $\mc C\otimes_{R_n} R_n/h$ with $B_0\lmod$. We claim that an object $B_0\in B_0\lmod$ can be lifted to a flat object $P$ in $\mc C$. Indeed, locally over each $U_i$ lift exists, and one can see via the standard deformation theory argument that the obstruction to a global lift from $R_m$ to $R_{m+1}$ is given by a class in $\Ext^2_{B_0}(B_0,B_0)$ which is 0. The set of such liftings is then a torsor over $\Ext^1_{B_0}(B_0,B_0)$, which is also zero; so we get that such an object $P$ exists, and is unique up to an isomorphism.

Set $B_n:=\End_{\mc C}(P)^{\op}$. More generally, define a Zariski sheaf $\mc B_n$ of $R_n$-algebras on $\mr{Prin}(X)$ sending $U$ to $\End_{\mc C(U)}(P_{|U})^{\op}$; then $B_n= \mc B_n(X)$. Locally (over each $U_i$) $P_{|U_i}$ is a flat lift of $B_{i,0}$, and by uniqueness we have $P_{|U_i}\simeq \ul{B_{i,n}}$, so ${\mc B_n}_{|U_i}\simeq \ul{B_{i,n}}$ (as sheaves of algebras on $U_i$). Let $\mc B_m:=\mc B_n/h^{m+1}$. The isomorphism $\alpha$ induces an equivalence $\mc{B}_0\simeq \mc O_X$. 

We claim that $B_n$ is flat over $R_n$, and that the restriction maps $B_n\to \mc B_n(U_i)=B_{i,n}$ identify $B_{i,n}$ with $B_n[S_{b_i}^{-1}]$ (where $U_i=\Spec B_0[b_i^{-1}]$). Given the claim, we can identify $\ul{B_n\lmod}$ with $\mc B_n\lmod$, and consider the functor
$$
\Phi\colon \sC \to \mc B_n\lmod \simeq \ul{B_n\lmod},
$$
locally sending $M\in \sC(U)$ to $\Hom_{\sC(U)}(P_{|U}, M)$. Since both sides are sheaves, and restricted to each $U_i$ this functor is an equivalence, $\Phi$ is an equivalence of $\mr{LinCat}_{R_n}^{\mr{Ab}}$-valued sheaves.

It remains to prove the claim. Assume $B_m=B_n/h^{m+1}$ is flat over $R_m$, and let us show that $B_{m+1}$ is flat over $R_{m+1}$; since for $m=0$, $B_0$ is flat over $k$ we would be done by induction. Since $p_m\colon R_{m+1}\to R_m$ is a square-zero extension one can check flatness of $B_{m+1}$ by the following criterion: if $J_m:=\ker(p_m)$, then the natural map 
$$
J_m\otimes_{R_{m+1}} B_{m+1} \to J_m\cdot B_{m+1}
$$
is an isomorphism. The sheaf version of this statement (for $\mc B_{m+1}$ in place of $B_{m+1}$) is true (because it is true after restricting to each $U_i$). However, since both sides are killed by $h$, via $\alpha$, they define quasi-coherent sheaves on $X$. Since $X$ is affine, the same statement is then true for the global sections. Now, consider the map $B_n\to \mc B_n(U_i)=B_{i,n}$. The image of any element in $S_{b_i}$ is invertible in $\mc B_n(U_i)$: indeed, its reduction modulo $h$ is given by $b_i^n$ for some $n$, which is invertible in $B_{i,0}$. Thus this map extends to the map
$$
B_n[S_{b_i}^{-1}] \to B_{i,n}.
$$
But this map is an isomorphism modulo $h$, and since both sides are flat over $R_n$, the map is an isomorphism.
\end{proof}
\section{Denominator-free Weyl symmetrization map}\label{rem:denominator-free_weyl} The algebra $\mc A_{{\mathbb P^1}\backslash \{\infty \}}$ carries an interesting action of the algebraic group $\Sp(2n)$.
To see it, identify  the algebra  $A_{{\mathbb P^1}\backslash \{\infty\}}$ of global sections with 
the quotient of the free associative $\bF_p[h]$-algebra on a symplectic $2n$-dimensional vector space $(V, \omega)$ modulo the relations: $v\cdot w - w\cdot v = h\,  \omega(v,w)$, $v^p=0$, for all $v,w\in V$. The action of $\Sp(2n)$ on $V$ extends uniquely to $A_{{\mathbb P^1}\backslash \{\infty\}}$. For $d<p$, the Weyl symmetrization map 
\begin{equation}\label{eq:weyl}
   \Sym^d(V) \rar{W}   A_{{\mathbb P^1}\backslash \{\infty\}}, \quad W(v_1\otimes v_2 \otimes \cdots \otimes v_d)= \frac{1}{d!}\sum_{\sigma\in S_d} v_{\sigma(1)} v_{\sigma(2)}\cdots v_{\sigma(d)}
\end{equation}
is $\Sp(2n)$-equivariant. Identify $A_0$ with the quotient of the symmetric algebra of $V$.
\begin{pr}\label{pr:denominator-free_weyl}
  For $p>2$, maps  \eqref{eq:weyl} extend to a $\bF_p[h]$-linear isomorphism 
\begin{equation}\label{eq:weyl_stat}
   W\colon A_0[h]\iso A_{{\mathbb P^1}\backslash \{\infty\}}
\end{equation}
of representations of $\Sp(2n)$: for every $d>0$ one defines
  \begin{equation}\label{eq:Weyl-matching}
W(v_1\otimes v_2 \otimes \cdots \otimes v_d)
=
\sum_{M\in\mathcal M_d}
\left(-\frac h2\right)^{|M|}
\left(
  \prod_{\substack{\{i,j\}\in M\\ i<j}}
  \omega(v_i,v_j)
\right)
\left(
  \prod_{k\in U(M)}^{\nearrow}v_k
\right).
\end{equation}
Here  $\mathcal M_d$ is the set of partial matchings on
$\{1,\ldots,d\}$ {\it i.e.}, collections of disjoint subsets of $\{1,\ldots,d\}$ of size $2$. 
For such a collection $M\in \mathcal M_d$, we denote by $U(M)\subset \{1,\ldots,d\}$ the complement to the union of subsets in $M$. The last product is taken in increasing order of indices in $U(M)$.
\end{pr}
\begin{proof}
   We start by showing that the right-hand side of \eqref{eq:Weyl-matching} is symmetric in 
   $v_1,\cdots ,v_d$. Let $\tilde V$ be a free $\Lambda:= \bZ[1/2]$-module with $\tilde V \otimes \bF_p\iso V$, $\tilde \omega$ a symplectic form on $\tilde V$ lifting $\omega$, and let
   $\cW$ be the free associative $\Lambda[h]$-algebra on $\tilde V$ modulo the relations: $v\cdot w - w\cdot v = h\,  \tilde \omega(v,w)$. This is a flat algebra over $\Lambda$.
   For every $d$, formulae \eqref{eq:weyl} and  \eqref{eq:Weyl-matching} define $2$ maps $(\tilde V)^{\otimes d} \to \cW\otimes \bQ$; the first one factors through $\Sym^d(\tilde V) $ while the second one lands in $\cW$. Let us check that the two maps are equal. Since the commutators $[v_i, v_j]$ are central in $\cW$, the Baker--Campbell--Hausdorff formula
terminates after its first commutator term:
\[
\exp\left(
\sum_i t_iv_i
\right)=e^{t_1v_1}\cdots e^{t_dv_d}
\exp\left(-\frac h2\sum_{i<j}t_it_j\tilde \omega(v_i,v_j)\right) 
\]
in $(\cW\otimes \bQ)[[t_1, \cdots, t_d]]$.
Comparing the coefficients at $t_1\cdots t_d$ gives the equality. This proves that \eqref{eq:Weyl-matching} defines a map $\Sym^\star V \to A_{{\mathbb P^1}\backslash \{\infty\}}$ which is clearly $\Sp(2n)$-equivariant. Let us check that the map factors
through $A_0$. It is enough to verify that  $W(v^{\otimes p}\otimes v_1\otimes  \cdots \otimes v_d)=0$. To see this, group the terms in
\eqref{eq:Weyl-matching} according to the number $r$ of the $p$ copies of $v$
which are paired with factors of $v_1\otimes  \cdots \otimes v_d$.
 If $r=0$, the unpaired product contains $v^p$, which vanishes in the
reduced Weyl algebra. If $r>0$, every fixed contraction occurs with multiplicity
\(
  p(p-1)\cdots(p-r+1),
\)
which is divisible by $p$. It remains to check that $W\colon A_0[h]\to A_{{\mathbb P^1}\backslash \{\infty\}}$ is an isomorphism. To see this observe, that $W$ preserves the grading 
(given by $\deg V=1$, $\deg h=2$) and it is the identity modulo $h$. This is enough.
\end{proof}

\section{The surjectivity of the map $L\!\GL(n) \to L\!\PGL(n)$ }\label{app:surjectivity}
For a group scheme $G$ over $\bZ$, let $L\, G$ be the {\it fppf} sheafification of the group presheaf on affine schemes over $\bZ$ sending a ring $R$ to $G(R((h)))$.
The original proof of the following result emerged from a conversation between the last author, Akhil Mathew, and Federico Scavia. The slightly different argument given below was suggested to us by ChatGPT.
\begin{pr}\label{pr:surjectivity}
Let $R$ be a ring and let $L$ be an invertible $R[x]$-module whose class
in $\Pic(R[x])$ has finite order. There is an fppf $R$-algebra $R'$ such
that
\[
  L\otimes_{R[x]}R'[x]\simeq R'[x].
\]
Consequently, for every $n\geq1$, the sequence 
\begin{equation}\label{eq:exact_seq_of_loop_groups}
 1\longrightarrow L\bG_m\longrightarrow L\!\GL(n)
   \longrightarrow L\!\PGL(n)\longrightarrow 1    
\end{equation}
is exact for the fppf topology on affine schemes over $\bZ$.
\end{pr}
\begin{proof}
Recall that an extension $A\subset B$ is \emph{subintegral} if it is
integral and induces a bijection on prime spectra and isomorphisms on
all residue fields. Every reduced ring $A$ has a seminormalization
$A\subset A^{\sn}$, which is a subintegral extension with $A^{\sn}$
seminormal; see \cite[Theorem~4.1]{Swan1980}. We will also use the
Traverso--Swan theorem in the following form:
\begin{equation}\label{eq:swan}
  \Pic(B)\xrightarrow{\sim}\Pic(B[x])
  \qquad\text{if $B$ is seminormal.}
\end{equation}
This is \cite[Theorem~1]{Swan1980}.

\begin{lm}\label{lem:saturation}
Let $A$ be a reduced ring such that every monic polynomial of positive
degree in $A[T]$ has a root in $A$. If $A\subset B$ is a subintegral
extension with $B$ reduced, then $B/A$ is torsion-free as an abelian
group. In particular, $A^{\sn}/A$ is torsion-free.
\end{lm}

\begin{proof}
First observe that, for $c\in B$ and an integer $m>0$,
\[
  c^2,c^3,mc\in A\quad\Longrightarrow\quad c\in A. \tag{$*$}
\]
Indeed, all powers $c^i$ with $i\geq2$ belong to $A$, so
\[
  (T-c)^m=T^m-mcT^{m-1}
       +\sum_{i=2}^m(-1)^i\binom mi c^iT^{m-i}
  \in A[T].
\]
By hypothesis this polynomial has a root $a\in A$. Since
$(a-c)^m=0$ and $B$ is reduced, we get $c=a\in A$.

Now suppose that $b\in B$ and $mb\in A$ for some integer $m>0$.
Set $C=A[b]$. This is a finite subintegral extension of $A$.
Indeed, $b$ is integral over $A$, and an intermediate extension of
a subintegral extension is subintegral. If $b$ satisfies a monic equation of
degree $d\geq1$, then $C$ is generated over $A$ by
$1,b,\ldots,b^{d-1}$. Consequently, for $N=m^{d-1}>0$,
\[
  NC\subset A.
\]
Consider the conductor
\[
  I=\{c\in C:cC\subset A\}
   =\operatorname{Ann}_A(C/A).
\]
It is an ideal of both $A$ and $C$. We claim that $I$ is radical in
$C$. Otherwise $C/I$ has a nonzero square-zero element: choose
$u\notin I$ with $u^2\in I$. For every $v\in C$,
\[
  (uv)^2,(uv)^3\in I\subset A,
  \qquad Nuv\in A.
\]
Applying $(*)$ to $uv$ shows that $uv\in A$. Thus $uC\subset A$,
which means $u\in I$, a contradiction.

Suppose that $C\ne A$, so that $I$ is proper. Choose a prime
$\mathfrak p\subset A$ minimal over $I$. Both
\[
  A_{\mathfrak p}/I_{\mathfrak p}
  \quad\text{and}\quad
  C_{\mathfrak p}/I_{\mathfrak p}
\]
are reduced: $C/I$ is reduced, and $A/I$ is a subring of $C/I$.
Moreover, each has exactly one prime ideal. For the first ring
this follows from the minimality of $\mathfrak p$; for the second
it follows additionally from the bijection on spectra supplied by
subintegrality. A reduced ring with exactly one prime ideal is a
field. Since subintegrality preserves residue fields, the natural
map is an isomorphism
\[
  A_{\mathfrak p}/I_{\mathfrak p}
  \xrightarrow{\sim}
  C_{\mathfrak p}/I_{\mathfrak p}.
\]
As $I_{\mathfrak p}\subset A_{\mathfrak p}$, this implies
$C_{\mathfrak p}=A_{\mathfrak p}$.

On the other hand, $C/A$ is a finite $A$-module with annihilator
$I$, so its localization at a prime containing $I$ cannot vanish.
Explicitly, if $(C/A)_{\mathfrak p}=0$, finitely many generators
give an element $s\notin\mathfrak p$ annihilating $C/A$. Then
$s\in I\subset\mathfrak p$, a contradiction.

Therefore $C=A$, and hence $b\in A$. Taking $B=A^{\sn}$ gives
the final assertion.
\end{proof}

\begin{lm}\label{lem:npic}
Let $A$ be reduced, and suppose that $A^{\sn}/A$ is torsion-free as
an abelian group. Then $\NPic(A)\colon=\ker(\Pic(A[x]) \rar{x=0}\Pic(A))$ is torsion-free.
\end{lm}

\begin{proof}
Put $B=A^{\sn}$. Suppose that $L$ is an invertible $A[x]$-module,
equipped with isomorphisms
\[
  \lambda:L/xL\xrightarrow{\sim}A,
  \qquad \theta:L^{\otimes n}\xrightarrow{\sim}A[x]
\]
for some $n\geq1$. Multiplying $\theta$ by a constant in $A^\times$,
we may assume that its restriction to $x=0$ is $\lambda^{\otimes n}$.

By \eqref{eq:swan}, $L_B:=L\otimes_{A[x]}B[x]$ is trivial.
Choose an isomorphism $\alpha:L_B\xrightarrow{\sim}B[x]$ whose
restriction at $x=0$ is $\lambda\otimes_A B$. The two trivializations
$\theta_B$ and $\alpha^{\otimes n}$ of $L_B^{\otimes n}$ differ by
a unit in $B[x]$. As $B$ is reduced, every unit of $B[x]$ is constant;
the normalizations at $x=0$ show that this unit is $1$.

The map $L\to L_B$ is injective: it is obtained by tensoring the
injection $A[x]\hookrightarrow B[x]$ with the flat $A[x]$-module $L$.
Using $\alpha$, identify $L$ with an $A[x]$-submodule
$I\subset B[x]$. We have
\begin{equation}\label{eq:ideal}
  I(0)=A,\qquad I^n=A[x]\quad\text{inside }B[x],
\end{equation}
where $I(0)$ is the image under evaluation at zero and $I^n$ is the
$A[x]$-module generated by products of $n$ elements of $I$.

We claim that
\begin{equation}\label{eq:coefficients}
  h\in B[x],\quad h(0)=1,\quad h^n\in A[x]
  \quad\Longrightarrow\quad h\in A[x].
\end{equation}
Write $h=1+\sum_{j\geq1}b_jx^j$. If $b_1,\ldots,b_{j-1}$ are in
$A$, the coefficient of $x^j$ in $h^n$ is $nb_j$ plus an integral
polynomial in these earlier coefficients. Thus $nb_j\in A$.
Torsion-freeness of $B/A$ gives $b_j\in A$, proving the claim by
induction.

Choose $f\in I$ with $f(0)=1$, which is possible by
\eqref{eq:ideal}. For any $g\in I$, both $f$ and $f+xg$ have constant
term $1$, and their $n$th powers belong to $A[x]$. By
\eqref{eq:coefficients}, $f,f+xg\in A[x]$, hence $g\in A[x]$.
Therefore $I$ is an ideal of $A[x]$, and
\[
  A[x]=I^n\subset I\subset A[x].
\]
It follows that $I=A[x]$, so $L$ is trivial.
\end{proof}

\begin{lm}\label{lem:cover}
Every ring $R$ admits an ind-fppf algebra $R_\infty$ such that every
monic polynomial of positive degree over $R_\infty$ has a root in
$R_\infty$.
\end{lm}

\begin{proof}
For a ring $C$, let $\mathcal P(C)$ be the set of monic polynomials
of positive degree over $C$, and form
\[
  C^+=\bigotimes_{P\in\mathcal P(C),\,C} C[T_P]/(P(T_P)).
\]
The tensor product here is the filtered colimit over finite subsets
of $\mathcal P(C)$. Each finite tensor product is finite free of
positive rank over $C$, and thus faithfully flat and finitely
presented. In particular $C\to C^+$ is ind-fppf.

Set $R_0=R$, $R_{i+1}=R_i^+$, and
$R_\infty=\varinjlim_{i\geq0}R_i$. Ind-fppf maps are stable under
composition and filtered colimits; see
\cite[\S2.3.3]{CS24}. Therefore $R\to R_\infty$ is ind-fppf. The
finitely many coefficients of any monic polynomial over $R_\infty$
come from some $R_i$, and the polynomial acquires a root over
$R_{i+1}$.
\end{proof}

Now we can prove the first assertion of Proposition 
\ref{pr:surjectivity}.
Choose $n\geq1$ with $L^{\otimes n}\simeq R[x]$. After a preliminary
fppf base change, we may suppose that $L/xL\simeq R$. For example,
take a finite principal-open cover trivializing $L/xL$ and replace
$R$ by the product of the corresponding localizations.

Choose $R_\infty$ as in Lemma~\ref{lem:cover} and put
$A=(R_\infty)_{\reduced}$. Every monic polynomial over $A$ lifts to a
monic polynomial over $R_\infty$, so it has a root in $A$.
Apply Lemma~\ref{lem:saturation} to $A\subset A^{\sn}$ and then
Lemma~\ref{lem:npic} to conclude that $\NPic(A)$ is torsion-free.
Since the class of $L_{A[x]}$ is
$n$-torsion and trivial at $x=0$, it follows that
\[
  L_{A[x]}\simeq A[x].
\]

Triviality lifts across a nil ideal. For completeness, if $J$ is a
nil ideal of a ring $C$ and $M$ is a finite projective $C$-module with
$M/JM\simeq C/J$, lift the chosen generator to $s\in M$. The map
$C\to M$ taking $1$ to $s$ is surjective by Nakayama's lemma, since
$J$ is contained in the Jacobson radical. It splits, and its kernel
is a finite projective module whose reduction modulo $J$ is zero;
a second application of Nakayama's lemma makes the kernel zero.
Applying this to $C=R_\infty[x]$ and
$J=\sqrt{0_{R_\infty}}[x]$, we obtain
\[
  L_{R_\infty[x]}\simeq R_\infty[x].
\]
Notice that $J$ is a nil ideal even if the nilradical of $R_\infty$
is not nilpotent.

By the definition of ind-fppf, write
$R_\infty=\varinjlim_\lambda R_\lambda$, with each $R_\lambda$
faithfully flat and finitely presented over the current base. The
isomorphism above and its inverse descend to some $R_\lambda[x]$:
the module $L$ is finitely presented, and the maps and the two
inverse identities involve only finitely many coefficients and
relations. Composing with the preliminary fppf base change gives
the required fppf algebra over the original ring $R$.

To prove the second assertion of Proposition \ref{pr:surjectivity} we need the following.
\begin{lm}\label{lem:integer-cohomology}
For every ring $R$, the group $H^1_{\et}(\Spec R,\bZ)$ is torsion-free.
\end{lm}

\begin{proof}
For every $n\geq1$, the map
\[
  H^0_{\et}(\Spec R,\bZ)\longrightarrow
  H^0_{\et}(\Spec R,\bZ/n\bZ)
\]
is surjective: a locally constant function with values in $\bZ/n\bZ$
lifts by choosing an integer representative of each value on its
open-and-closed inverse image. The cohomology sequence of
$0\to\bZ\xrightarrow{n}\bZ\to\bZ/n\bZ\to0$ then shows that
multiplication by $n$ on $H^1_{\et}(\Spec R,\bZ)$ is injective.
\end{proof}

\begin{cor}\label{cor:laurent}
Every torsion line bundle over $R((t))$ becomes trivial after an
fppf base change $R\to R'$.
\end{cor}
\begin{proof}
Gabber's formula, as proved in
\cite[Theorem~3.1.7, equation~(3.1.7.1), with $T=\bG_m$]{bc}, is a
natural isomorphism
\begin{equation}\label{eq:gabber}
  \Pic(R[t^{-1}])\oplus H^1_{\et}(\Spec R,\bZ)
  \xrightarrow{\sim}\Pic(R((t))).
\end{equation}
By Lemma~\ref{lem:integer-cohomology}, a torsion class on the right
comes from a torsion class in $\Pic(R[t^{-1}])$. Apply the first
assertion of Proposition \ref{pr:surjectivity}, already proved, with $x=t^{-1}$,
and use the naturality of \eqref{eq:gabber}.
\end{proof}

Put $C=R((t))$ and let $g\in\PGL_n(C)$. The fppf exact sequence
\[
  1\longrightarrow\bG_m\longrightarrow\GL_n
    \longrightarrow\PGL_n\longrightarrow1
\]
associates to $g$ a class $\delta(g)\in\Pic(C)$. Its image in
$H^1_{\mathrm{fppf}}(C,\GL_n)$ is trivial. Thus its associated line
bundle $M$ satisfies
\[
  M^{\oplus n}\simeq C^n,
  \qquad\text{and consequently}\qquad
  M^{\otimes n}\simeq C
\]
by taking determinants. Corollary~\ref{cor:laurent} provides an fppf
$R$-algebra $R'$ over which this obstruction vanishes. Therefore
$g$ lifts to $\GL_n(R'((t)))$.

The kernel of $\GL_n(R((t)))\to\PGL_n(R((t)))$ consists of scalar
matrices, namely $R((t))^\times$. This proves exactness of the
claimed sequence of fppf sheaves.

 \end{proof}
\begin{rem}\label{rem:the_case_of_F_p_and_Q-algebras} If $R$ is an $\bF_p$-algebra then, for \emph{any}  $L\in \Pic(R[x])$,
there is an {\it fppf} $R$-algebra $R'$ such that the image of $L$ in $\Pic(R'[x])$ is trivial. This is a result of Drinfeld (\cite[Cor. 3.4]{bmv}). In general, the above stronger assertion is false; see \S 3.1 of {\it loc.cit.}, for a counterexample due to Gabber. 
Using Gabber's formula,
\cite[Theorem~3.1.7]{bc} the same result holds for  $L\in \Pic(R((x)))$ {\it i.e.},
for $\bF_p$-algebras,  the torsion assumption in Corollary \ref{cor:laurent} is redundant. 
On the other hand, if $n$ is invertible in $R$ then $H^1_{\et}(\Spec R,\mu_n)\iso H^1_{\et}(\Spec R[x],\mu_n)$ implies that the $n$-torsion part of $\NPic(R)$ is trivial. Together with Gabber's result \eqref{eq:gabber} it implies
that sequence \eqref{eq:exact_seq_of_loop_groups} is exact even for the \emph{Zariski} topology on the category of affine schemes over $\bZ[n^{-1}]$. 
\end{rem}

\makeatletter
\let\original@lbibitem\@lbibitem
\def\@lbibitem[#1]#2{%
  \ifstrequal{#2}{BKMVZ}
    {\original@lbibitem[BKMVZ]{#2}}
    {\original@lbibitem[#1]{#2}}%
}
\makeatother

\bibliographystyle{alphaurl}
\bibliography{QC-references}

\end{document}